\documentclass{amsart}
\usepackage{graphics}
\usepackage[all]{xy}

\newtheorem{defn}[subsubsection]{Definition}
\newtheorem{thm}[subsubsection]{Theorem}
\newtheorem{lem}[subsubsection]{Lemma}
\newtheorem{cor}[subsubsection]{Corollary}

\newtheorem{prop}[subsubsection]{Proposition}

\newtheorem*{MT}{Main Theorem}

\newcommand{\Aspace}{\mathcal{A}}
\newcommand{\Bspace}{\mathcal{B}}
\newcommand{\apply}{\vdash}
\newcommand{\ab}{[a,b]}
\newcommand{\abpow}{[[a,b]]}
\newcommand{\Wspace}{\mathcal{W}}

\newcommand{\What}{\widehat{\mathcal{W}}}

\newcommand{\Whatwedge}{\widehat{\mathcal{W}}_\wedge}

\newcommand{\WhatF}{\widehat{\mathcal{W}}_{\mathrm{\bf F}}}

\newcommand{\basebulltoF}{\text{B}_{\bullet\rightarrow\mathrm{F}}}

\newcommand{\intoop}{
\text{B}_{
\begin{subarray}{c}{\,\bullet\,\mapsto\partial_a} \\
{\perp\mapsto\partial_b}\end{subarray}}}
\newcommand{\applyhash}{\#}
\newcommand{\baselegtopartial}{
\text{B}_{l\mapsto\partial_a}}

\begin{document}
\title[Differential Operators and The Wheels Power Series]{Differential Operators \\ and the Wheels power series}
\author[A. Kricker]{Andrew Kricker}
\address{Division of Mathematical Sciences \\ School of
Mathematical and Physical Sciences \\ Nanyang Technological
University, Singapore, 637616} \email{ajkricker@ntu.edu.sg}
\thanks{The author thanks Dror Bar-Natan and Eckhard Meinrenken for their support at the University of Toronto during the most important part of
this work.}
\begin{abstract}
An earlier work of the author's showed that it was possible to adapt the Alekseev-Meinrenken Chern-Weil proof of the Duflo isomorphism to obtain a completely combinatorial proof of the Wheeling isomorphism. That work depended on a certain combinatorial identity, which said that a certain composition of elementary combinatorial operations arising from the proof was precisely the Wheeling operation. The identity can be summarized as follows: The Wheeling operation is just a graded averaging map in a space enlarging the space of Jacobi diagrams. The purpose of this paper is to present a detailed and self-contained proof of this identity. The proof broadly follows similar calculations in the Alekseev-Meinrenken theory, though the details here are somewhat different, as the algebraic manipulations in the original are replaced with arguments concerning the enumerative combinatorics of formal power series of graphs with graded legs.
\end{abstract}
 \maketitle

At first glance, the operation which appears in the statement of the Wheeling isomorphism -- $\chi_\Bspace\circ\partial_\Omega : \Bspace \rightarrow \Aspace$ -- does not
seem to be a particularly natural operation. The purpose of this paper is to provide a detailed proof of an identity which says that $\chi_\Bspace\circ \partial_\Omega$
can be factored into a sequence of elementary combinatorial operations. In summary: Wheeling is just a graded averaging map in a space which enlarges $\Aspace$.

\section{Recalling the identity and the spaces and maps involved}
First we'll state the identity in question. After that we'll recall the definitions of the various spaces and maps that are involved.
See \cite{K} for more detailed definitions, if required. Figure \ref{illustrfig} illustrates the first of the two compositions appearing in the main theorem.
\begin{MT}\label{maincombinatorialtheorem}
The composition of maps
\[
\xymatrix{ \mathcal{B} \ar[r]^{\Upsilon} & \mathcal{W} \ar[r]^{\chi_\Wspace} & \widetilde{\mathcal{W}} \ar[r]^{\pi} &
\widehat{\mathcal{W}} \ar[r]^{\basebulltoF} & \widehat{\mathcal{W}}_{\mathrm{F}}\ar[r]^\lambda  &  \widehat{\mathcal{W}}_{\wedge} }
\]
is equal to the composition
\[
\xymatrix{ \mathcal{B} \ar[r]^{\partial_\Omega} & \mathcal{B} \ar[r]^{\chi_{\mathcal{B}}} & \mathcal{A} \ar[r]^{\phi_{\mathcal{A}}} & \widehat{\mathcal{W}}_{\wedge}}.
\]
\end{MT}

Each of the spaces in the above theorem is defined as a $\mathbb{Q}$-vector space consisting of formal finite $\mathbb{Q}$-linear combinations of abstract graphs with vertices of degree 1 and degree 3, modulo certain relations which depend on the specific space. Some of the spaces above have two types of degree 1 vertex, others have only one, as will be described below. The edges are always unoriented. The trivalent vertices are always oriented, which means that at each trivalent vertex, the set of incoming edges is given a cyclic ordering.  Each of the above spaces includes the AS and IHX relations in its set of relations (if unfamiliar with these perennials, see \cite{BarNatan}). So from the viewpoint of the {\it internal} ($=$trivalent) vertices, all of the above spaces are the same.

The only place these spaces differ is in how they treat the vertices of degree 1. We will describe the differences in some detail.

The first thing to say is that in each of the above spaces, except $\mathcal{B}$, the set of degree 1 vertices of a diagram is totally ordered. The diagrams are drawn by ordering the degree 1 vertices along an {\it ordering line} at the bottom of the diagram. We'll remind the reader of which space a diagram is to be considered an element of by drawing the arrow on the ordering line with a different style for each space.

\subsection{The averaging map $\chi_\Bspace:\Bspace \rightarrow \Aspace$.}
The space $\Bspace$ only has one type of degree 1 vertex. As mentioned above, these vertices are not ordered. The diagrams which generate this space will be called {\it symmetric Jacobi diagrams} in this paper. The only relations in $\Bspace$ are the always-present IHX and AS relations.
The space $\Aspace$, on the other hand, is based on diagrams where the degree 1 vertices are ordered.
These diagrams will be called {\it ordered Jacobi diagrams} in this paper.
The space $\Aspace$ will be taken modulo AS, IHX, and STU relations.
The averaging map $\chi_\Bspace$ is the linear extension of the map which maps a symmetric Jacobi diagram
in $\Bspace$ to the average of all the possible ordered Jacobi diagrams obtained by ordering the degree 1-vertices. For example:
\[
\chi_\Bspace\left(
\raisebox{-5ex}{\scalebox{0.24}{\includegraphics{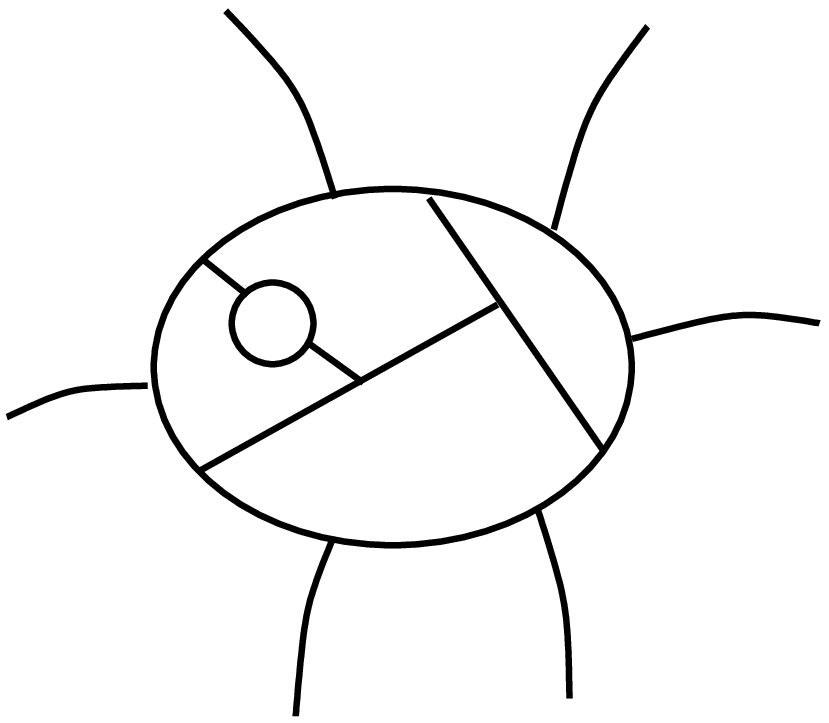}}}
\right) = \frac{1}{6!}\sum_{\sigma}
\raisebox{-6ex}{\scalebox{0.24}{\includegraphics{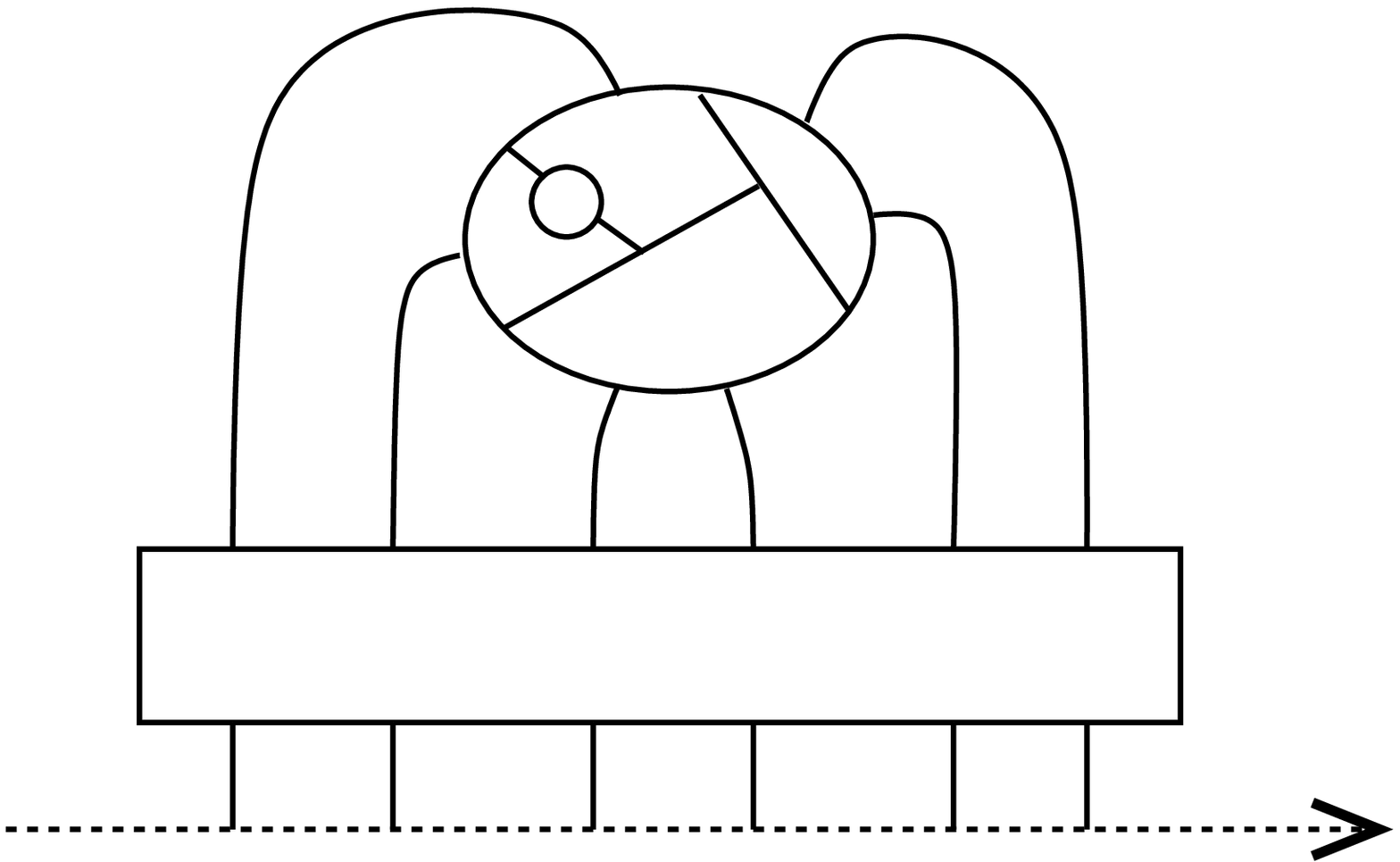}}}
\hspace{-2.25cm}
\raisebox{-2.5ex}{\scalebox{1.25}{$\sigma$}}\hspace{1.5cm}\ \ \in \Aspace.
\]

It turns out that while this map is an isomorphism of vector spaces (this is the formal Poincare-Birkhoff-Witt theorem described in \cite{BarNatan}), it is {\it not} an isomorphism of algebras. In other words, if $u$ and $v$ are elements of $\Bspace$ then it is not in general true that $\chi_\Bspace(u\sqcup v)=\chi_\Bspace(u)\# \chi_\Bspace(v)$, where $\sqcup$ is the ``disjoint-union" product on $\Bspace$, and $\#$ is the ``juxtaposition" product on $\Aspace$.

\subsection{The Wheeling map, $\partial_\Omega: \Bspace \rightarrow \Bspace$.}

The Wheeling operation $\partial_\Omega: \Bspace \rightarrow \Bspace$ is a linear isomorphism of $\Bspace$ which promotes $\chi_\Bspace$ to an algebra isomorphism:
\[
\left(\chi_\Bspace \circ \partial_\Omega\right)(u\sqcup v)=\left(\chi_\Bspace \circ \partial_\Omega\right)(u) \# \left(\chi_\Bspace \circ \partial_\Omega\right)(v)
\]
for all $u,v\in \Bspace$.

This identity - the Wheeling isomorphism - is a combinatorial strengthening of a Lie theoretic result from the 1970's known as the Duflo isomorphism \cite{D}.
It was conjectured by Bar-Natan, Garoufalidis, Rozansky, and Thurston \cite{BGRT}, who were mostly motivated by the theory of the Kontsevich integral (a topological invariant of framed tangles), and constructions related to it.
Bar-Natan, Le, and Thurston gave an elegant proof of this identity which directly employed the Kontsevich integral (``$1+1=2$"), \cite{BLT}.
Since then Wheeling has proved to be an indispensable tool in the study of the structure of the Kontsevich integral and related constructions, (see e.g. \cite{GK}).

The BLT proof shows that Wheeling is deeply bound up with the theory of the Kontsevich integral,
and such things as the theory of associators, the monodromy of the Knizhnik-Zamolodchikov equations, and the theory of quantum groups.
The aim of this paper, together with its companion \cite{K}, is to describe a completely combinatorial proof of the Wheeling isomorphism, with the goal of discovering new approaches to these topics. This proof derives from work of Alekseev and Meinrenken \cite{AleksMein,AleksMein2}, as is discussed in \cite{K}.

So, what is this map $\partial_\Omega$? First we must recall what is $\partial_X(Y)\in \Bspace$, the result of operating on a symmetric Jacobi diagram $Y$ with a symmetric Jacobi diagram $X$. The result is the sum of all the possible symmetric Jacobi diagrams that you obtain by gluing all of the legs of $X$ to some (possibly all) of the legs of $Y$. This is extended linearly to define $\partial_u(v)$, for abitrary $u,v\in \Bspace$.

In the case that $X$ has more legs than $Y$, $\partial_X(Y)$ will be zero. This means that it is meaningful to consider operations of the form $\partial_{\mathcal P}:\Bspace\rightarrow \Bspace$, where ${\mathcal P}$ is an infinite combination of symmetric Jacobi diagrams (a ``formal power series of diagrams"), as long as for each $b\in \mathbb{N}$, the piece of ${\mathcal P}$ consisting of the diagrams whose number of degree 1 vertices is less than $b$ is finite.
(These issues are carefully discussed, from a more general viewpoint, in Section \ref{diagoperators}.)

To recall $\Omega$, the power series appearing in Wheeling, we'll use the following convenient notation for generating a formal power series of symmetric Jacobi diagrams:
\begin{multline*}
\raisebox{-6ex}{\scalebox{0.21}{\includegraphics{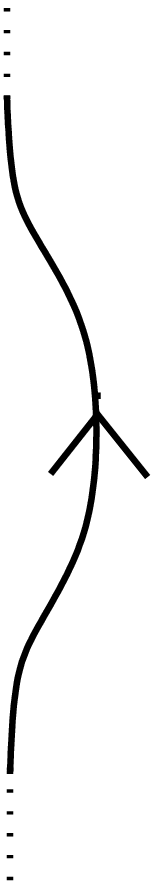}}} \
c_0 +  c_1a + c_2a^2 + c_3a^3 + \ldots  \\  \ \ \mbox{denotes}\ \
\ \ \ \ c_0\
\raisebox{-6ex}{\scalebox{0.21}{\includegraphics{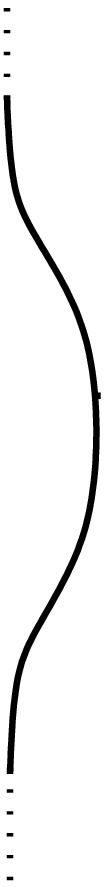}}}\
\ +\ \  c_1\
\raisebox{-6ex}{\scalebox{0.21}{\includegraphics{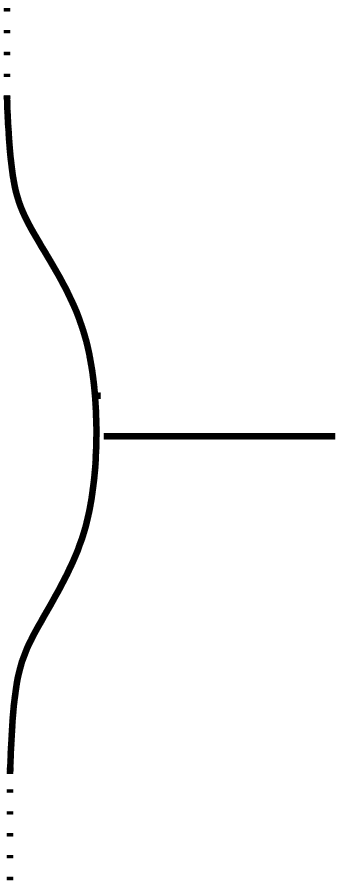}}} \
\ +\ \  c_2\
\raisebox{-6ex}{\scalebox{0.21}{\includegraphics{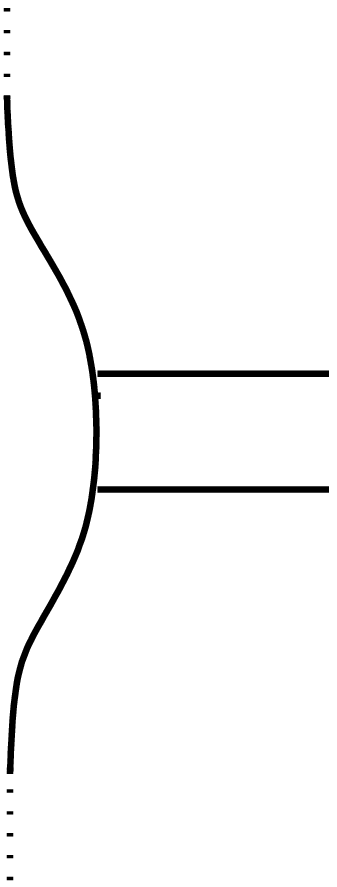}}} \
\ +\ \ c_3\
\raisebox{-6ex}{\scalebox{0.21}{\includegraphics{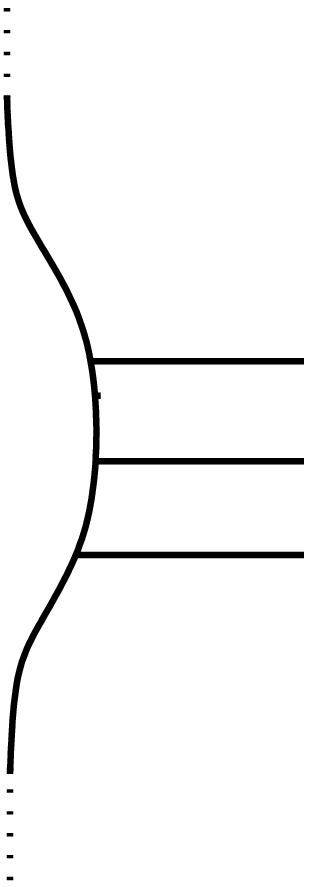}}} \
\ +\ \ldots\ \ \ \ .
\end{multline*}
The following precise statement uses $\Bspace^n$, which is the subspace of $\Bspace$ generated by diagrams with precisely $n$ degree 1 vertices.
\begin{defn}
The Wheels element, $\Omega$, is the formal power series of
symmetric Jacobi diagrams defined by the expression
\[
\Omega = \exp_{\sqcup}\left(\frac{1}{2}\
\raisebox{-3.4ex}{\scalebox{0.28}{\includegraphics{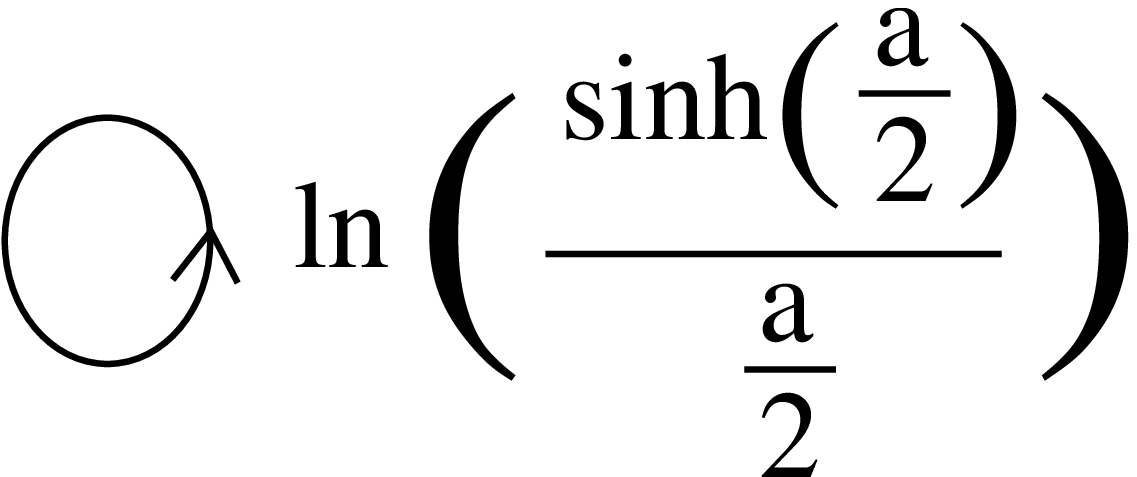}}}
\right)\ \in\ \prod_{n=0}^{\infty}\Bspace^n\ .
\]
\end{defn}

Now we turn our attention to the first sequence of compositions described in the main theorem: $\lambda\circ\basebulltoF\circ\pi\circ\chi_{\Wspace}\circ\Upsilon$. This is a sequence of elementary combinatorial operations which the main theorem claims has the same effect on a symmetric Jacobi diagram as operating with $\Omega$.

\subsection{The space $\Wspace$.}

In the earlier work \cite{K}, $\Wspace$ was introduced as an ``$\iota$-complex", which was a pair of cochain complexes equipped with a degree $-1$ map between them. In this work we have no need for all this extra structure,
and $\Wspace$ just denotes the vector space underlying the structure.

Recall: the diagrams which generate $\Wspace$, which will be called ``symmetric Weil diagrams" in this paper, have degree 1 vertices of two different types. There are ``leg-grade 1" vertices, which are drawn without any decoration, and ``leg-grade 2" vertices, which are drawn with a fat dot. The space $\Wspace$ consists of formal finite $\mathbb{Q}$-linear combinations of Weil diagrams, modulo AS and IHX relations, and also relations which say that when we transpose the position of two adjacent legs in the ordering, we pick up a sign $(-1)^{xy}$, where $x$ and $y$ are the leg-grades of the involved legs.

So, for example, the following equations hold in $\Wspace$:
\[
\raisebox{-3ex}{\scalebox{0.25}{\includegraphics{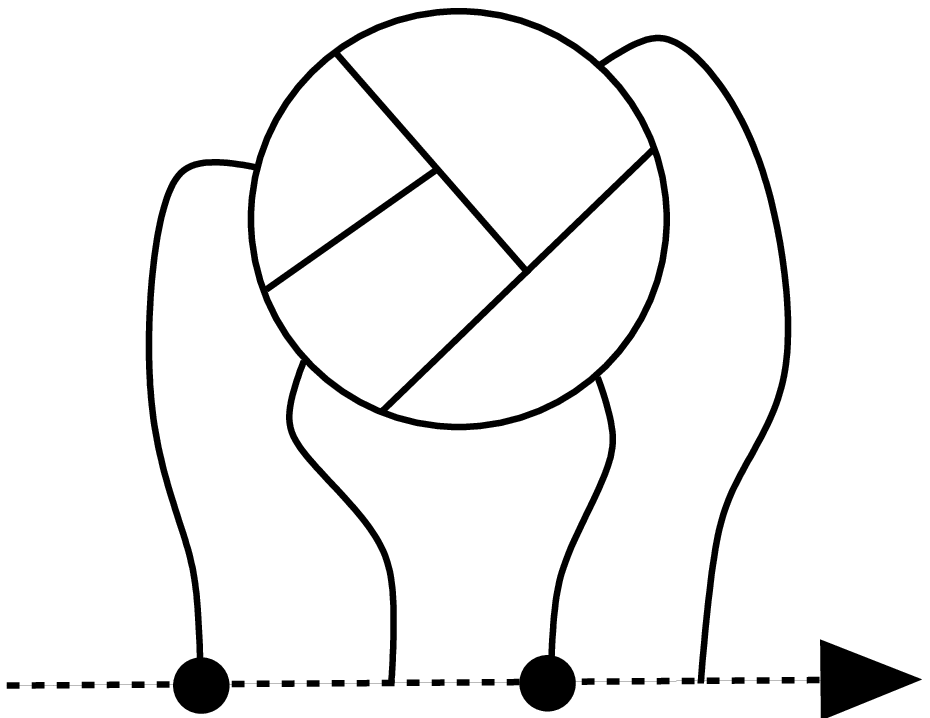}}} \ \
\ =\ \ \
\raisebox{-3ex}{\scalebox{0.25}{\includegraphics{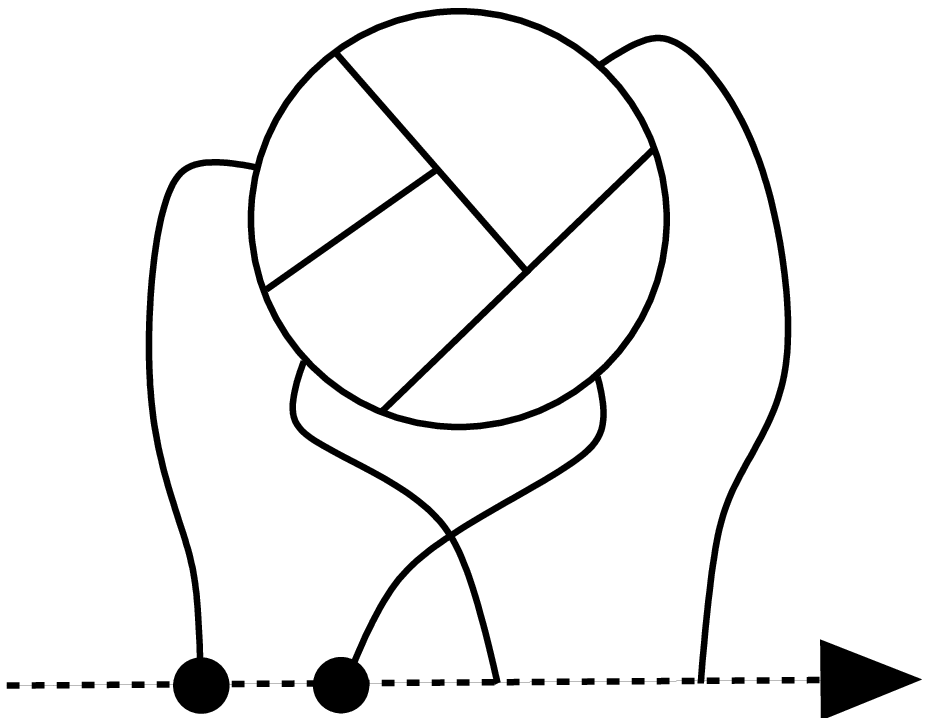}}} \ \
\ =\ \ \ -\
\raisebox{-3ex}{\scalebox{0.25}{\includegraphics{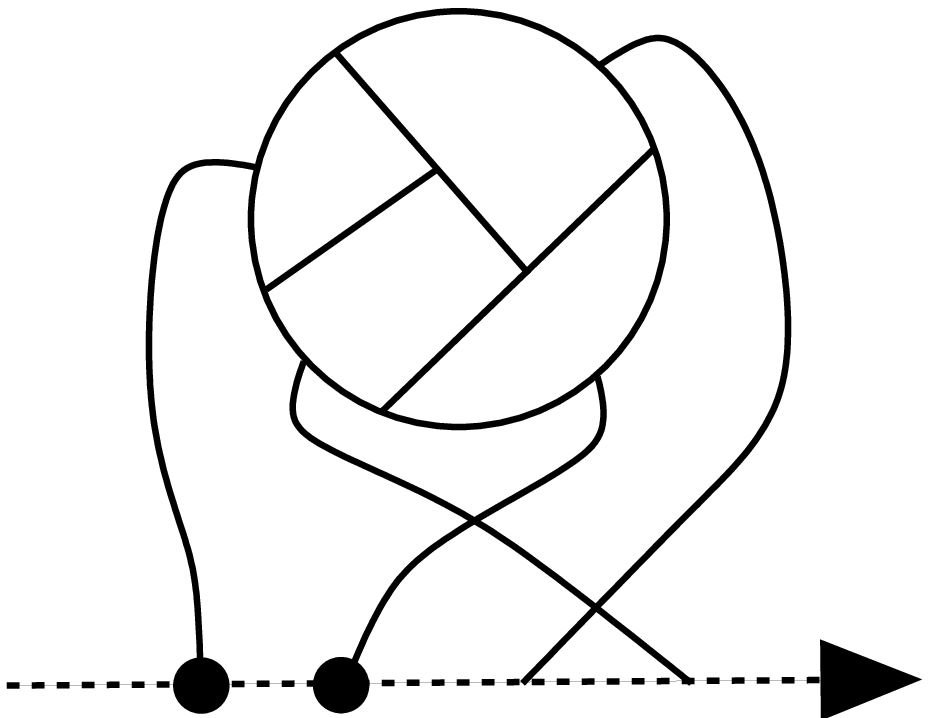}}}\ \
\ \ \ \mbox{in\ \ ${\mathcal{W}}$.}
\]\vspace{0.1cm}

Observe the arrow-head with which symmetric Weil diagrams are drawn.

\subsection{The Hair-splitting map, $\Upsilon: \Bspace\rightarrow\Wspace$.}
\label{hsd}

Now we'll recall the map $\Upsilon$ which embeds $\Bspace$ into the space $\Wspace$. On some symmetric Jacobi diagram $v$, the map is just to choose an ordering of the degree 1 vertices of $v$ (sometimes called the ``hair" of $v$), and then to replace each degree 1 vertex according to the rule:
\[
\begin{array}{ccccl}
\Upsilon :\ \ \ \ \
\raisebox{-2.6ex}{\scalebox{0.28}{\includegraphics{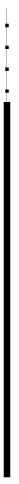}}}\ \ &
\mapsto &
\raisebox{-3ex}{\scalebox{0.27}{\includegraphics{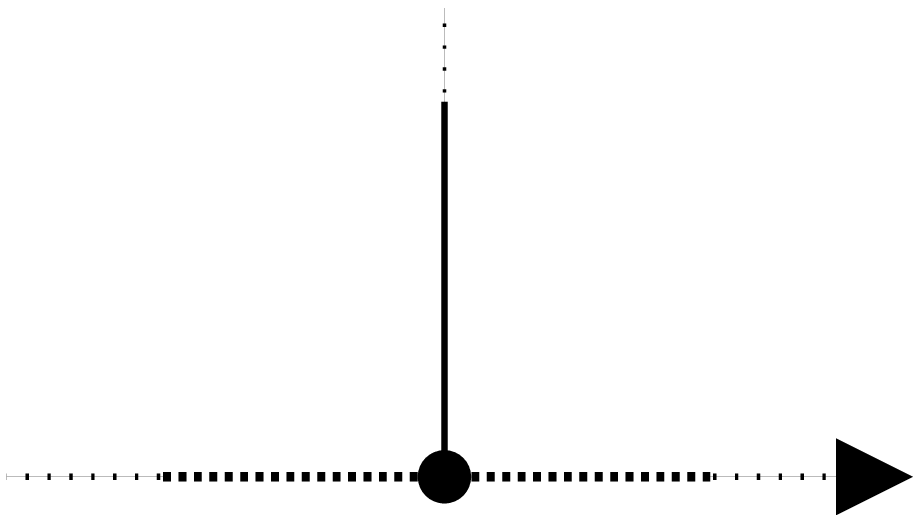}}}
-\frac{1}{2}
\raisebox{-3ex}{\scalebox{0.27}{\includegraphics{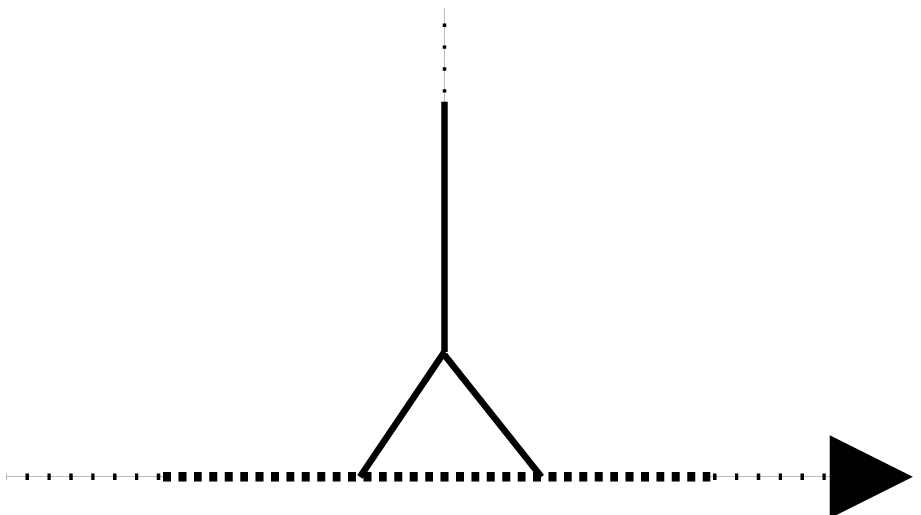}}}\ .
\end{array}
\]
So, for example:
\begin{eqnarray*}
\Upsilon\left(\,\raisebox{-3.5ex}{\scalebox{0.22}{\includegraphics{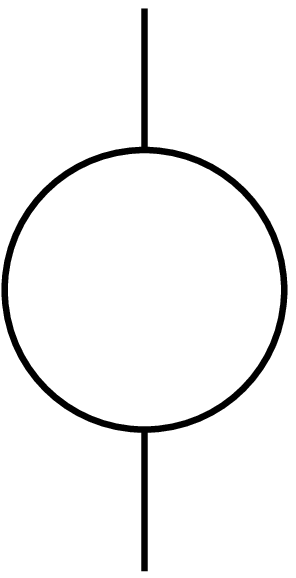}}}\,\right)
& = & \raisebox{-3ex}{\scalebox{0.24}{\includegraphics{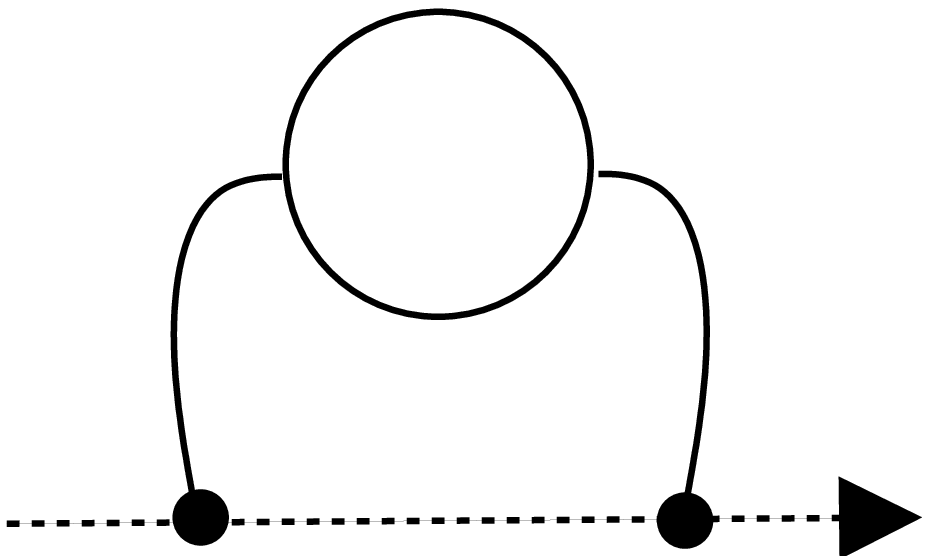}}}
-\frac{1}{2}
\raisebox{-3ex}{\scalebox{0.24}{\includegraphics{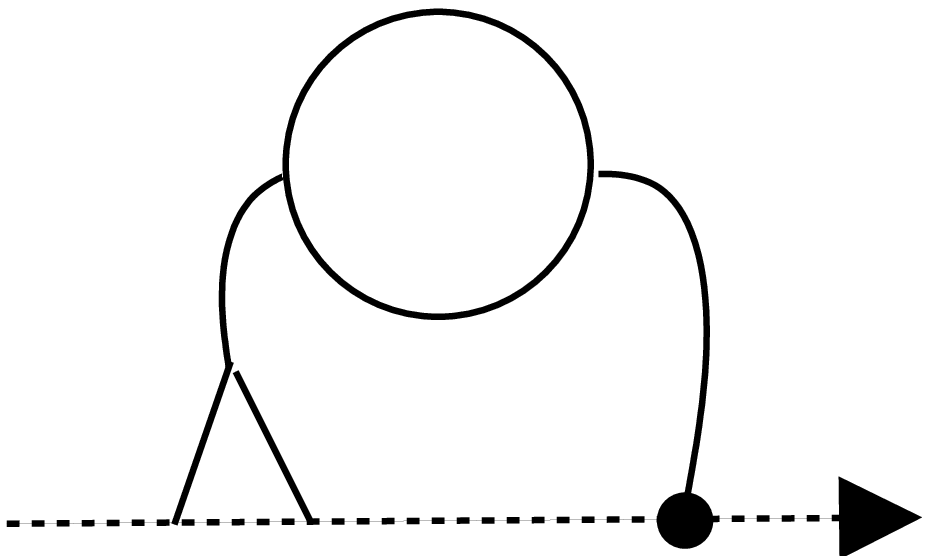}}} \\
& &  -\frac{1}{2}
\raisebox{-3ex}{\scalebox{0.24}{\includegraphics{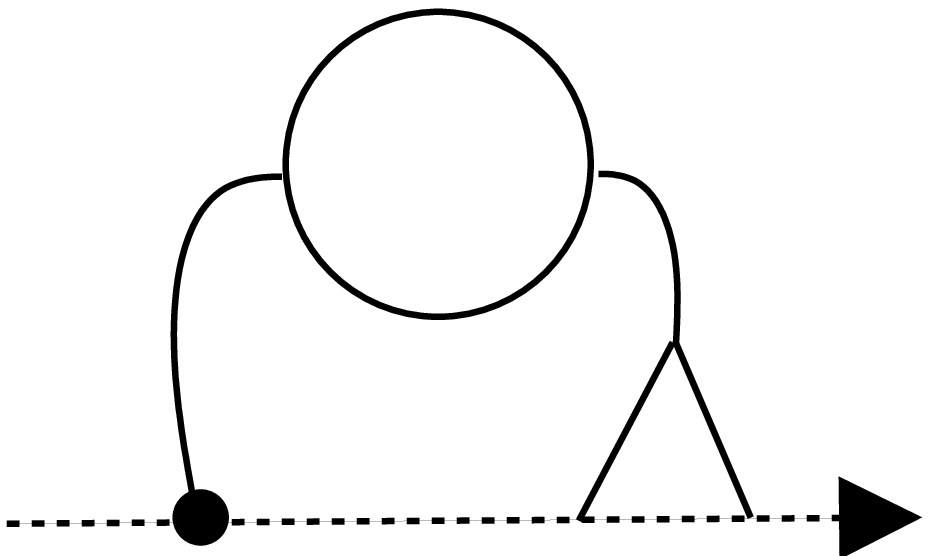}}} \
+\frac{1}{4}
\raisebox{-3ex}{\scalebox{0.24}{\includegraphics{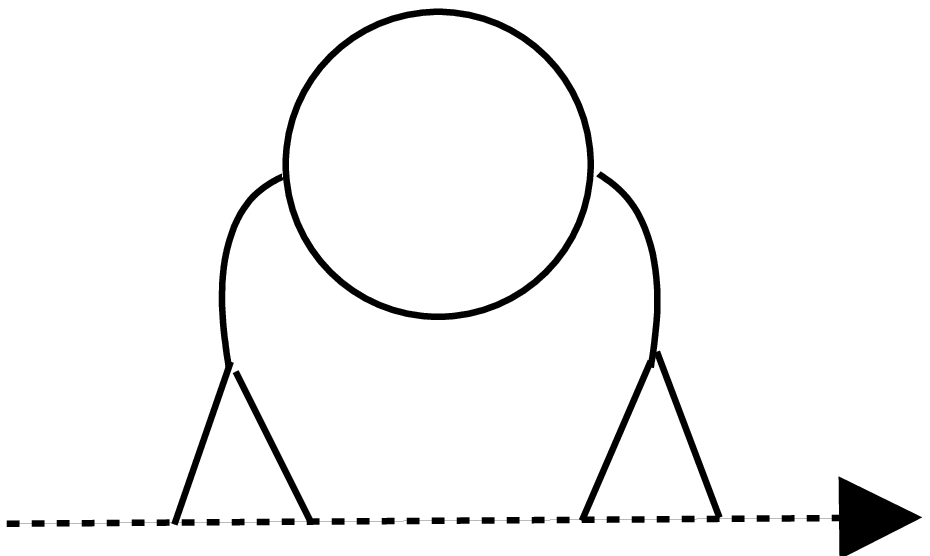}}}\ \in\
\Wspace\,.
\end{eqnarray*}\vspace{0.1cm}

We can think of this map intuitively as ``splitting hair in all possible ways", or ``gluing in forks in all possible ways".

\subsection{The graded averaging map, $\chi_\Wspace:\Wspace \rightarrow \widetilde{\Wspace}$.}

The space $\widetilde{\Wspace}$ is defined in exactly the same way as the space $\Wspace$, but without introducing the leg transposition relations. The diagrams which
generate this space will be called ``non-commutative Weil diagrams" in this paper. So
the relationship between $\Wspace$ and $\widetilde{\Wspace}$ is analogous to the relationship between the  symmetric algebra and tensor algebra on some vector space $V$.

We can embed $\Wspace$ into $\widetilde{\Wspace}$ by means of the graded averaging map $\chi_\Wspace:\Wspace \rightarrow \widetilde{\Wspace}$. This is the linear extension of the map
which takes a symmetric Weil diagram to the average of all possible rearrangements of the legs of the diagram, accompanied by the sign that arises when that permutation is performed in $\Wspace$.
For example:
\begin{eqnarray*}
\chi_\Wspace\left(
\raisebox{-3ex}{\scalebox{0.22}{\includegraphics{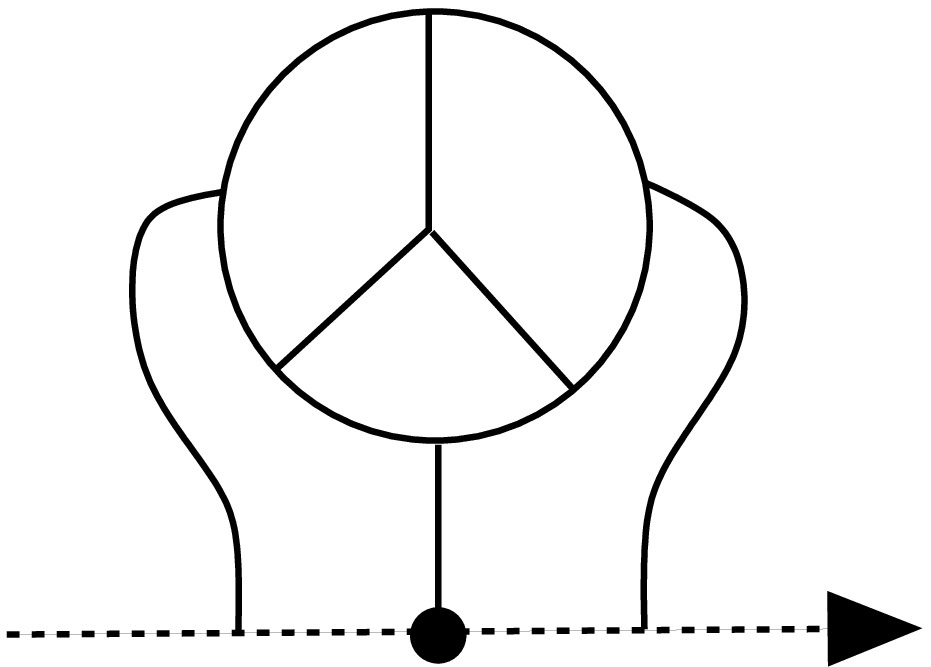}}}
\right) & = & \frac{1}{3!}\left(
\raisebox{-3ex}{\scalebox{0.22}{\includegraphics{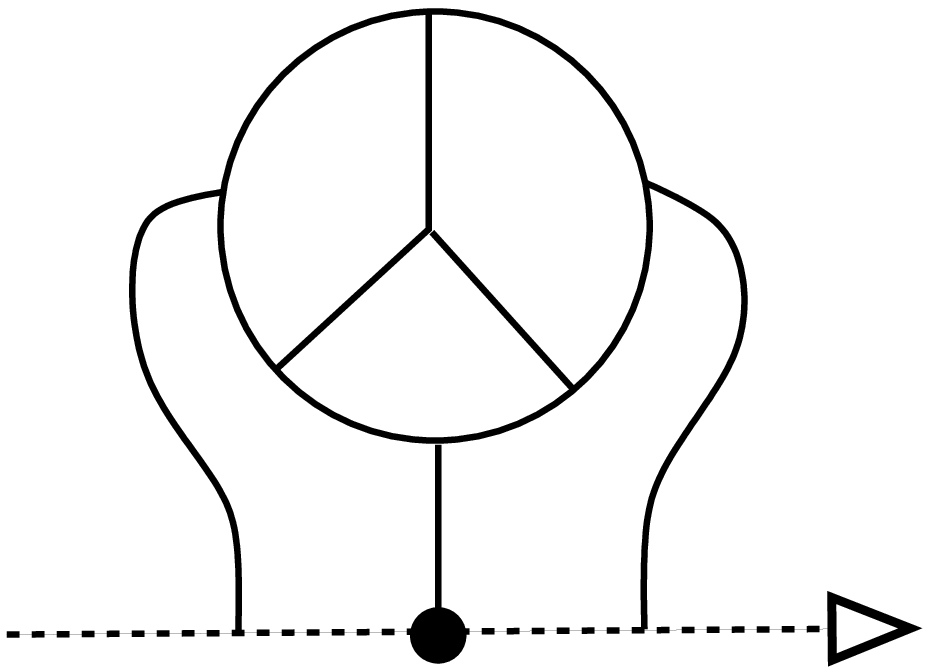}}}
+
\raisebox{-3ex}{\scalebox{0.22}{\includegraphics{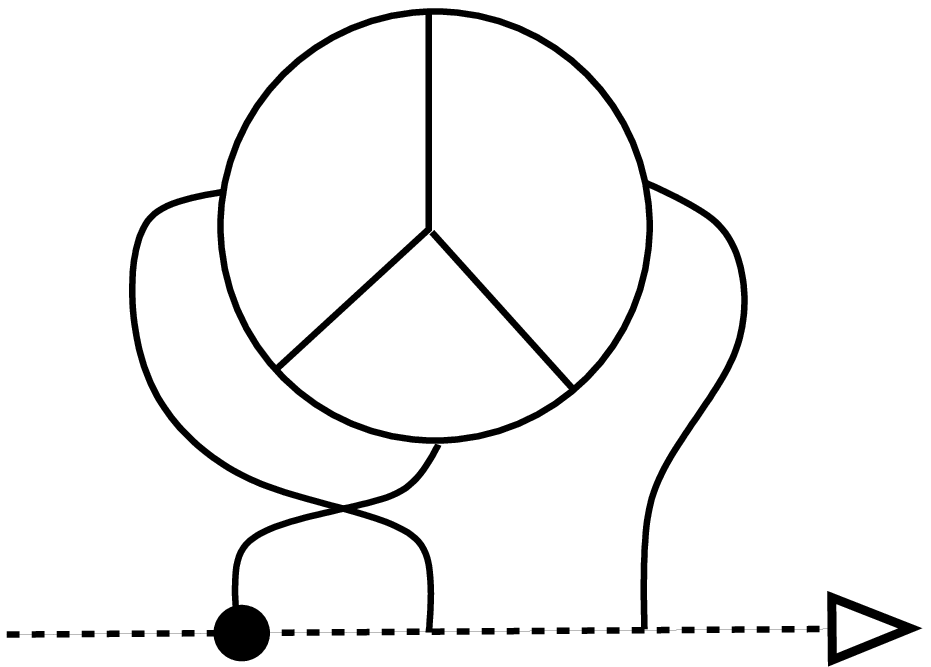}}}
-
\raisebox{-3ex}{\scalebox{0.22}{\includegraphics{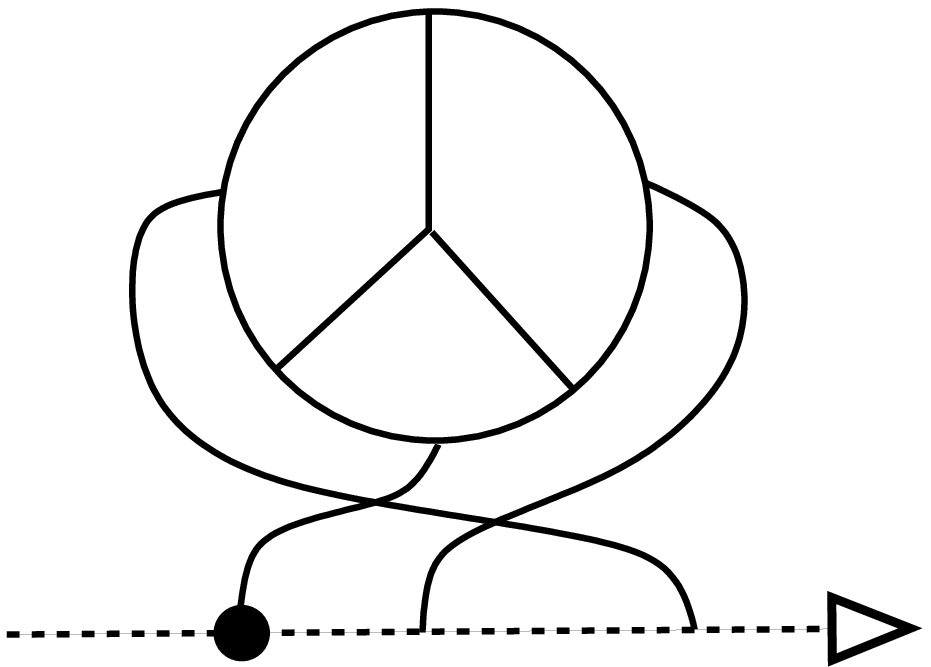}}}
\right.
\\ & & \left.
\ \ \ \,-
\raisebox{-3ex}{\scalebox{0.22}{\includegraphics{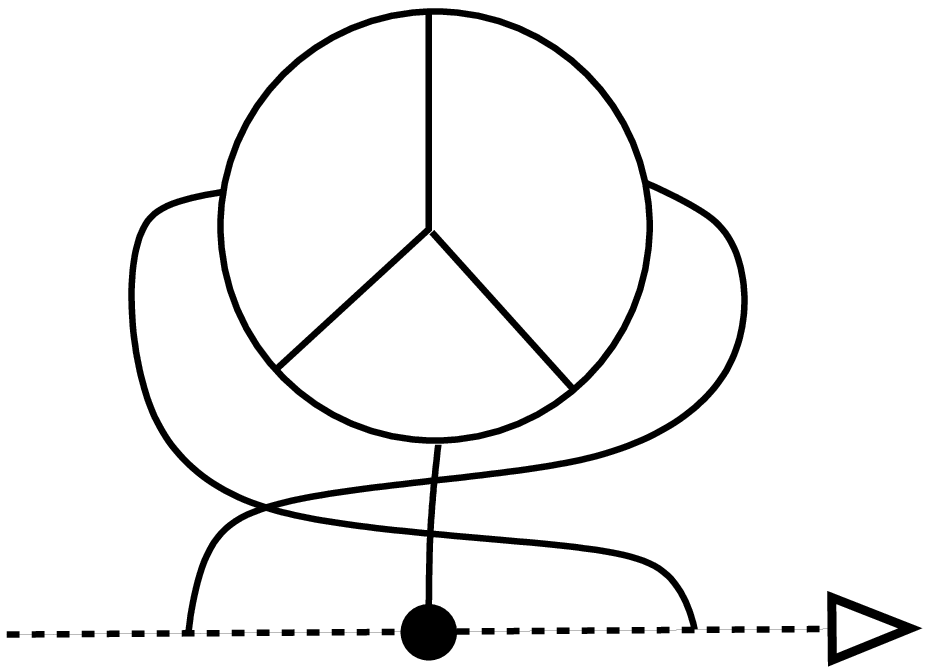}}}
-
\raisebox{-3ex}{\scalebox{0.22}{\includegraphics{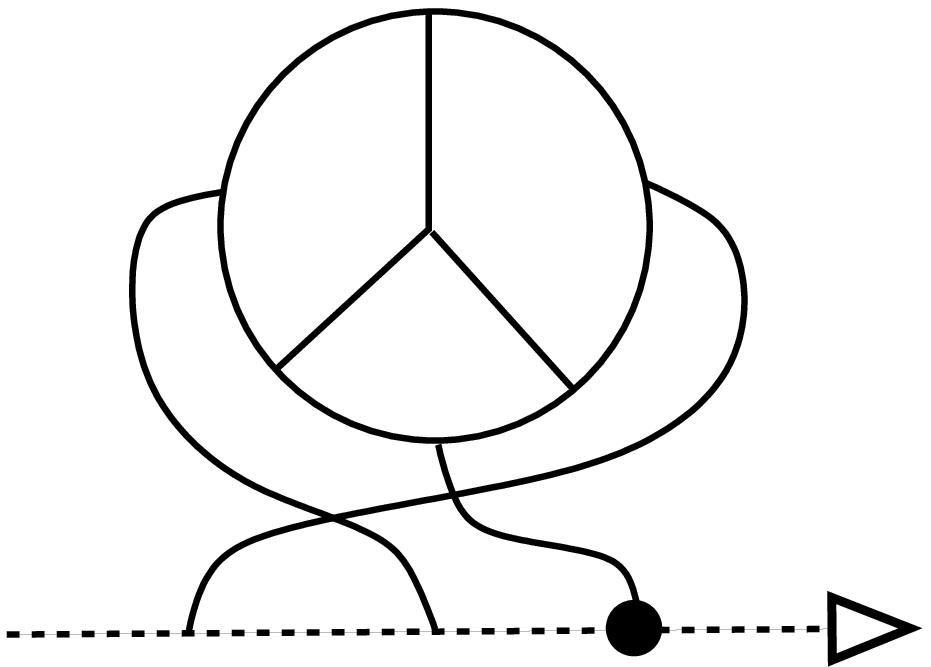}}}
+
\raisebox{-3ex}{\scalebox{0.22}{\includegraphics{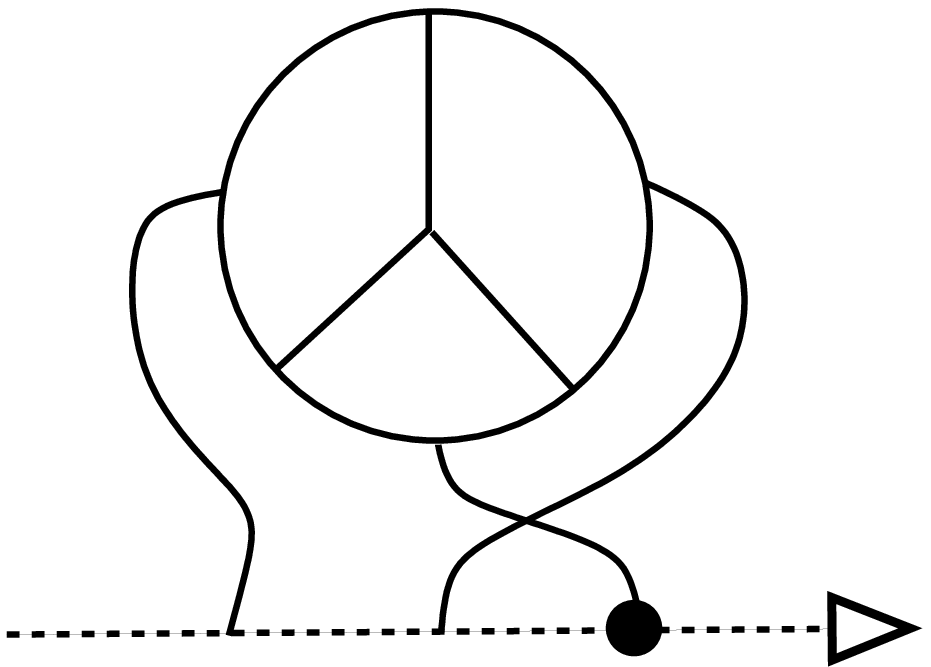}}}\right).
\end{eqnarray*}
The reader can check that this map respects leg transposition relations.

\subsection{The space $\widehat{\Wspace}$, and the map $\pi: \widetilde{\Wspace} \rightarrow \widehat{W}$.}

Non-commutative Weil diagrams have no relations that relate different orderings of their legs. But the space we are heading towards -- $\Aspace$ -- has STU
relations, so we had better introduce them. As discussed in \cite{K} (basically following results of Alekseev and Meinrenken), when we introduce STU relations
amongst the leg-grade 2 legs, there are some other classes of relations that we must introduce at the same time, so as to retain the algebraic structure of an
$\iota$-complex.

The complete set of relations that we introduce when we introduce STU is as follows:
\[
\begin{array}{ccccc}
\raisebox{-2ex}{\scalebox{0.22}{\includegraphics{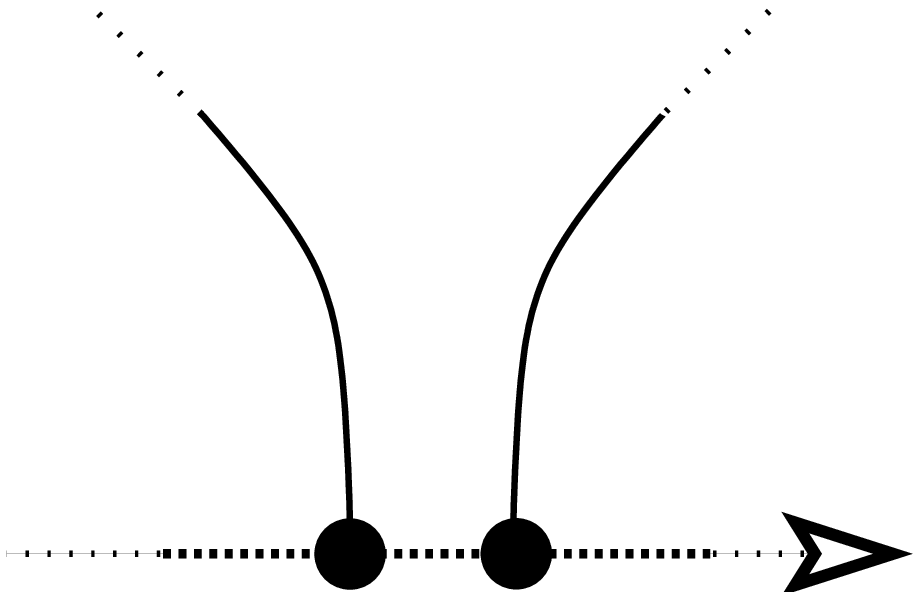}}} & -
& \raisebox{-2ex}{\scalebox{0.22}{\includegraphics{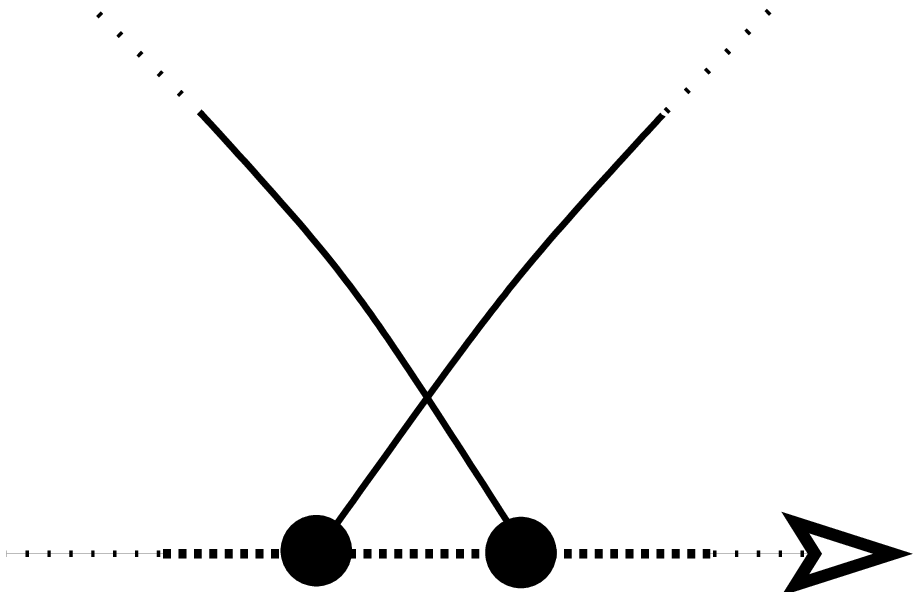}}}  &
= & \raisebox{-2ex}{\scalebox{0.22}{\includegraphics{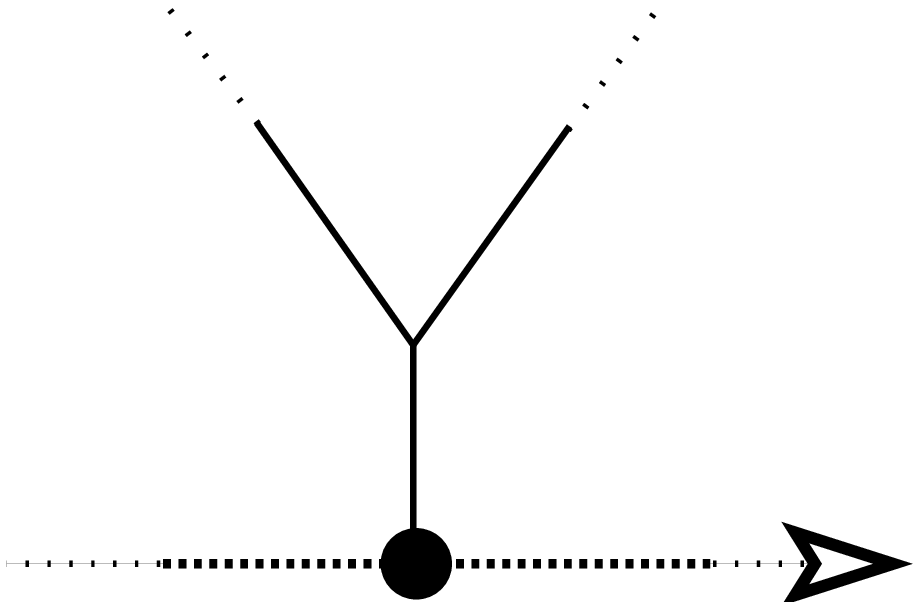}}}\
\ ,
\\[0.55cm]
\raisebox{-2ex}{\scalebox{0.22}{\includegraphics{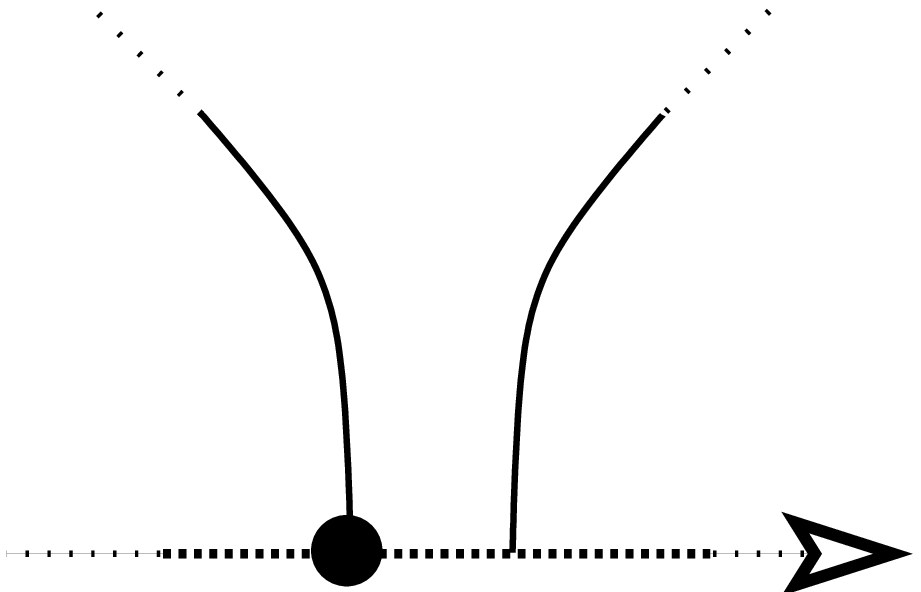}}} & -
& \raisebox{-2ex}{\scalebox{0.22}{\includegraphics{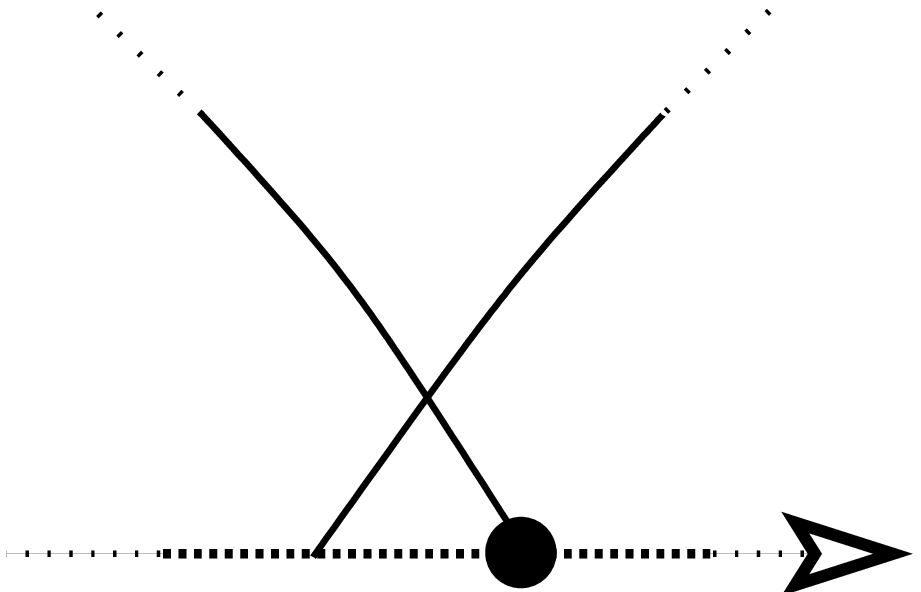}}} &
= &
\raisebox{-2ex}{\scalebox{0.22}{\includegraphics{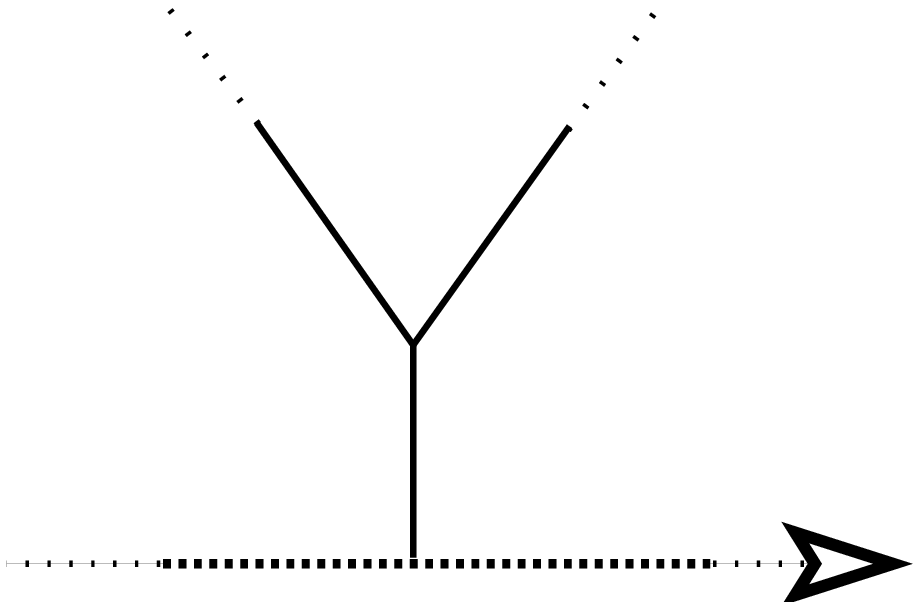}}}\ \
,
\\[0.55cm]
\raisebox{-2ex}{\scalebox{0.22}{\includegraphics{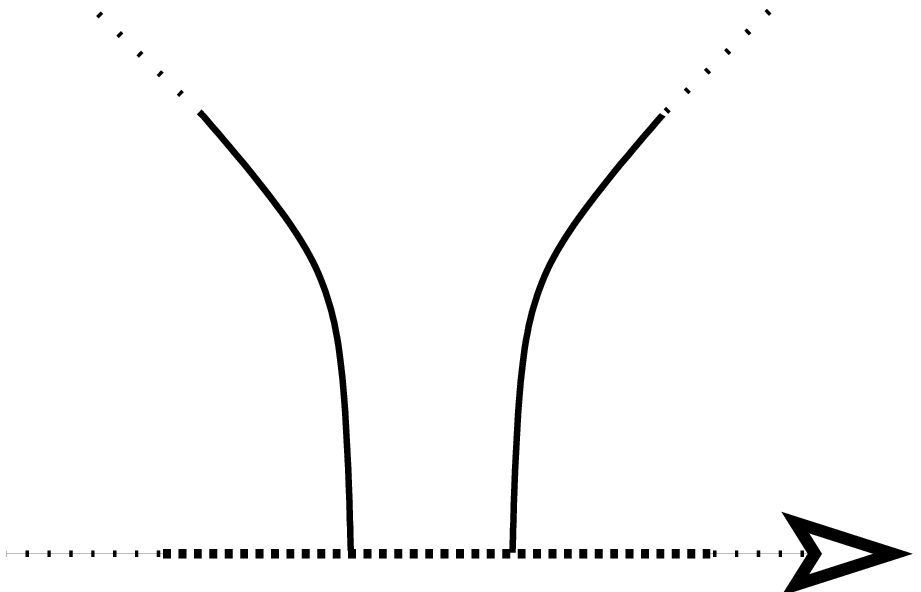}}} &
+ &
\raisebox{-2ex}{\scalebox{0.22}{\includegraphics{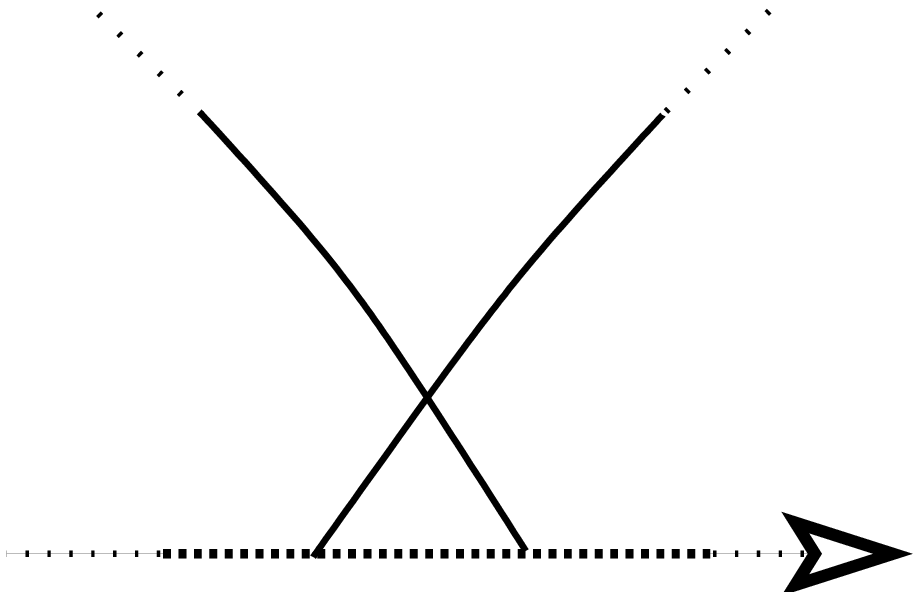}}} &
= &
\raisebox{-2ex}{\scalebox{0.22}{\includegraphics{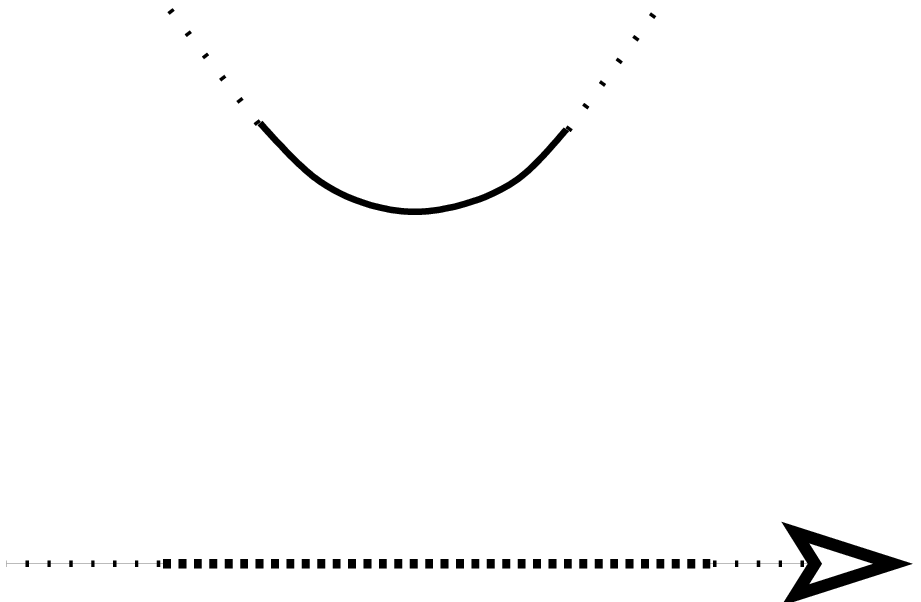}}}\ \
.
\end{array}
\]\vspace{0.1cm}

Observe that the third class presented above is some sort of formal analogue of the defining relation of a Clifford algebra.

The map $\pi: \widetilde{\Wspace} \rightarrow \widehat{W}$ is just to introduce these relations, with no other effect on a diagram.

\subsection{Curvature legs and the map $\basebulltoF : \widehat{\mathcal{W}} \rightarrow   \widehat{\mathcal{W}}_{\mathrm{F}}$.}
Instead of the usual leg-grade 2 legs that have appeared in the discussion up to this point (the legs drawn with a fat dot), it is possible to work
with a different choice of leg-grade 2 leg, which we'll call {\it curvature legs} in this work.
The relationship between the two choices can be expressed by the equation (with arrow-head appropriate to the space):
\[
\begin{array}{rccl}
\raisebox{-3ex}{\scalebox{0.25}{\includegraphics{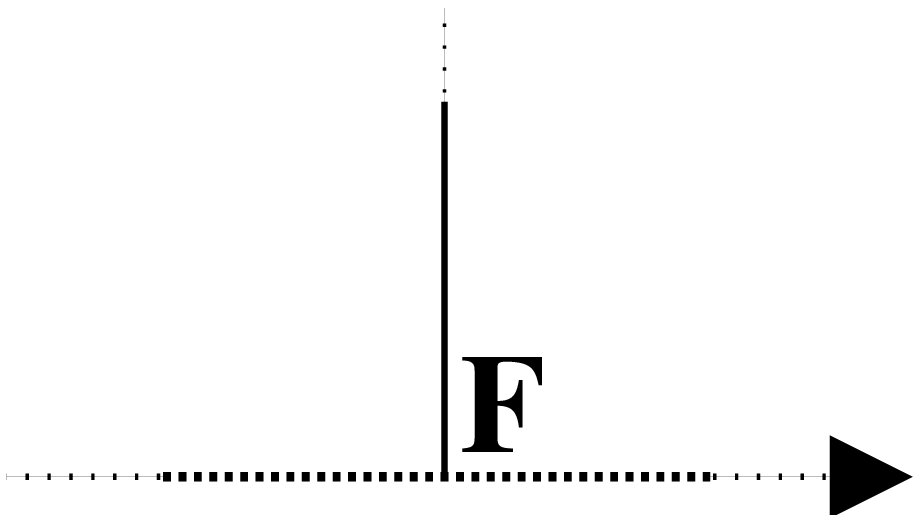}}} &
= &
\raisebox{-3ex}{\scalebox{0.25}{\includegraphics{curvB}}}
-\frac{1}{2}
\raisebox{-3ex}{\scalebox{0.25}{\includegraphics{curvC}}}\ .
\end{array}
\]
For example, the Hair-splitting map $\Upsilon$ discussed above in Section \ref{hsd} can alternatively be defined as follows: Choose some ordering for the degree 1 vertices
of the diagram, then make every leg a curvature leg.

In the algebraic theory, this is just a different choice of generators within a common algebra. In the current work, from the combinatorial point of view, it is clearer
to view diagrams that are based on curvature legs as generators of a different, though isomorphic, vector space, and the above equality should only be viewed heuristically.
Curvature legs are introduced into the theory to simplify the map $\iota$, though at the expense of a more complicated differential.

If we base the space $\widehat{\Wspace}$ on curvature legs, instead of the usual leg-grade 2 legs, we are led to the space $\widehat{\Wspace}_{\mathrm{F}}$, which has the following relations
\[
\begin{array}{ccccc}
\raisebox{-2ex}{\scalebox{0.23}{\includegraphics{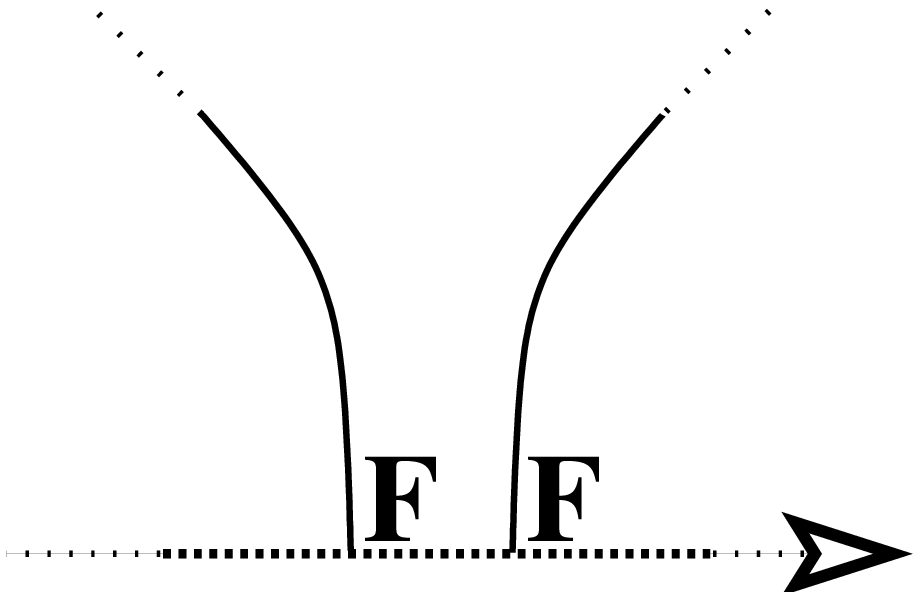}}} & - &
\raisebox{-2ex}{\scalebox{0.23}{\includegraphics{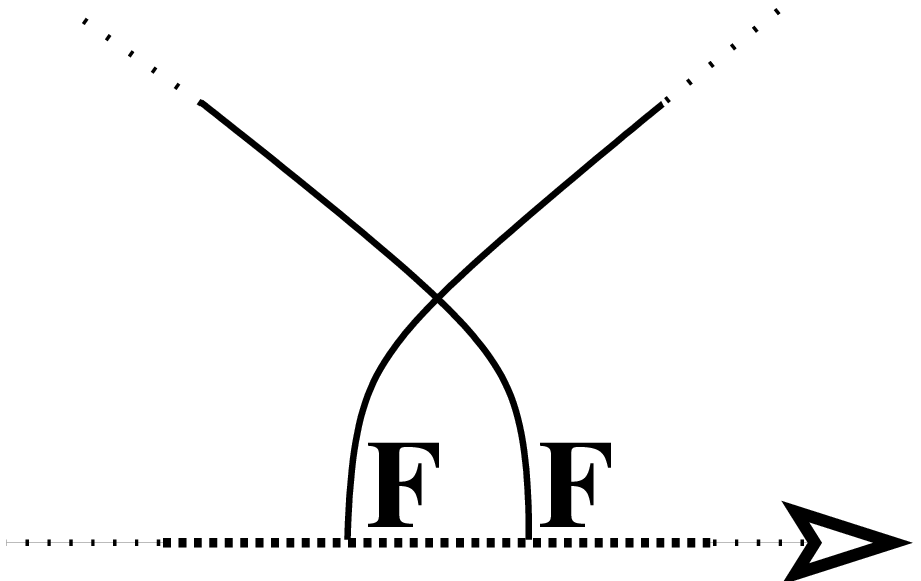}}} & = &
\raisebox{-2ex}{\scalebox{0.23}{\includegraphics{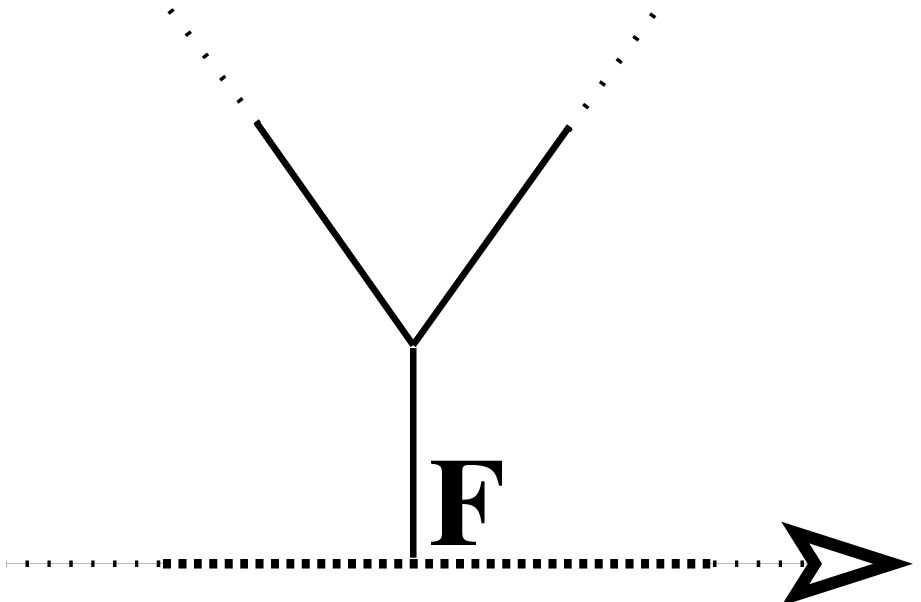}}}\ \ ,
\\[0.55cm]
\raisebox{-2ex}{\scalebox{0.23}{\includegraphics{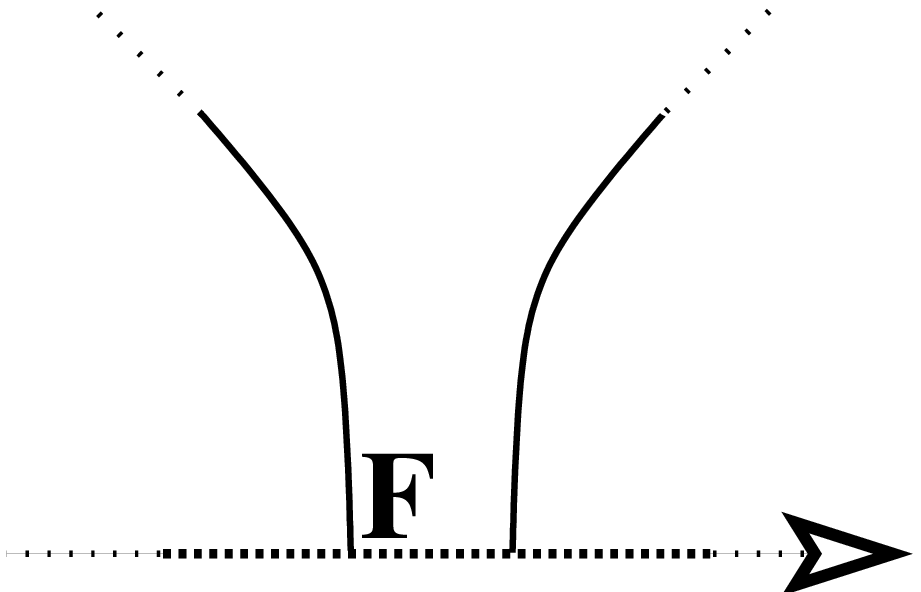}}} &  -
& \raisebox{-2ex}{\scalebox{0.23}{\includegraphics{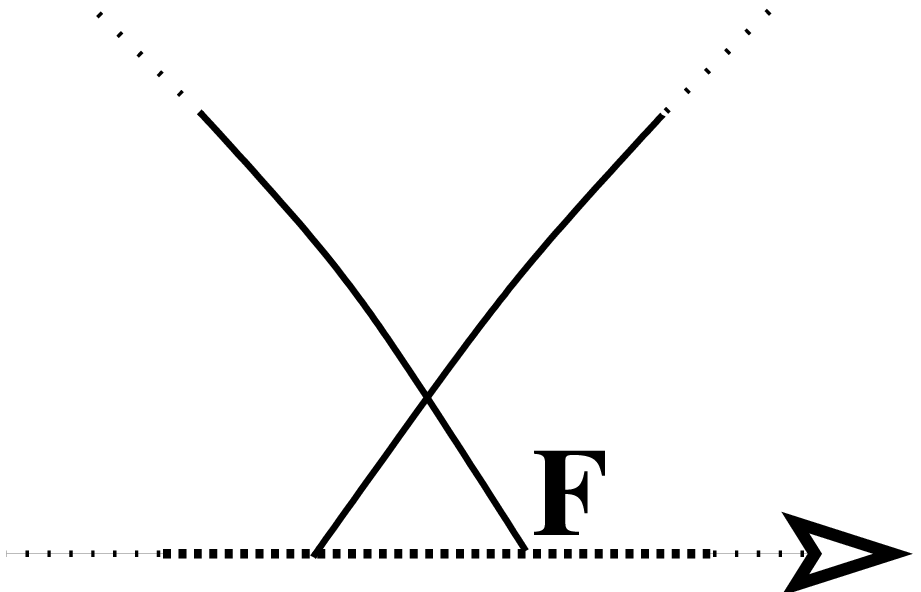}}} & =
& 0\ \ ,
\\[0.55cm]
\raisebox{-2ex}{\scalebox{0.23}{\includegraphics{whatrelnAPP}}} &
+ &
\raisebox{-2ex}{\scalebox{0.23}{\includegraphics{whatrelnBPP}}} &
= &
\raisebox{-2ex}{\scalebox{0.23}{\includegraphics{whatrelnCPP}}}\ \
.
\end{array}
\]
\vspace{0.1cm}

The ``change of basis" map, $\basebulltoF : \widehat{\mathcal{W}} \rightarrow   \widehat{\mathcal{W}}_{\mathrm{F}}$, is just to replace
every leg-grade 2 leg with a curvature leg, via the operation:
\[
\begin{array}{rccl}
\raisebox{-3ex}{\scalebox{0.25}{\includegraphics{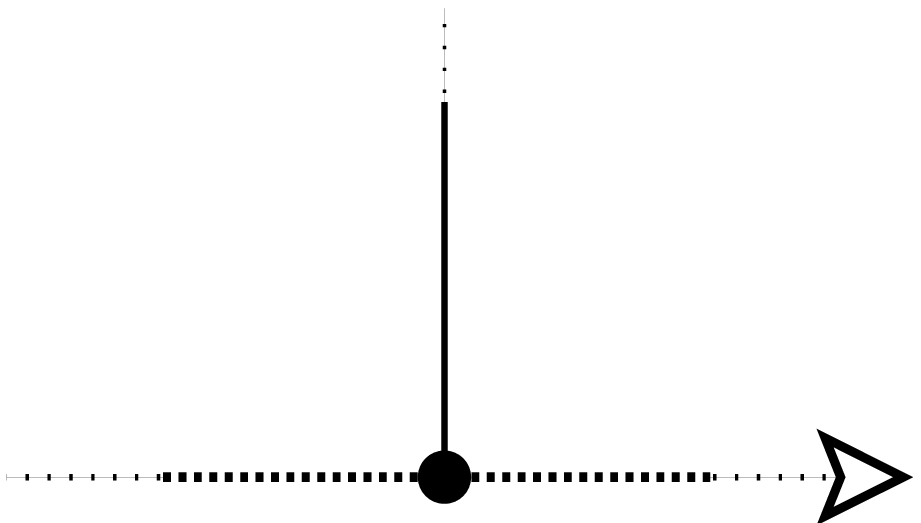}}} &
\mapsto &
\raisebox{-3ex}{\scalebox{0.25}{\includegraphics{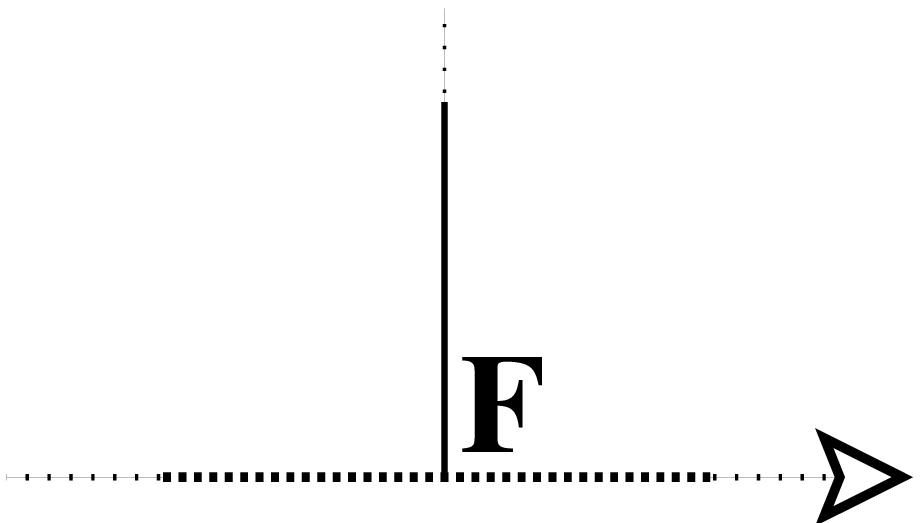}}}
+\frac{1}{2}
\raisebox{-3ex}{\scalebox{0.25}{\includegraphics{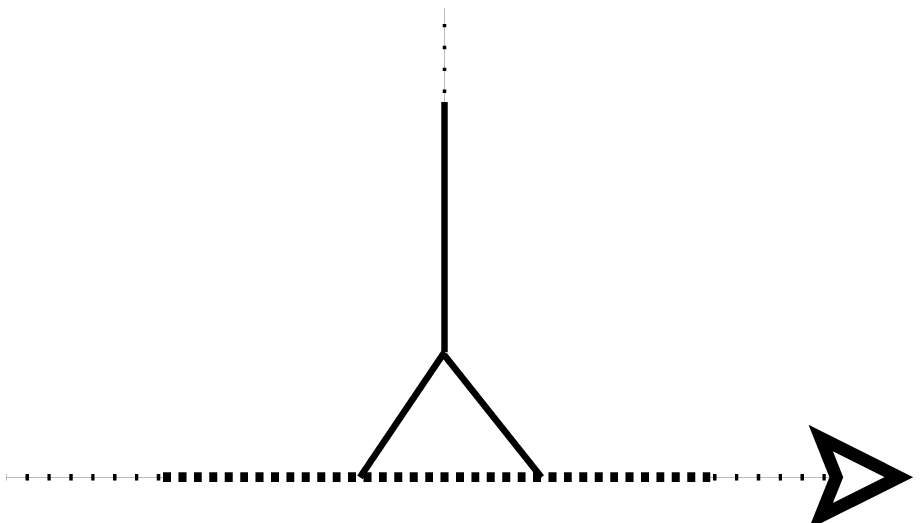}}}\ .
\end{array}
\]\vspace{0.2cm}

\subsection{The space $\widehat{\mathcal{W}}_\wedge$.}
The final space to recall is the space $\widehat{\mathcal{W}}_\wedge$. This space consists of formal finite $\mathbb{Q}$-linear combinations
of diagrams with leg-grade 1 legs and curvature legs (i.e. exactly the same diagrams as is used by $\widehat{\mathcal{W}}_{\mathrm{F}}$), taken modulo AS, IHX,
and the following three classes of relations:
\[
\begin{array}{ccccc}
\raisebox{-2ex}{\scalebox{0.23}{\includegraphics{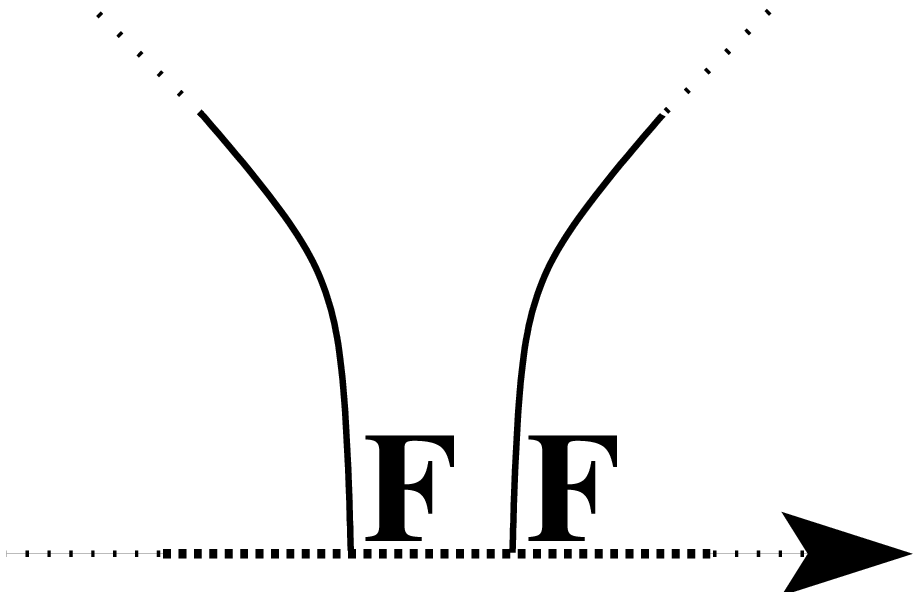}}} & -
& \raisebox{-2ex}{\scalebox{0.23}{\includegraphics{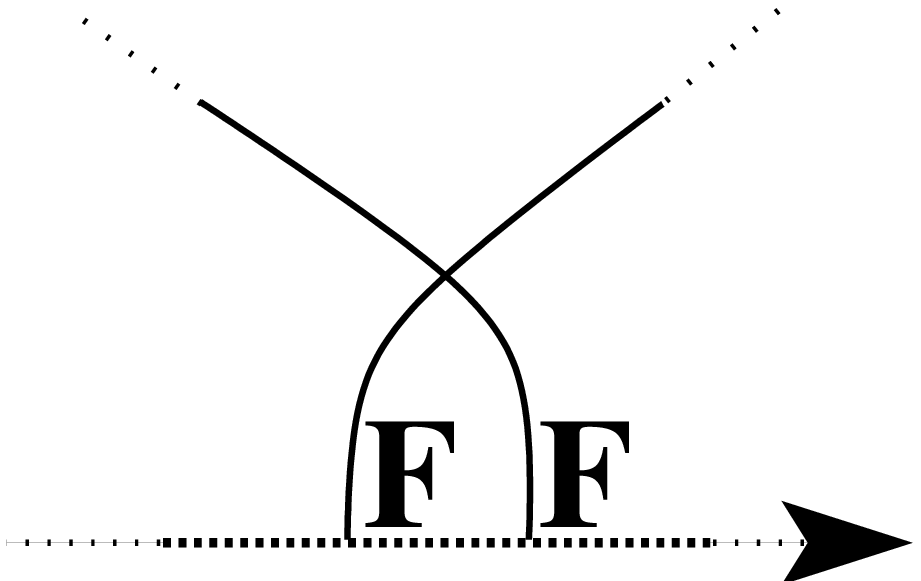}}} &
= &
\raisebox{-2ex}{\scalebox{0.23}{\includegraphics{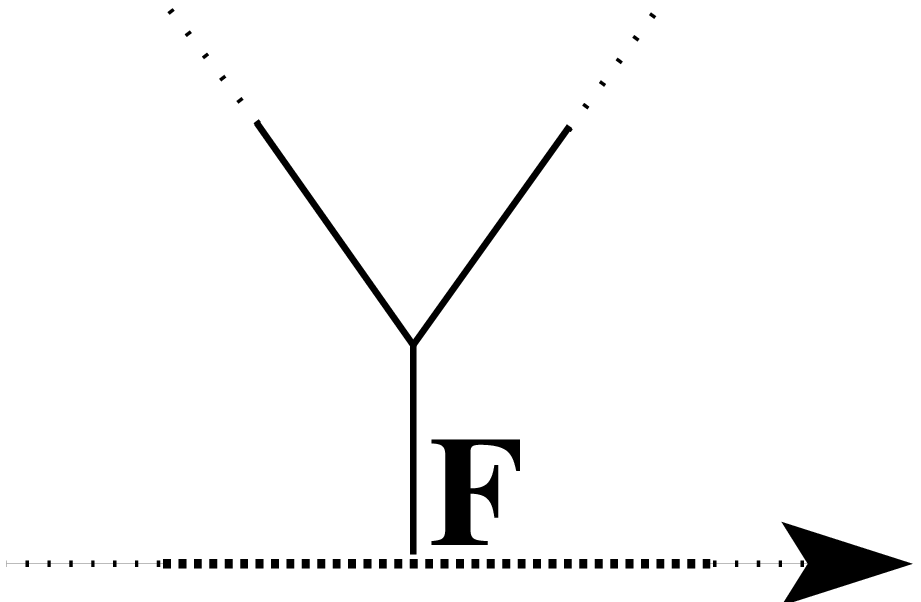}}}\ \
,
\\[0.5cm]
\raisebox{-2ex}{\scalebox{0.23}{\includegraphics{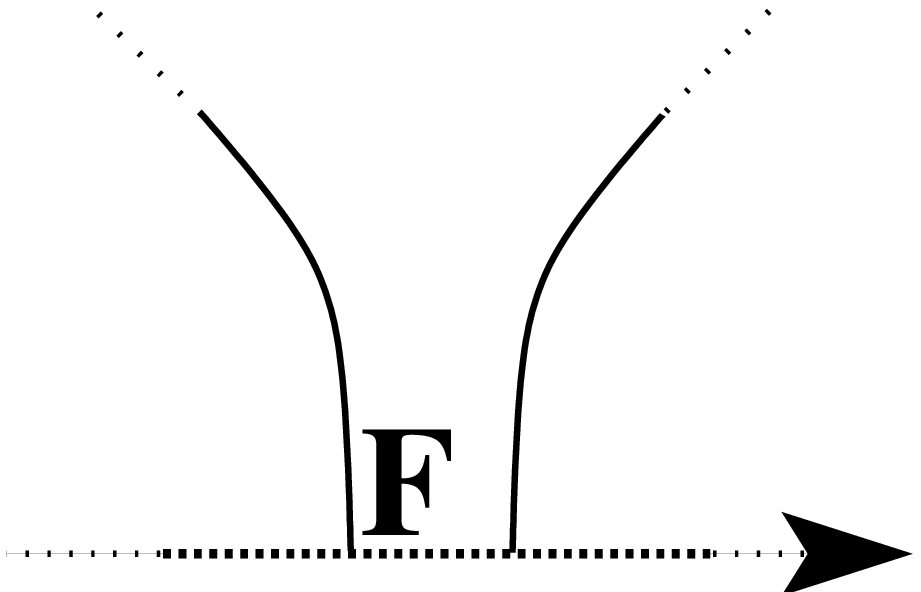}}} &
- &
\raisebox{-2ex}{\scalebox{0.23}{\includegraphics{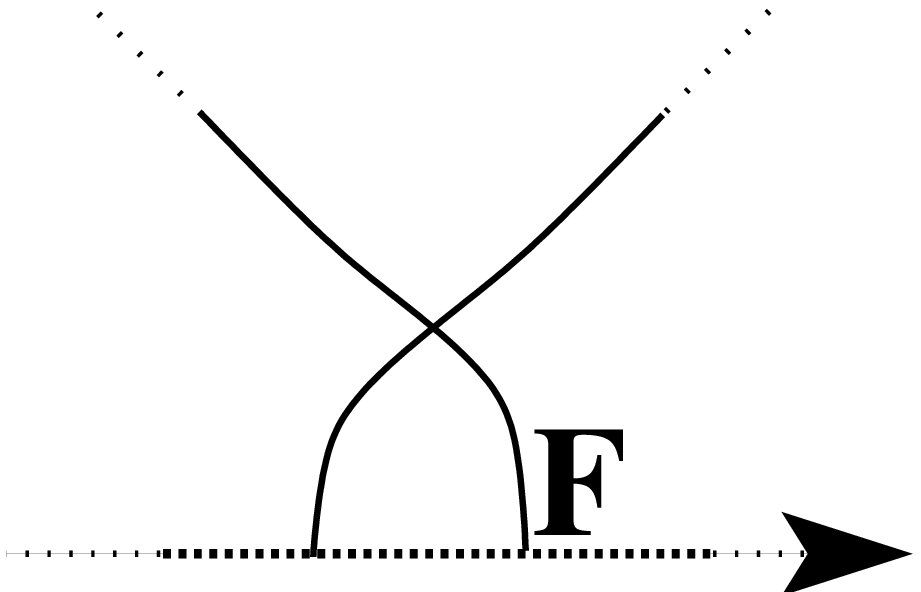}}} &
= & 0\ \ \ ,
\\[0.5cm]
\raisebox{-2ex}{\scalebox{0.23}{\includegraphics{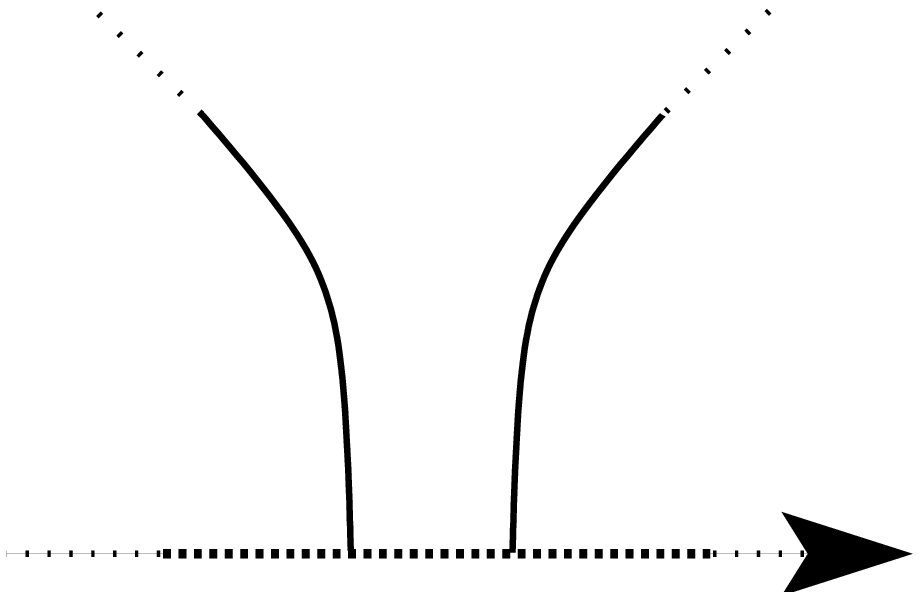}}} &
+ &
\raisebox{-2ex}{\scalebox{0.23}{\includegraphics{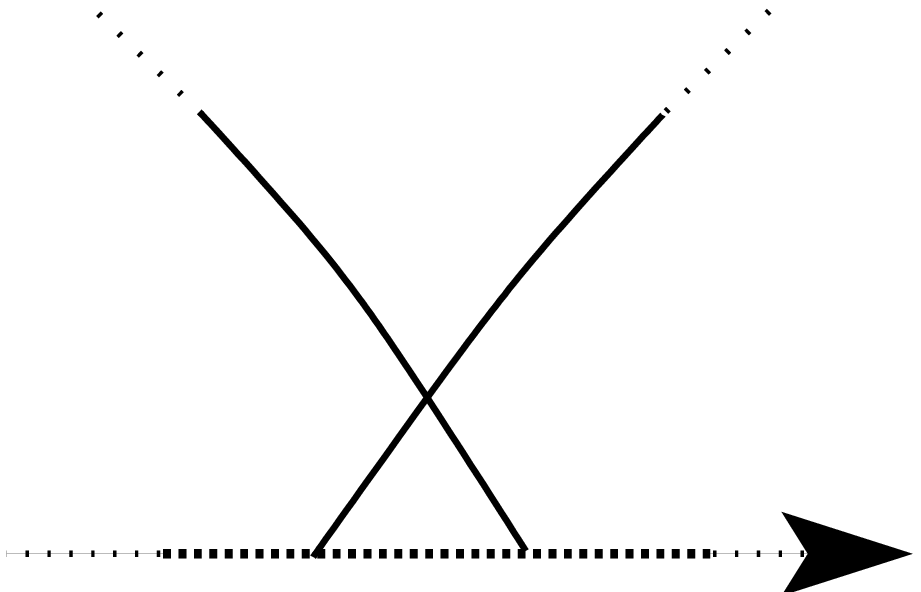}}} &
= & 0\ \ \ .
\end{array}
\]\vspace{0.1cm}

Observe that in this space the leg-grade 1 legs can be moved about freely, up to sign. In particular, this space is graded by the number of leg-grade 1 legs that a diagram has.

This space can be thought of as $\widehat{\mathcal{W}}_{\mathrm{F}}$ viewed with respect to generators in which the leg-grade 1 legs have been symmetrized.
To be precise: we have a well-defined map, $\chi_\wedge: \widehat{\mathcal{W}}_\wedge \rightarrow \widehat{\mathcal{W}}_{\mathrm{F}}$, which graded averages the leg-grade 1 legs.
For example:

\begin{eqnarray*}
\chi_\wedge\left(
\raisebox{-2.75ex}{\scalebox{0.28}{\includegraphics{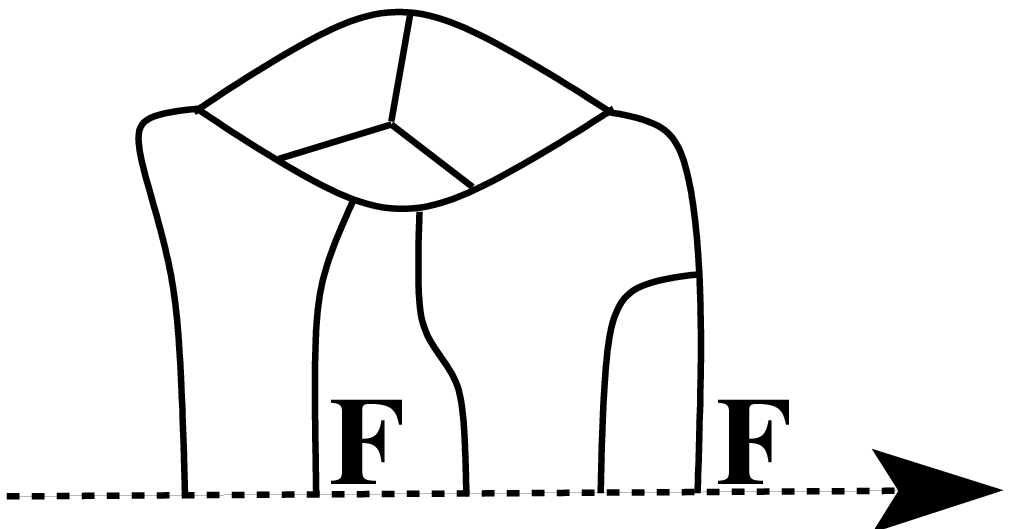}}}
\right) & = &\frac{1}{3!}\left(
\raisebox{-2.75ex}{\scalebox{0.28}{\includegraphics{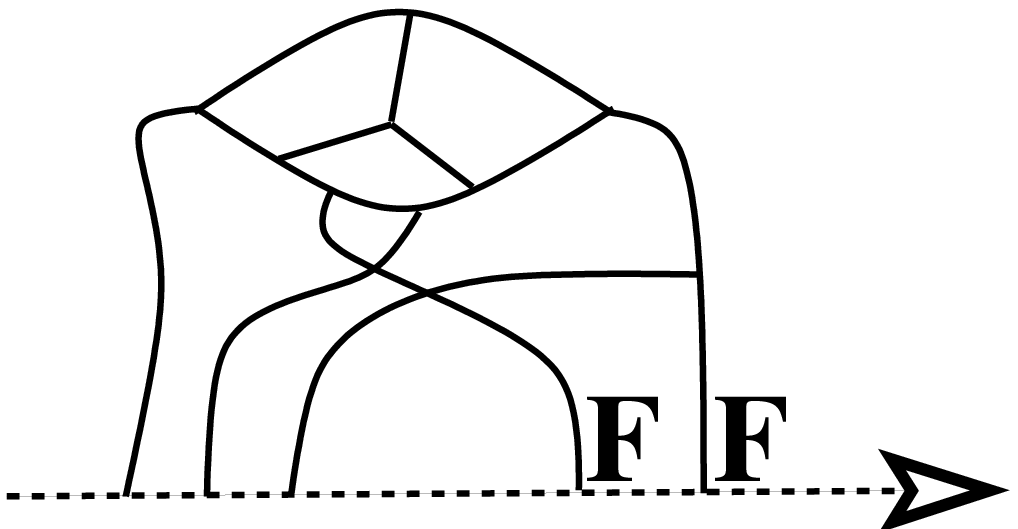}}}
\ -\
\raisebox{-2.75ex}{\scalebox{0.28}{\includegraphics{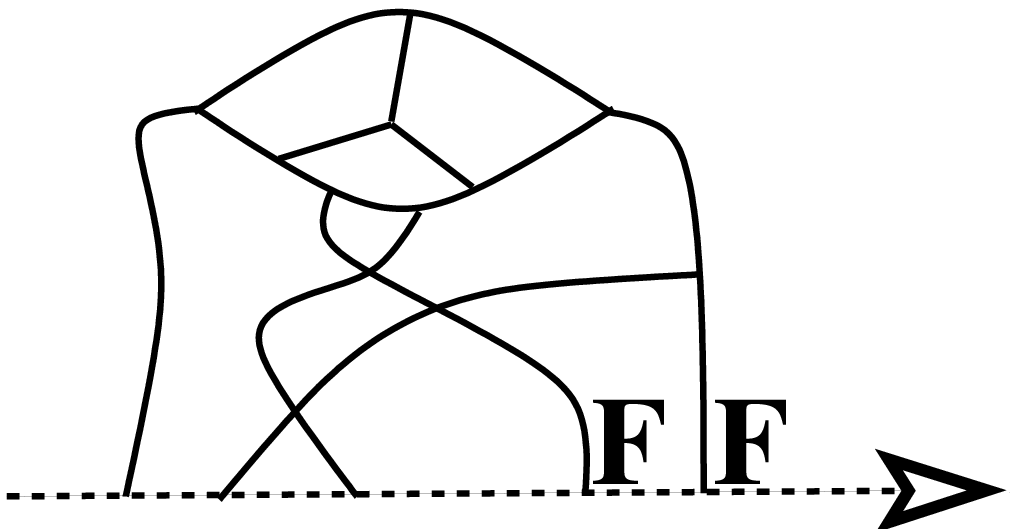}}}\right.
\\
& & \ +\
\raisebox{-2.75ex}{\scalebox{0.28}{\includegraphics{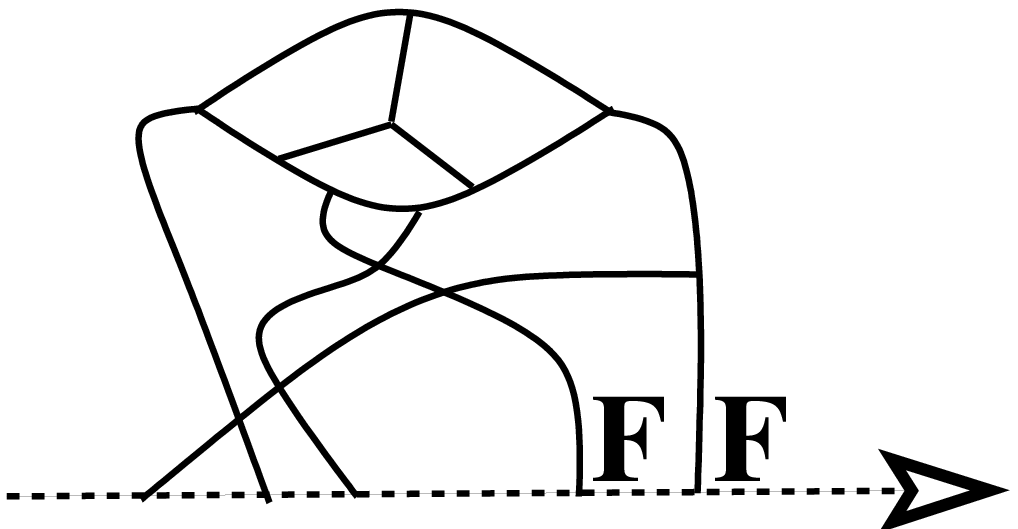}}}
\ -\
\raisebox{-2.75ex}{\scalebox{0.28}{\includegraphics{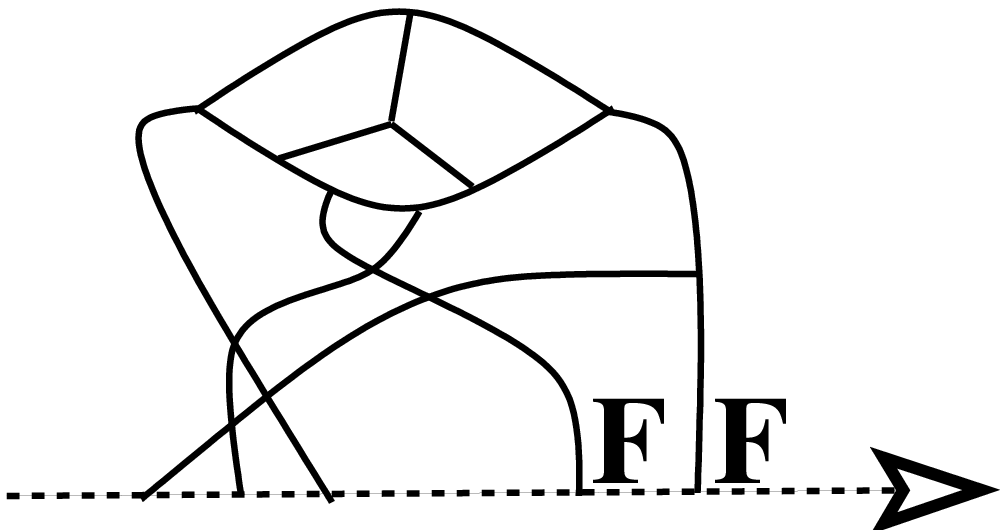}}}
\\[0.4cm]
& &\left. \ +\
\raisebox{-2.75ex}{\scalebox{0.28}{\includegraphics{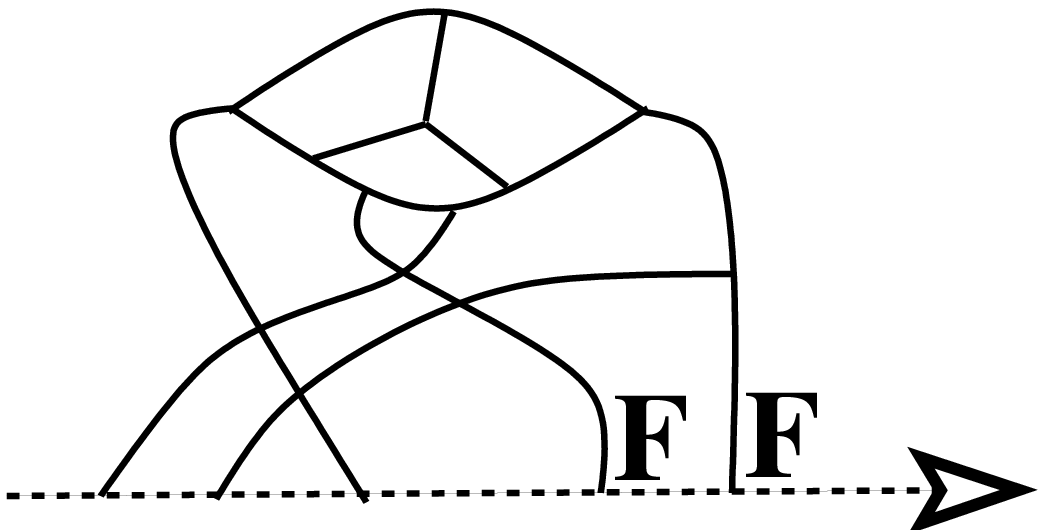}}}
\ -\
\raisebox{-2.75ex}{\scalebox{0.28}{\includegraphics{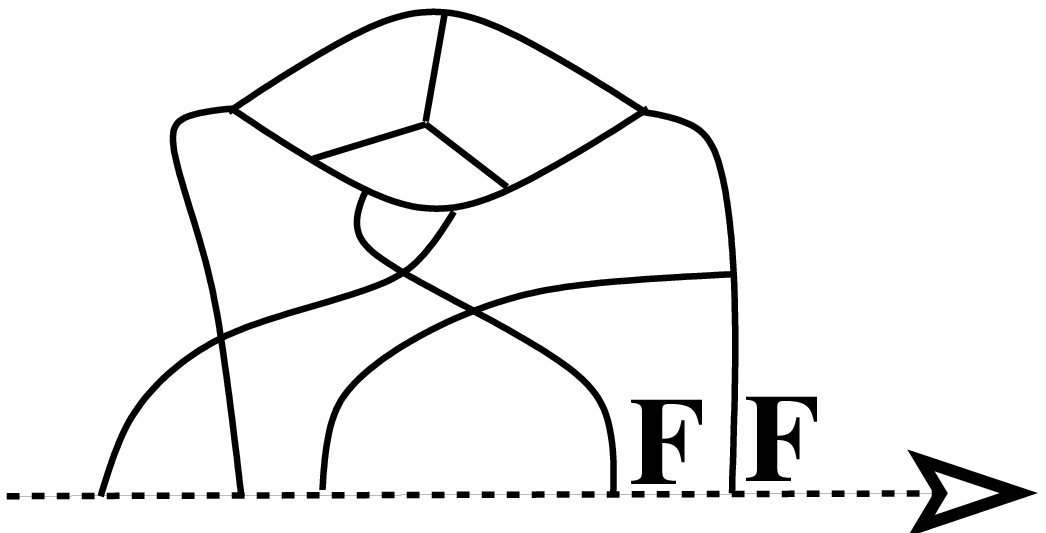}}}\right).
\end{eqnarray*}

Just like the case $\chi_\Bspace$, which is the formal PBW isomorphism considered in \cite{BarNatan}, the map $\chi_\wedge$ is a vector space isomorphism.
However, something is true in this case which is not true for $\chi_\Bspace$: the inverse map has an elementary construction.

\subsection{The map $\lambda: \widehat{\mathcal{W}}_{\mathrm{F}} \rightarrow \widehat{\mathcal{W}}_\wedge$.}
\label{lambdarecall}

Here we'll recall the definition of the map $\lambda$ which inverts $\chi_\wedge$. For detailed proofs that it
is well-defined, and actually inverts $\chi_\wedge$, see \cite{K}. In \cite{K} we described two approaches to $\lambda$, a ``combinatorial" definition
which was useful for proving facts about the construction, and a ``visual" definition which was more useful for doing calculations. We'll recall the second approach here.

The definition can be summarized in the following way: Glue the grade 1 legs to each other in all possible ways, with appropriate coefficients. To be precise,
recall that a {\it pairing} of the grade 1 legs of a diagram $w$ is a collection, possibly empty, of disjoint 2-element subsets of the set of grade 1 legs of $w$. Let $\mathcal{P}(w)$
denote the set of pairings of the diagram $w$. Then $\lambda$ is defined as the linear extension of the map which sends a diagram $w$ to a certain sum
\[
\lambda(w)=\sum_{\wp\in\mathcal{P}(w)}\mathcal{D}_\wp(w),
\]
where $\mathcal{D}_\wp(w)$ denotes $w$ with its grade 1 legs glued together according to the pairing $\wp$, equipped with an appropriate coefficient.
To recall the exact coefficient, we'll follow through the following example:
\[
w=\raisebox{-6ex}{\scalebox{0.25}{\includegraphics{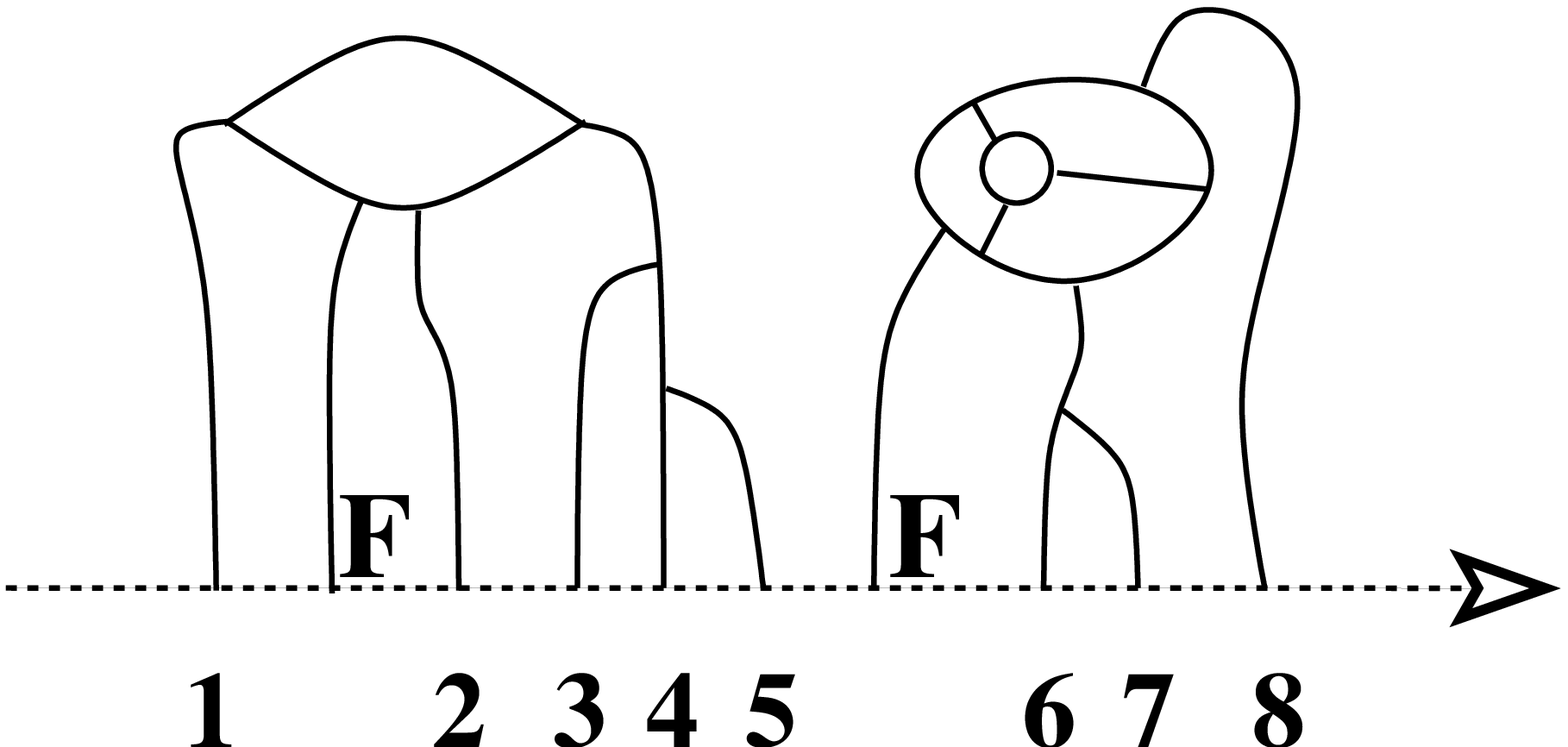}}}\
\ \text{and}\ \ \wp = \{\{1,3\},\{2,4\},\{5,7\}\}.
\]\vspace{0.1cm}

Begin by introducing a second
orienting line underneath the diagram, with a gap separating the
two orienting lines.
Then, for every pair of legs in the pairing $\wp$, add an arc,
using a full line, between the corresponding legs of the diagram
(such that the introduced arc has no self-intersections and stays
within the gap between the two orienting lines).
Finally, carry all the remaining legs straight down onto the
bottom orienting line, using a full line for the grade 1 legs and
a dashed line for the grade 2 legs:
\[
\raisebox{-4ex}{\scalebox{0.25}{\includegraphics{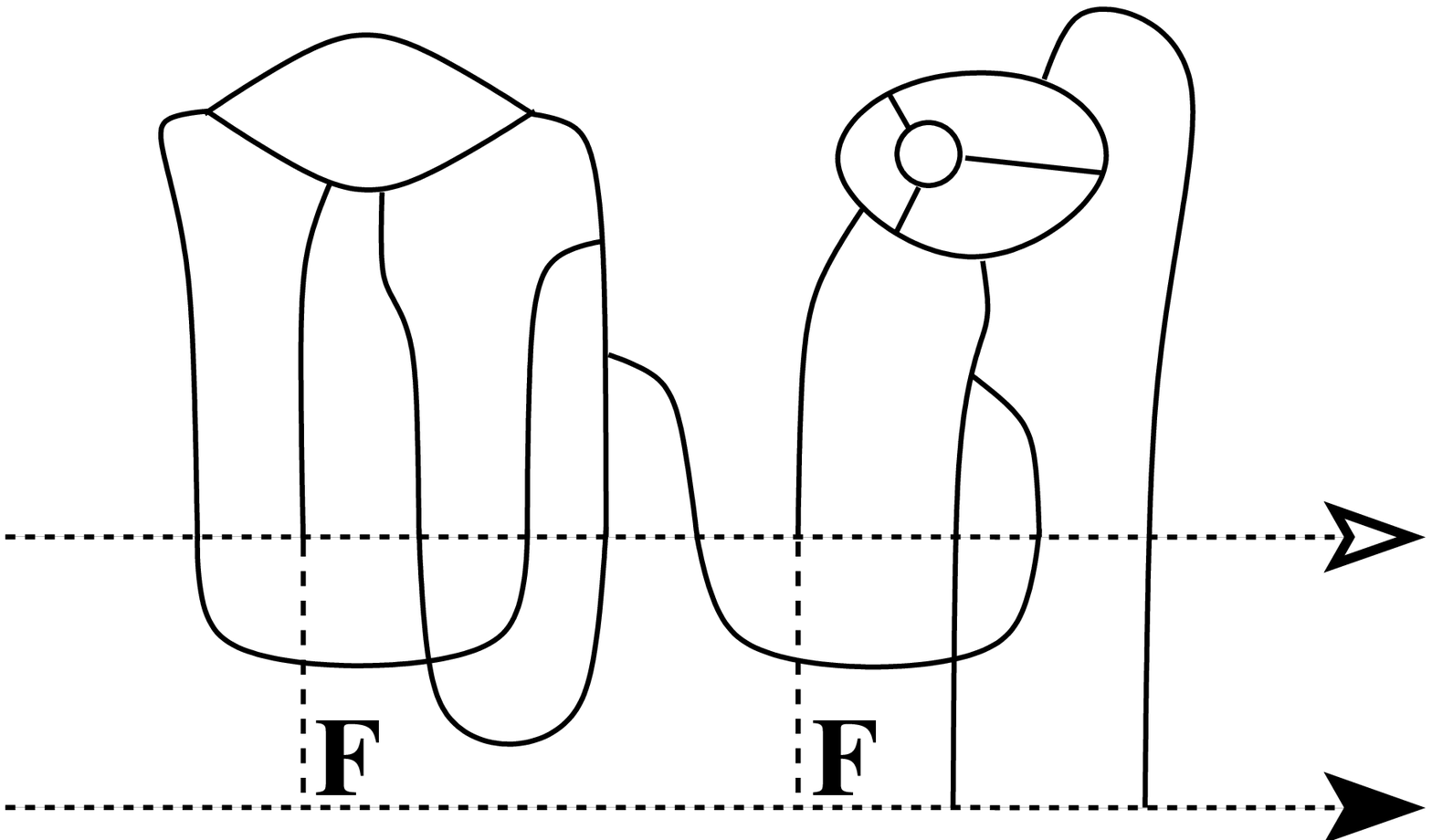}}}\
.
\]\vspace{0.1cm}

Let $x$ denote the number of intersections between full lines
displayed within the gap. The term $\mathcal{D}_\wp(w)$ is this
diagram (with the original orienting line forgotten and the dashed
lines filled in) multiplied by $(-1)^x$ and one factor of
$(\frac{1}{2})$ for every pair of legs glued together.

Thus, in the example at hand:
\[
\mathcal{D}_{\{\{1,3\},\{2,4\},\{5,7\}\}}
=(-1)^2\left(\frac{1}{2}\right)^3\raisebox{-3ex}{\scalebox{0.25}{\includegraphics{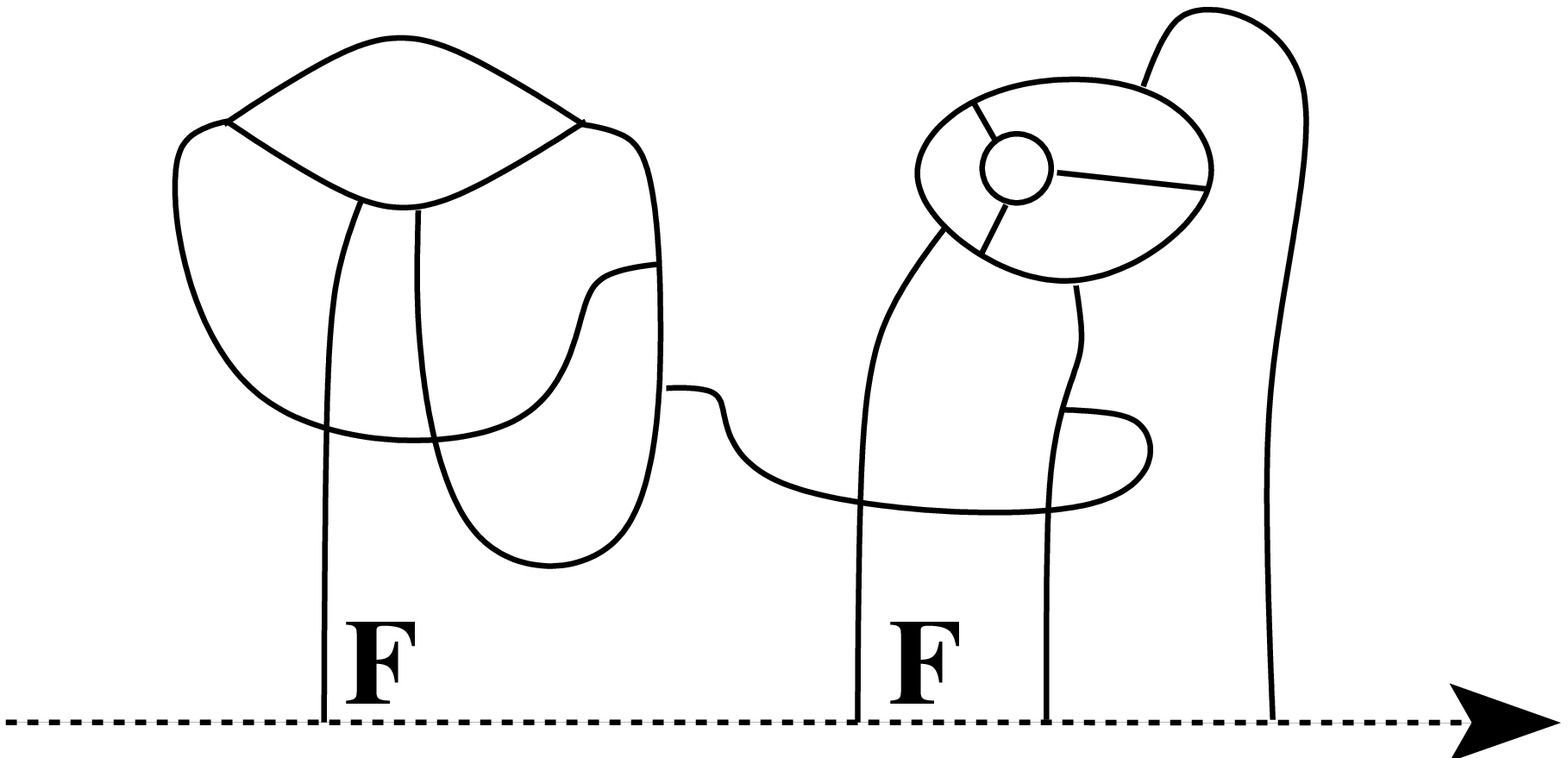}}}\
.
\]\vspace{0.1cm}

Figure \ref{lambdaexample} gives an example of the result of a full calculation of $\lambda$.

\begin{figure}\label{lambdaexample}
\caption{An example of $\lambda$.}
\begin{eqnarray*}
\lefteqn{\lambda\left(
\raisebox{-4ex}{\scalebox{0.25}{\includegraphics{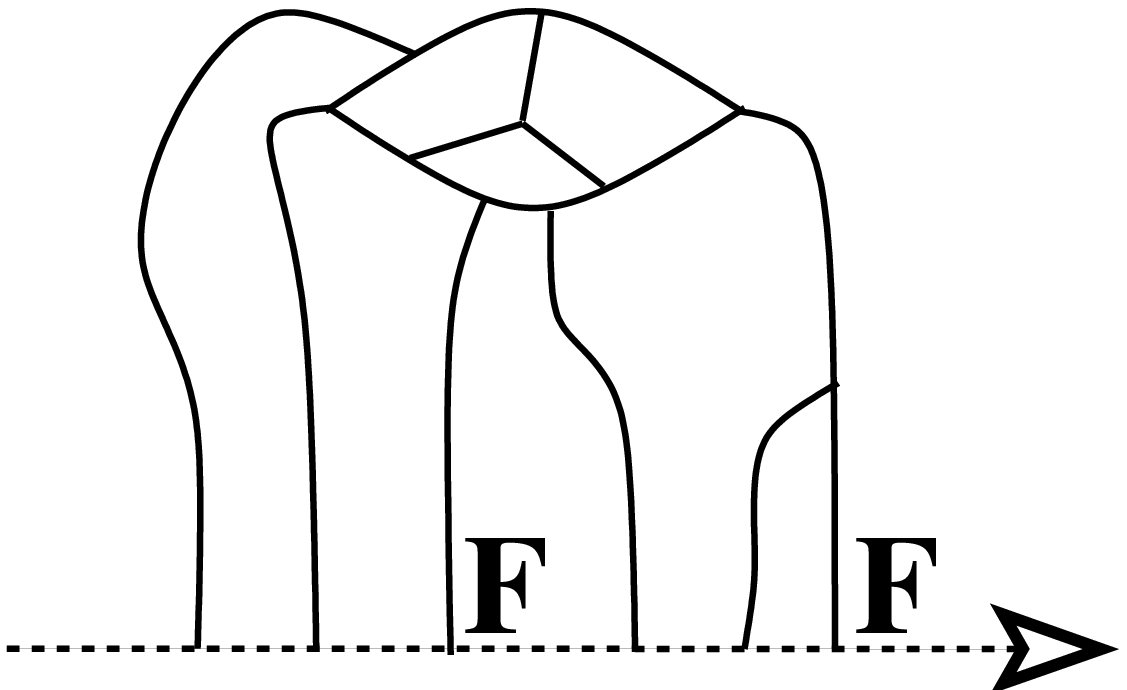}}}\right)\
\ =\ \
\raisebox{-4ex}{\scalebox{0.25}{\includegraphics{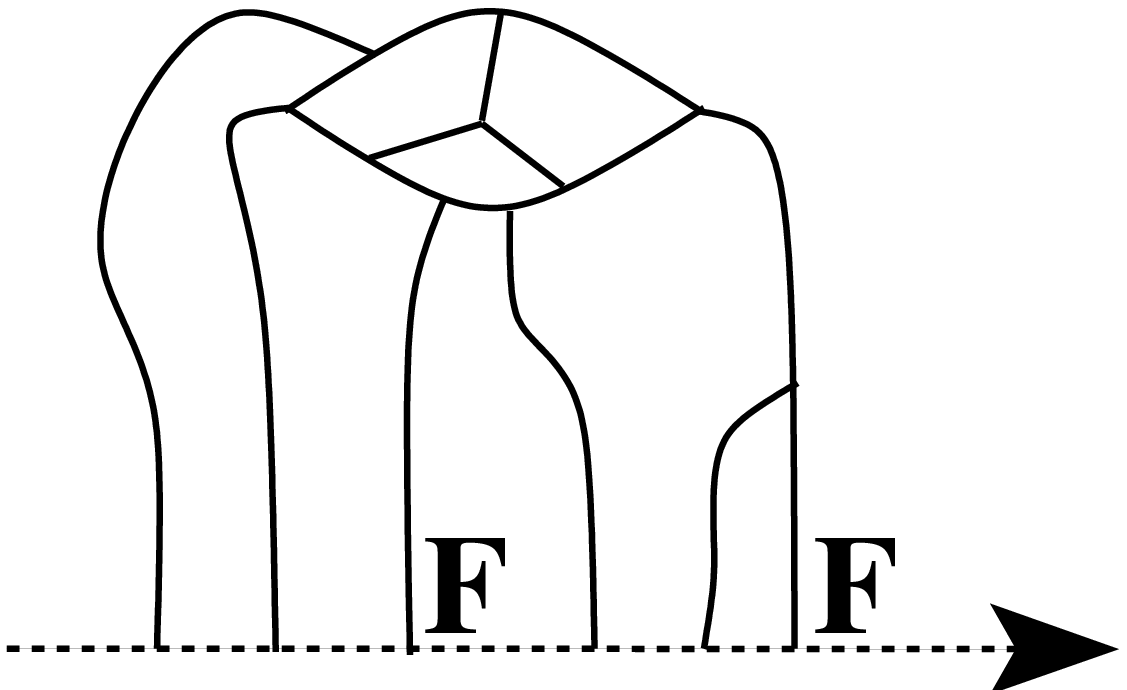}}}}\\[0cm]
& & + \frac{1}{2}\
\raisebox{-3ex}{\scalebox{0.25}{\includegraphics{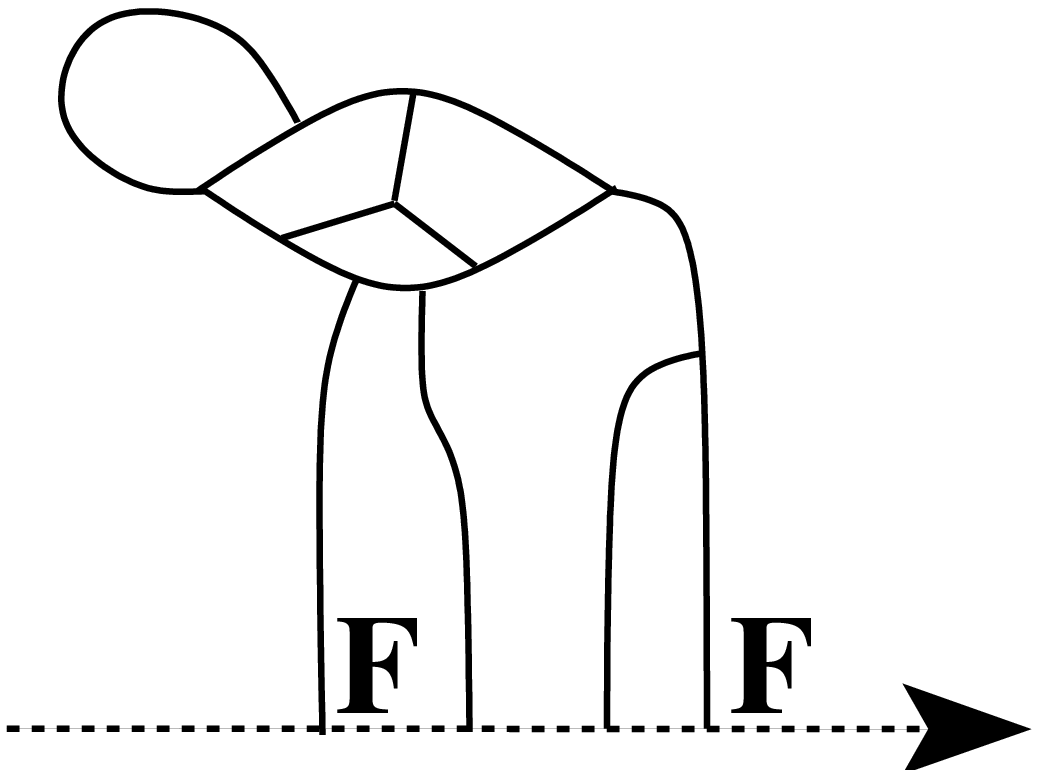}}}\
-\frac{1}{2}\
\raisebox{-3ex}{\scalebox{0.25}{\includegraphics{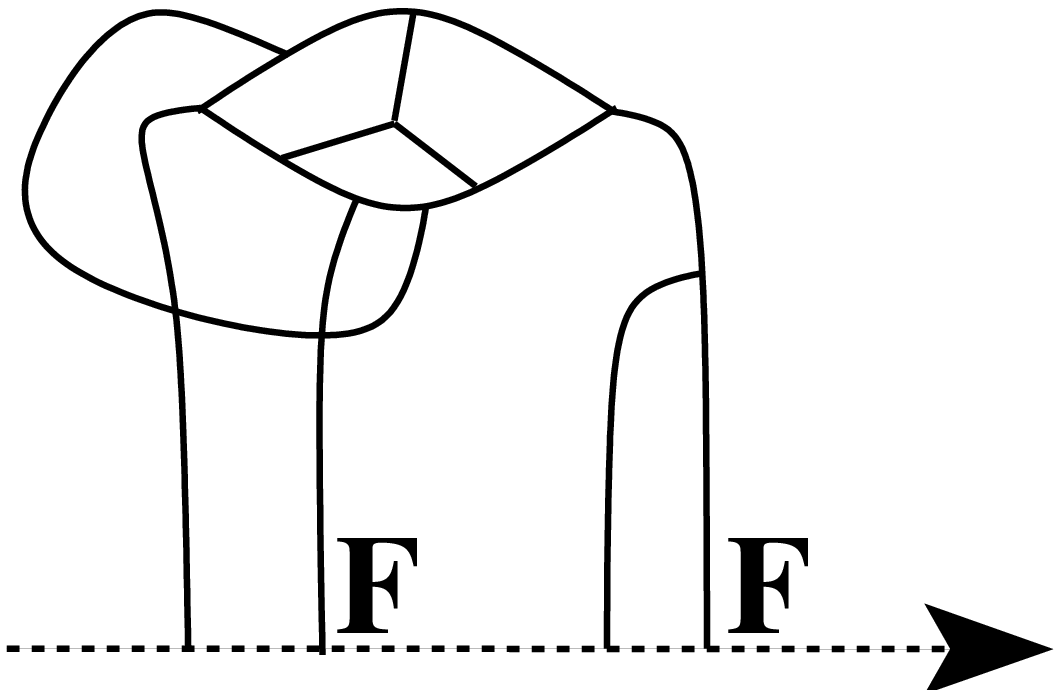}}}\
+\frac{1}{2}\
\raisebox{-3ex}{\scalebox{0.25}{\includegraphics{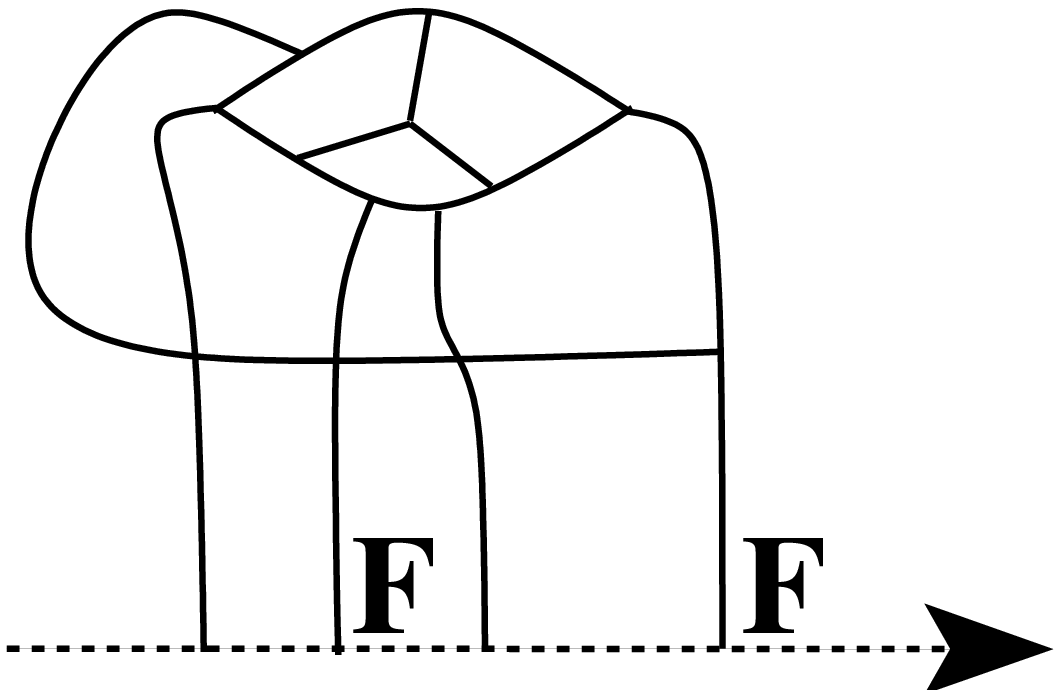}}}\
\\[0.15cm]
& & +\frac{1}{2}\
\raisebox{-3ex}{\scalebox{0.25}{\includegraphics{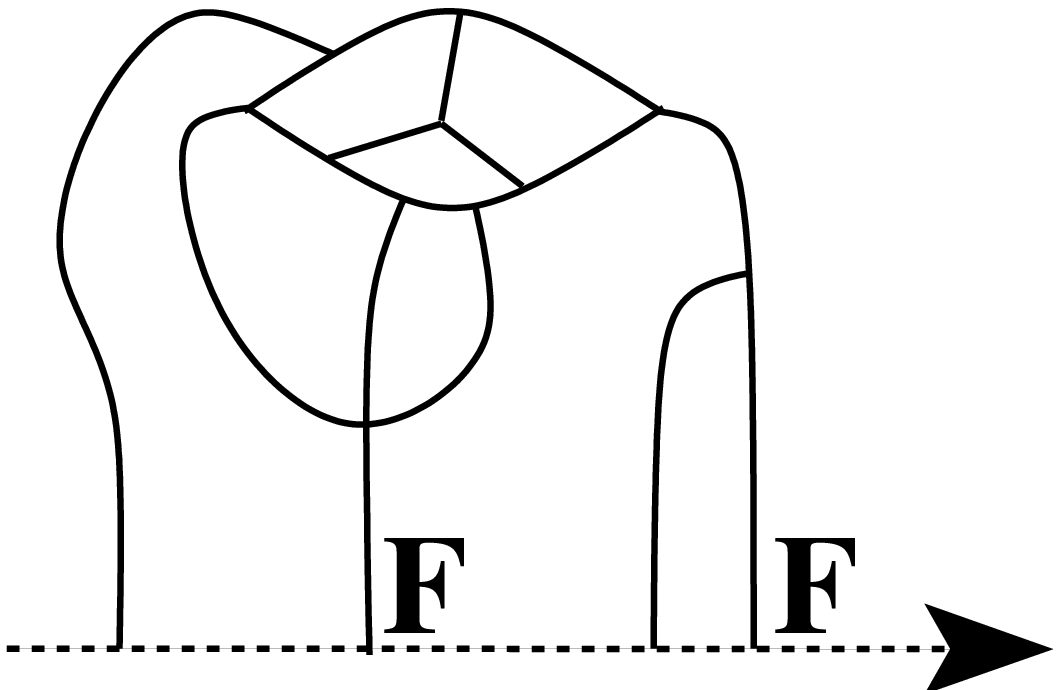}}}\
-\frac{1}{2}\
\raisebox{-3ex}{\scalebox{0.25}{\includegraphics{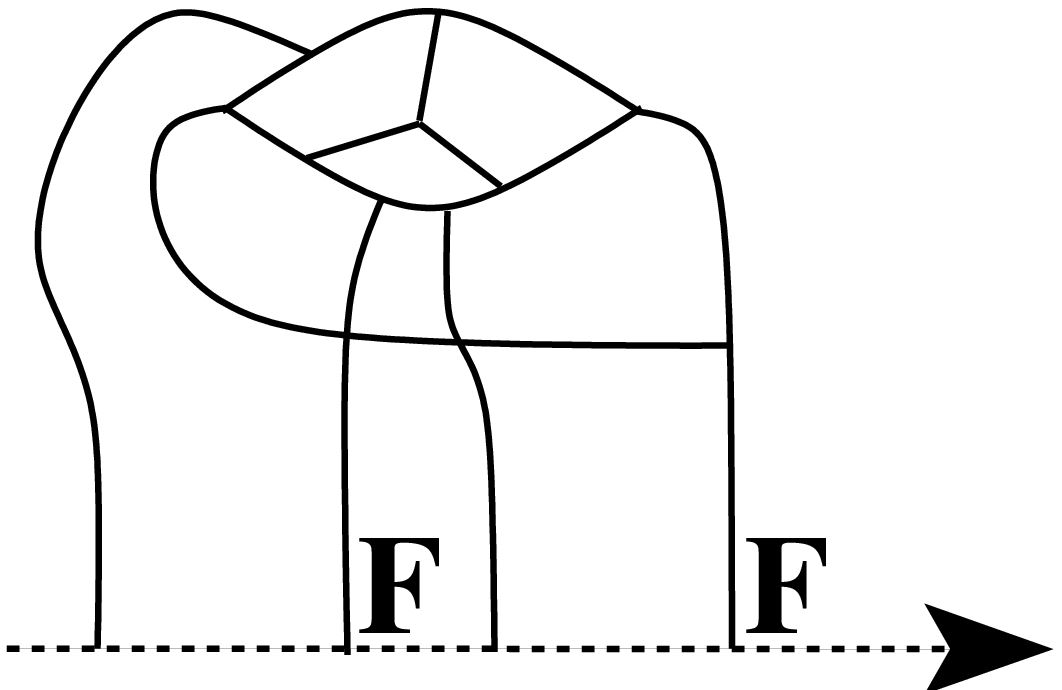}}}\
+\frac{1}{2}\
\raisebox{-3ex}{\scalebox{0.25}{\includegraphics{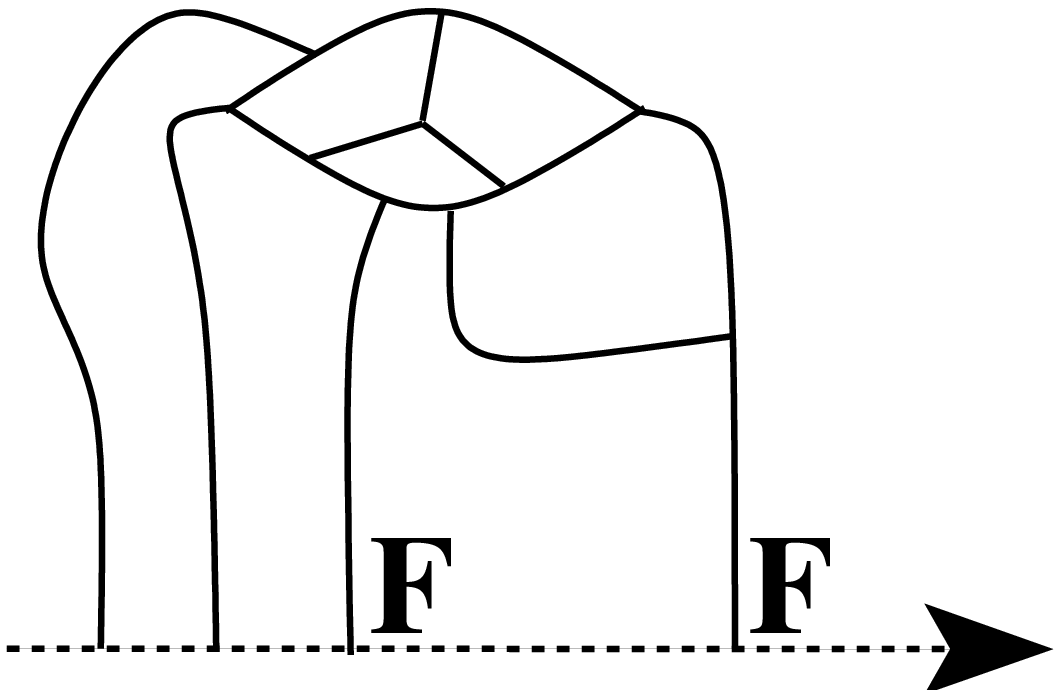}}}
\\[0.15cm]
& & +\frac{1}{4}\
\raisebox{-3ex}{\scalebox{0.25}{\includegraphics{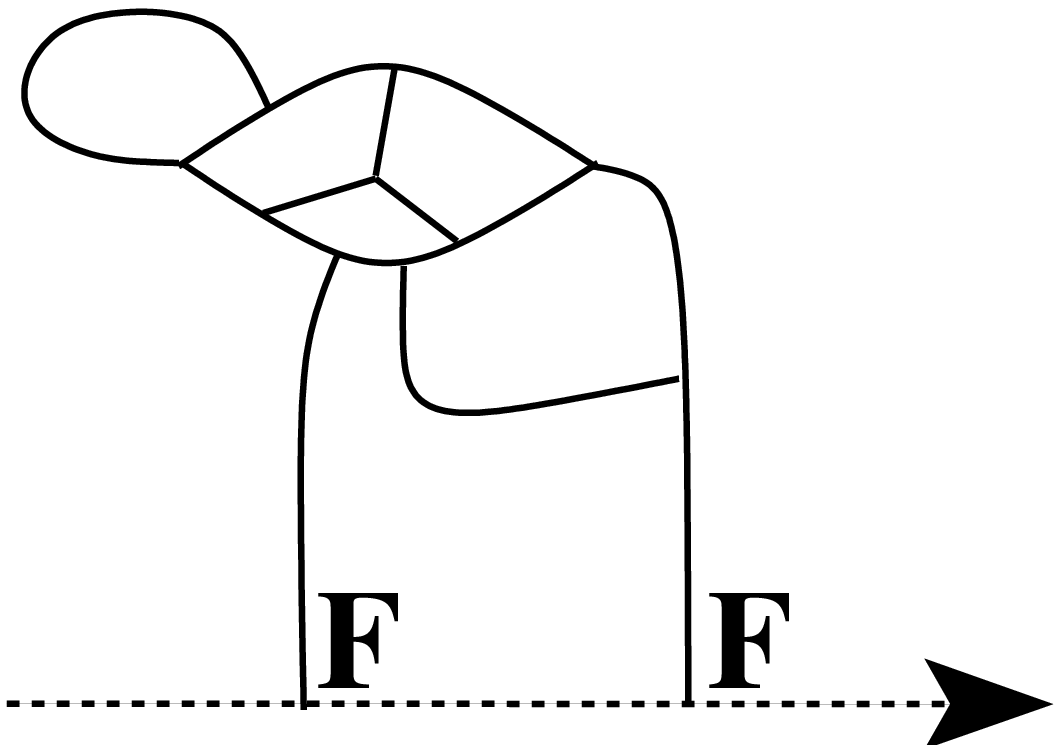}}}\
-\frac{1}{4}\
\raisebox{-3ex}{\scalebox{0.25}{\includegraphics{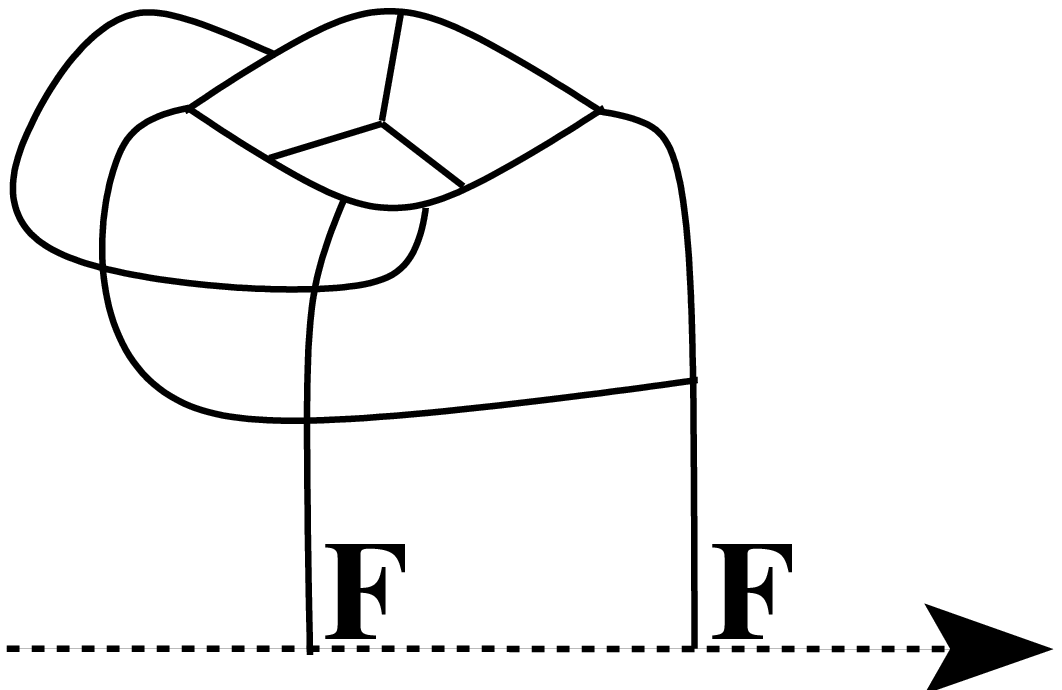}}}\
+\frac{1}{4}\
\raisebox{-3ex}{\scalebox{0.25}{\includegraphics{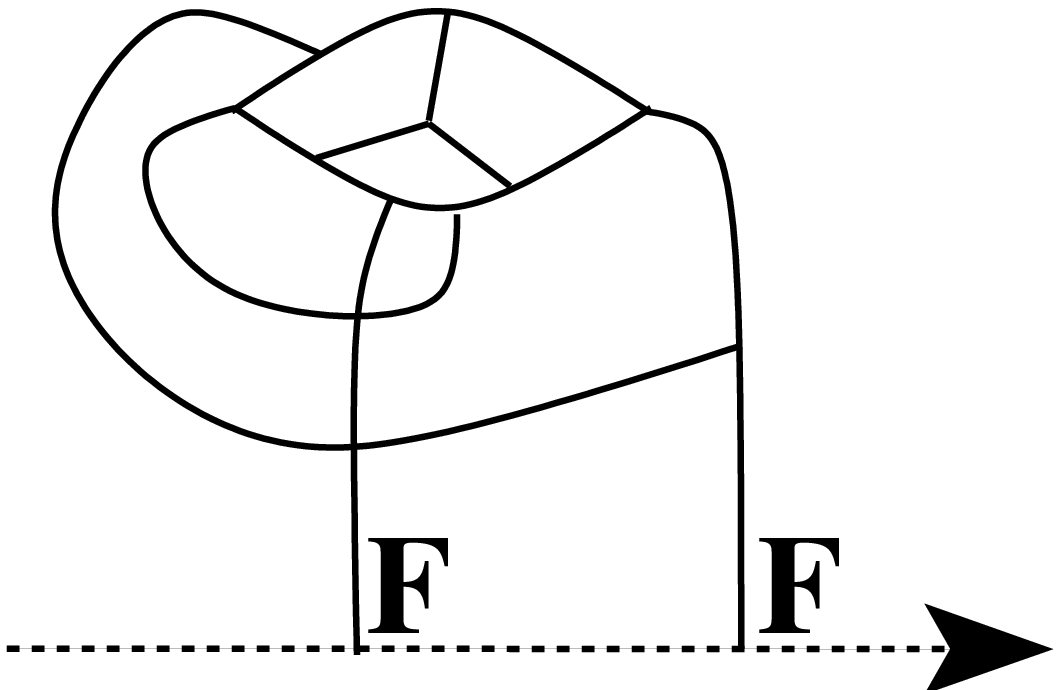}}}\
.\\
\end{eqnarray*}

\underline{\hspace{7cm}}
\end{figure}

\subsection{The map $\phi_\Aspace: \Aspace \rightarrow \widehat{\mathcal{W}}_\wedge$.}

Observe that the space $\widehat{\mathcal{W}}_\wedge$ has no relations which involve the leg-grade 1 legs, except relations which say that when
we transpose an adjacent pair of such legs then we pick up a minus sign. This means that $\widehat{\mathcal{W}}_\wedge$ is graded by the number of leg-grade 1
legs that a diagram has:
\[
\widehat{\mathcal{W}}_\wedge \simeq \bigoplus_{i=0} \widehat{\mathcal{W}}^i_\wedge\ ,
\]
where $\widehat{\mathcal{W}}^i_\wedge$ denotes the subspace  generated by diagrams with exactly $i$ leg-grade $1$ legs. The space $\widehat{\mathcal{W}}^0_\wedge$ is
clearly isomorphic to $\Aspace$, and $\phi_\Aspace:\Aspace \rightarrow \widehat{\mathcal{W}}_\wedge$ is the corresponding
embedding, whose action on an ordered Jacobi diagram
is simply to make every leg a curvature leg, e.g.:
\[
\phi_\Aspace\left(
\raisebox{-4.5ex}{\scalebox{0.25}{\includegraphics{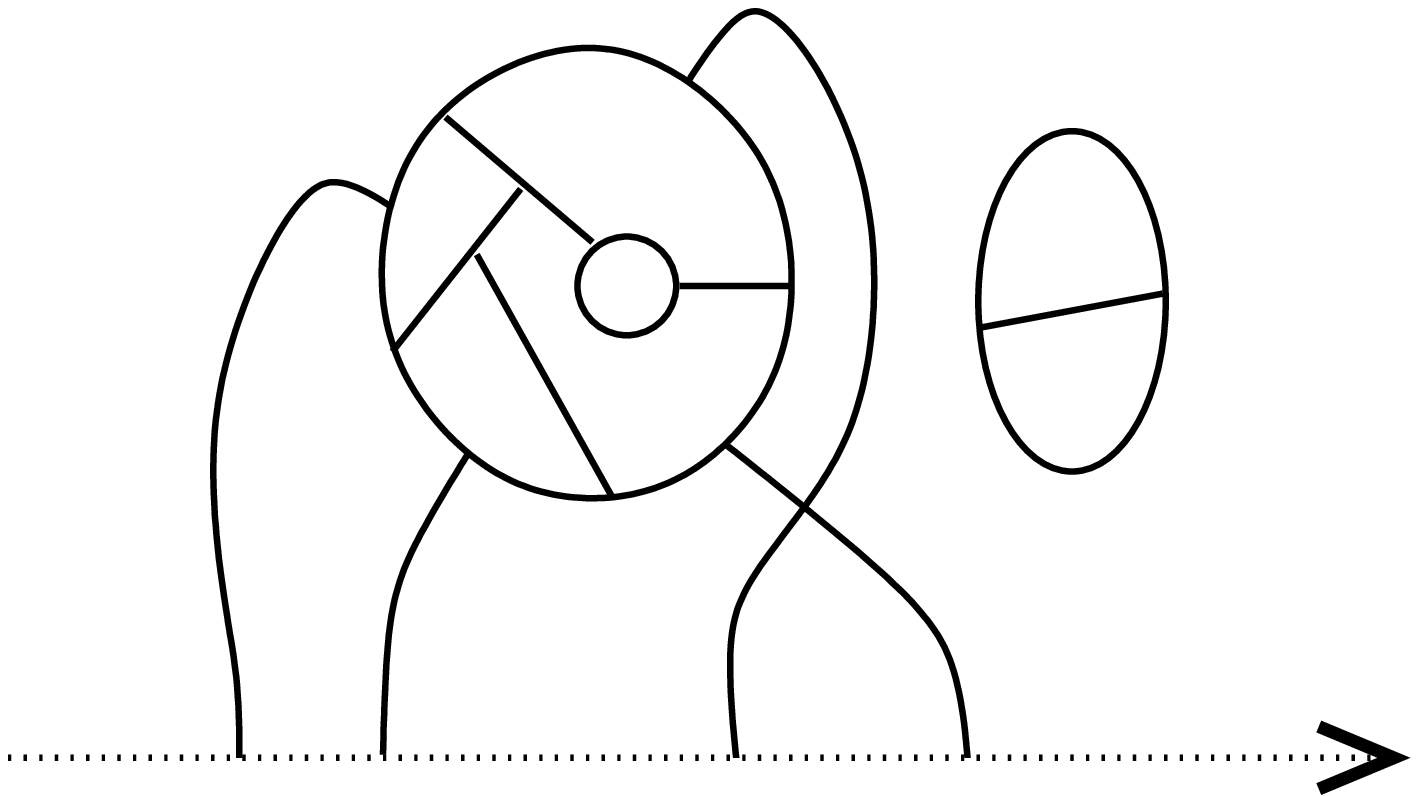}}}
\right) \ =\
\raisebox{-4.5ex}{\scalebox{0.25}{\includegraphics{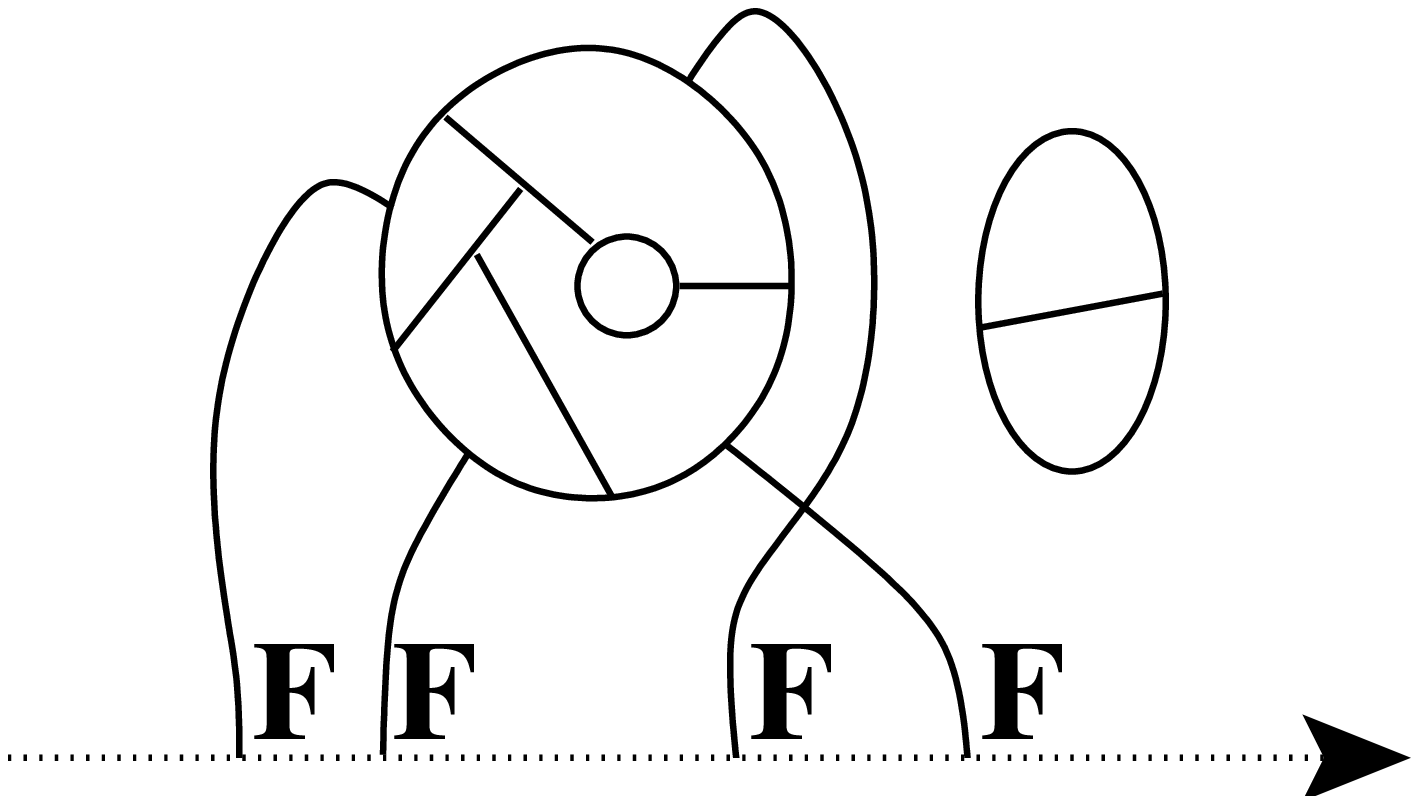}}}\ \ .
\]

That concludes our recollection of the spaces and maps that are involved in the identity that is the subject
of the main theorem:
\[
\lambda\circ\basebulltoF\circ\pi\circ\chi_{\Wspace}\circ\Upsilon = \phi_\Aspace \circ \chi_\Bspace \circ \partial_\Omega.
\]

\section{Outline of the contents of this paper}

The fact that this sequence of elementary combinatorial operations is just the
Wheeling operation $\phi_\Aspace\circ \chi_\Bspace \circ \partial_\Omega$ may be a little surprising. Just how do these maps lead to wheels being glued into legs?

In figure \ref{illustrfig} we illustrate the mechanism which produces wheels by taking a random symmetric Jacobi diagram
and following it as it maps through this composition. At each stage in the composition we have drawn a diagram that is
typical of the diagrams appearing in the sum at that point. In words: first we glue in a few forks (the map $\Upsilon$), then we rearrange the legs
according to some permutation (the map $\chi_\Wspace$), then we glue in some more forks (the map $\basebulltoF$), and then we join up some pairs of
leg-grade 1 legs (the map $\lambda$).

\begin{figure}
\caption{An illustration of how ``wheels" appear.}
\[
\label{illustrfig} \xymatrix{ *{\begin{array}{c} \ \\ \
\\ \end{array}\mathcal{B}\begin{array}{c} \ \\ \
\\ \end{array}} \ar[d]_{\Upsilon} & *{
\raisebox{-4ex}{\scalebox{0.24}{\includegraphics{illustrA}}}} \\
*{\begin{array}{c} \ \\ \
\\ \end{array}\mathcal{W} \begin{array}{c} \ \\ \
\\ \end{array}}\ar[d]_{\chi_{\mathcal{W}}} & *{
\raisebox{-4ex}{\scalebox{0.24}{\includegraphics{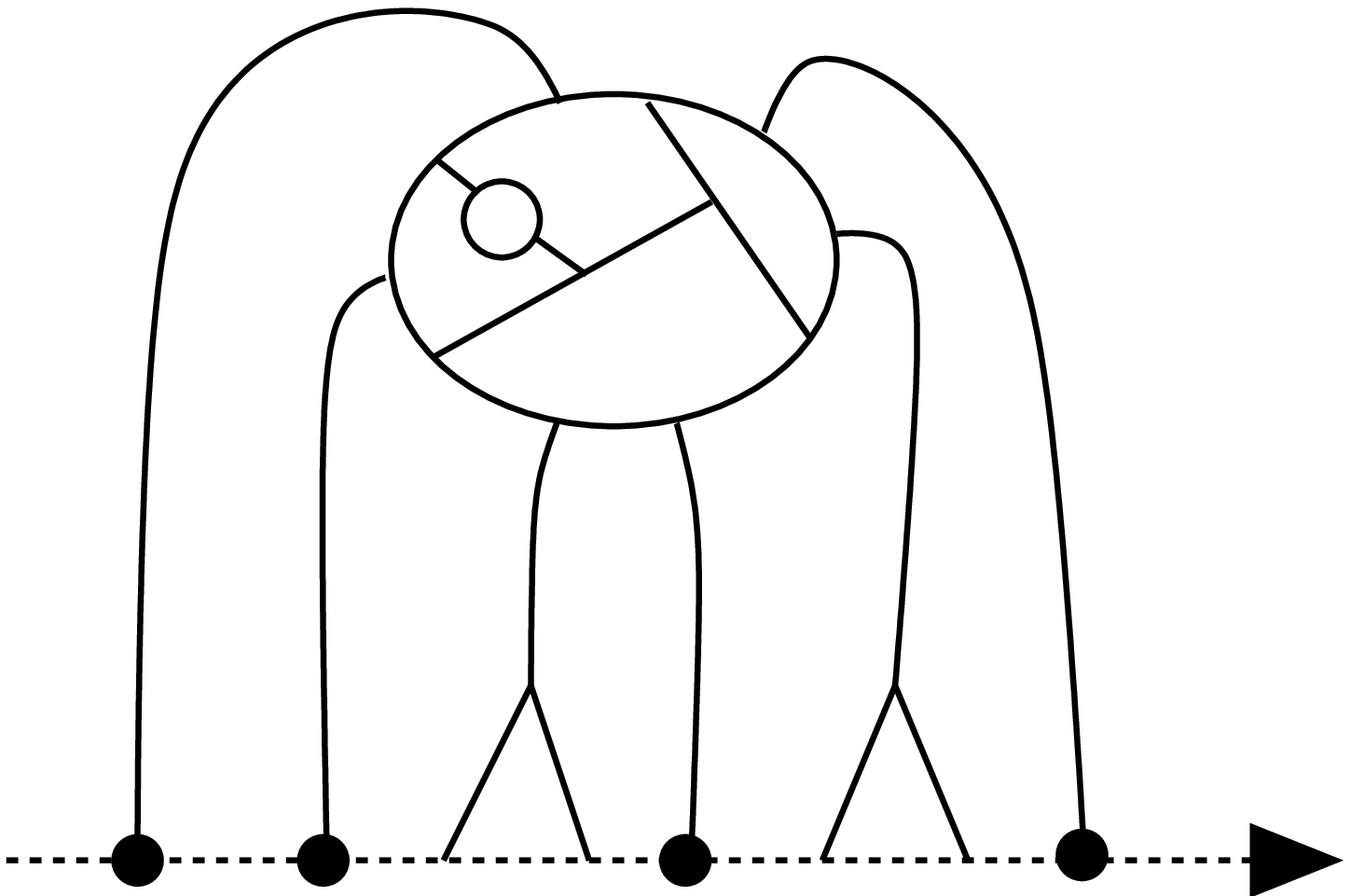}}}}
\\ *{\begin{array}{c} \ \\ \
\\ \end{array}\widetilde{\mathcal{W}}\begin{array}{c} \ \\ \
\\ \end{array}} \ar[d]_{(B_{\bullet\rightarrow\mathrm{F}})\circ\pi}& *{
\raisebox{-4ex}{\scalebox{0.24}{\includegraphics{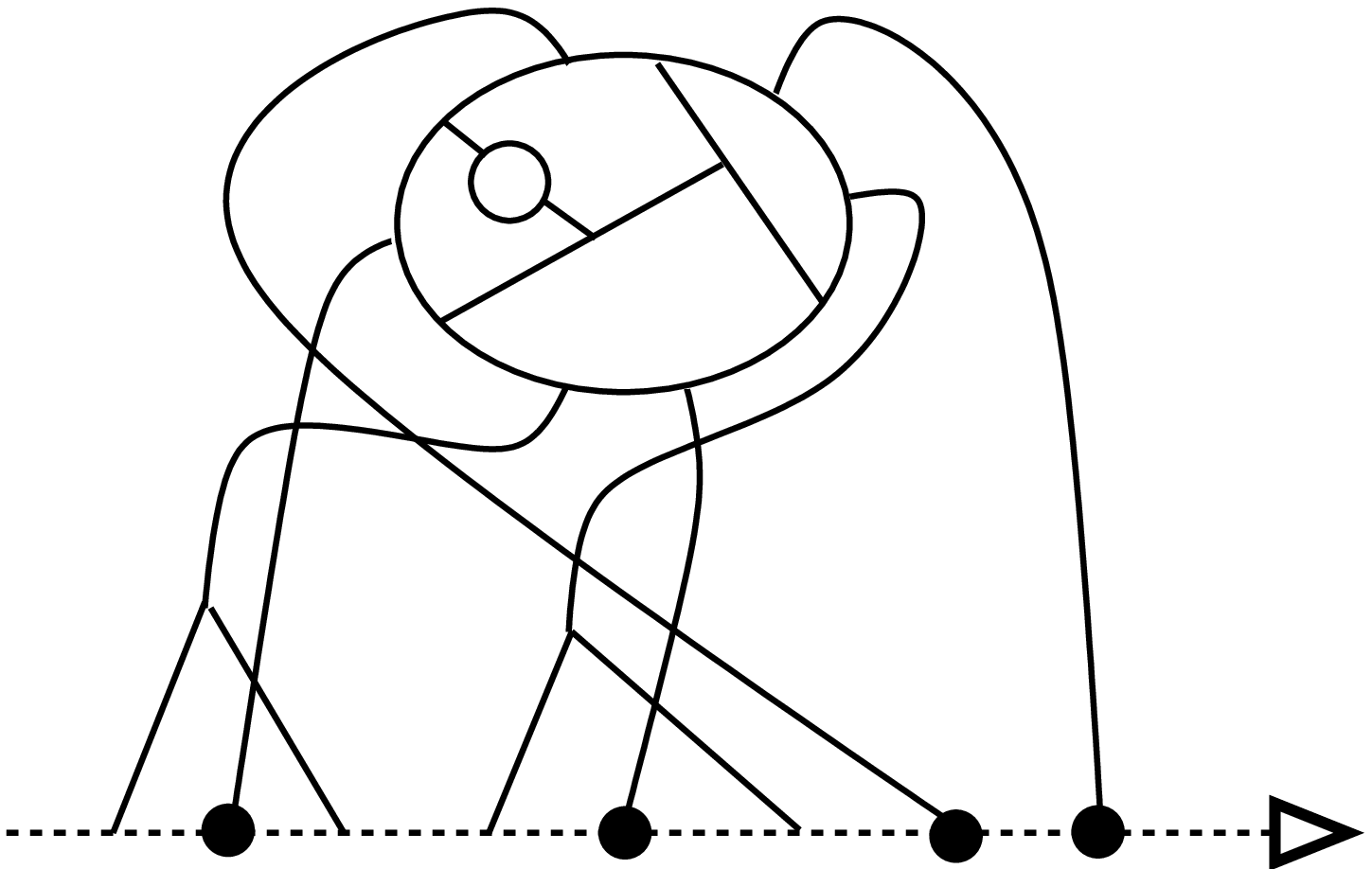}}}}
\\
*{\begin{array}{c} \ \\ \
\\ \end{array}\widehat{\mathcal{W}}_{\mathrm{F}}\begin{array}{c} \ \\ \
\\ \end{array} } \ar[d]_{\lambda} & *{
\raisebox{-4ex}{\scalebox{0.24}{\includegraphics{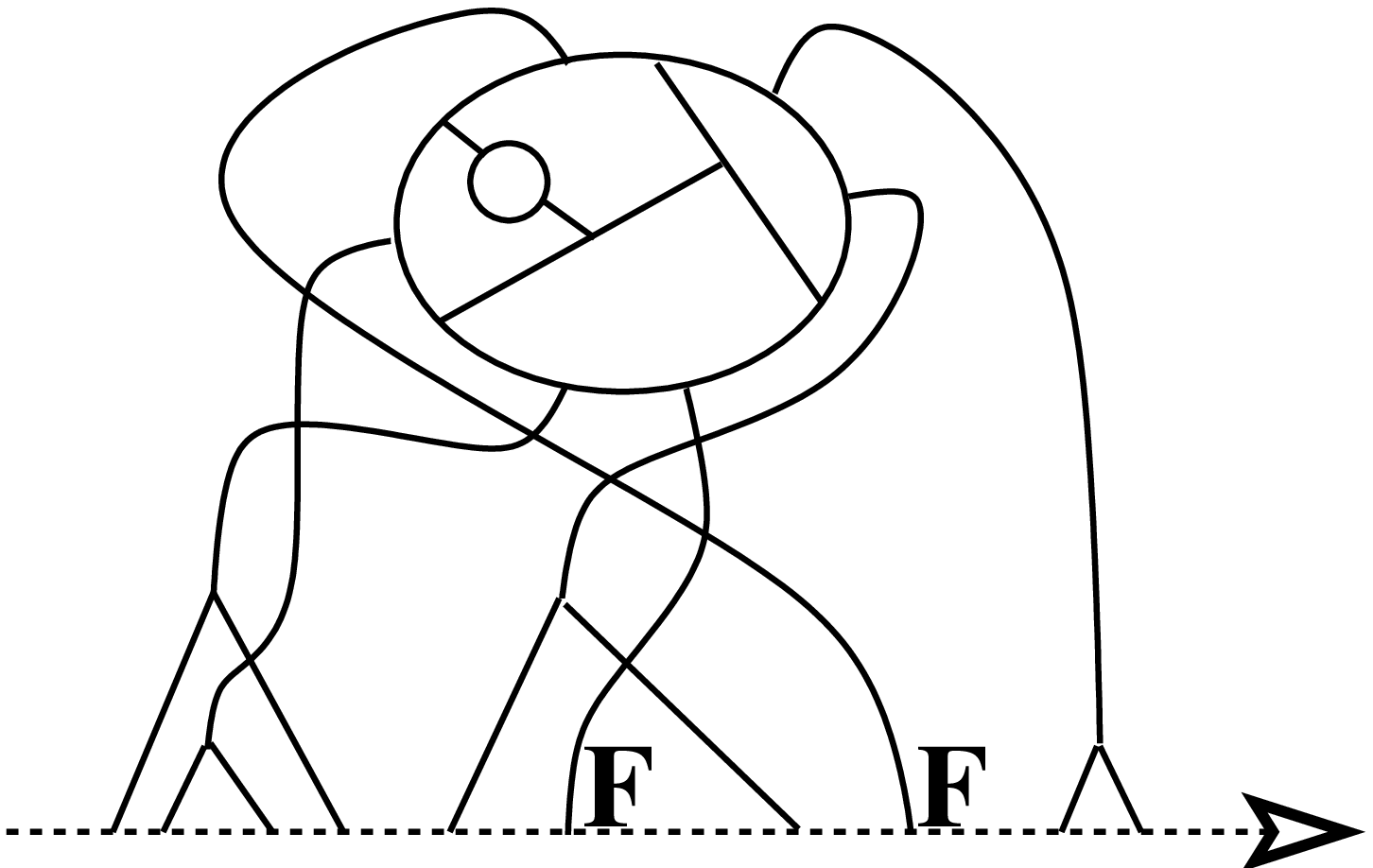}}}}
\\
*{\begin{array}{c} \ \\ \
\\ \end{array}\widehat{\mathcal{W}}_{\wedge}\begin{array}{c} \ \\ \
\\ \end{array}} & *{
\raisebox{-4ex}{\scalebox{0.24}{\includegraphics{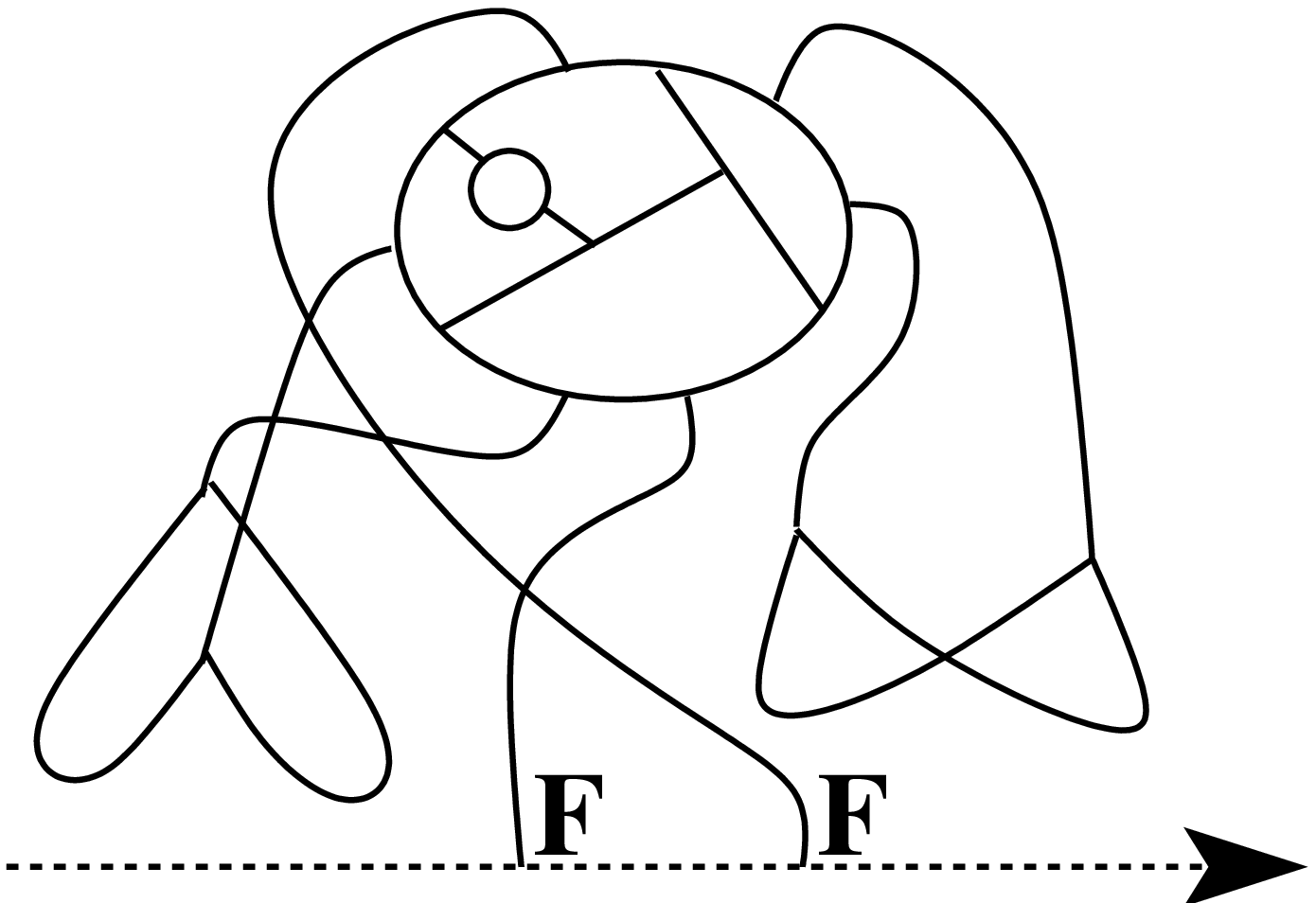}}}} }
\] \

\underline{\hspace{7cm}}
\end{figure}
Observe that the last diagram is equal to the image under
$\phi_\Aspace$ of
\[
\raisebox{-4ex}{\scalebox{0.24}{\includegraphics{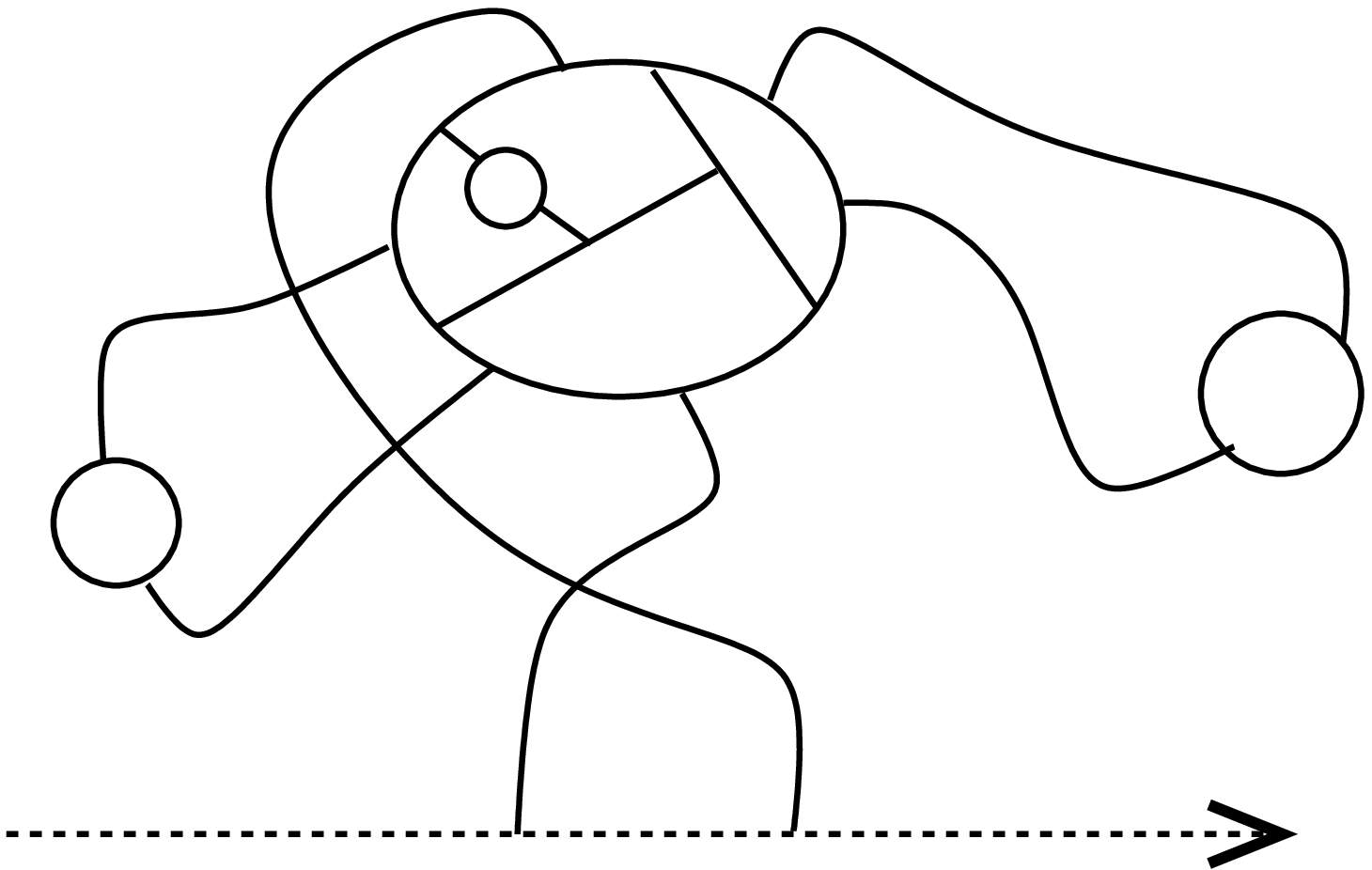}}}\ \ ,
\]
the original diagram with two wheels glued into its legs and its
remaining legs given some ordering.

This shows that it is immediate that the composition $\lambda\circ\basebulltoF\circ\pi\circ\chi_{\Wspace}\circ\Upsilon$
produces the same sort of diagrams as the Wheeling operation, $\phi_\Aspace \circ \chi_\Bspace \circ \partial_\Omega$, (as well
as some other diagrams).
Showing that the diagrams are produced with the correct coefficients, and that all other diagrams produced cancel,
is the difficulty of the proof.

The paper is organized in the following way:
\begin{itemize}
\item{In Section \ref{diagoperators} we explain how formal power series of diagrams of various sorts operate on each other.
In the case that the legs are ungraded this is a familiar
operation, being, for example, the calculus in which the {\AA}rhus 3-manifold invariant is constructed \cite{arhus}.}
\item{In Section \ref{compopexp} we develop a certain expression (Theorem \ref{importantexpression}) for the composition $\lambda\circ B_{\bullet\rightarrow\mathrm{\bf
F}} \circ \pi \circ \chi_\Wspace \circ \Upsilon$ in terms of these diagram operations. The key idea is that we can (graded) symmetrize a diagram by turning it into
a differential operator and then applying it to a suitable formal exponential (Proposition \ref{symmpro}).}
\item{The remaining two sections do a direct calculation of the expression developed in Theorem \ref{importantexpression}. The key idea in these computations
is that the result of an exponential of connected diagrams operating on another exponential of connected diagrams is the exponential of all the connected diagrams
you can construct from the logarithms of the pieces. This is a familiar story in the ungraded case - we provide detailed proofs that it holds in certain cases of the graded setting as well.
So all we need to do is calculate all the possible connected diagrams that can be constructed. This turns out to be a manageable combinatorial problem in the cases that arise.}
\end{itemize}

\section{Operating with diagrams on
diagrams.}\label{diagoperators}
\subsection{Operator Weil diagrams}
We'll now begin to operate on Weil diagrams with other Weil
diagrams.
%
To introduce this formalism we'll build a vector space
$\WhatF\abpow$. (The constructions to follow adapt in an
unambiguous way to build spaces like $\Whatwedge\abpow$, etc.)
Intuitively: we are taking the vector space $\WhatF$, adjoining a
formal grade $2$ variable $a$ and a formal grade $1$ variable $b$
and their corresponding differential operators, and then taking
power series with respect to those introduced symbols.

Formally: this space will be built from diagrams which may have
the usual legs for $\WhatF$, namely
\[
\mbox{grade $2$ legs}\ \
\raisebox{-3ex}{\scalebox{0.25}{\includegraphics{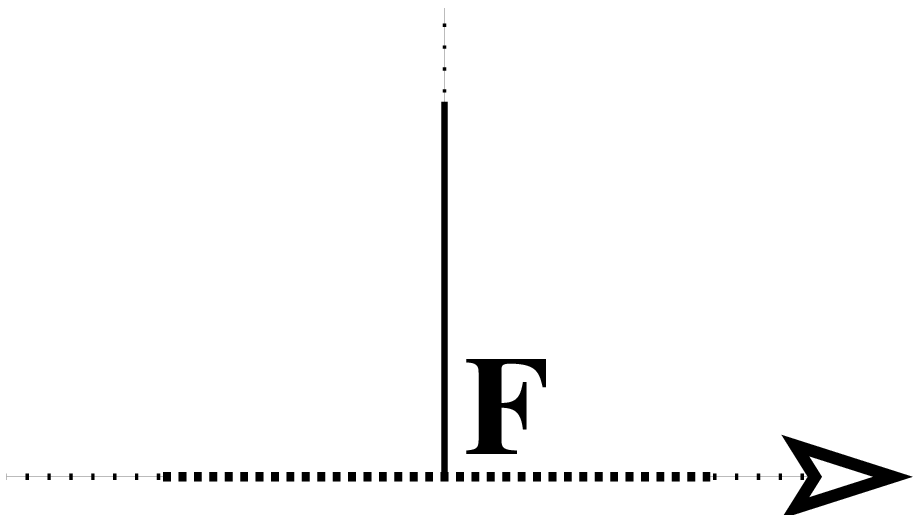}}}\ \
\mbox{and grade $1$ legs}\ \
\raisebox{-3ex}{\scalebox{0.25}{\includegraphics{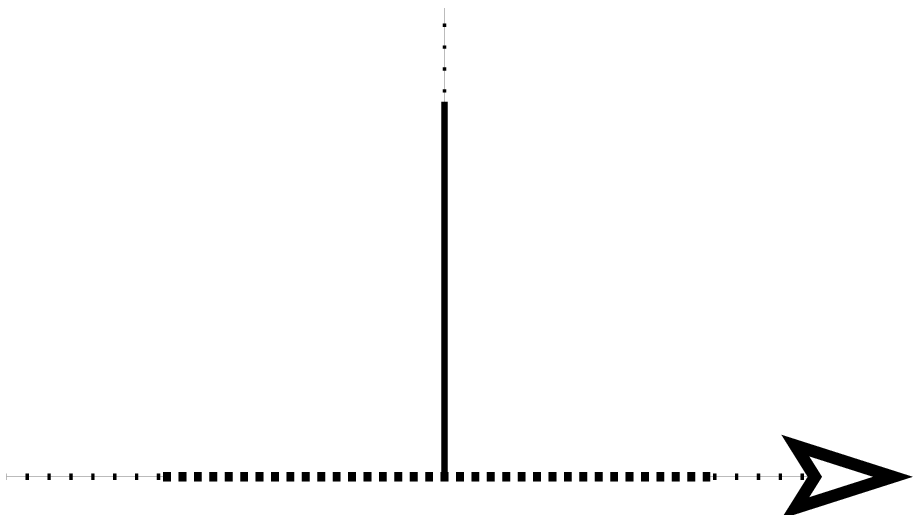}}}\ \
,\\
\]
but which may have, in addition, {\bf parameter legs}
\[
\raisebox{-5ex}{\scalebox{0.25}{\includegraphics{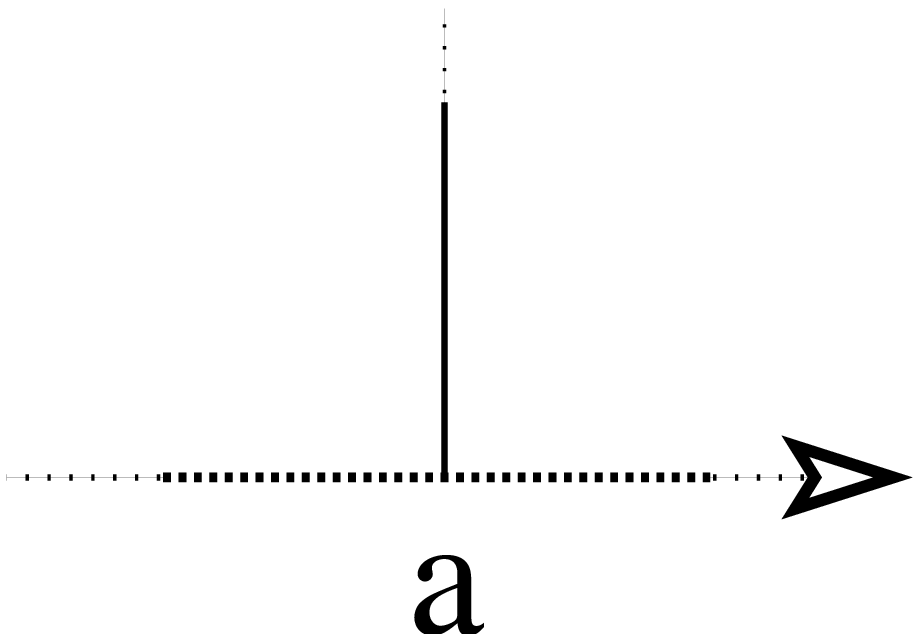}}}\ \
\text{and}\ \
\raisebox{-5ex}{\scalebox{0.25}{\includegraphics{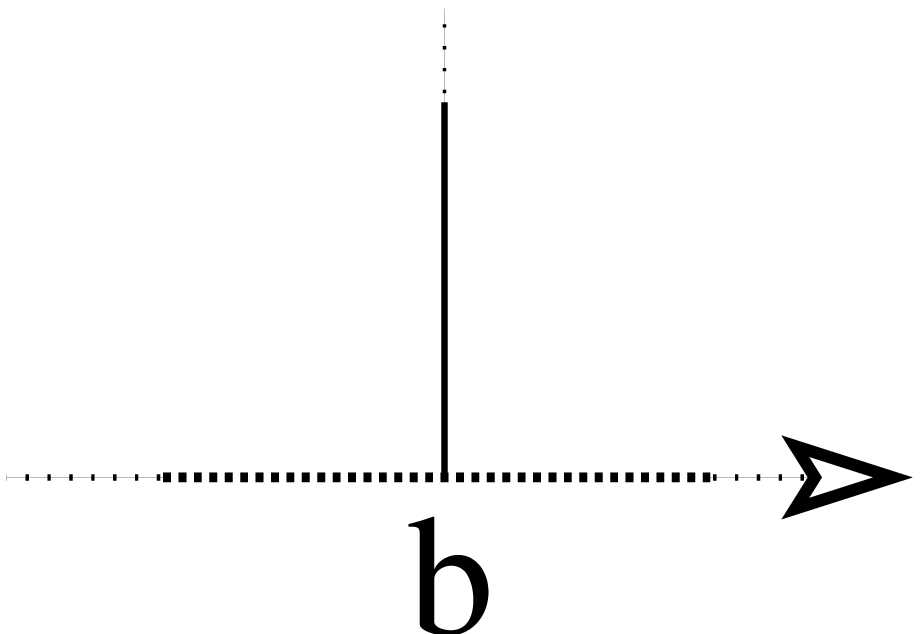}}}\ ,
\]
and their corresponding {\bf operator legs}
\[
\raisebox{-5.35ex}{\scalebox{0.25}{\includegraphics{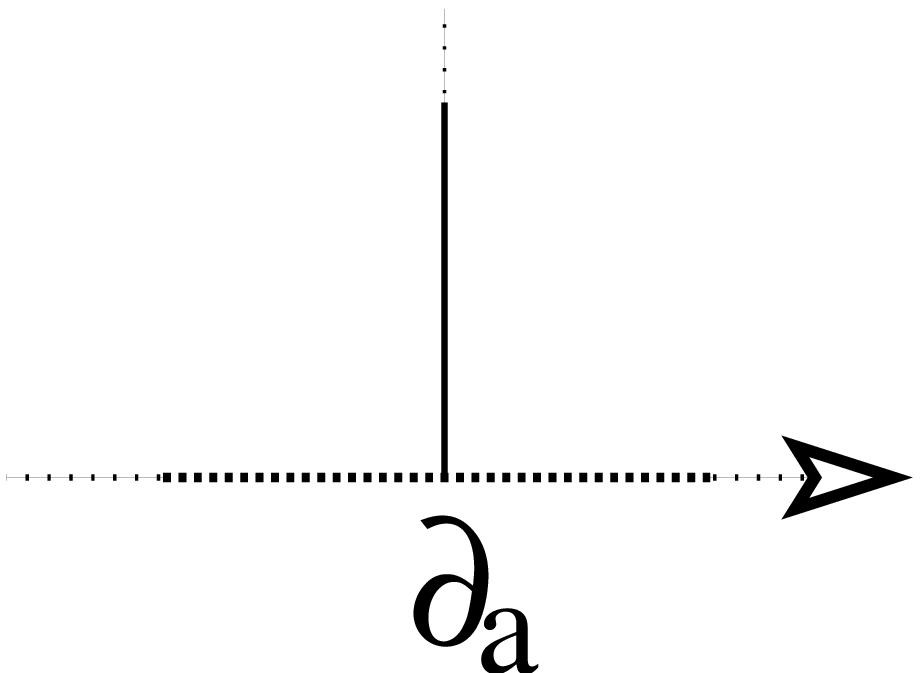}}}\
\ \text{and}\ \
\raisebox{-5.35ex}{\scalebox{0.25}{\includegraphics{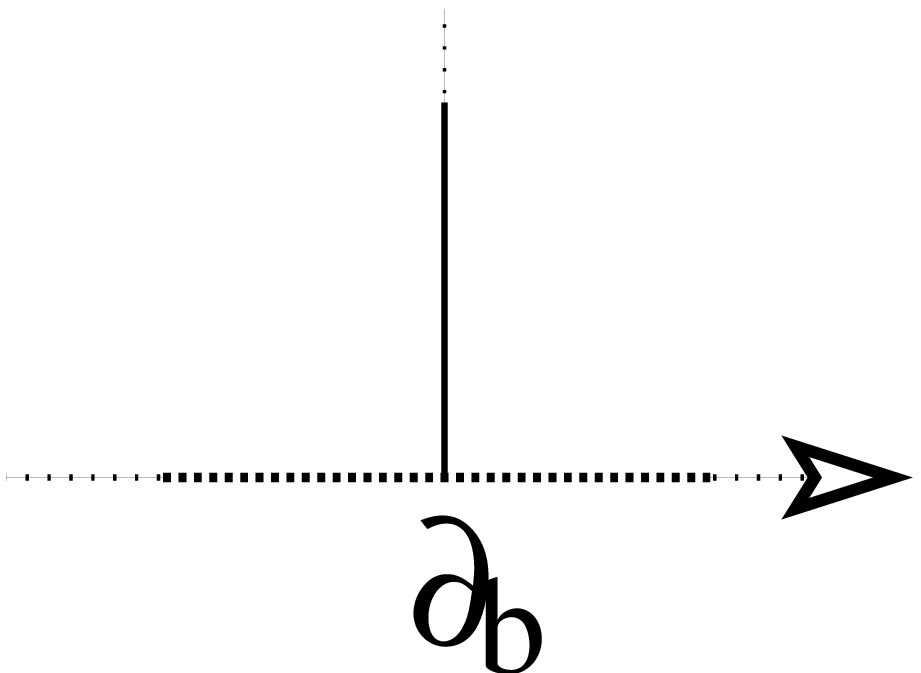}}}\
 .
\]
We require that the operator legs appear in a group at the far
right-hand end of the diagram. The parameter legs may appear
amongst the usual legs in any order. Here is an example of such an
{\bf operator Weil diagram}:
\[
\raisebox{-4ex}{\scalebox{0.32}{\includegraphics{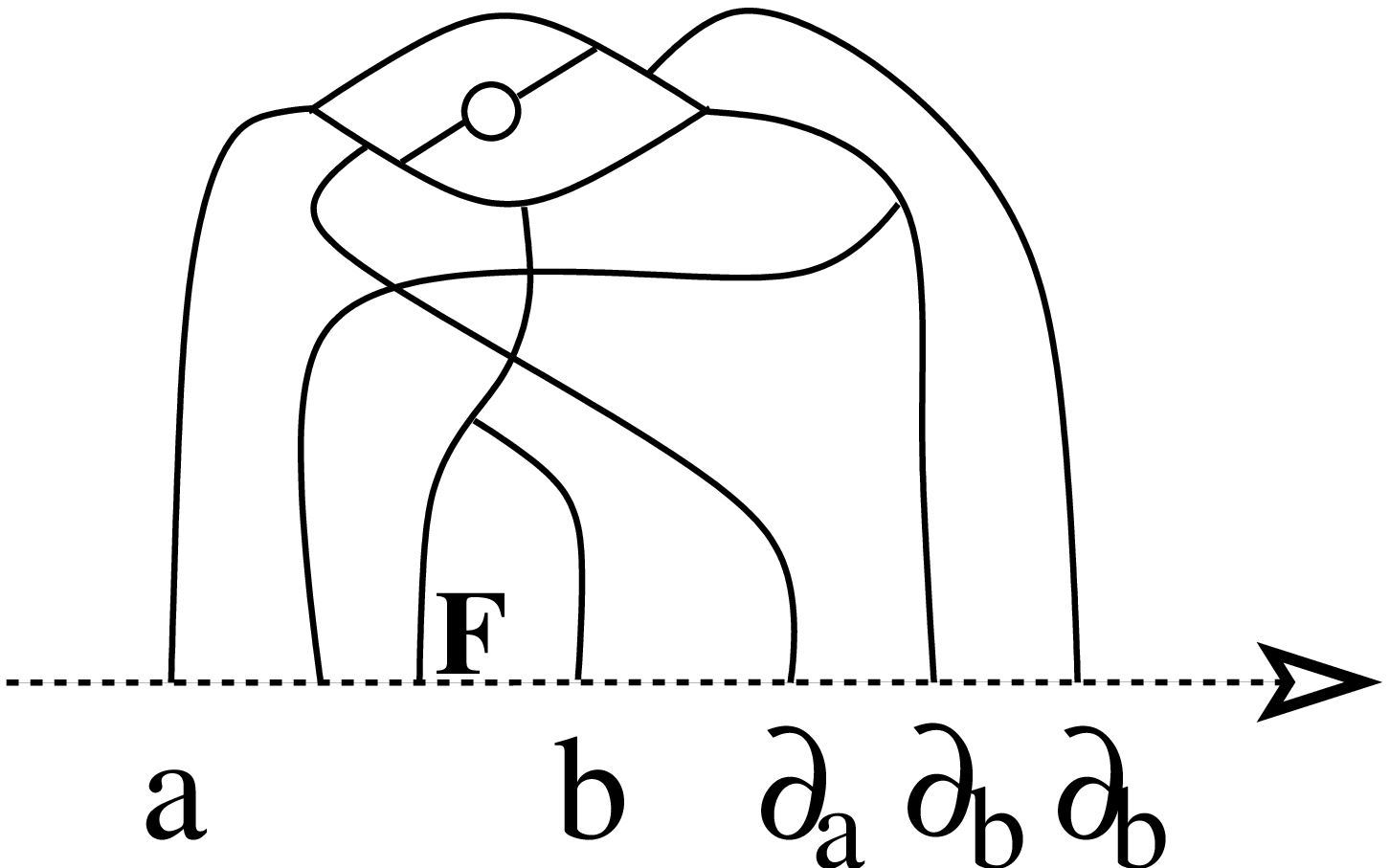}}}
\]
Let the {\bf parameter-grade} of a diagram be the total grade of
its parameter legs, where $a$-labelled legs count for $2$ and
$b$-labelled legs count for $1$. The diagram above has
parameter-grade $3$. Similarly define the quantity {\bf
operator-grade}; the diagram above has operator-grade $4$. If a
diagram has parameter-grade $i$ and operator-grade $j$ then we say
that it is an $(i,j)$-operator Weil diagram. The pair $(i,j)$ will
be referred to as the {\bf type} of the diagram. Thus, the
operator Weil diagram above has type $(3,4)$.

\begin{defn}
Define the vector space $\WhatF\ab^{(i,j)}$ to be the space of
formal $\mathbb{Q}$-linear combinations of operator Weil diagrams
of type $(i,j)$, subject to the same relations that the space
$\WhatF$ uses, together with relations that say that the parameter
and operator legs can be moved about freely (up to the appropriate
sign), as long as the operator legs all stay at the far-right hand
end of the orienting line.
\end{defn}
For example, the following equations hold in
$\WhatF\ab^{(3,2)}$:\\[0cm]
\[
\begin{array}{c}
\ \ \
\raisebox{-4ex}{\scalebox{0.25}{\includegraphics{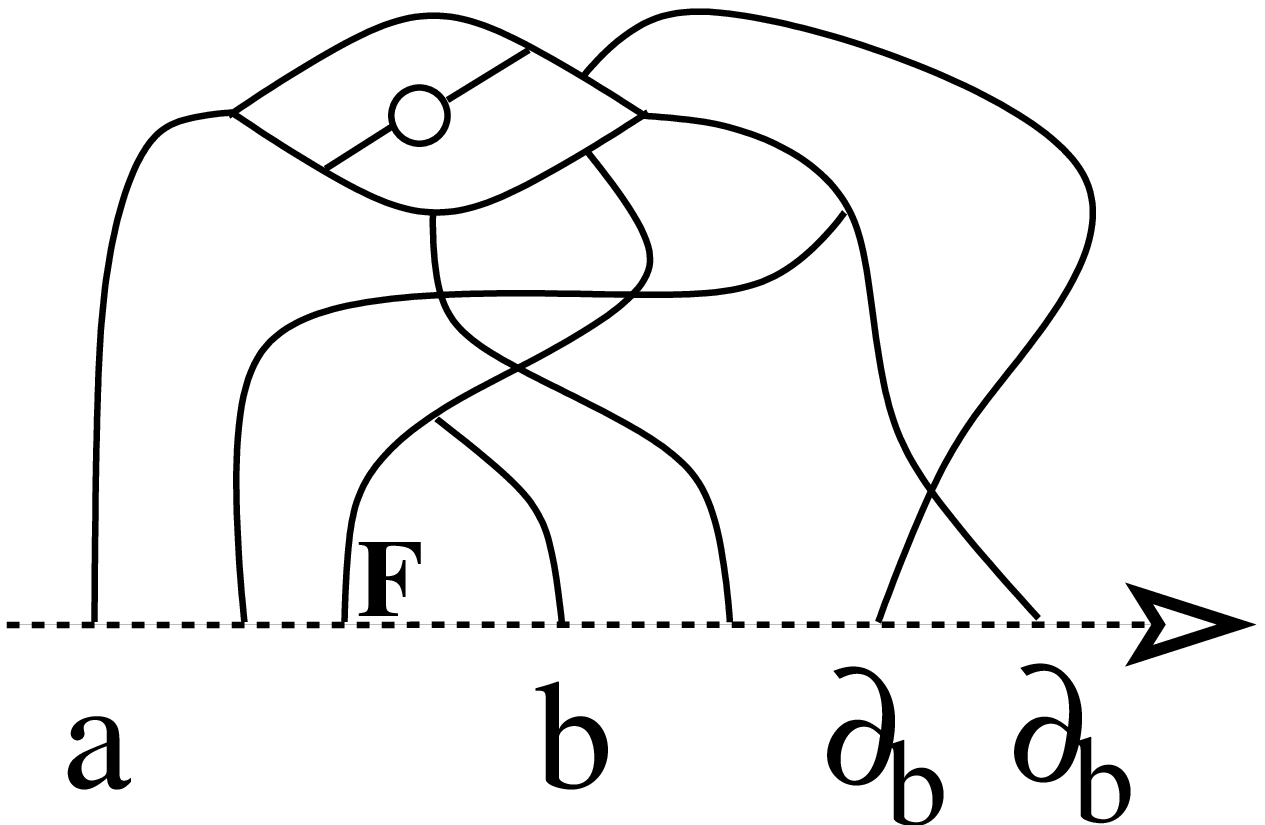}}}\ =\
-\ \raisebox{-4ex}{\scalebox{0.25}{\includegraphics{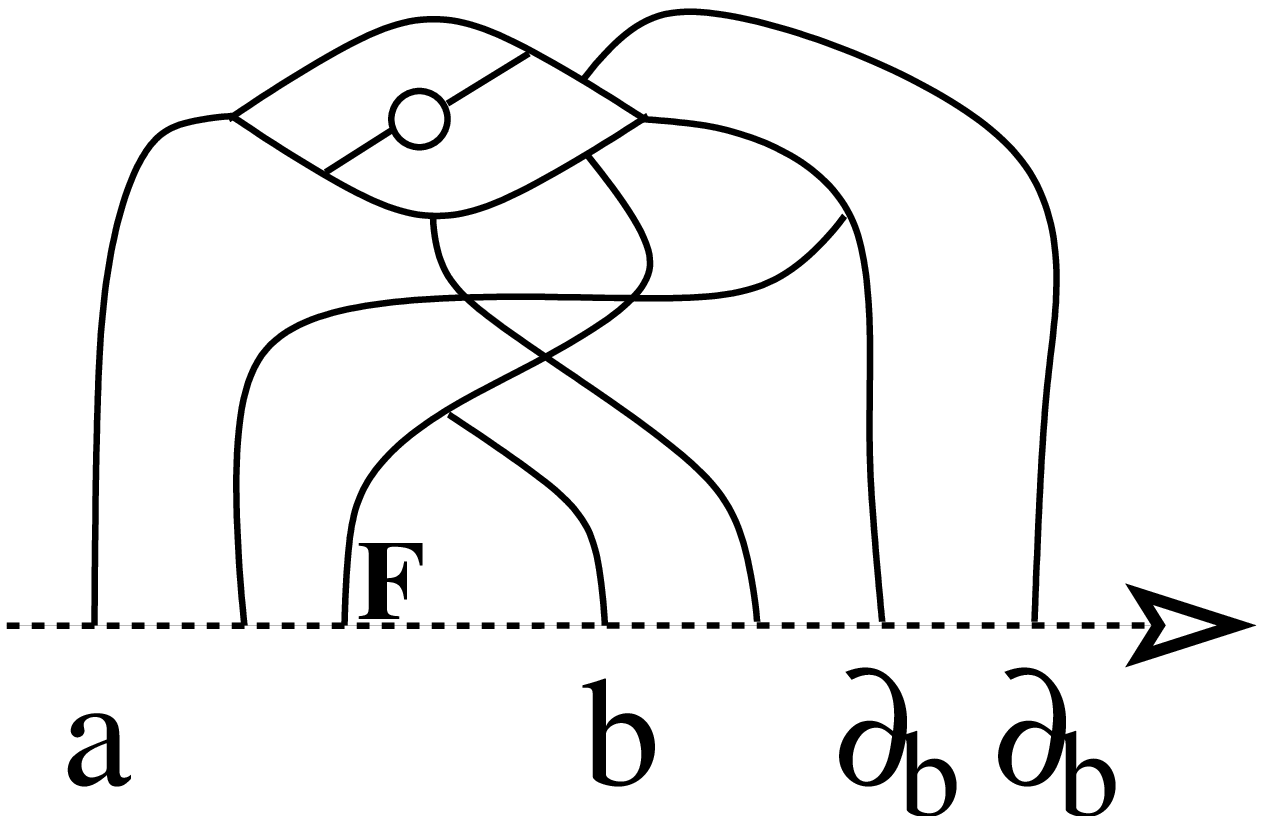}}}\
=\ \raisebox{-4ex}{\scalebox{0.25}{\includegraphics{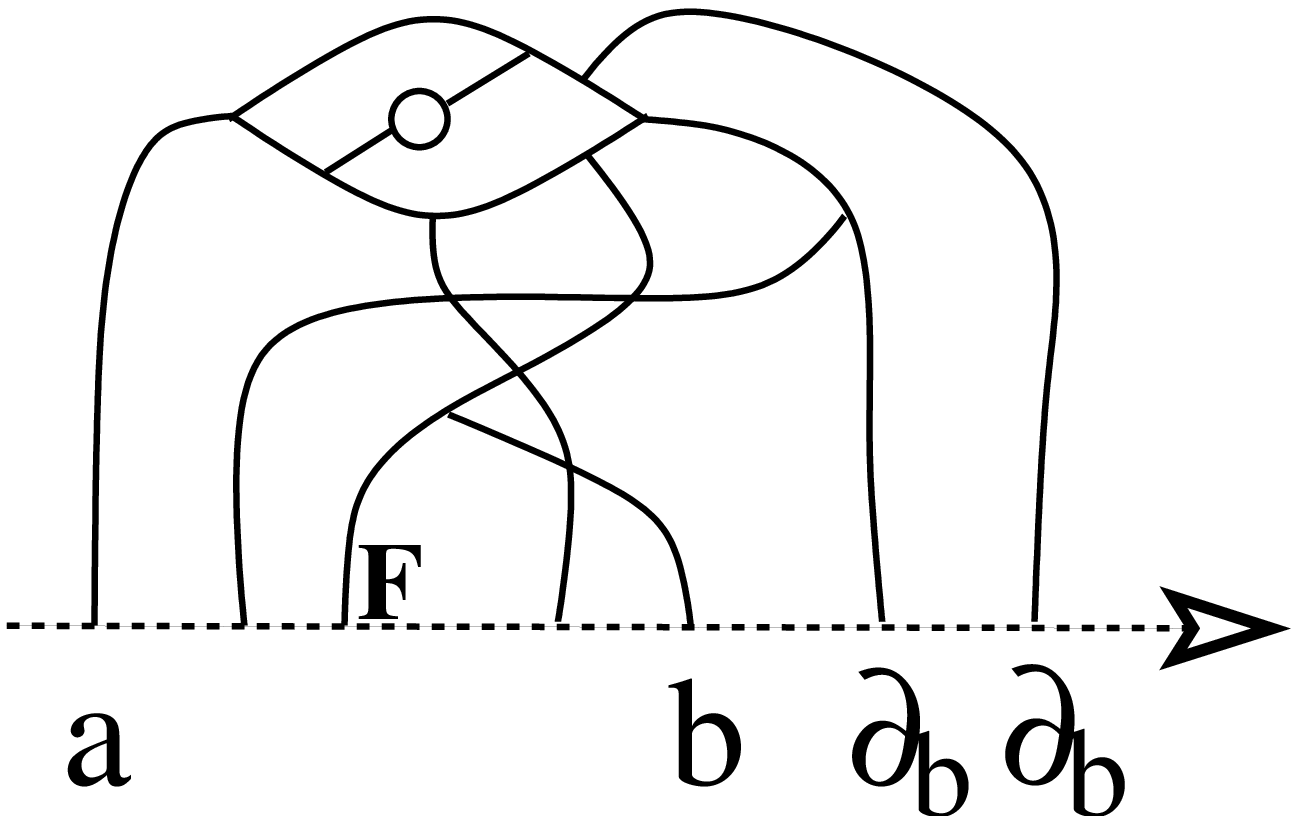}}}\
\\[0.75cm] =\
\raisebox{-4ex}{\scalebox{0.25}{\includegraphics{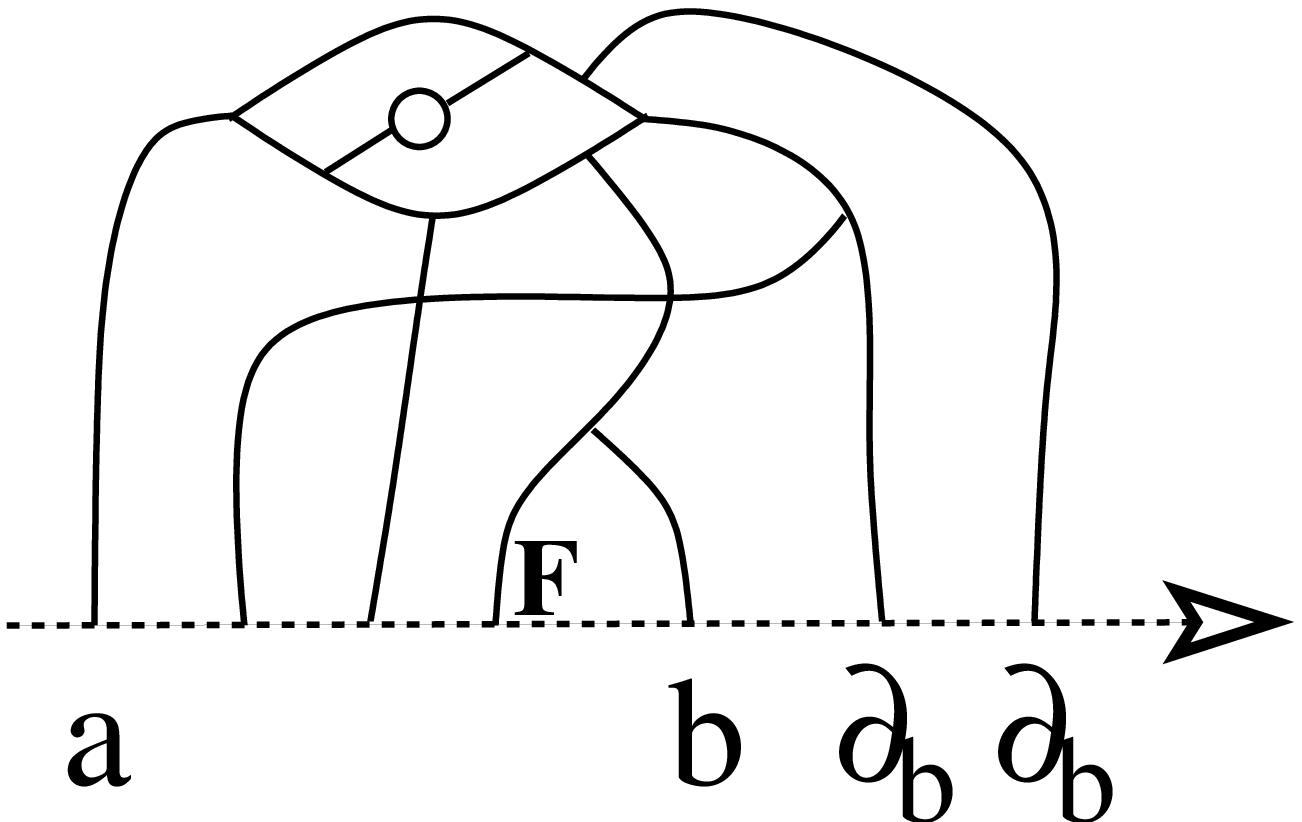}}}\ =\
-\ \raisebox{-4ex}{\scalebox{0.25}{\includegraphics{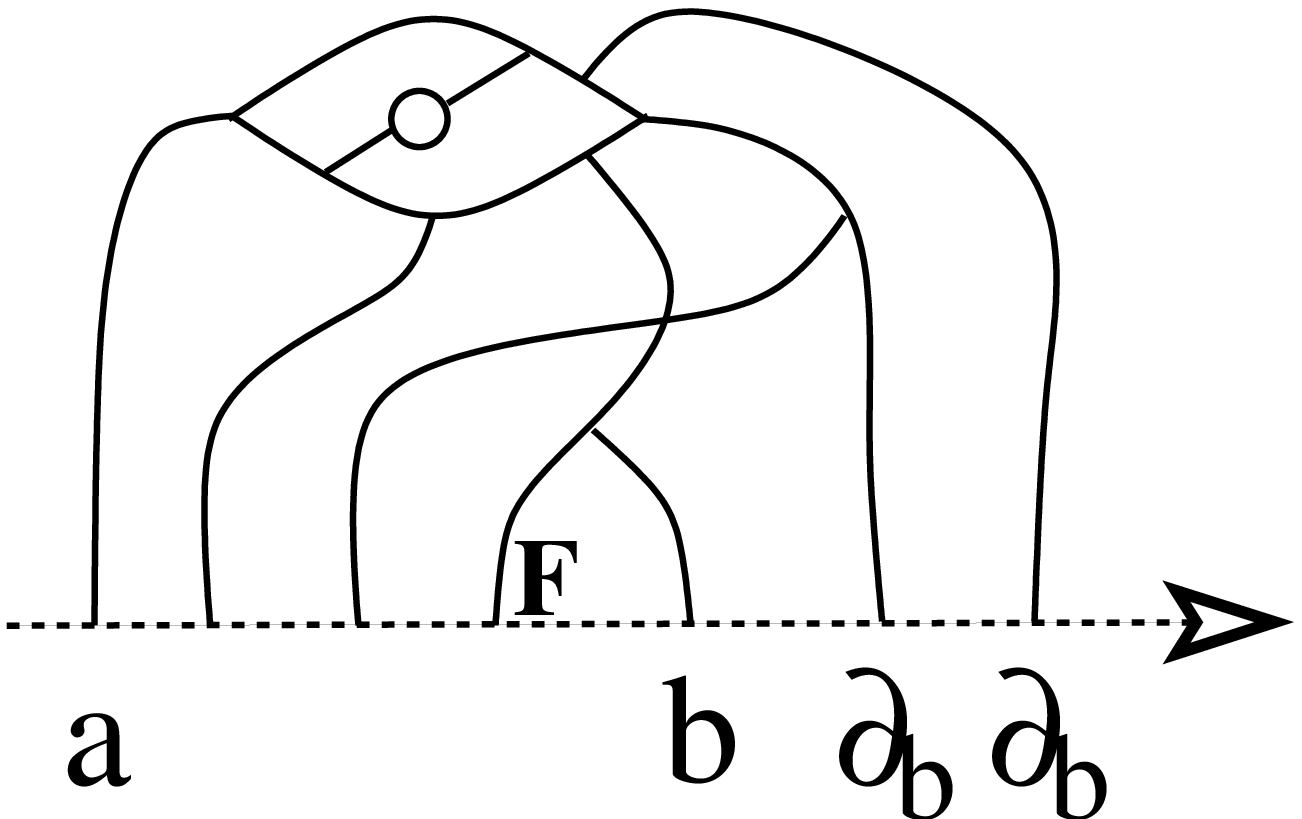}}}\
+\ \raisebox{-4ex}{\scalebox{0.25}{\includegraphics{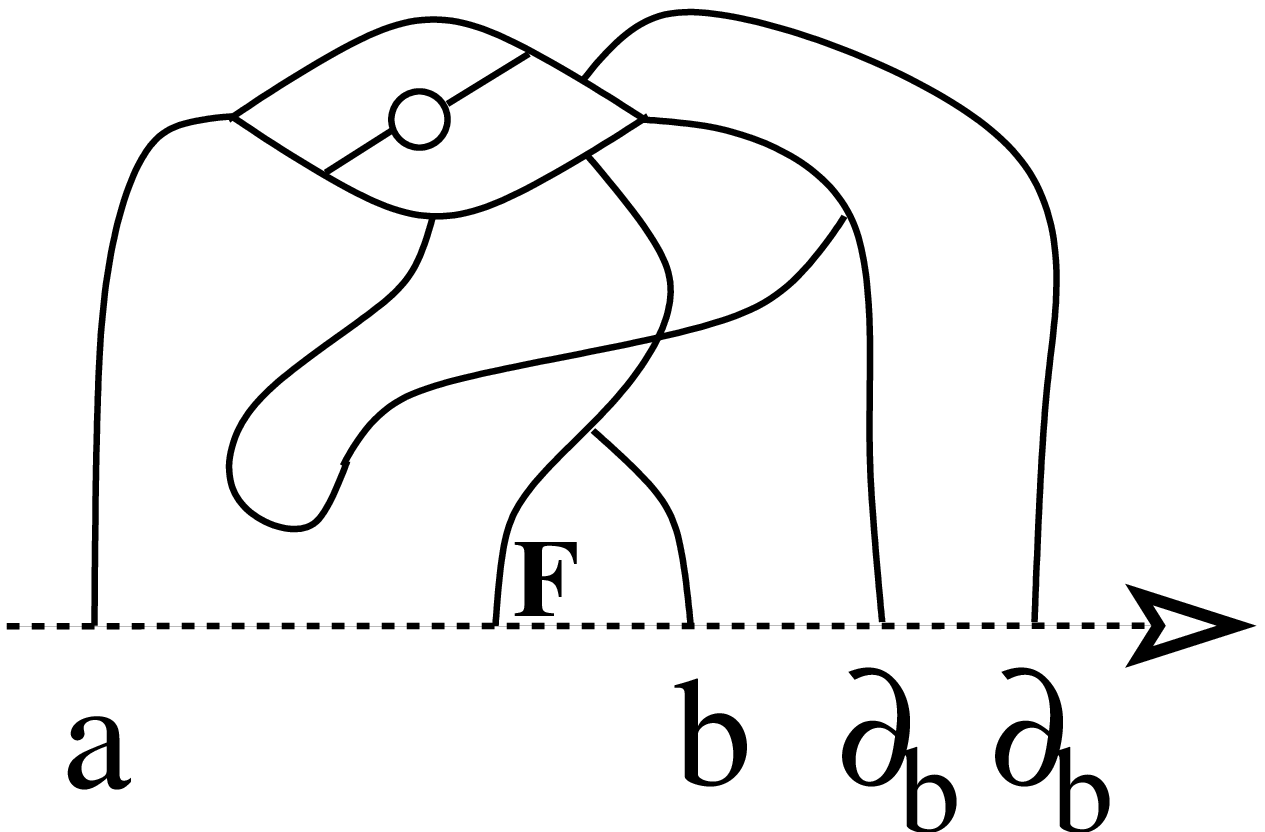}}}\
.\\[0.75cm]
\end{array}
\]
We will work with {\bf power series} of operator Weil diagrams.
Here is what we mean by that:
\begin{defn}
Define the space of formal power series of operator Weil diagrams
in the following way:
\[ \WhatF\abpow= \prod_{(i,j)\in
\mathbb{N}_0\times\mathbb{N}_0}\WhatF\ab^{(i,j)},
\]
where $\mathbb{N}_0$ denotes the set of non-negative integers.
\end{defn}
We suggest taking a moment to decode this: a {\it formal power
series} of operator diagrams is a choice, for every pair $(i,j)$
of non-negative integers, of a vector from $\WhatF\ab^{(i,j)}$.
\subsection{The operator pairing.}
We will now introduce a bilinear pairing on these power series: \[
\apply\,:\, \WhatF\abpow \times' \WhatF\abpow \rightarrow
\WhatF\abpow.
\]
The notation $\times'$ is to record the fact that the pairing is
only defined (only ``{converges}") on certain pairs of power
series:
\[
\WhatF\abpow \times' \WhatF\abpow\subset \WhatF\abpow\times
\WhatF\abpow.
\]
The discussion below requires the projection map
\[
\pi^{(i,j)} : \WhatF\abpow \rightarrow \WhatF\ab^{(i,j)}.
\]
\subsubsection{How to operate with a diagram.}
\label{firstopdefn} The purpose of operator Weil diagrams, of
course, is to have them operate on each other. We'll first define
how individual diagrams operate on each other, and then extend
that action to power series. Consider, then, two operator Weil
diagrams:
\[
\raisebox{-5ex}{\scalebox{0.28}{\includegraphics{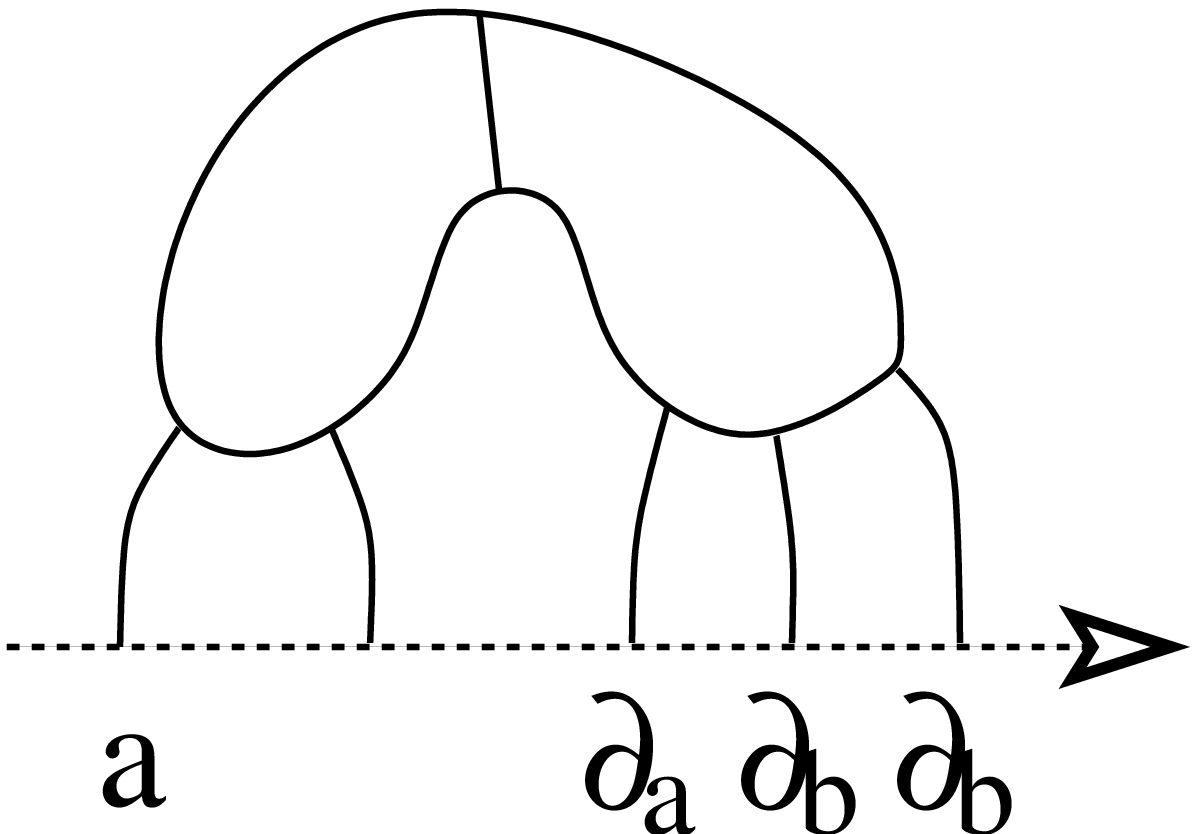}}}\ \
\apply \ \
\raisebox{-5ex}{\scalebox{0.28}{\includegraphics{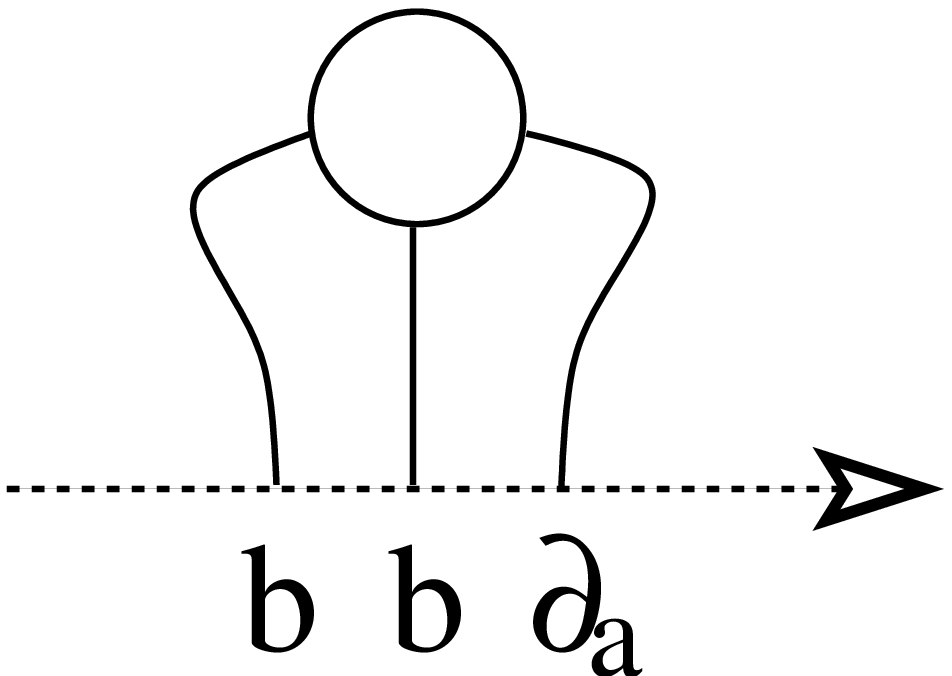}}}\ \
.
\]
To operate with the first on the second, you begin by placing the
two diagrams adjacent to each other on the orienting line.
Then you proceed to push the operator legs to the far-right hand
side of the resulting diagram by using substitution rules which
declare that the operator legs act as graded differential
operators. To be precise: if the operator leg encounters a
parameter leg corresponding to the same parameter, then it
operates on that leg:
\[
\raisebox{-5ex}{\scalebox{0.22}{\includegraphics{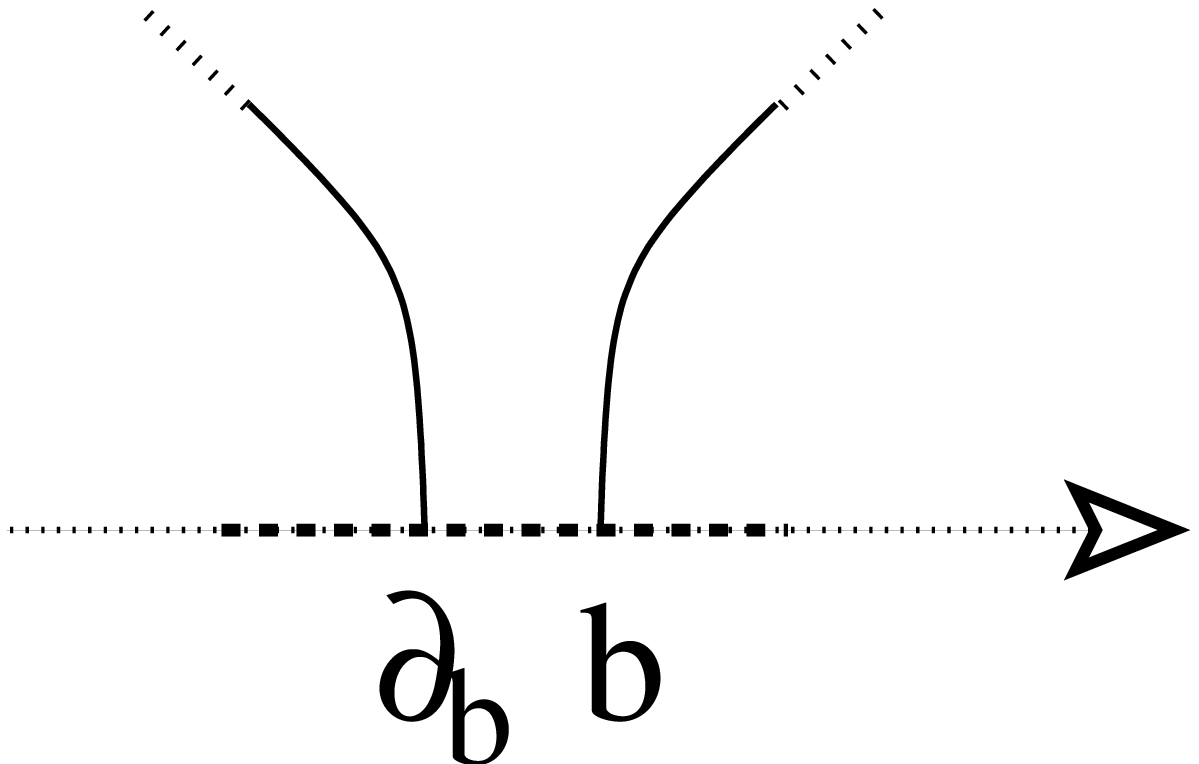}}}\ \
\leadsto\ \ -
\raisebox{-5ex}{\scalebox{0.22}{\includegraphics{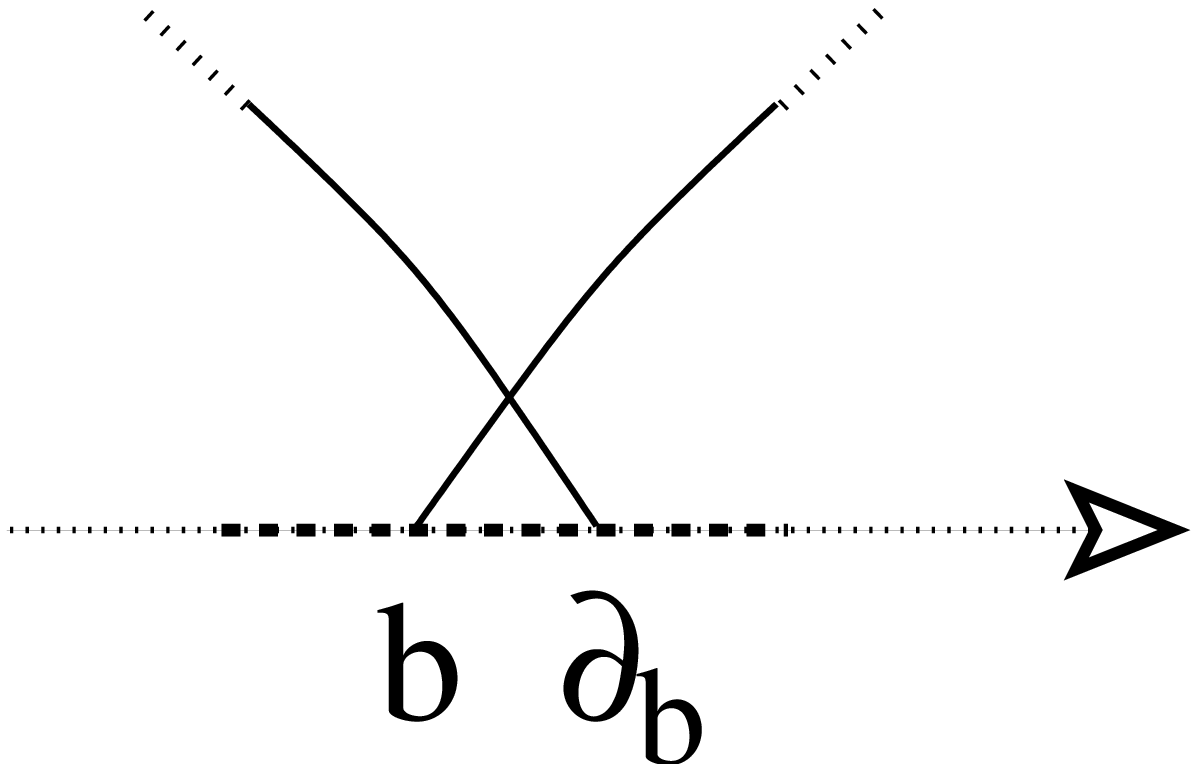}}}\ +\
\raisebox{-2.5ex}{\scalebox{0.22}{\includegraphics{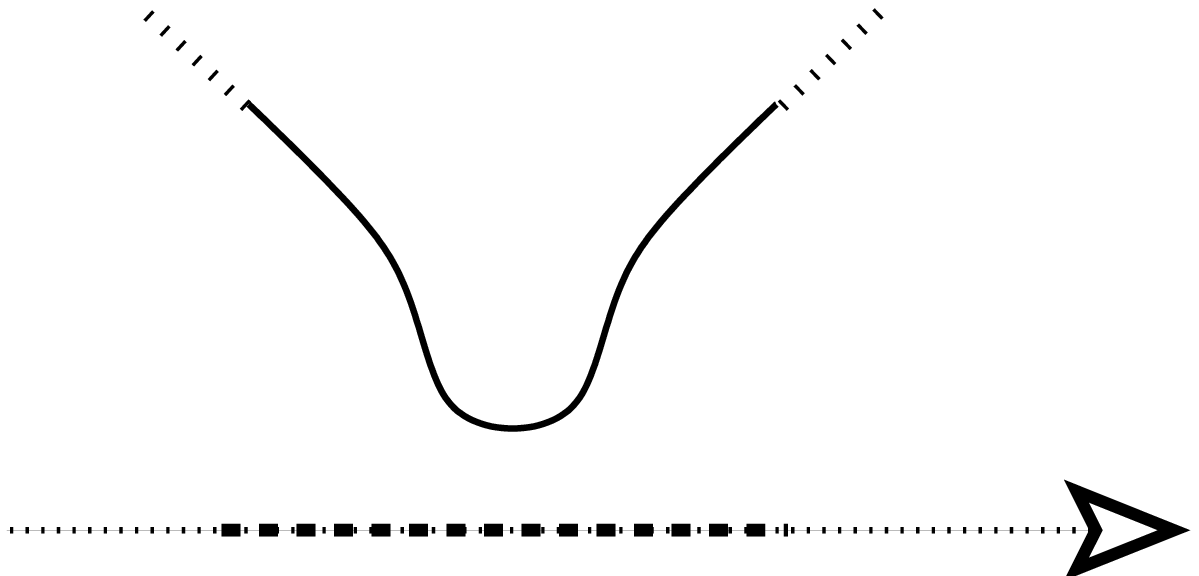}}}\ \
,
\]
or
\[
\raisebox{-5ex}{\scalebox{0.22}{\includegraphics{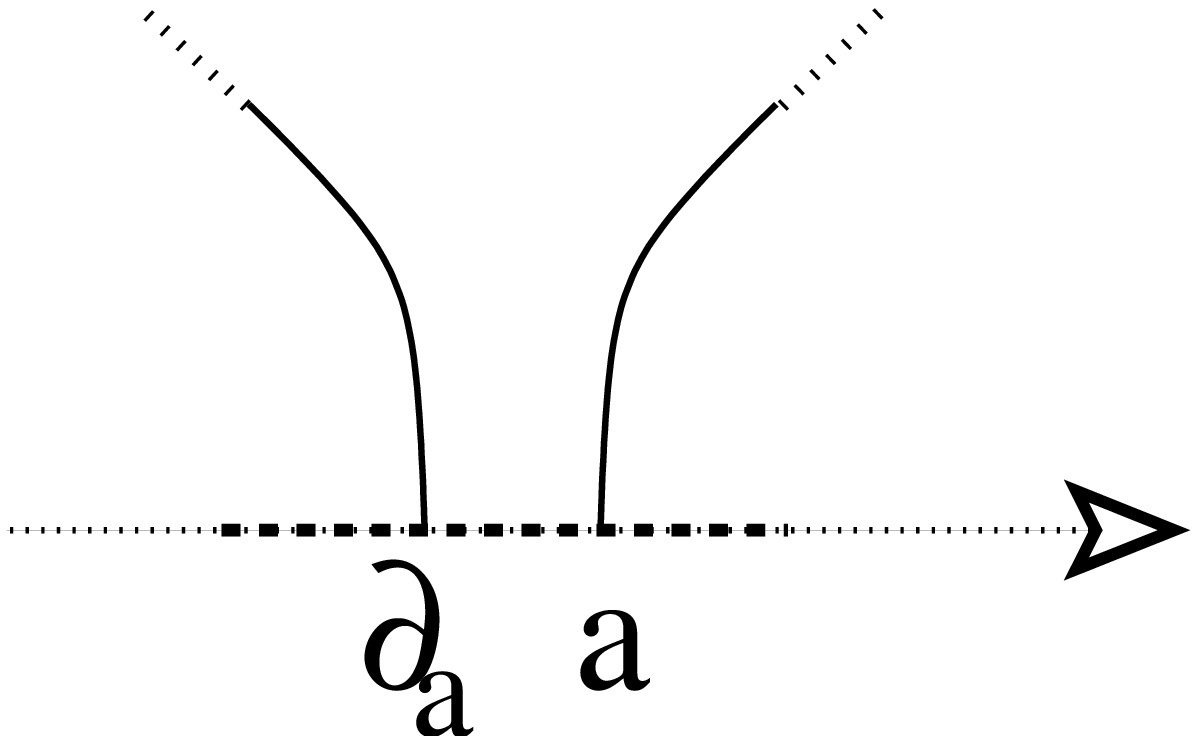}}}\ \
\leadsto\ \ \ \ \
\raisebox{-5ex}{\scalebox{0.22}{\includegraphics{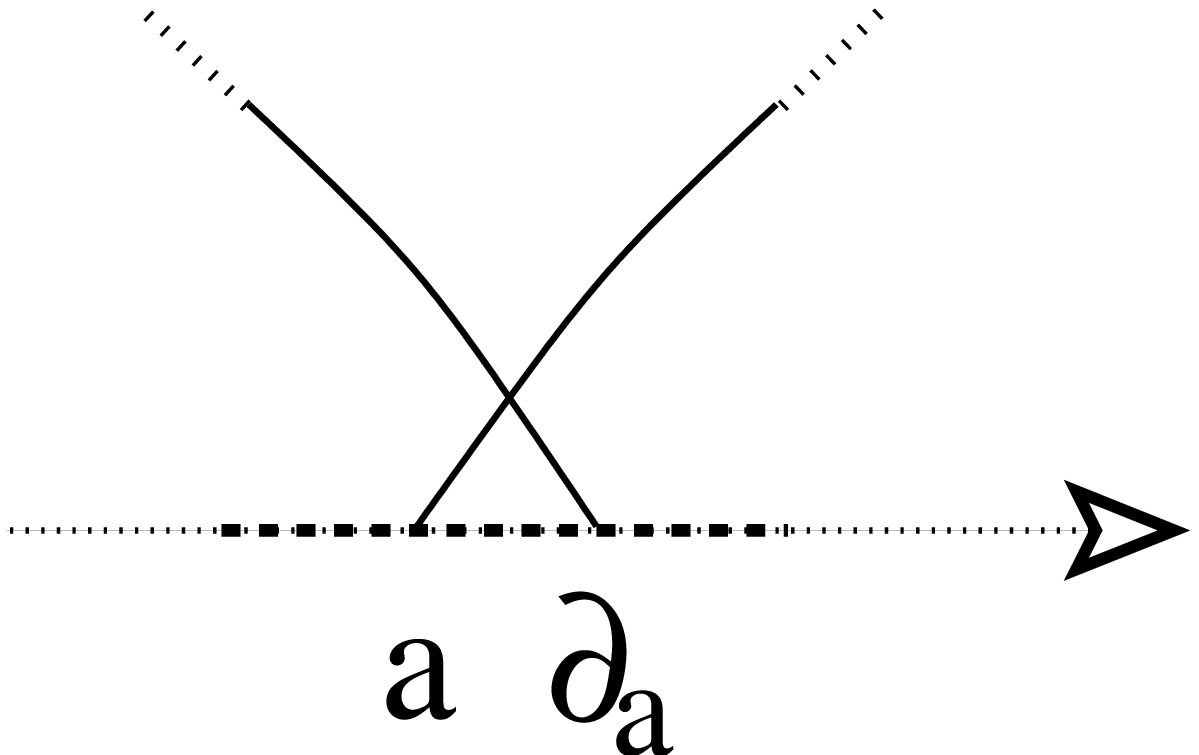}}}\ +\
\raisebox{-2.5ex}{\scalebox{0.22}{\includegraphics{pushpastG}}}\ \
.
\]
If, on the other hand, the operator leg encounters any leg it is
not matched to, then the operator is just pushed past the leg,
incurring the appropriate sign ($(-1)^{g_1g_2}$, where $g_1$ and
$g_2$ are the leg-grades of the two legs involved).
Figure \ref{opcompillustr} illustrates such a computation. The
reader might like to check that they get the final result, which
is contained in Figure \ref{finalopcompresult}.
\begin{figure}
 \caption{\label{opcompillustr}An illustration of one diagram
operating on another.}
\begin{eqnarray*}
 \lefteqn{
\raisebox{-5ex}{\scalebox{0.22}{\includegraphics{compexampA}}}\ \
\apply \ \
\raisebox{-5ex}{\scalebox{0.22}{\includegraphics{compexampB}}} }
\\[0.25cm]
 &
\leadsto &
\raisebox{-4ex}{\scalebox{0.22}{\includegraphics{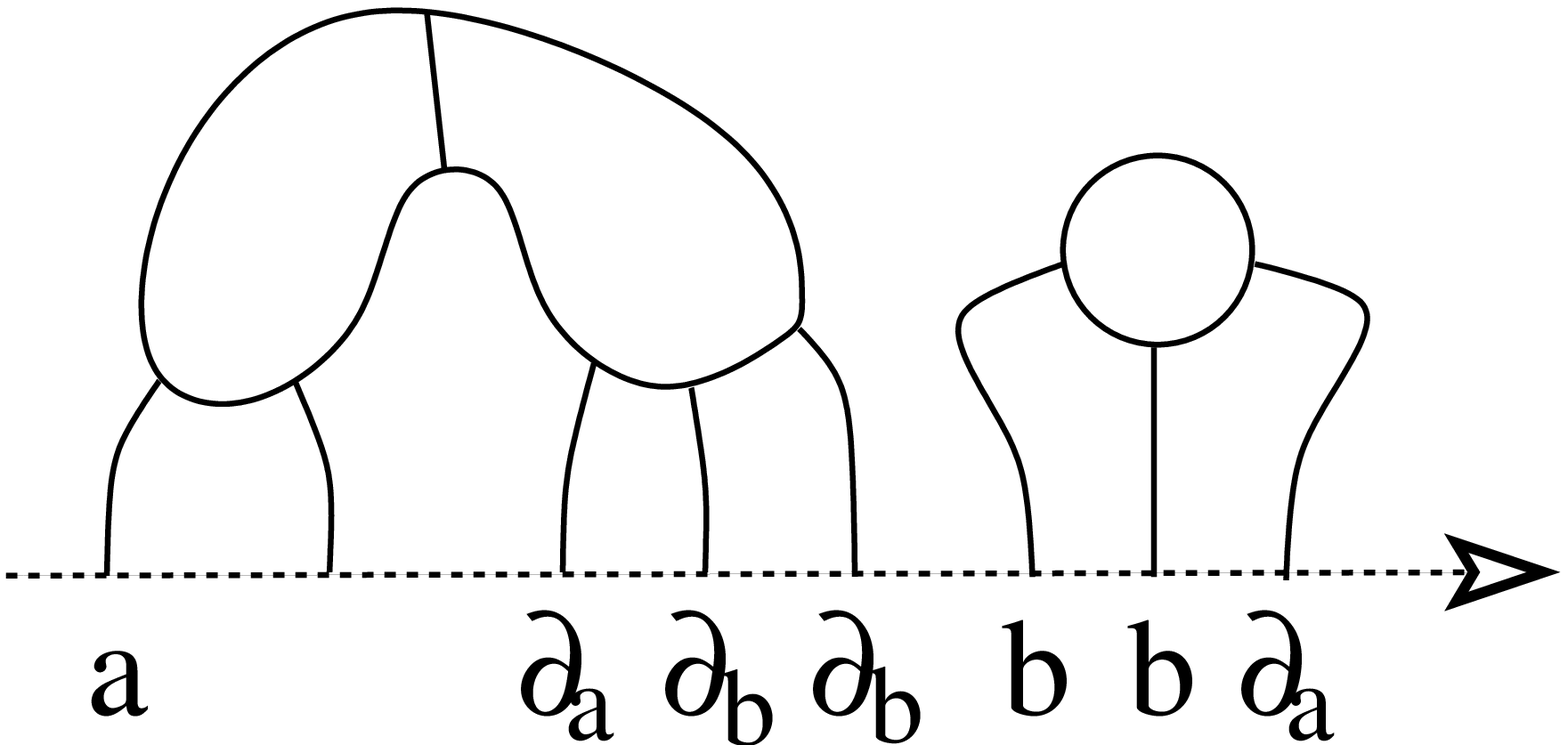}}}
\\[0.25cm]
& \leadsto &
\raisebox{-4ex}{\scalebox{0.22}{\includegraphics{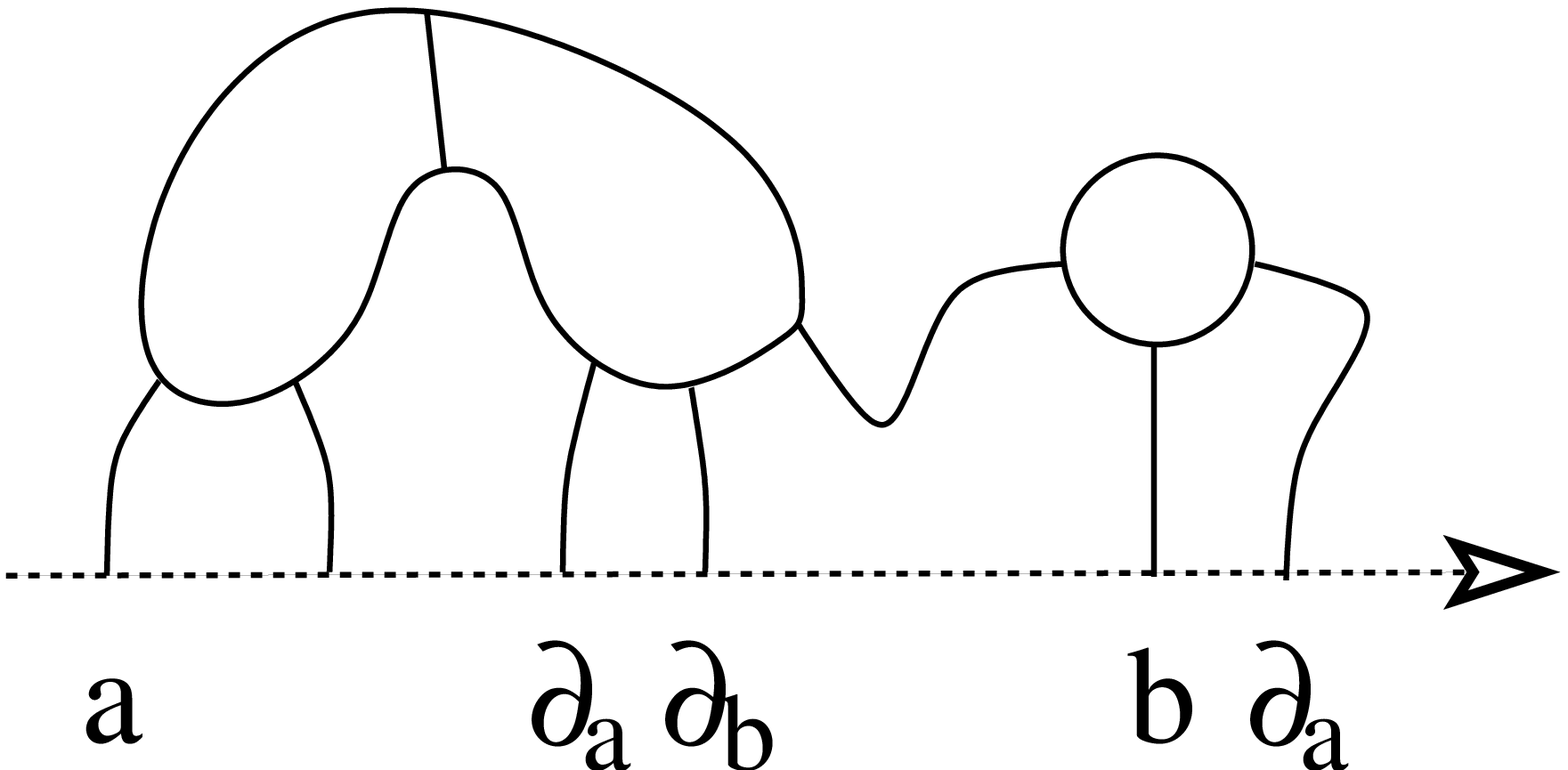}}}\ -\
\raisebox{-4ex}{\scalebox{0.22}{\includegraphics{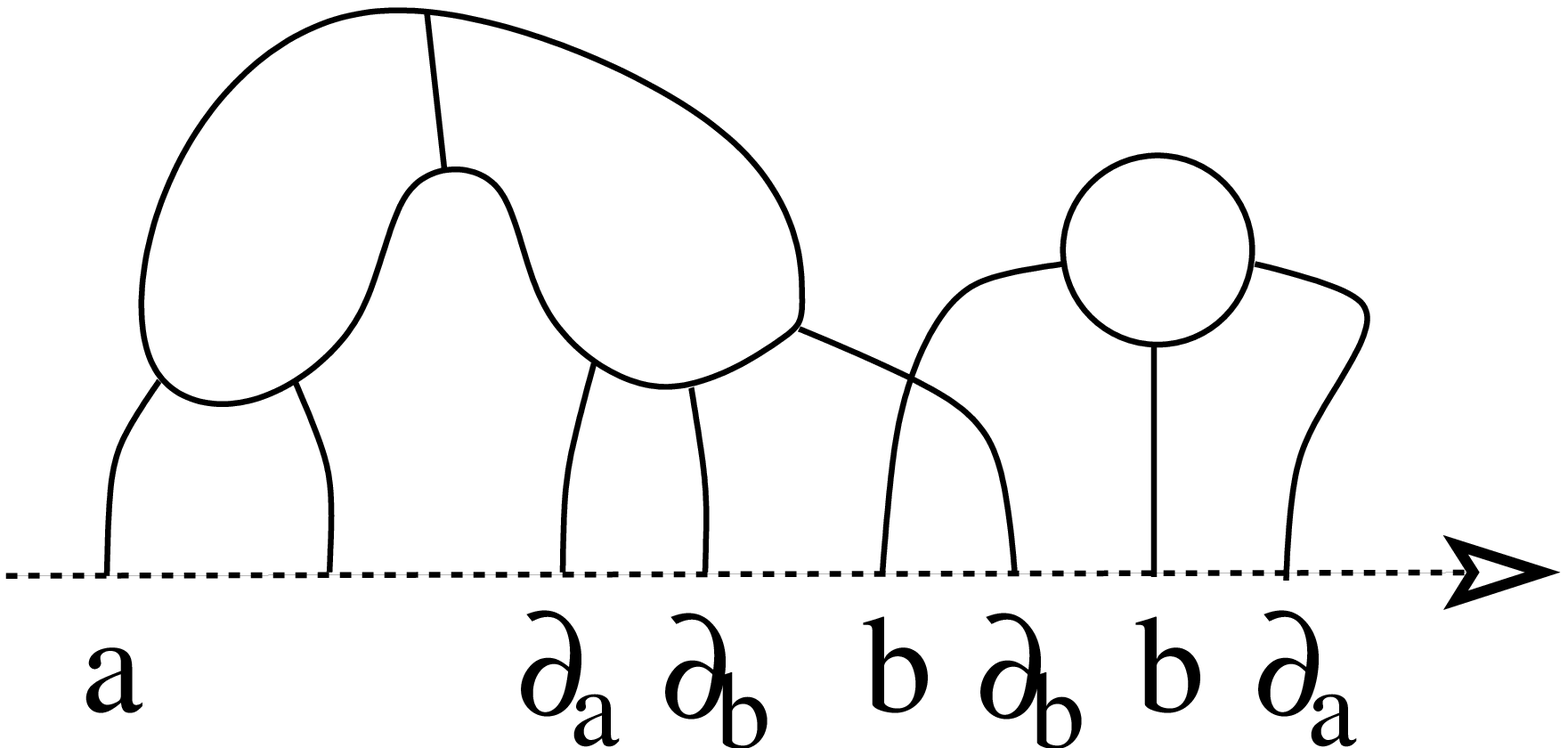}}}
\\[0.25cm]
& \leadsto &
\raisebox{-4ex}{\scalebox{0.22}{\includegraphics{compexampD}}}\ -\
\raisebox{-4ex}{\scalebox{0.22}{\includegraphics{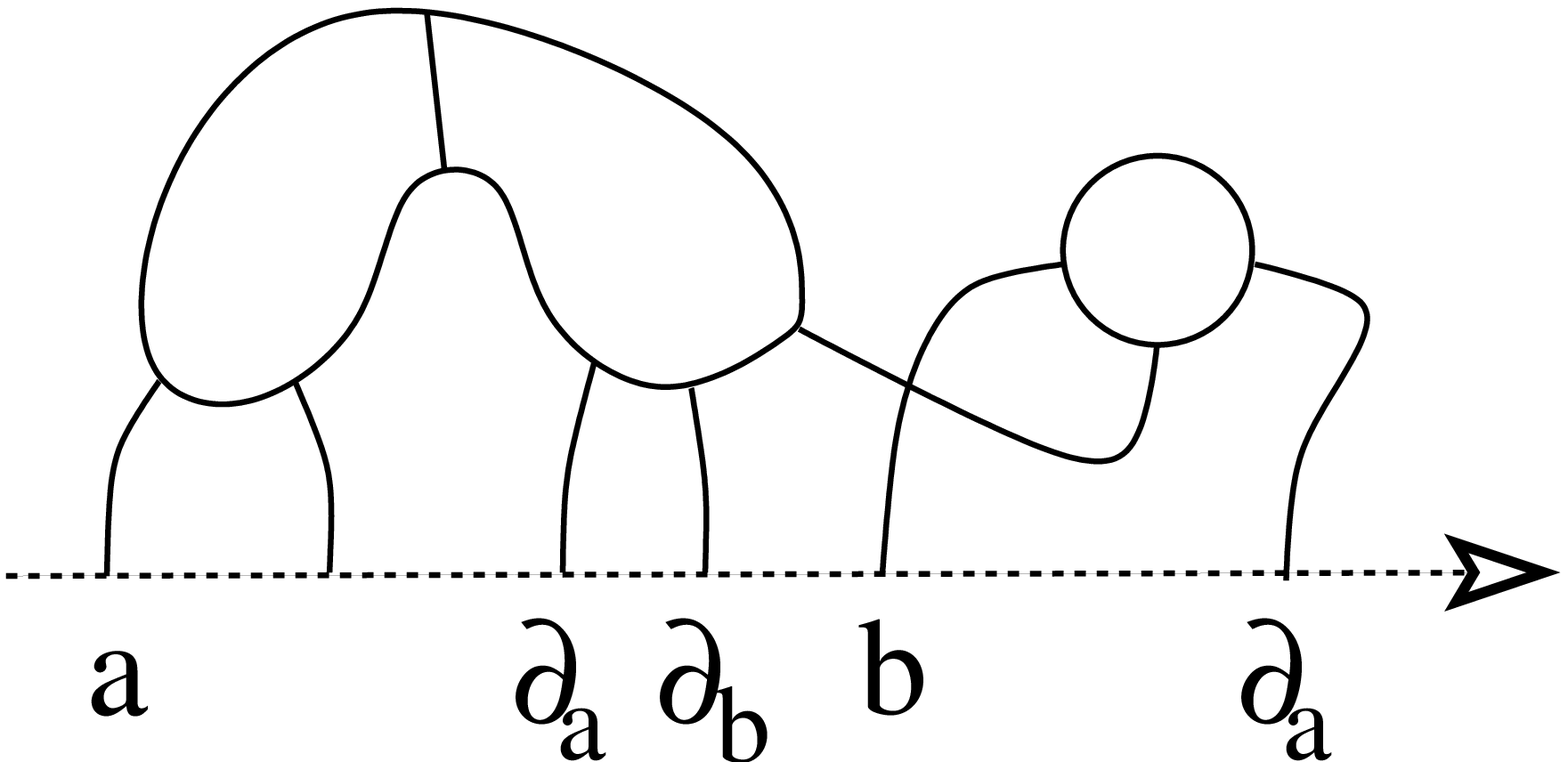}}}\\[0.25cm]
& &  +\
\raisebox{-4ex}{\scalebox{0.22}{\includegraphics{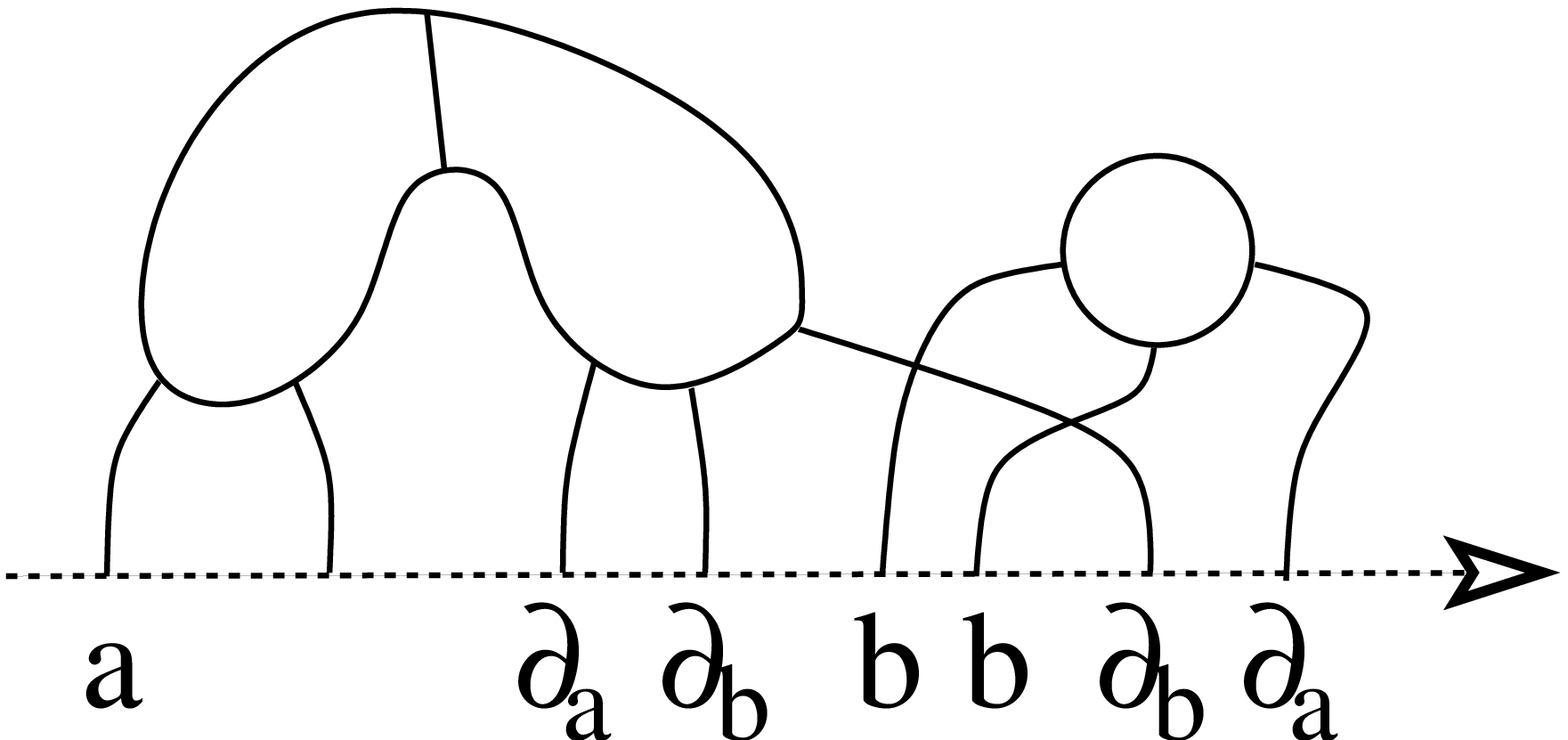}}}\ \
\ ,\ \ \ \ \text{etc.}
\end{eqnarray*}
\underline{\hspace{7cm}}
\end{figure}

\begin{figure}
\caption{\label{finalopcompresult}The final result of the diagram
operation started in Figure \ref{opcompillustr}.}
\[
\begin{array}{l}
\ \ \ \ \
\raisebox{-4ex}{\scalebox{0.22}{\includegraphics{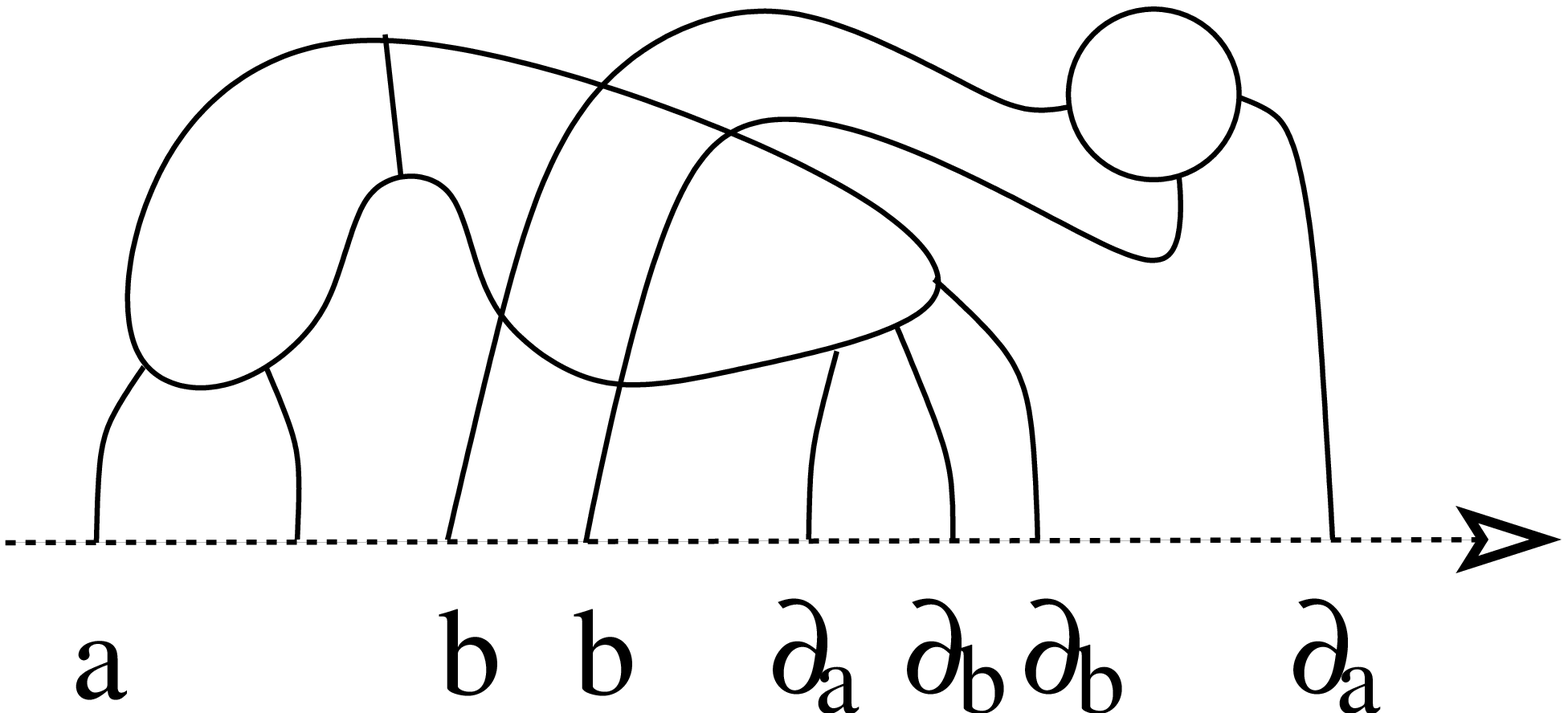}}}\ \
+\ \
\raisebox{-4ex}{\scalebox{0.22}{\includegraphics{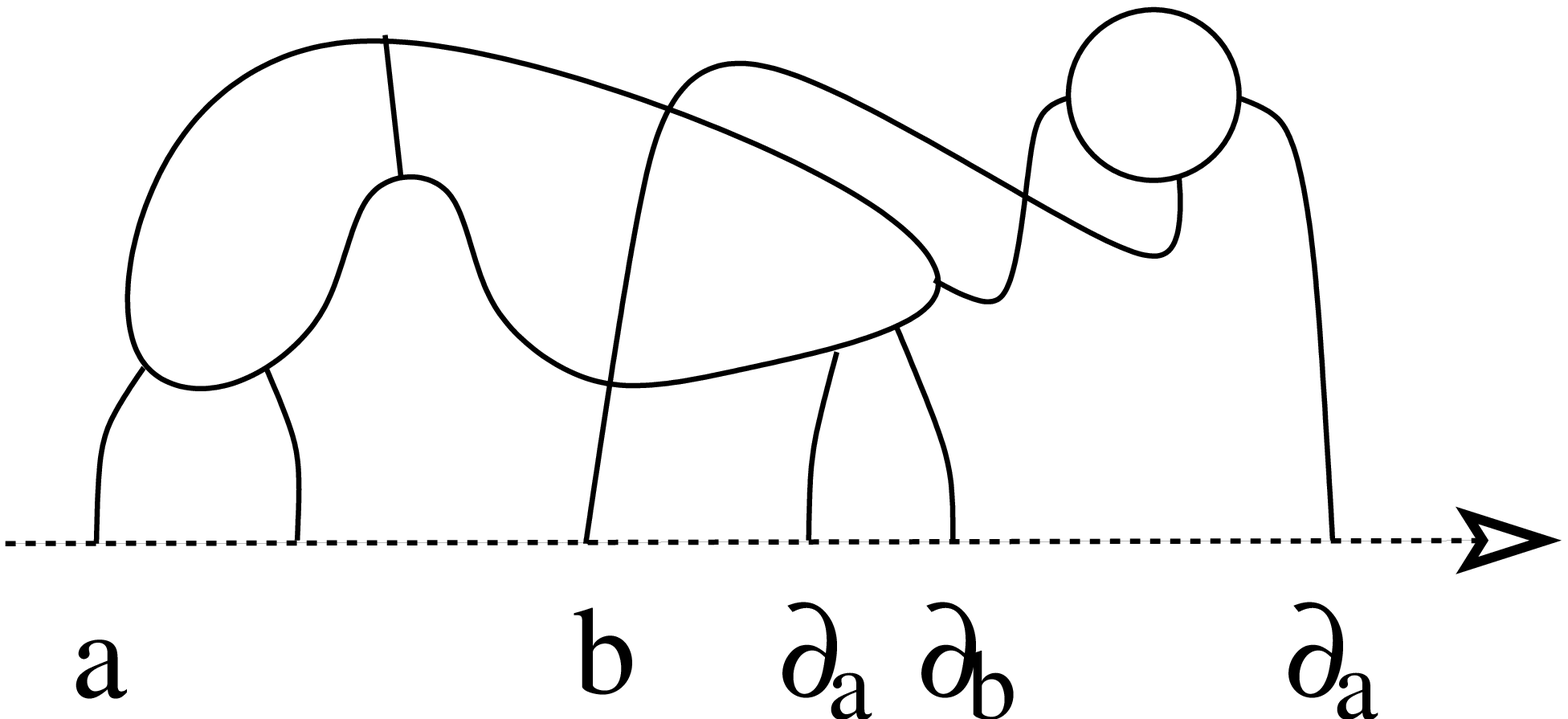}}}
\\[0.75cm]
\ \ -\
\raisebox{-4ex}{\scalebox{0.22}{\includegraphics{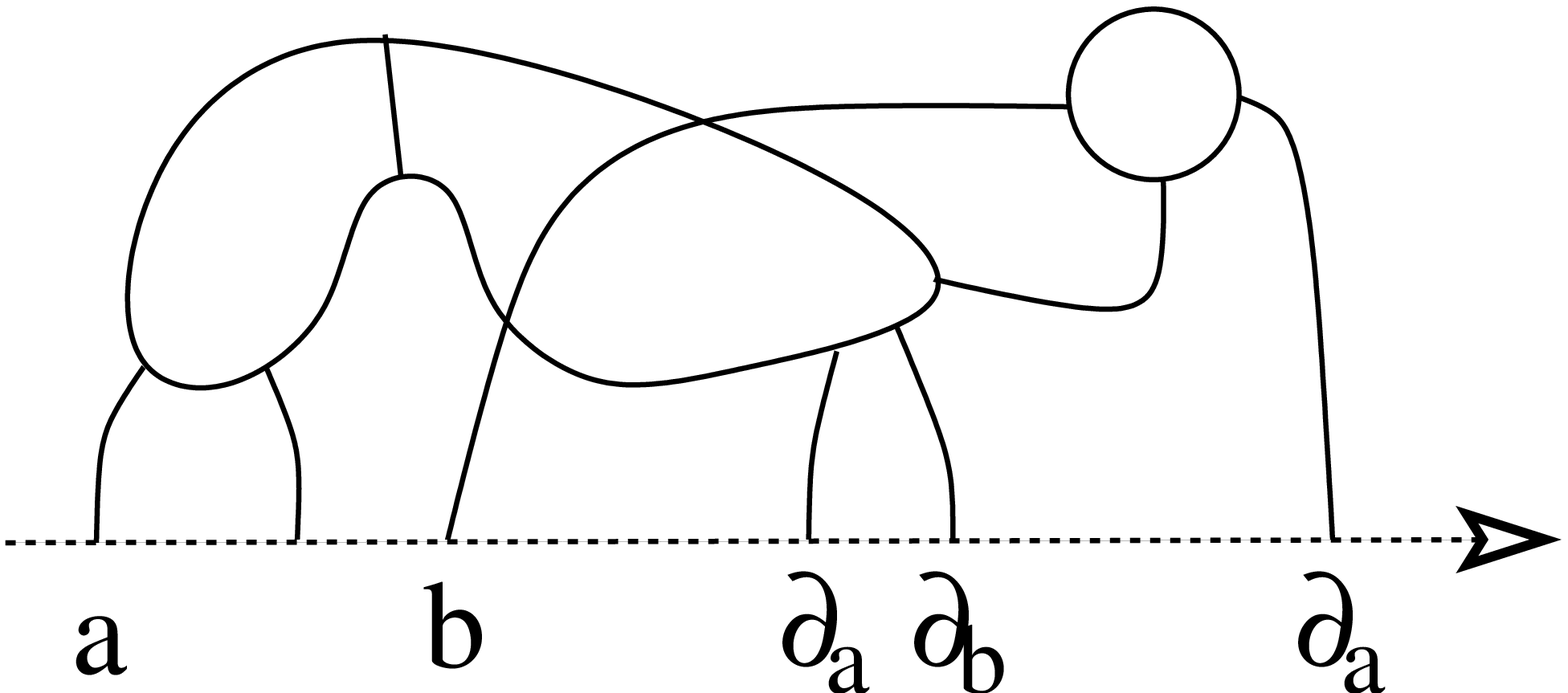}}} \
\ -\ \
\raisebox{-4ex}{\scalebox{0.22}{\includegraphics{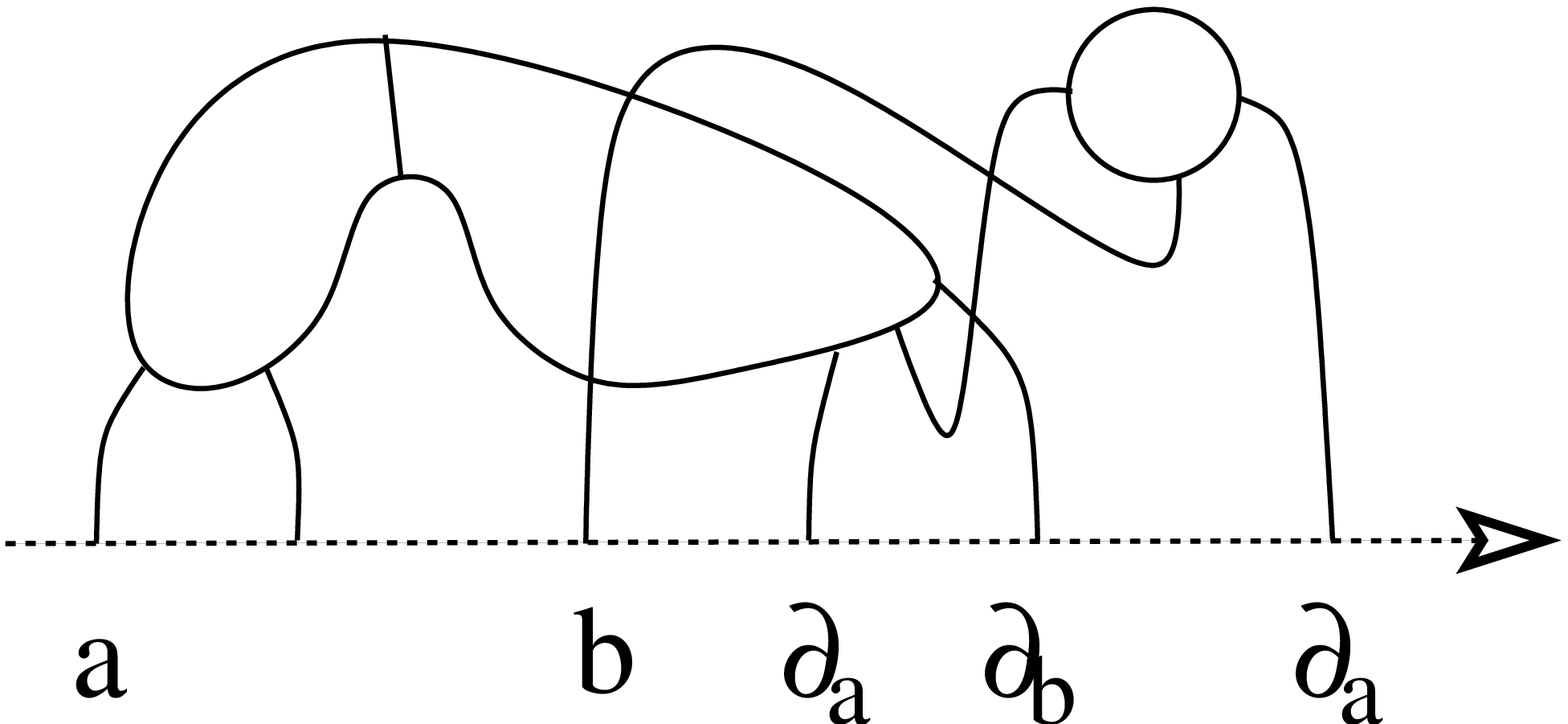}}}
\\[0.75cm] \ \ +\
\raisebox{-4ex}{\scalebox{0.22}{\includegraphics{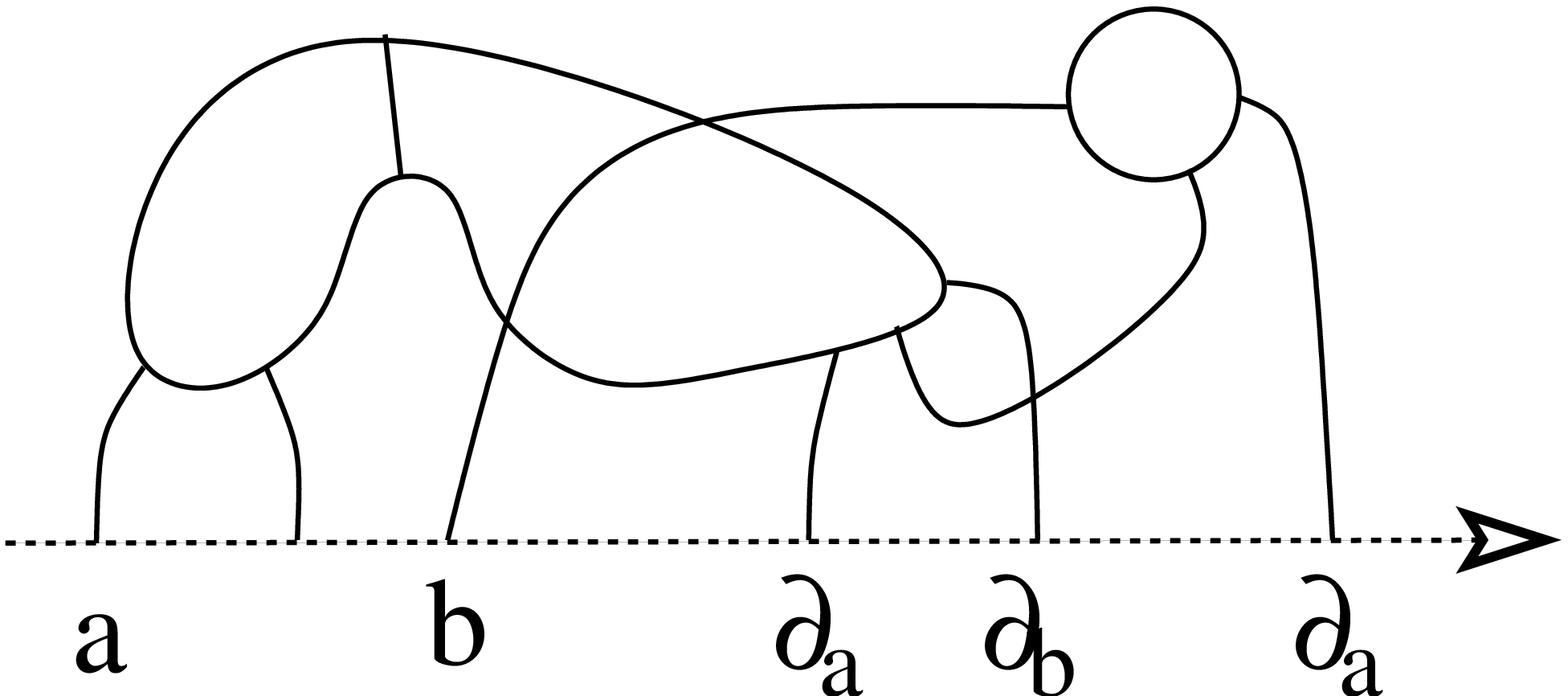}}} \
\ +\ \
\raisebox{-4ex}{\scalebox{0.22}{\includegraphics{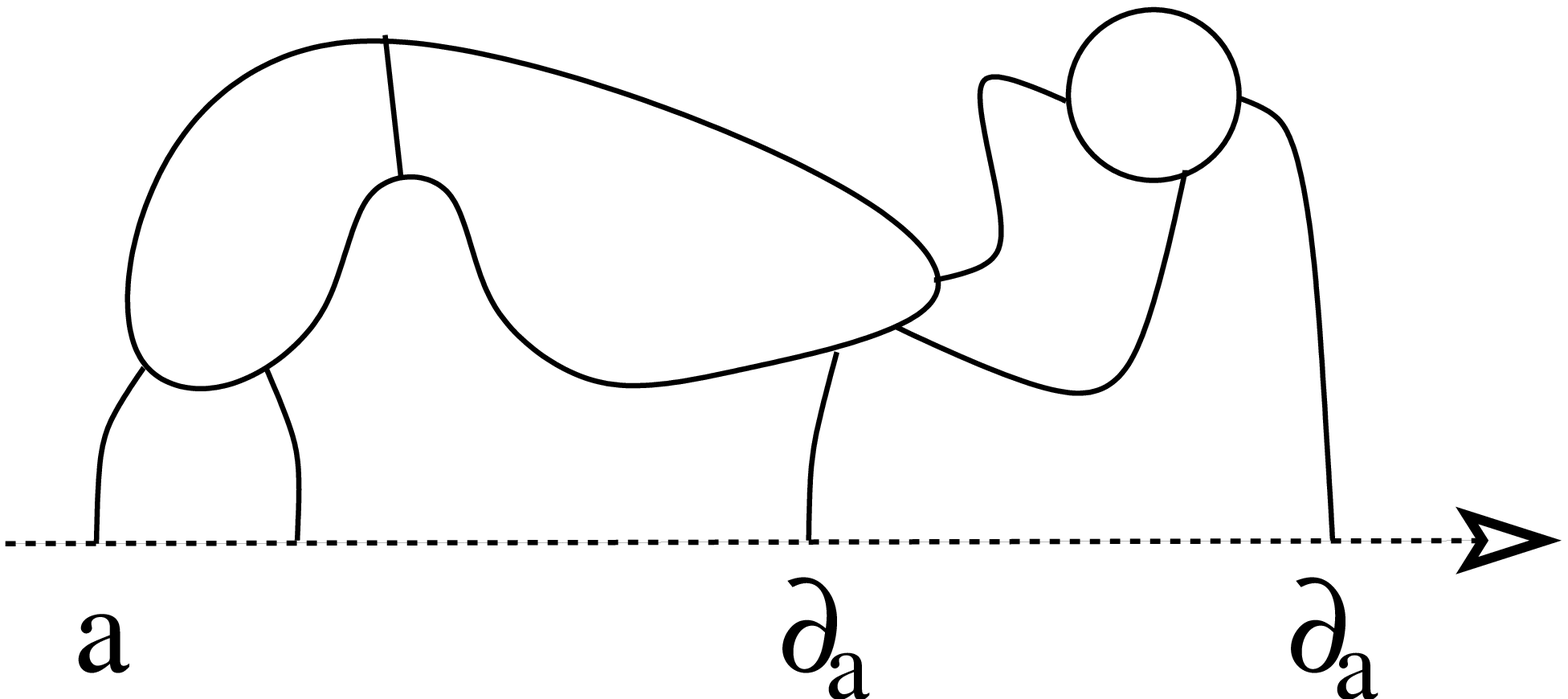}}}
\\[0.75cm]
\ \ -\
\raisebox{-4ex}{\scalebox{0.22}{\includegraphics{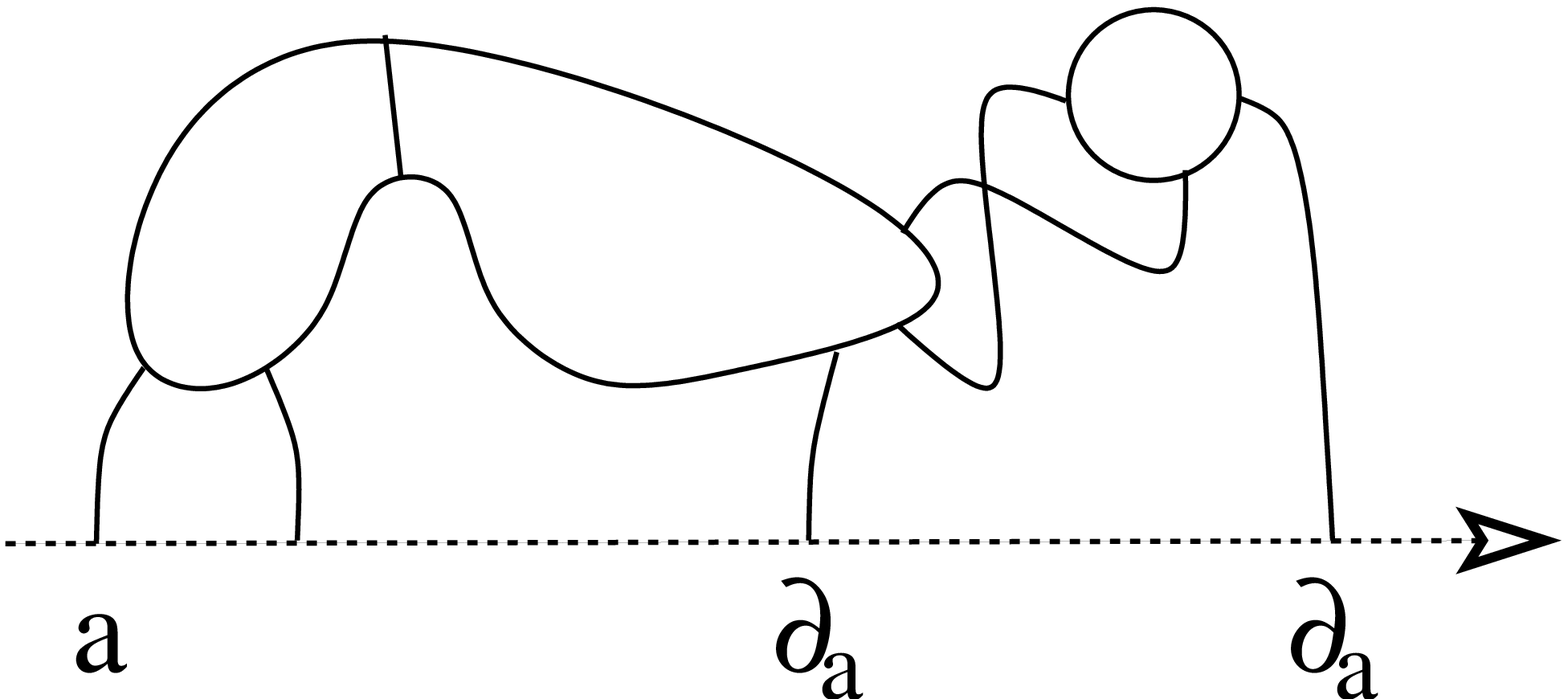}}}\ \
.
\end{array}
\]
\underline{\hspace{7cm}}
\end{figure}

\subsubsection{Checking relations.}
Having defined the operation $\apply$ on the level of diagrams,
let us proceed to verify that it respects the relations that we
introduce amongst those diagrams.
\begin{prop}
Let $i$, $j$, $k$ and $l$ be elements of $\mathbb{N}_0$. The
linear extension of the above definition of\, $\apply$ gives a
well-defined bilinear map
\[
\apply\  :\ \WhatF\ab^{(i,j)}\times \WhatF\ab^{(k,l)}\rightarrow
\WhatF\abpow\ . \]
\end{prop}
\begin{proof}
 We require that the relations that define
$\WhatF\ab^{(i,j)}$ and $\WhatF\ab^{(k,l)}$ are sent to zero in
$\WhatF\abpow$. The sort of thing that we must show is that
\[
\raisebox{-6ex}{\scalebox{0.24}{\includegraphics{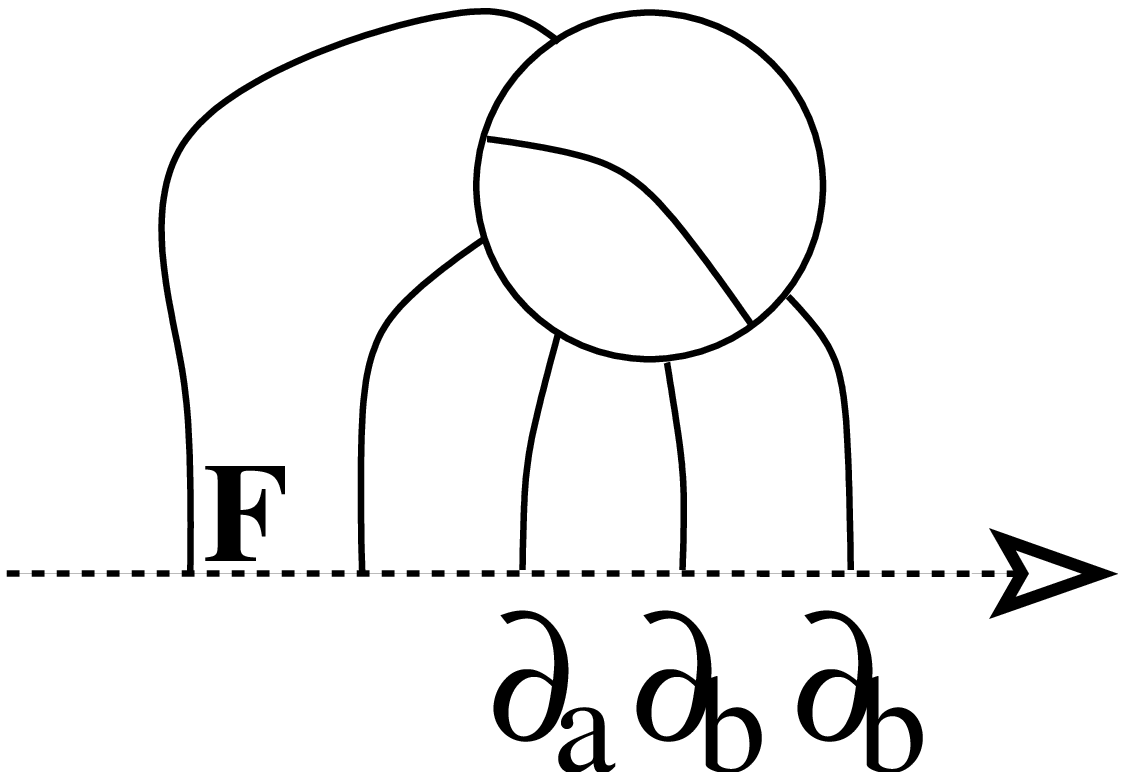}}}\apply
\left(
\raisebox{-6ex}{\scalebox{0.24}{\includegraphics{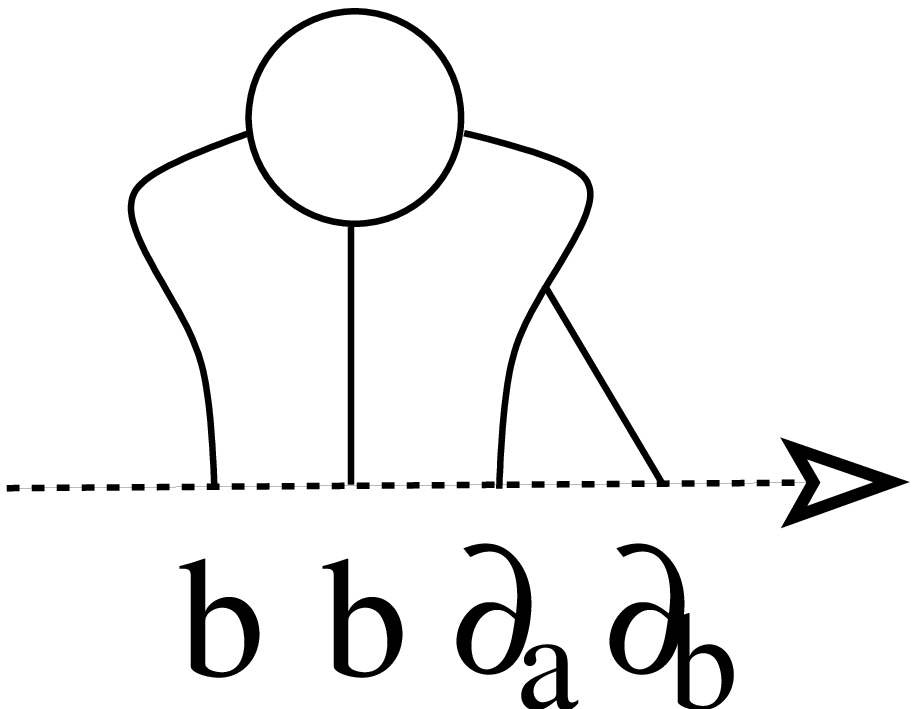}}}\
+\
\raisebox{-6ex}{\scalebox{0.24}{\includegraphics{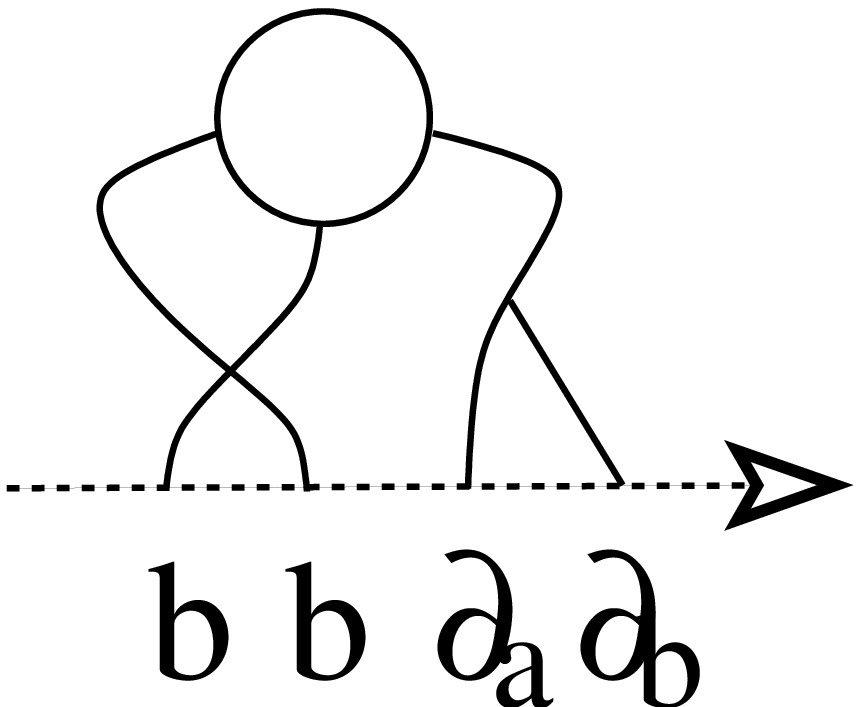}}}\right)\
=\ 0.
\]
In this discussion we'll use a box to represent the part of the
diagram that varies in a relation vector. For example, the
equation above will be represented schematically as follows: \[
\raisebox{-6ex}{\scalebox{0.24}{\includegraphics{comprelncheckA}}}\
\apply \
\raisebox{-6ex}{\scalebox{0.24}{\includegraphics{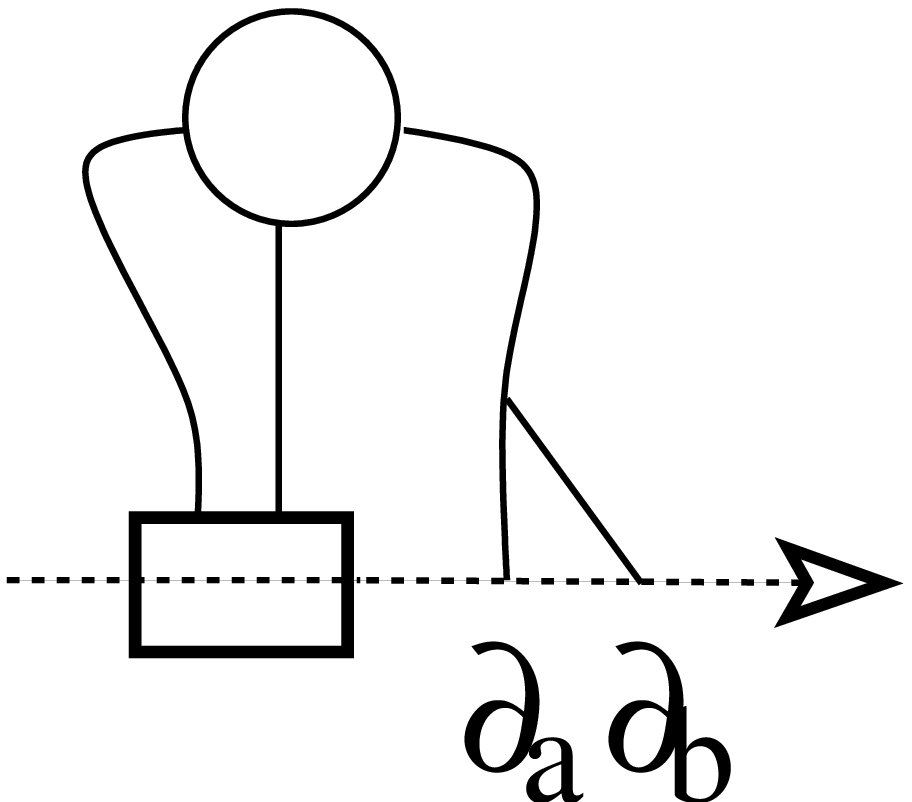}}}\
=\ 0 .
\]
There are four classes of relations that must be checked:
\begin{enumerate}
\item{ Amongst the operator legs of $\WhatF\ab^{(i,j)}$. For
example:
\[ \raisebox{-4ex}{\scalebox{0.21}{\includegraphics{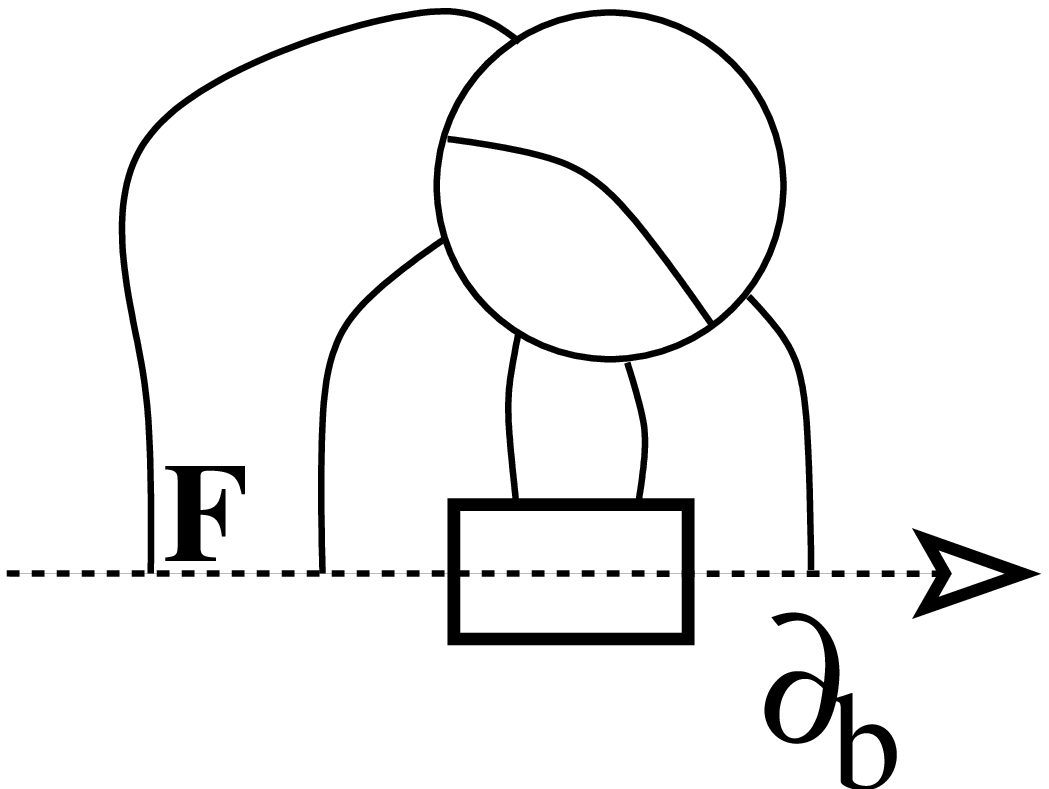}}}\
\apply \
\raisebox{-4ex}{\scalebox{0.21}{\includegraphics{comprelncheckB}}}\
=\ 0 .
\]} \item{Amongst the non-operator legs of
$\WhatF\ab^{(i,j)}$. For example: \[
\raisebox{-4ex}{\scalebox{0.21}{\includegraphics{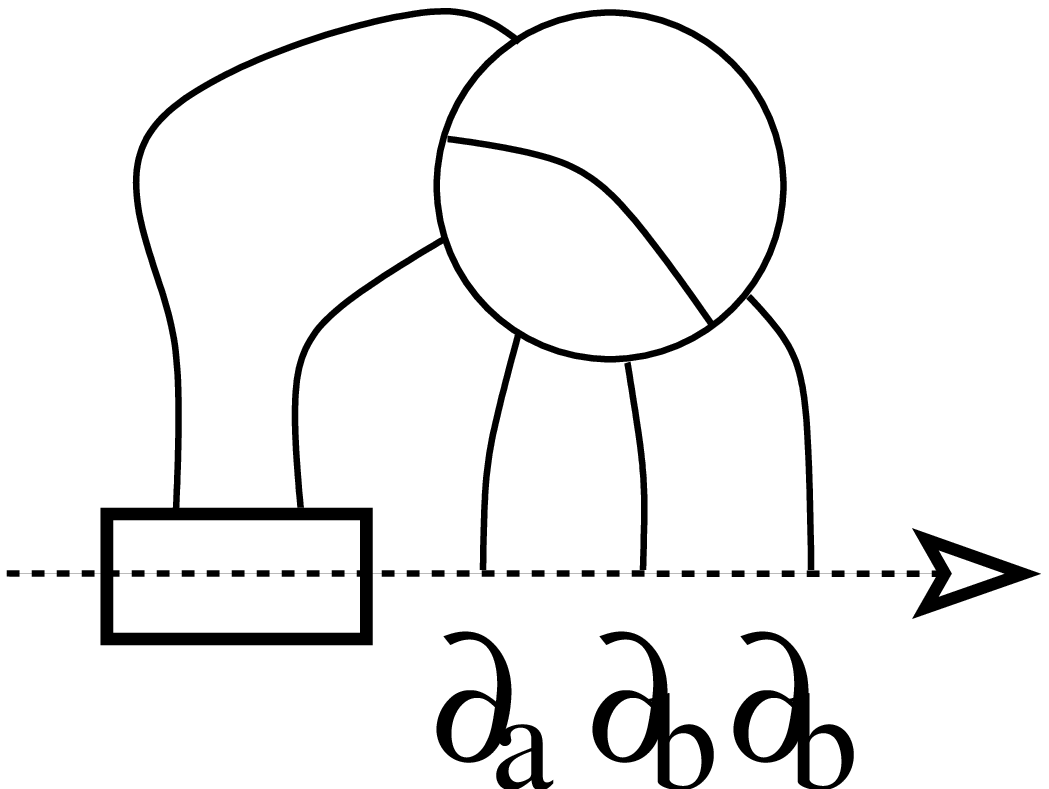}}}\ \apply \
\raisebox{-4ex}{\scalebox{0.21}{\includegraphics{comprelncheckB}}}\
=\ 0 .
\]}
\item{Amongst the operator legs of $\WhatF\ab^{(i,j)}$. For
example:
\[
\raisebox{-4ex}{\scalebox{0.21}{\includegraphics{comprelncheckA}}}\
\apply \ \raisebox{-4ex}{\scalebox{0.21}{\includegraphics{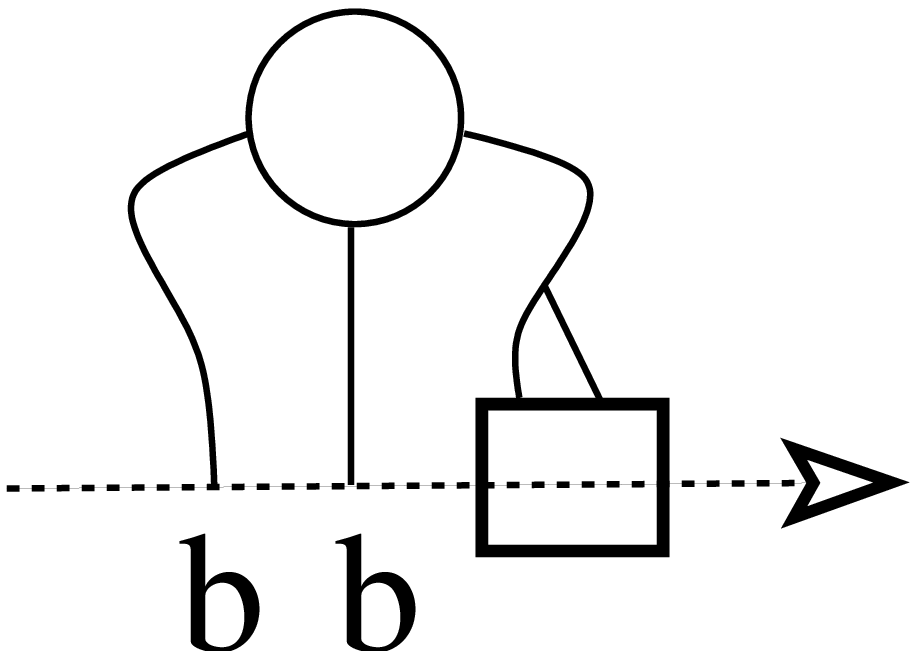}}}\
=\ 0 .
\]
} \item{Amongst the other legs of $\WhatF\ab^{(k,l)}$. For
example:
\[
\raisebox{-4ex}{\scalebox{0.21}{\includegraphics{comprelncheckA}}}\
\apply \
\raisebox{-4ex}{\scalebox{0.21}{\includegraphics{comprelncheckBP}}}\
=\ 0 .
\]
 }
 \end{enumerate}
Classes 2 and 3 are obviously sent to zero because they play no
part in the calculation. We will restrict ourselves to checking
Class 4 as these relations will play a role in the subsequent
calculation. Class 1 is also straightforward. To show why such relations are respected we'll consider a
specific example. We must show that \[
\raisebox{-4ex}{\scalebox{0.25}{\includegraphics{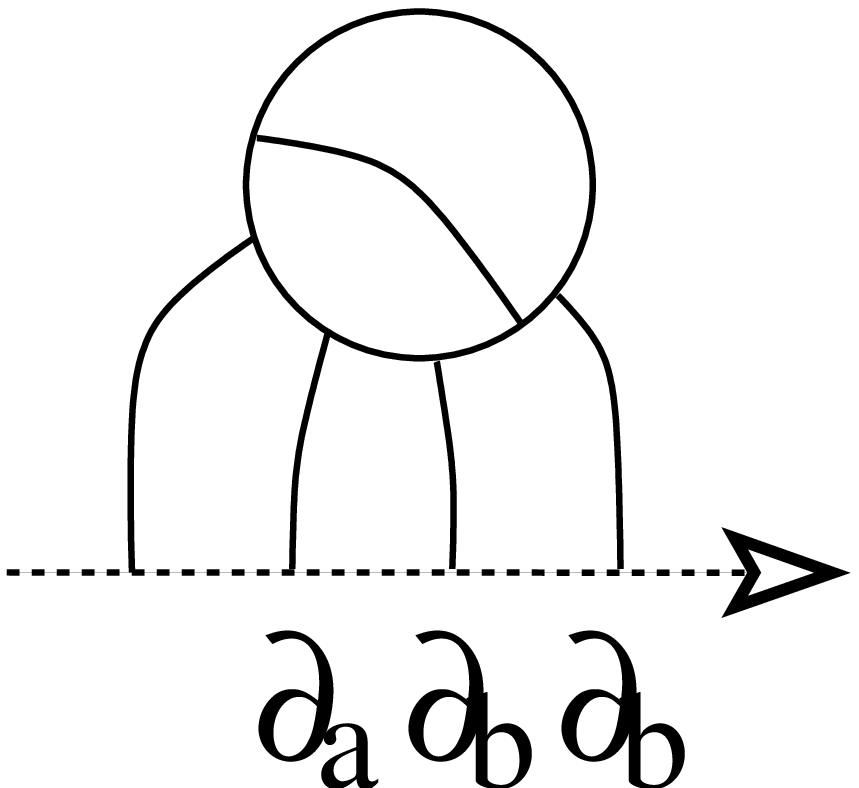}}}\
\apply \
\raisebox{-4ex}{\scalebox{0.25}{\includegraphics{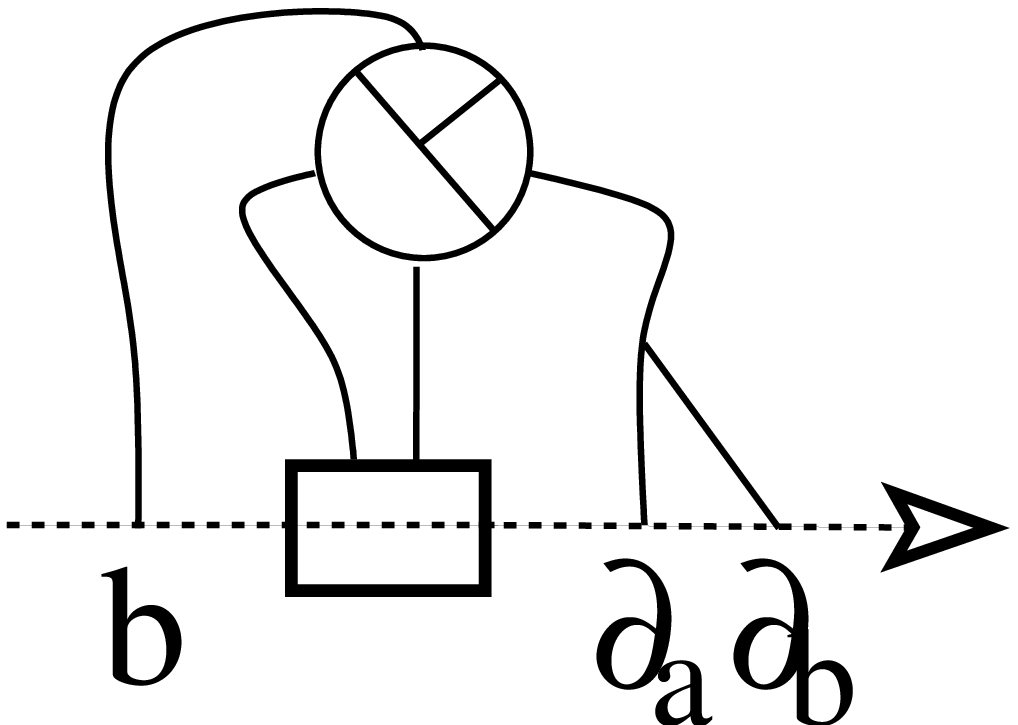}}}\ =\
0,
\]
where the box can represent any one of ten possible relations. It
might be one of the usual relations for
$\widehat{\mathcal{W}}_{\mathrm{F}}$:
\[
\begin{array}{ccccccc}
\raisebox{-2.2ex}{\scalebox{0.22}{\includegraphics{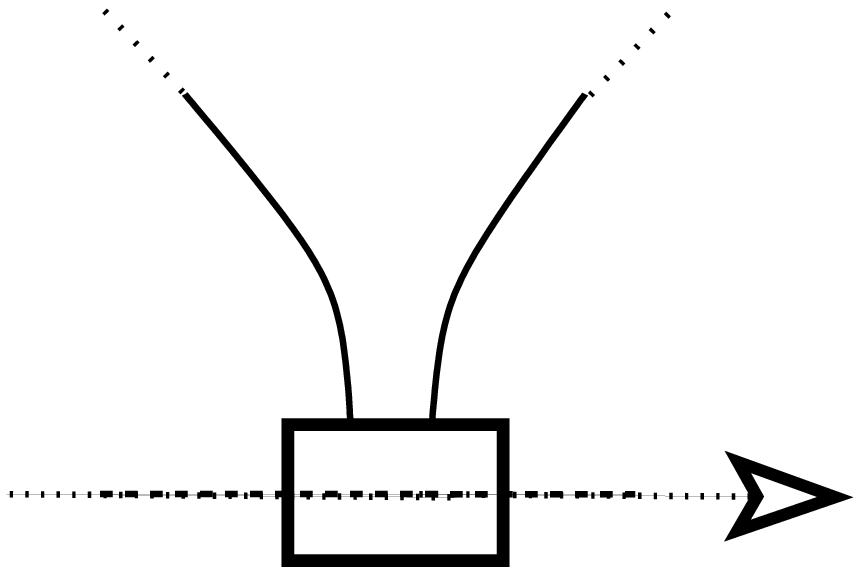}}} & = &
\raisebox{-2ex}{\scalebox{0.22}{\includegraphics{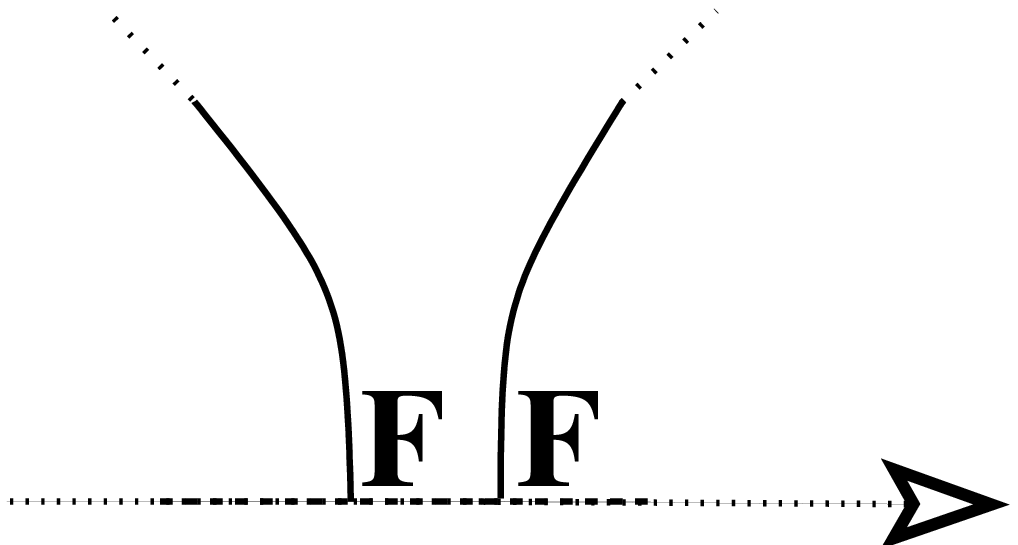}}} & - &
\raisebox{-2ex}{\scalebox{0.22}{\includegraphics{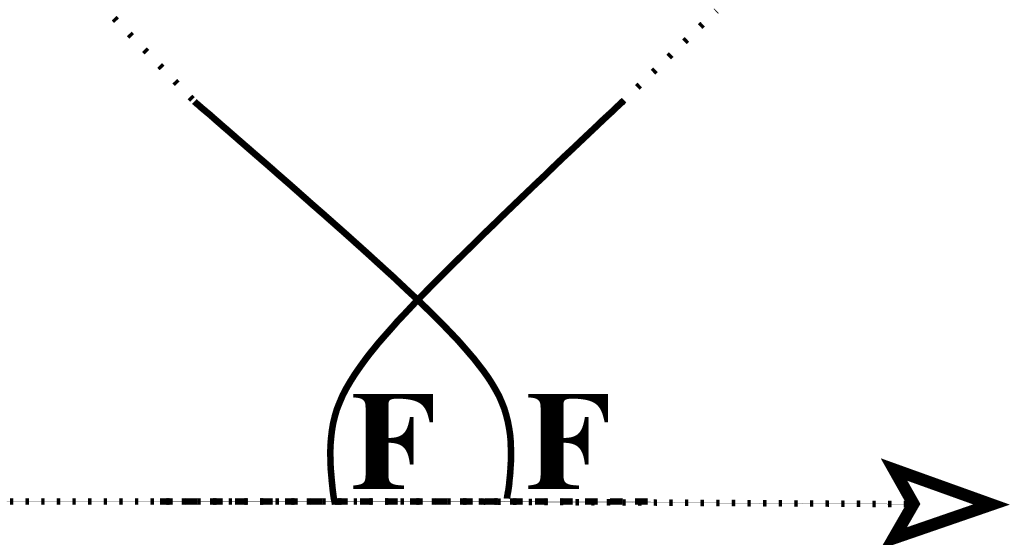}}} & - &
\raisebox{-2ex}{\scalebox{0.22}{\includegraphics{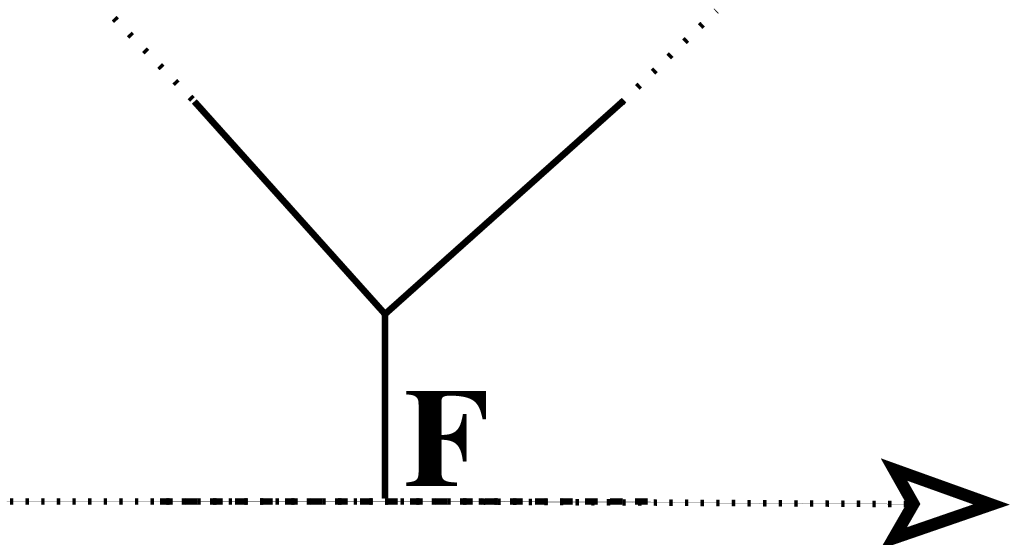}}}\,,
\\[0.5cm]
\raisebox{-2.2ex}{\scalebox{0.22}{\includegraphics{inbox}}} & = &
\raisebox{-2ex}{\scalebox{0.22}{\includegraphics{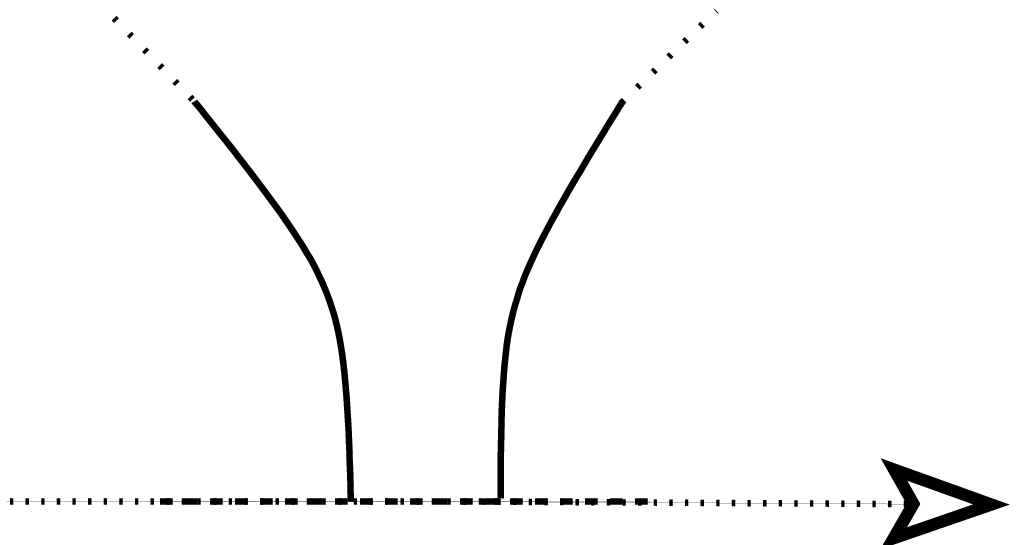}}} & + &
\raisebox{-2ex}{\scalebox{0.22}{\includegraphics{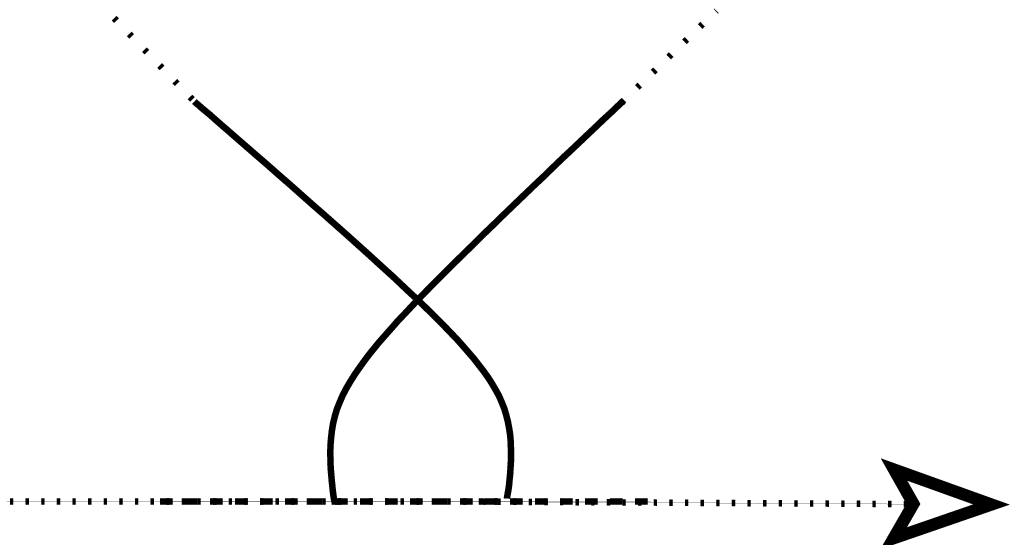}}} & - &
\raisebox{-2ex}{\scalebox{0.22}{\includegraphics{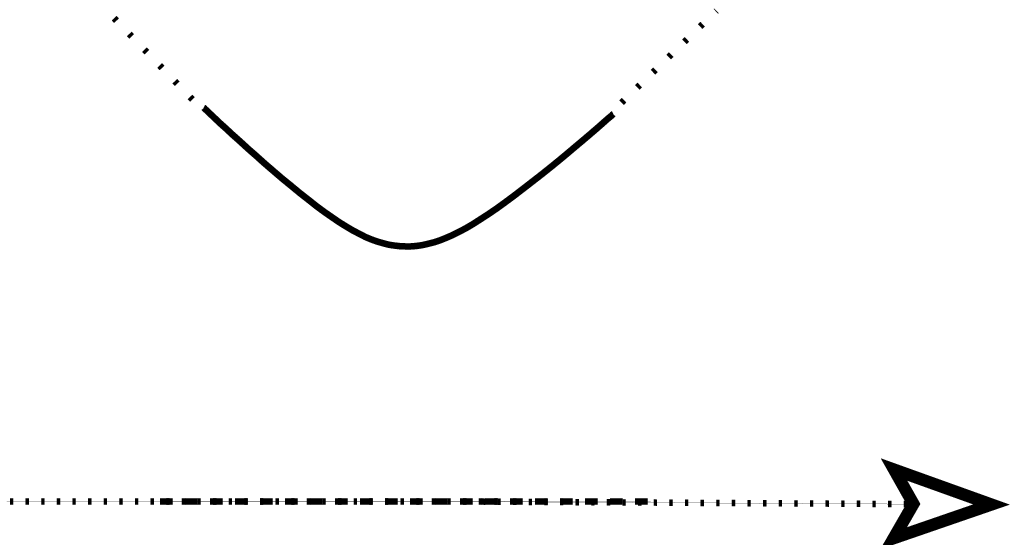}}}\,,
\\[0.5cm] \raisebox{-2.2ex}{\scalebox{0.22}{\includegraphics{inbox}}} & =
& \raisebox{-2ex}{\scalebox{0.22}{\includegraphics{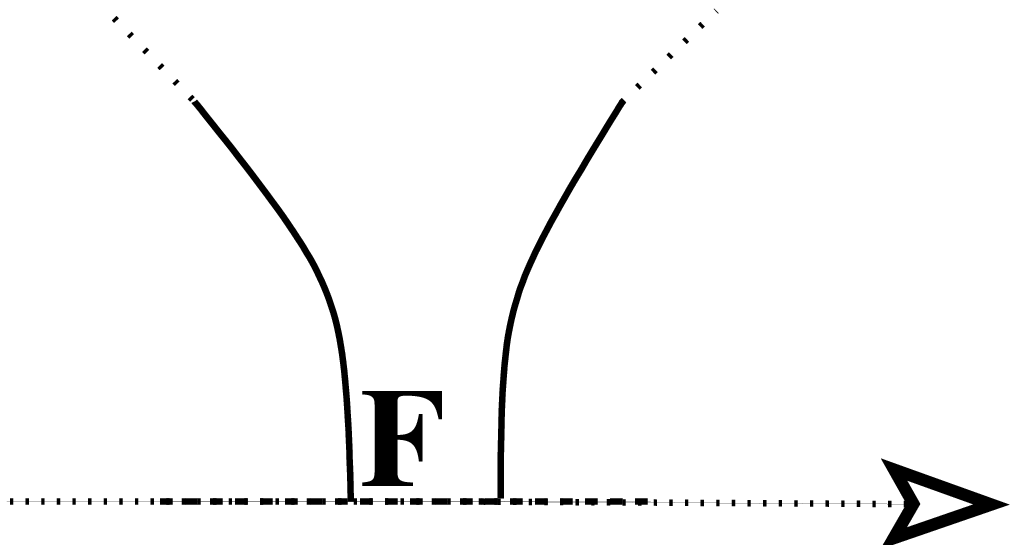}}} & - &
\raisebox{-2ex}{\scalebox{0.22}{\includegraphics{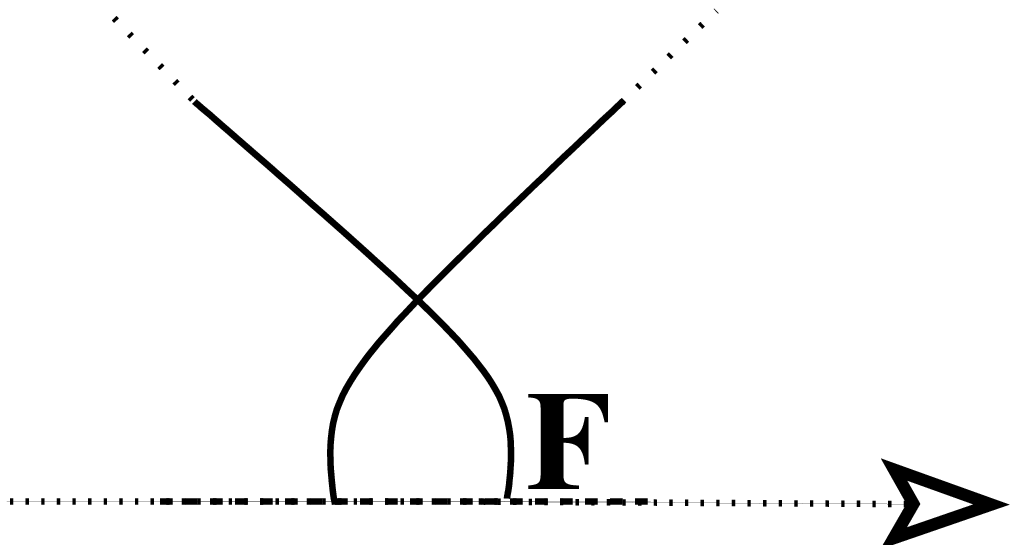}}}\ .
\end{array}
\]
On the other hand it might be one of the relations involving one
of the introduced parameters:
\[
\raisebox{-3.1ex}{\scalebox{0.22}{\includegraphics{inbox}}}\ =\
\raisebox{-4ex}{\scalebox{0.22}{\includegraphics{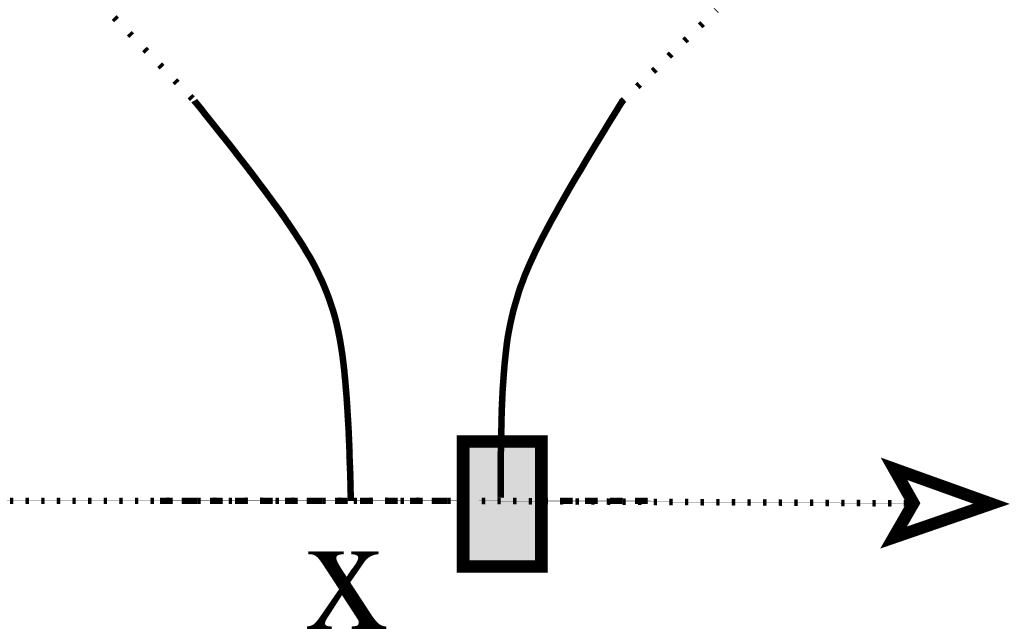}}}\ -\
(-1)^{|x|\mu}\
\raisebox{-4ex}{\scalebox{0.22}{\includegraphics{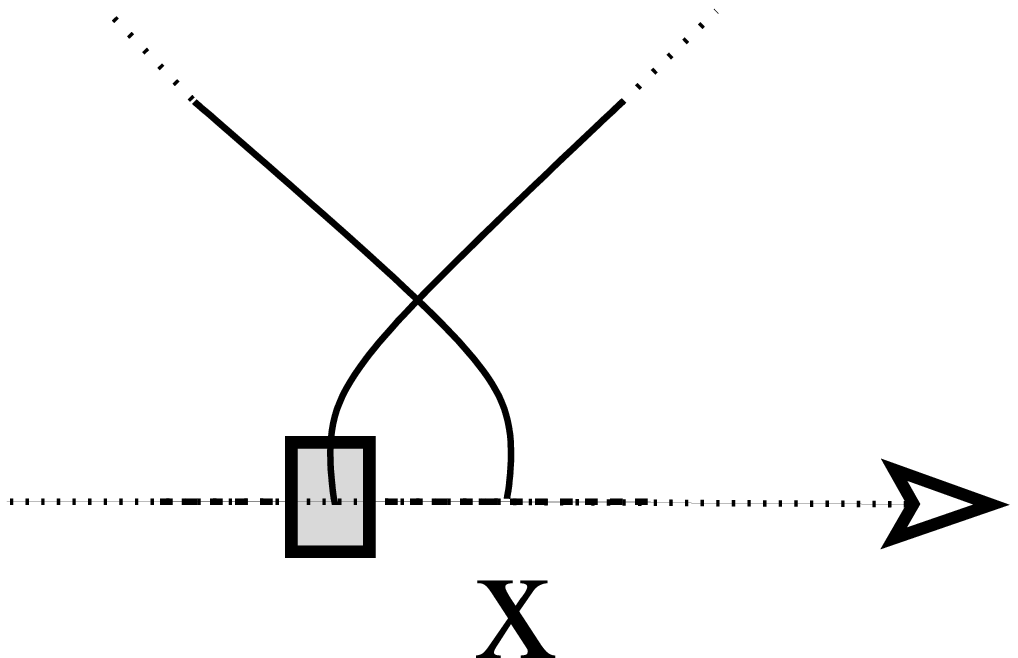}}}\ \ \ ,
\]
where $x$ can be either $a$ or $b$, where $|x|$ is the grade of
that leg, and  where there can be any type of leg in the shaded
box with $\mu$ representing its grade. To proceed, note that we
can perform the calculation on all the diagrams involved in the
relation at the same time. We begin:
\[
\begin{split}
\raisebox{-6ex}{\scalebox{0.25}{\includegraphics{instanceA}}} &
\apply
\raisebox{-6ex}{\scalebox{0.25}{\includegraphics{instanceB}}}
 \leadsto  \raisebox{-7ex}{\scalebox{0.25}{\includegraphics{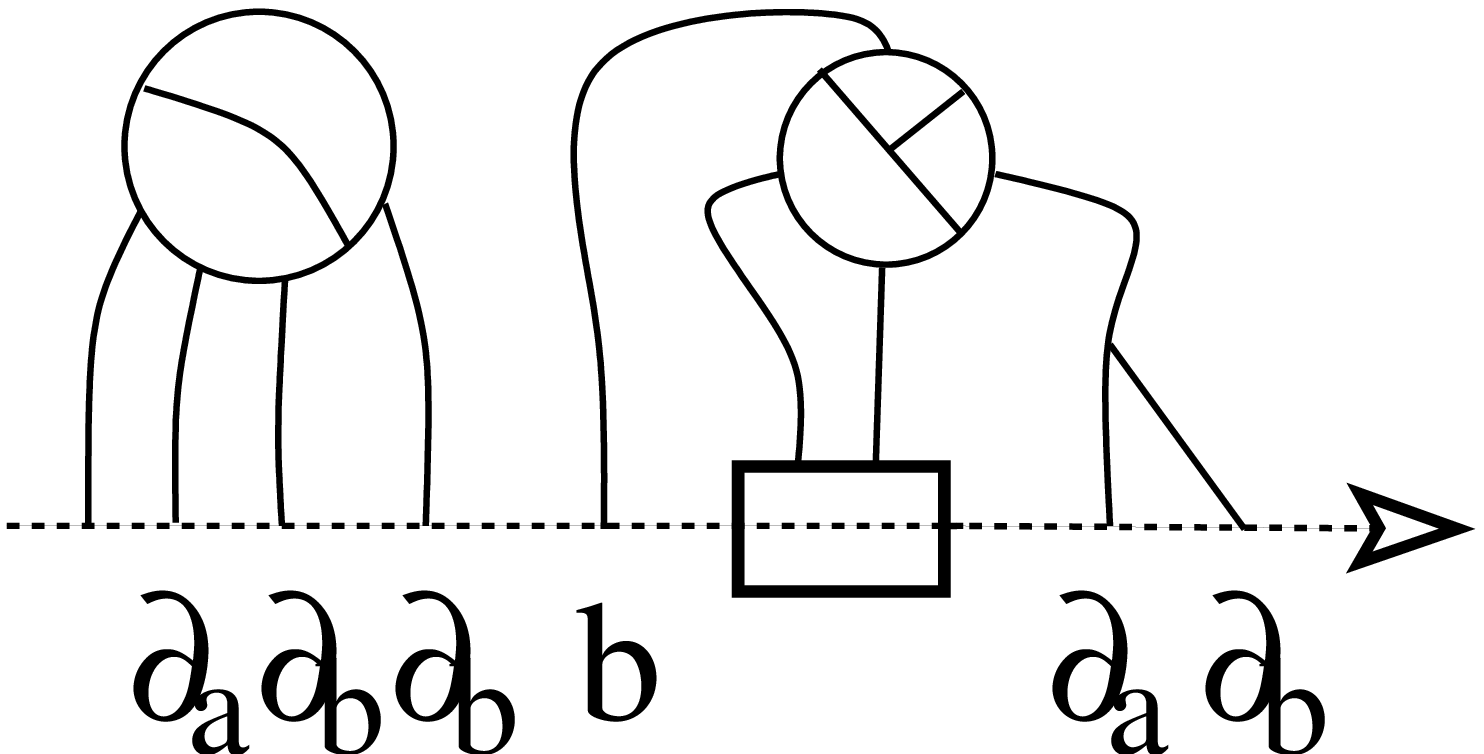}}}\\
{}& \leadsto
\raisebox{-6ex}{\scalebox{0.25}{\includegraphics{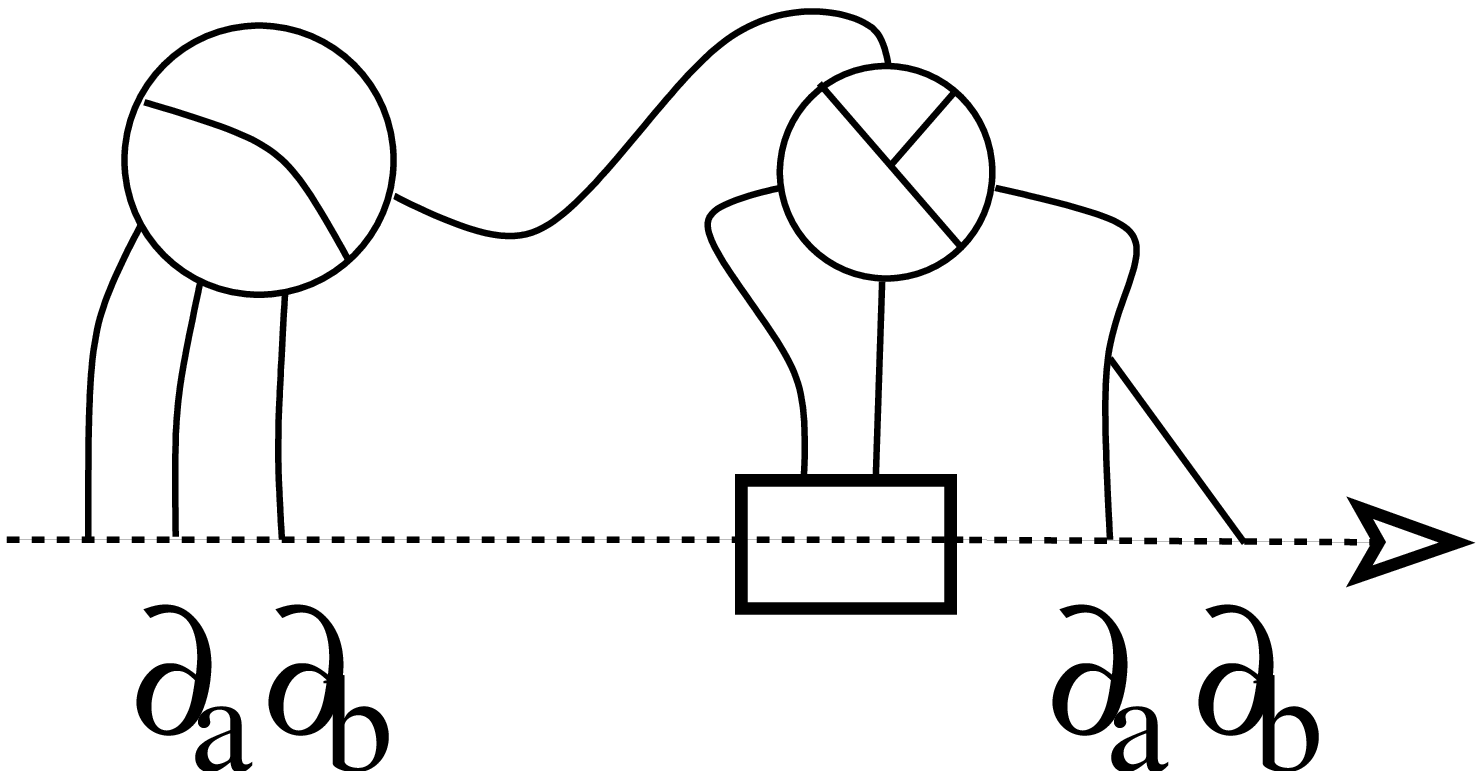}}}\ -\
\raisebox{-6ex}{\scalebox{0.25}{\includegraphics{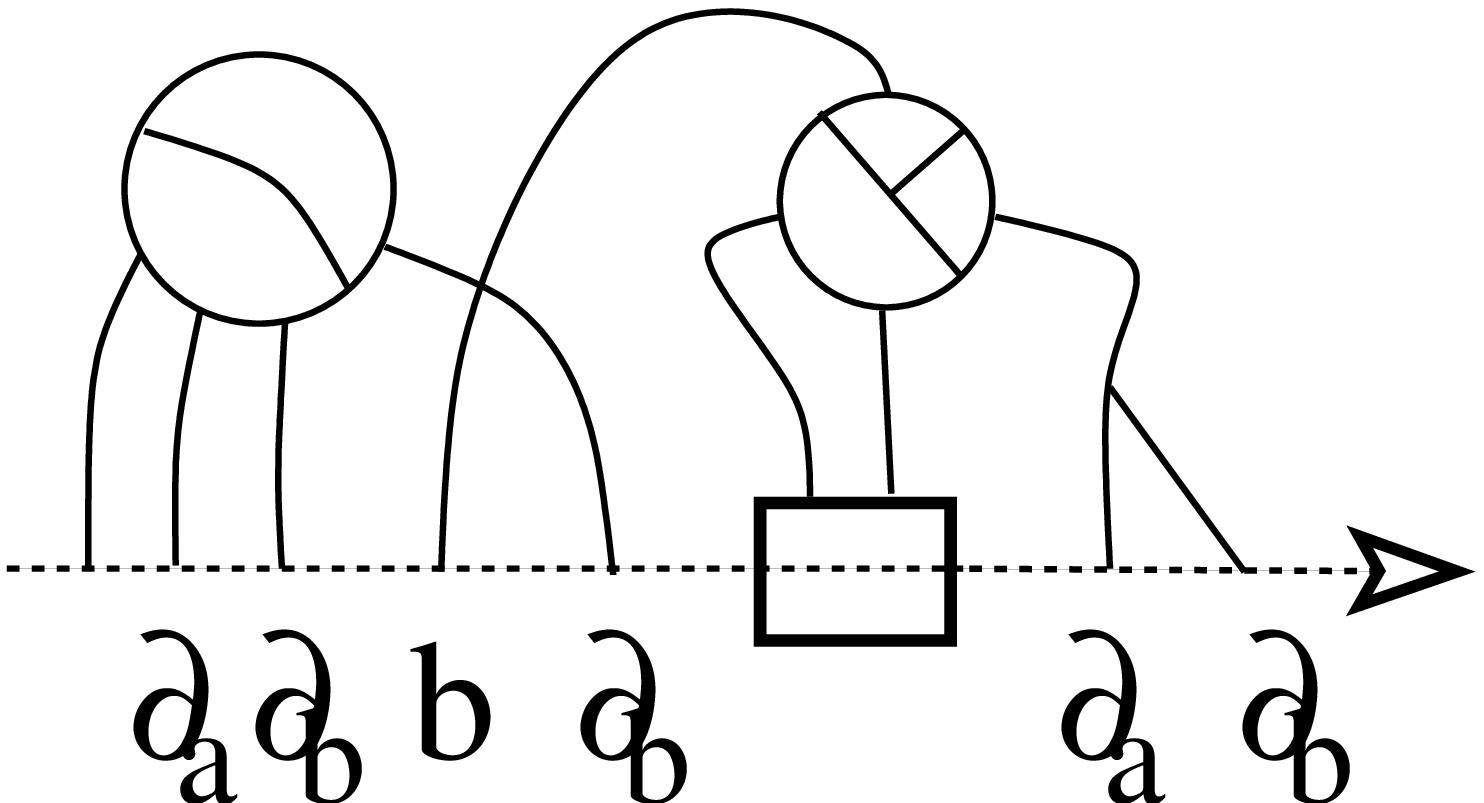}}}
\end{split}
\]
Now we employ the following formulas, which the reader can check
are true for each of the ten possible replacements for the box
(where $\mu$ is the total grade of the legs in the box):
\[
\begin{array}{cccc}
\raisebox{-4ex}{\scalebox{0.25}{\includegraphics{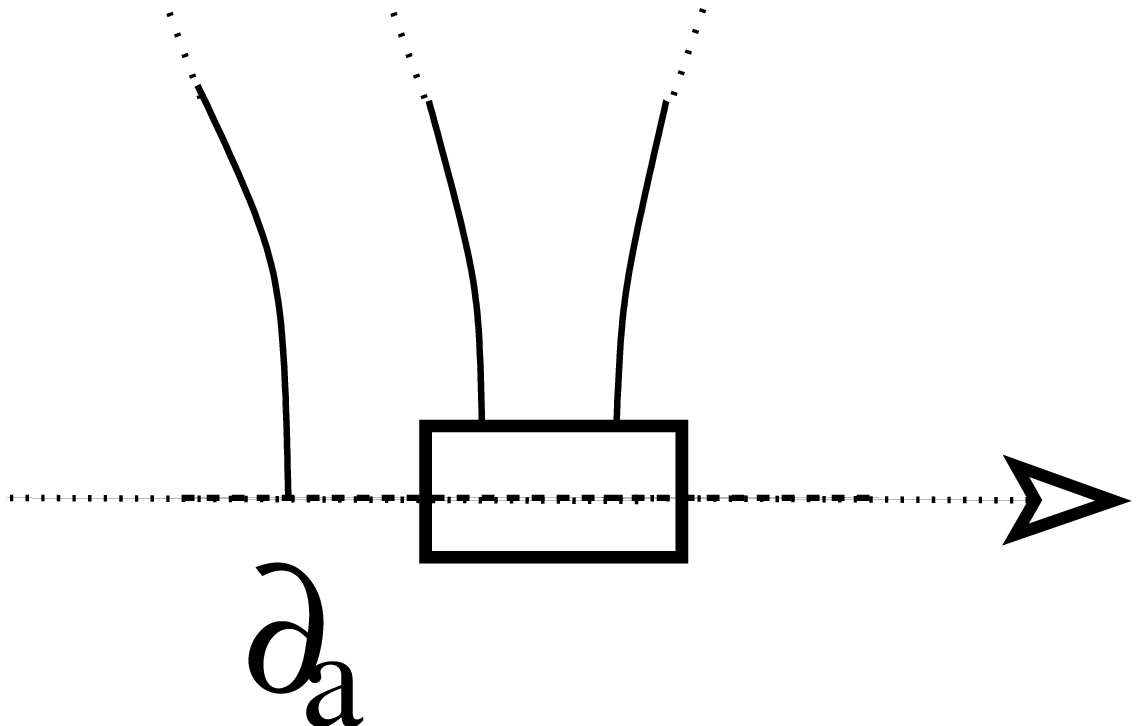}}}
& \leadsto &  &
\raisebox{-4ex}{\scalebox{0.25}{\includegraphics{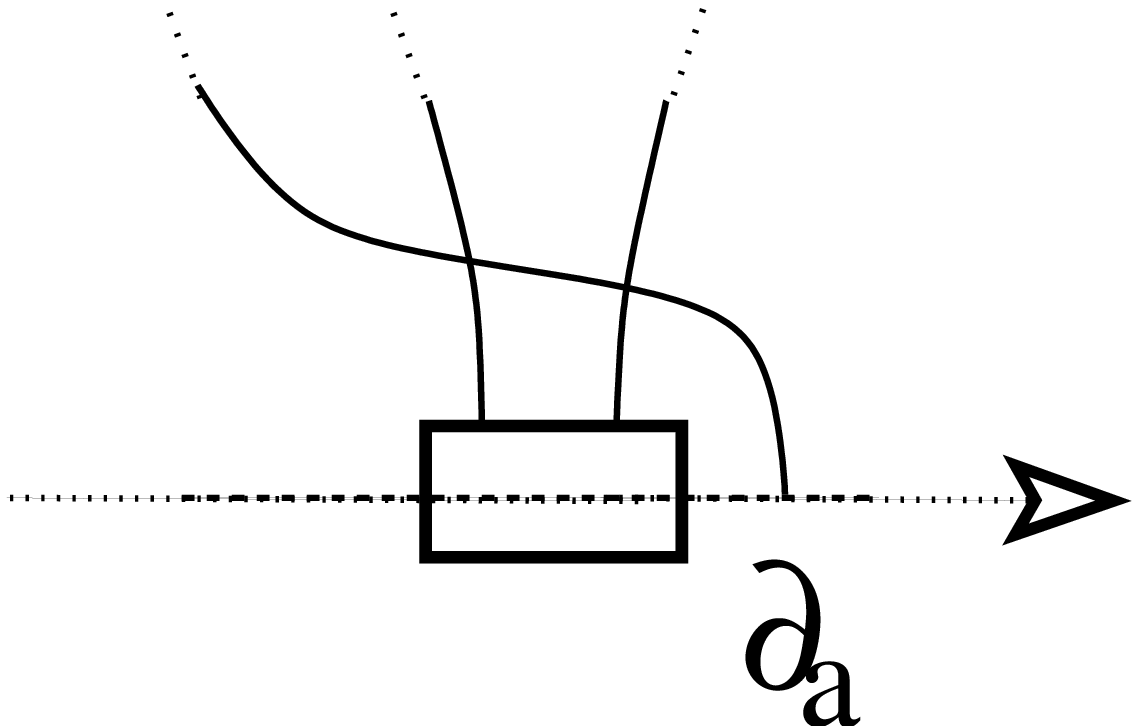}}}\
\ ,
\\
\raisebox{-4ex}{\scalebox{0.25}{\includegraphics{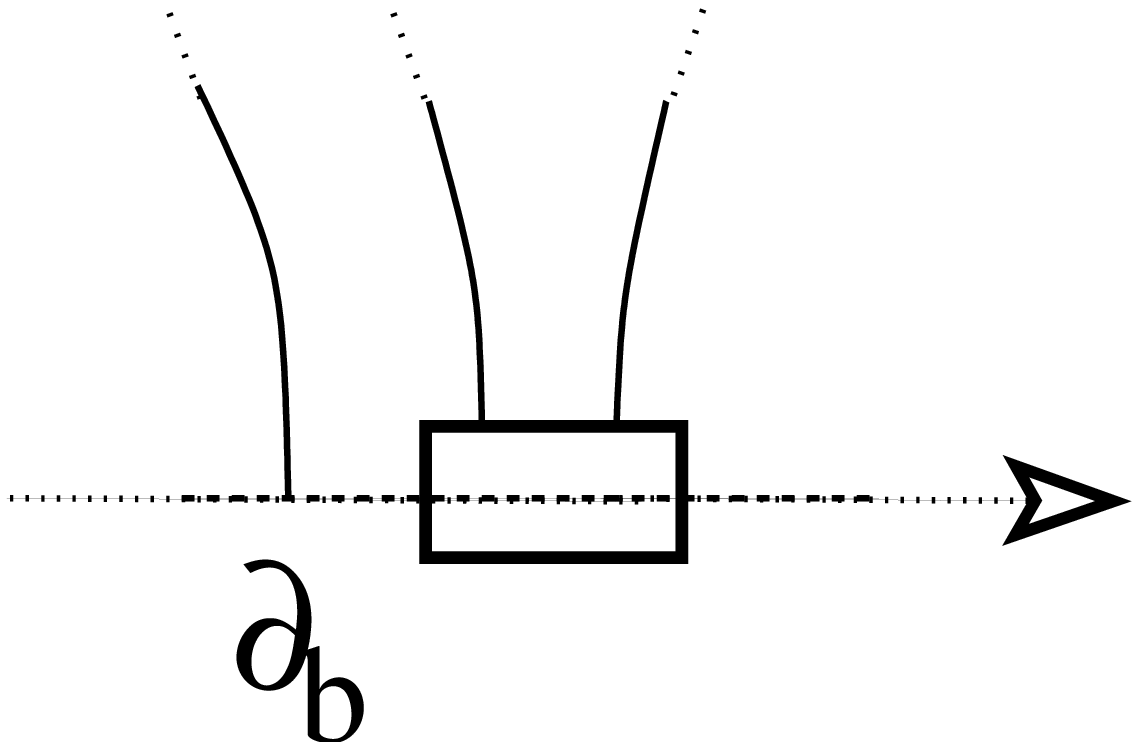}}}
& \leadsto & (-1)^\mu &
\raisebox{-4ex}{\scalebox{0.25}{\includegraphics{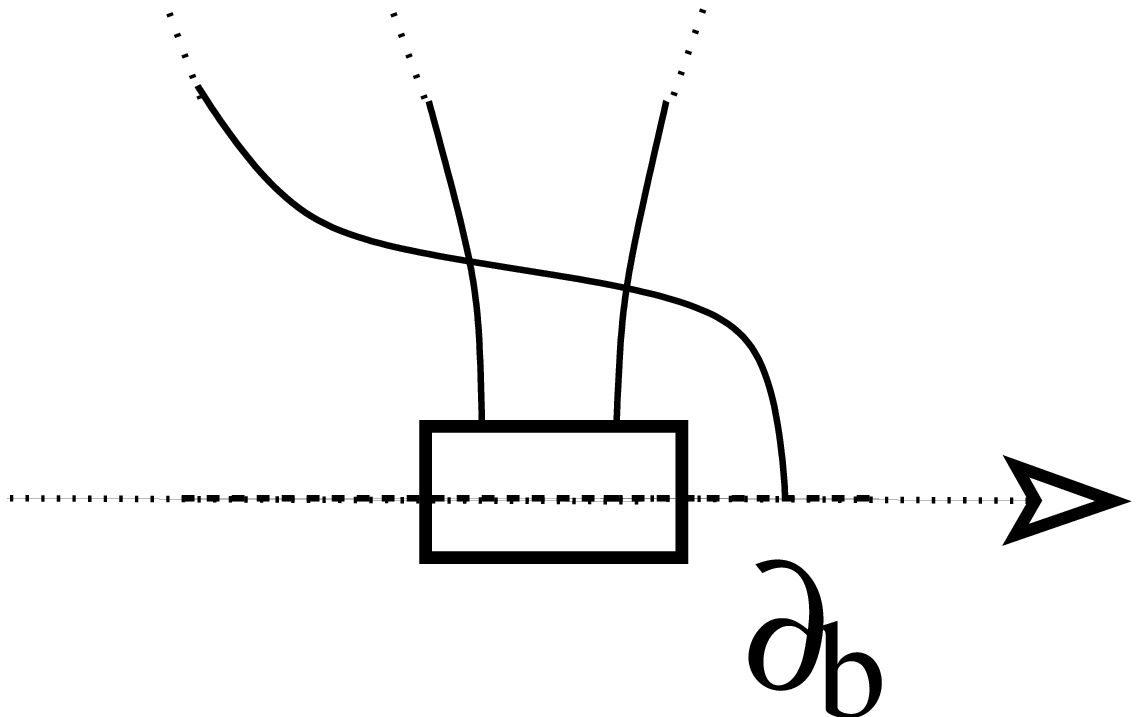}}}\
\ \ .
\end{array}
\]
Continuing with the example from earlier we obtain:
\[
\begin{array}{l}
(-1)^\mu
\raisebox{-6ex}{\scalebox{0.25}{\includegraphics{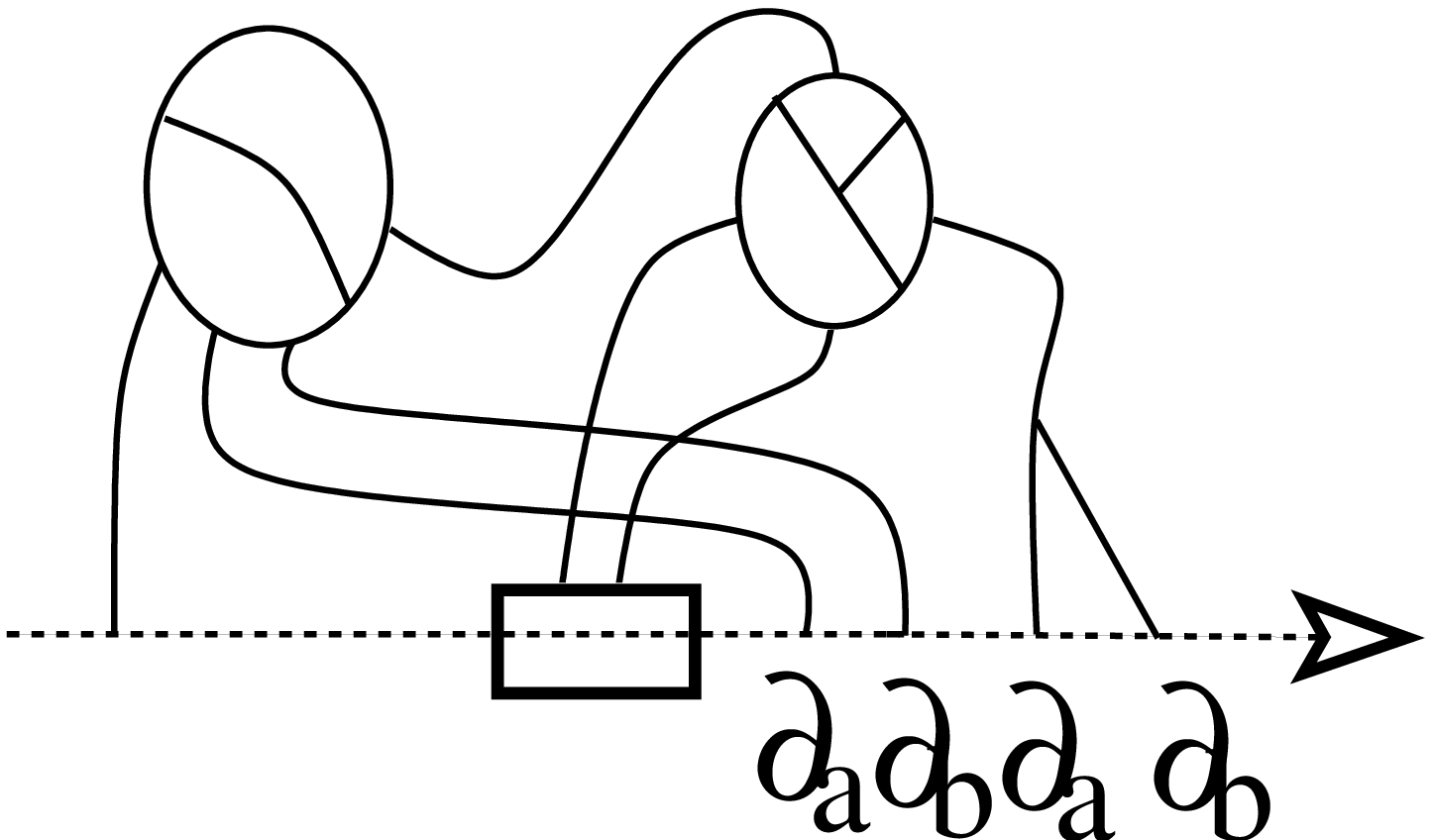}}}\
\ -\ (-1)^\mu
\raisebox{-6ex}{\scalebox{0.25}{\includegraphics{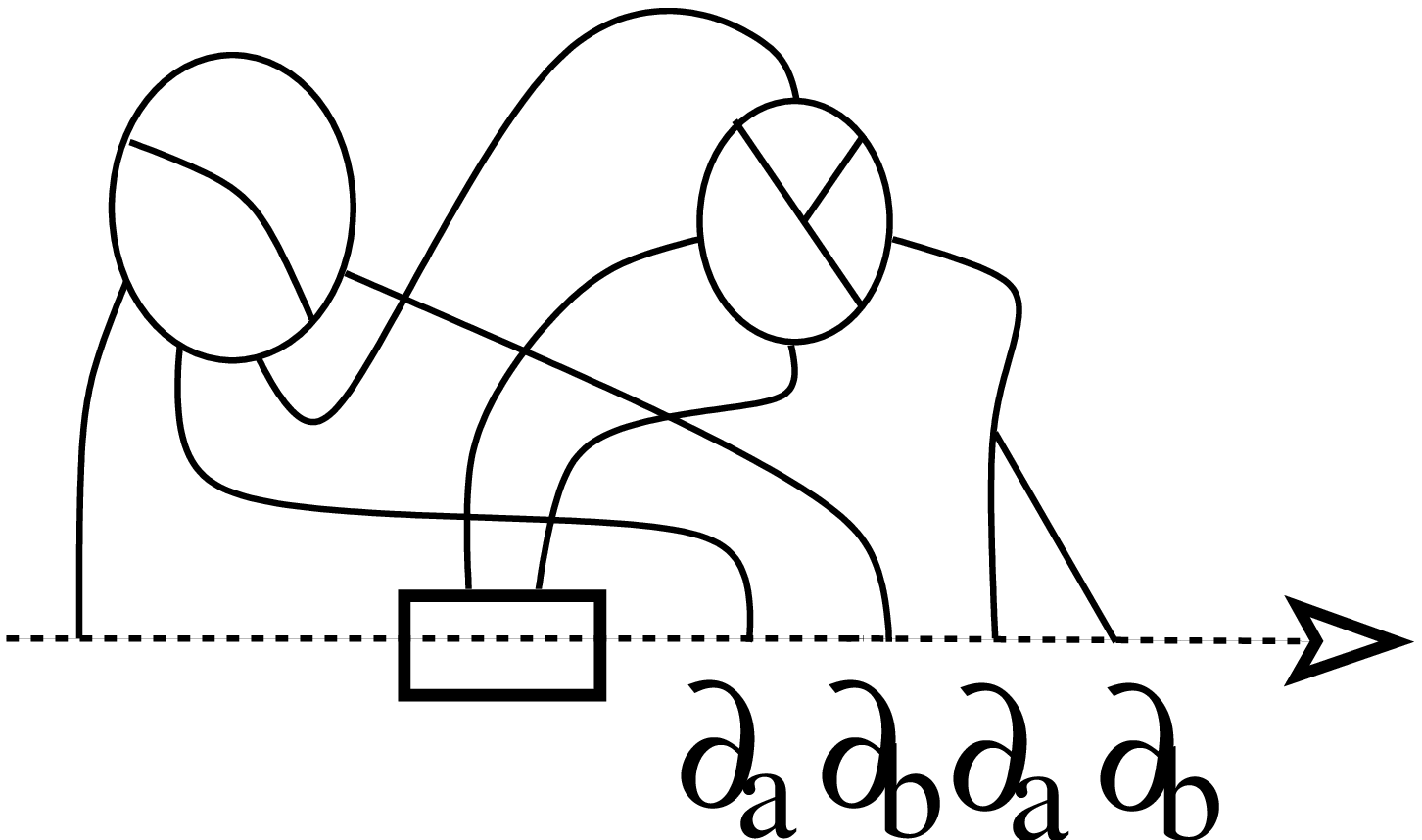}}}
\\[1.1cm]  \ \ \ \ \ \ \ \ \ \  \ +\
\raisebox{-6ex}{\scalebox{0.25}{\includegraphics{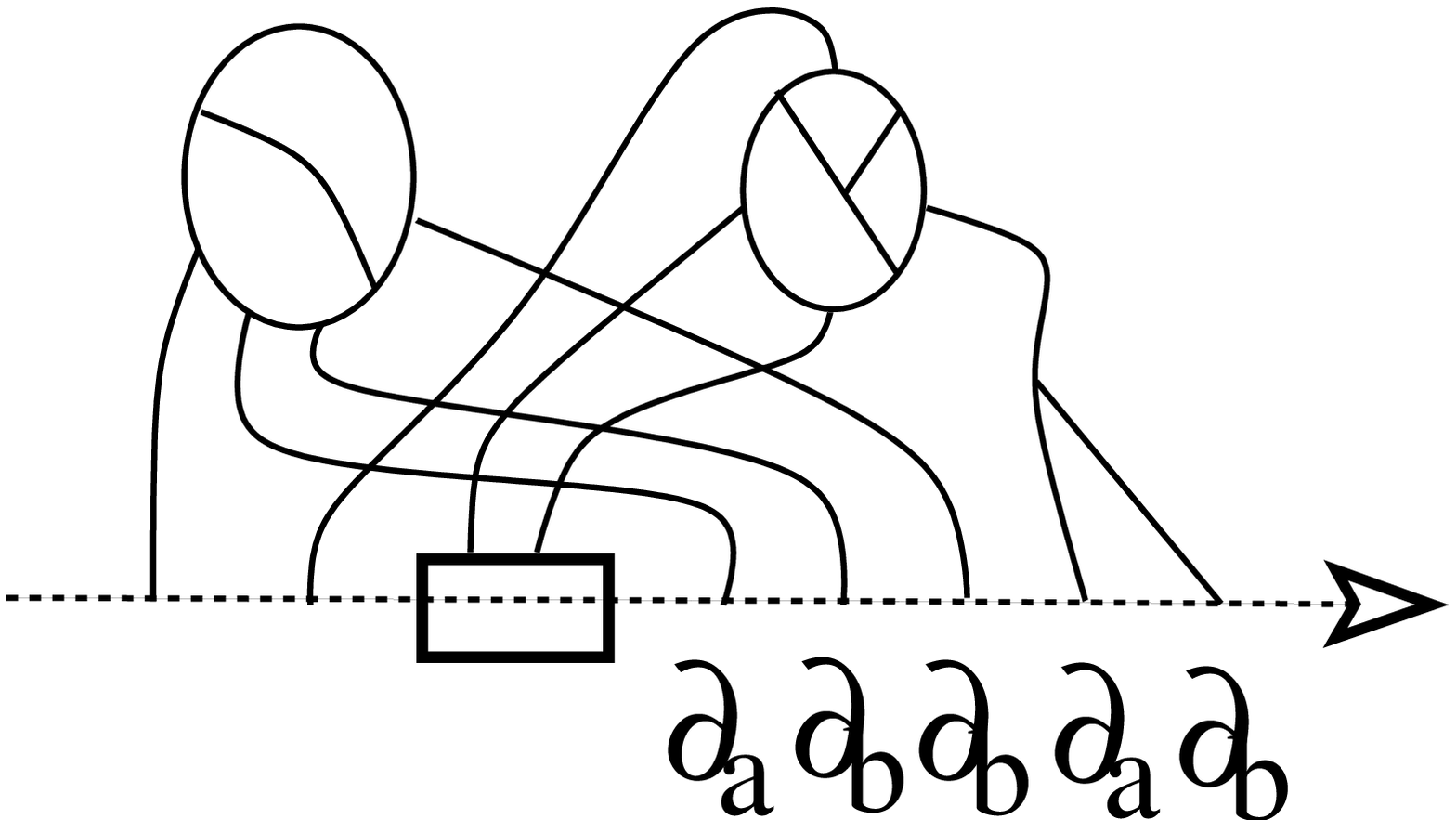}}}\
\ .
\end{array}
\]
This is clearly zero in $\WhatF\abpow$ (because it is a
combination of relations).
\end{proof}
\subsubsection{The extension to power series.}
 Recall that we are working with
``formal power series" of Weil operator diagrams, which are a
choice, for every pair $(i,j)$ of non-negative integers, of a
vector from $\WhatF[a,b]^{(i,j)}$:
\[
\WhatF\abpow = \prod_{(i,j)\in\mathbb{N}_0\times\mathbb{N}_0}
\WhatF\ab^{(i,j)}.
\]
In this section we will extend the operation product $\apply$ to
these power series in the obvious way: to multiply two power
series, do it term-by-term, then add up the results. Because there
is the possibility of infinite sums coming out of this, we must be
careful making statements in generality about this product.

For a power series $v$ write
\[
v = \sum_{(i,j)\in\mathbb{N}_0\times\mathbb{N}_0} v^{(i,j)},
\]
where $v^{(i,j)} = \pi^{(i,j)}(v)$, the type $(i,j)$ piece of $v$.
\begin{defn}
Let $v$ and $w$ be power series from $\WhatF\abpow$. If for every
pair $(i,j)\in\mathbb{N}_0\times\mathbb{N}_0$ it is true that
\[
\pi^{(i,j)}\left( v^{(k,l)}\apply w^{(m,n)} \right) = 0
\]
for all but finitely many pairs $\left((k,l),(m,n)\right)$, then
we say that the product $v\apply w$ is {\bf convergent}, in which
case it is defined to be
\[
\sum_{(k,l),(m,n)\in \mathbb{N}_0\times\mathbb{N}_0}
\left(v^{(k,l)}\apply w^{(m,n)}\right)\ \in\WhatF\abpow.
\]
\end{defn}
It is not difficult to construct non-convergent products. Here is
one:
\[
\exp_{\#}\left(
\raisebox{-0.85cm}{\scalebox{0.24}{\includegraphics{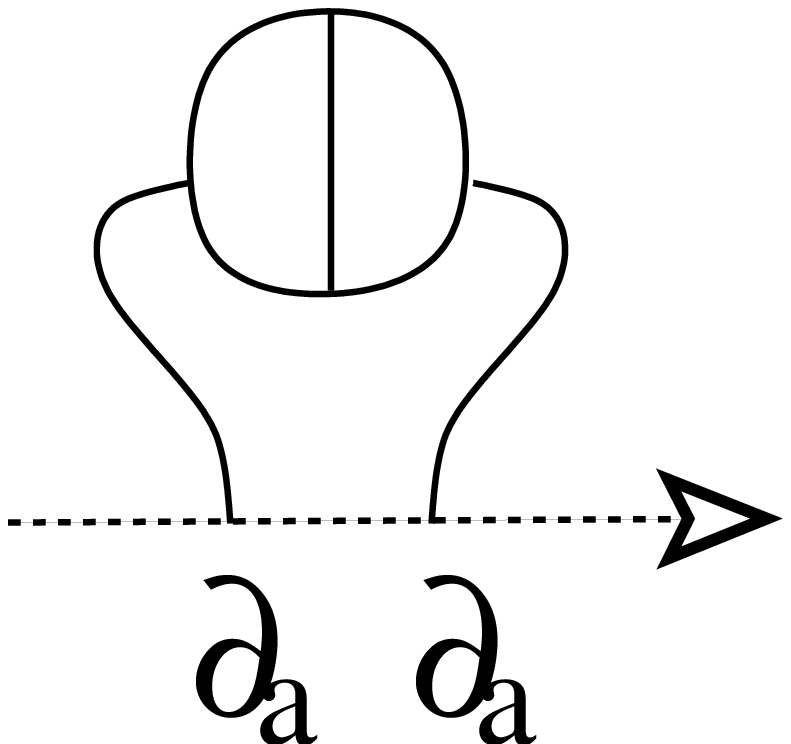}}}
\right) \apply \exp_{\#}\left(
\raisebox{-0.75cm}{\scalebox{0.24}{\includegraphics{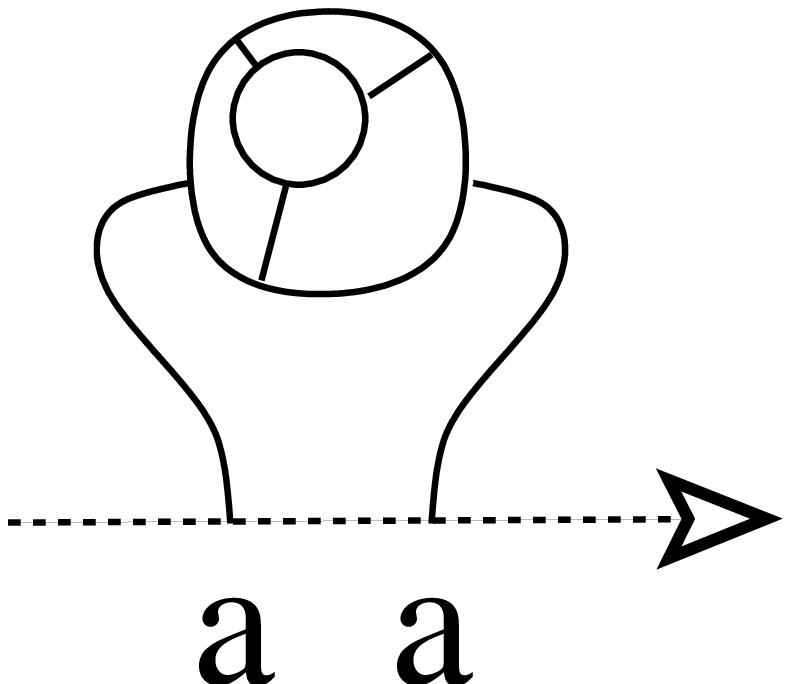}}}
\right).
\]
\subsection{A convenient graphical method for doing diagram
operations.} \label{applygraphical}\label{alternativedefn} In
certain of the computations to come later in this paper, we'll
need to be able to give a direct construction of all the terms
that contribute to a diagram operation, all at once, together with an easy way
to determine the signs of those contributions.

For this purpose we'll now introduce a convenient visual method
for doing a diagram operation. In this method, the ``operation" is
defined to be a sum of one term for every {\it gluing} of the two
diagrams:
\begin{defn}
Consider two operator Weil diagrams $v$ and $w$. Let
$\text{Op}(v)$ denote the set of operator legs of $v$, and let
$\text{Par}(w)$ denote the set of parameter legs of $w$. A {\bf
gluing} of $v$ onto $w$ is an injection of a (possibly empty)
subset of $\text{Op}(v)$ into $\text{Par}(w)$ that respects labels
(so a $\partial_a$-labelled leg of $v$ will only be mapped to an
$a$-labelled leg of $w$, and similarly for the $b$-labels). Let
$\mathcal{G}(v,w)$ denote the set of gluings of $v$ onto $w$.
\end{defn}
Below we will explain how to associate a term \[ t(v,w,\sigma) \]
to a gluing $\sigma\in \mathcal{G}(v,w)$, and then we'll define
the operation of $v$ on $w$ to be:
$
v\apply w = \sum_{\sigma\in \mathcal{G}(v,w)} t(v,w,\sigma).
$
This construction will obviously agree with the first definition
of the operation that we gave in Section \ref{firstopdefn}.

 In the discussion to follow, we'll consider the example of the
diagrams
\[
v =
\raisebox{-8ex}{\scalebox{0.22}{\includegraphics{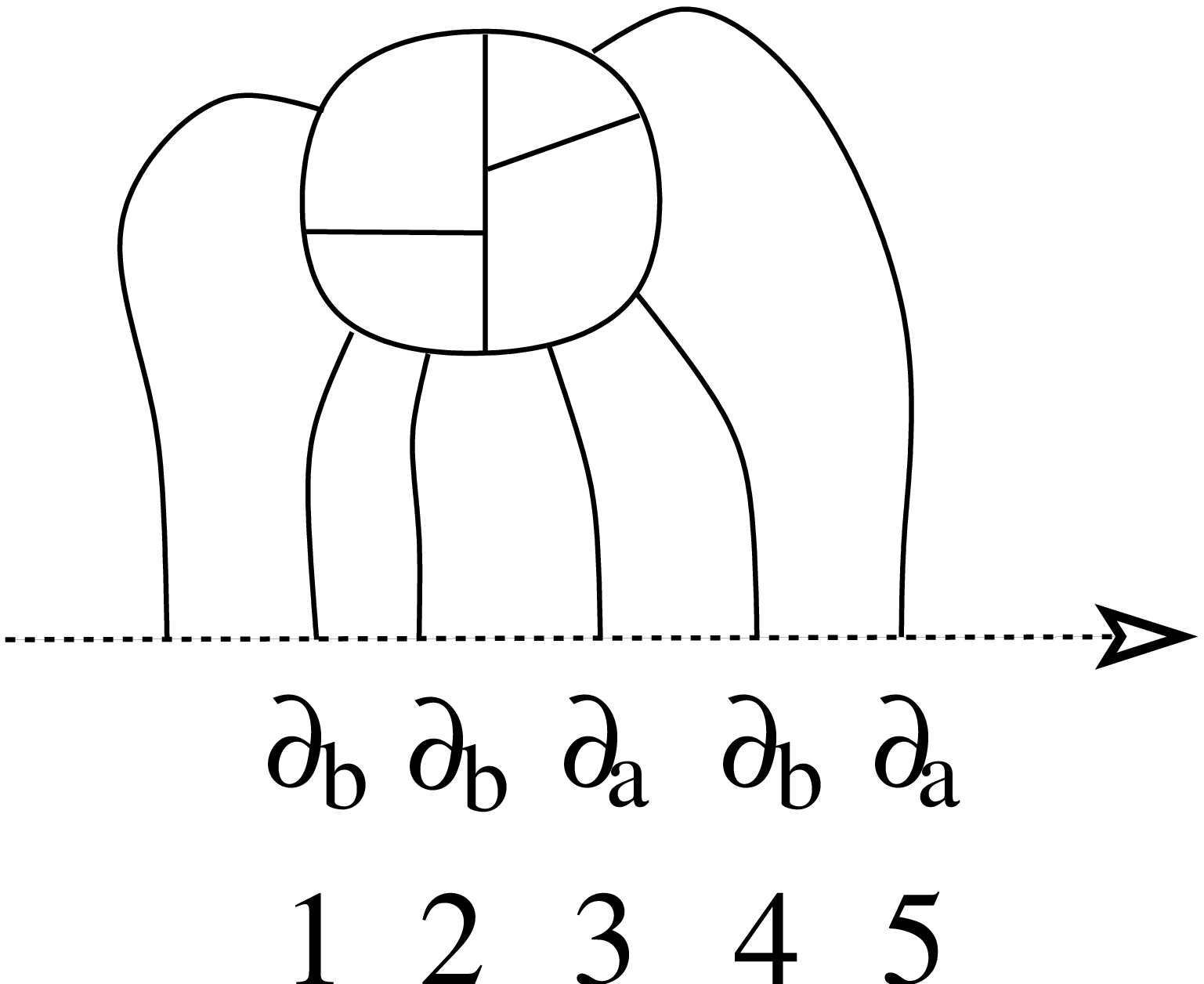}}} \
\ \text{and}\ \ w =
\raisebox{-8ex}{\scalebox{0.22}{\includegraphics{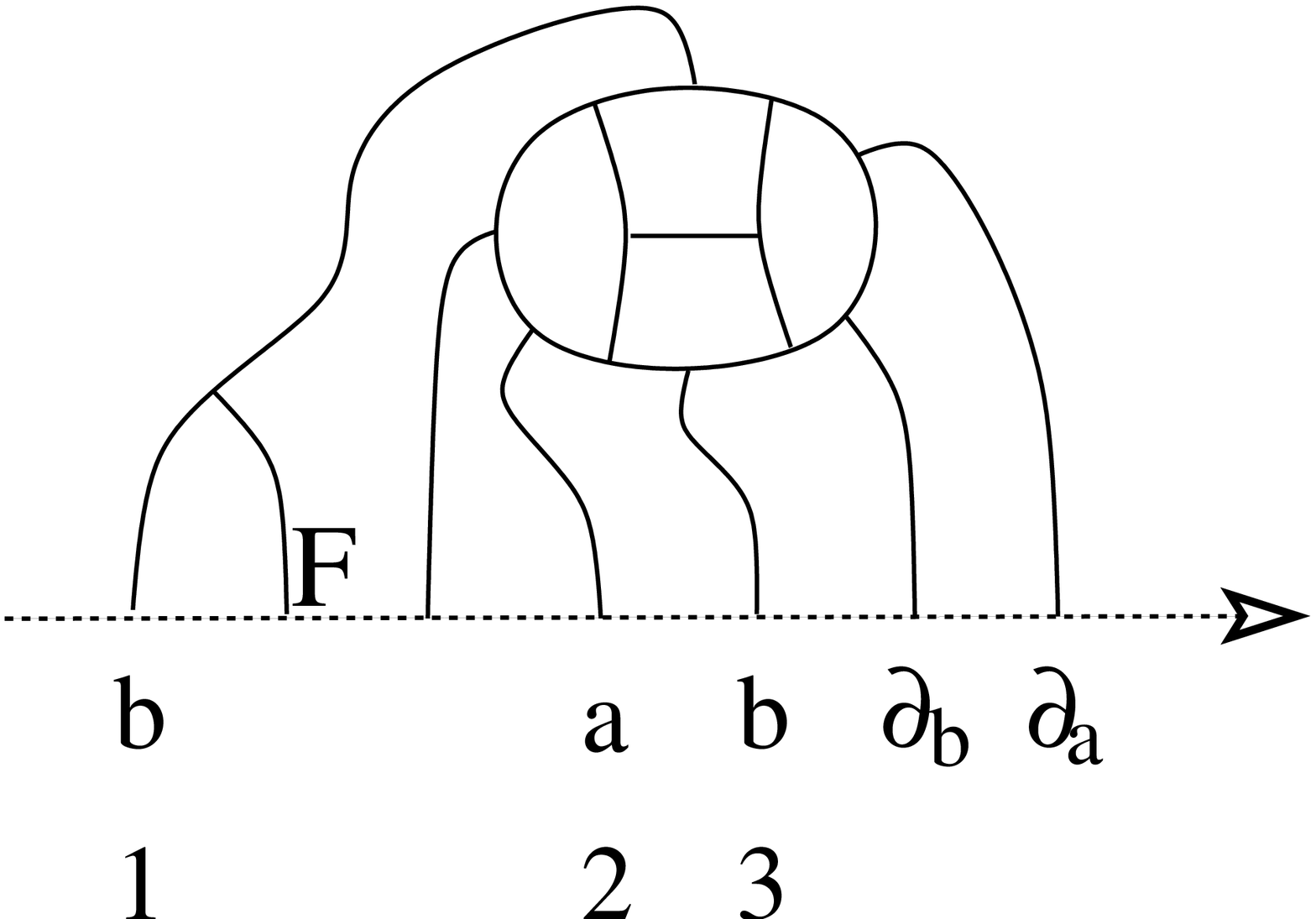}}}\ \
.
\]
For the purposes of the discussion, the operator legs of $v$ have been numbered from left to right, and so
have the parameter legs of $w$, as displayed by the above diagrams.
Below, we'll construct the term $t\left(v,w, \left({2 4 5    \atop
1 3 2}\right)\right)$.

The term $t(v,w,\sigma)$ corresponding to some gluing
$\sigma\in\mathcal{G}(v,w)$ will be constructed by the following
procedure. To begin, place the operator legs of $v$ up the
left-hand side of a grid, and the non-operator legs of $w$ across
the top of the grid, above an orienting line:
\[ \raisebox{-10ex}{\scalebox{0.2}{\includegraphics{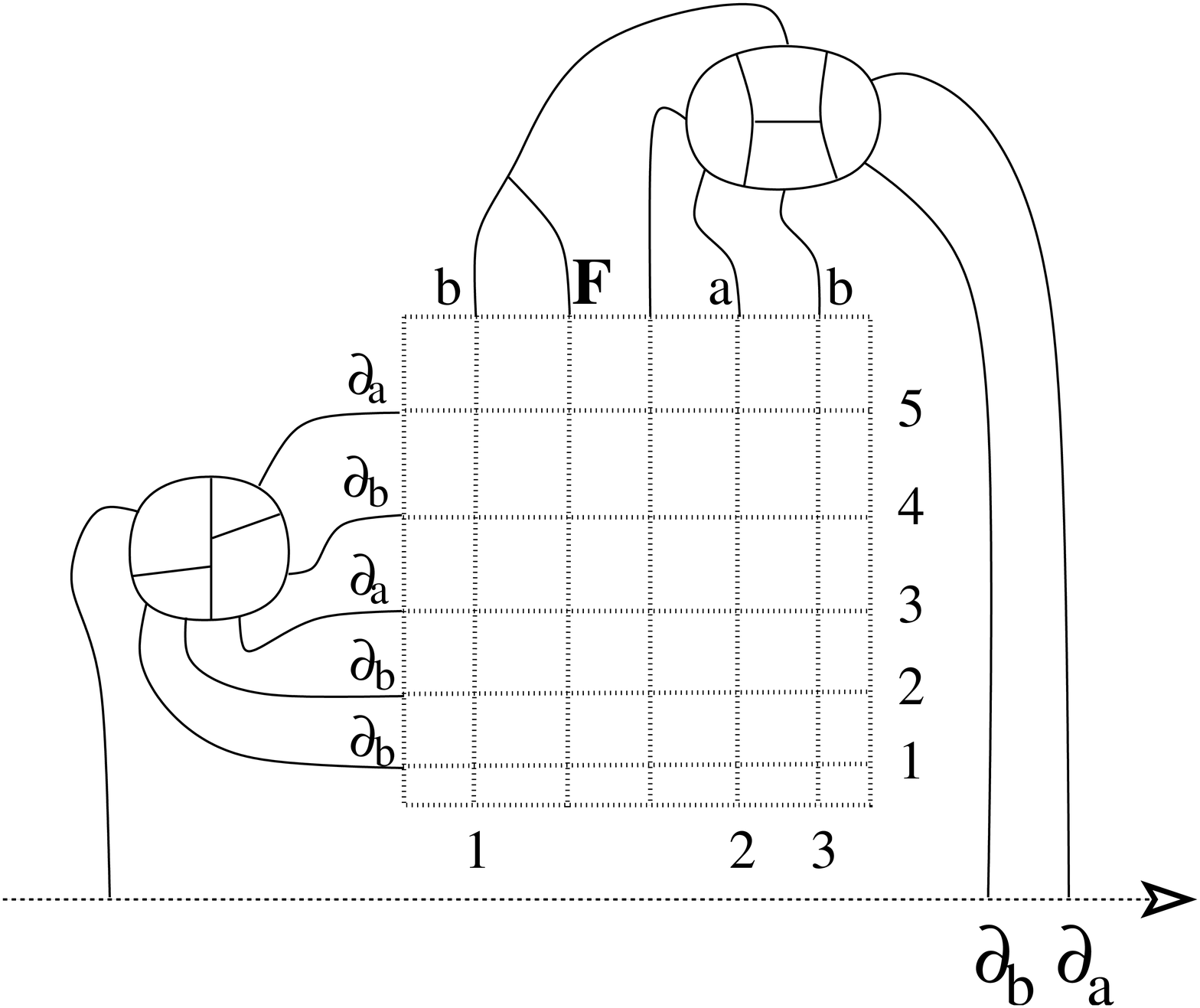}}}\ \ .
\]
The next step is to join up legs along the grid according to the
gluing. Join up grade 1 legs using a full line, and grade 2 legs
using a dashed line (this is so that we will be able to easily
read off the sign of the term at the end of the construction).
Continuing the example:
\[\raisebox{-10ex}{ \scalebox{0.15}{\includegraphics{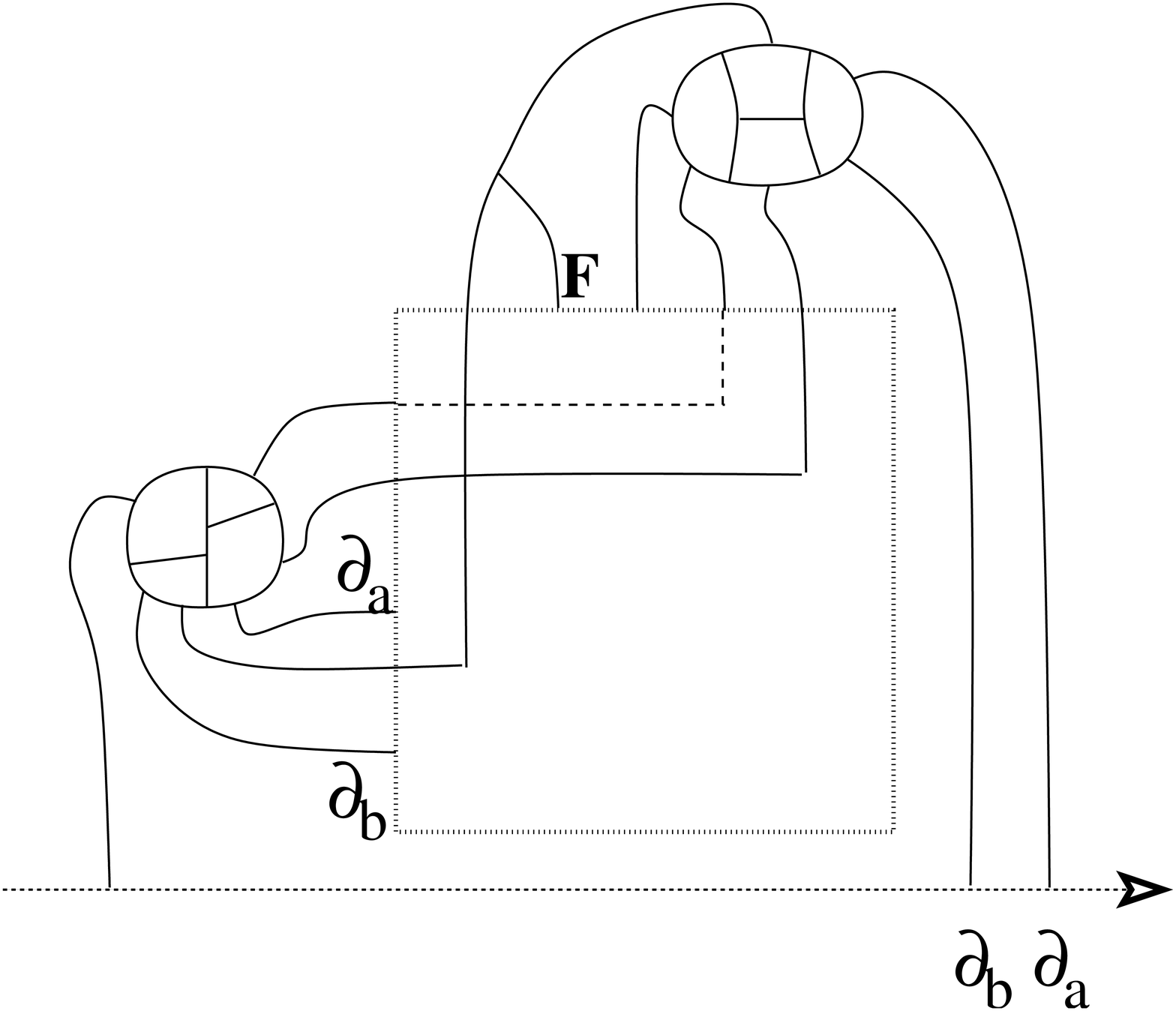}}}\ \ .
\]
Join any remaining legs on the top of the grid to the orienting
line. For grade 1 legs use a full line and for grade 2 legs use a
dashed line: \[
\raisebox{-10ex}{\scalebox{0.15}{\includegraphics{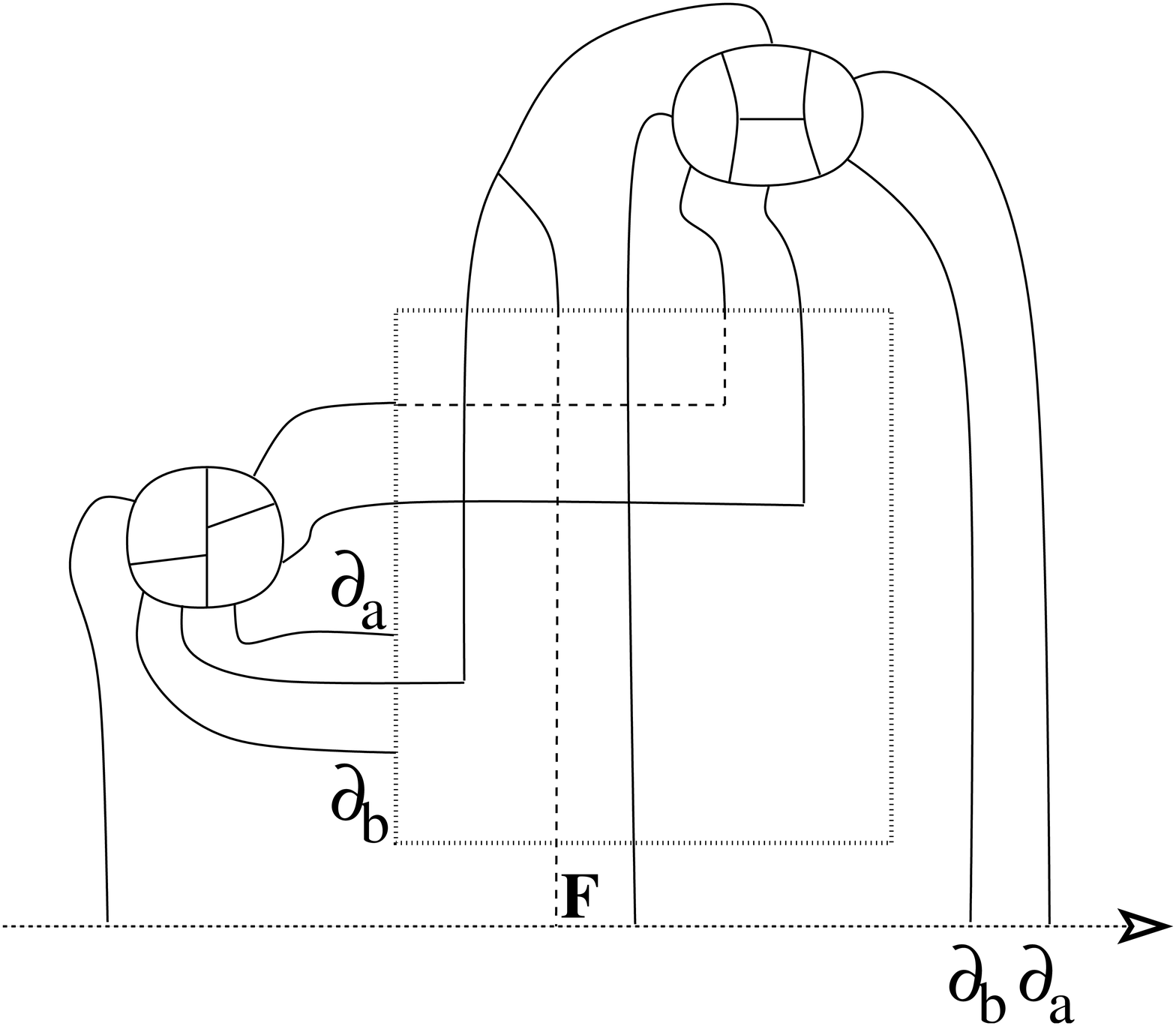}}}\ \ .
\]
Finally, carry any remaining operator legs lying along the
left-hand side of the grid to the far side of the grid, and then
place them on the orienting line using nested right-angles:
\[
\raisebox{-10ex}{\scalebox{0.15}{\includegraphics{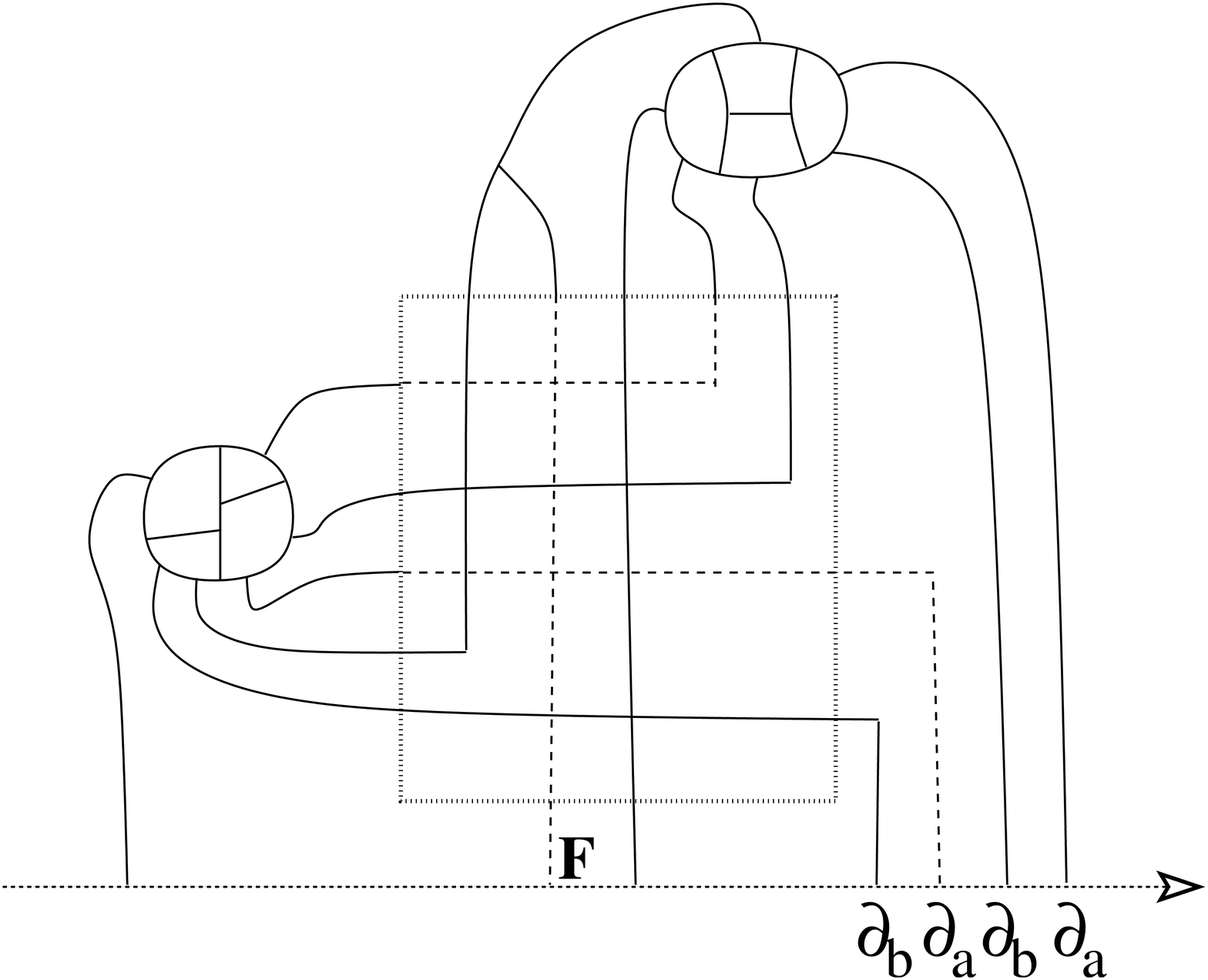}}}\ \ .
\]
Let $x$ denote the number of intersections between full lines
displayed within the box. The term $t(v,w,\sigma)$ is just the
diagram that has been constructed (with the dashed lines now
filled in), with a sign $(-1)^x$ out the front.

In the example at hand $x=3$, so:\vspace{-0.5cm}
\[
t\left(v,w, \left({2 4 5    \atop 1 3 2}\right)\right)\ =\ (-1)^3
\raisebox{-6ex}{\scalebox{0.15}{\includegraphics{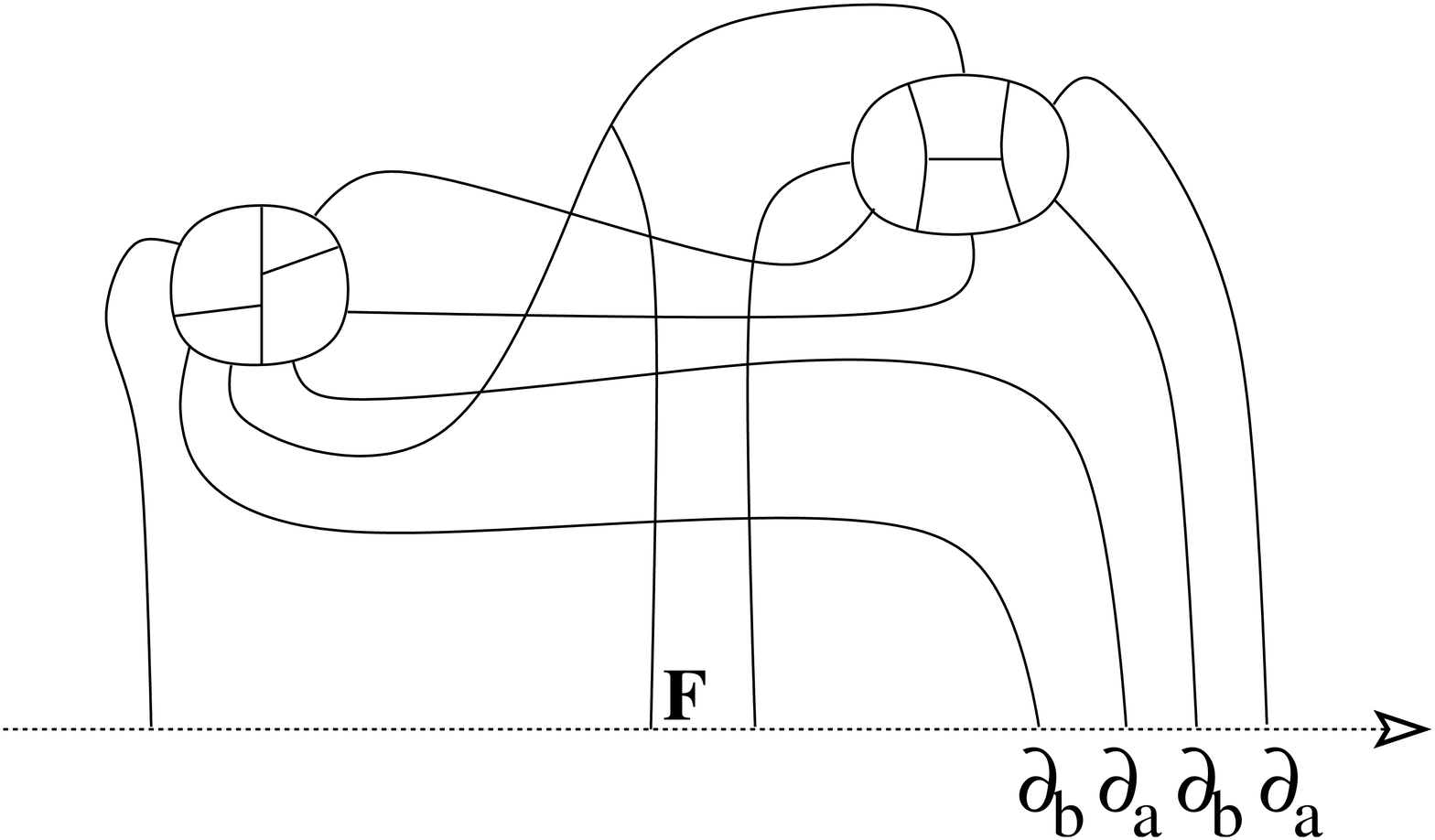}}}\ \
.\\[0.2cm]
\]

\subsection{Associativity}
As an illustration of this graphical method we'll now use it to
give a (probably more detailed than necessary) proof that the
operation product is {\it associative}:
\begin{prop}
Let $u$, $v$ and $w$ denote operator Weil diagrams. Then
\[
u\apply (v\apply w) = (u\apply v)\apply w.
\]
\end{prop}
\begin{proof}
We'll illustrate the discussion with the example of:\\[-0.15cm]
\[ u =
\raisebox{-0.75cm}{\scalebox{0.2}{\includegraphics{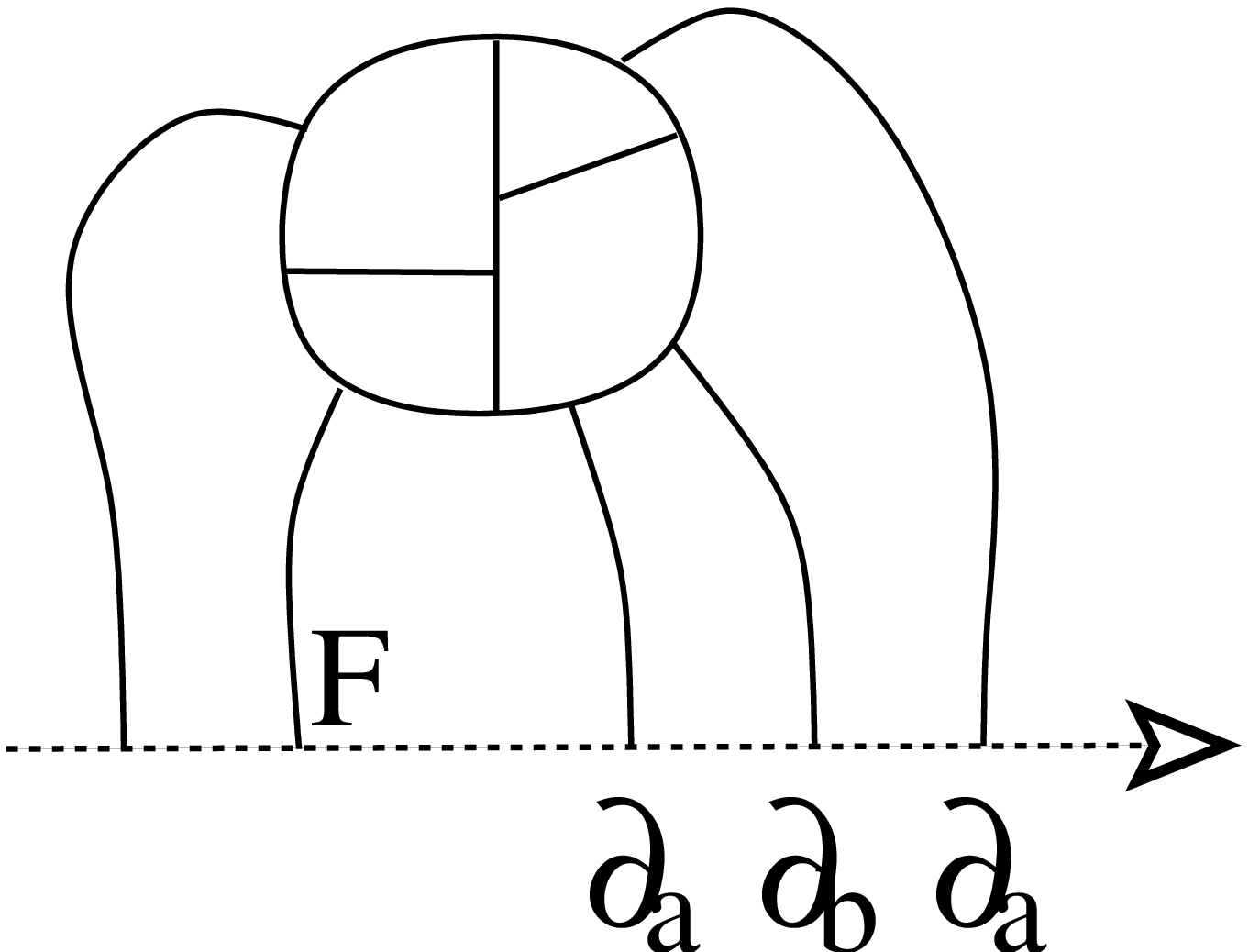}}}\ ,\ \
v = \raisebox{-0.75cm}{\scalebox{0.2}{\includegraphics{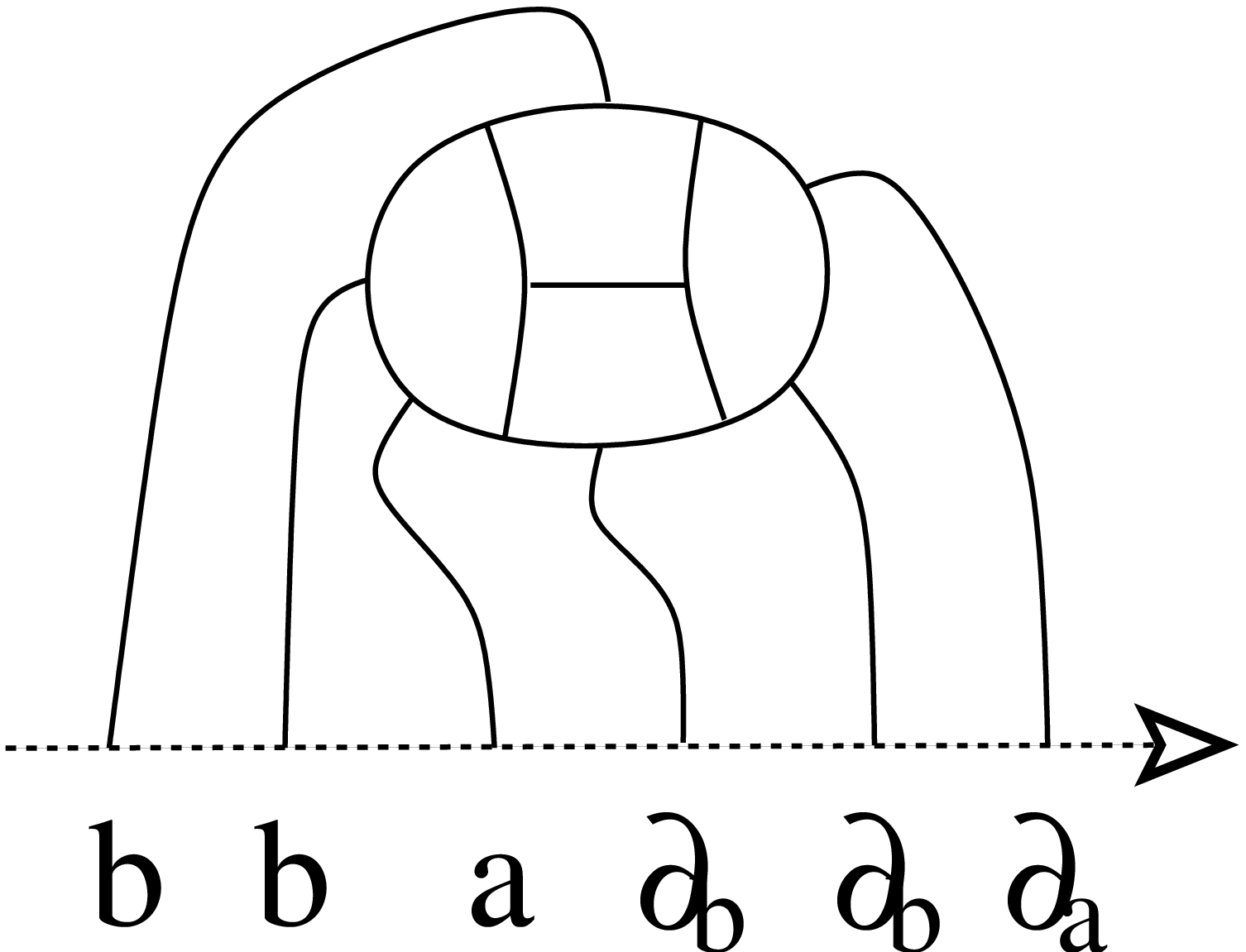}}}\
\text{and}\ w =
\raisebox{-0.75cm}{\scalebox{0.2}{\includegraphics{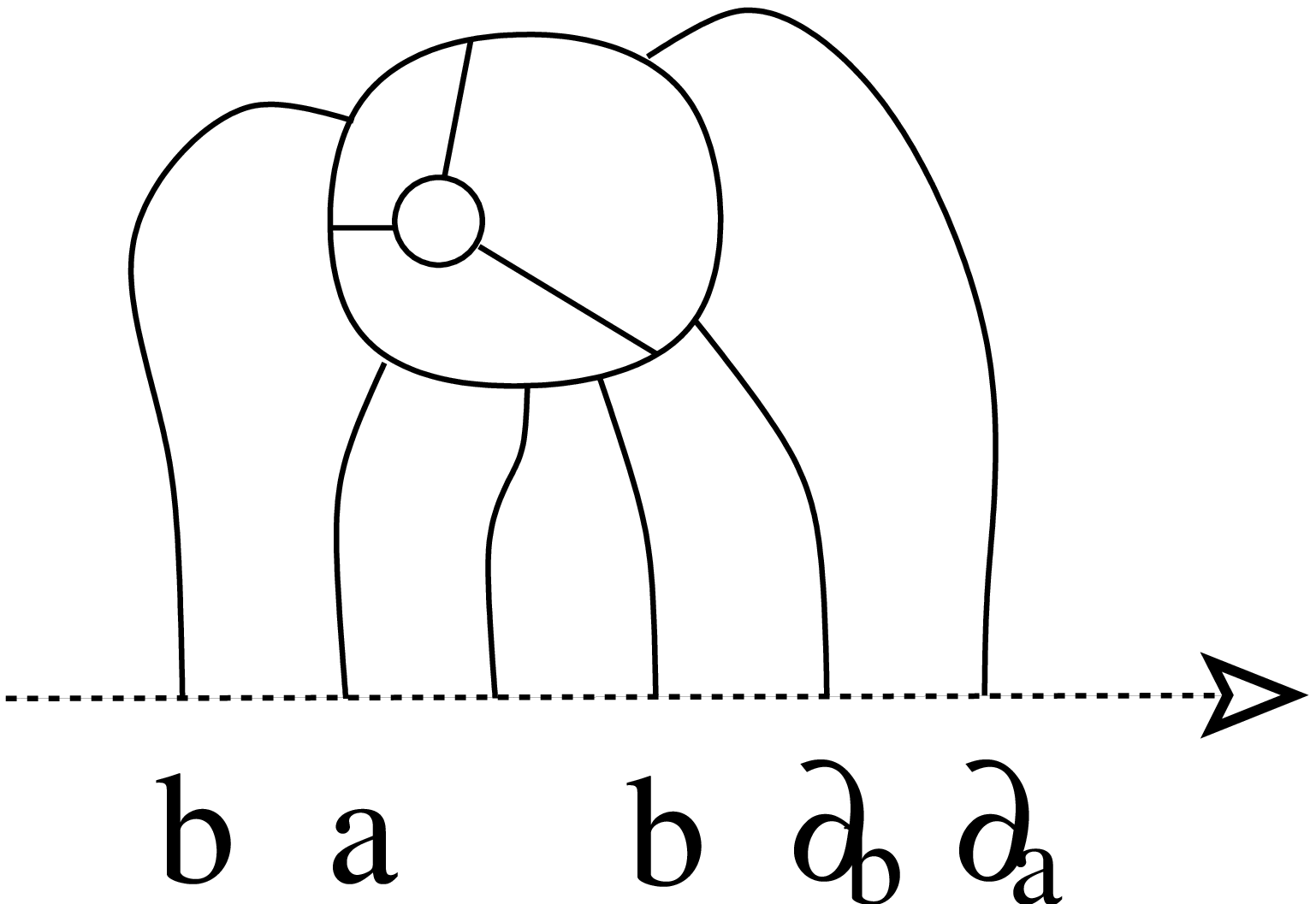}}}
.\\[0.15cm]
\]
To begin, assemble the three diagrams around the edges of a
``step-ladder" grid, in the following way:
\[
\scalebox{0.24}{\includegraphics{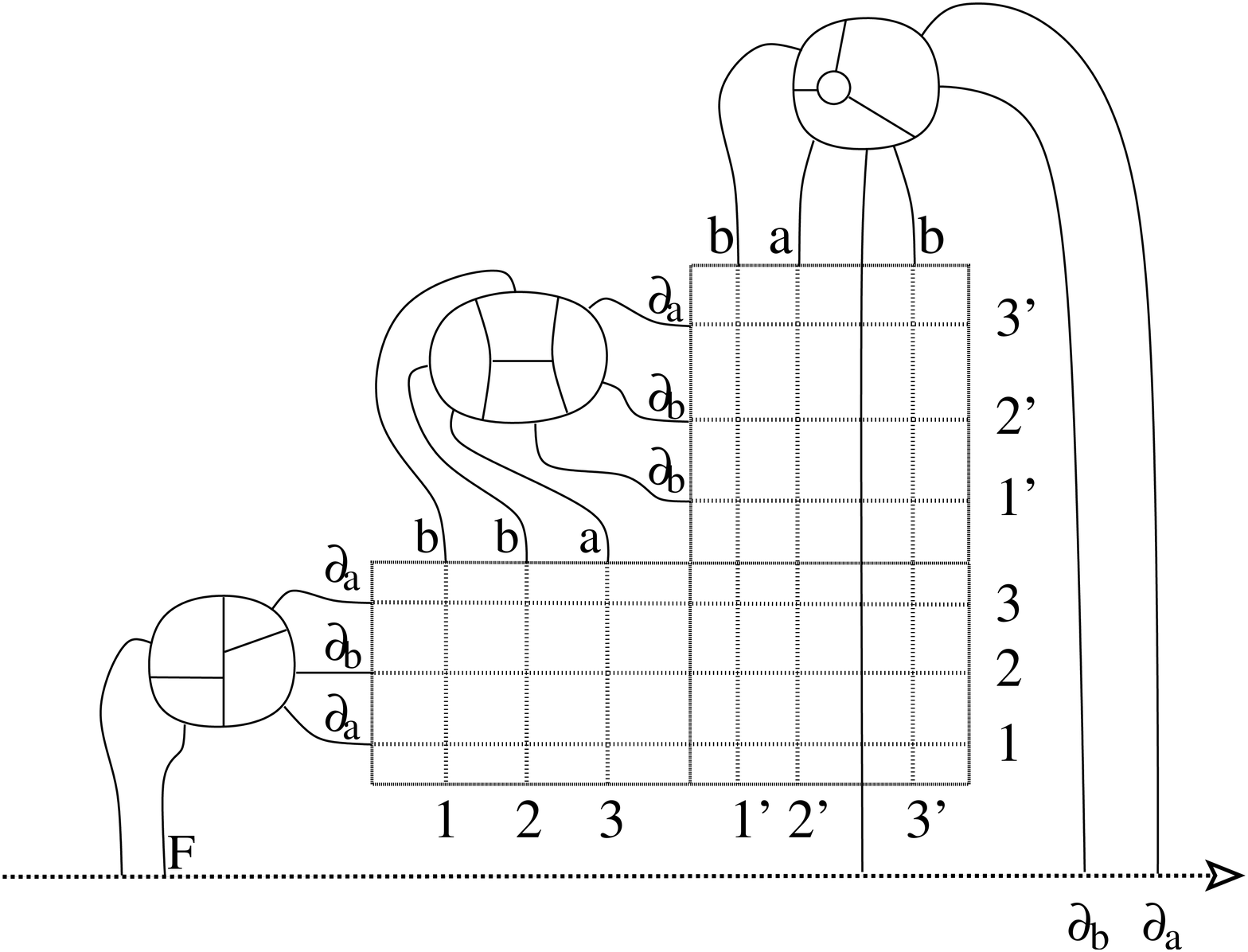}}
\]
%
To prove associativity we'll:
\begin{itemize}
\item{Show how to express $u\apply(v\apply w)$ as a sum of
diagrams built from this step-ladder grid.} \item{Show how to
express $(u\apply v)\apply w$ in the same way.} \item{Observe that
the terms of the two sums correspond.}
\end{itemize}

So focus first on the product $u\apply(v\apply w)$. If we follow
the definitions directly, we learn that this product has one term
for every pair of gluings (recalling that a gluing is a
parameter-respecting injection): \[
\begin{split} \rho & : \text{Op}(v)\supset K \to \text{Par}(w), \\
\rho' & : \text{Op}(u)\supset K'\to \text{Par}(v)\cup
\left(\text{Par}(w)\backslash\text{Im}(\rho)\right).
\end{split}
\]
Let the set of such pairs be denoted by $\mathcal{G}_{1}$. The
term $t_1(\rho,\rho')$ corresponding to some pair $(\rho,\rho')\in
\mathcal{G}_1$ is constructed by wiring up the top box of the
step-ladder using $\rho$, and then wiring up the bottom two boxes
using $\rho'$. The sign of the term, as usual, is a product of a
$(-1)$ for every intersection between full lines displayed by the
diagram.

For example, to construct the contribution $t_1\left(\left({1' 2'
\atop 1' 3'}\right),\left({1\, 2 \atop  2' 1}\right)\right)$, we
start by wiring up the top box using the gluing $\left({1' 2'
\atop 1' 3'}\right)$, which gives \[
\raisebox{-1cm}{\scalebox{0.2}{\includegraphics{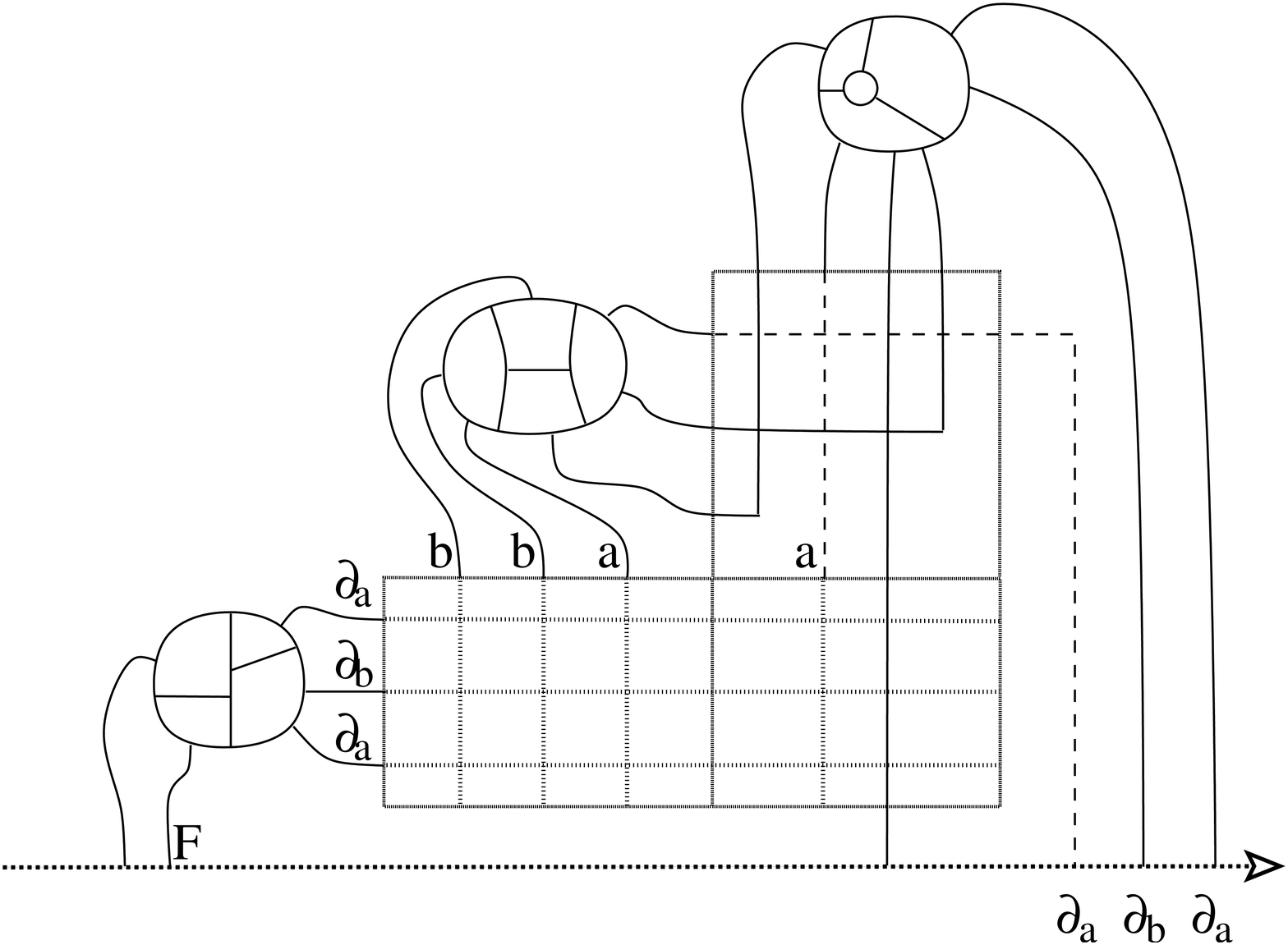}}}.
\]
Then we wire up the bottom two boxes using $\left({1\, 2 \atop  2'
1}\right)$, which gives
\[
\raisebox{-1cm}{\scalebox{0.2}{\includegraphics{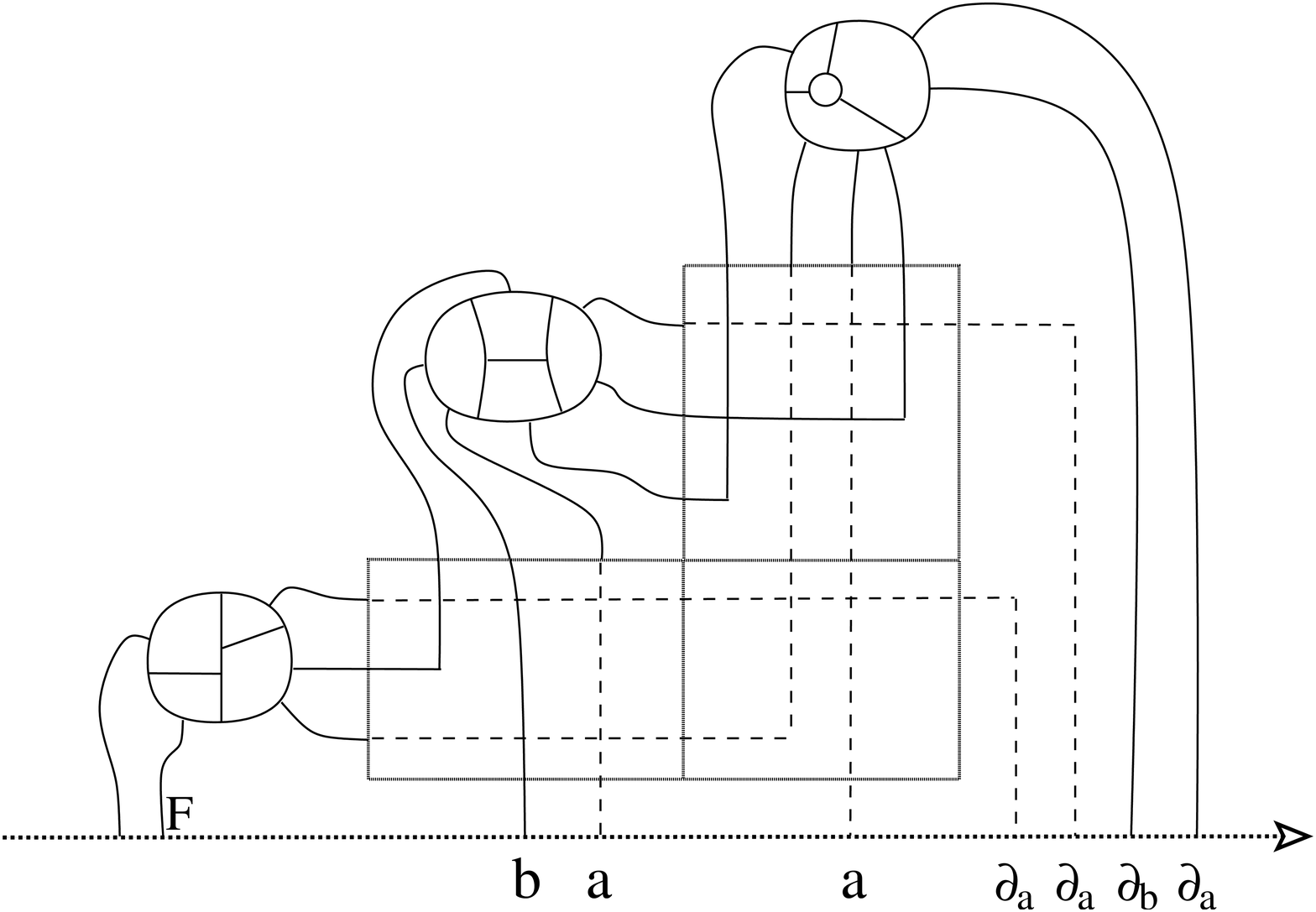}}}.
\]
Is is clear, as a direct application of the definitions, that
\[
u\apply (v\apply w) = \sum_{(\rho,\rho')\in
\mathcal{G}_1}t_1(\rho,\rho').
\]

Next we'll consider the other bracketting: $(u\apply v)\apply w$.
In this case the appropriate indexing set is $\mathcal{G}_2$, the
set of pairs of gluings
\[
\begin{split} \varsigma & : \text{Op}(u)\supset L \to \text{Par}(v), \\
\varsigma' & : \left(\left(\text{Op}(u)\backslash L\right)\cup
\text{Op}(v)\right) \supset L'\to \text{Par}(w).
\end{split}
\]
The corresponding expression is:
\[
(u\apply v)\apply w = \sum_{(\varsigma.\varsigma')\in
\mathcal{G}_2} t_2(\varsigma,\varsigma').
\]

To finish, note that there is an obvious correspondence
$C:\mathcal{G}_1 \simeq \mathcal{G}_2$ between the terms of the
two expressions:  $t_1(\rho,\rho') =
t_2\left(C(\rho,\rho')\right).$
\end{proof}

%

\subsection{Associativity and power series.}
We recommend this section for the second reading; it consists of
some unsurprising details about associativity and convergence.

In the computations that are the core of this work, we'll need to
re-bracket certain products of power series. We've just shown that
we can re-bracket products of the {\it generators}; to re-bracket
products of {\it power series} proves to be a more delicate affair
(because of convergence issues). To avoid getting bogged down by
the logic of our definitions, we'll introduce a simple finiteness
condition (Condition $(\S)$, below). When Condition $(\S)$ holds,
for a triple $u$, $v$ and $w$ of power series from $\WhatF\abpow$,
then it will be true that:
\[
(u\apply v)\apply w = u \apply (v\apply w).
\]
We'll state the
condition as a lemma:
\begin{lem}\label{powerseriesassociate}
Let $u$, $v$ and $w$ be power series from $\WhatF\abpow$. Assume
that:
\begin{itemize}
\item{The product $u\apply v$ converges.} \item{The product
$v\apply w$ converges.} \item{For all $(i,j)$, there are only
finitely many triples $((k,l),(m,n),(p,q))$ with the property that
\[
\pi^{(i,j)}\left( \left(u^{(k,l)} \apply v^{(m,n)} \right) \apply
w^{(p,q)}\right) \neq 0\ \ \ \ \ (\S).
\]
}
\end{itemize}
Then the products $(u\apply v) \apply w$ and $u\apply (v\apply w)$
converge, and:
\[
(u\apply v)\apply w=u\apply (v\apply w).
\]
\end{lem}
There is also a version for the other bracketting. It will be a
straightforward matter to check that this condition holds in any
situation that we perform a re-bracketting.
\begin{proof}
First, note that $(u\apply v)\apply w$ obviously converges.
(Otherwise Condition $(\S)$ would be violated.) Second, note that
because we can re-bracket the generators, Condition $(\S)$ implies
its re-bracketted version: that for each $(i,j)$ there are only
finitely many triples $((k,l),(m,n),(p,q))$ such that
\[
\pi^{(i,j)}\left( u^{(k,l)} \apply \left( v^{(m,n)}  \apply
w^{(p,q)}\right)\right) \neq 0\ \ \ \ (\S').
\]
Thus $u\apply(v\apply w)$ converges as well. Thirdly, note that
Condition $(\S)$ implies that the expression
\[
\sum_{(k,l),(m,n),(p,q)} \left(u^{(k,l)} \apply v^{(m,n)}\right)
\apply w^{(p,q)}
\]
makes sense. It is almost tautological that it is equal to
$(u\apply v)\apply w$. And similarly, $(\S')$ implies that
\[
u\apply (v\apply w) =\sum_{(k,l),(m,n),(p,q)} u^{(k,l)} \apply
\left(v^{(m,n)} \apply w^{(p,q)}\right).
\]
Associativity of products of generators gives the required
equality.
\end{proof}
The major reason we need such detail is that there are products
$(u\apply v)\apply w$ which don't satisfy Condition $(\S)$ but
which nevertheless converge. This makes general statements
difficult.


\section{Expressing the composition as an operator product.}\label{compopexp}
The computation that is the subject of Theorem
\ref{maincombinatorialtheorem} is based on an expression of the
value of the composition \[
\mathcal{B}\stackrel{\Upsilon}{\longrightarrow} \mathcal{W}
\stackrel{\chi_{\mathcal{W}}}{\longrightarrow}
\widetilde{\mathcal{W}}\stackrel{\pi}{\longrightarrow}
\widehat{\mathcal{W}} \stackrel{B_{\bullet\rightarrow\mathrm{\bf
F}}}{\longrightarrow} \widehat{\mathcal{W}}_{\mathrm{\bf F}}
\stackrel{\lambda}{\longrightarrow} \widehat{\mathcal{W}}_\wedge
\]
on some symmetric Jacobi diagram in terms of the operation product
$\apply$. The statement of the expression uses a linear map $
\text{B}_{l\mapsto\partial_a}: \Bspace\to \Whatwedge\abpow $ which
acts on a Jacobi diagram $v$ by first choosing an ordering of the
legs of $v$ and then labelling every leg with a $\partial_a$. For
example:
\[
\baselegtopartial\left(
\raisebox{-5ex}{\scalebox{0.25}{\includegraphics{illustrA}}}
\right) \ \ =\ \
\raisebox{-8ex}{\scalebox{0.25}{\includegraphics{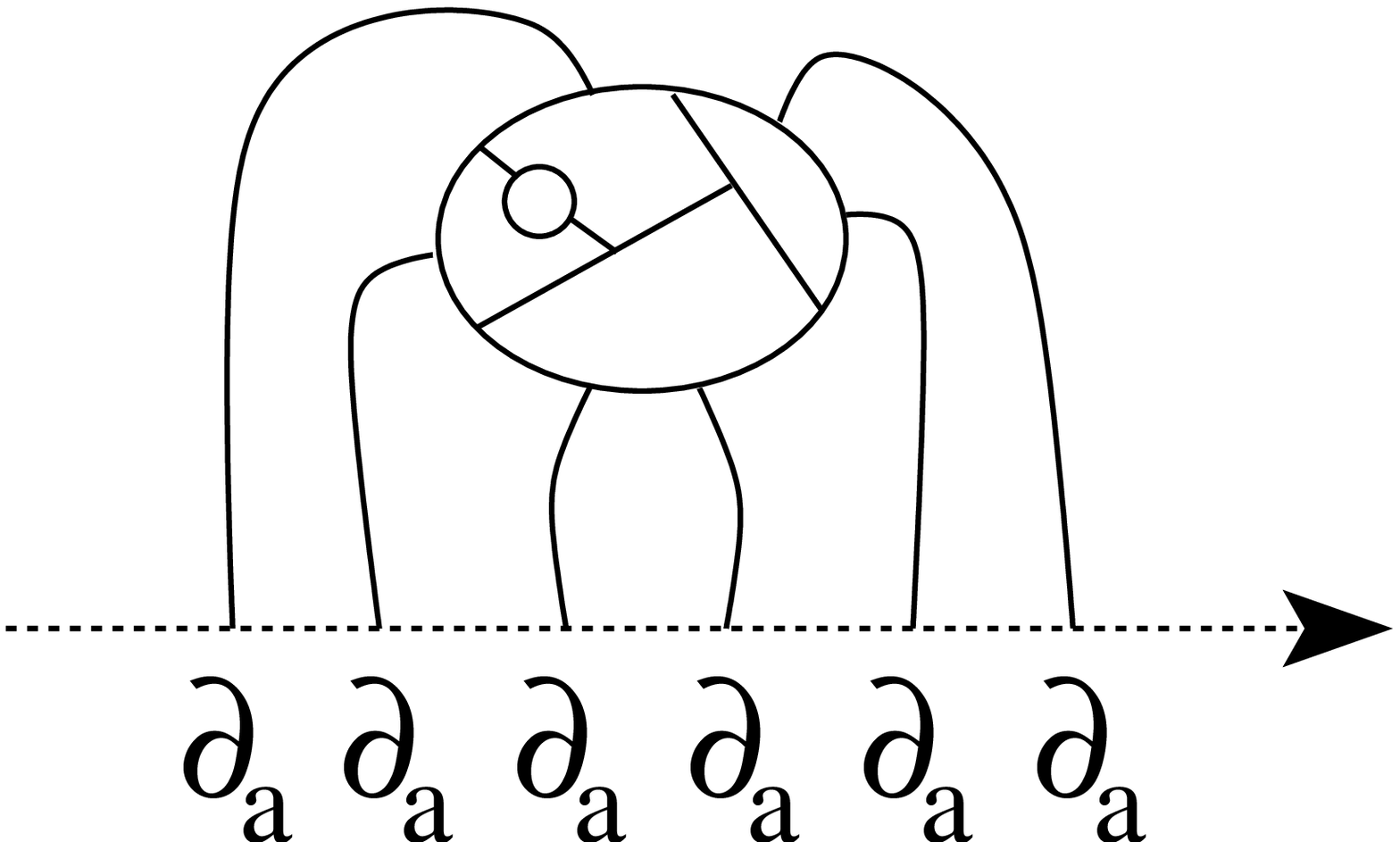}}}\,.\\[0.1cm]
\]
The purpose of this section is to prove the following theorem:
\begin{thm}
\label{importantexpression}
 \label{mainreexpressthm} Let $v$ be an
element of $\Bspace$, the space of symmetric Jacobi diagrams. Then
the element $ \left(\lambda\circ \basebulltoF \circ \pi \circ
\chi_\Wspace \circ \Upsilon\right)(v) $ is equal to
\[
\left[B_{l\mapsto\partial_a}(v)\,\apply \,\left(\exp_{\# }\left(
\raisebox{-3.5ex}{\scalebox{0.25}{\includegraphics{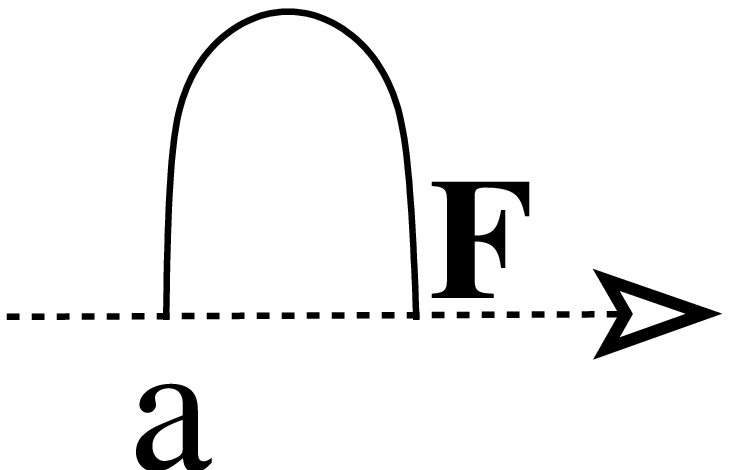}}}\right)
\#\, \mathcal{X} \right)\right]_{{a,b,\partial_a,\partial_b}=0}
\]
where $\mathcal{X}$ is equal to
\[
\exp_\apply
\left(-\frac{1}{2}\,\raisebox{-3.5ex}{\scalebox{0.25}{\includegraphics{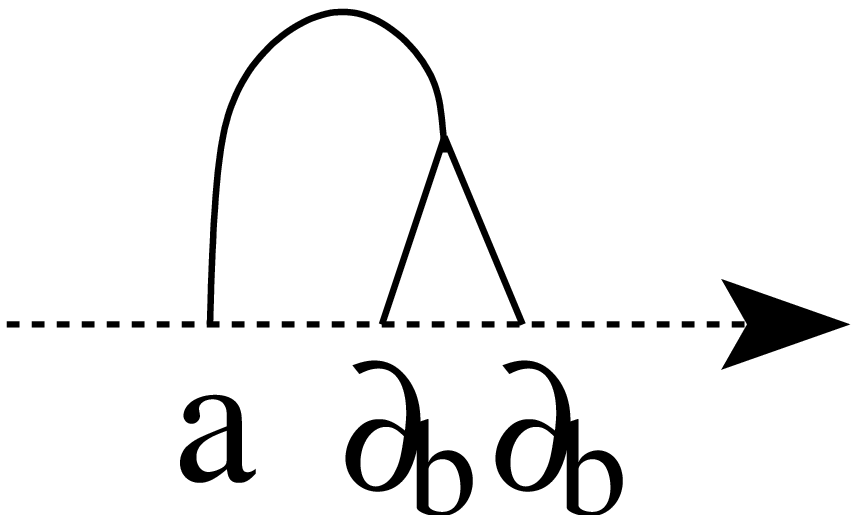}}}
\right) \apply \lambda\left(\exp_\#\left(\frac{1}{2}\,
\raisebox{-3.5ex}{\scalebox{0.25}{\includegraphics{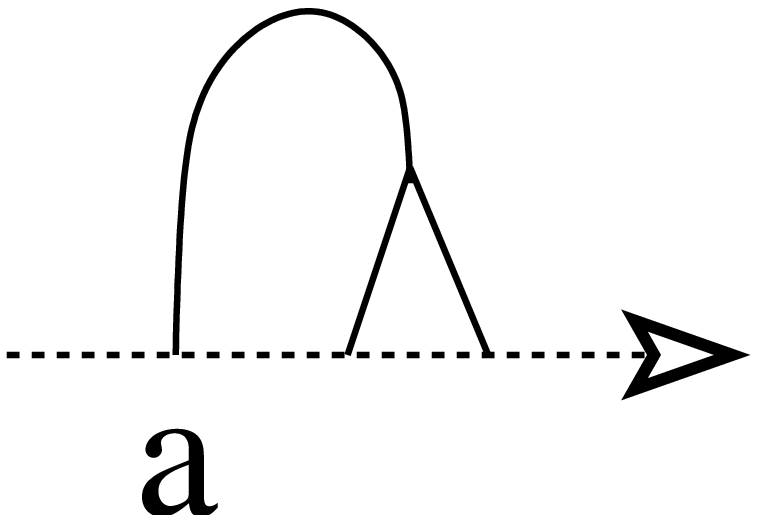}}} +
\raisebox{-3.5ex}{\scalebox{0.25}{\includegraphics{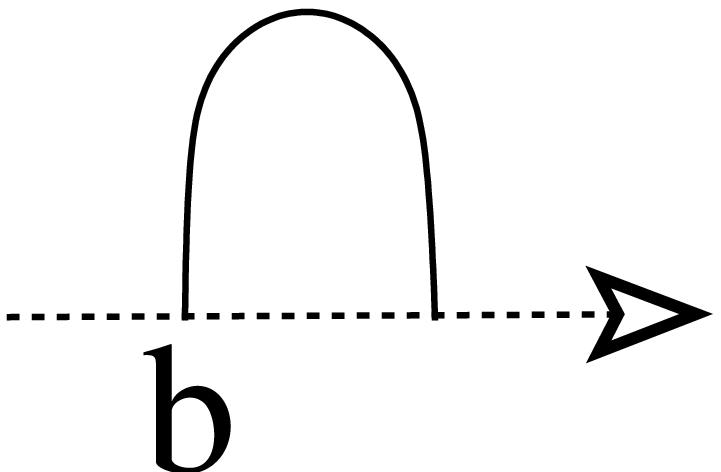}}}\right)
\right).
\]
\end{thm}

In Section \ref{computingtheoperatorexpressionA} we commence the
computation of $\mathcal{X}$ by performing the $\lambda$ operation
in the second factor above. Section
\ref{computingtheoperatorexpressionB} takes that result and
performs the remaining operation product.

To state the final result we'll employ a certain notation for the
$a$-labelled legs. Note that the $a$-labelled legs commute with
every type of leg, and so can be moved around freely. It proves
useful, then, to avoid drawing them in explicitly. We'll record
the $a$-labelled legs by (locally) orienting the edge they are
incident to, and labelling that edge with some power series in
$a$. For example:
\[
\raisebox{-1.5ex}{\scalebox{0.25}{\includegraphics{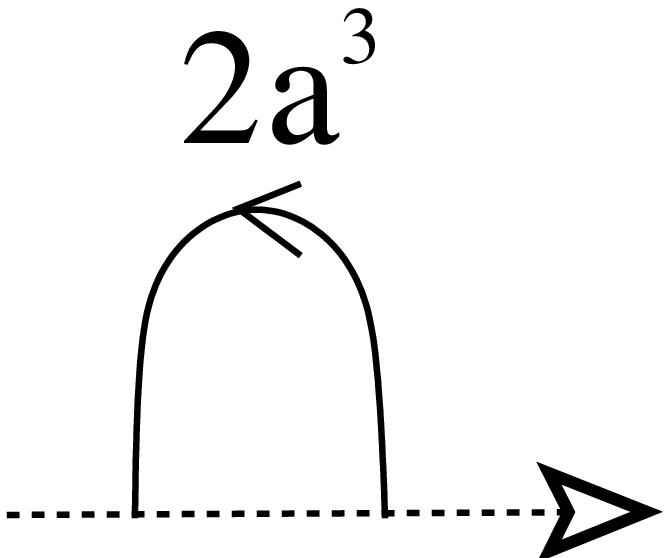}}}\ \ \
=\ \ \
2\raisebox{-3.8ex}{\scalebox{0.25}{\includegraphics{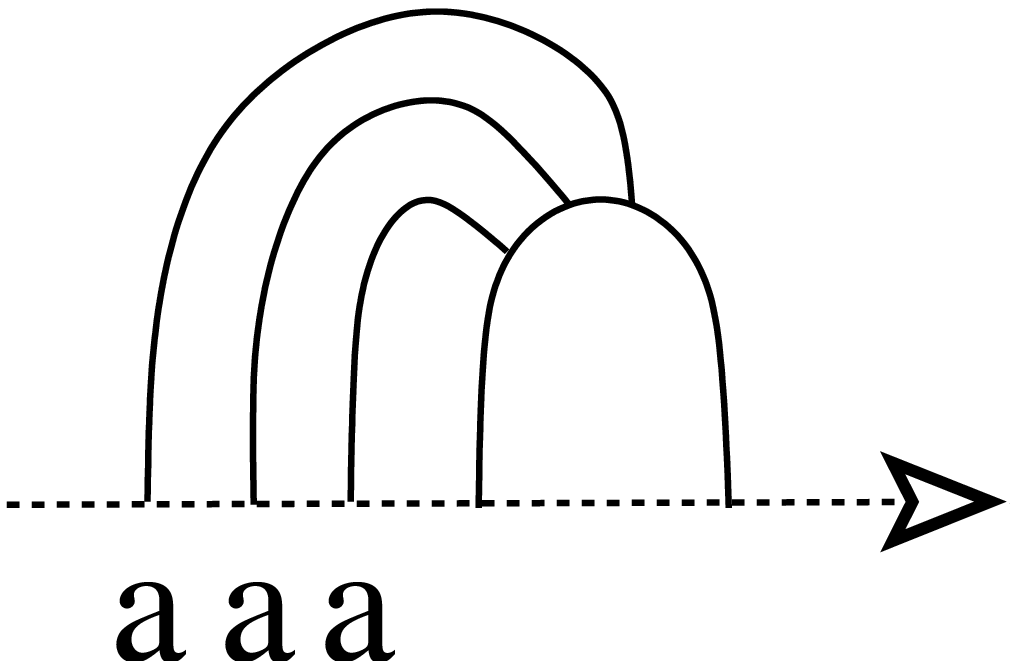}}}\ .
\]

Section \ref{computingtheoperatorexpressionB} completes the
computation that:
\begin{thm}\label{Xcalculate}
\[
\left[ \mathcal{X} \right]_{b,\partial_b=0} = \exp_\#\left(
\,\raisebox{-6ex}{\scalebox{0.25}{\includegraphics{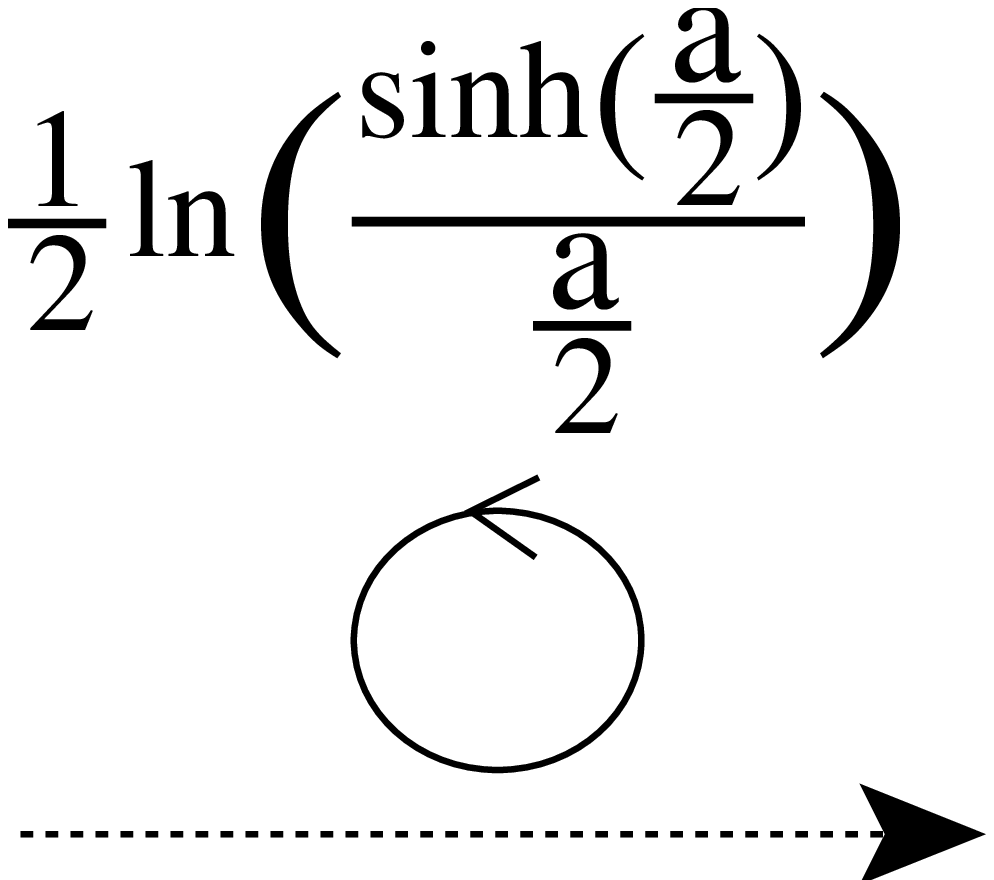}}}
\,\right).
\]
\end{thm}
Substituting this computation into Theorem
\ref{importantexpression} completes the proof of Theorem
\ref{maincombinatorialtheorem}, that for $v\in \Bspace$,
\[
\left(\phi_\Aspace^{-1}\circ\lambda\circ B_{\bullet\rightarrow
F}\circ\pi\circ\chi_\mathcal{W}\circ\Upsilon\right)(v) =
\left(\chi_\mathcal{B}\circ \partial_\Omega\right)(v).
\]

\subsection{Using operator diagrams to average.}
\label{operatorexpression} We'll build up to Theorem
\ref{mainreexpressthm} piece-by-piece. The construction begins
with the following piece:
\[ \mathcal{B}\stackrel{\Upsilon}{\longrightarrow}
\underbrace{\mathcal{W}
\stackrel{\chi_{\mathcal{W}}}{\longrightarrow}
\widetilde{\mathcal{W}}\stackrel{\pi}{\longrightarrow}
\widehat{\mathcal{W}}} \stackrel{B_{\bullet\rightarrow\mathrm{\bf
F}}}{\longrightarrow} \widehat{\mathcal{W}}_{\mathrm{\bf F}}
\stackrel{\lambda}{\longrightarrow} \widehat{\mathcal{W}}_\wedge.
\]
To present this piece, we'll turn a diagram $v\in \Wspace$ into an
operator diagram, and then have that diagram operate on an
exponential of parameters.
\begin{defn}
Define a linear map
\[
\intoop: \Wspace \rightarrow \What\abpow
\]
by replacing legs according to the rules
\[
\raisebox{-2ex}{\scalebox{0.25}{\includegraphics{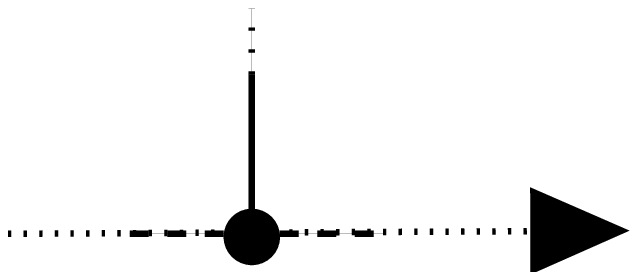}}}
\mapsto
\raisebox{-4.5ex}{\scalebox{0.25}{\includegraphics{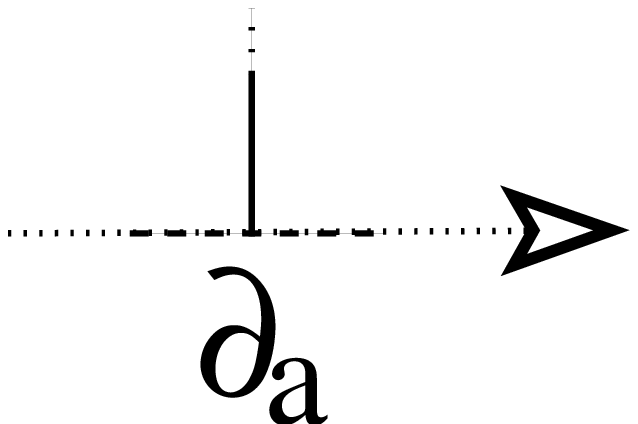}}}\ \
\ \mbox{and}\ \ \
\raisebox{-2ex}{\scalebox{0.25}{\includegraphics{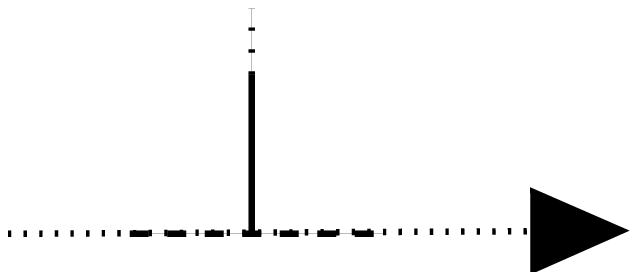}}}
\mapsto
\raisebox{-4.5ex}{\scalebox{0.25}{\includegraphics{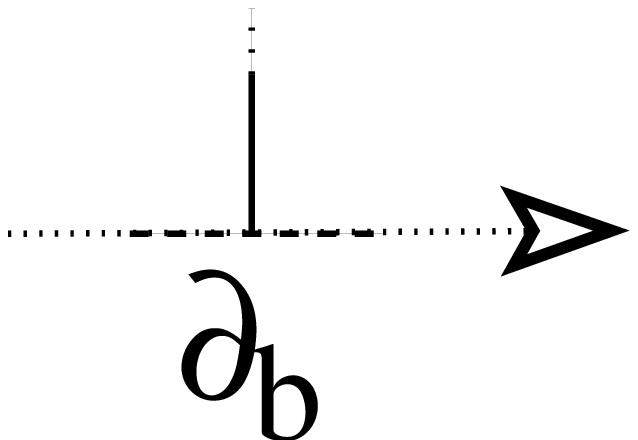}}}\ \
\ .
\]
\end{defn}
For example:
\[
\intoop\left(
\raisebox{-3ex}{\scalebox{0.25}{\includegraphics{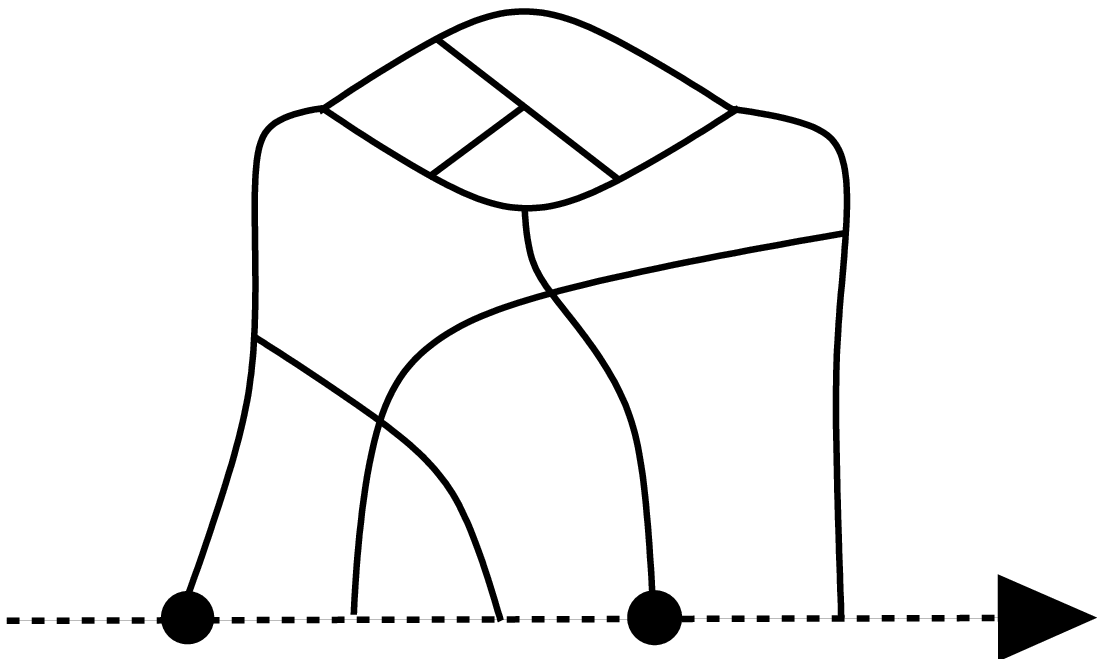}}}
\right) \ =\
\raisebox{-5.2ex}{\scalebox{0.25}{\includegraphics{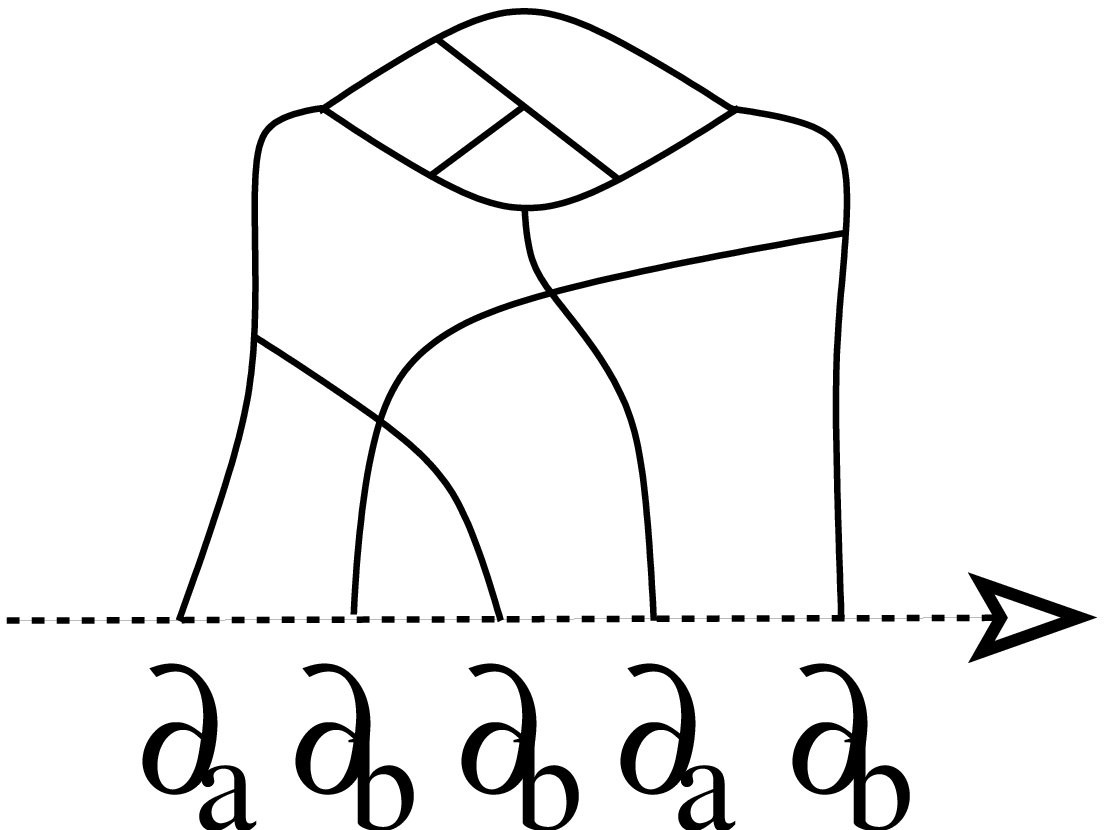}}}\ .
\]
The key proposition follows. It says that we can take the signed
average of a diagram by changing it into an operator, applying the
resulting operator to an exponential of formal parameters, then
setting all parameters to zero. The map which sets all the
parameters to zero, denoted $[.]_{a,b,\partial_a,\partial_b=0}$
below, is precisely pr$^{(0,0)}$, the projection of the $(0,0)$
factor out of the power series.
\begin{prop}\label{symmpro}
Let $v\in\mathcal{W}$. Then
\begin{equation}\label{symprop}
(\pi\circ\chi_\mathcal{W})(v)= \left[\intoop(v)\apply
\left(\exp_{\applyhash }\left(
\raisebox{-3ex}{\scalebox{0.25}{\includegraphics{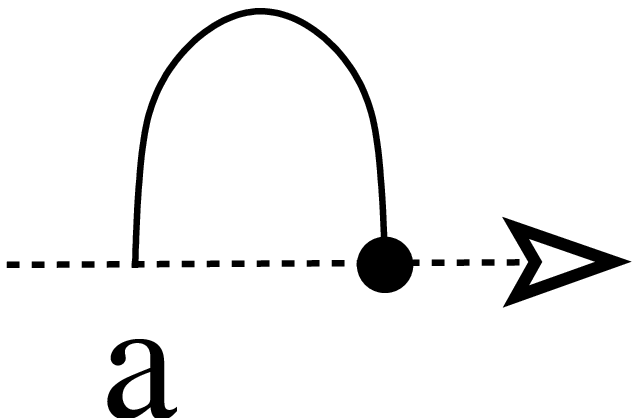}}} +
\raisebox{-3ex}{\scalebox{0.25}{\includegraphics{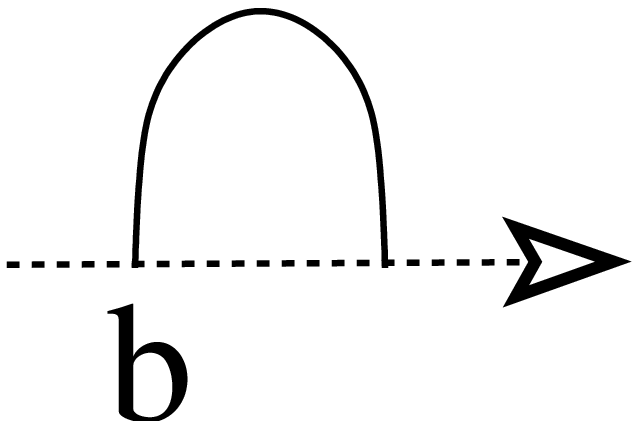}}}
\right)\right)\right]_{a,b,\partial_a,\partial_b=0}.
\end{equation}
\end{prop}
\begin{proof}
 Both sides are linear
maps, so it suffices to check this formula on generators. So take
some symmetric Weil diagram $v$, and assume that it has $p$ grade
2 legs and $q$ grade 1 legs. We'll evaluate the value that the
right-hand side takes on $v$ and observe that it is precisely the
signed average of $v$, as required.

For convenience, write
\[
v_{\mathrm{op}}\ \ \mbox{for}\ \ \ \intoop(v).
\]
Consider, then, the exponential that $v_{\mathrm{op}}$ is to be
applied to. We'll index its terms by the set of words (including
the empty word) that can be built from the symbols $A$ and $B$.
Given some such word $w$, let $f_w$ denote the corresponding
diagram. For example:
\[
f_{BABBA}\ =\
\raisebox{-4.5ex}{\scalebox{0.25}{\includegraphics{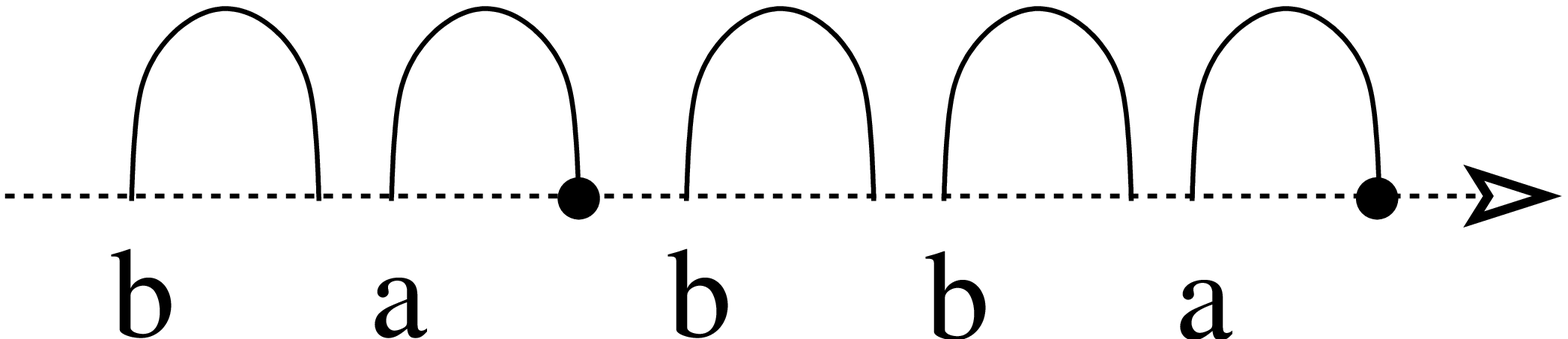}}}\
.
\]
Thus we can expand the right-hand-side of equation
\ref{symprop} to get:
\begin{equation}\label{expandedexp} \left[ v_{\mathrm{op}} \apply  \left(
\raisebox{0.1ex}{$\sum_{
\begin{subarray}{c}
\mathrm{Words}\ w\ \mathrm{made\ from} \\
\mathrm{the\ symbols}\ A\ \mathrm{and}\
B.\end{subarray}}\frac{1}{|w|!} f_{w}$} \right)
\right]_{a,b,\partial_a,\partial_b=0}.
\end{equation} Because $v_{\mathrm{op}}$ has exactly $p$
legs labelled by $\partial_a$  and exactly $q$ legs labelled by
$\partial_b$, the only terms that will survive when all the
parameters are set to zero will arise from the terms $f_w$
corresponding to words $w$ built from exactly $p$ copies of $A$
and exactly $q$ copies of $B$. Restricting expression
\ref{expandedexp} to these terms we write: \begin{equation}
\label{ees} \frac{1}{(p+q)!} \sum_{
\begin{subarray}{c}
\mathrm{Words}\ w\ \mathrm{built}\ \\ \mathrm{from}\ p\
\mathrm{copies}\ \mathrm{of}\ A\ \\ \mathrm{and}\  q\
\mathrm{copies}\ \mathrm{of}\ B.
\end{subarray}}
\left[v_{\mathrm{op}}\apply  f_{w}
\right]_{a,b,\partial_a,\partial_b=0}.
\end{equation}
Fix, then, such a word $w$ and let's proceed to calculate $[v_\mathrm{op}\apply
f_w]_{a,b,\partial_a,\partial_b=0}$. For example, if
\[
v =\
\raisebox{-3.1ex}{\scalebox{0.25}{\includegraphics{makediffA}}}\ \
\
\]
and $w=BABBA$, then we wish to calculate
\[
\left[
\left(\raisebox{-5.5ex}{\scalebox{0.23}{\includegraphics{makediffB}}}\right)
\apply
\raisebox{-5ex}{\scalebox{0.23}{\includegraphics{exptermexamp}}}
\right]_{a,b,\partial_a,\partial_b=0}.
\]
To do such a computation directly, it helps to employ the
graphical method for doing diagram operations that was described
in Section \ref{alternativedefn}.

Let us take a moment to recall this method. We begin by placing
the operator legs of $v_{\mathrm{op}}$ up the left-hand side of a grid, and the legs of
$f_w$ along the top of the grid: \[
\scalebox{0.25}{\includegraphics{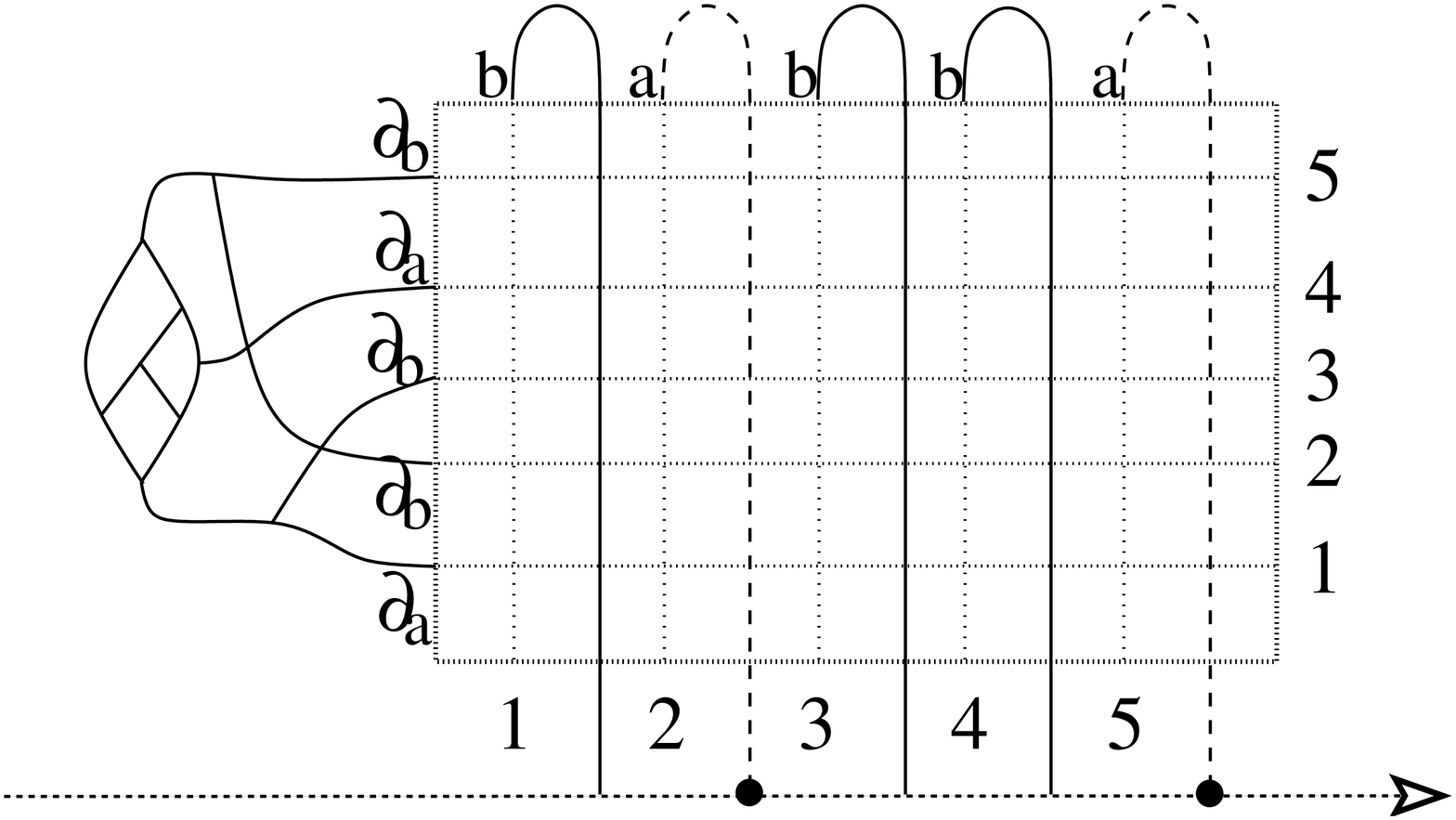}}
\]
According to this method, we get precisely one contribution
$t(v_{\mathrm{op}},f_{w},\sigma)$ to the operation for every
permutation
\[
\sigma : \{1,2,3,4,5\} \to \{1,2,3,4,5\}
\]
which respects the parameters. (To be precise, these are the contributions that will survive after the parameters get sent to zero at the end of the operation.) So if we let Perm$_n$
denote the set of all permutations on $\{1,\ldots,n\}$ and let
$\text{Perm}_{p+q}(v,w)\subset\text{Perm}_{p+q}$ denote the set of
permutations respecting the parameters, then we can write
\begin{equation}\label{opcontr}
\left[v_{\mathrm{op}}\apply
f_w\right]_{a,b,\partial_a,\partial_b=0} = \sum_{\sigma \in
\mathrm{Perm}_{p+q}(v,w)} t(v_{op},f_w,\sigma),
\end{equation}
where $t(v_{op},f_w,\sigma)$ is determined by the usual graphical
method. For example, the term corresponding to the permutation
$\left({1 2 3 4 5 \atop 2 4 3 5 1}\right)$ is determined by
joining up legs in the following way:
\[
\scalebox{0.25}{\includegraphics{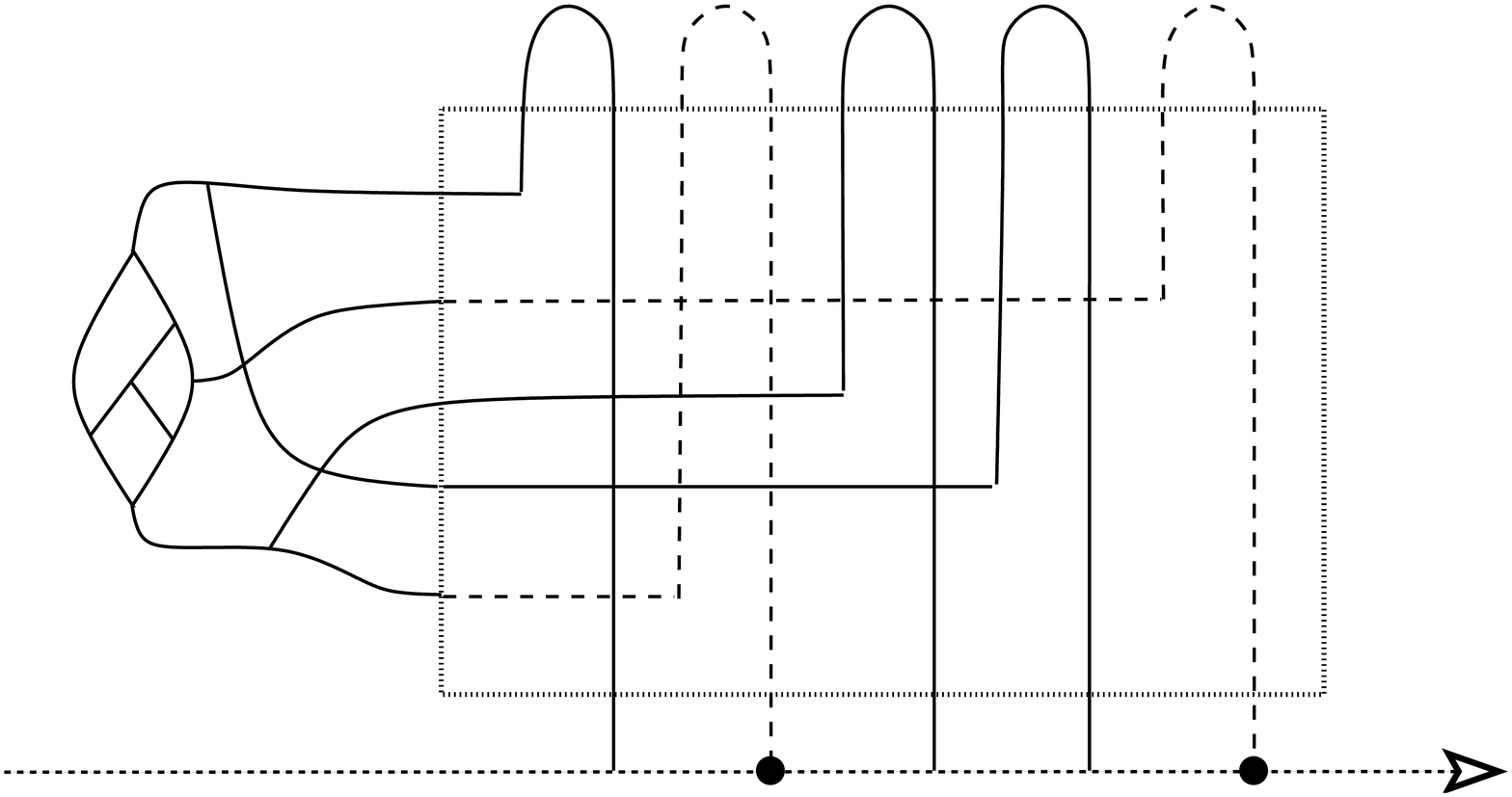}}
\]
Recall that the sign of the contribution is determined by counting
the number of intersections between full lines displayed within
the box. Thus: \[
 t\left(v_{\mathrm{op}},f_w,\left({1 2 3 4 5 \atop 2 4 3 5
1}\right)\right)\ =\
(-1)^3\raisebox{-6ex}{\scalebox{0.25}{\includegraphics{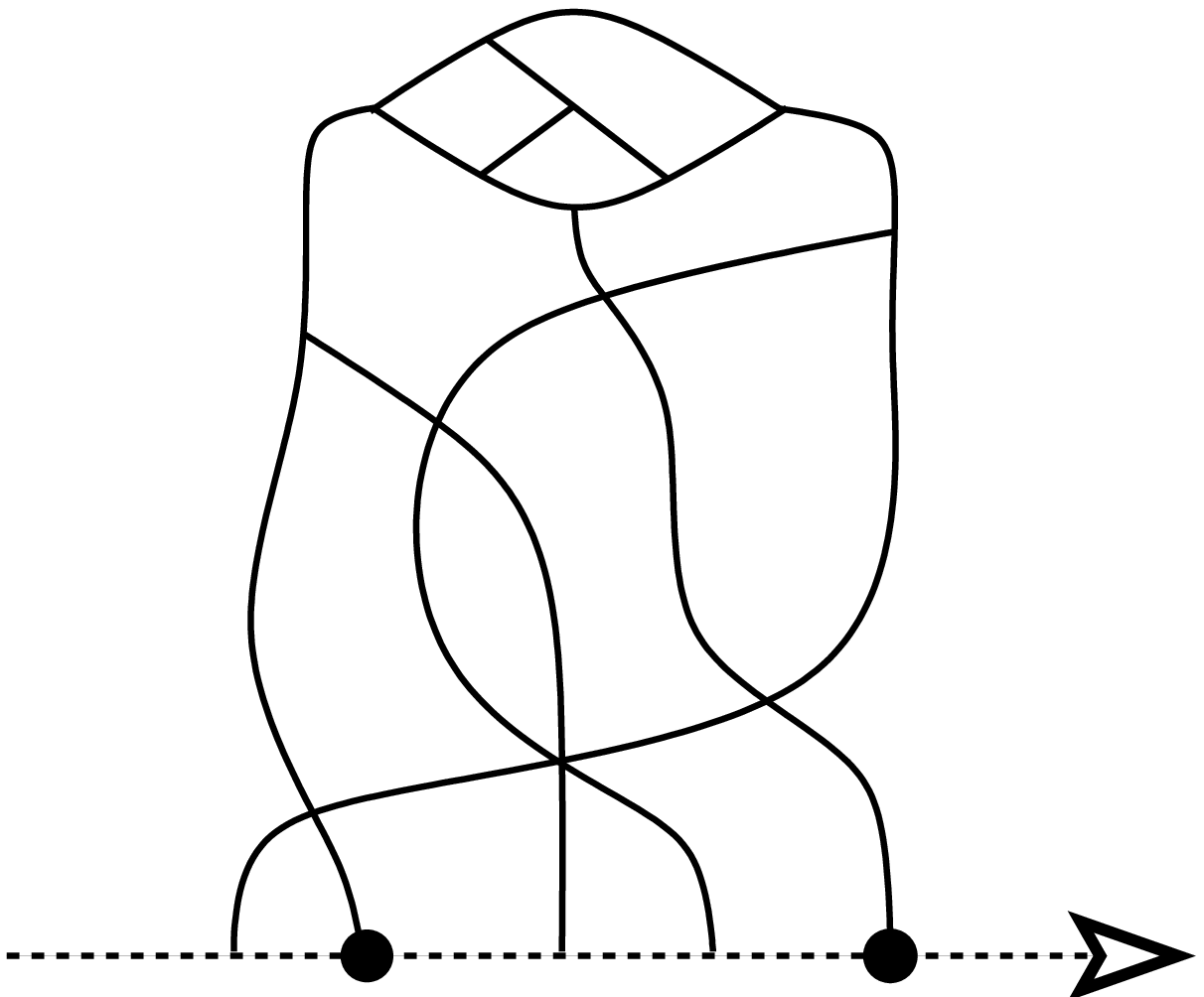}}}\
.
\]
Here is what we need to notice: the diagram
$t(v_\mathrm{op},f_w,\sigma)$ put out by this operation is (up to
some to-be-determined sign) precisely what we get by permuting the
legs of the original diagram using the permutation $\sigma$. To
state this observation precisely, let's introduce some notation:
Given a permutation $\sigma\in \text{Perm}_{p+q}$, let $v^\sigma$
denote the diagram one gets by permuting (without introducing
signs) the legs of $v$ according to $\sigma$. For example, if
\[
v = \raisebox{-4ex}{\scalebox{0.25}{\includegraphics{makediffA}}}\
\ \text{then}\ \
v^{\left( \begin{subarray}{c} 1  2  3  4  5 \\
1 3  5  2  4 \end{subarray} \right)}\ =\ \
\raisebox{-4ex}{\scalebox{0.25}{\includegraphics{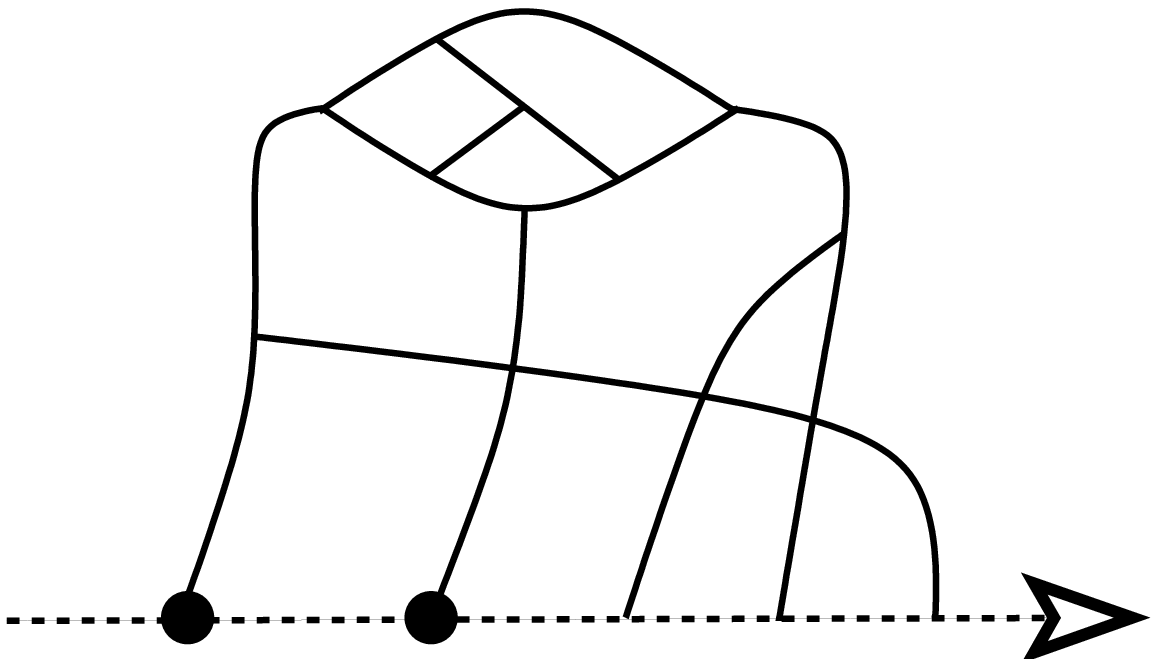}}}\ .
\]
The required observation is:
\begin{lem}
Let $v$ be some diagram from $\Wspace$ with $p$ grade $2$ legs and
$q$ grade $1$ legs. Let $w$ be some word which uses $p$ copies of
the symbol $A$ and $q$ copies of the symbol $B$. Let $\sigma$ be
some permutation from $\text{Perm}_{p+q}(v,w)$. Then:
\[
t(v_{\mathrm{op}},f_w,\sigma) = \varphi(\sigma)v^\sigma,
\]
where $\varphi(\sigma)$ is the product of a $(-1)$ for every pair
of grade $1$ legs of $v$ which reverse their order in
$v^{\sigma}$.
\end{lem}
{\it Proof of the Lemma:} The fact that
$t(v_{\mathrm{op}},f_w,\sigma)$ is some sign multiplied by
$v^\sigma$ is quite clear. The subtlety is in the sign. This is
quite clear too, once we regard the problem from the point of view
of the visual method of doing the gluing. Here is the general
picture of a gluing:
\[
\scalebox{0.30}{{\includegraphics{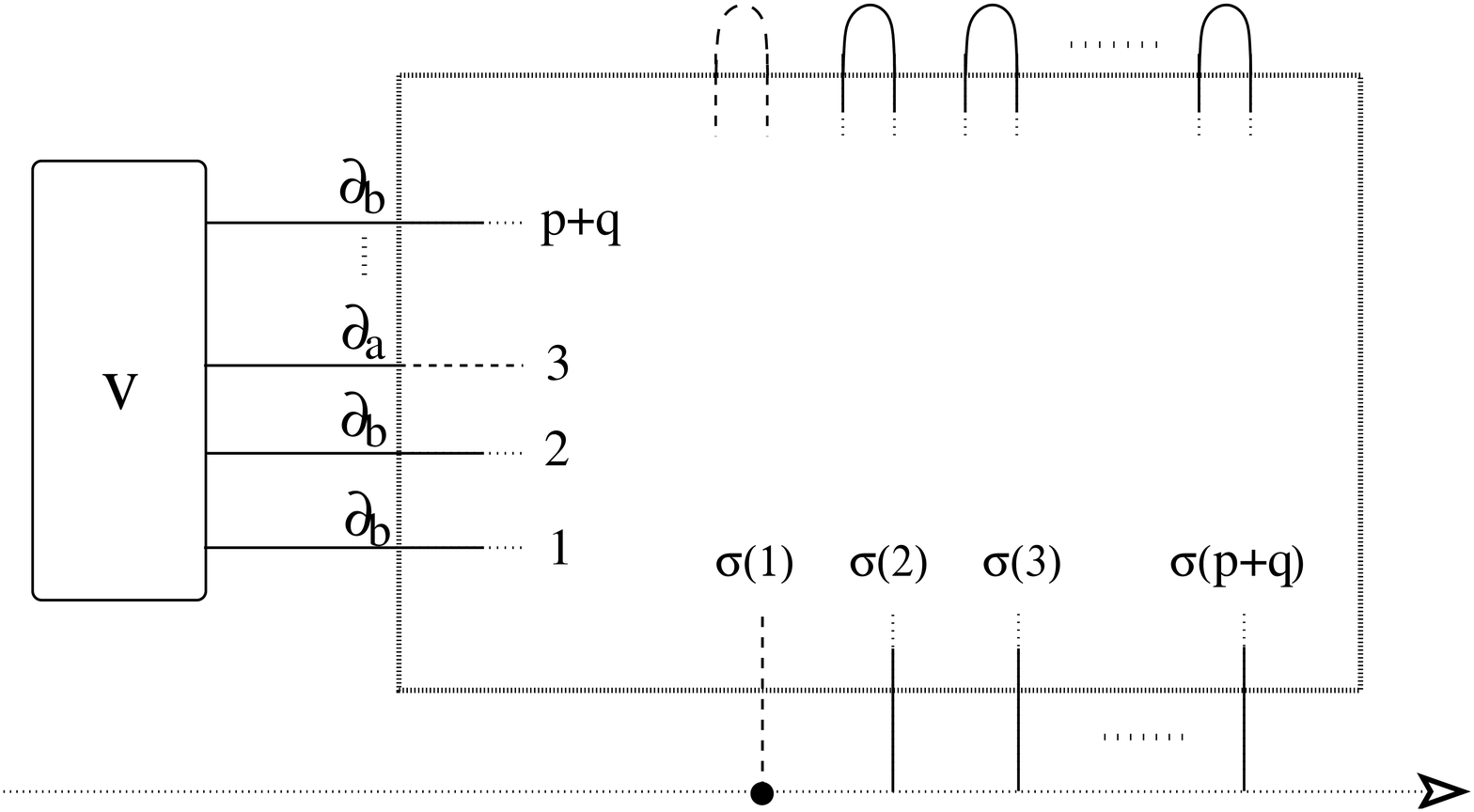}}}\ \ \ .
\]
The legs of $v$ appear in their original order up the left-hand
side of the grid. From there, they continue through the box (and the arcs at
the top) along paths determined by the gluing, then finish along the
orienting line at the bottom in the order determined by the
permutation $\sigma$. The sign we are trying to determine is
$(-1)$ raised to the number of intersections between full lines
displayed inside the box. The topology of the arrangement
requires that the number of intersections between full lines is
given, modulo 2, by the stated formula. {\it End of the proof of the Lemma.}

Now we can put it all together. According to Equation \ref{ees}
and Equation \ref{opcontr}, the right-hand side of the Equation
\ref{symprop} can be written \[ \frac{1}{(p+q)!} \sum_{
\begin{subarray}{c}
\mathrm{Words}\ w\ \mathrm{built}\ \\ \mathrm{from} p\
\mathrm{copies}\ \mathrm{of}\ A\ \\ \mathrm{and}\  q\
\mathrm{copies}\ \mathrm{of}\ B.
\end{subarray}}
\sum_{\sigma \in \mathrm{Perm}_{p+q}(v,w)} t(v_{op},f_w,\sigma).
\]
Substituting the result of the lemma into this, we get:
\[
\frac{1}{(p+q)!} \sum_{
\begin{subarray}{c}
\mathrm{Words}\ w\ \mathrm{built}\ \\ \mathrm{from} p\
\mathrm{copies}\ \mathrm{of}\ A\ \\ \mathrm{and}\  q\
\mathrm{copies}\ \mathrm{of}\ B.
\end{subarray}}
\sum_{\sigma \in \mathrm{Perm}_{p+q}(v,w)}
\varphi(\sigma)v^\sigma\ .
\]
This, of course, is just $\pi\circ\chi_\Wspace$, the required
graded averaging map.
\end{proof}

It is a simple step to extend this proposition to give a formula
for the following piece of the composition
\[ \Bspace\stackrel{\Upsilon}{\longrightarrow}
\underbrace{\ \mathcal{W}
\stackrel{\chi_{\mathcal{W}}}{\longrightarrow}
\widetilde{\mathcal{W}}\stackrel{\pi}{\longrightarrow}
\widehat{\mathcal{W}} \stackrel{B_{\bullet\mapsto\mathrm{\bf
F}}}{\longrightarrow} \widehat{\mathcal{W}}_{\mathrm{\bf F}}}
\stackrel{\lambda}{\longrightarrow} \widehat{\mathcal{W}}_\wedge.
\]

\begin{cor}\label{comblemcor}
Let $v\in\mathcal{W}$. Then
$\left(\basebulltoF\circ\pi\circ\chi_{\mathcal{W}}\right)(v)$ is
given by the expression
\[
\left[\intoop(v)\apply  \left(\exp_{\# }\left(
\raisebox{-3.5ex}{\scalebox{0.25}{\includegraphics{paramj}}}
+\frac{1}{2}
\raisebox{-3.5ex}{\scalebox{0.25}{\includegraphics{paramk}}} +
\raisebox{-3.5ex}{\scalebox{0.25}{\includegraphics{paraml}}}
\right)\right)\right]_{a,b,\partial_a,\partial_b=0}.
\]
\end{cor}

\subsection{Hair-splitting with operator diagrams.}
We now turn our focus to the first step in the composition:
\[ \underbrace{\ \mathcal{B}\stackrel{\Upsilon}{\longrightarrow}
{\mathcal{W}}} \stackrel{\chi_{\mathcal{W}}}{\longrightarrow}
\widetilde{\mathcal{W}}\stackrel{\pi}{\longrightarrow}
\widehat{\mathcal{W}} \stackrel{B_{\bullet\rightarrow\mathrm{\bf
F}}}{\longrightarrow} \widehat{\mathcal{W}}_{\mathrm{\bf F}}
\stackrel{\lambda}{\longrightarrow} \widehat{\mathcal{W}}_\wedge.
\]

\begin{prop}\label{howtosplit}
Let $v\in\mathcal{B}$. Then the equation
\[
\baselegtopartial(v) \apply
\mathrm{exp}_\apply\left(-\frac{1}{2}\,
\raisebox{-3.5ex}{\scalebox{0.25}{\includegraphics{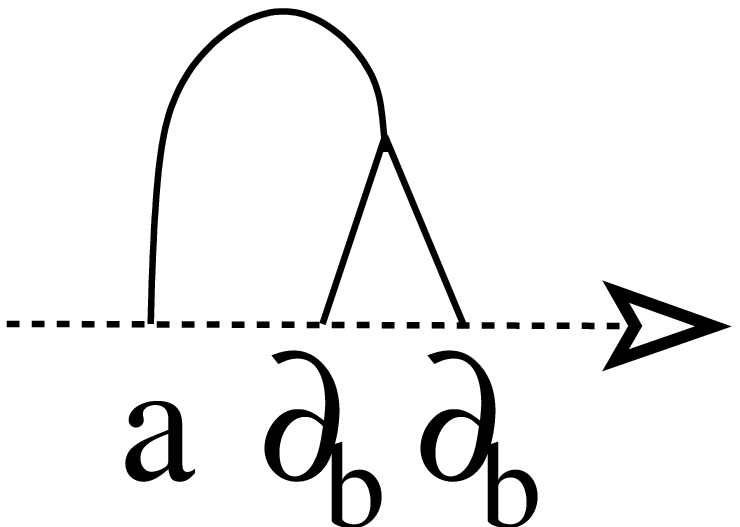}}}
\right) = \mathrm{exp}_\apply \left(-\frac{1}{2}\,
\raisebox{-3.5ex}{\scalebox{0.25}{\includegraphics{paramZ}}}
\right)\apply  \intoop\left(\Upsilon(v)\right)
\]
holds in $\WhatF\abpow$.
\end{prop}
\begin{proof} As the two sides of this equation are both linear maps, it
suffices to show that the equation holds for generators of $\mathcal{B}$. So let
$v$ be a symmetric Jacobi diagram. We begin by expanding the
left-hand side of Equation \ref{howtosplit} in the following way:
\[ \sum_{n=0}^{\infty}\left(\left(-\frac{1}{2}\right)^n \frac{1}{n!}\right)
B_{l\mapsto\partial_a}(v)\apply
\left(\raisebox{-6ex}{\scalebox{0.25}{\includegraphics{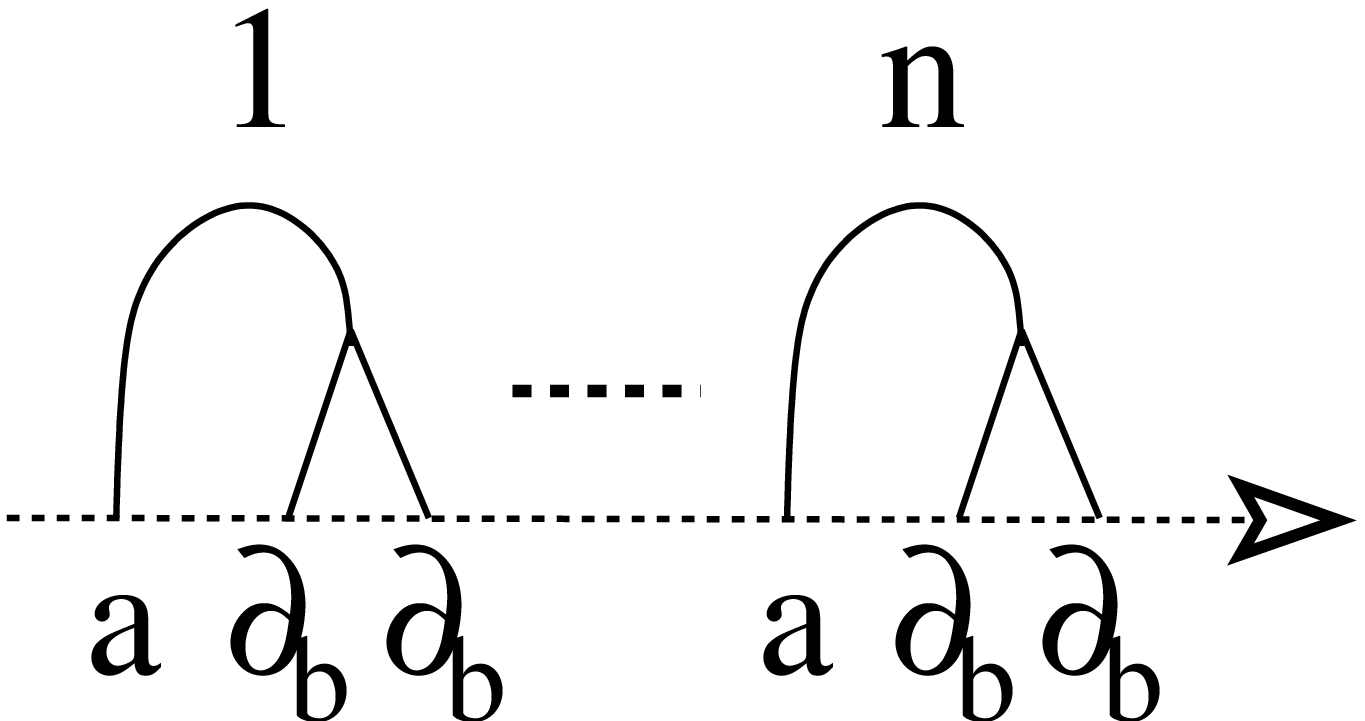}}}\right)\
\ .
\]
Now consider the diagram operation in the above sum. Recalling how
to do diagram operations (see Section \ref{alternativedefn}), we
get (letting  $\mathcal{L}$ denote the set of the legs of $v$):
\[ \sum_{n=0}^{\infty}\left(\left(-\frac{1}{2}\right)^n
\frac{1}{n!}\right) \sum_{{\mathrm{Injections}\atop
\phi:\mathcal{L} \supset S \rightarrow \{1,\ldots,n\} }} \left(
\raisebox{-6ex}{\scalebox{0.25}{\includegraphics{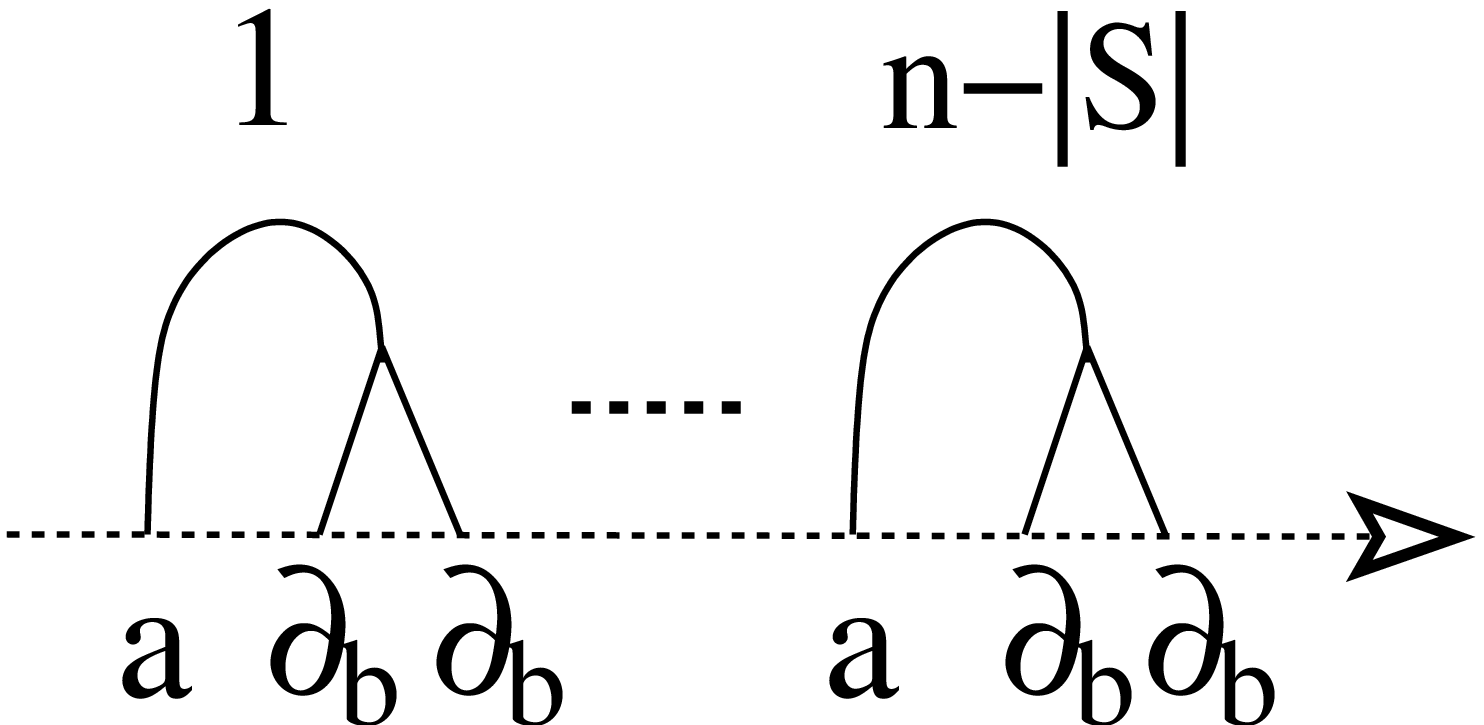}}}\right)\apply
v_S \ ,
\]
where, for example,
\[
\mbox{if}\ \ v\ =\
\raisebox{-6ex}{\scalebox{0.25}{\includegraphics{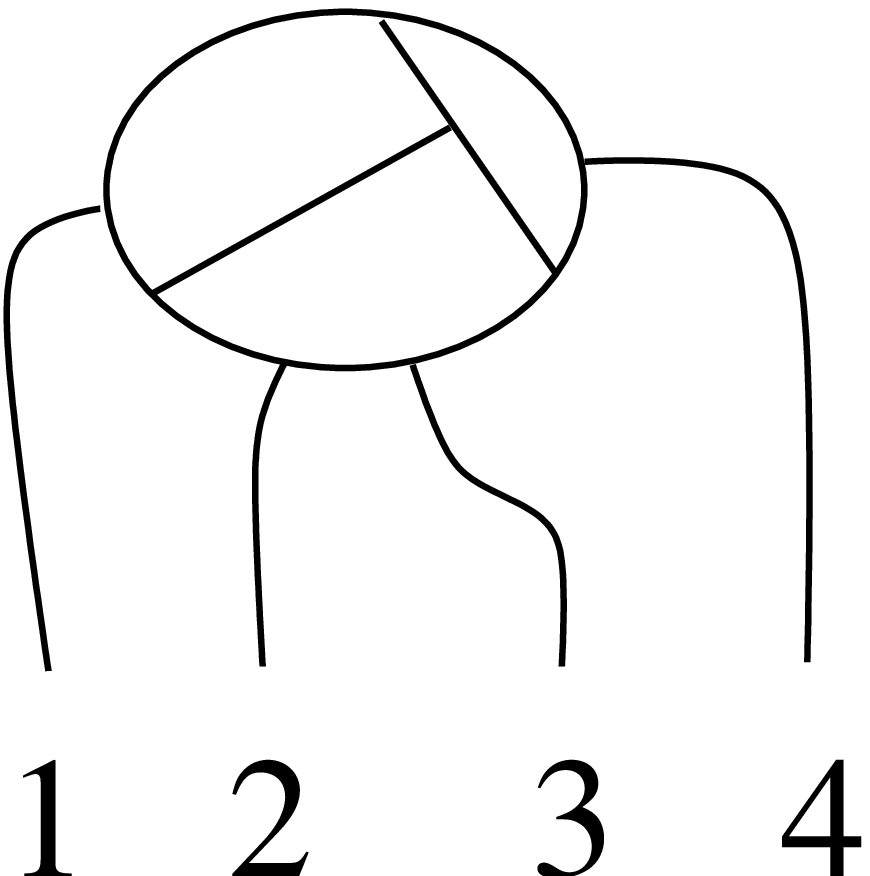}}}\ \
\ \mbox{then}\ \ v_{\{1,3,4\}}\ =\
\raisebox{-8ex}{\scalebox{0.25}{\includegraphics{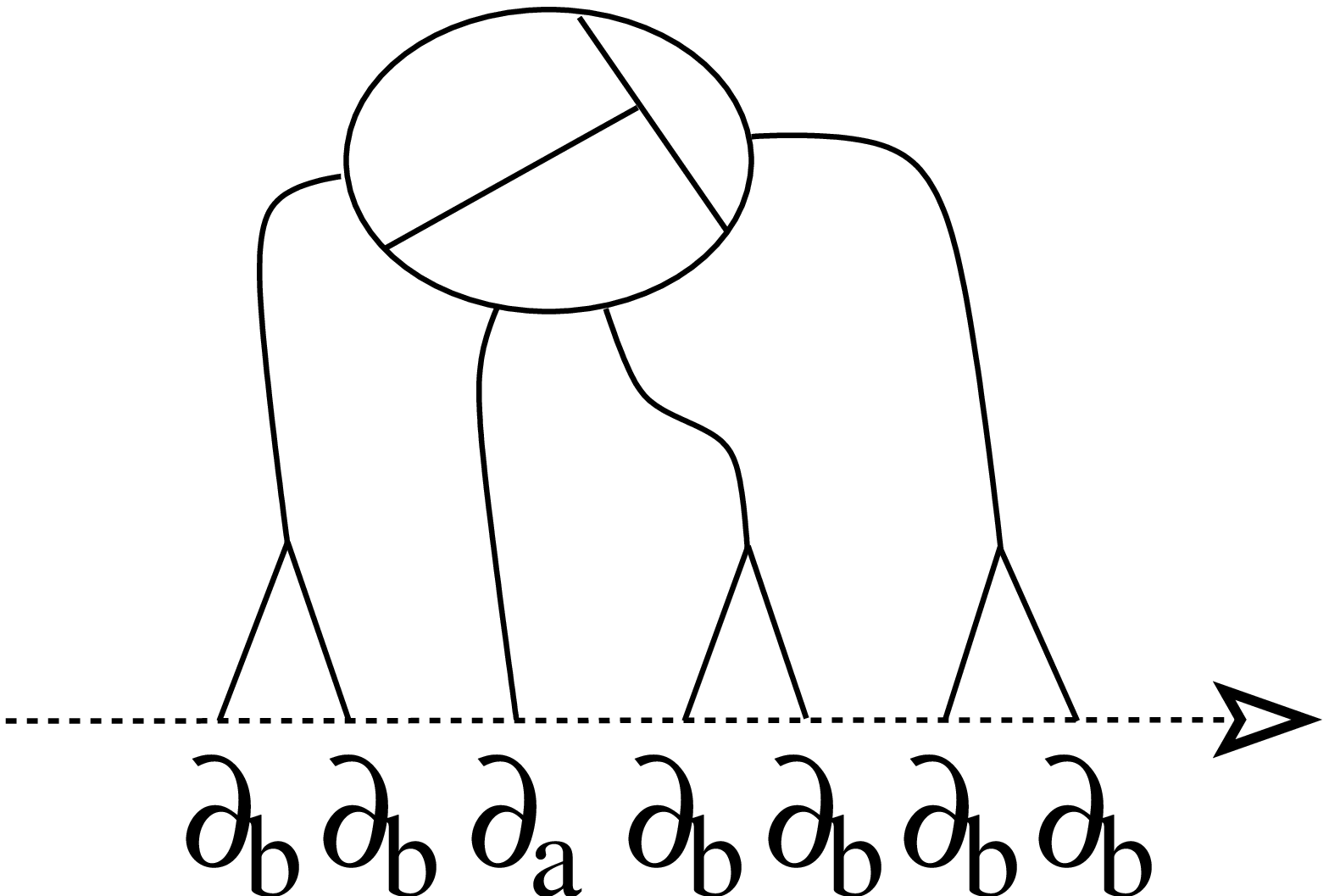}}}\
.
\]
Now for every subset $S\subset \mathcal{L}$ of legs there are
$\frac{n!}{(n-|S|)!}$ injections $\phi:S\rightarrow
\{1,\ldots,n\}$. Thus the above expression may be rewritten
\begin{eqnarray*}
& &  \sum_{n=0}^{\infty}\left(\left(-\frac{1}{2}\right)^n
\frac{1}{n!}\right) \sum_{{\mathrm{Subsets}\ S\subset
\mathcal{L}} \atop \mathrm{with}\ |S|\leq n} \frac{n!}{(n-|S|)!} \left(
\raisebox{-5ex}{\scalebox{0.25}{\includegraphics{paramZE}}}\right)\apply
v_S \ \\[0.2cm]
 & = &
\left( \sum_{p=0}^{\infty}
 \left( \left(-\frac{1}{2}\right)^p \frac{1}{p!}\right)
\raisebox{-5ex}{\scalebox{0.25}{\includegraphics{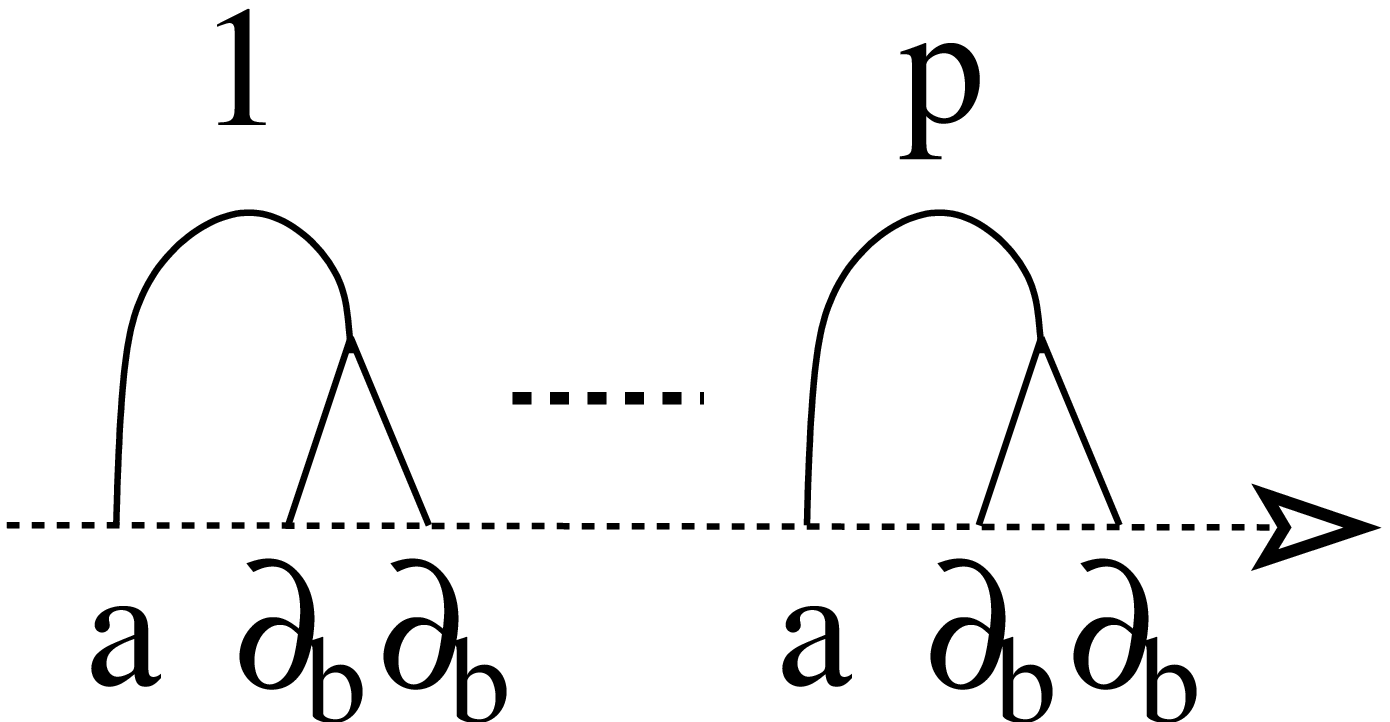}}}\right)\apply
\left( \sum_{{\mathrm{Subsets}\ S\subset
\mathcal{L}}}\left(-\frac{1}{2}\right)^{|S|}v_S\right).\\
\end{eqnarray*}
This is the right hand side of Equation \ref{howtosplit}.
\end{proof}

\subsection{Putting the pieces together.}
Now we'll put these pieces together to obtain an expression for
the following part of the composition:
\[ \underbrace{\ \mathcal{B}\stackrel{\Upsilon}{\longrightarrow}
\mathcal{W} \stackrel{\chi_{\mathcal{W}}}{\longrightarrow}
\widetilde{\mathcal{W}}\stackrel{\pi}{\longrightarrow}
\widehat{\mathcal{W}} \stackrel{B_{\bullet\mapsto\mathrm{\bf
F}}}{\longrightarrow} \widehat{\mathcal{W}}_{\mathrm{\bf F}}}
\stackrel{\lambda}{\longrightarrow} \widehat{\mathcal{W}}_\wedge.
\]
So let $v$ be an element of $\Bspace$. If we substitute
$\Upsilon(v)$ directly into Corollary \ref{comblemcor} then we are
given the following expression for $\left(\basebulltoF\circ
\pi\circ \chi_\Wspace\circ\Upsilon\right)(v)$:
\[
\left[ \intoop(\Upsilon(v))\,\apply \,\mathrm{exp}_\# \left(
\raisebox{-3.5ex}{\scalebox{0.24}{\includegraphics{paramj}}}
+\frac{1}{2}
\raisebox{-3.5ex}{\scalebox{0.24}{\includegraphics{paramk}}} +
\raisebox{-3.5ex}{\scalebox{0.24}{\includegraphics{paraml}}}
\right)\right]_{{a,b,\partial_a,\partial_b}=0}.
\]
We now wish to use Proposition \ref{howtosplit} to re-express this
as a direct function of $v$. To this end, we begin by inserting
the missing piece of that proposition into the front of this
expression, giving: \begin{multline*} \left[ \mathrm{exp}_\apply
\left(-\frac{1}{2}\,
\raisebox{-3.5ex}{\scalebox{0.25}{\includegraphics{paramZ}}}
\right)\apply\left( \intoop(\Upsilon(v))\,\right.\right. \\
\left.\left.\apply \,\mathrm{exp}_\# \left(
\raisebox{-3.5ex}{\scalebox{0.24}{\includegraphics{paramj}}}
+\frac{1}{2}
\raisebox{-3.5ex}{\scalebox{0.24}{\includegraphics{paramk}}} +
\raisebox{-3.5ex}{\scalebox{0.24}{\includegraphics{paraml}}}
\right)\right)\right]_{{a,b,\partial_a,\partial_b}=0}.
\end{multline*}
We can do this because all operation products in the resulting
expression converge and because, when we set the parameter $a$ to
zero, all the introduced terms will vanish.

Now we perform an associativity rearrangement (an ultra-careful reader
may wish to read Lemma \ref{powerseriesassociate} and then check
that Condition $(\S)$ holds before rearranging):
\begin{multline*} \left[ \left(\mathrm{exp}_\apply
\left(-\frac{1}{2}\,
\raisebox{-3.5ex}{\scalebox{0.25}{\includegraphics{paramZ}}}
\right)\apply\intoop(\Upsilon(v))\right)\,\right. \\
\left.\left.\apply \,\mathrm{exp}_\# \left(
\raisebox{-3.5ex}{\scalebox{0.24}{\includegraphics{paramj}}}
+\frac{1}{2}
\raisebox{-3.5ex}{\scalebox{0.24}{\includegraphics{paramk}}} +
\raisebox{-3.5ex}{\scalebox{0.24}{\includegraphics{paraml}}}
\right)\right)\right]_{{a,b,\partial_a,\partial_b}=0}.
\end{multline*}
Using Proposition \ref{howtosplit} to replace the bracket, and
then doing another associativity rearrangement (again checking Condition
$(\S)$) we get: \begin{multline*} \left[
\baselegtopartial(v)\apply\left(\exp_\apply \left(-\frac{1}{2}\,
\raisebox{-3.5ex}{\scalebox{0.25}{\includegraphics{paramZ}}}
\right)\,\right.\right. \\
\left.\left.\apply \,\mathrm{exp}_\# \left(
\raisebox{-3.5ex}{\scalebox{0.24}{\includegraphics{paramj}}}
+\frac{1}{2}
\raisebox{-3.5ex}{\scalebox{0.24}{\includegraphics{paramk}}} +
\raisebox{-3.5ex}{\scalebox{0.24}{\includegraphics{paraml}}}
\right)\right)\right]_{{a,b,\partial_a,\partial_b}=0}.
\end{multline*}
Finally, observe that $a$-labelled legs commute with all other
types of legs, and {\bf F}-legs commute with all other legs
(except other {\bf F}-legs), so some straightforward
rearrangements allow us to write (using $\{\}$ instead of $()$
only to make this equation easier on the eye):
\begin{multline}
\left[ \baselegtopartial(v)\apply\left\{\exp_\# \left(
\raisebox{-3.5ex}{\scalebox{0.25}{\includegraphics{paramj}}}\right)
\# \right.\right.\\ \left(\exp_\apply \left(-\frac{1}{2}\,
\raisebox{-3.5ex}{\scalebox{0.25}{\includegraphics{paramZ}}}
\right)\,\apply \right. \\ \left.\left.\left. \exp_\# \left(
\frac{1}{2}
\raisebox{-3.5ex}{\scalebox{0.25}{\includegraphics{paramk}}} +
\raisebox{-3.5ex}{\scalebox{0.25}{\includegraphics{paraml}}}
\right)\right)\right\}\right]_{{a,b,
\partial_a,\partial_b}=0}.\label{almostfinal}
\end{multline}
To complete the construction of the composition
\[ \mathcal{B}\stackrel{\Upsilon}{\longrightarrow} \mathcal{W}
\stackrel{\chi_{\mathcal{W}}}{\longrightarrow}
\widetilde{\mathcal{W}}\stackrel{\pi}{\longrightarrow}
\widehat{\mathcal{W}} \stackrel{B_{\bullet\rightarrow\mathrm{\bf
F}}}{\longrightarrow} \widehat{\mathcal{W}}_{\mathrm{\bf F}}
\stackrel{\lambda}{\longrightarrow} \widehat{\mathcal{W}}_\wedge
\]
it remains for us to apply $\lambda$ to the Expression
\ref{almostfinal}. The next section will show that we can commute
$\lambda$ through this expression and begin by applying it to the
right-most exponential above.

\subsection{Commuting $\lambda$ through the expression.}
For the purposes of this discussion we'll refer to legs of
the form
\[
\raisebox{-0.375ex}{\scalebox{0.25}{\includegraphics{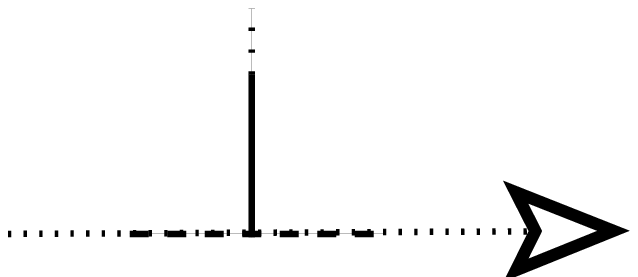}}}
\]
as {\bf $\bot$-legs}.

In Section \ref{lambdarecall} we recalled $\lambda$ as a map\[
\lambda : \widehat{\mathcal{W}}_\mathrm{\bf
F}\rightarrow\widehat{\mathcal{W}}_\mathrm{\wedge},\] basically
defined by ``gluing $\bot$-legs together in all possible ways
(with coefficients)". In this section we'll need the natural extension of
$\lambda$ to the situation where there are some parameter and
operator legs present (we'll make this precise shortly):
\[
\lambda : \WhatF\abpow\rightarrow\Whatwedge\abpow.\] The purpose
of this section is to prove the following proposition.
\begin{prop}
Let $v$ and $w$ be elements of $\WhatF\abpow$ such that the product $v\apply w$ converges, and assume that $v$
can be expressed without $\bot$-legs. Then:
\[
\lambda(v\apply w) = v\apply \lambda(w).
\]
\end{prop}

Before turning to the proof, let's repeat the definition of
$\lambda$ (in this more general context). Consider some diagram
$w$, a generator of $\WhatF\abpow$. Let $\mathcal{L}_\bot(w)$
denote the set of $\bot$-legs of $w$. Recall that a {\bf pairing}
of $w$ is a (possibly empty) set of disjoint 2-element subsets of
$\mathcal{L}_\bot$. As before, $\mathcal{P}(w)$ denotes the set of
pairings of $w$. We then define $\lambda$ by
\[
\lambda(w) = \sum_{\wp\in\mathcal{P}(w)} \mathcal{D}_\wp(w)
\]
where the term $\mathcal{D}_\wp(w)$ is constructed by the
graphical procedure described in Section
\ref{lambdarecall}. The procedure is, recall: draw another
orienting line under the existing one, separated by a gap; then
join up the $\bot$-legs according to the pairing $\wp$ using
non-self-intersecting full arcs lying entirely in the gap between
the two orienting lines; then carry any legs remaining on the
original orienting line down to the new orienting line, using full
lines for (all) the grade 1 legs and dashed lines for (all) the
grade 2 legs. Finally, write down a coefficient of
$(-1)^x\left(\frac{1}{2}\right)^y$ where $x$ is the number of
self-intersections displayed by full lines between the two
orienting lines and $y$ is the number of pairs in $\wp$.

For example, if
\[ w=\ \ \
\raisebox{-6ex}{\scalebox{0.2}{\includegraphics{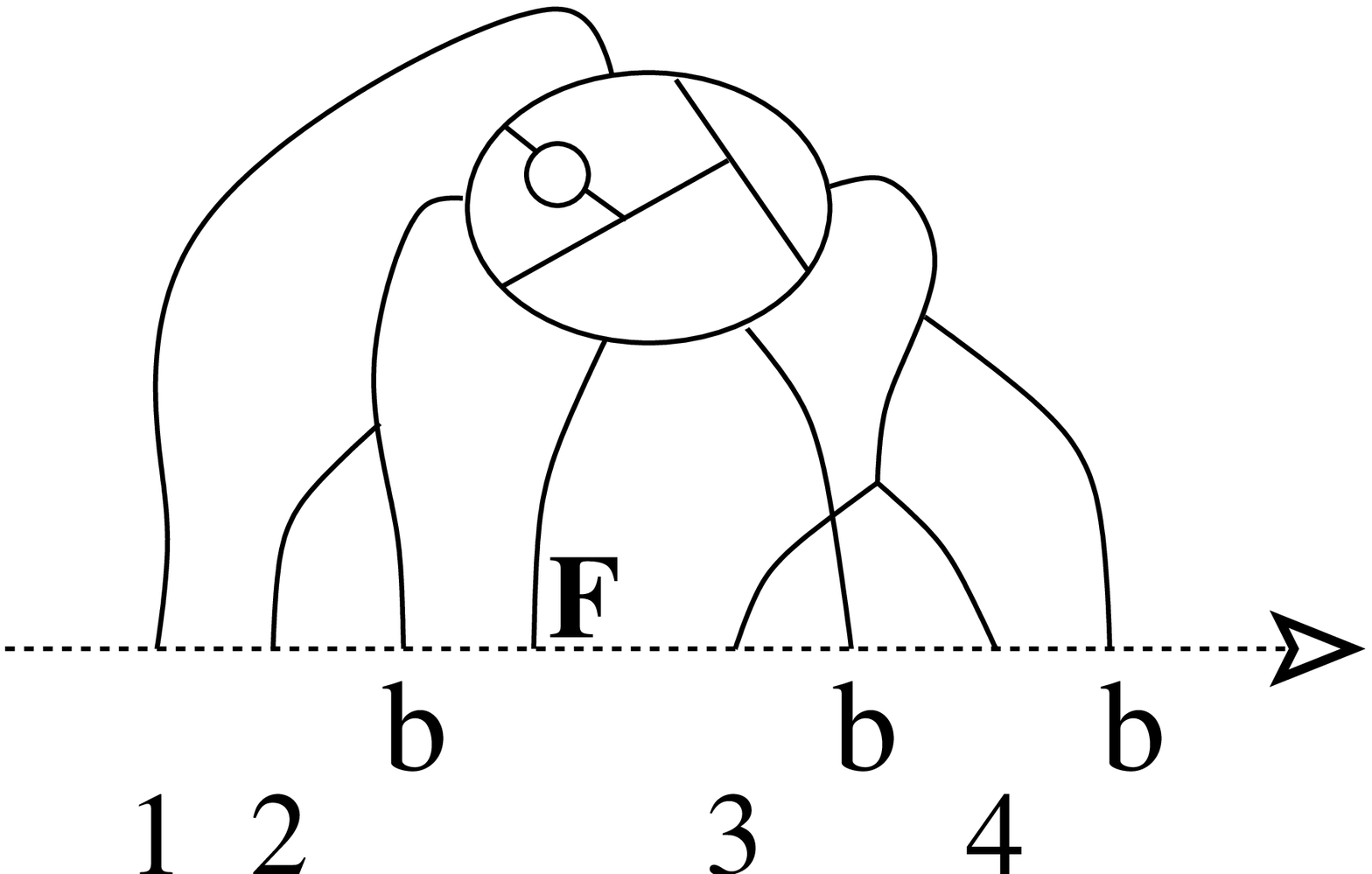}}}\
,\\[0.2cm]
\]
and we wished to construct $\mathcal{D}_{\{\{1,4\},\{2,3\}\}}(w)$,
then we would draw the diagram
\[
\raisebox{-6ex}{\scalebox{0.2}{\includegraphics{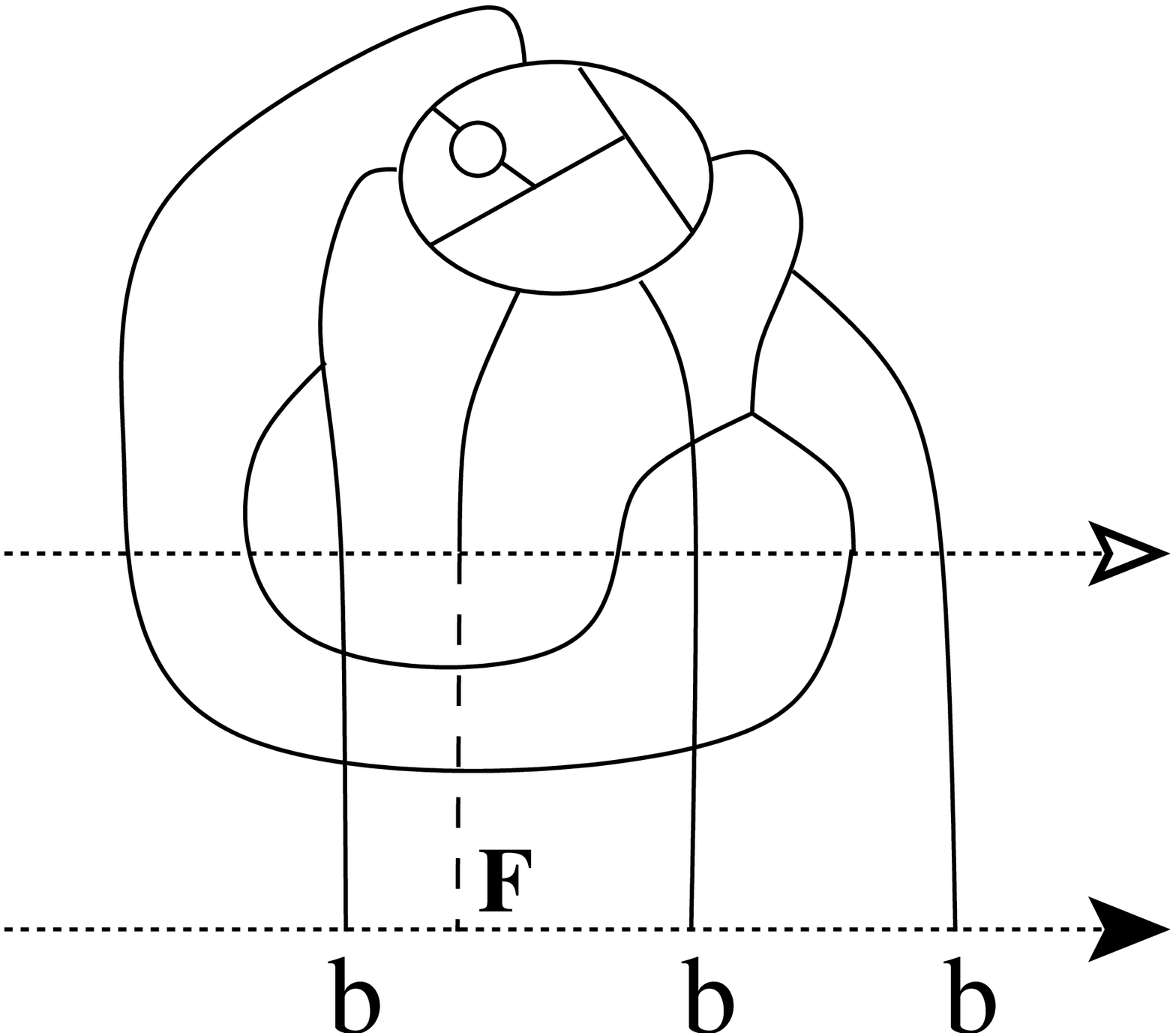}}}\ \ ,
\]
and thus deduce that:
\[
\mathcal{D}_{\{\{1,4\},\{2,3\}\}}(w) =
(-1)^3\left(\frac{1}{2}\right)^2
\raisebox{-5ex}{\scalebox{0.2}{\includegraphics{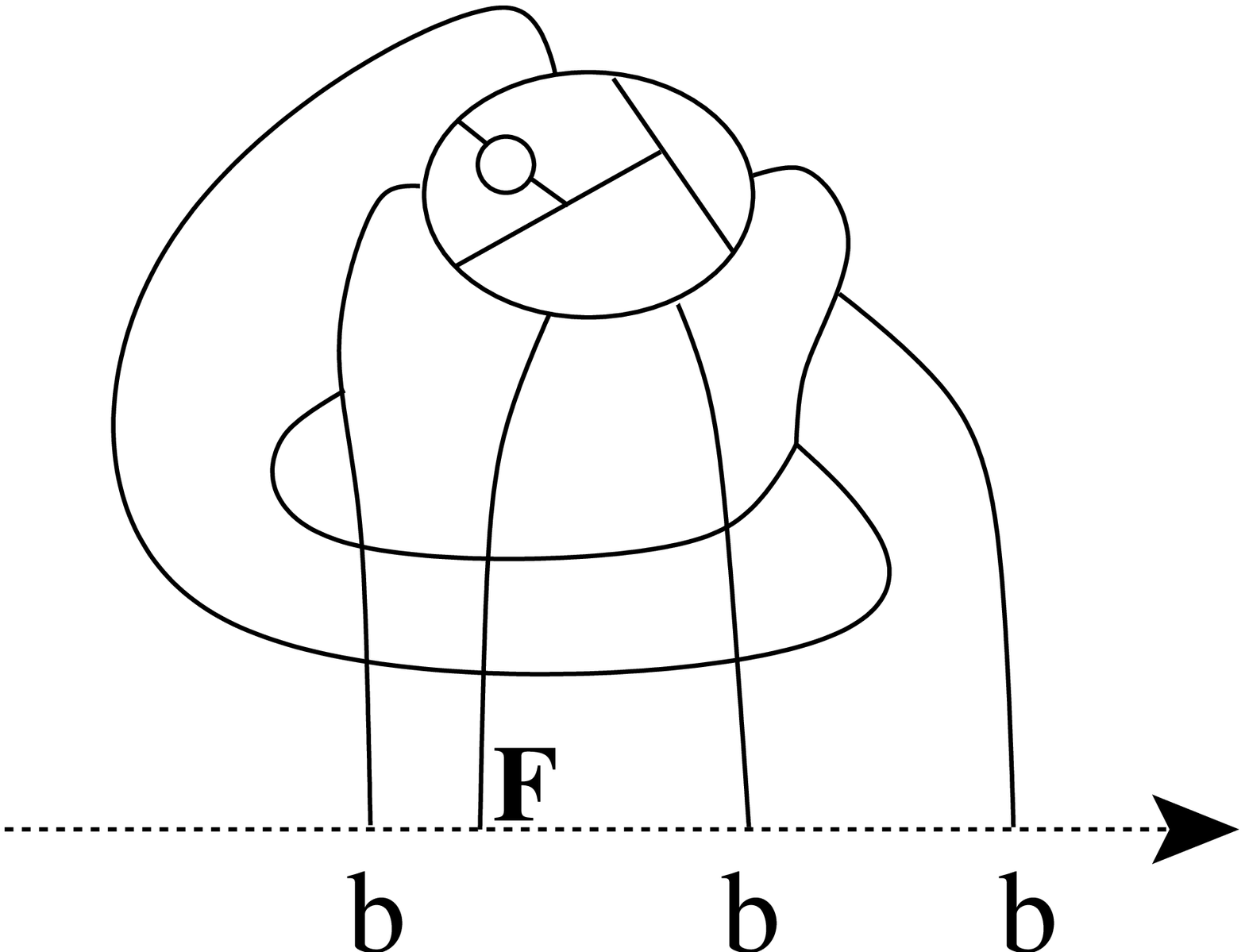}}}.
\]
\begin{proof}
It suffices to show that
\begin{equation}\label{uphere}
\lambda(v\apply w) = v\apply \lambda(w)
\end{equation}
is true for generators.

So let $v$ and $w$ be operator Weil diagrams and assume that $v$
has no $\bot$-legs. For the purposes of the discussion below we'll
consider the example where:
\[
v =
\raisebox{-1cm}{\scalebox{0.25}{\includegraphics{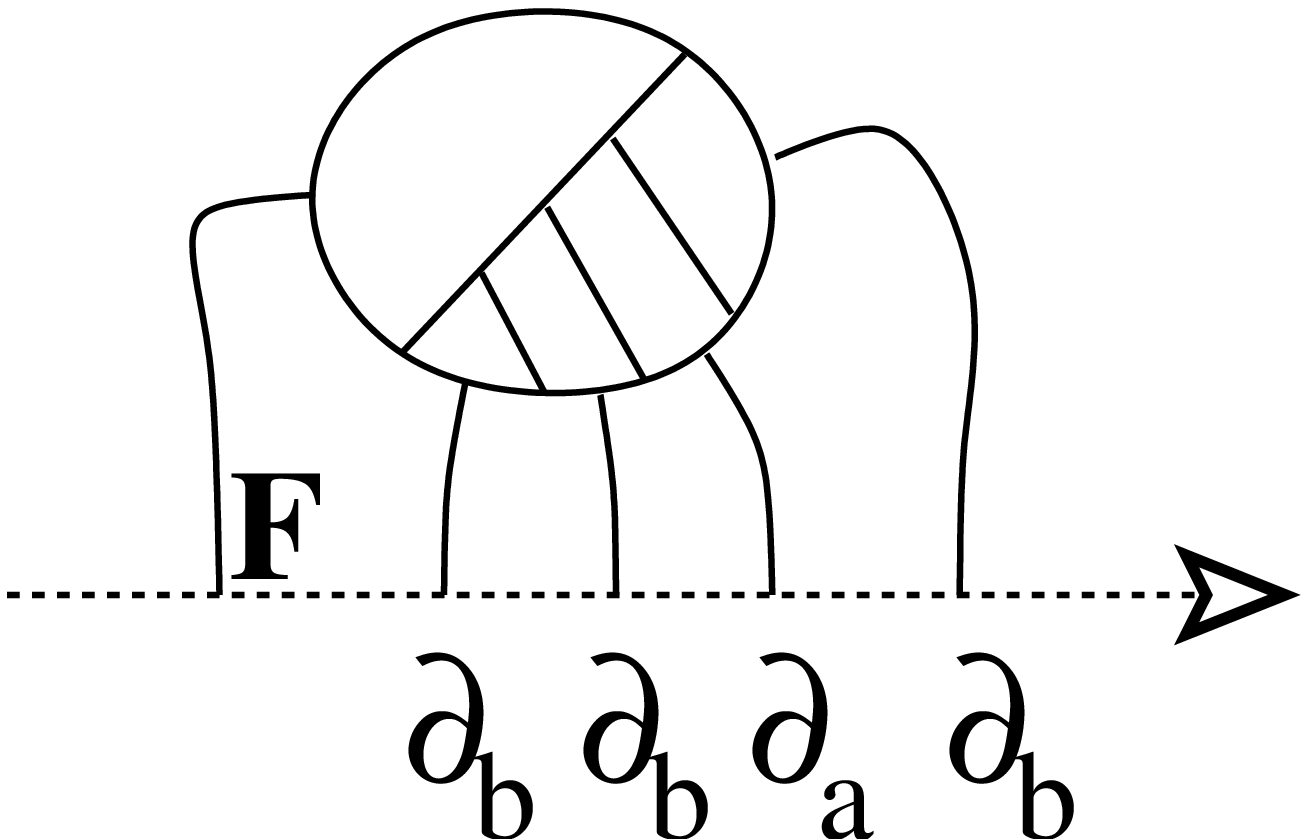}}}\
\ \ \ \ \text{and}\ \ \ \ \ w =
\raisebox{-0.75cm}{\scalebox{0.25}{\includegraphics{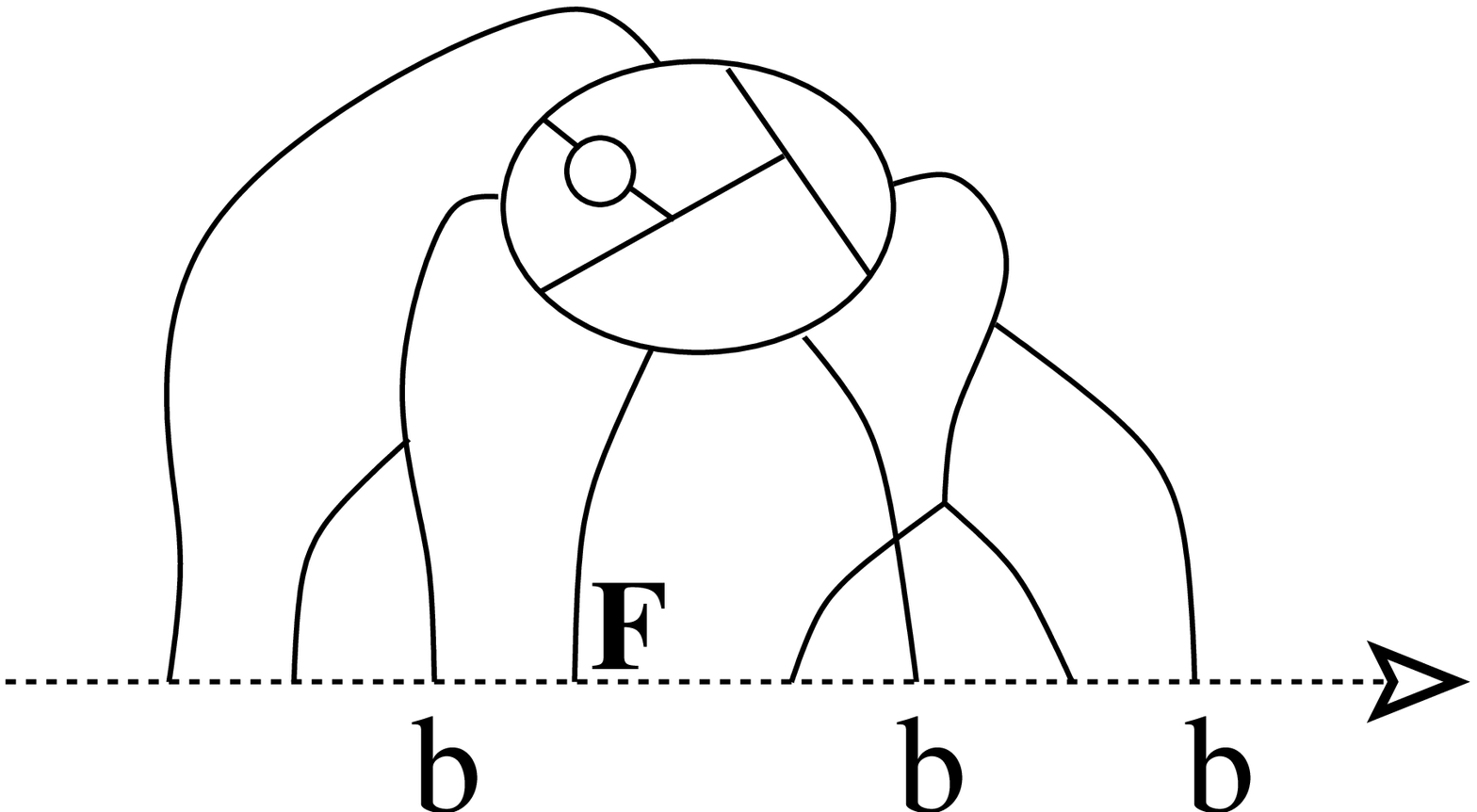}}}\
.
\]
It follows from the construction of the operations $\lambda$ and
$\apply$ that the two sides of Equation \ref{uphere} are sums
indexed by the same set. Define that set $\mathcal{PG}$ to consist
of pairs $(\wp,\sigma)$, where $\wp$ is a pairing
$\wp\in\mathcal{P}(w)$ and $\sigma$ is a gluing $\sigma\in
\mathcal{G}(v,w)$. Then:
\begin{enumerate}
\item{$\lambda(v\apply w) = \sum_{(\wp,\sigma)\in \mathcal{PG}}
\mathcal{D}_{\wp}\left(t(v,w,\sigma)\right),$} \item{$
v\apply\lambda(w) = \sum_{(\wp,\sigma)\in \mathcal{PG}}
t(v,\mathcal{D}_\wp(w),\sigma),$}
\end{enumerate}
where $t(v,w,\sigma)$ is the notation used in Section \ref{applygraphical} for the term that arises when $v$
is applied to $w$ using the gluing $\sigma$.
The required equality will be
established if we can show that for every pair
$(\wp,\sigma)\in\mathcal{PG}$:
\[
\mathcal{D}_\wp\left(t(v,w,\sigma)\right) =
t\left(v,\mathcal{D}_\wp(w),\sigma\right).
\]

We'll begin by illustrating the desired equality for the case of
the given example and the pair
$(\wp,\sigma)=\left(\{\{1,3\},\{2,4\}\},\left({4 1 \atop 3
2}\right)\right)$. A direct application of the definitions tells
us that to construct the term
$\mathcal{D}_\wp\left(t(v,w,\sigma)\right)$ we draw the diagram
shown in Figure \ref{gluethenpair}. On the other hand: a direct
application of the definitions tells us that to construct
$t\left(v,\mathcal{D}_\wp(w),\sigma\right)$ we draw the diagram
shown in Figure \ref{pairthenglue}.
\begin{figure}
\caption{The diagram you draw to build
\(\mathcal{D}_\wp\left(t(v,w,\sigma)\right)\).\label{gluethenpair}}
\[
\scalebox{0.22}{\includegraphics{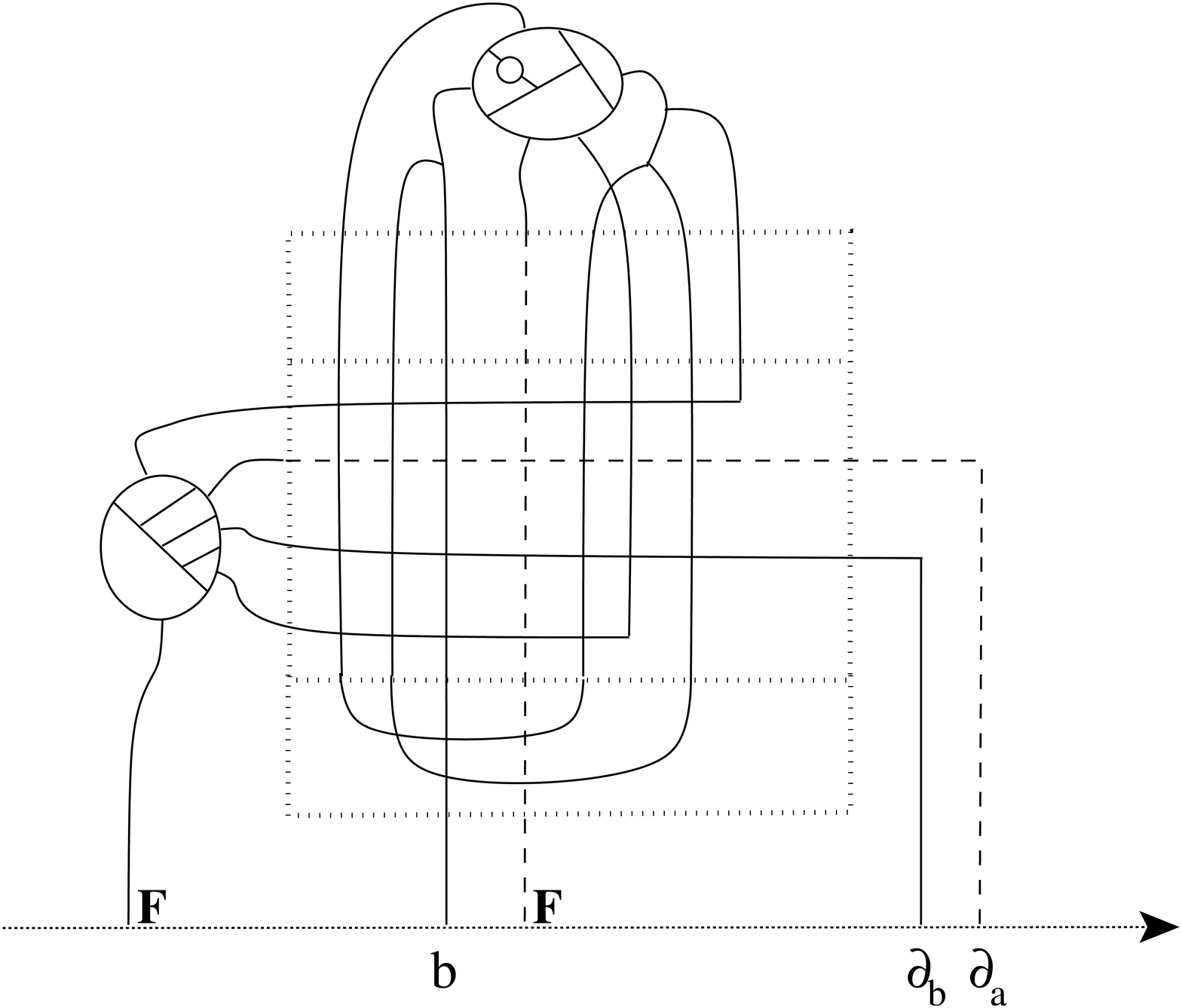}}\ .
\]
\end{figure}
\begin{figure}
\caption{The diagram that you draw to build
$t\left(v,\mathcal{D}_\wp(w),\sigma\right)$.\label{pairthenglue}}
\[
\scalebox{0.22}{\includegraphics{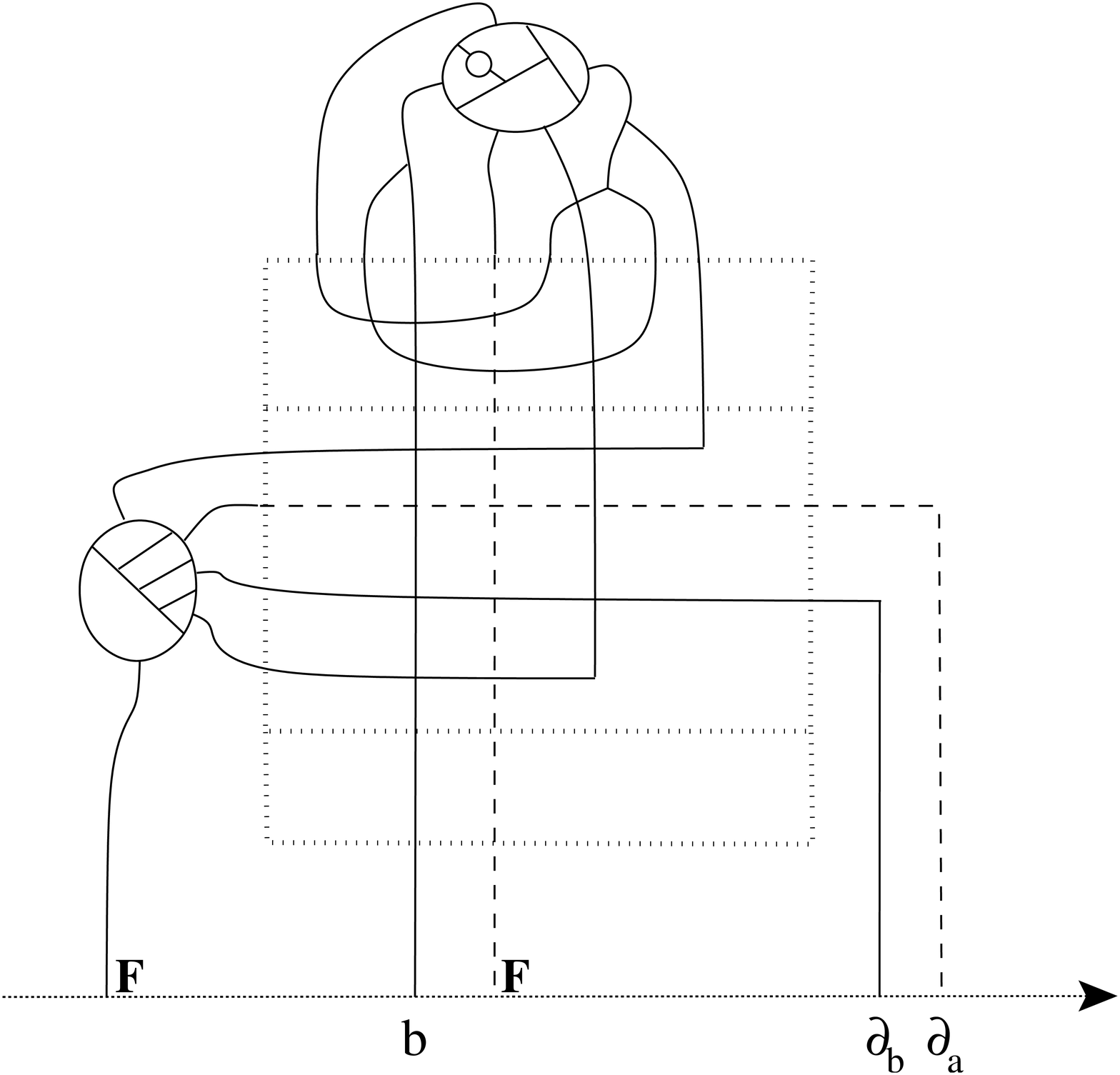}}\ .
\]
\end{figure}
In both cases the coefficient of the contributing term is given by
$(-1)^x\left(\frac{1}{2}\right)^y$, where $x$ is the respective number of
intersections between full lines displayed inside the total box
(by {\it total box} we'll refer to the union of the three
displayed dotted boxes), and $y$ is the number of pairs in $\wp$
(in this case, 2). The underlying diagrams are obviously the same,
so it remains to understand why the two diagrams have
the same number of intersections between full lines (mod $2$).

To explain why, in generality: First observe that the total box
cuts the full lines that go through the total box into a number of
{\it arcs}. Then notice that there is obvious correspondence
between the arcs of the first diagram and the arcs of the second
diagram. Finally, notice that the ends of any pair of arcs will
have the same relative position around the edge of the total box
in both diagrams; hence each pair of arcs will have the same number
of intersections mod $2$ in both diagrams.
\end{proof}

\section{Computing the operator product I: The inner-most piece.}
\label{computingtheoperatorexpressionA} We'll begin the
computation of the expression in Theorem \ref{importantexpression}
with the inner-most piece. The objective of this section is to
prove the following theorem.
\begin{thm}\label{connectedtheorem}
The following equality holds in $\Whatwedge\abpow$.
\begin{multline*}
\lambda\left(\exp_{\# }\left(
\raisebox{-3ex}{\scalebox{0.27}{\includegraphics{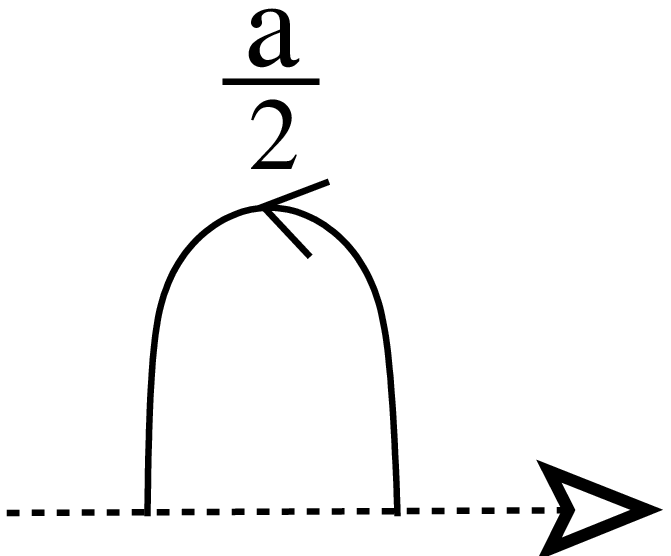}}} +
\raisebox{-4.85ex}{\scalebox{0.27}{\includegraphics{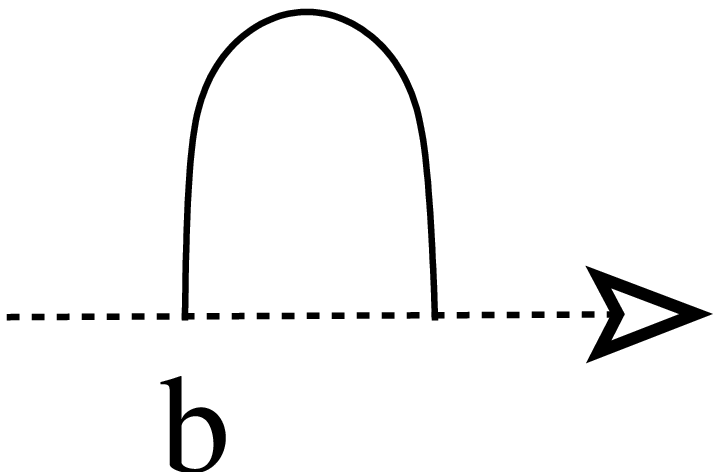}}}
\right)\right)  \\
= \exp_{\#}\left(\left(\frac{1}{2}\right)
\raisebox{-5ex}{\scalebox{0.27}{\includegraphics{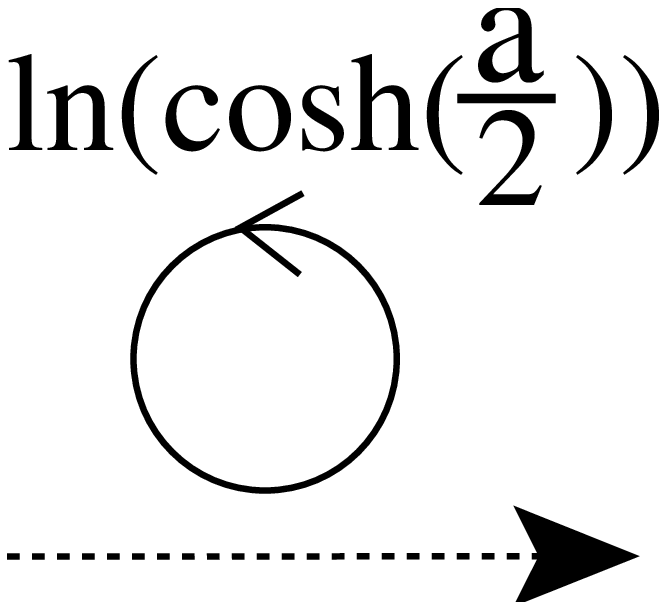}}} +
\raisebox{-5ex}{\scalebox{0.27}{\includegraphics{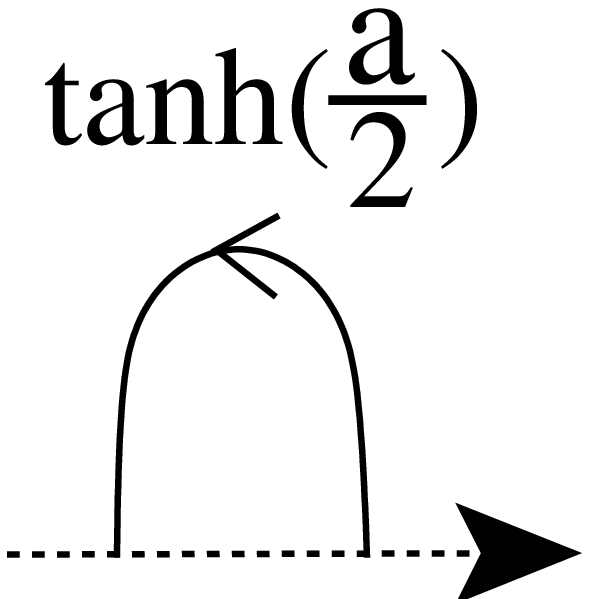}}}
-\raisebox{-6.9ex}{\scalebox{0.27}{\includegraphics{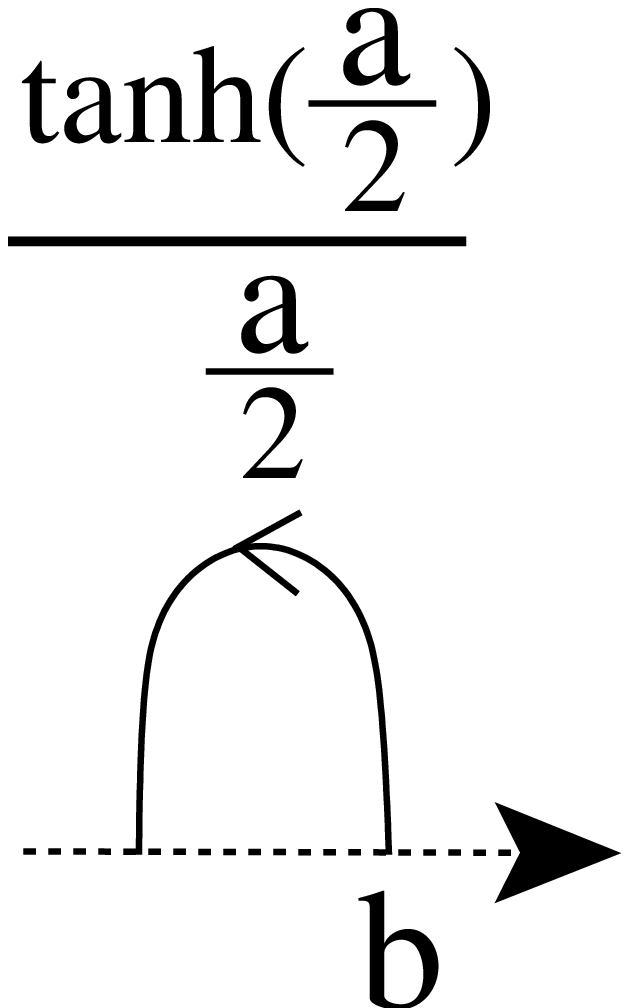}}}
-\left(\frac{1}{2}\right)^2
\raisebox{-6.75ex}{\scalebox{0.27}{\includegraphics{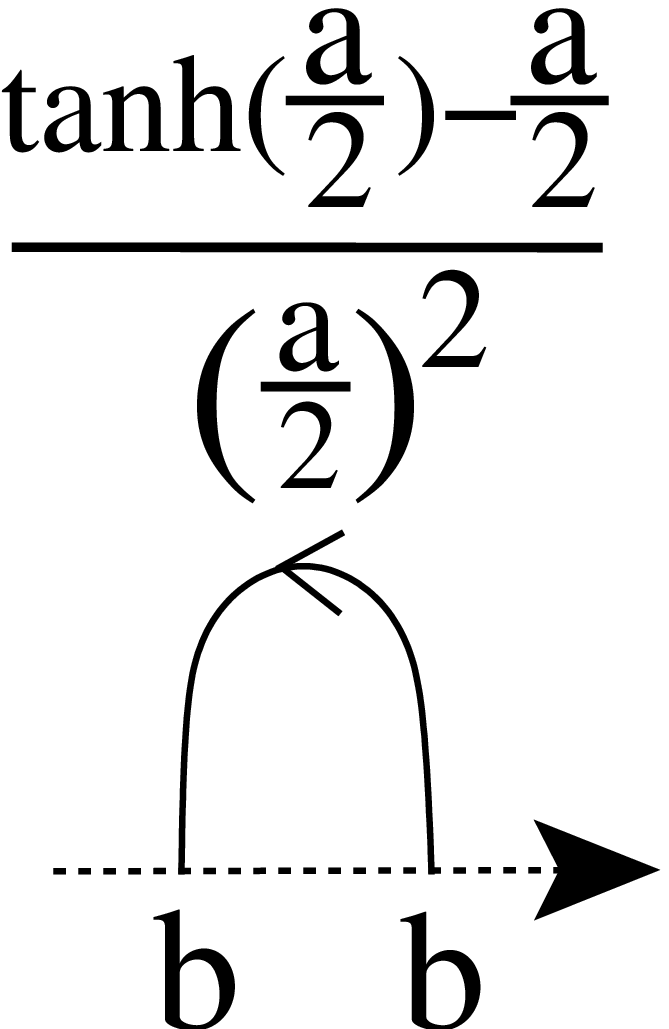}}}\
\right) .
\end{multline*}
\end{thm}
This section consists of two subsections. In Section
\ref{isconnected} we show that the left-hand side of the above
equation can be expressed as an exponential of the series of terms
with {\bf connected} diagrams that arise from the evaluation of
$\lambda$. In Section \ref{determineconnected} we'll perform a
detailed calculation of that series.

\subsection{An exponential
of connected diagrams.}\label{isconnected}

Consider the terms that arise when you compute the left-hand side
of the above equation. We'll index these terms with a certain set
$\mathcal{T}$. The set $\mathcal{T}$ is defined to be the set of
pairs $(w,\wp)$ consisting of a non-empty word $w$ in the symbols
$A$ and $B$ and a disjoint (possibly empty) family $\wp$ of
2-element subsets of the set $\{1,2,3,\,\ldots\,,2\#A+\#B\}$
(where $\#A$ and $\#B$ denote the number of appearances in the
word $w$ of the symbols $A$ and $B$ respectively). If
$\tau=(w,\wp)$ we'll often write $|\tau|$ for $|w|$, the length of
the word $w$.

Given such a pair $(w,\wp)$, the corresponding term $T_{(w,\wp)}$
is constructed in two steps. The first step is to place a number
of copies of \[
\raisebox{-6ex}{\scalebox{0.26}{\includegraphics{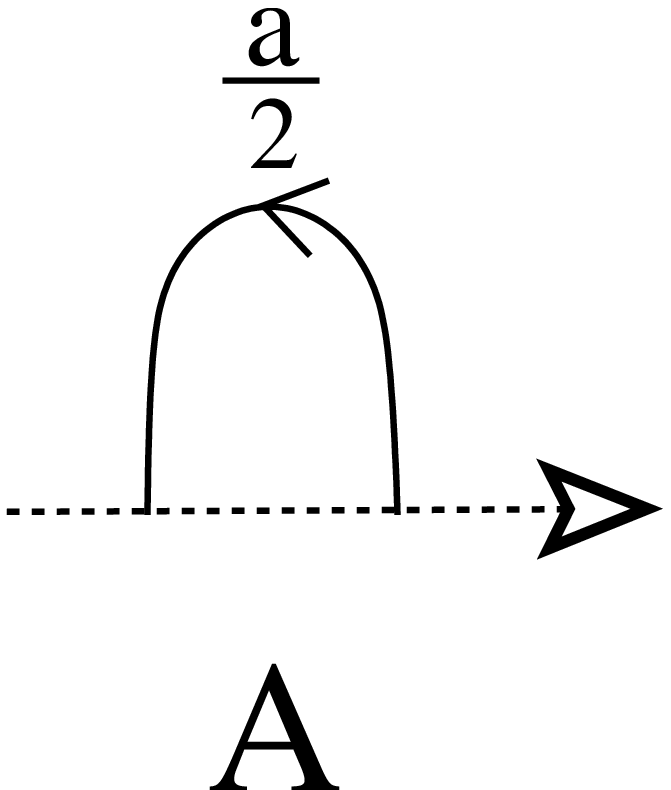}}}\ \ \
\ \ \ \mbox{and}\ \ \ \ \ \
\raisebox{-6.5ex}{\scalebox{0.26}{\includegraphics{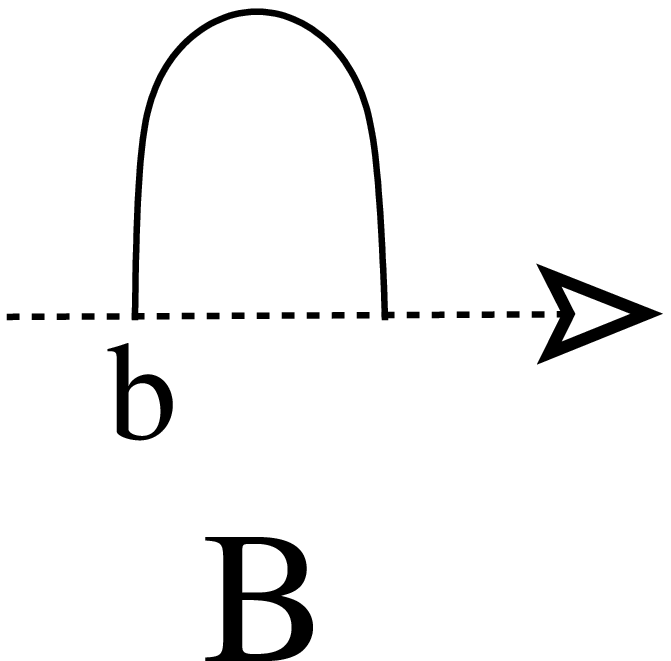}}}
\]
along an orienting line in the order dictated by $w$. The second
step is to pair up (with appropriate signs) the $\bot$-legs of
this diagram according to the pairing $\wp$, multiplying by a
factor of $\left(\frac{1}{2}\right)$ for every pair of legs glued
together (in other words, by a factor of
$\left(\frac{1}{2}\right)^{|\wp|}$).

For example, to construct $T_{\left(AABAB,
\{\{1,5\},\{2,3\},\{4,6\},\{7,8\}\}\right)}$, begin by writing
down
\[\raisebox{-3.5ex}{\scalebox{0.26}{\includegraphics{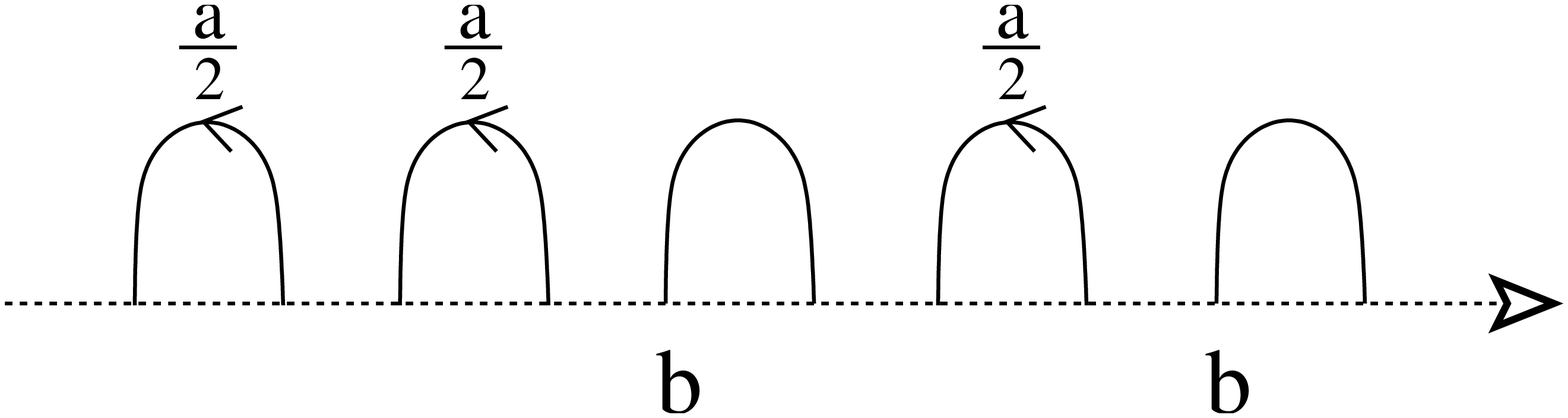}}}\ .
\]
Then pair up legs according to the pairing information, giving:
\[
(+1)\left(\frac{1}{2}\right)^3\,
\raisebox{-5ex}{\scalebox{0.26}{\includegraphics{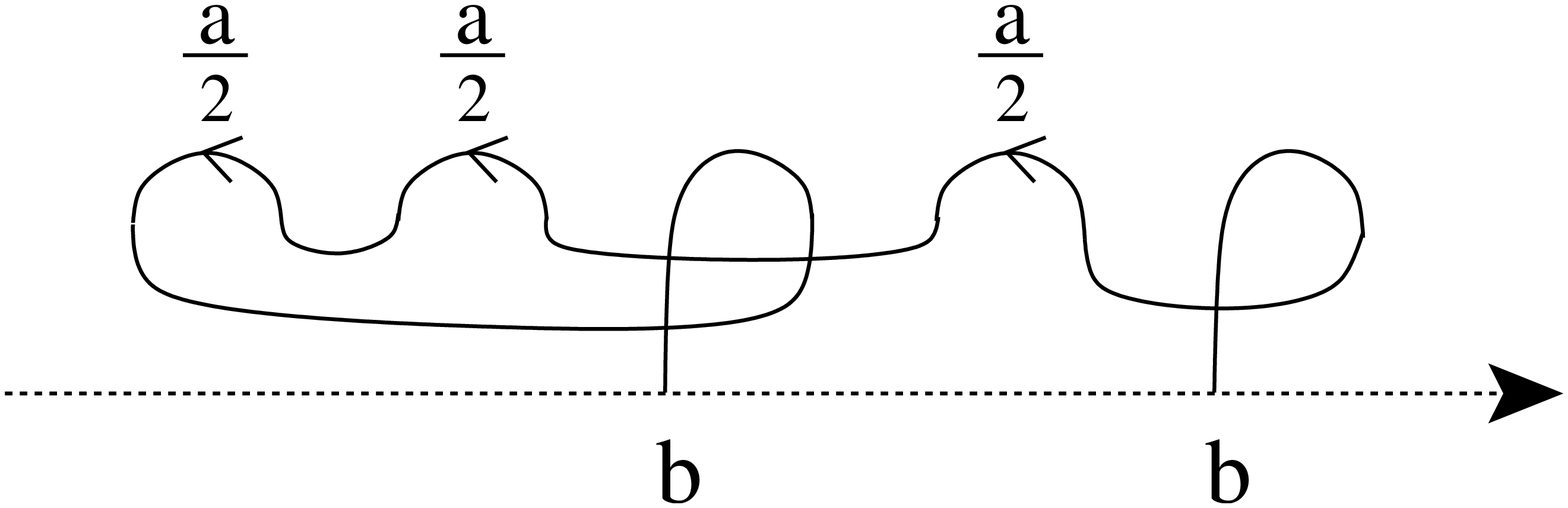}}}\ .
\]
The follow proposition is just a restatement of the definitions.
\begin{prop}\label{restatement}
\[
\lambda\left( \exp_{\# }\left(
\raisebox{-3ex}{\scalebox{0.25}{\includegraphics{paramkay}}} +
\raisebox{-4.6ex}{\scalebox{0.25}{\includegraphics{paraml}}}
\right)\right) = 1 + \sum_{\tau\in\mathcal{T}}
\frac{1}{|\tau|!}T_{\tau}.
\]
\end{prop}
Now let $\mathcal{T}_{\mathcal{C}}\subset \mathcal{T}$ denote the subset consisting of those pairs
$(w,\wp)$ whose corresponding term $T_{(w,\wp)}$ is {\bf
connected}. (By connected we mean that the graph is connected when
the orienting line is ignored.) In this discussion we'll call
$\mathcal{T}_\mathcal{C}$ the set of {\bf connected types}.

\begin{thm}\label{decompintoconnecteds}
The following equation holds in $\Whatwedge\abpow$:
\[
\lambda\left( \exp_{\# }\left(
\raisebox{-3ex}{\scalebox{0.25}{\includegraphics{paramkay}}} +
\raisebox{-4.5ex}{\scalebox{0.25}{\includegraphics{paraml}}}
\right)\right) = \exp_{\#
}\left(\sum_{\tau\in\mathcal{T}_\mathcal{C}}
\frac{1}{|\tau|!}T_\tau \right).
\]
\end{thm}
We'll build up to this theorem with a number of combinatorial
lemmas. The computation of
$\sum_{\tau\in\mathcal{T}_\mathcal{C}}\frac{1}{|\tau|!}T_\tau$,
the series of {\it connected} diagrams that can arise from the
computation, is the subject of the next section. Certain readers
may feel that the above theorem is obvious, and we encourage such
readers to skip to Section \ref{determineconnected}. (We
remark that the subtlety in this situation is mostly to do with the
signs.)

\subsubsection{The content of a pairing.}
Consider some pair $(w,\wp)$. The corresponding diagram
$T_{(w,\wp)}$ decomposes into a number a connected components. To
each connected component $x$ we can associate some other pair
$(w_x,\wp_x)$. Namely, just delete every other component and write
down the pair $(w_x,\wp_x)$ which produces the remaining
component.

For example, consider the diagram corresponding to
\[(AABAAA,\{\{2,5\},\{4,6\},\{3,9\},\{7,11\},\{8,10\}\}\,).\]It is:
\[
\raisebox{-3.5ex}{\scalebox{0.24}{\includegraphics{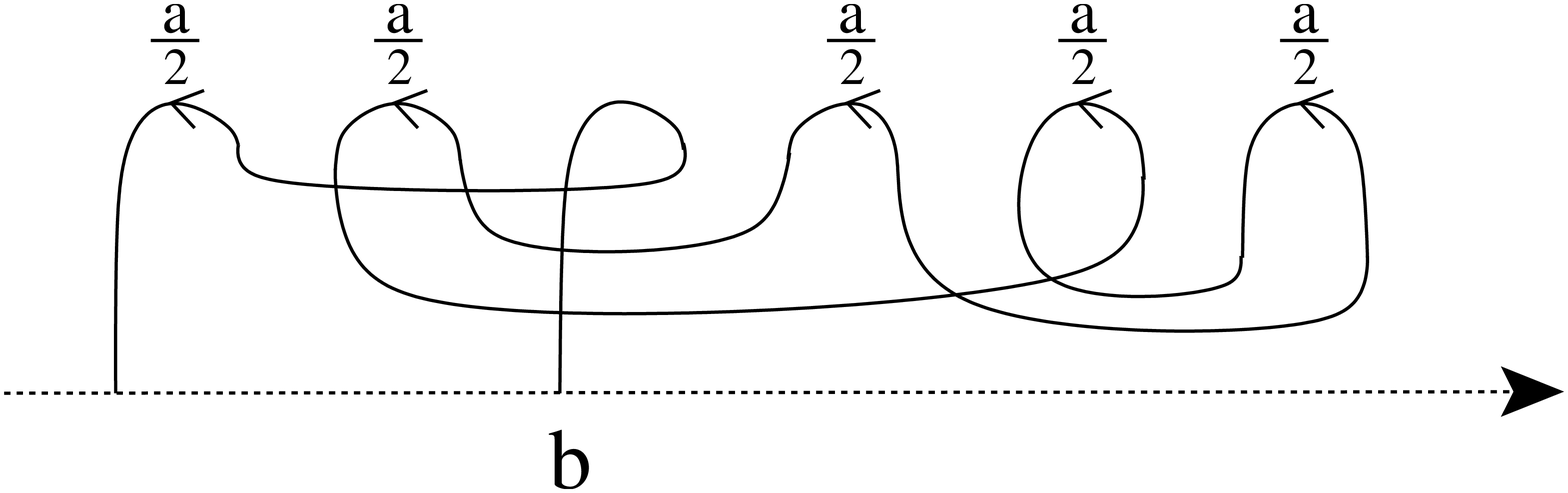}}}\
 .
\]
This diagram has 2 connected components. One component corresponds
to the pair $(AAAA,\{\{1,6\},\{2,3\},\{4,8\},\{5,7\}\}\,)$:
\[
\raisebox{-3.5ex}{\scalebox{0.24}{\includegraphics{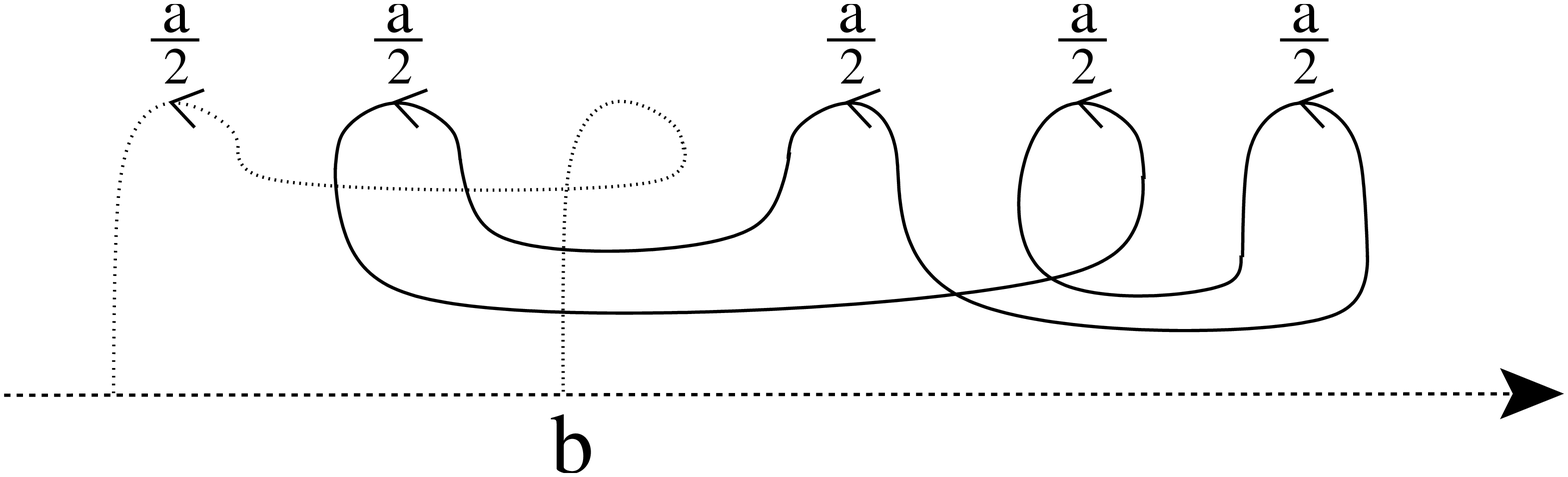}}}\
.
\]
The other component corresponds to the pair $(AB,\{\{2,3\}\}\,)$:
\[
\raisebox{-3.5ex}{\scalebox{0.24}{\includegraphics{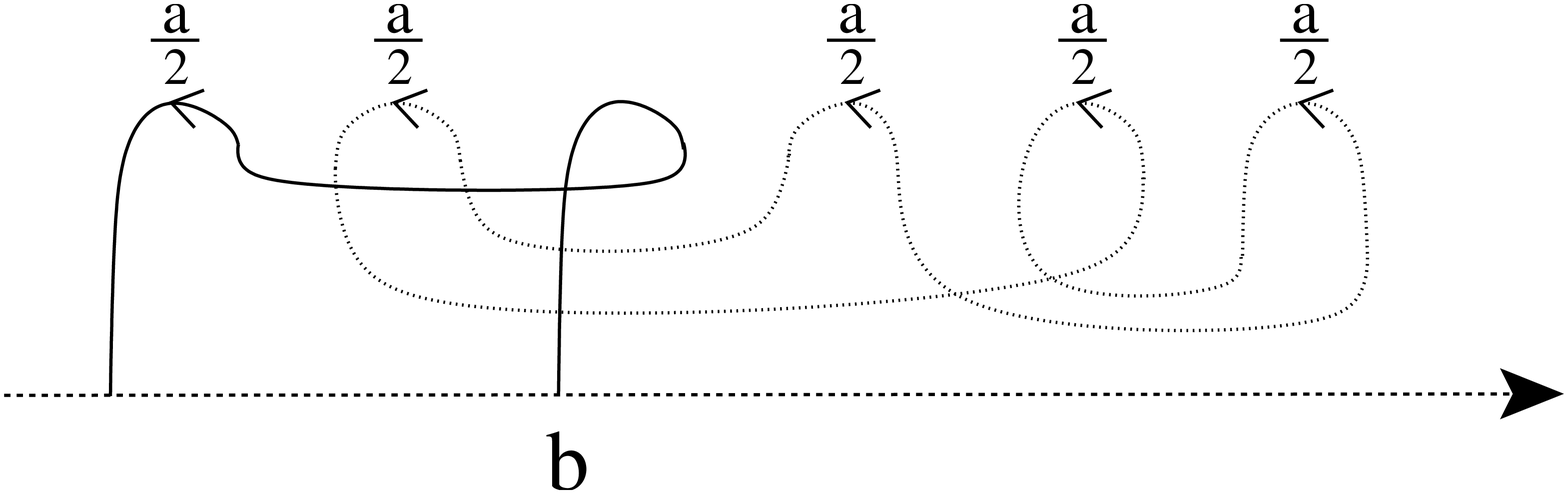}}}\
 .
\]

Define now the {\bf content} of a pair $(w,\wp)$ to be the map
$c_{(w,\wp)}:\mathcal{T}_\mathcal{C}\rightarrow \mathbb{N}$ which
counts the different types of the connected components. For
example, the diagram corresponding to the pair
\[
(w,\wp)=(AAAAABAA ,
\{\{1,10\},\{2,9\},\{3,5\},\{4,6\},\{7,13\},\{8,12\},\{11,14\}\})
\]
is \[
\raisebox{-3.5ex}{\scalebox{0.24}{\includegraphics{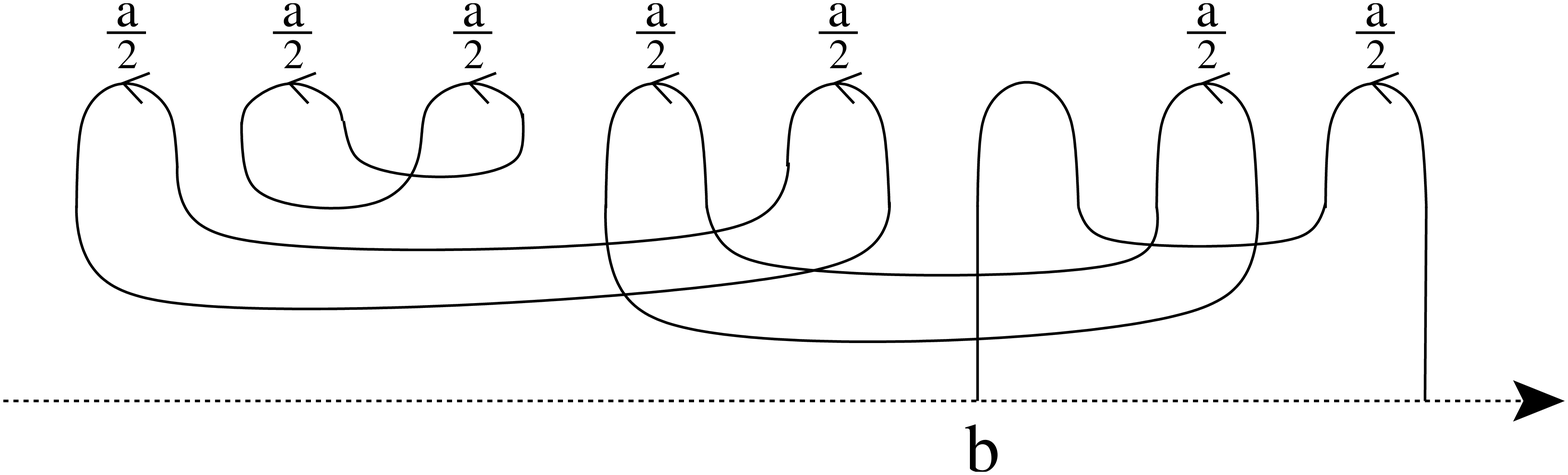}}}\
\ \ \ ,
\]
and its corresponding content is thus:
\[
c_{(w,\wp)}((\omega,\rho)) = \left\{
\begin{array}{cl}
2 & \mbox{if $(\omega,\rho)=(AA,\{\{1,4\},\{2,3\}\})$}, \\
1 & \mbox{if $(\omega,\rho)=(AA,\{\{1,3\},\{2,4\}\})$}, \\
1 & \mbox{if $(\omega,\rho)=(BA,\{\{1,2\}\})$}, \\
0 & \mbox{otherwise.}
\end{array}
\right.
\]

\subsubsection{The ``Pairings factorize" lemma.}
\begin{lem}\label{factorlemma}
Consider some $\tau\in \mathcal{T}$. Then:
\[
T_\tau =
\prod^\#_{\kappa\in\mathcal{T}_{\mathcal{C}}}\left(T_\kappa\right)^{\#c_\tau(\kappa)}.
\]
\end{lem}
The proof of this lemma will appear shortly, in Section
\ref{pairingsfactorize}. To illustrate it, consider the case that
\[\tau=\left(AAAAAB, \{\{2,9\},\{4,5\},\{8,11\}\}\,\right).
\]In this case, the
left hand side of the above equation is equal to \[
(+1)\left(\frac{1}{2}\right)^3\raisebox{-6ex}{\scalebox{0.24}{\includegraphics{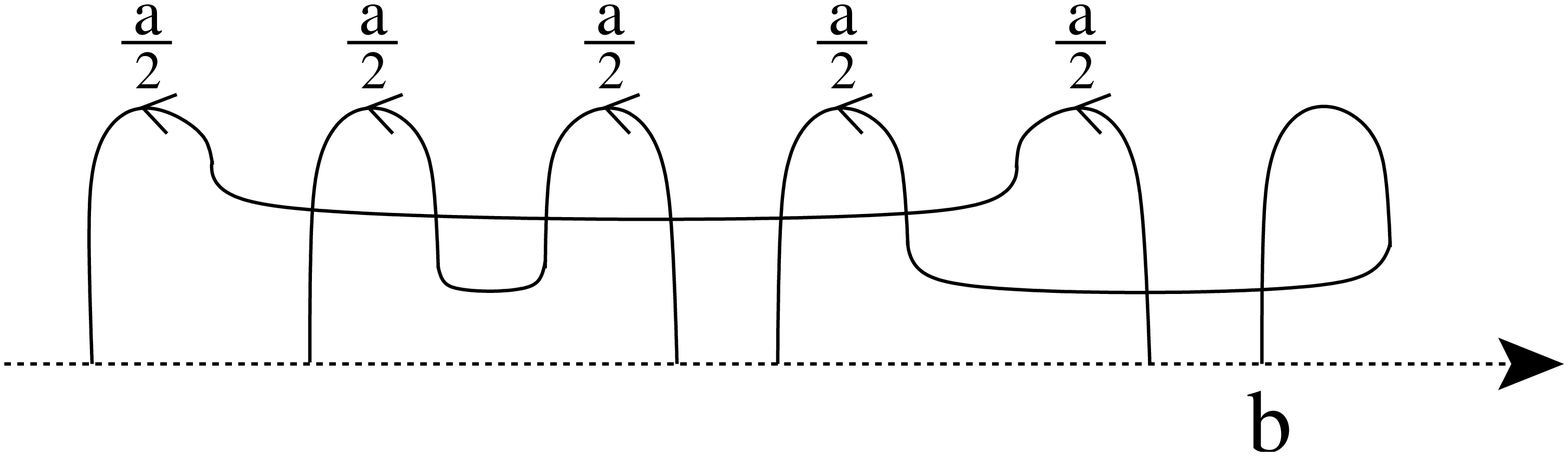}}},
\]
and the right hand side is
\[
\left((+1)\left(\frac{1}{2}\right)
\raisebox{-3.5ex}{\scalebox{0.24}{\includegraphics{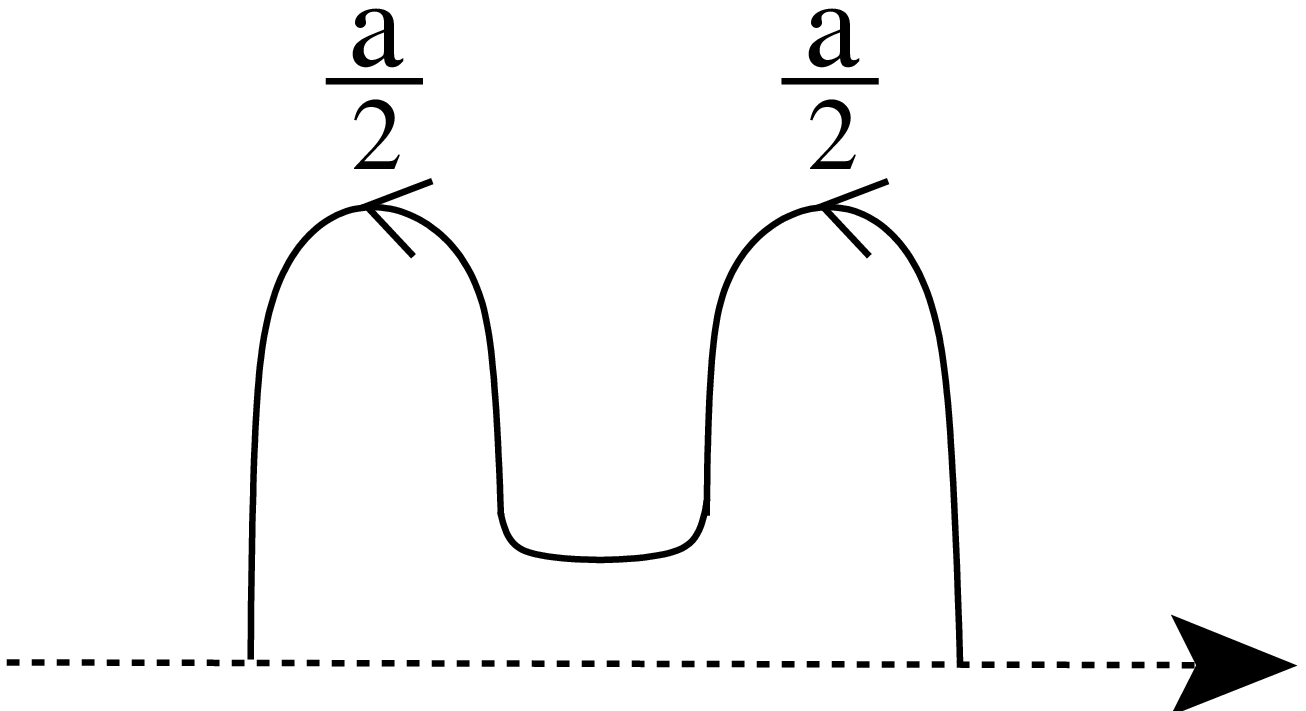}}}
\right)^{\#2}\# \left((-1)\left(\frac{1}{2}\right)
\raisebox{-5.7ex}{\scalebox{0.24}{\includegraphics{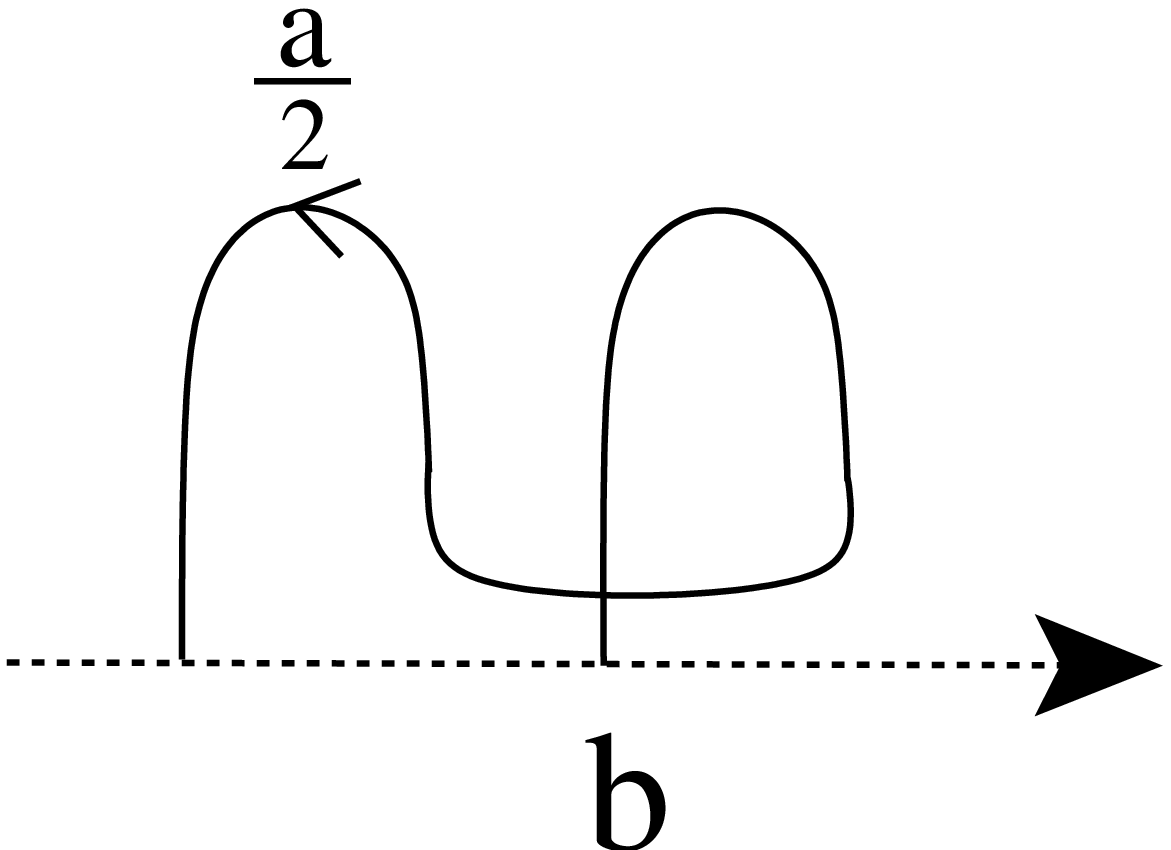}}}
\right) .
\]
The equality of these two expressions (in $\Whatwedge\abpow$) may
be immediately observed.

Applying this lemma to Proposition \ref{restatement} we can write:
\begin{multline}
1 + \sum_{\tau\in\mathcal{T}} \frac{1}{|\tau|!}T_{\tau} \\ =  1
+ \sum_{{\mathrm{content\ functions}\atop 0\neq c:
\mathcal{T}_{\mathcal{C}}\rightarrow
\mathbb{N}_0}}\left(\begin{array}{c}\mbox{Number of} \\ \mbox{pairs $\tau\in\mathcal{T}$} \\
\mbox{with content
$c$.}\end{array}\right)\frac{1}{\left(\sum_{\tau\in\mathcal{T}_\mathcal{C}}c(\tau)|\tau|\right)!}
  \prod^\#_{\kappa\in\mathcal{T}_{\mathcal{C}}}
  T_\kappa^{\#c(\kappa)}. \label{aboveeqn}
\end{multline}

It remains for us to count the number of pairs $(w,\wp)$ with some
given content
$c:\mathcal{T}_{\mathcal{C}}\rightarrow\mathbb{N}_0$. That is
achieved by the next lemma, which is proved in Section
\ref{pairingsfactorize}.
\begin{lem}\label{howmanycontent}
Consider some content function
$\tau:\mathcal{T}_\mathcal{C}\rightarrow\mathbb{N}_0$. The number
of pairs $(w,\wp)$ with this content is
\[
\frac{\left(\sum_{\tau\in\mathcal{T}_\mathcal{C}}c\left(\tau\right)|\tau|\right)!}
{\prod_{\kappa\in\mathcal{T}_\mathcal{C}}\left(|\kappa|!\right)^{c(\kappa)}\left(c(\kappa)\right)!}.
\]
\end{lem}
If we substitute this computation into the right-hand side of the
above equation, we get:
\[
1 + \sum_{{\mathrm{content\ functions}\atop 0\neq c:
\mathcal{T}_{\mathcal{C}}\rightarrow \mathbb{N}_0}}
\prod^\#_{\kappa\in\mathcal{T}_{\mathcal{C}}}
\frac{1}{c(\kappa)!}\left(\frac{T_\kappa}{|\kappa|!}\right)^{\#c(\kappa)}.
\]
This completes the proof of Theorem \ref{decompintoconnecteds}.

\subsubsection{The proofs of the two lemmas.}
\label{pairingsfactorize} This section contains proofs of the two
technical lemmas that were used in the proof of Theorem
\ref{decompintoconnecteds}.

 {\it Proof of Lemma
\ref{factorlemma}.} We are asked to show that, for
$(w,\wp)\in\mathcal{T}$, \begin{equation}\label{factorhere}
T_{(w,\wp)} =
\prod^\#_{\kappa\in\mathcal{T}_{\mathcal{C}}}\left(T_\kappa\right)^{\#c_{(w,\wp)}(\kappa)}.
\end{equation}
The right hand side of this equation is just the left hand side,
factored into its connected components. The equality is obvious,
except for the possibility that the signs may differ. To establish
this equality we'll begin by drawing the diagram representing the
left-hand side in a canonical way; then we'll push the legs around
(using the signed permutation relations) until the diagram is
separated into its constituent components. Our task is to keep
track of what happens to the sign out the front of the term during
this process.

We'll illustrate the following discussion with the example:
\[(w,\wp)=\left(AAAAAB
, \{\{2,9\},\{4,5\},\{8,11\}\}\,\right).
\]
We begin with the left-hand side. Construct the corresponding term
$T_{(w,\wp)}$ using the graphical approach to $\lambda$ discussed
in Section \ref{lambdarecall}. In the given example, we would
draw:
\[
T_{(w,\wp)}=(-1)^8\left(\frac{1}{2}\right)^3\
\raisebox{-8ex}{\scalebox{0.24}{\includegraphics{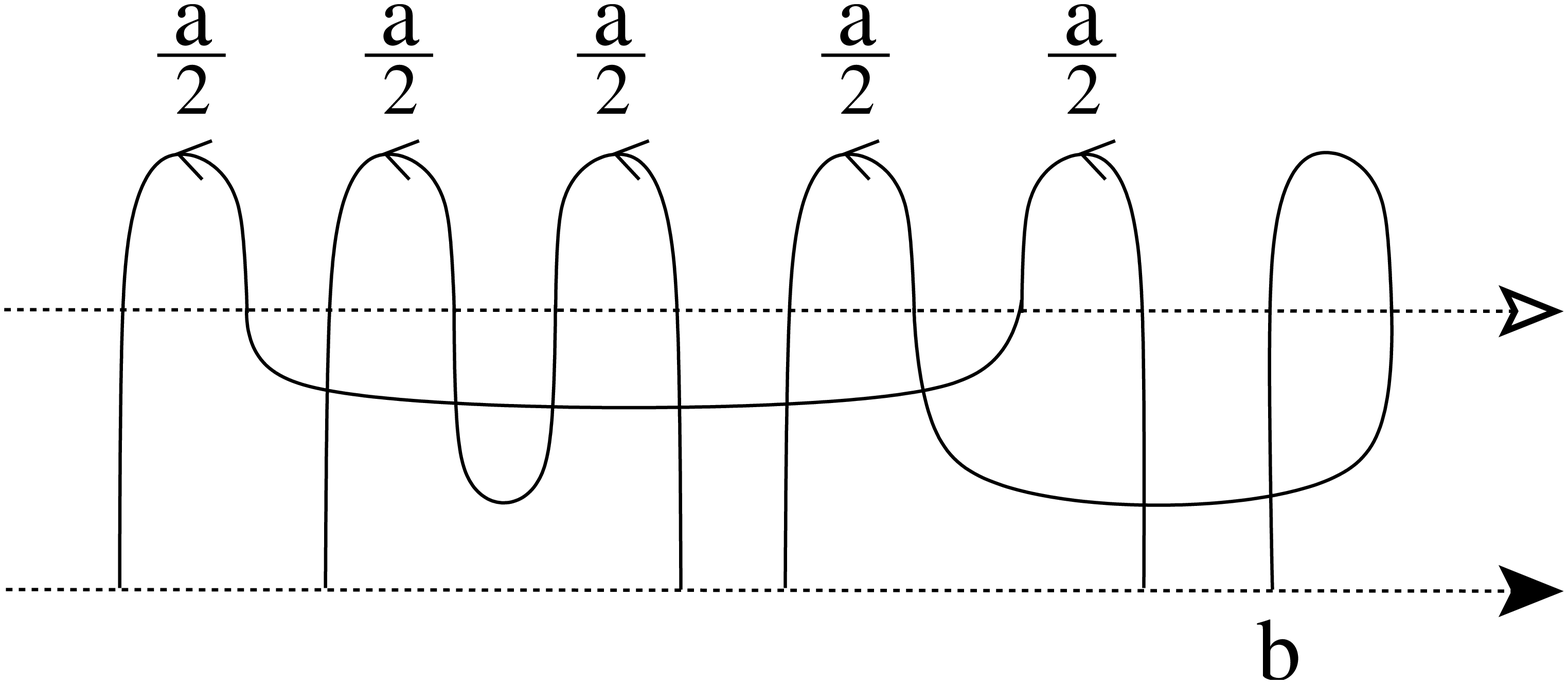}}}\
.
\]
Before we proceed, observe the crucial point: {\it if we draw the
diagram in this fashion, then the sign out the front of the term
is precisely a product of a $(-1)$ for every intersection
displayed by the drawing}. Call this observation ($\dagger$).

So this is the left hand side, $T_{(w,\wp)}$. To connect this with
the expression on the right-hand side, we will now separate this
diagram into its connected components. To do this we have to
perform permutations of the legs. Every time we permute a pair of
legs we pick up a $(-1)$, but also pick up an extra intersection
point in the drawing. So throughout this process observation $(\dagger)$ still holds. Continuing with our example:
\begin{eqnarray*}
T_{(w,\wp)}& = & (-1)^9\left(\frac{1}{2}\right)^3 \raisebox{-6ex}{\scalebox{0.24}{\includegraphics{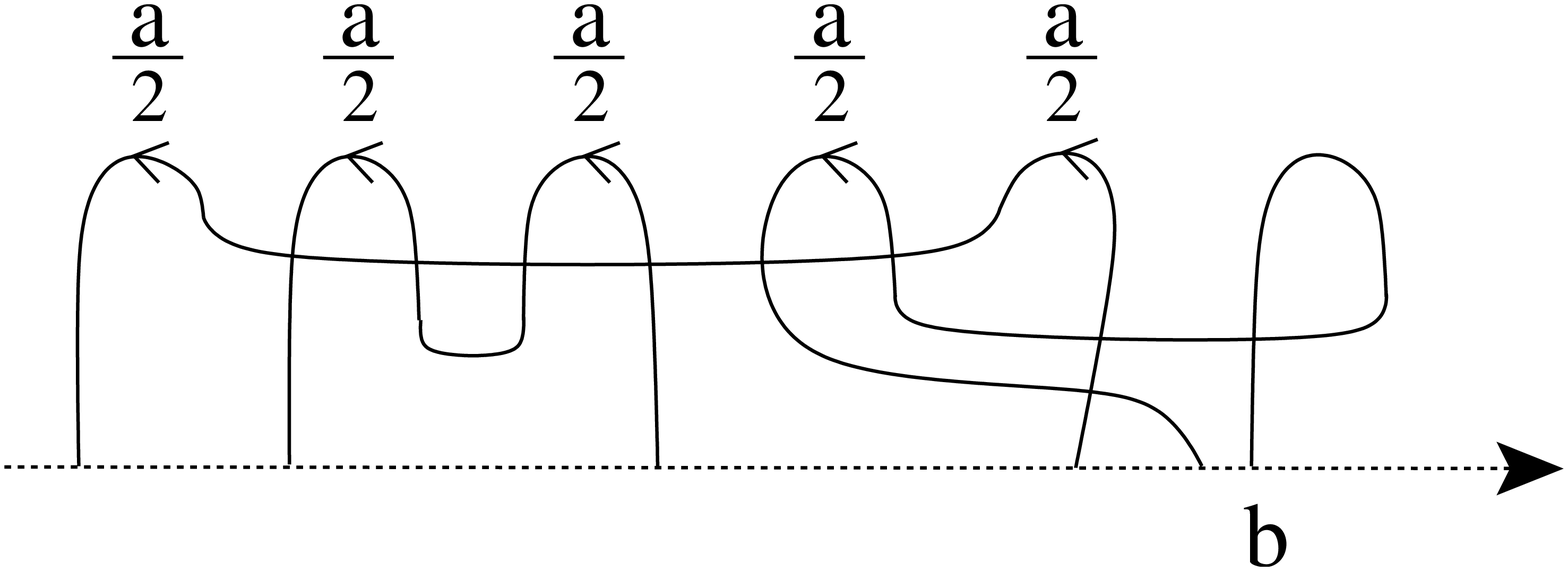}}},\\
&= & (-1)^{10}\left(\frac{1}{2}\right)^3
\raisebox{-6ex}{\scalebox{0.24}{\includegraphics{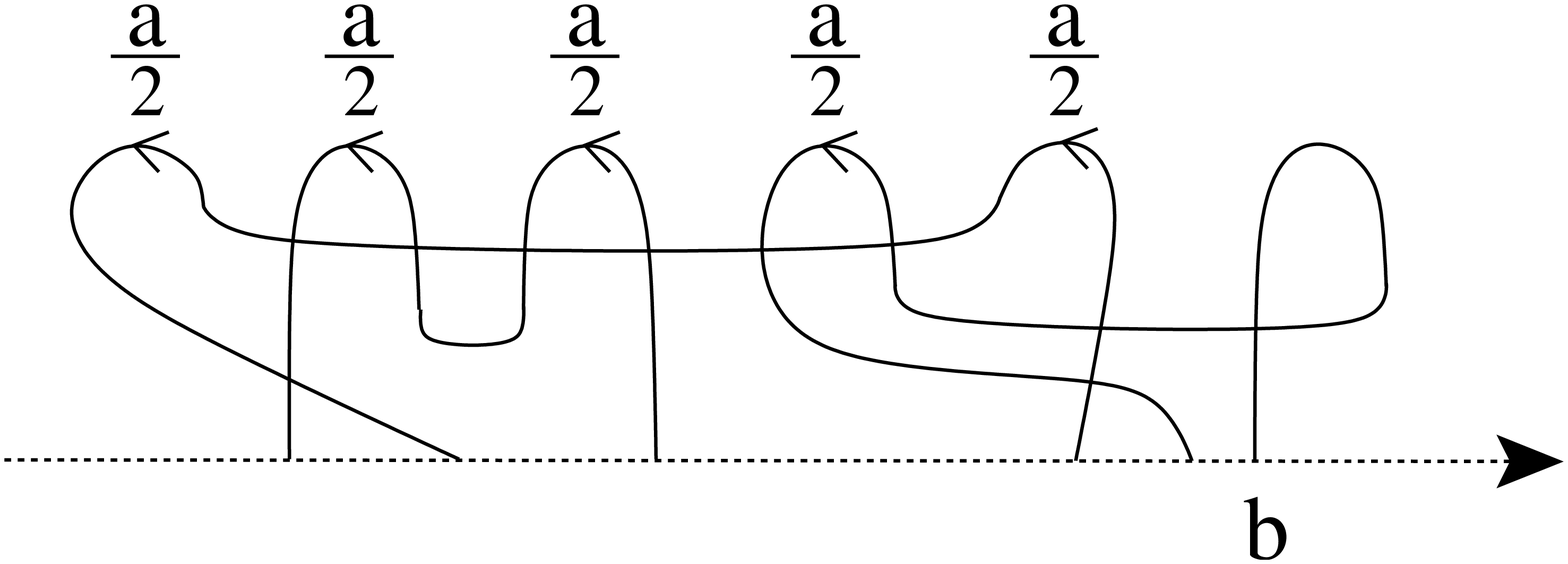}}},
\\
& = &(-1)^{11}\left(\frac{1}{2}\right)^3
\raisebox{-6ex}{\scalebox{0.24}{\includegraphics{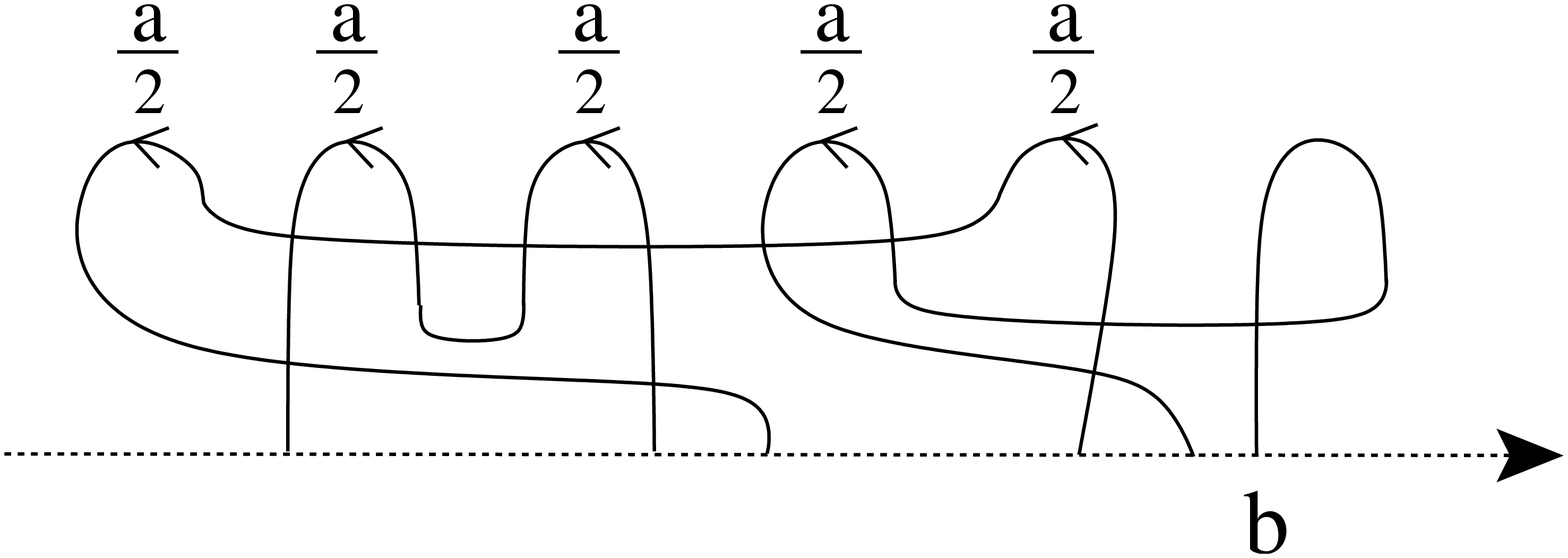}}}.
\end{eqnarray*}
We have now separated the legs of the connected components into
their respective components.

We finish by fully separating the connected components in the
drawing. To be precise, we can now do a combination of the
following two moves (where the dashed line indicates that there
are two connected components involved):
\[
\raisebox{-3.5ex}{\scalebox{0.16}{\includegraphics{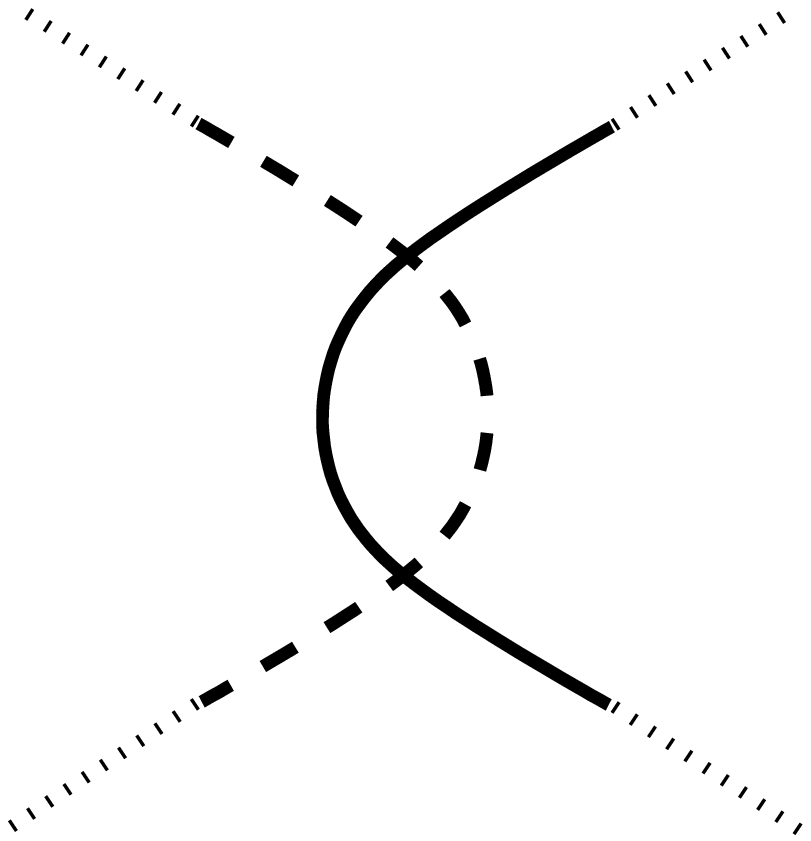}}}
\rightleftharpoons
\raisebox{-3.5ex}{\scalebox{0.16}{\includegraphics{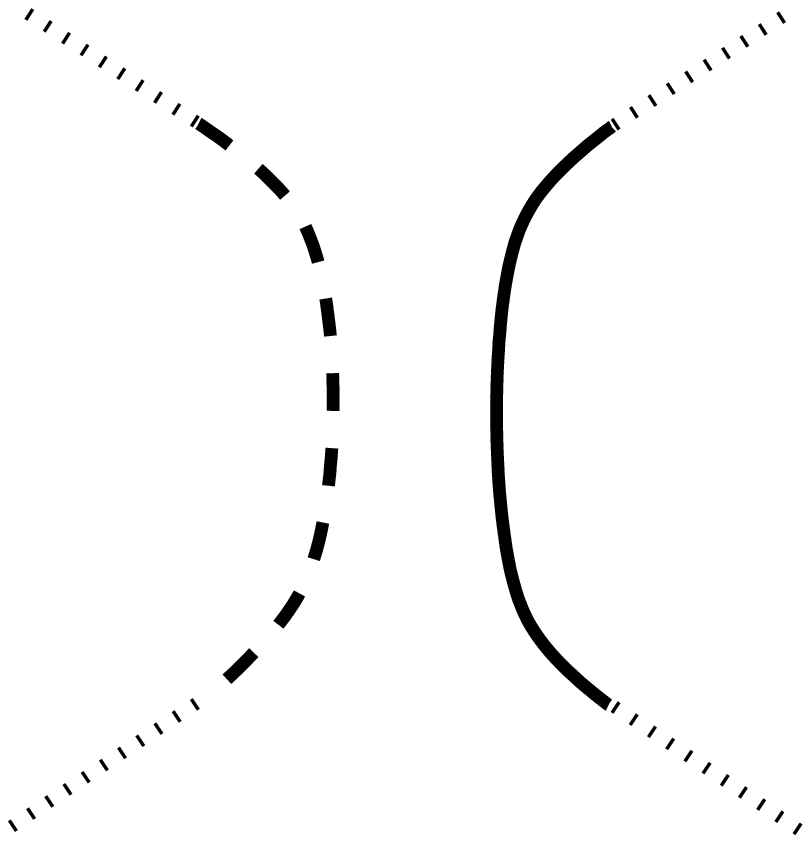}}} \ \ \ \
\text{and}\ \ \ \
\raisebox{-3.5ex}{\scalebox{0.18}{\includegraphics{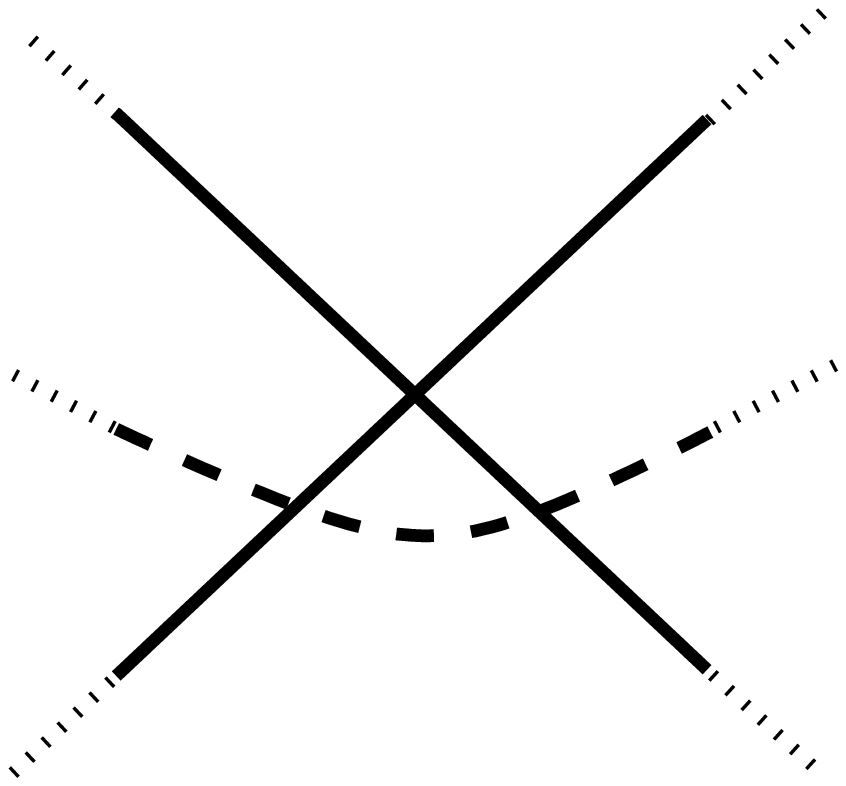}}}
\rightleftharpoons
\raisebox{-3.5ex}{\scalebox{0.18}{\includegraphics{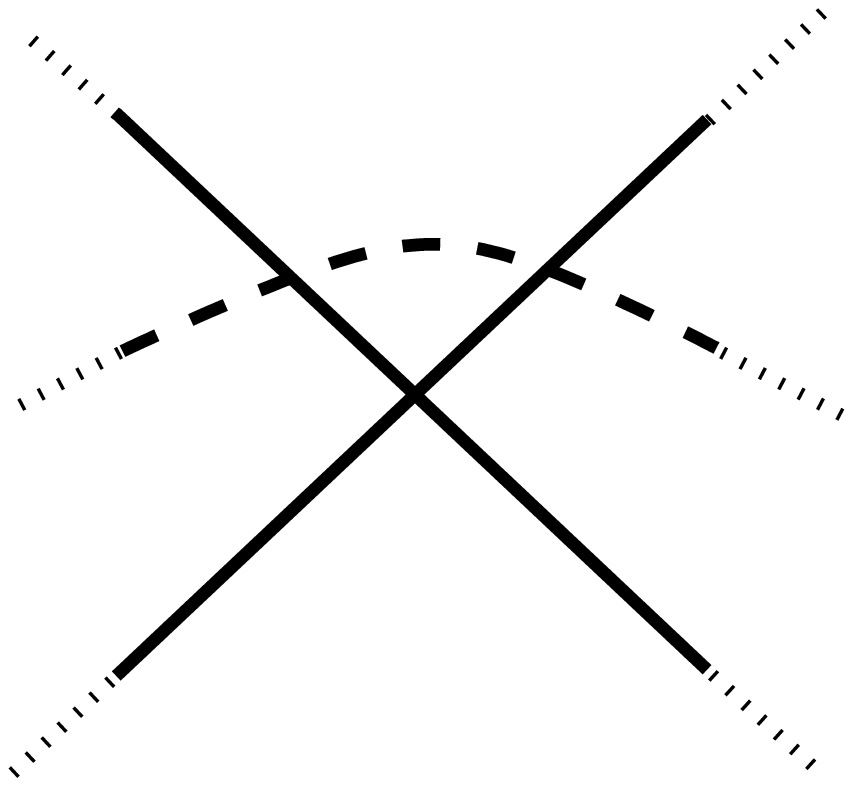}}}\ \ \ ,
\]
to separate the connected components in the drawing without
affecting the graphical structure of the drawings of the connected
components themselves. Note that these moves only change the
number of displayed intersections by an even number, so observation $(\dagger)$ still
holds after doing these moves.
In our example:
\begin{eqnarray*}
T_{(w,\wp)}& = & (-1)^{5}\left(\frac{1}{2}\right)^3
\raisebox{-6ex}{\scalebox{0.24}{\includegraphics{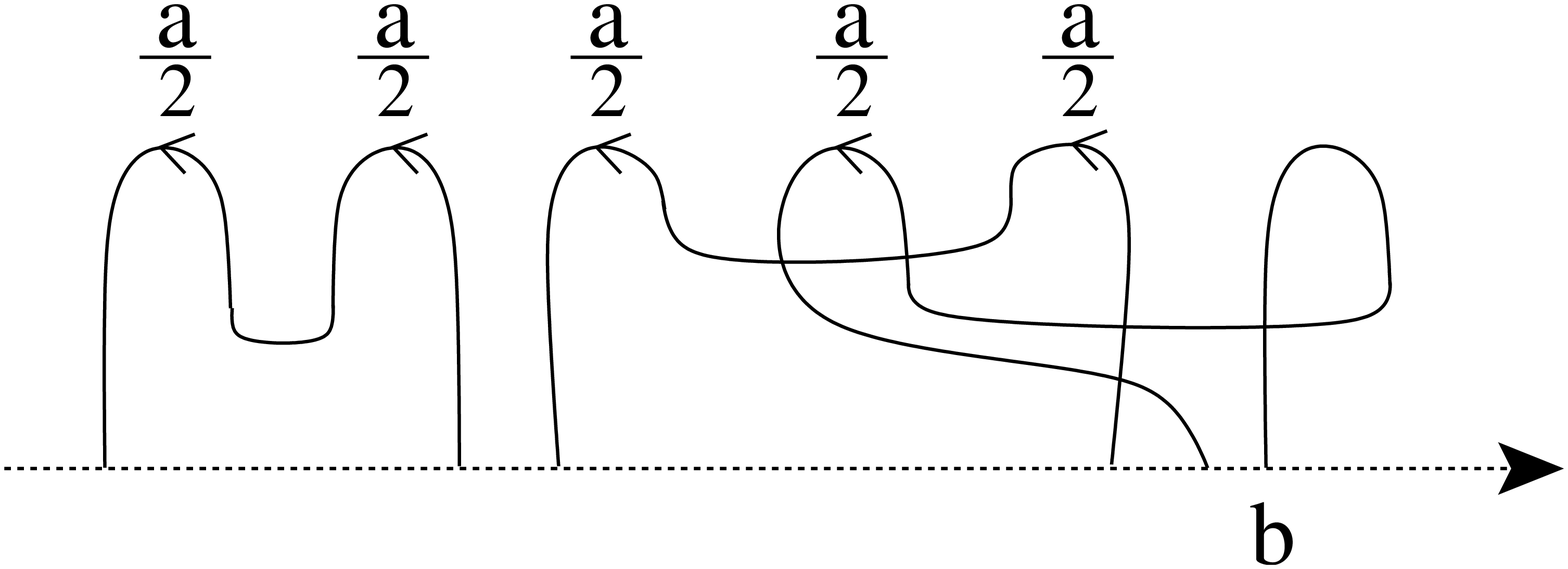}}}\ \
,\\& = & (-1)^1\left(\frac{1}{2}\right)^3
\raisebox{-6ex}{\scalebox{0.24}{\includegraphics{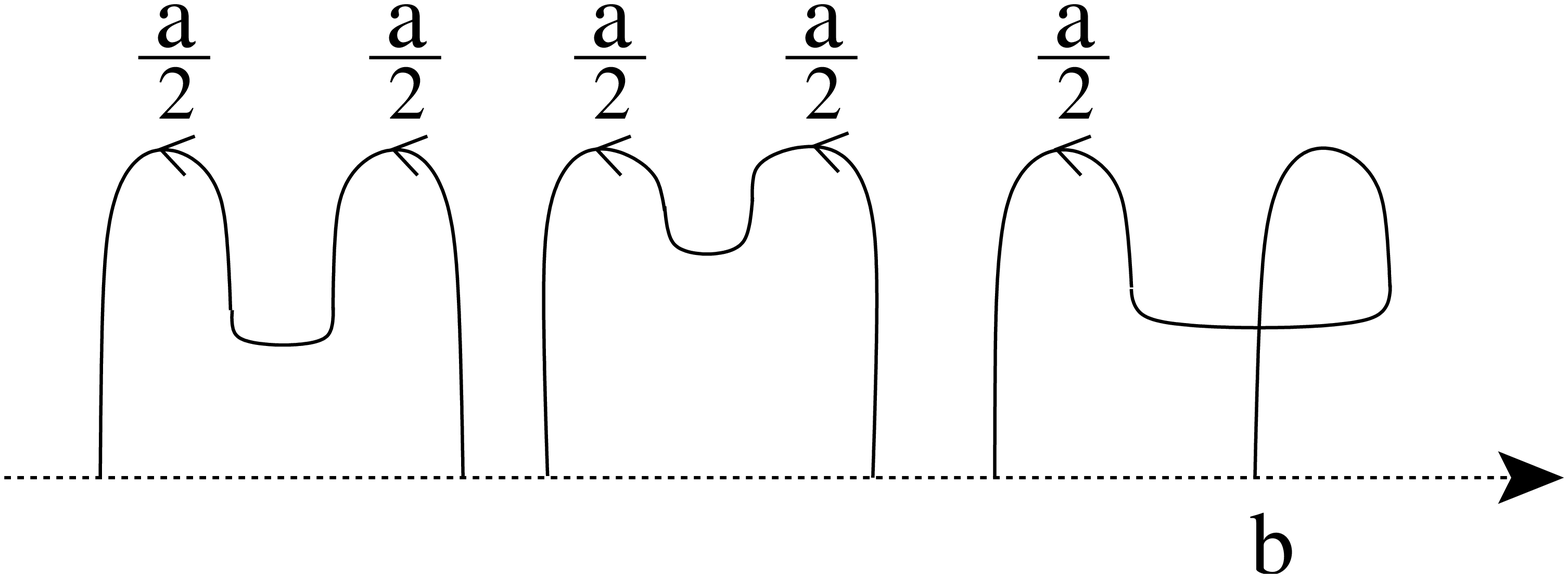}}}\ \
.
\end{eqnarray*}
After this factorization procedure the sign that we are left with,
then, is a $(-1)$ for every intersection displayed by the drawing.
Thus the sign is a $(-1)$ for every self-intersection of the
connected components, which exactly gives the right-hand side of
Equation \ref{factorhere}.
\begin{flushright}
$\Box$
\end{flushright}
{\it Proof of Lemma \ref{howmanycontent}.} So, we are given a
specific content function
\[
\tau : \mathcal{T}_\mathcal{C} \to \mathbb{N}_0
\]
and asked to count how many pairs $(w,\wp)$ have this content
function. For convenience, just say that the given content
consists of $n_1$ copies of the pair $(w_1,\wp_1)$, $n_2$ copies
of the pair $(w_2,\wp_2)$, and so on, up to $n_m$ copies of the
pair $(w_m,\wp_m)$.

Let $X\subset \mathcal{T}$ denote the set of pairs $(w,\wp)$
having this content. To count $X$ we'll construct a bijection from
$X$ to a certain set $Y$ of words. Consider the following set of
symbols:
\[
\{\zeta_{i,j}, 1\leq i\leq m, 1\leq j\leq n_i\}.
\]
For every such $i$ and $j$, take $|w_i|$ copies of the symbol
$\zeta_{i,j}$. Let $Y$ denote the set of words that can be built
from this collection of symbols which satisfy the restriction that
for every $i$ and $j<k$, the first appearance of the symbol
$\zeta_{i,j}$ appears to the left of the first appearance of the
symbol $\zeta_{i,k}$.

The map from $X$ to $Y$ is just to scan the corresponding diagram
$T_{(w,\wp)}$ factor by factor, and for each factor to write down
a symbol $\zeta_{i,j}$ if that factor is used by the $j$-th copy
of the $i$-th connected type.

For example, consider the following content:
\[
c_{(w,\wp)}((\omega,\rho)) = \left\{
\begin{array}{cl}
2 & \mbox{if $(\omega,\rho)=(AA,\{\{1,4\},\{2,3\}\})$}, \\
1 & \mbox{if $(\omega,\rho)=(AA,\{\{1,3\},\{2,4\}\})$}, \\
1 & \mbox{if $(\omega,\rho)=(BA,\{\{1,2\}\})$}, \\
0 & \mbox{otherwise.}
\end{array}
\right.
\]

Here is an example of how a pair $(w,\wp)$ with this content gives
a word in these symbols:
\[
\raisebox{-6.5ex}{\scalebox{0.24}{\includegraphics{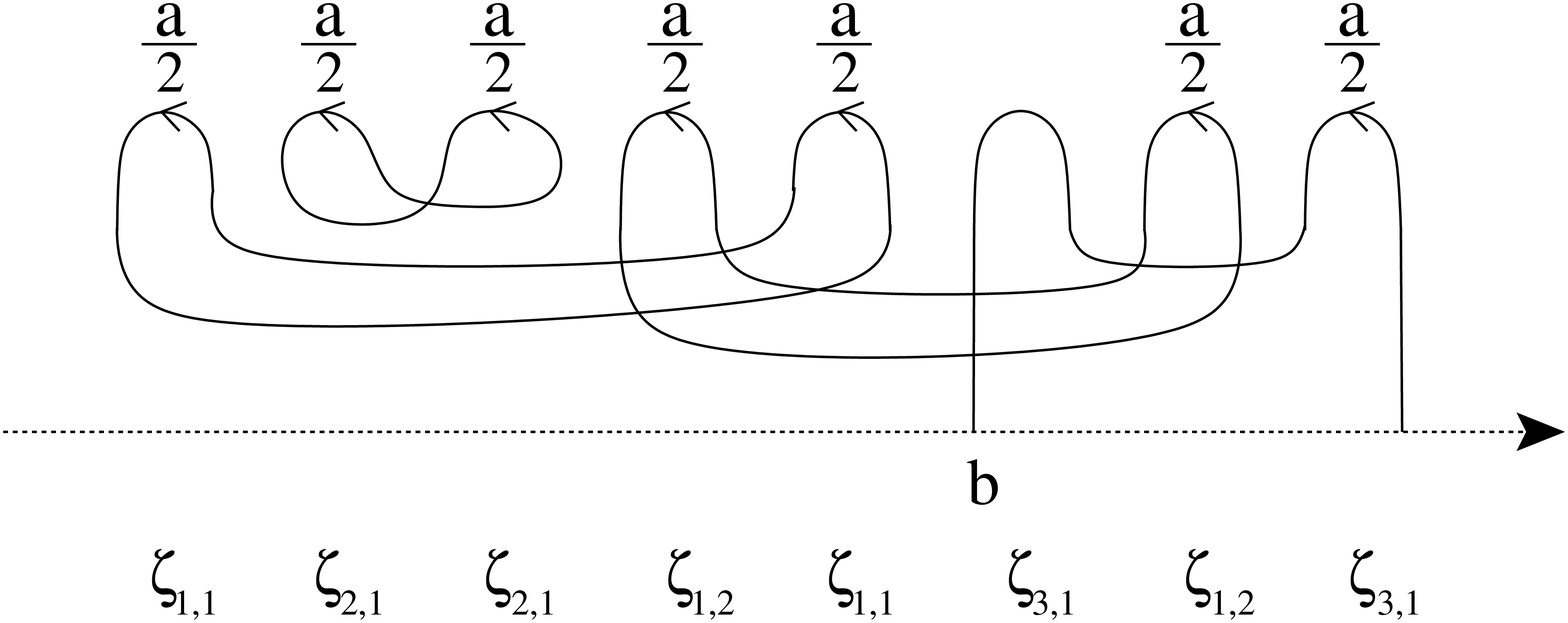}}}\
\ \ \ .\\[0.25cm]
\]
This map sets up a bijection between $X$ and $Y$. It is a
straightforward combinatorial problem to count $Y$.
\begin{flushright}
$\Box$
\end{flushright}
\subsection{The computation of $\sum_{\tau\in\mathcal{T}_\mathcal{C}}\frac{1}{|\tau|!}T_\tau$.}
\label{determineconnected}
Our task in this section is to write down the series of all
possible terms that can arise by the following procedure:
\begin{enumerate}
\item{Putting down a number of copies of the diagrams
\[
\raisebox{-2ex}{\scalebox{0.26}{\includegraphics{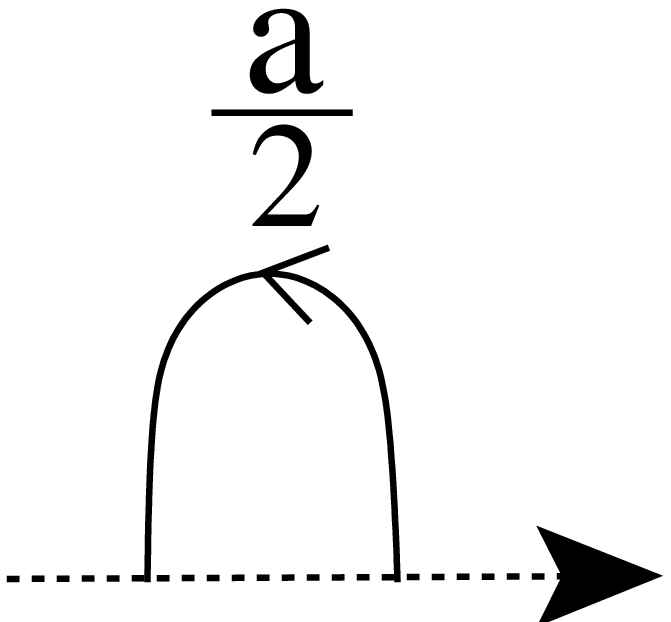}}}\ \ \
\ \mbox{and}\ \ \ \
\raisebox{-3.5ex}{\scalebox{0.26}{\includegraphics{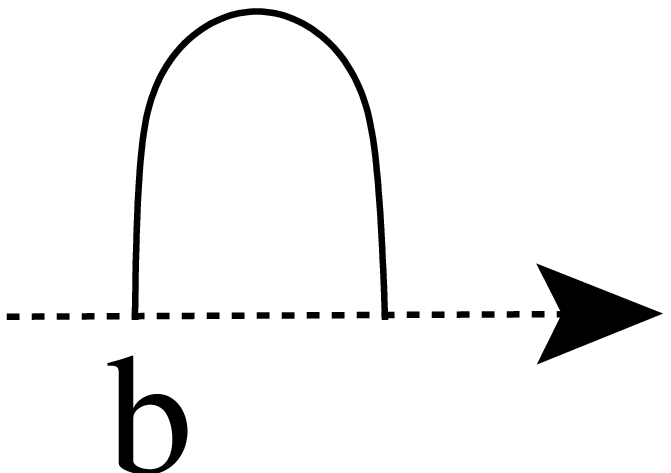}}}
\]
in some order along an orienting line. If we use $n$ factors in
total then we multiply the diagram by $\frac{1}{n!}$.}
\item{Joining up (with signs) the $\bot$-legs in such a way as to
produce a {\it connected} diagram  (multiplying by a factor of
$\frac{1}{2}$ for every pair of legs joined up).}
\end{enumerate}
We  begin by observing that the connected diagrams that can arise
in this way fall into exactly four groups. Below, we'll refer to a
leg of the form
$\raisebox{-0.5ex}{\scalebox{0.25}{\includegraphics{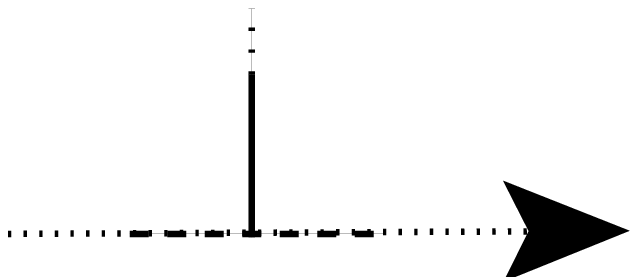}}}$
as a $\bot$-leg and a leg of the form
$\raisebox{-2.2ex}{\scalebox{0.25}{\includegraphics{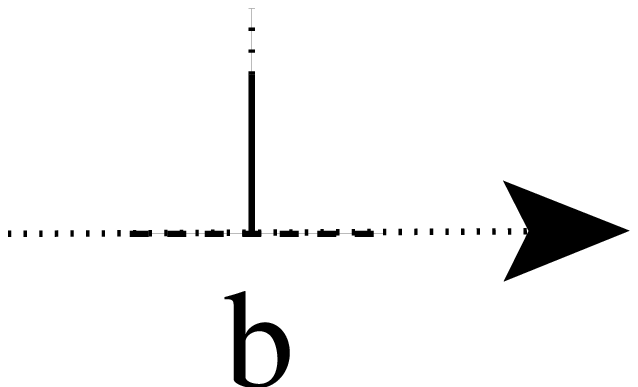}}}$\
as a $b$-leg. The four possibilities are:
\begin{itemize}
\item{ The resulting diagram has exactly two remaining legs, and
they are both $\bot$-legs. For example: \[
\raisebox{-3.5ex}{\scalebox{0.24}{\includegraphics{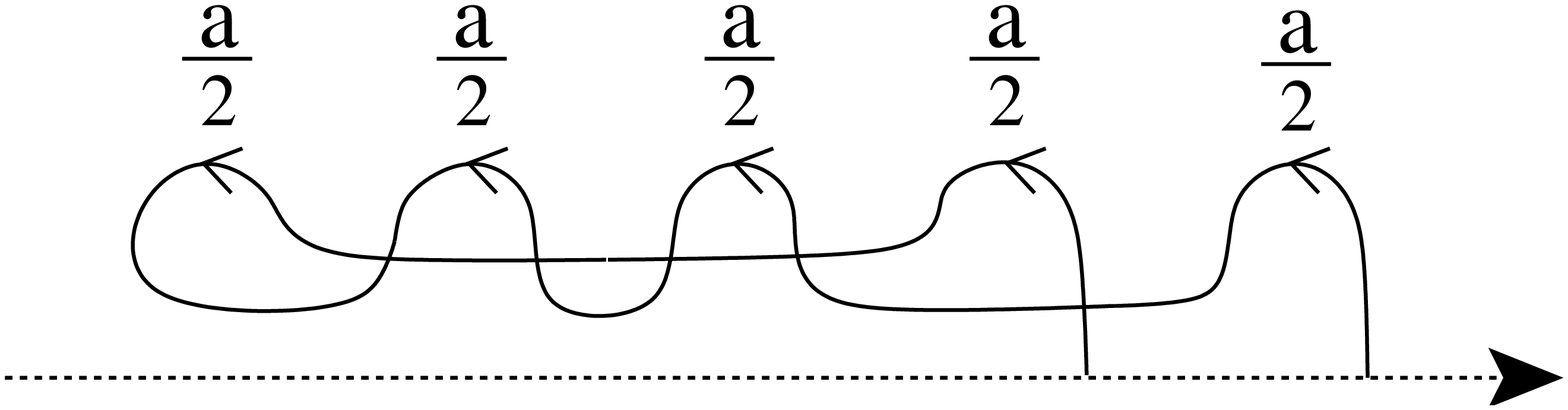}}}\ .\\
\]} \item{The resulting diagram has exactly two
remaining legs, and they are both $b$-legs. For example: \[
\raisebox{-3.5ex}{\scalebox{0.24}{\includegraphics{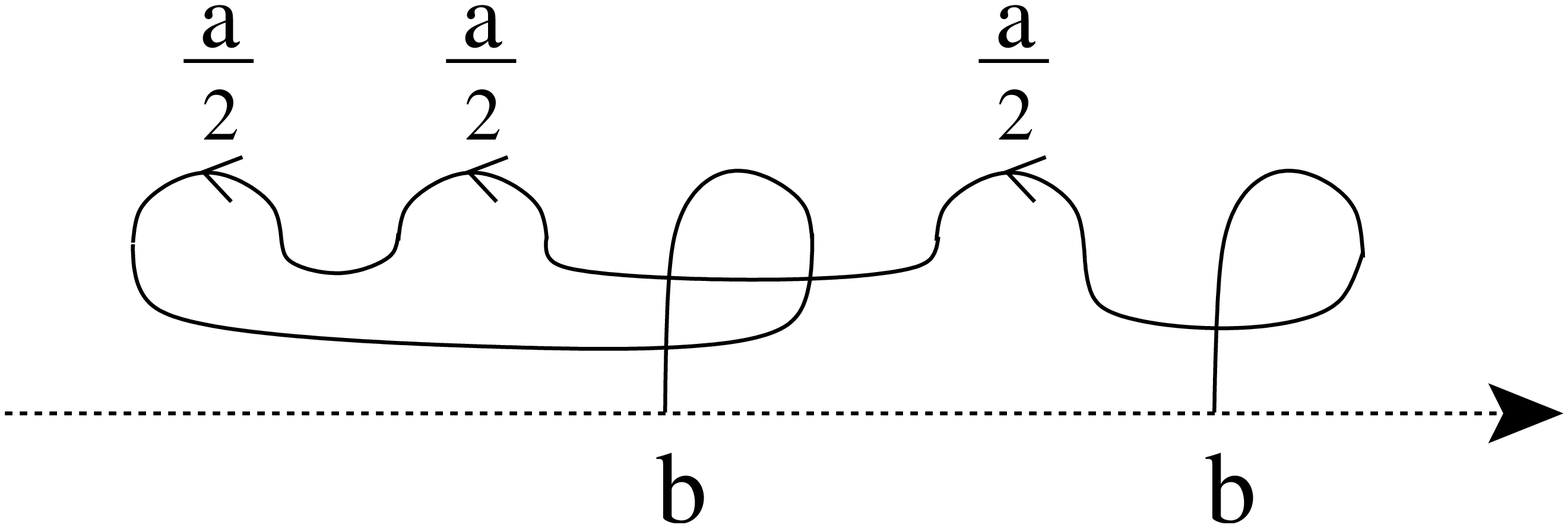}}}\ .\\
\]} \item{The resulting diagram has exactly two
remaining legs, one $\bot$-leg and one $b$-leg. For example: \[
\raisebox{-3.5ex}{\scalebox{0.24}{\includegraphics{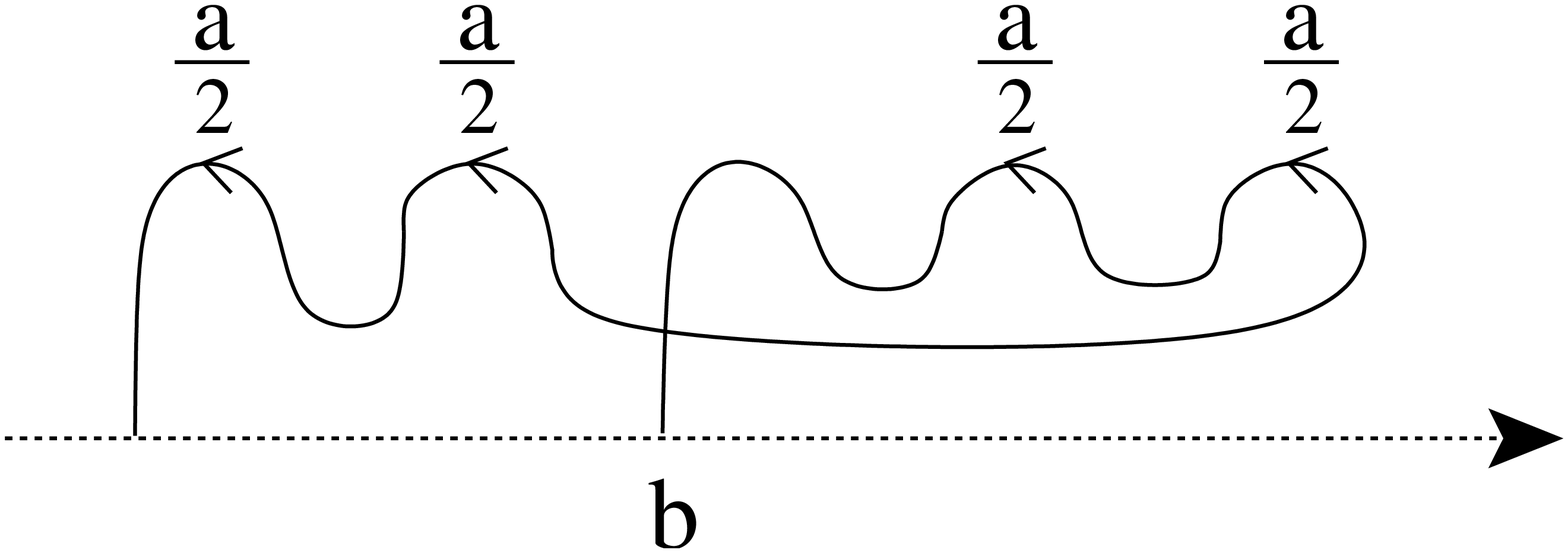}}}\ .\\
\]} \item{The resulting diagram has no remaining
legs. For example:
\[
\raisebox{-3.5ex}{\scalebox{0.24}{\includegraphics{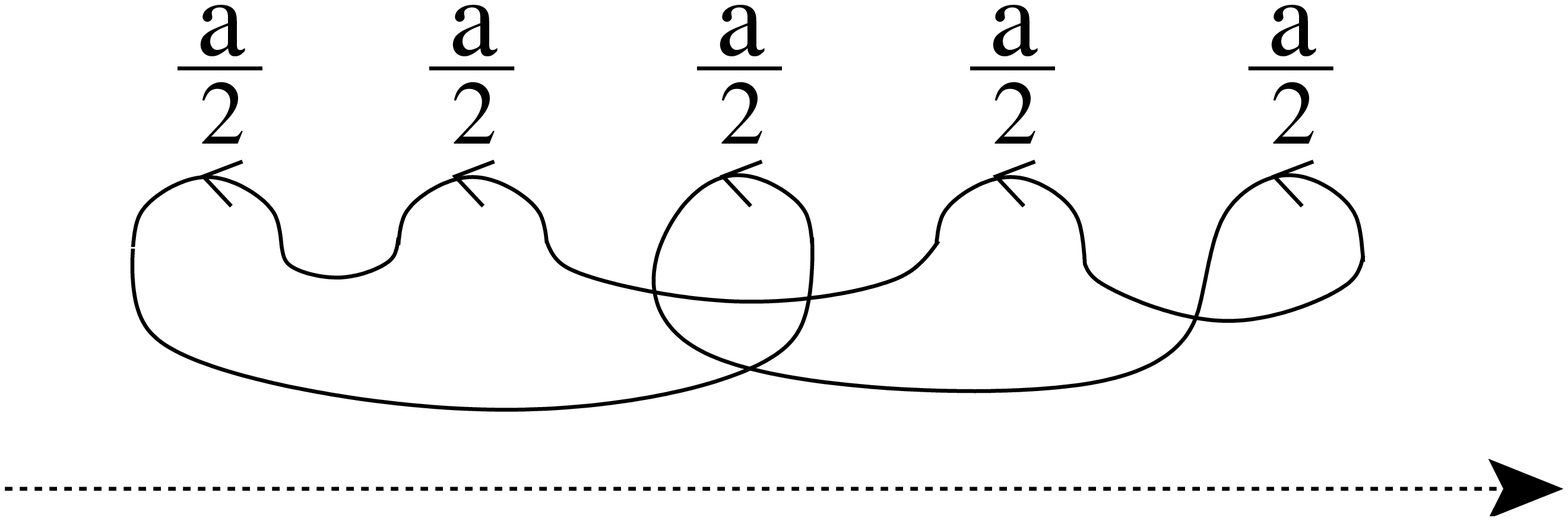}}}\ .\\
\]}
\end{itemize}
Denote these different contributions in the following way:
\[ \sum_{\tau\in\mathcal{T}_\mathcal{C}}\frac{1}{|\tau|!}T_\tau =
C_{||}+C_{|b}+C_{bb}+C_o .
\]
We'll compute these different contributions in turn. The
computation of the contribution $C_{||}$ will be described in some
detail. The other contributions will be computed in much the same
way and will be described in less detail. We remark that the
subsection describing $C_{||}$ introduces certain definitions used
in the other subsections.
\subsubsection{The contribution $C_{||}$.}
The goal of this subsection is the computation that:
\begin{equation}\label{iicomp}
C_{||} =
\raisebox{-3.5ex}{\scalebox{0.27}{\includegraphics{ansii}}}.
\end{equation}

So consider some $n\geq 2$ ($n=1$ we'll put in by hand). We wish
to compute the contributions from the {\it connected} diagrams
with {\it 2} legs that we can get by doing signed pairings of the
legs of the following term (where the {\bf blocks} have been
numbered for convenience): \[ \frac{1}{n!}\
\raisebox{-4ex}{\scalebox{0.24}{\includegraphics{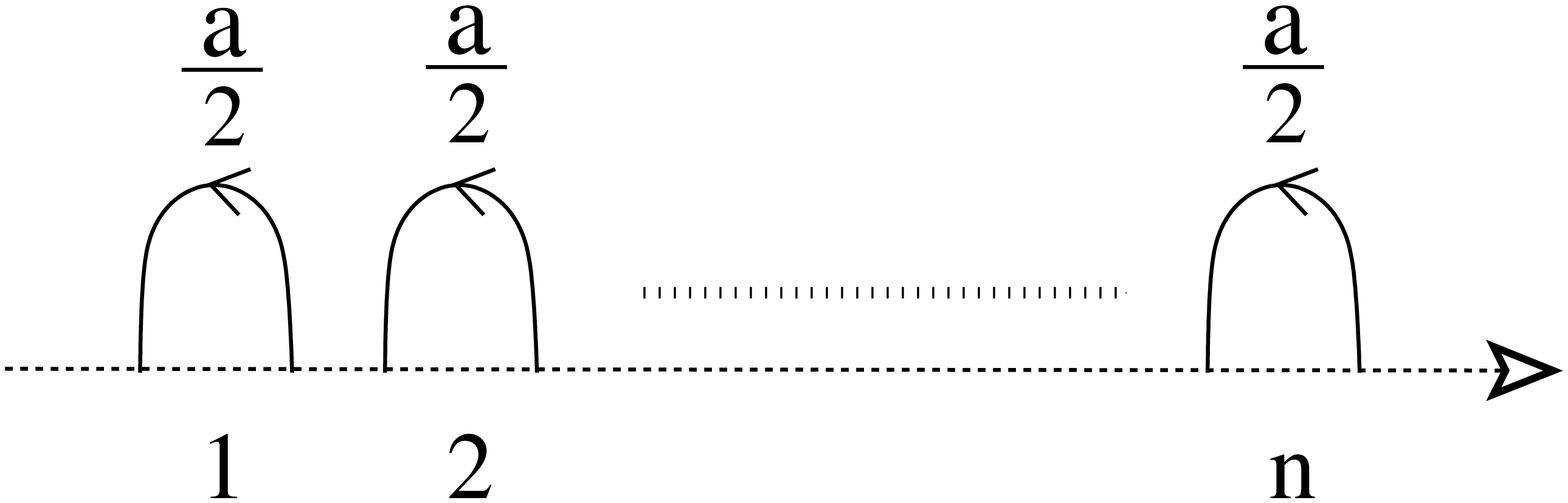}}}\ .
\]
To enumerate the possible pairings we'll construct a certain set
$\overrightarrow{\Gamma}_n$. An element of the set
$\overrightarrow{\Gamma}_n$ is a word which uses each of the
symbols $\{1,2,\ldots,n\}$ precisely once, where, in addition,
each symbol $s$ is decorated by either an arrow pointing to the
left $\overleftarrow{s}$ or an arrow pointing to the right
$\overrightarrow{s}$. For example, we'll soon see that
\[
\overrightarrow{2}\overrightarrow{3}\overleftarrow{5}\overrightarrow{1}\overrightarrow{4}
\in \overrightarrow{\Gamma}_5.
\]
The set $\overrightarrow{\Gamma}_n$ is defined to be the set of
words of this form based on the symbols $\{1,\ldots,n\}$ subject
to the single constraint that the first symbol is less than the
final symbol. To see that they agree with our definition the
reader might like to check that
$|\overrightarrow{\Gamma}_n|=n!2^{n-1}$.

For example:
\begin{eqnarray*}
\overrightarrow{\Gamma}_3 & = & \left\{
\overrightarrow{1}\overrightarrow{2}\overrightarrow{3}\ ,\
\overrightarrow{1}\overrightarrow{2}\overleftarrow{3}\ ,\
\overrightarrow{1}\overleftarrow{2}\overrightarrow{3}\ ,\
\overrightarrow{1}\overleftarrow{2}\overleftarrow{3}\ ,\ \right.\\
& &\ \ \overleftarrow{1}\overrightarrow{2}\overrightarrow{3}\ ,\
\overleftarrow{1}\overrightarrow{2}\overleftarrow{3}\ ,\
\overleftarrow{1}\overleftarrow{2}\overrightarrow{3}\ ,\
\overleftarrow{1}\overleftarrow{2}\overleftarrow{3}\ ,\ \\
& &\ \ \overrightarrow{2}\overrightarrow{1}\overrightarrow{3}\ ,\
\overrightarrow{2}\overrightarrow{1}\overleftarrow{3}\ ,\
\overrightarrow{2}\overleftarrow{1}\overrightarrow{3}\ ,\
\overrightarrow{2}\overleftarrow{1}\overleftarrow{3}\ ,\ \\
& &\ \ \overleftarrow{2}\overrightarrow{1}\overrightarrow{3}\ ,\
\overleftarrow{2}\overrightarrow{1}\overleftarrow{3}\ ,\
\overleftarrow{2}\overleftarrow{1}\overrightarrow{3}\ ,\
\overleftarrow{2}\overleftarrow{1}\overleftarrow{3}\ ,\ \\
& &\ \ \overrightarrow{1}\overrightarrow{3}\overrightarrow{2}\ ,\
\overrightarrow{1}\overrightarrow{3}\overleftarrow{2}\ ,\
\overrightarrow{1}\overleftarrow{3}\overrightarrow{2}\ ,\
\overrightarrow{1}\overleftarrow{3}\overleftarrow{2}\ ,\ \\
& &\left.\ \
\overleftarrow{1}\overrightarrow{3}\overrightarrow{2}\ ,\
\overleftarrow{1}\overrightarrow{3}\overleftarrow{2}\ ,\
\overleftarrow{1}\overleftarrow{3}\overrightarrow{2}\ ,\
\overleftarrow{1}\overleftarrow{3}\overleftarrow{2}\ \right\}.
\end{eqnarray*}

We'll demonstrate how the different pairings correspond with the
elements of the sets $\overrightarrow{\Gamma}_n$ by means of the
following example:
\[
\raisebox{-3.5ex}{\scalebox{0.24}{\includegraphics{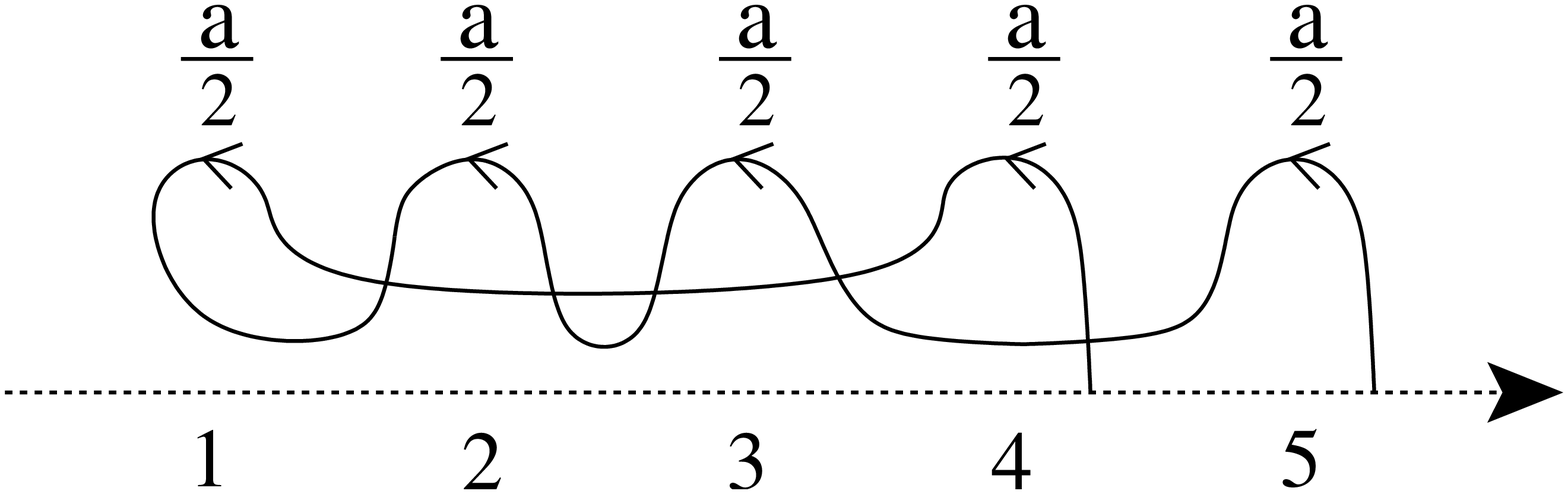}}}\ .\\
\]
To write down the word corresponding to some gluing, begin at the
base of the left-most of the two remaining legs. Now traverse the
graph by simply following the edge until you reach the second of
the two remaining legs. Write down the symbols $\{1,\ldots,n\}$ in
the order in which you visit the different blocks. In this case,
you should write down:
\[
41235.
\]
Now decorate this word with arrows to record how you traverse each
block (whether from left to right, or from right to left). (Take
care not to confuse this arrow with the arrow which locally
orients the edges around the $a$ labels.) The decoration in this
case is this:
\[
\overleftarrow{4}\overleftarrow{1}\overrightarrow{2}\overrightarrow{3}\overrightarrow{5}.
\]
This word contains sufficient instructions for uniquely
reconstructing the pairing, so we get precisely one of the
relevant pairings for each element of $\overrightarrow{\Gamma}_n$.

Given some word $w\in \overrightarrow{\Gamma}_n$, let $\gamma_w$
denote the corresponding contribution to $C_{||}$ (that is,
including the factorial out the front as well as the signs that
arise from the gluing and a factor of $\left(\frac{1}{2}\right)$
for every joined pair). For example:
\begin{eqnarray*}
\gamma_{\overleftarrow{4}\overleftarrow{1}\overrightarrow{2}\overrightarrow{3}\overrightarrow{5}}
& = & \frac{(-1)^5}{5!}\left(\frac{1}{2}\right)^4
\raisebox{-1.5ex}{\scalebox{0.25}{\includegraphics{possibilityA}}}
\\[0.15cm]
& = & -\frac{1}{2^45!}\,
\raisebox{-1.5ex}{\scalebox{0.22}{\includegraphics{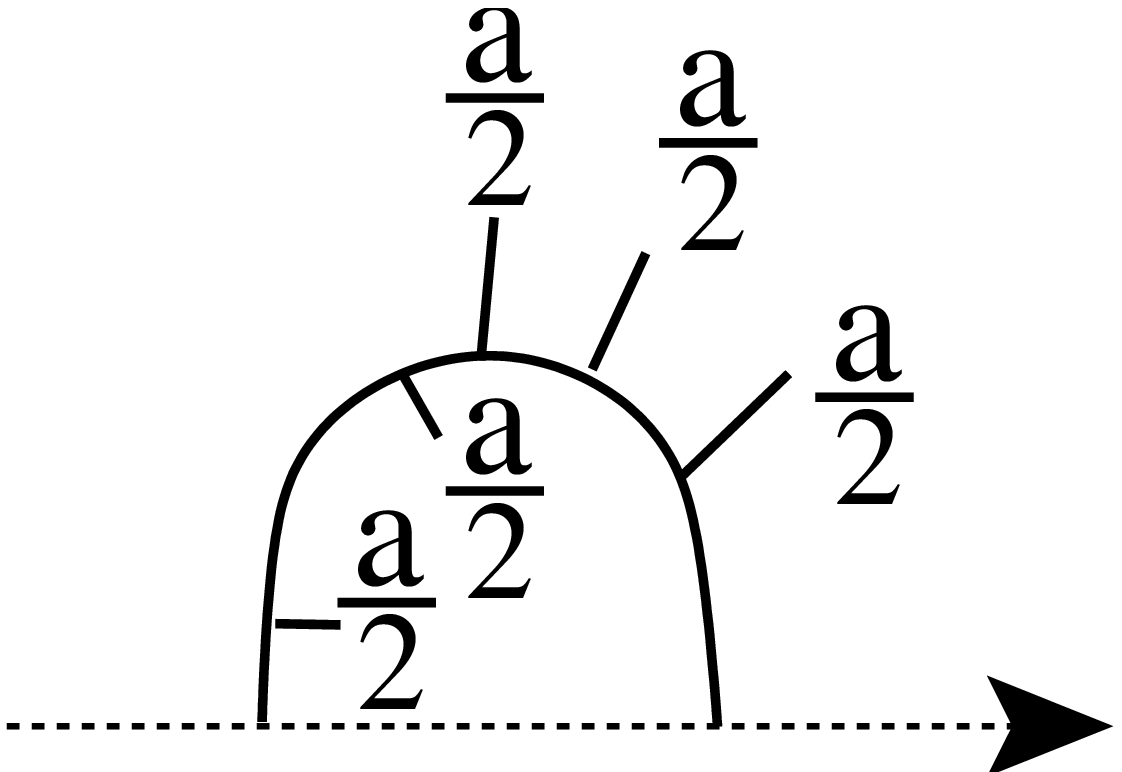}}}\
\ =\ \ \frac{1}{2^45!}\,
\raisebox{-1.5ex}{\scalebox{0.22}{\includegraphics{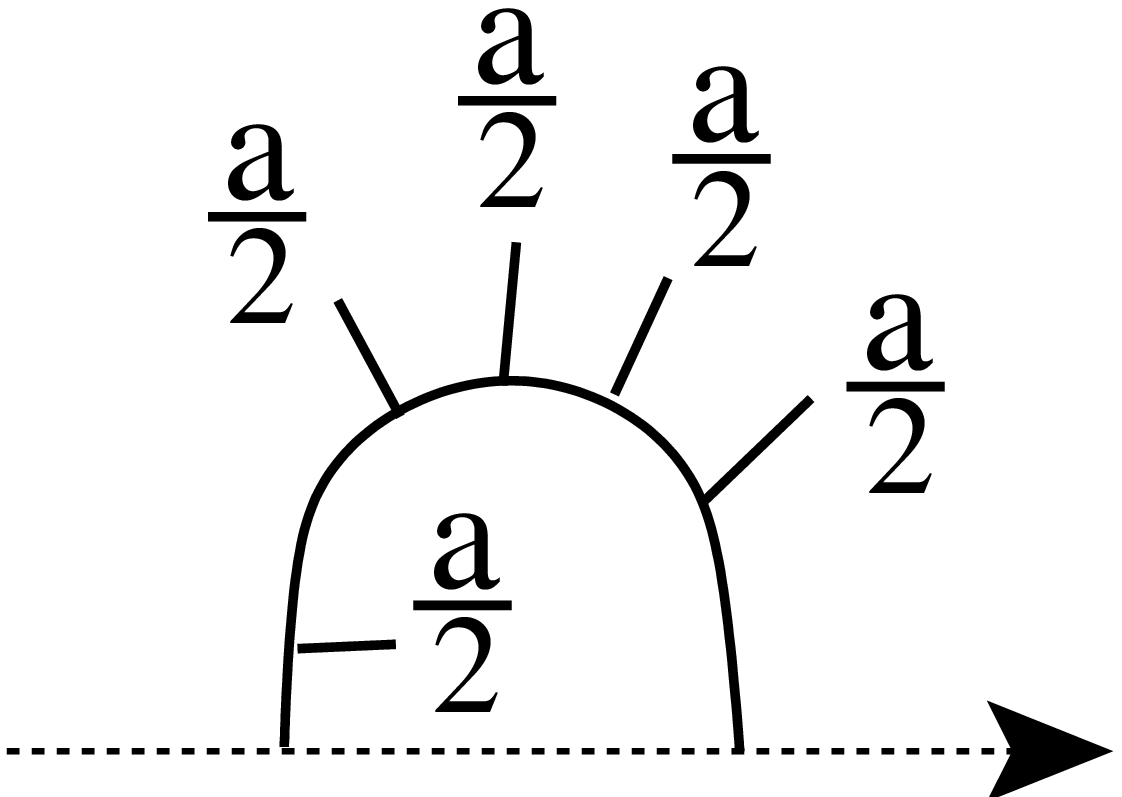}}}
\\[0.15cm]
& = & -\frac{1}{2^45!}\,
\raisebox{-1.5ex}{\scalebox{0.22}{\includegraphics{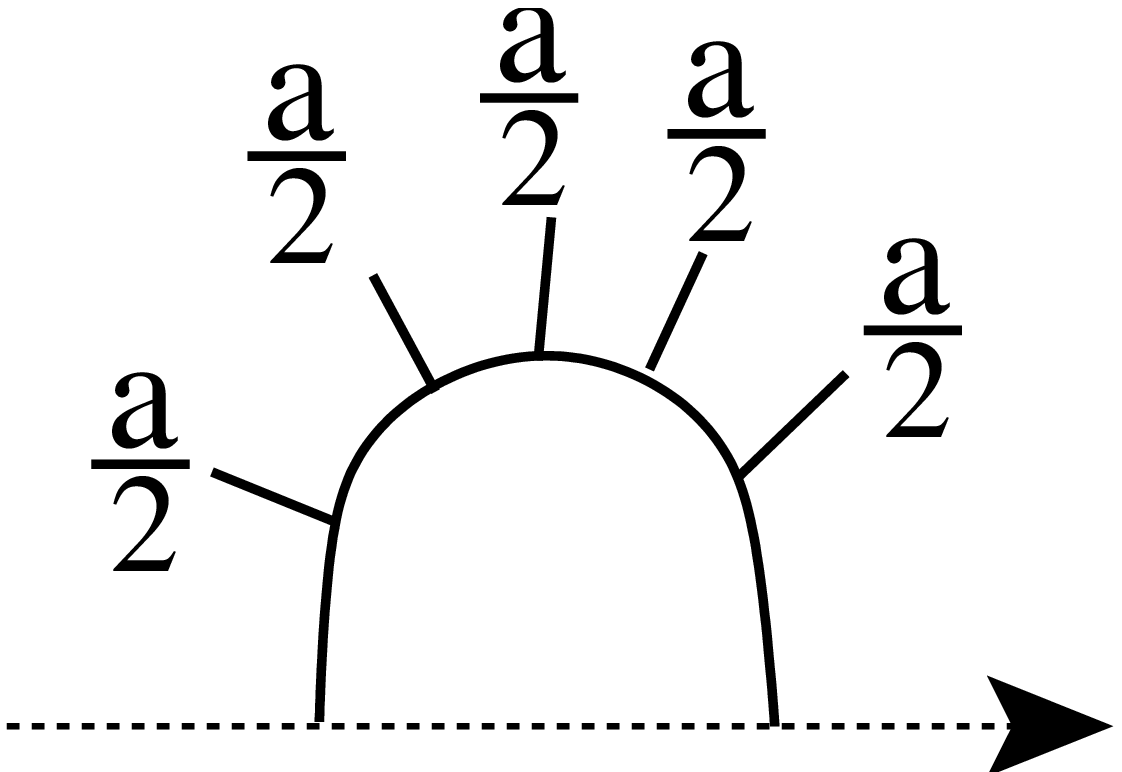}}}
\ \  =\ \ -\frac{1}{2^45!}\,
\raisebox{-1.5ex}{\scalebox{0.22}{\includegraphics{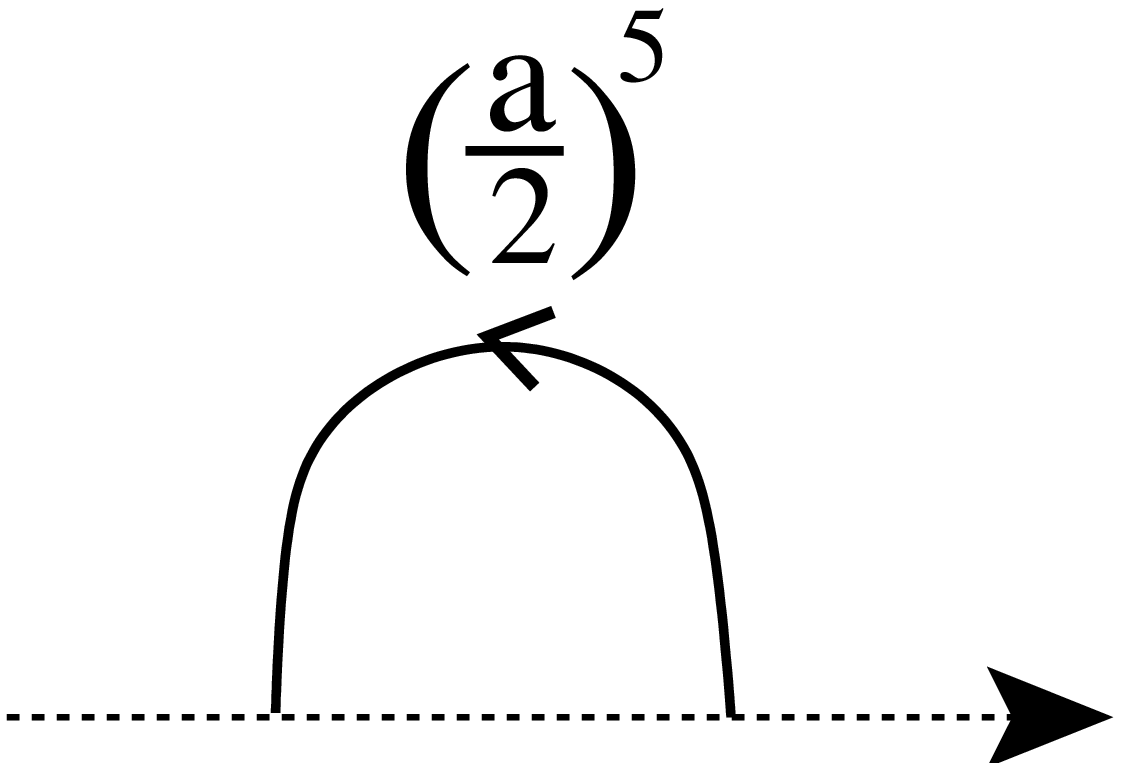}}}.\\[0.05cm]
\end{eqnarray*}With this definition we can write the contribution that we are seeking to compute in
the following way:
\begin{equation}\label{||eqn}
C_{||} =
\raisebox{-1.75ex}{\scalebox{0.17}{\includegraphics{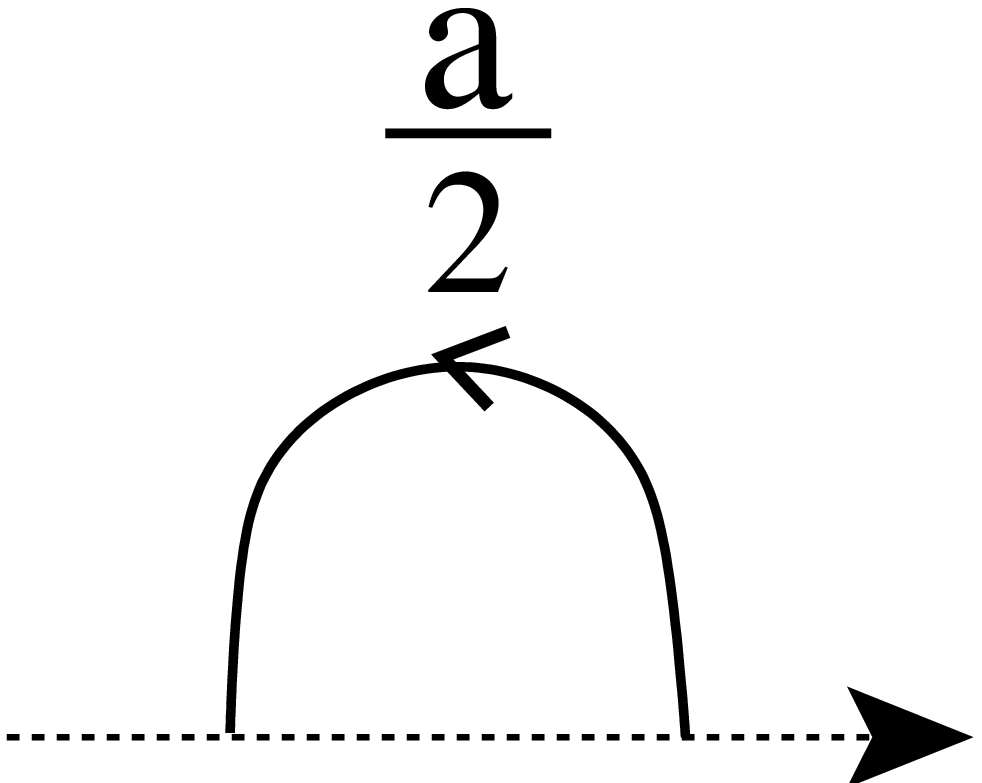}}}+\sum_{n=2}^\infty
\sum_{w\in \overrightarrow{\Gamma}_n} \gamma_w.
\end{equation}
Observe that, for some fixed $n$, all the $\gamma_w$, for $w\in
\overrightarrow{\Gamma}_n$, are equal, up to sign:
\[
\gamma_w = \pm \frac{1}{2^{n-1}n!}\,
\raisebox{-2ex}{\scalebox{0.2}{\includegraphics{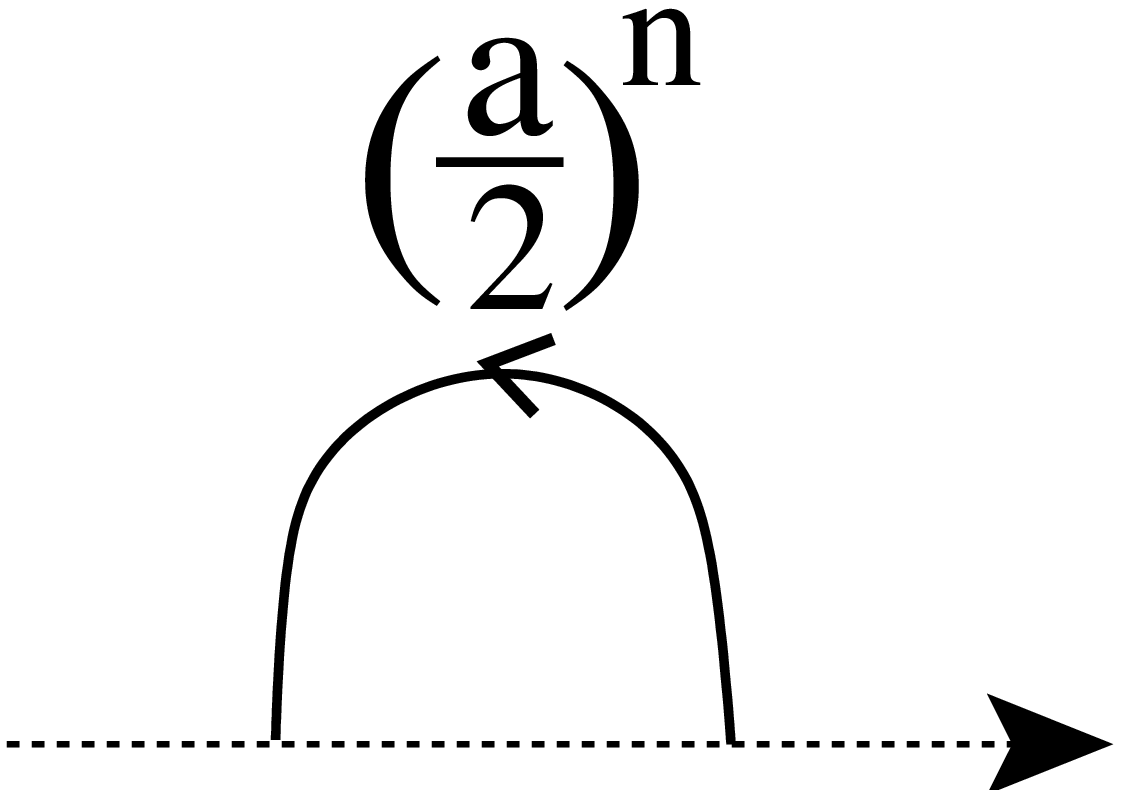}}}.
\]
The difficulty, then, in computing the sum $\left(\sum_{w\in
\overrightarrow{\Gamma}_n}\gamma_w\right)$, is to determine the
signs of the various $\gamma_w$. This difficulty is dealt with by
the next lemma, whose proof is later in this section.

Let $\Gamma_n$ denote the set of words in the symbols
$\{1,\ldots,n\}$ with the property that the right-most symbol in a
word has greater value than the left-most symbol. There is an
obvious $2^n$-to-$1$ forgetful map:
\[
f_\Gamma : \overrightarrow{\Gamma}_n\rightarrow \Gamma_n.
\]
Define the {\it descent} of a word $w\in\Gamma_n$, denoted $d(w)$,
to be the number of times in which the value of the symbol
decreases as you scan the word from left to right. For example:
\[
d(41235)=1,
\]
because the value decreases once (going from $4$ to $1$).
\begin{lem}\label{standardformlemma}
Let $n\geq 2$ and let $w\in \overrightarrow{\Gamma}_n$. Then:
\[
\gamma_w = (-1)^{d(f_\Gamma(w))}\frac{1}{2^{n-1}n!} \,
\raisebox{-3ex}{\scalebox{0.2}{\includegraphics{Nterm}}}.
\]
\end{lem}
Substituting this computation into Equation \ref{||eqn}, we find
that:
\begin{eqnarray*}C_{||} & = &
\raisebox{-1.5ex}{\scalebox{0.19}{\includegraphics{firstterm}}}+\sum_{n=2}^\infty
\sum_{w\in \overrightarrow{\Gamma}_n} \gamma_w\ , \\
& = &
\raisebox{-1.5ex}{\scalebox{0.19}{\includegraphics{firstterm}}}+\sum_{n=2}^\infty\sum_{w\in
\overrightarrow{\Gamma}_n}\left(
(-1)^{d(f_\Gamma(w))}\frac{1}{2^{n-1}n!} \,
\raisebox{-1.95ex}{\scalebox{0.19}{\includegraphics{Nterm}}}\right),
\\
& = &
\raisebox{-1.5ex}{\scalebox{0.19}{\includegraphics{firstterm}}}+\sum_{n=2}^\infty
\left(\frac{2\sum_{w\in\Gamma_n}(-1)^{d(w)}}{n!}\right)
\raisebox{-1.95ex}{\scalebox{0.19}{\includegraphics{Nterm}}}\ .
\end{eqnarray*}
(In the last equality above the $2^{n-1}$ is cancelled by a $2^n$
arising from the fact that there are $2^n$ words in
$\overrightarrow{\Gamma}_n$ for every word in $\Gamma_n$.) Now
note that if $n$ is even then:
\[
\raisebox{-1.95ex}{\scalebox{0.19}{\includegraphics{Nterm}}} = 0\
.
\]
For example:
\begin{eqnarray*}
\raisebox{-1.95ex}{\scalebox{0.19}{\includegraphics{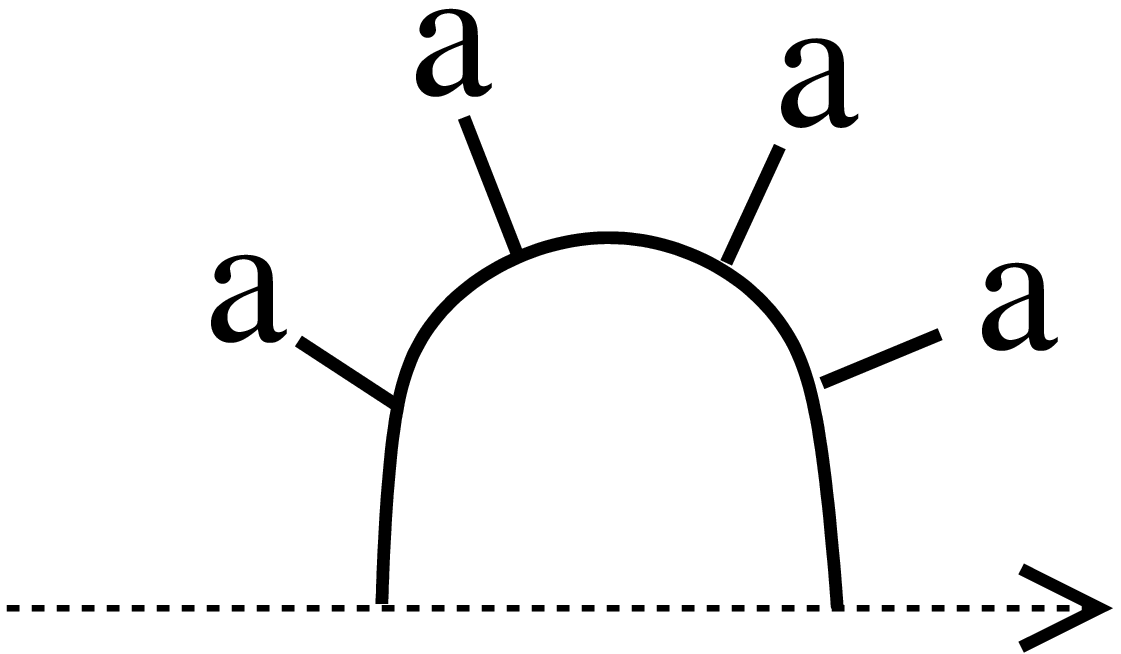}}}
& = &
(+1)\raisebox{-1.95ex}{\scalebox{0.19}{\includegraphics{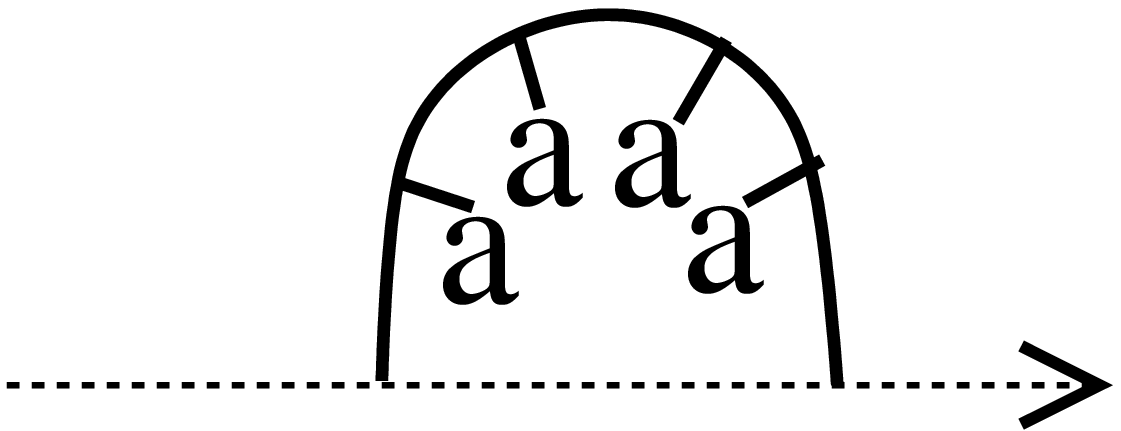}}}
\ =\
(-1)\raisebox{-1.95ex}{\scalebox{0.19}{\includegraphics{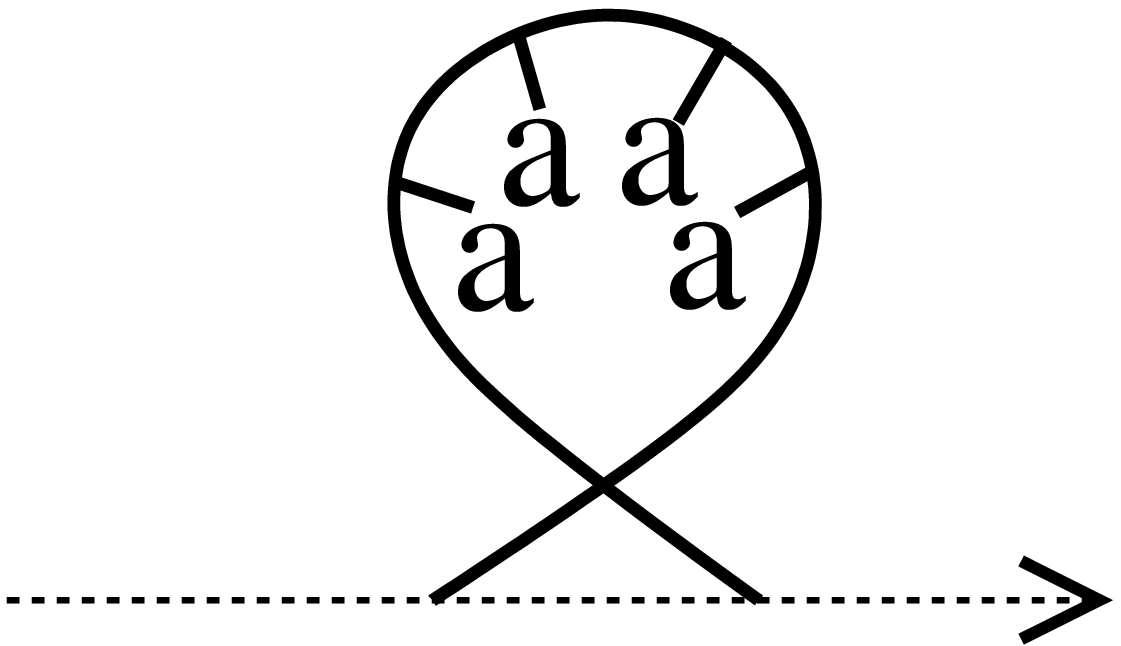}}}
\\[0.15cm]
& = &
(-1)\raisebox{-1.95ex}{\scalebox{0.19}{\includegraphics{evenvanishA}}}\
\ .
\end{eqnarray*}
Thus we may write: \[ C_{||} =
\raisebox{-1.95ex}{\scalebox{0.17}{\includegraphics{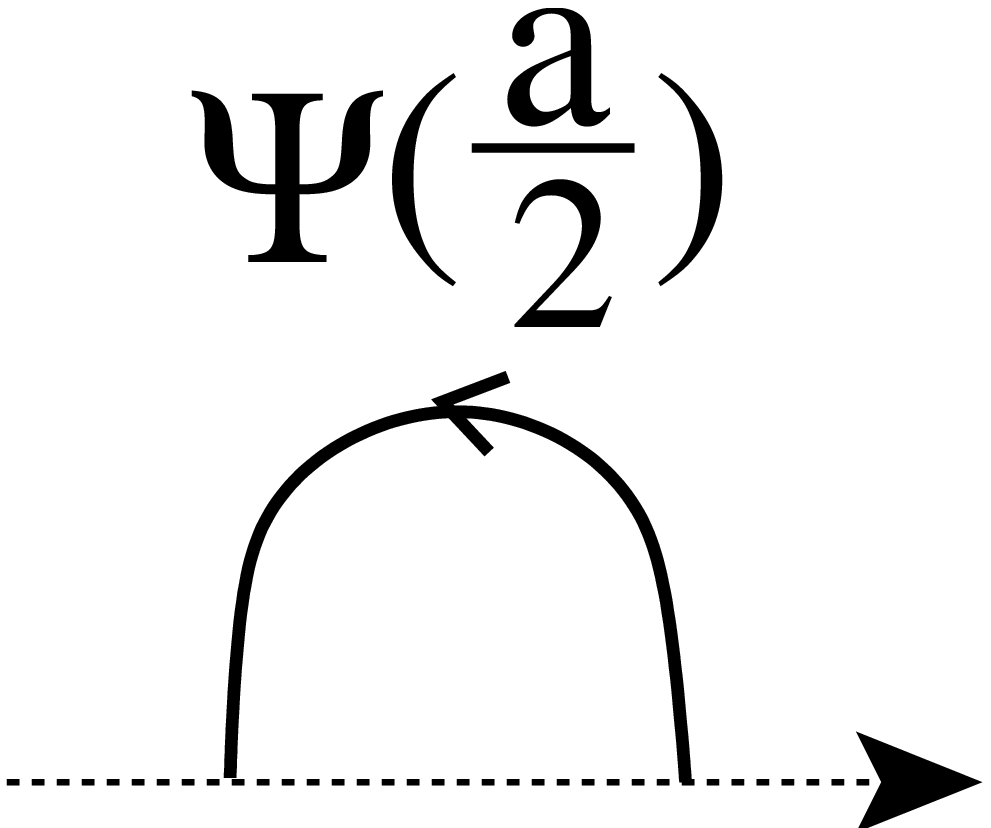}}}\ ,
\]
where $\Psi(x)$ is the formal power series defined by $\Psi(x) =
\sum_{n=1}^{\infty} \frac{\psi(n)}{n!}\,x^n$, with
\[
\psi(n) = \left\{
\begin{array}{ll}
1 & \mbox{if $n=1$,} \\[0.1cm]
2\sum_{w\in\Gamma_n}(-1)^{d(w)}& \mbox{if $n>1$ and $n$ is odd,}
\\[0.1cm]
0 & \mbox{if $n>1$ and $n$ is even.}
\end{array}
\right.
\]
The required computation, Equation \ref{iicomp}, is completed by
the following proposition.
\begin{prop}\label{psicalc}
\[
\Psi(x) = \tanh(x).
\]
\end{prop}
\begin{proof}
We'll begin by replacing $\psi(n)$ with a function that is easier
to use.
 For every $n\geq 1$ let
$\Sigma_n$ denote the set of words that can be made using each of
the symbols $\{1,\ldots,n\}$ exactly once (with no restrictions on
order) and define
\[
\phi(n)=\sum_{w\in\Sigma_n}(-1)^{d(w)}.
\]
Let $\nu:\Sigma_n\rightarrow\Sigma_n$ be the involution of
$\Sigma_n$ which writes a word in its reverse order. Define the
{\it descent} $d(w)$ of a word  in the obvious way. Notice that:
\begin{equation}\label{involanddescent}
d(\nu(w)) = \left\{
\begin{array}{cl}
+d(w) & \mbox{if $n$ is odd,} \\[0.1cm]
-d(w) & \mbox{if $n$ is even.}
\end{array}
\right.
\end{equation}
Thus:
\[ \phi(n) = \left\{
\begin{array}{ll}
1 & \mbox{if $n=1$,} \\[0.1cm]
2\sum_{w\in\Gamma_n}(-1)^{d(w)}& \mbox{if $n>1$ and $n$ is odd,}
\\[0.1cm]
0 & \mbox{if $n>1$ and $n$ is even.}
\end{array}
\right.
\]
In other words, $\phi(n)=\psi(n)$, and our task is to calculate
the power series:
\[
\Psi(x) = \sum_{n=1}^{\infty} \frac{\phi(n)}{n!}x^n.
\]
We'll calculate this power series by writing down a recursion
relation which determines the function $\frac{\phi(n)}{n!}$, and
then we'll identify a power series whose coefficients solve the
recursion relation.

To deduce the appropriate recursion relation we'll partition the
set $\Sigma_n$ according to the position of the symbol $n$. Let
$\Sigma_n^i\subset \Sigma_n$ denote the subset consisting of the
words where the symbol $n$ appears in position $i$. Then, for
$n\geq 3$:
\begin{eqnarray*}
\frac{\phi(n)}{n!} & = & \frac{1}{n!}\sum_{w\in\Sigma_n}(-1)^{d(w)}, \\
& = & \frac{1}{n!}\sum_{i=1}^n\left( \sum_{w\in
\Sigma_n^i}(-1)^{d(w)}
\right), \\
& = & -\frac{1}{n!}\sum_{i=2}^{n-1} \left({n-1 \atop
i-1}\right)\phi(i-1)\phi(n-i), \\
& = &
-\frac{1}{n}\sum_{i=2}^{n-1}\frac{\phi(i-1)}{(i-1)!}\frac{\phi(n-i)}{(n-i)!}.
\end{eqnarray*}
Now observe that this recursion relation, together with the
initial conditions $\phi(1)=1$ and $\phi(2)=0$, completely
determines the sequence $\phi(n)$. It follows from this recursion
relation that $\Psi(x)$ is the unique power series satisfying the
functional equation:
\[
\frac{\mathrm{d}}{\mathrm{d}x}\left[\Psi(x)\right] = 1 -
\Psi(x)^2,
\]
with intial terms
\[
\Psi(x) = x + (\mbox{terms of degree at least $3$}).
\]
Thus, $\Psi(x)=\tanh(x)$.
\end{proof}
 {\it Proof of Lemma \ref{standardformlemma}.} So consider the
diagram arising from the pairing corresponding to some word $w\in
\overrightarrow{\Gamma}_n$. Draw the diagram of the pairing
canonically. That is, start with:
\[
\raisebox{-4ex}{\scalebox{0.2}{\includegraphics{lemmaproofZ}}}\,,
\]
and, when doing the pairing, introduce only transversal
double-point intersections each lying above the orienting line. If
you do this then note that the sign of the resulting term is
precisely $(-1)$ raised to the number of intersections displayed
in drawing of the diagram. For example:
\[
\gamma_{\overleftarrow{2}\overleftarrow{1}\overleftarrow{5}\overrightarrow{3}\overleftarrow{4}}
= \frac{(-1)^4}{5!}\left(\frac{1}{2}\right)^4
\raisebox{-3.5ex}{\scalebox{0.2}{\includegraphics{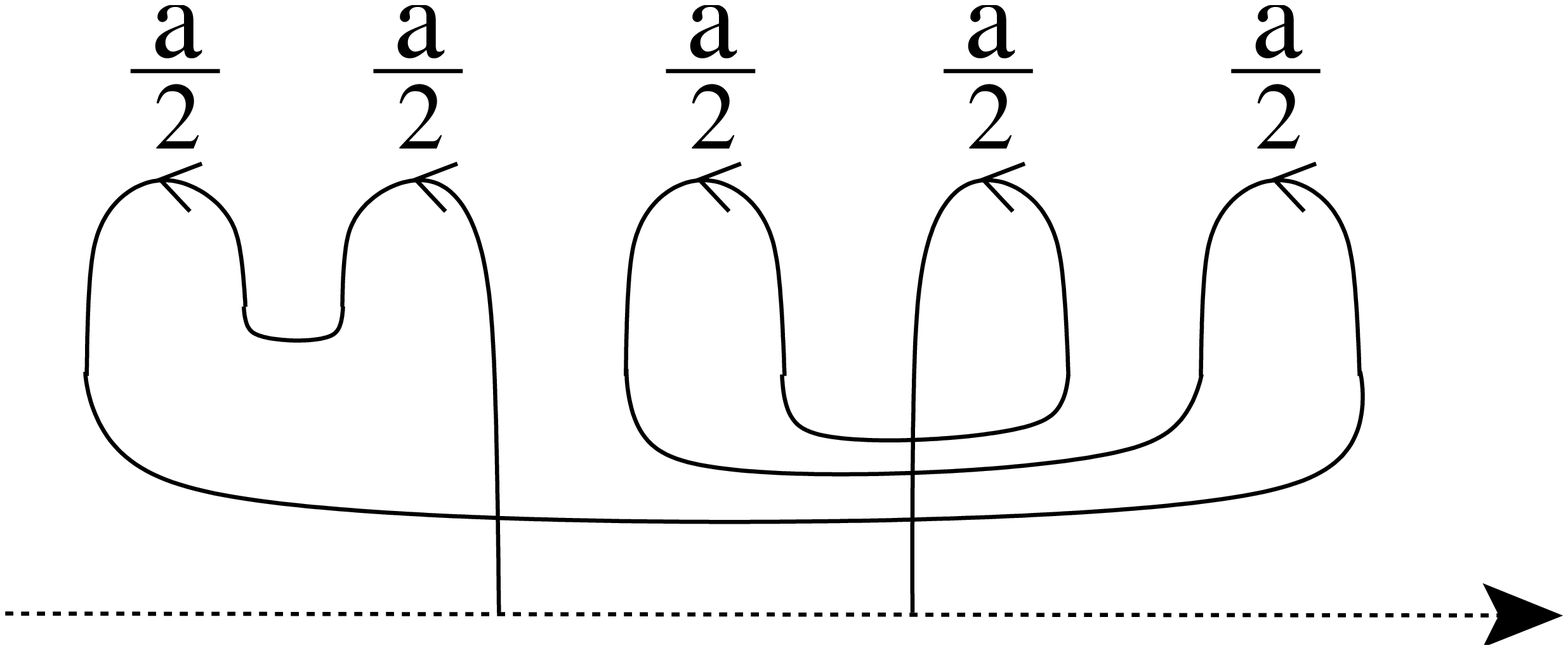}}}.
\\[0.15cm]
\]
Our problem is to work out what further signs must be introduced
to make all the $a$-legs lie on the same side of the edge, and to
write the final sign as a function of $w$.

We'll put diagrams into a standard form with two steps. The first
step will be to add permutations of the following form to the top
of the drawing: \[
\raisebox{-3.5ex}{\scalebox{0.2}{\includegraphics{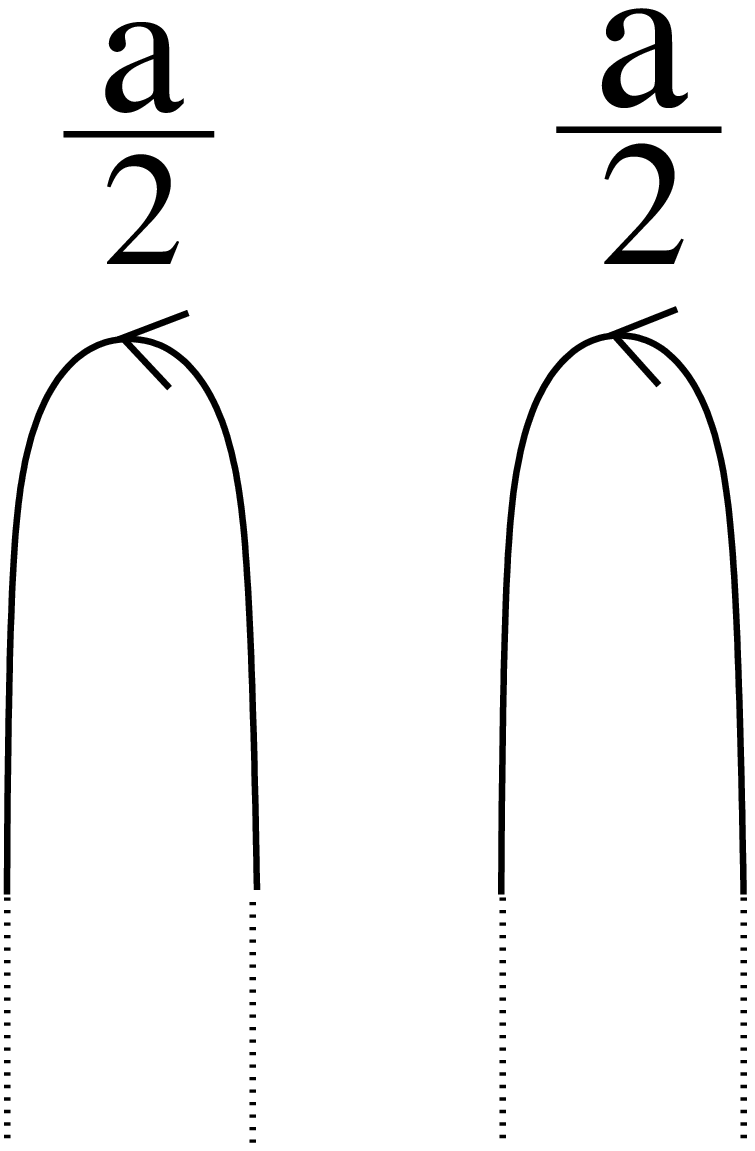}}}\ \
=\ \
\raisebox{-3.5ex}{\scalebox{0.2}{\includegraphics{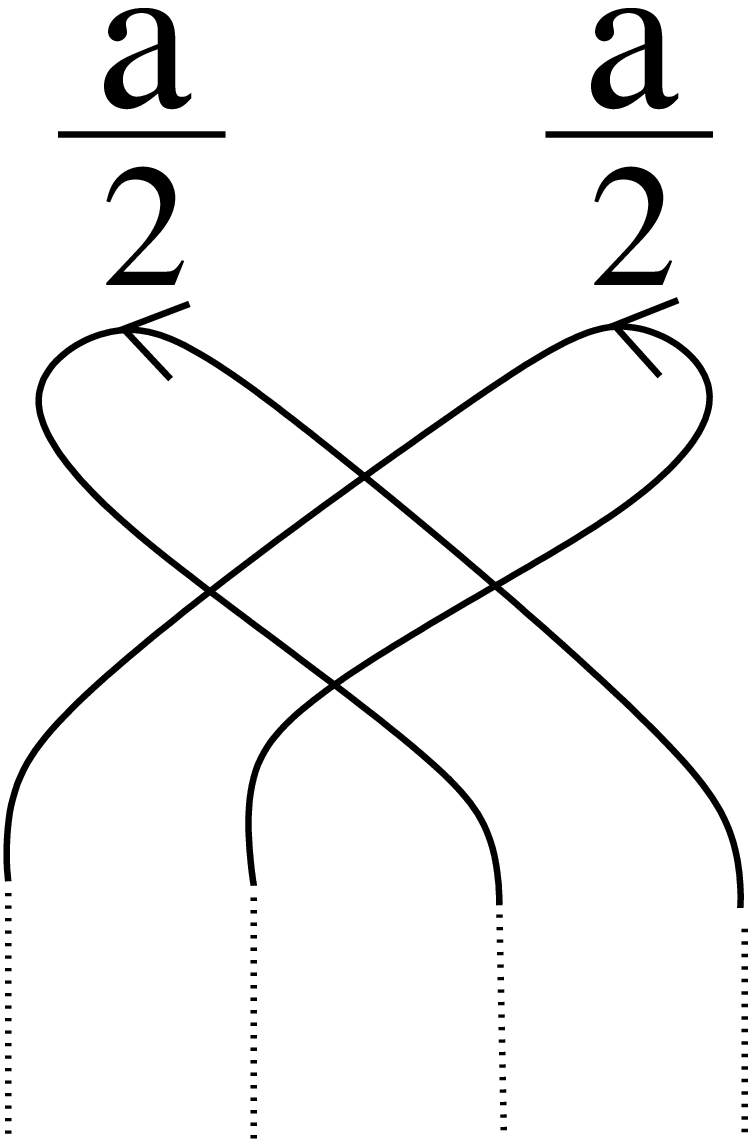}}}\ \ .
\]
Add permutations so that the tops of the blocks appear in the same
order as they appear in the word $w$. Notice that a single such
permutation introduces $4$ intersections into the diagram, so it
is still true, after such a move, that the sign of the term is
precisely $(-1)$ raised to the number of intersections displayed
in the drawing of the diagram. Continuing with our example:
\[
\gamma_{\overleftarrow{2}\overleftarrow{1}\overleftarrow{5}\overrightarrow{3}\overleftarrow{4}}
= \frac{(-1)^{16}}{5!}\left(\frac{1}{2}\right)^4
\raisebox{-5.5ex}{\scalebox{0.2}{\includegraphics{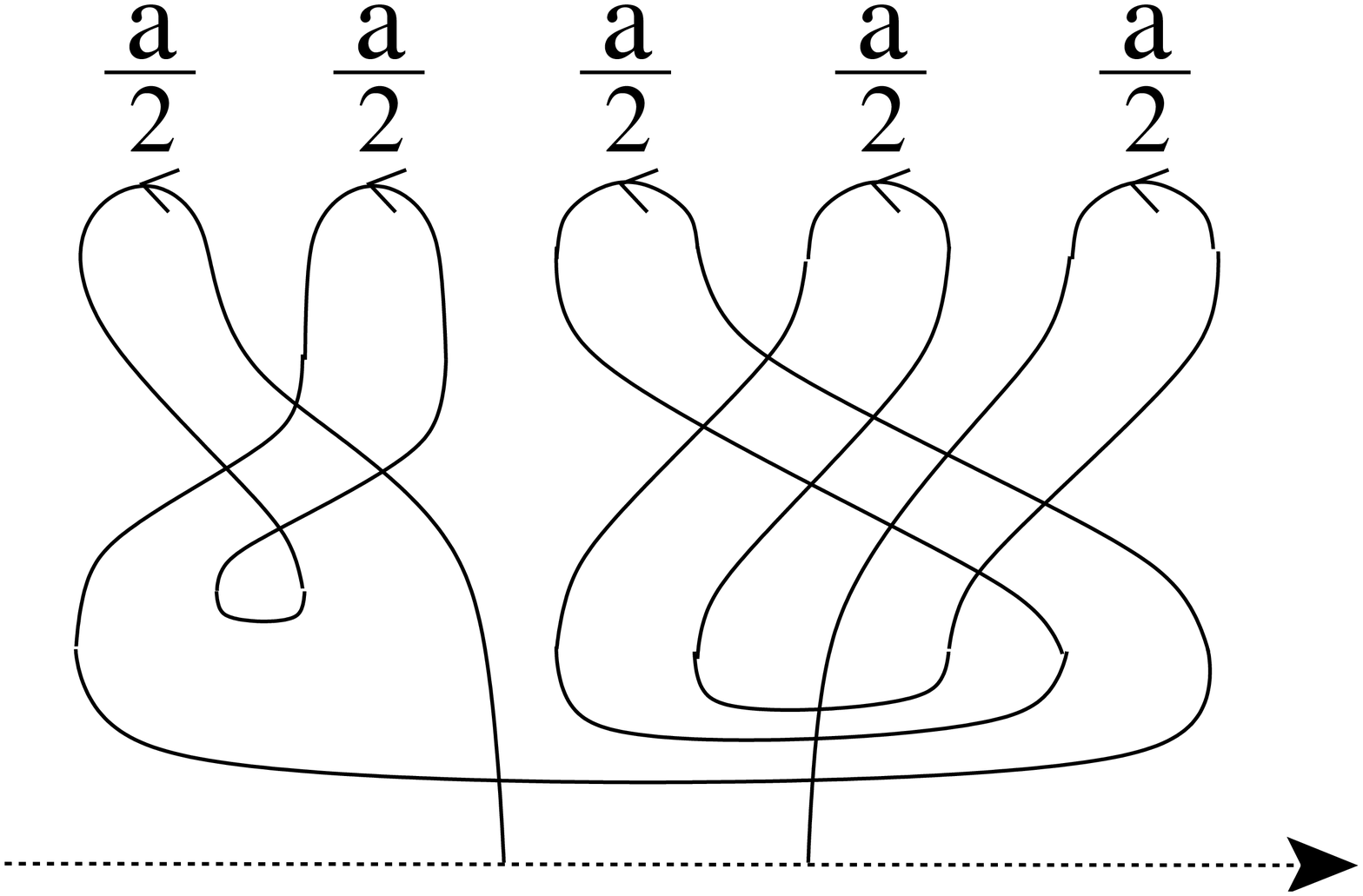}}}\ .
\\[0.15cm]
\]
Notice that as you traverse the edge from the base of the left leg
to the base of the right leg, then some factors are traversed from
left to right (the 4th factor from the left, above), while some
factors are traversed from right to left (the other factors). The
second and final step in the procedure to put the diagram into
standard form is to add twists of the following form:
\[
\raisebox{-3.5ex}{\scalebox{0.2}{\includegraphics{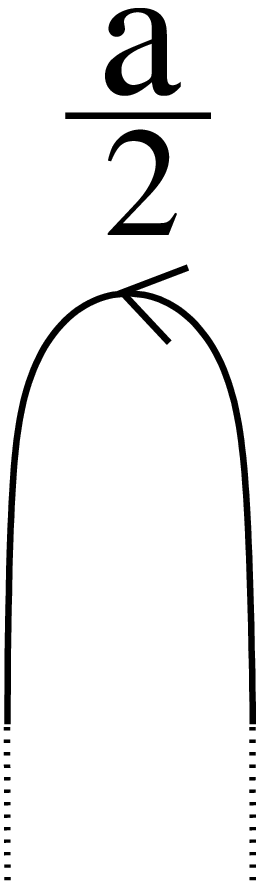}}}\ \
=\ \ (-1)\
\raisebox{-3.5ex}{\scalebox{0.2}{\includegraphics{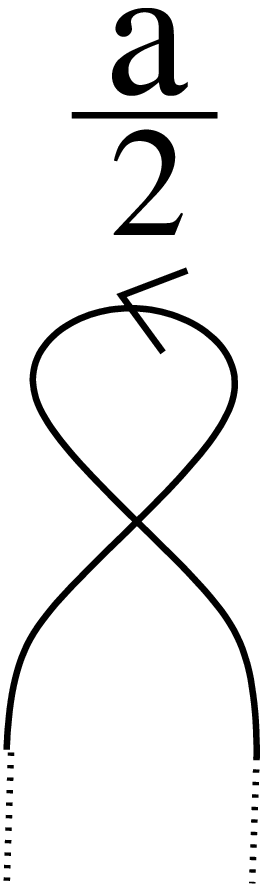}}}
\]
in order that every factor is traversed from left to right. These
diagrams differ by a $(-1)$ which arises from an AS relation which
is employed to shift the $a$-leg to the other side of the edge.
Because this move introduces an extra intersection, it is still
true that the sign of the term is just $(-1)$ raised to the number
of intersections in the drawing. Continuing our example: \[
\gamma_{\overleftarrow{2}\overleftarrow{1}\overleftarrow{5}\overrightarrow{3}\overleftarrow{4}}
= \frac{(-1)^{20}}{5!}\left(\frac{1}{2}\right)^4
\raisebox{-5.5ex}{\scalebox{0.2}{\includegraphics{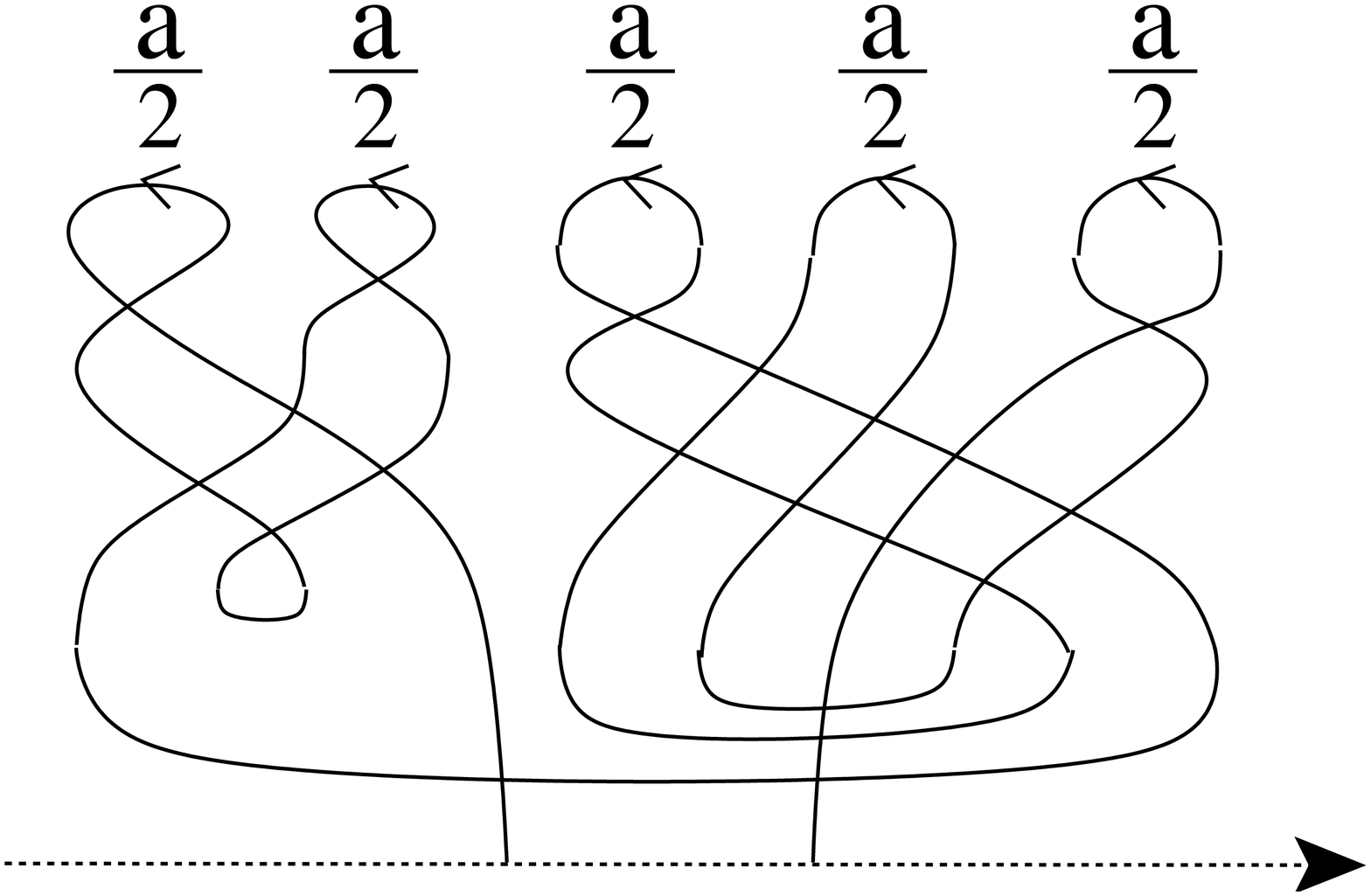}}}\ .
\\[0.15cm]
\]
Notice that after these two steps, the initial diagram has been
transformed into the following standard form, \[
\raisebox{-5.5ex}{\scalebox{0.18}{\includegraphics{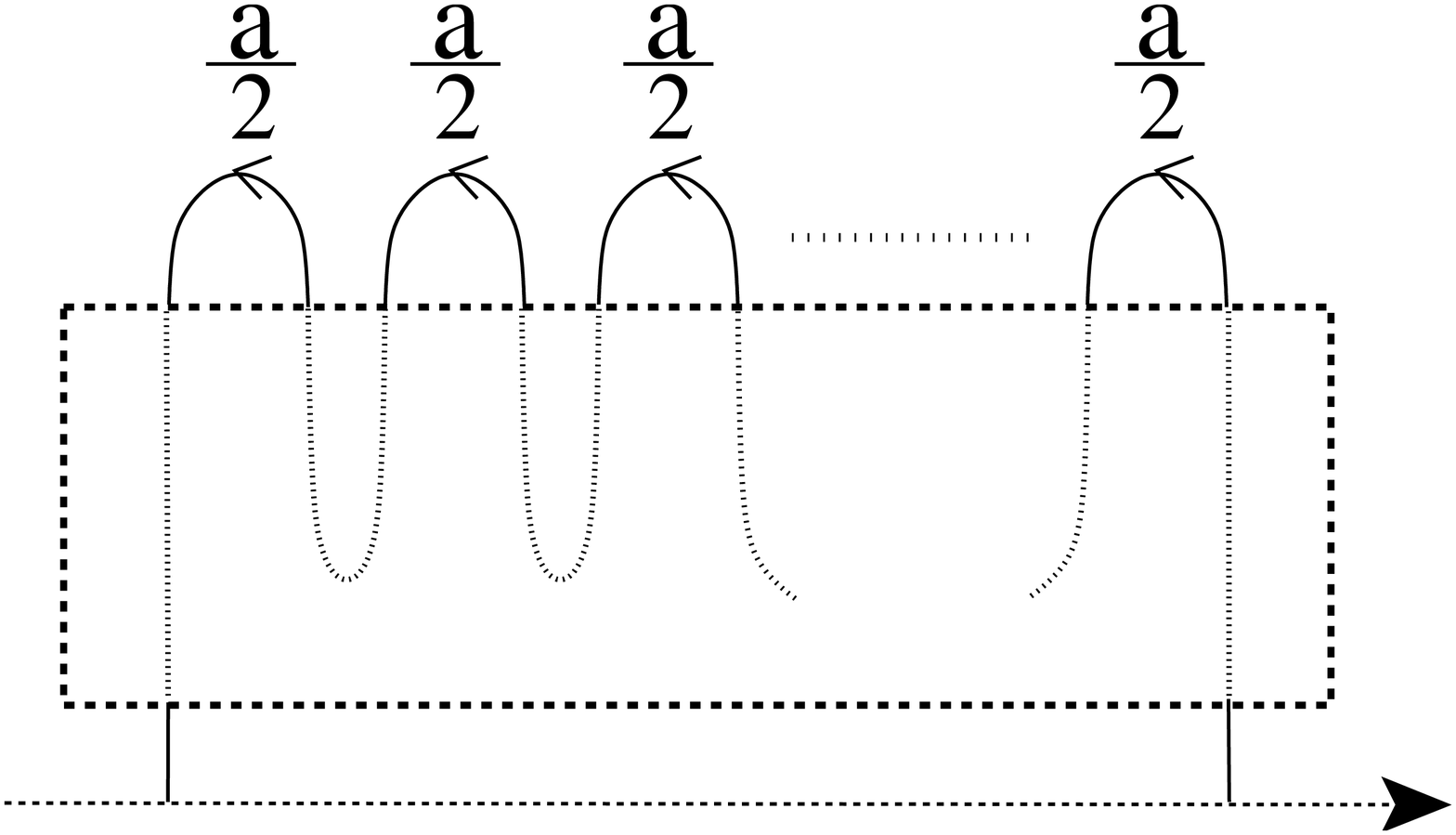}}}\ \ \ ,
\\[0.15cm]
\]
where the parts of the edges within the dashed box follow
complicated, possibly self-intersecting paths.

On account of our procedure,  the sign of the corresponding term
is precisely $(-1)$ raised to the number of intersection points
displayed in the drawing that we have just obtained.  To finish
the calculation, then, it remains for us to count the number of
intersection points displayed by the diagram within the dashed
box.

Notice that the dashed box cuts the edge up into pieces. We'll
call the pieces inside the box the {\it arcs}. They correspond in a direct way to the arcs that were introduced
when this diagram was created in the operation $\lambda$. (Think of those arcs as being coloured
when they are introduced by $\lambda$.)
The number of intersections of
two different arcs must be even for elementary topological reasons. It remains, then, to count the number of
self-intersection points of these arcs. Well, the arcs had no
self-intersections at the beginning (consider the construction $\lambda$), so we we just have to trace
how many self-intersection points were created by our procedure.
Self-intersections may be created in the first step of our
procedure, when a permutation is added to the top of the diagram,
in the following way:
\[
\raisebox{-3.5ex}{\scalebox{0.18}{\includegraphics{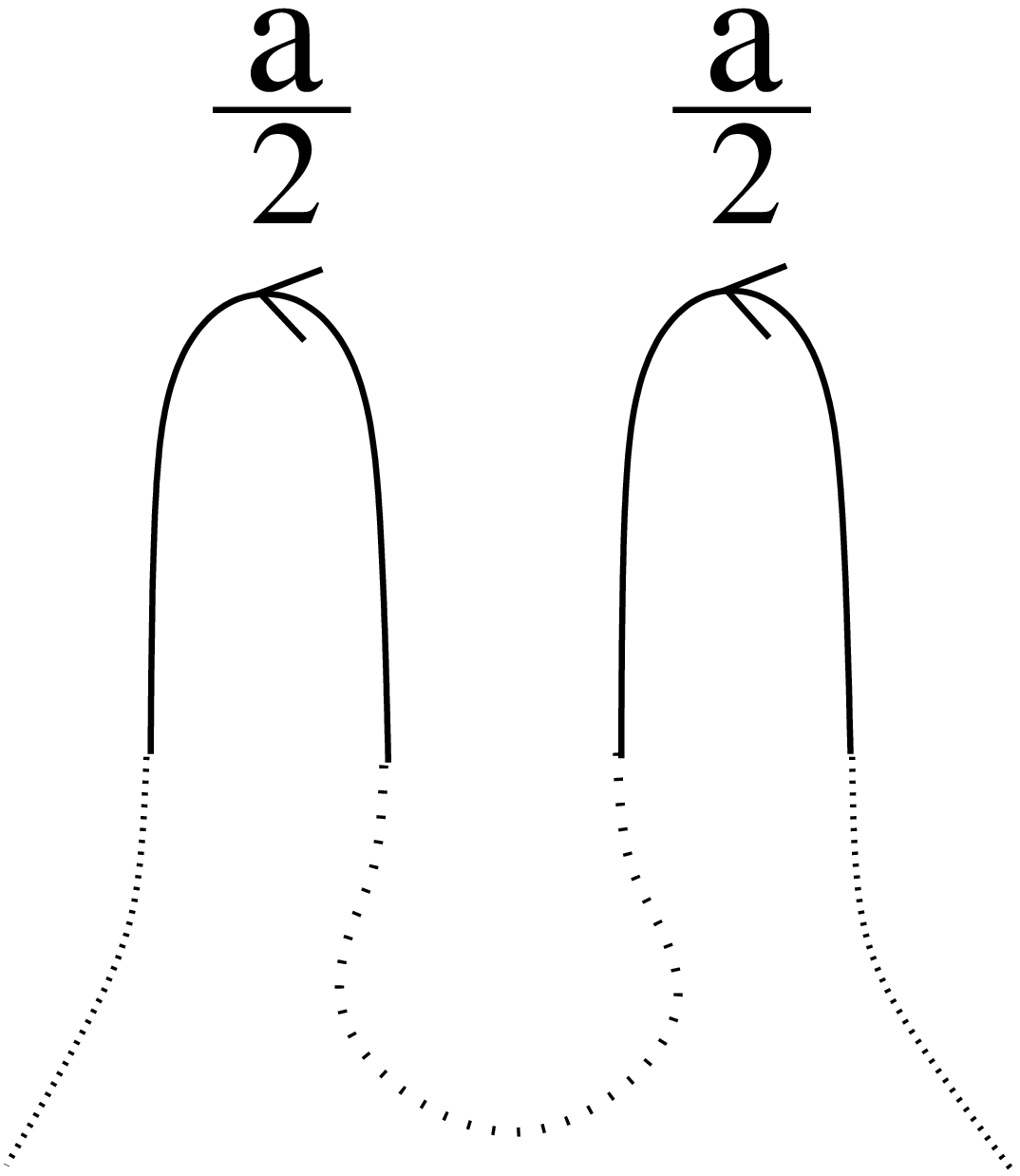}}}\ \
=\ \
\raisebox{-3.5ex}{\scalebox{0.18}{\includegraphics{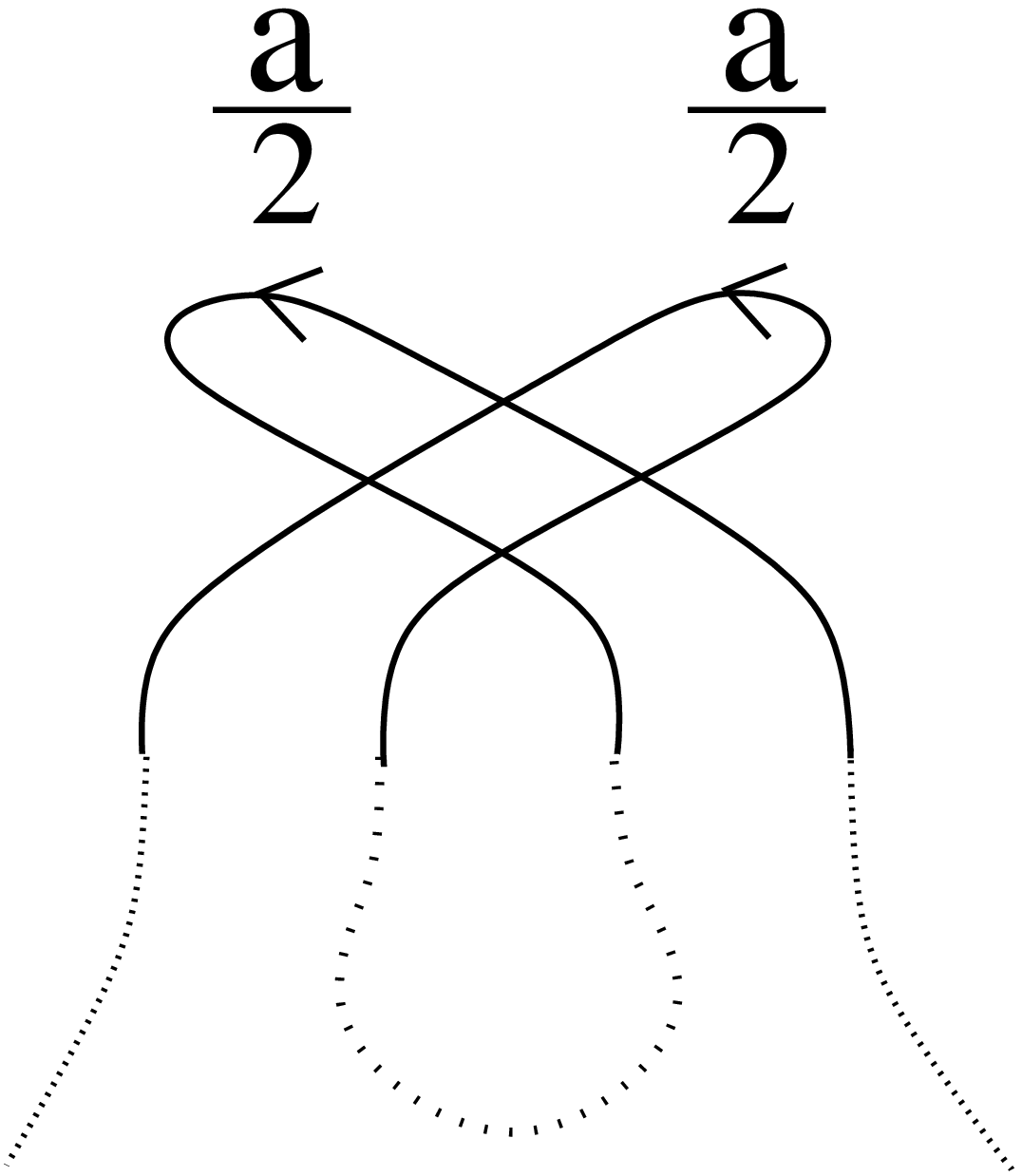}}}\ \ .
\]
There will be one of these for every case of two factors
consecutive in $w$, with the factor appearing later in $w$ having
the smaller value. Thus the sign is $(-1)^{d(w)}$.
\begin{flushright}
$\Box$
\end{flushright}

\subsubsection{The contribution $C_o$.}
The goal of this subsection is the computation that
\[
C_{o} = +\frac{1}{2}
\raisebox{-3.5ex}{\scalebox{0.27}{\includegraphics{ansloop}}}\ .
\]
So consider some $n\geq 2$ (the $n=1$ case we'll observe is zero).
In this subsection we wish to compute the contributions from the
{\it connected} diagrams with {\it zero} legs that we can get by
doing signed pairings of the legs of the following term:
\[ \frac{1}{n!}\
\raisebox{-4ex}{\scalebox{0.24}{\includegraphics{lemmaproofZ}}}\ .
\]
To enumerate these terms we'll employ a certain set
$\overrightarrow{\Xi}_n$. This set consists of the words that can
be made by using each of the symbols $\{1,\ldots,n\}$ precisely
once, such that the left-most symbol is a 1, and where every
symbol $s$ except the initial $1$ is decorated by either an arrow
pointing to the right $\overrightarrow{s}$ or an arrow pointing to
the left $\overleftarrow{s}$. For example:
\begin{eqnarray*}
\overrightarrow{\Xi}_3 & = & \left\{
1\overrightarrow{2}\overrightarrow{3}\ ,\
1\overrightarrow{2}\overleftarrow{3}\ ,\
1\overleftarrow{2}\overrightarrow{3}\ ,\
1\overleftarrow{2}\overleftarrow{3}\ ,\ \right.\\
& &\left.\, \ 1\overrightarrow{3}\overrightarrow{2}\ ,\
1\overrightarrow{3}\overleftarrow{2}\ ,\
1\overleftarrow{3}\overrightarrow{2}\ ,\
1\overleftarrow{3}\overleftarrow{2}\ \right\}.
\end{eqnarray*}
Let $\Xi_n$ denote the set defined in the same way but without the
arrow decorations and let
\[
f_\Xi : \overrightarrow{\Xi}_n \rightarrow \Xi_n
\]
denote the corresponding $2^{n-1}$-to-$1$ forgetful map.

Consider, then, some pairing which uses all the available legs and
results in exactly one connected component. For example: \[
\raisebox{-3.5ex}{\scalebox{0.24}{\includegraphics{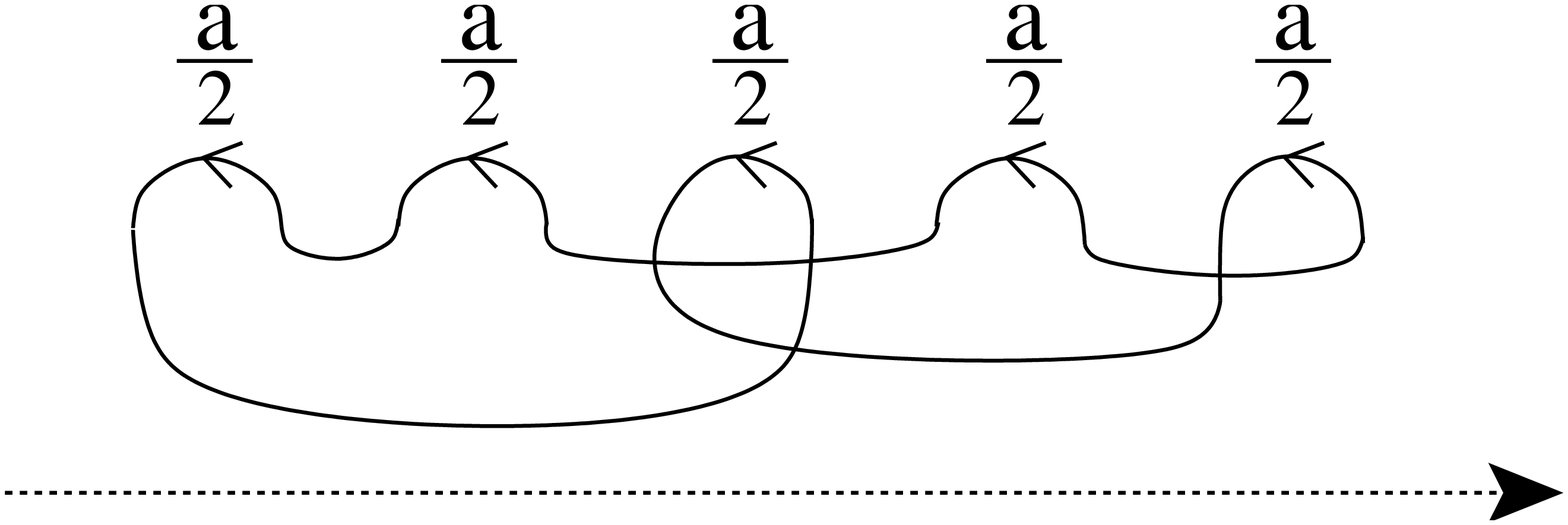}}}\ .\\
\]
The word that corresponds with a gluing is determined in the
following way. To begin, ignore the arc that terminates on the
left-hand side of the left-most block:
\[
\raisebox{-3.5ex}{\scalebox{0.24}{\includegraphics{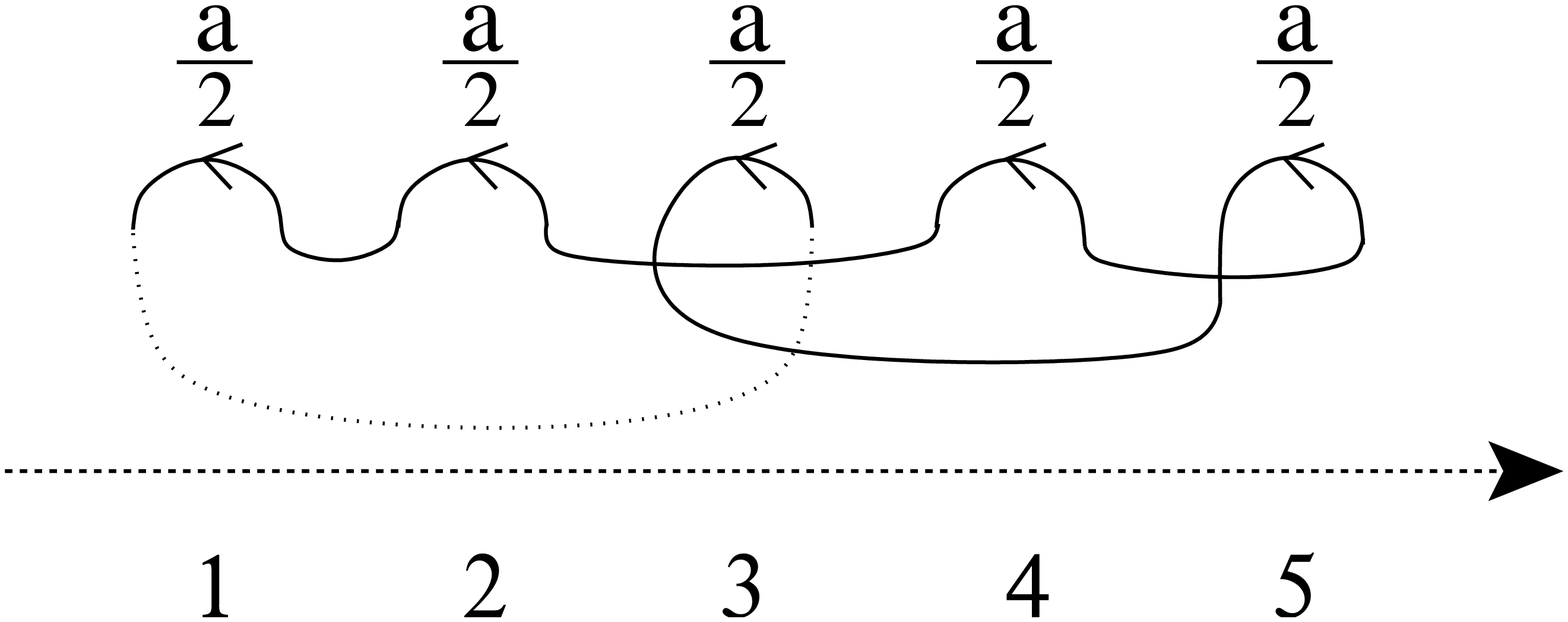}}}\ \ .\\
\]
Now traverse the graph, starting with the left-most block, writing
down the order in which blocks are visited together with the
corresponding directions. The example given leads to:
\[
1\overrightarrow{2}\overrightarrow{4}\overleftarrow{5}\overrightarrow{3}
.
\]
This word contains sufficient information to reconstruct the
pairing, so we have just set up a bijection between the set
$\overrightarrow{\Xi}_n$ and the set of pairings that we are
concerned with presently.

Given some word $w\in\overrightarrow{\Xi}_n$ let $\xi_w$ denote
the corresponding contribution.
\begin{lem} Let $n\geq 2$ and let $w\in\overrightarrow{\Xi}_n$.
Then:
\[
\xi_w =
(-1)^{d(f_\Xi(w))}\frac{1}{2^nn!}\,\raisebox{-2.5ex}{\scalebox{0.2}{\includegraphics{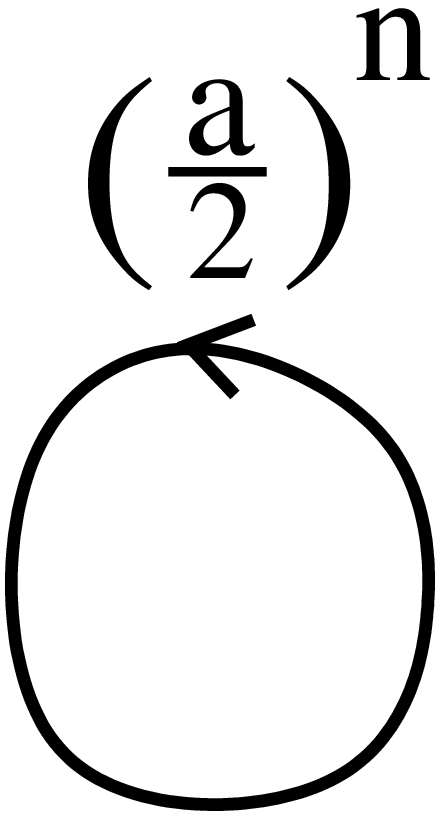}}}\
\ .
\]
\end{lem}
{\it Sketch of the proof.} This proof proceeds in much the same
way as the proof of Lemma \ref{standardformlemma}. That is, we
start by drawing the diagram of the graph that results from the
gluing canonically. Then we put it in standard form in two steps.
In the first step we add permutations to the top of the diagram so
that the factors appear in the diagram in the same order in which
they appear in the word $w$. In the second step we add any twists
that are required in order for the diagram to appear in the
following standard form:
\[
\raisebox{-5.5ex}{\scalebox{0.24}{\includegraphics{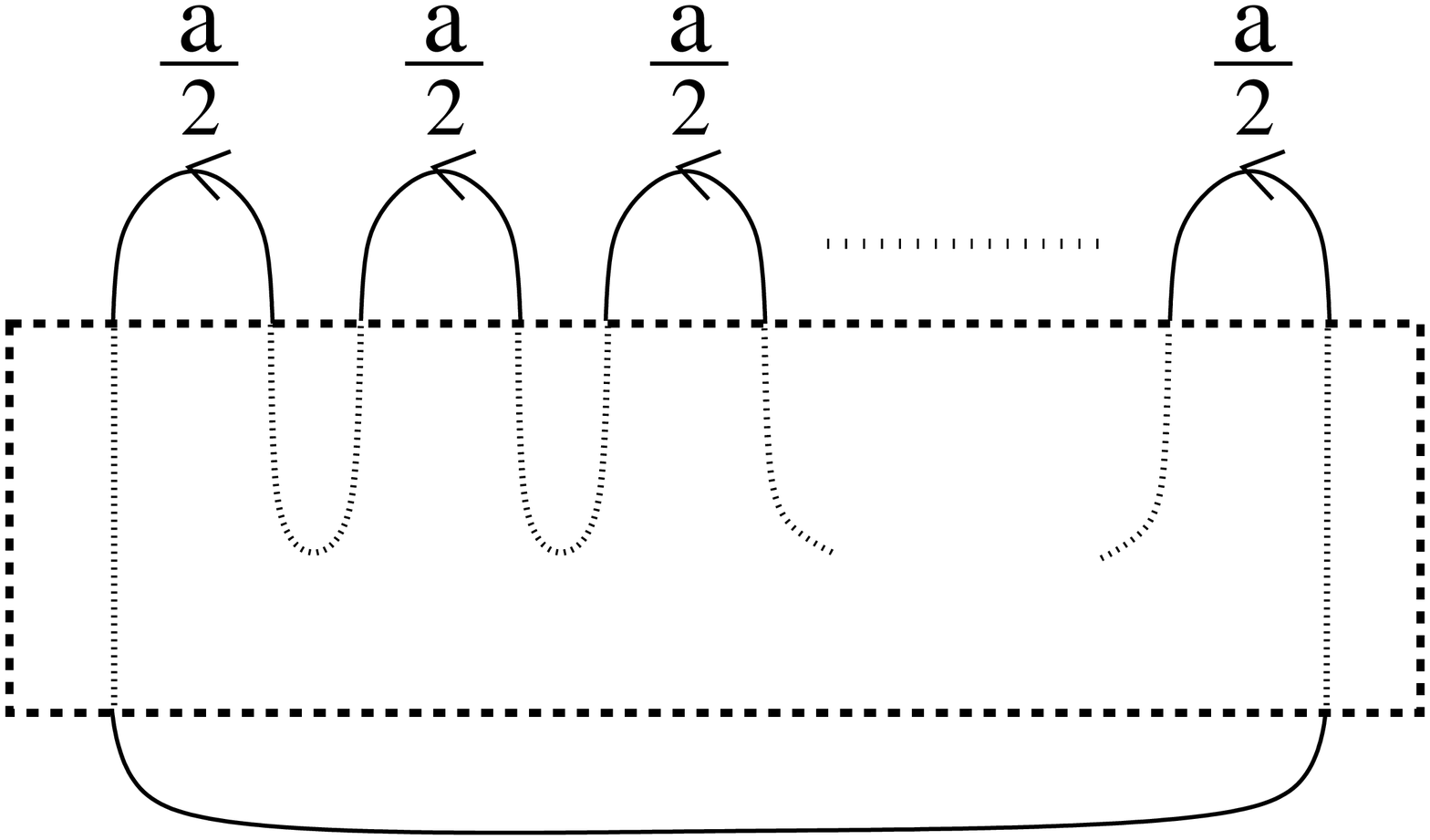}}}\ \ .
\\[0.15cm]
\]
We can keep track of signs in exactly the same way as we did in
the proof of Lemma \ref{standardformlemma}, and we are led to the
given conclusion. Observe that this term has an extra factor of $2$ in the denominator (in comparison with Lemma \ref{standardformlemma}) because
this term has an extra arc attached.
\begin{flushright}
$\Box$
\end{flushright}

Now we can compute the contribution $C_o$. To begin, note that the
term that arises in the $n=1$ case is zero:
\[
\frac{1}{2}\
\raisebox{-2.6ex}{\scalebox{0.2}{\includegraphics{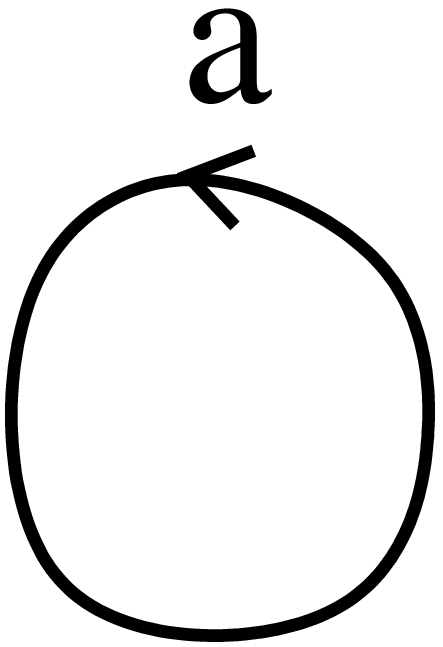}}}\
 =\ 0\ .
\]
Proceeding:
\begin{eqnarray*}
C_o & = &
\sum_{n=2}^{\infty}\sum_{w\in\overrightarrow{\Xi}_n}\xi_w, \\
& = &
\sum_{n=2}^{\infty}\sum_{w\in\overrightarrow{\Xi}_n}\left(\frac{1}{2^nn!}(-1)^{d(f_\Xi(w))}
\raisebox{-2.6ex}{\scalebox{0.2}{\includegraphics{Nloop}}}\,\right),
\\
& = &
\sum_{n=2}^{\infty}\left(\frac{2^{n-1}\sum_{w\in{\Xi}_n}(-1)^{d(w)}}{2^nn!}\right)
\raisebox{-2.6ex}{\scalebox{0.2}{\includegraphics{Nloop}}}\ ,
\\
 & = &
\frac{1}{2}\sum_{n=2}^{\infty}\left(\frac{1}{n}\frac{\phi(n-1)}{(n-1)!}\right)
\raisebox{-2.6ex}{\scalebox{0.2}{\includegraphics{Nloop}}}\ .
\end{eqnarray*}
Thus we have computed that:
\[
C_{o} =\frac{1}{2}\
\raisebox{-2.5ex}{\scalebox{0.2}{\includegraphics{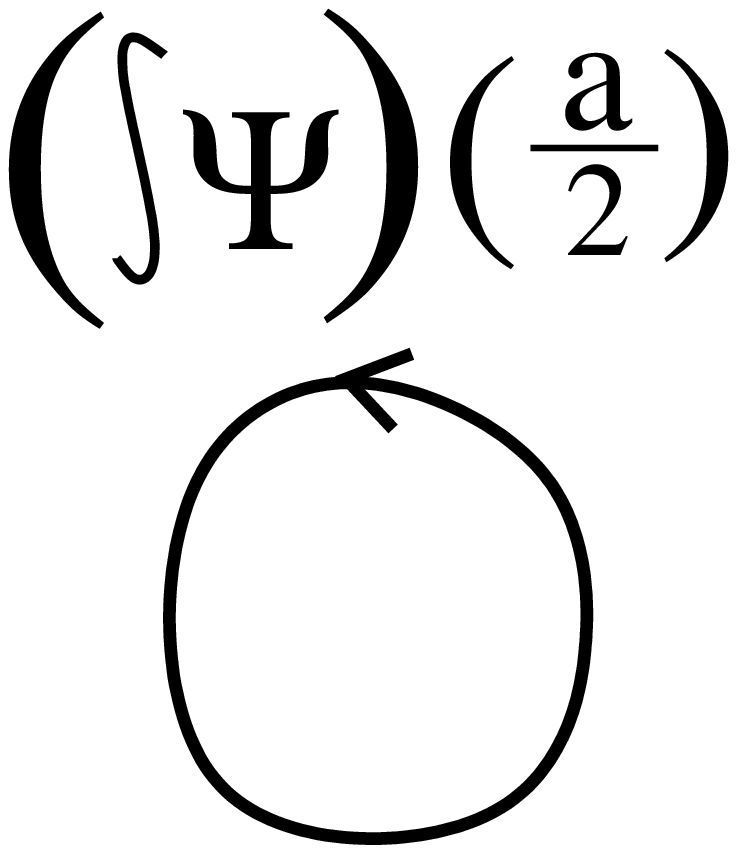}}}\ ,
\]
where $\left(\int\hspace{-0.75ex}\Psi\right)(x)$ denotes
$\sum_{n=2}^{\infty}\frac{1}{n}\frac{\phi(n-1)}{(n-1)!}x^n$, the
unique power series with zero constant term whose formal
derivative is $\Psi(x)$. That power series is given by:
\[
\left(\int\Psi\right)(x) = \ln \cosh x\ .
\]

\

\subsubsection{The contribution $C_{bb}$.} We now wish to compute
the contribution of pairings which lead to connected diagrams with
exactly two $b$-legs. For example:
\[
\raisebox{-3.5ex}{\scalebox{0.24}{\includegraphics{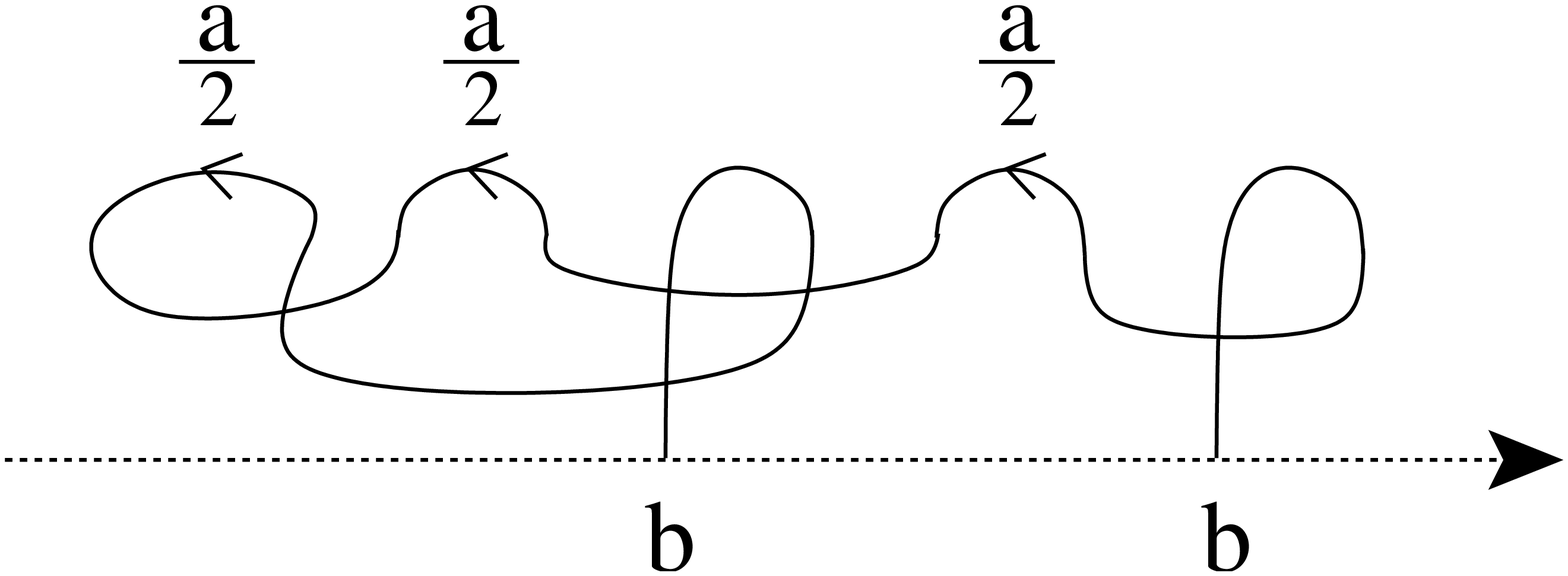}}}\ \ .\\
\]
The set which enumerates these gluings is the following
$\overrightarrow{\Omega}_n$. The elements of
$\overrightarrow{\Omega}_n$ are certain words that use each of the
symbols $\{1,\ldots,n\}$ precisely once and where each symbol
except the first and last symbols in the word is decorated by an
arrow. The set $\overrightarrow{\Omega}_n$ is defined to be all words of
this type with the property that the last symbol is greater than
the first symbol.

To write down the word corresponding to some gluing, traverse the
graph, starting at the left-most of the 2 legs, writing down the
order that blocks are encountered as you traverse (decorating with
the appropriate arrow). For example, the gluing above corresponds
with the word:
\[
3\overleftarrow{1}\overrightarrow{2}\overrightarrow{4}5 \in
\overrightarrow{\Omega}_5\ .
\]
Let
\[
f_\Omega : \overrightarrow{\Omega}_n \rightarrow \Gamma_n
\]
denote the $2^{n-2}$-to-$1$ forgetful map. Let $\omega_w$ denote
the contribution corresponding to some word $w\in
\overrightarrow{\Omega}_n$. We leave the proof of the following
lemma as an exercize for the reader.
\begin{lem}\label{standardformlemmaB}
Let $n\geq 2$ and let $w\in \overrightarrow{\Omega}_n$. Then:
\[
\omega_w = -(-1)^{d(f_\Omega(w))}\frac{1}{2^{n-1}n!} \,
\raisebox{-6ex}{\scalebox{0.25}{\includegraphics{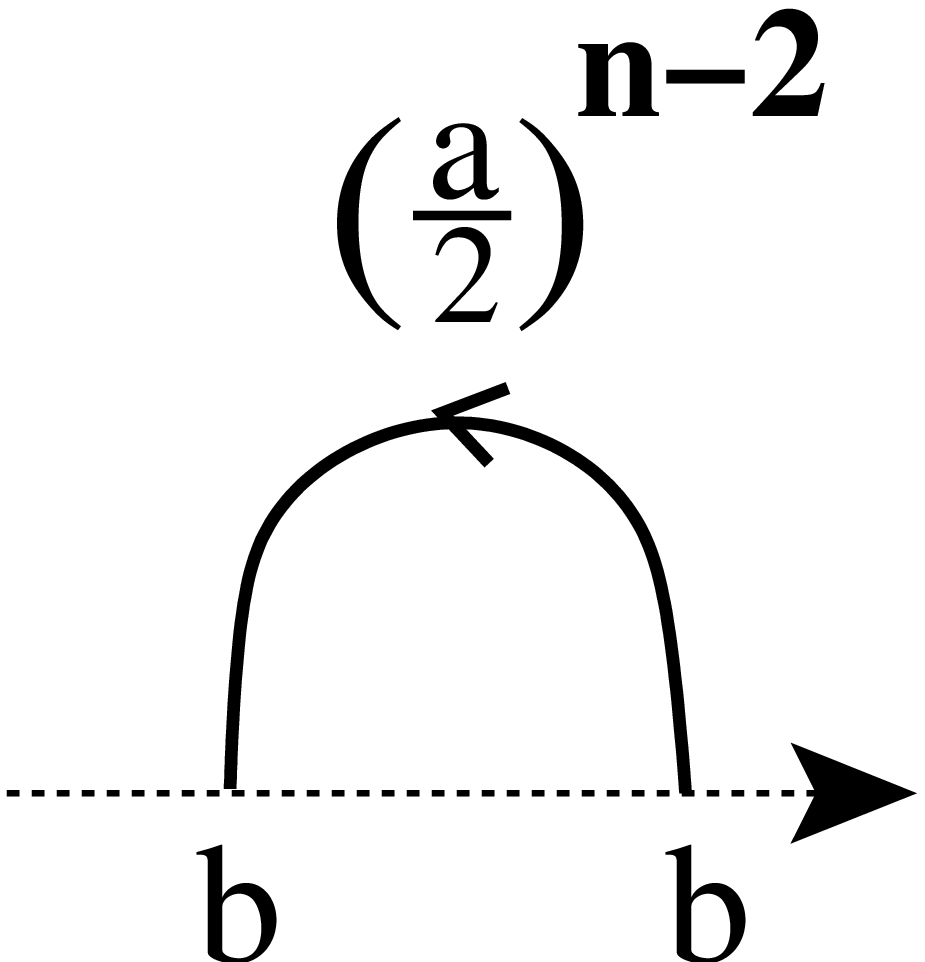}}}\ .
\]
\end{lem}
We can now complete the computation of the contribution $C_{bb}$:
\begin{eqnarray*}
C_{bb} & = & \sum_{n=2}^{\infty}\sum_{w\in \overrightarrow{\Omega}_n}\omega_w\ , \\
& = & -\sum_{n=2}^{\infty}\sum_{w\in \overrightarrow{\Omega}_n}
(-1)^{d(f_\Omega(w))}\frac{1}{2^{n-1}n!} \,
\raisebox{-3.5ex}{\scalebox{0.2}{\includegraphics{NtermBB}}}\ , \\
& = &
-\sum_{n=2}^{\infty}\left(\frac{2^{n-2}}{2^{n-1}n!}\sum_{w\in
{\Gamma}_n}(-1)^{d(w)}\right)
\raisebox{-3.5ex}{\scalebox{0.2}{\includegraphics{NtermBB}}}\ ,\\
& = &
-\left(\frac{1}{2}\right)^2\sum_{n=2}^{\infty}\left(\frac{\psi(n)}{n!}\right)
\raisebox{-3.5ex}{\scalebox{0.2}{\includegraphics{NtermBB}}}\ ,\\[0.15cm]
& = & -\left(\frac{1}{2}\right)^2
\raisebox{-3.5ex}{\scalebox{0.2}{\includegraphics{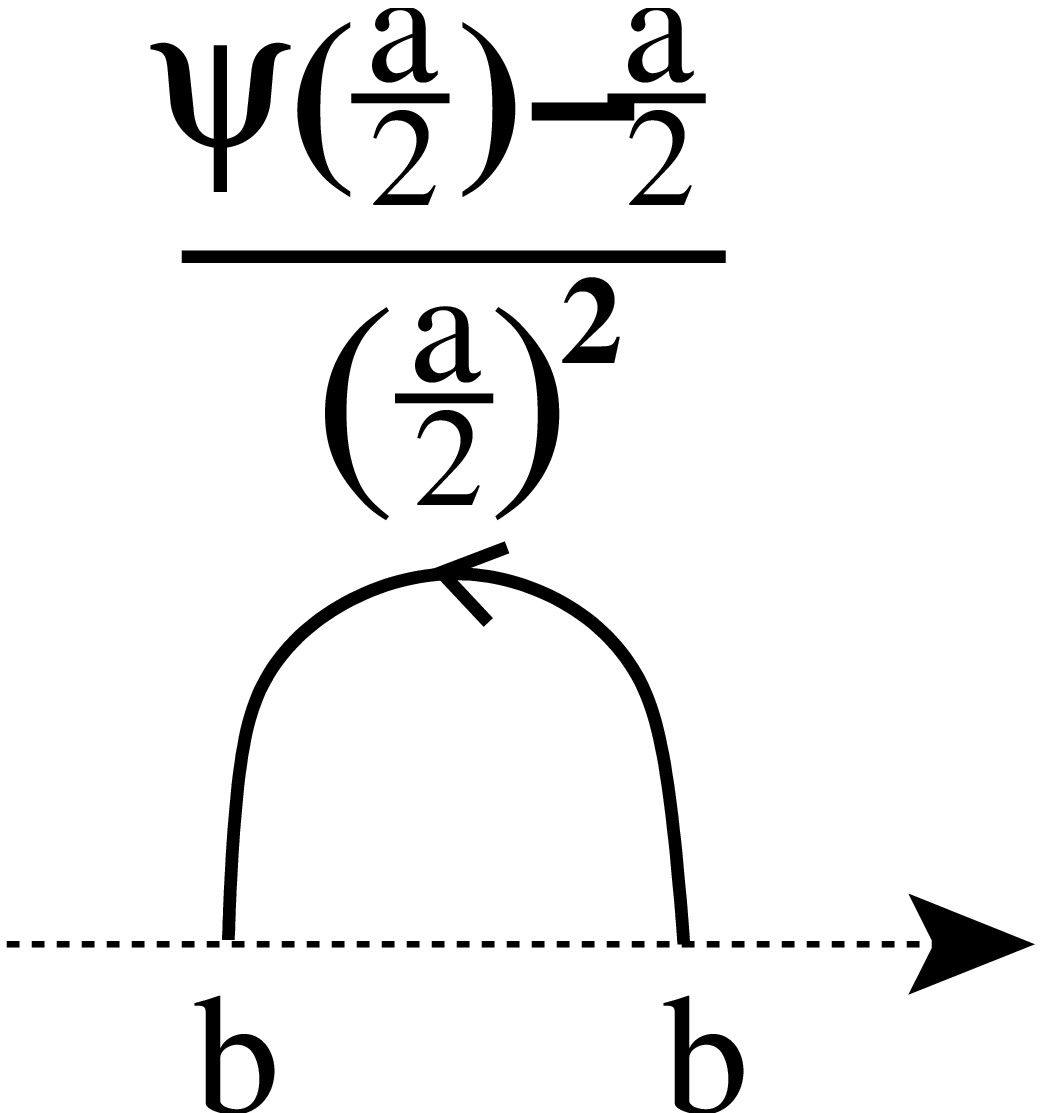}}}\
.
\end{eqnarray*}

\subsubsection{The contribution $C_{|b}$.}

To enumerate these terms we'll employ a set
$\overrightarrow{\Delta}_n$. The elements of this set are words
which use each one of the symbols $\{1,\ldots,n\}$ exactly once,
where every symbol except the last symbol is decorated by an
arrow. The set $\overrightarrow{\Delta}_n$ is defined to be the
set of all such words. For example:
\[
\overleftarrow{4}\overrightarrow{2}\overleftarrow{3}1 \in
\overrightarrow{\Delta}_5.
\]
To write down the word that corresponds with a given pairing,
traverse the resulting diagram, starting at the $\bot$-leg and
proceeding until you reach the $b$-leg. As you traverse, record
the order in which you visit the different blocks and the
directions in which you travel as you visit the blocks. For
example, the pairing
\[
\raisebox{-3.5ex}{\scalebox{0.24}{\includegraphics{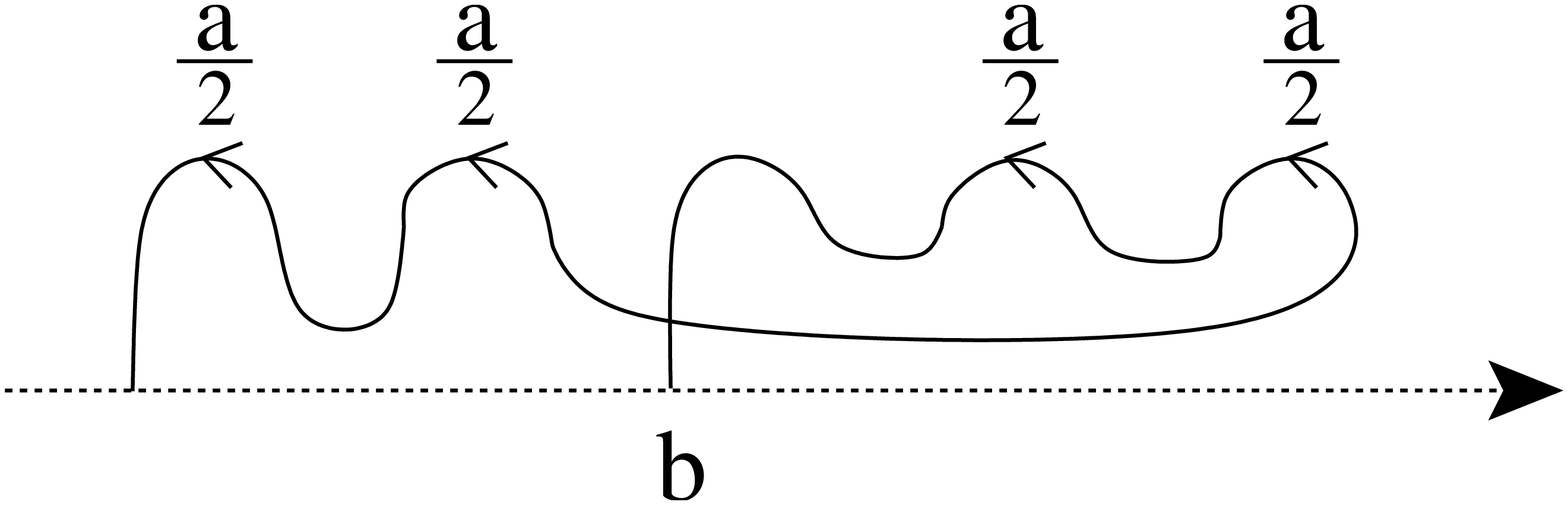}}}\ \\
\]
corresponds with the word
\[
\overrightarrow{1}\overrightarrow{2}\overleftarrow{5}\overleftarrow{4}3\
\in\ \overrightarrow{\Delta}_5\ .
\]
For some $w\in \overrightarrow{\Delta}_n$, let $\delta_w$ denote
the corresponding term. Let
\[
f_\Delta : \overrightarrow{\Delta}_n \rightarrow \Sigma_n
\]
denote the $2^{n-1}$-to-$1$ forgetful map.
\begin{lem}
Let $n\geq 2$ and let $w\in \overrightarrow{\Delta}_n$. Then:
\[
\delta_w = -(-1)^{d(f_\Delta(w))}\frac{1}{2^{n-1}n!} \,
\raisebox{-3.5ex}{\scalebox{0.2}{\includegraphics{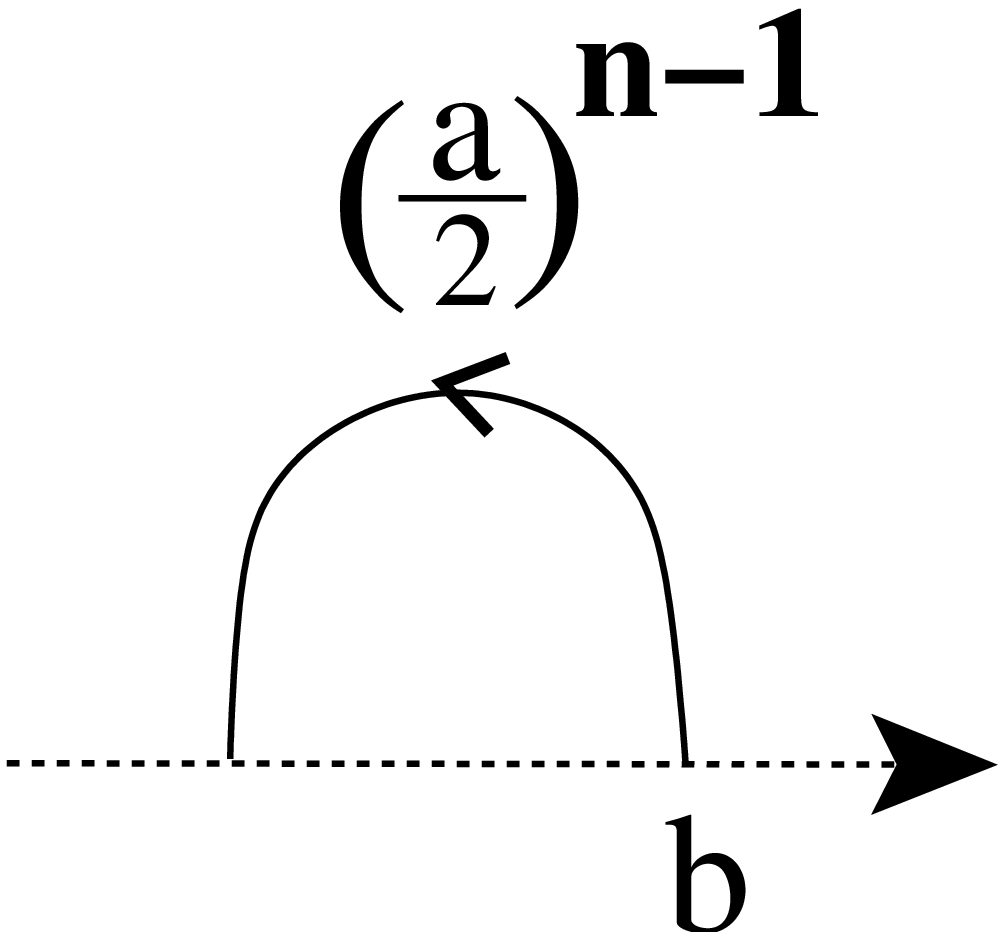}}}\ .
\]
\end{lem}
Thus we can calculate:
\begin{eqnarray*}
C_{|b} & = &
\sum_{n=1}^{\infty}\sum_{w\in\overrightarrow{\Delta}_n}\delta_w\ , \\
& = &
-\sum_{n=1}^{\infty}\sum_{w\in\overrightarrow{\Delta}_n}(-1)^{d(\delta(w))}\frac{1}{2^{n-1}n!}
\raisebox{-3.5ex}{\scalebox{0.18}{\includegraphics{NtermB}}}\ ,
\\[0.1cm]
& = &
-\sum_{n=1}^{\infty}\frac{\sum_{w\in\Sigma_n}(-1)^{d(w)}}{n!}
\raisebox{-3.5ex}{\scalebox{0.18}{\includegraphics{NtermB}}}\ ,
\\[0.2cm]
& = &
-\raisebox{-3.5ex}{\scalebox{0.18}{\includegraphics{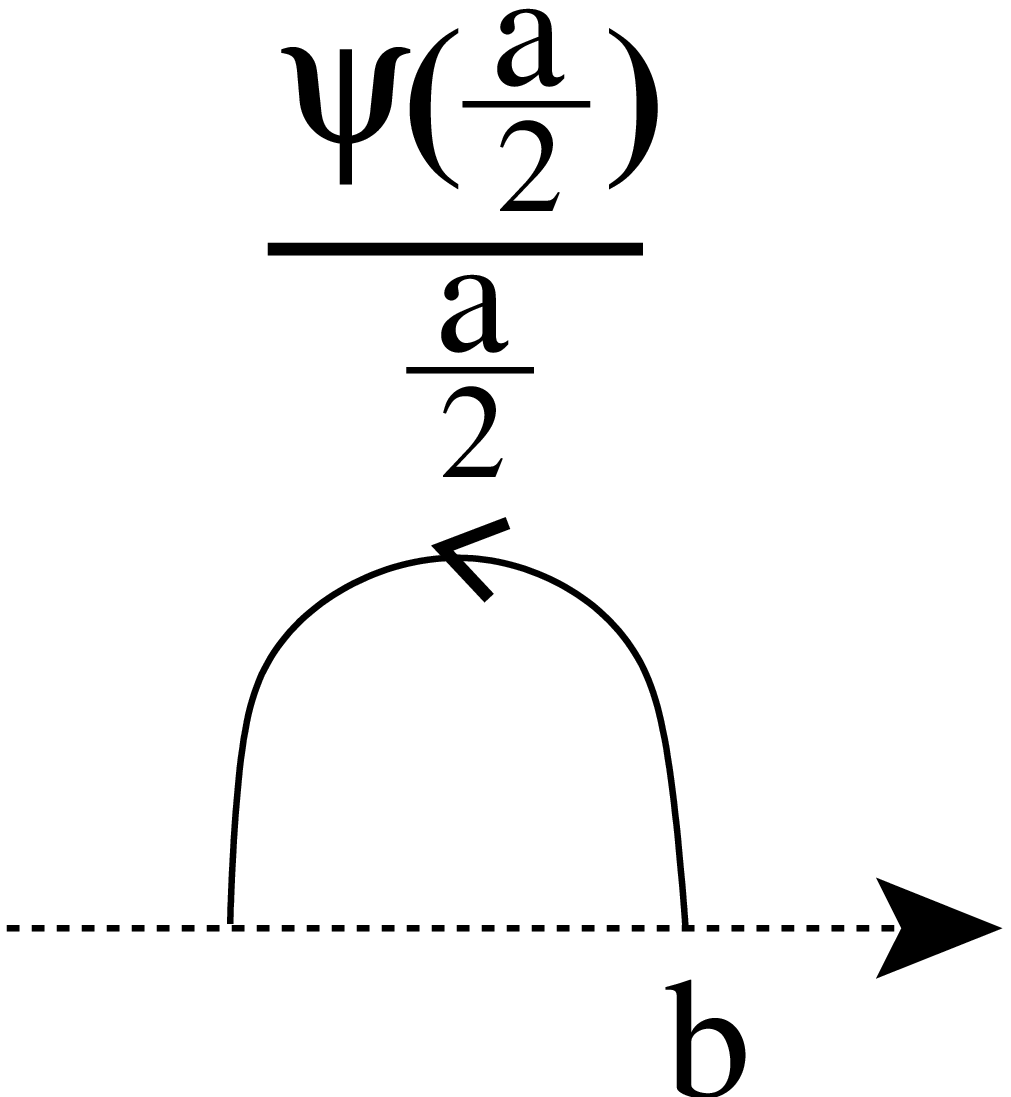}}}\
.
\end{eqnarray*}
\section{Computing the operator product II.}
\label{computingtheoperatorexpressionB} In this section we'll take
the computation of the last section (Theorem
\ref{connectedtheorem}) and use it to compute the series
$\left[\mathcal{X} \right]_{b,\partial_b=0}$ (where $\mathcal{X}$ is the series introduced in Theorem \ref{importantexpression}); this will give us
Theorem \ref{Xcalculate}. A direct substitution of the previous
section's result into $\mathcal{X}$ (taking terms without $b$-legs
out the front of the expression) yields that
$[\mathcal{X}]_{b,\partial_b=0}$ is equal to:
\begin{equation}\label{xcalclast} \exp_{\#}\left(\left(\frac{1}{2}\right)
\raisebox{-3ex}{\scalebox{0.27}{\includegraphics{ansloop}}} +
\raisebox{-3ex}{\scalebox{0.27}{\includegraphics{ansii}}}\right)\,\#\,
\mathcal{Y}
\end{equation}
where $\mathcal{Y}$ is equal to
\[
\left[ \exp_\apply
\left(-\frac{1}{2}\,\raisebox{-4.5ex}{\scalebox{0.25}{\includegraphics{paramzzz}}}
\right)\apply \exp_\#\left(
-\raisebox{-6.9ex}{\scalebox{0.25}{\includegraphics{ansB}}}
-\left(\frac{1}{2}\right)^2
\raisebox{-6.75ex}{\scalebox{0.25}{\includegraphics{ansBB}}}
\right)\right]_{b,\partial_b=0}\hspace{-0.5cm}.\\[0.1cm]
\]
This series $\mathcal{Y}$ will be computed with the following
lemma, whose proof is the subject of this section. Observe that the minus that appears inside the first term has been removed by appropriately 
choosing the orientation on the corresponding edge.
\begin{lem}\label{opcompprop}\label{secondapply}
Let $Y(a)$ be a power series containing only even powers of $a$
and let $Z(a)$ be a power series containing only odd powers of
$a$. Then:
\begin{multline*}
\left[ \mathrm{exp}_\apply \left(
\raisebox{-5.5ex}{\scalebox{0.2}{\includegraphics{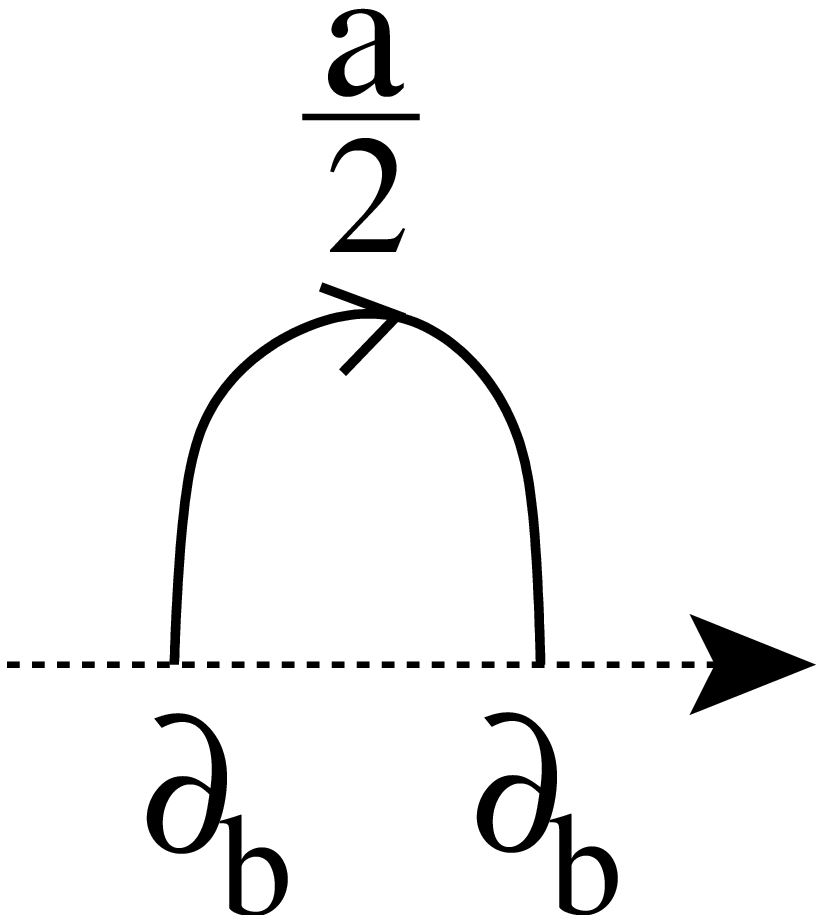}}}
\right)\apply \exp_{\# }\left(
\raisebox{-4.75ex}{\scalebox{0.2}{\includegraphics{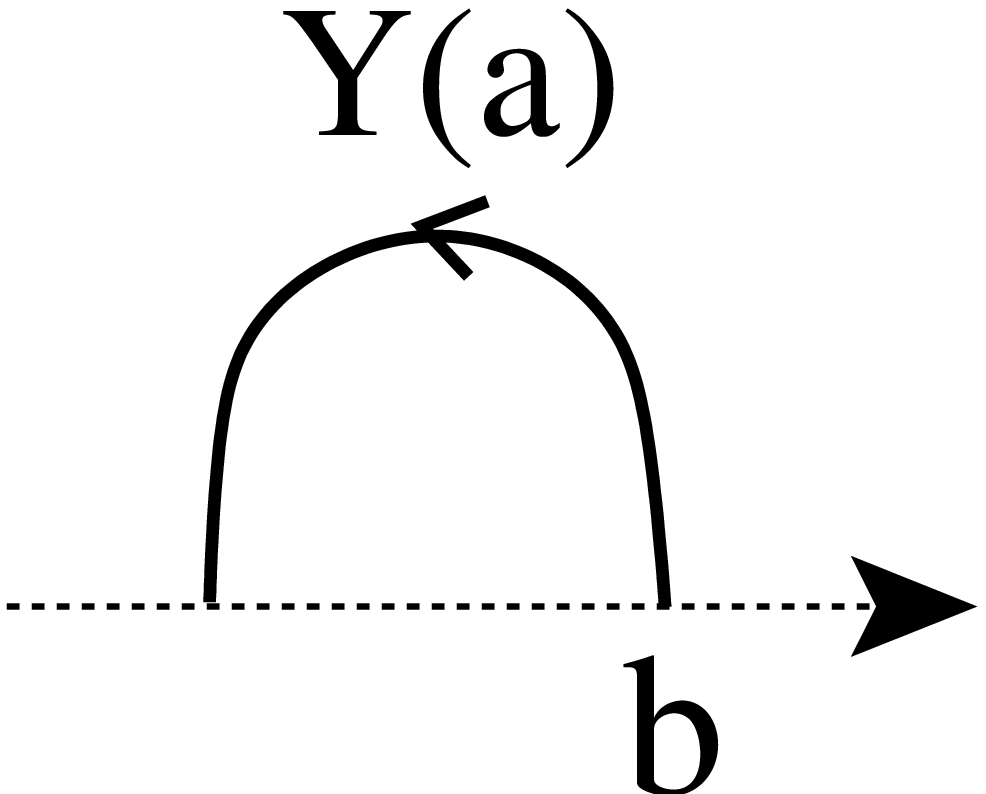}}} +
\raisebox{-4.5ex}{\scalebox{0.2}{\includegraphics{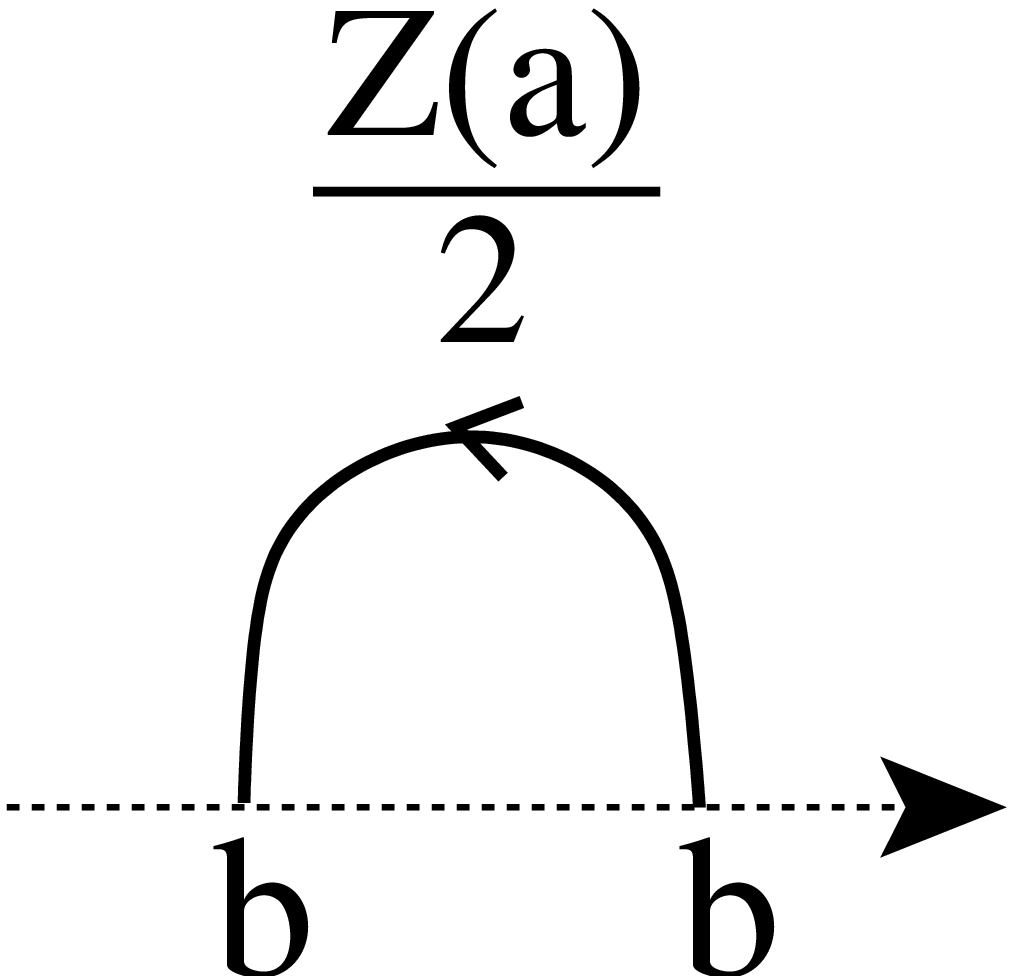}}}
\right)\right]_{{b,\partial_b}=0} \\[0.15cm]  =  \exp_{\#}\left(
-\frac{1}{2}\sum_{n=1}^{\infty}\frac{1}{n}\,
\raisebox{-3.75ex}{\scalebox{0.22}{\includegraphics{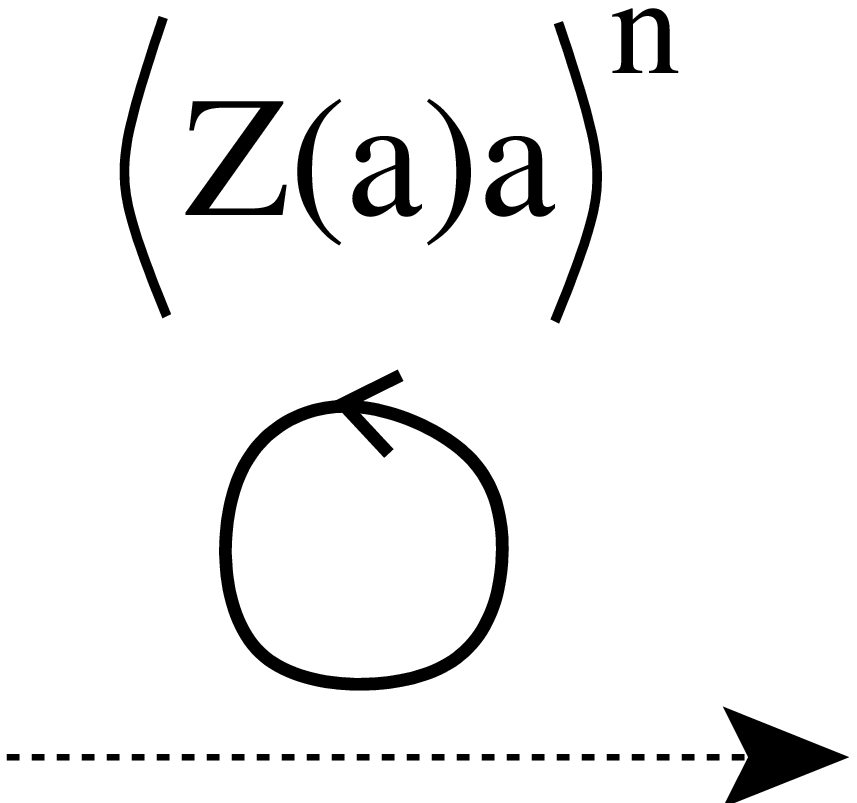}}}\
+\frac{1}{2}\sum_{n=0}^{\infty}\
\raisebox{-3.5ex}{\scalebox{0.22}{\includegraphics{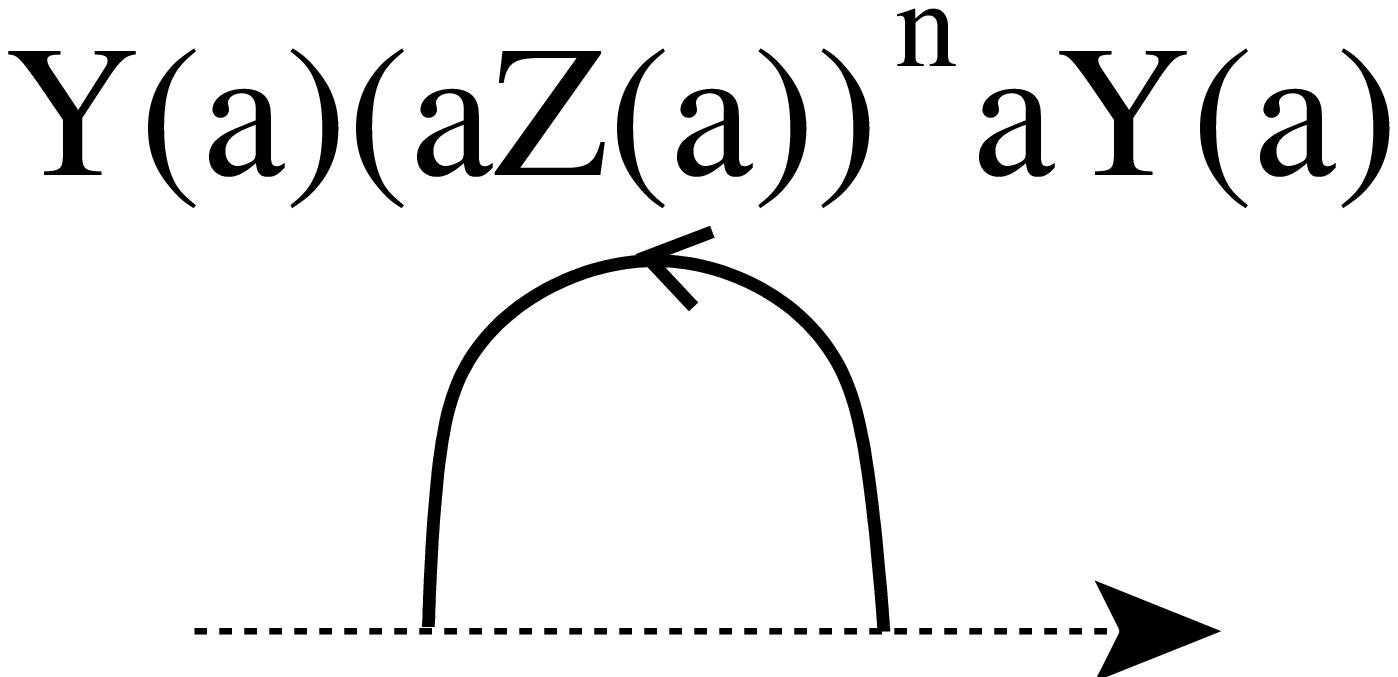}}}\
\right).
\end{multline*}
\end{lem}
To apply this lemma to Equation \ref{xcalclast} we set
$Y(a) = -\frac{\tanh(\frac{a}{2})}{\frac{a}{2}}$ and  $Z(a)=
-\frac{1}{2}\frac{\tanh\left(\frac{a}{2}\right)-\left(\frac{a}{2}\right)}{\left(\frac{a}{2}\right)^2}.$
With these assignments:
\[
\sum_{n=1}^{\infty}\frac{1}{n}\left( Z(a)a \right)^n =
\sum_{n=1}^{\infty}\frac{1}{n}\left(1-\frac{\tanh\left(\frac{a}{2}\right)}{\frac{a}{2}}\right)^n
=
-\ln\left(\frac{\tanh\left(\frac{a}{2}\right)}{\frac{a}{2}}\right).
\]
%
Furthermore:
\begin{eqnarray*}
\frac{1}{2}\sum_{n=0}^{\infty}Y(a)\left(aZ(a)\right)^naY(a) & = &
\sum_{n=0}^{\infty}\frac{\tanh(\frac{a}{2})}{\frac{a}{2}} \left(1
- \frac{\tanh\left(\frac{a}{2}\right)}{\frac{a}{2}}\right)^n
\frac{a}{2}\frac{\tanh(\frac{a}{2})}{\frac{a}{2}}, \\
& = & \frac{\tanh(\frac{a}{2})}{\frac{a}{2}} \frac{1}{1- \left(1 -
\frac{\tanh\left(\frac{a}{2}\right)}{\frac{a}{2}}\right)}
\tanh\left(\frac{a}{2}\right), \\
& = & \tanh\left(\frac{a}{2}\right).
\end{eqnarray*}
With these calculations in hand, a direct application of Lemma
\ref{secondapply} to Equation \ref{xcalclast} yields that
$\mathcal{Y}$ is equal to
\[
\exp_{\#}\left(\left(\frac{1}{2}\right)
\raisebox{-5ex}{\scalebox{0.27}{\includegraphics{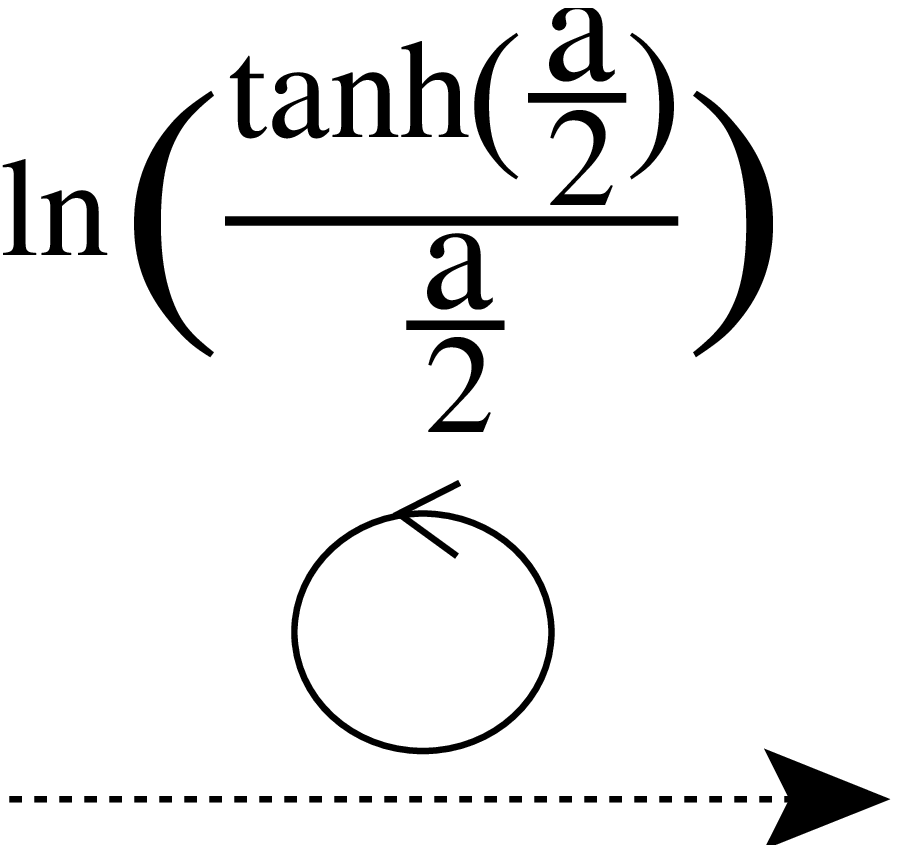}}} -
\raisebox{-5ex}{\scalebox{0.27}{\includegraphics{ansii}}}\right).
\]
Substituting this computation into expression \ref{xcalclast}, we deduce that $\left[\mathcal{X}\right]_{b,\partial_b=0}$ is equal to
\[
\exp_\#\left(\left(\frac{1}{2}\right)\raisebox{-6ex}{\scalebox{0.27}{\includegraphics{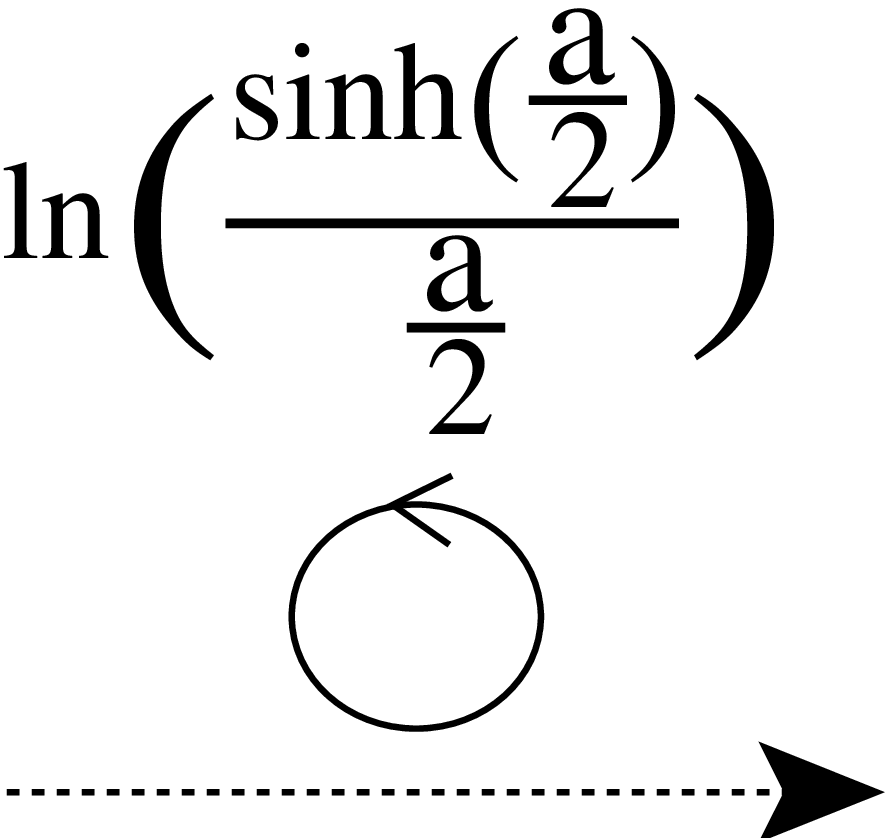}}}\right),
\]
as required. This completes the proof of Theorem \ref{Xcalculate}.

\subsection{Proof of Lemma \ref{secondapply}.}
Our task in this section is the combinatorial computation of the
following expression:
\begin{equation}\label{tobecomputed}
\left[ \mathrm{exp}_\apply \left(
\raisebox{-5ex}{\scalebox{0.175}{\includegraphics{arightZ}}}
\right)\apply \exp_{\# }\left(
\raisebox{-4.25ex}{\scalebox{0.175}{\includegraphics{YTERM}}} +
\raisebox{-4ex}{\scalebox{0.175}{\includegraphics{ZTERM}}}
\right)\right]_{{b,\partial_b}=0}\hspace{-0.5cm}.
\end{equation}
We'll begin with some general observations. First of all note
that, because we set $b$ and $\partial_b$ to zero at the end of
the calculation, there will only be contributions from those terms
arising from the expansion of the exponentials with the property
that the number of $\partial_b$ legs in the first factor is equal
to the number of $b$ legs in the second factor. A typical
contributing term is:
\begin{multline*}
\frac{1}{4!}\,\raisebox{-0.75cm}{\scalebox{0.175}{\includegraphics{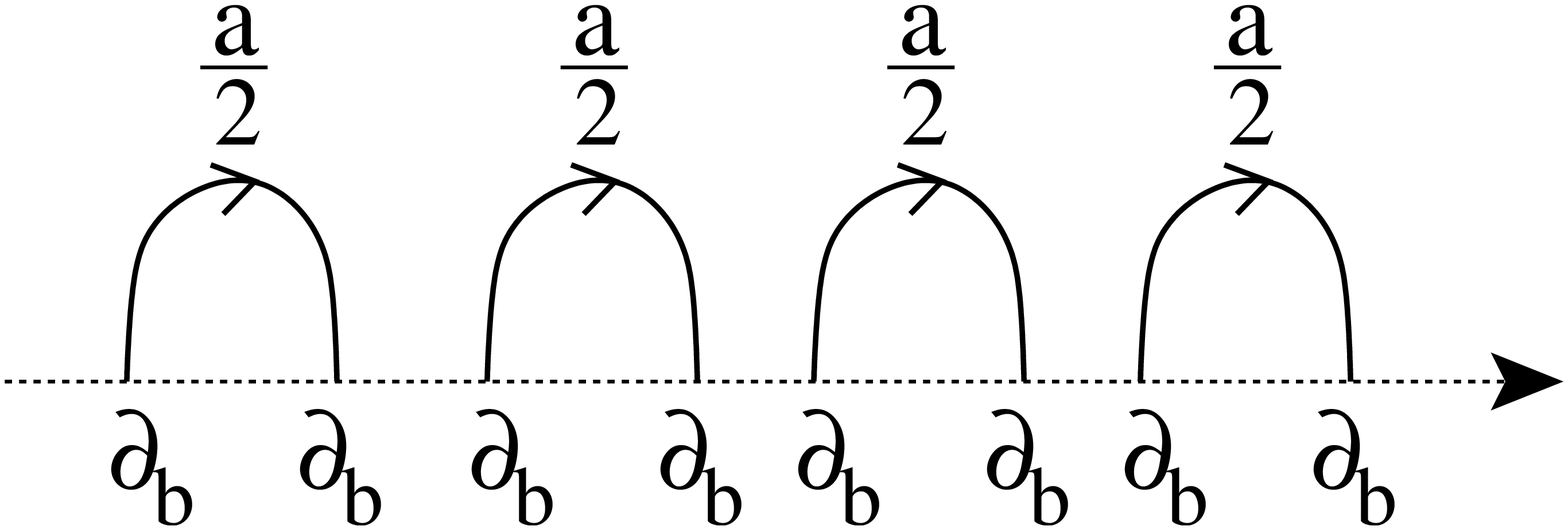}}} \\
\apply
\frac{1}{5!}\,\raisebox{-0.75cm}{\scalebox{0.175}{\includegraphics{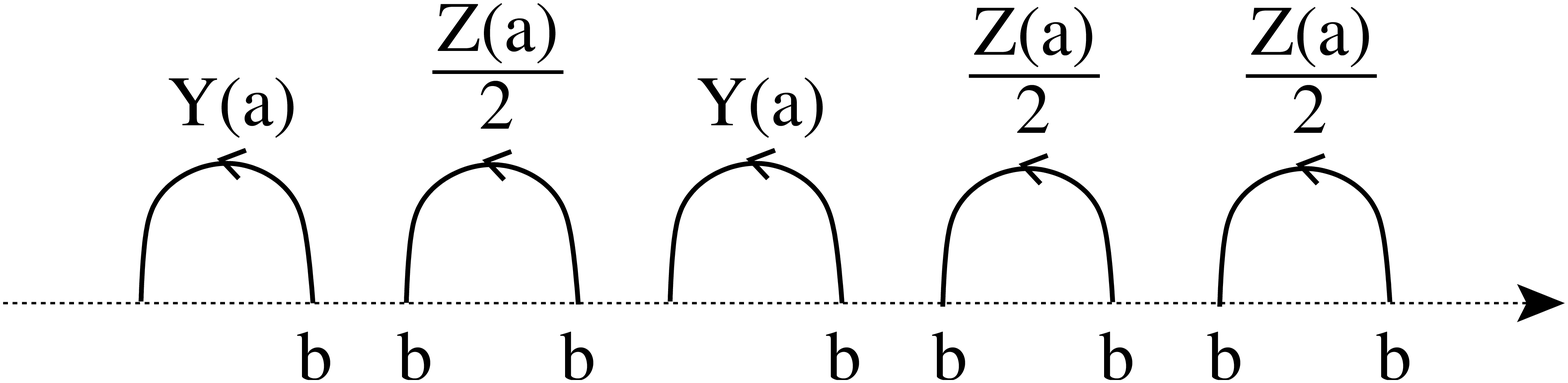}}}.
\end{multline*}
Let's briefly recall, then, how to compute such an operation
product using the graphical method described in Section
\ref{applygraphical}. We draw a grid over an orienting line,
placing the legs of the first factor in order up the left-hand
side of the grid, and the legs of the second factor in order along
the top of the grid:
\[
\raisebox{-3ex}{\scalebox{0.2}{\includegraphics{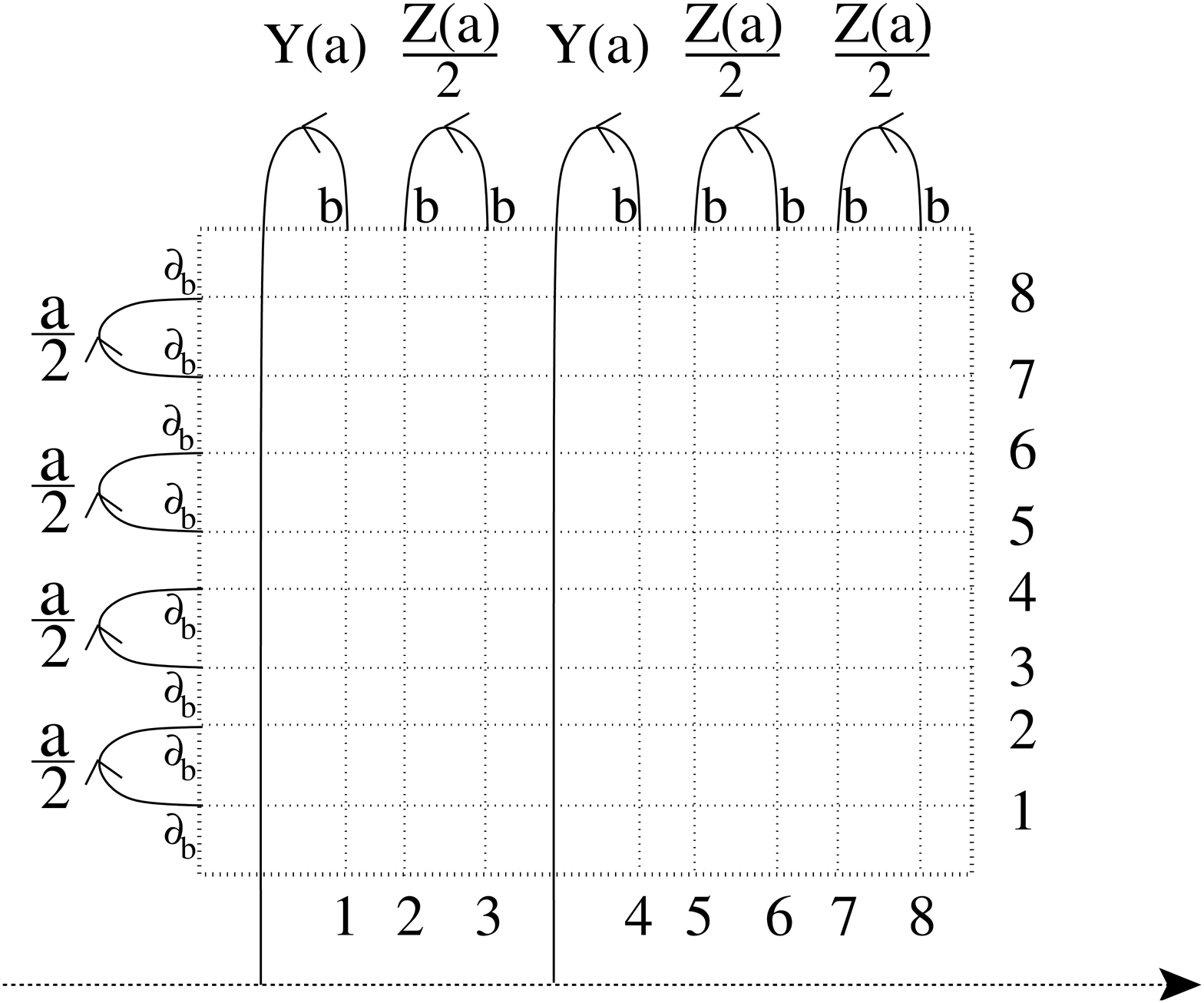}}}\ .
\]
Because we are going to set $b$ and $\partial_b$ to zero at the
end of the calculation, we will get exactly one contribution for
every different way of wiring {\it all} the $\partial_b$-legs to
{\it all} the $b$-legs. In other words, we get exactly one
contribution for every bijection
\[
\phi : \{1,\ldots,8\}\to\{1,\ldots,8\}.
\]
E.g., to construct the contribution corresponding to the bijection
$\left({1 2 3 4 5 6 7 8 \atop 1 5 2 7 4 6 8 3}\right)$ we, first
of all, wire up the grid using this bijection:
\[
\raisebox{-3ex}{\scalebox{0.2}{\includegraphics{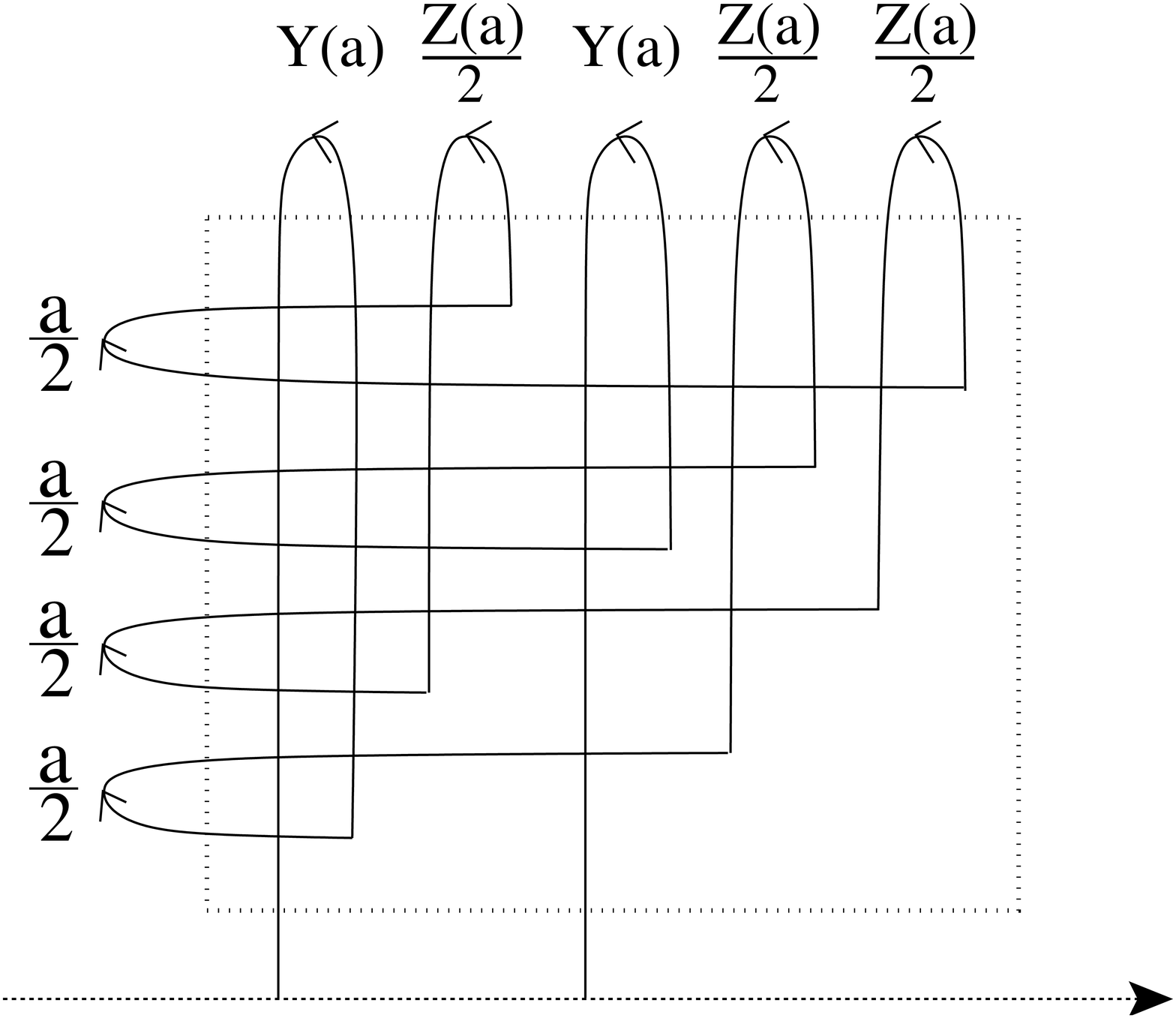}}}\ .
\]
And then the contribution is this diagram multiplied by
$\frac{1}{4!}\frac{1}{5!}(-1)^x$, where $x$ denotes the number of
intersections displayed within the grid. In this example, then,
the contribution is:
\begin{eqnarray*}
& & \frac{1}{4!}\frac{1}{5!}(-1)^{32}\,
\raisebox{-3ex}{\scalebox{0.175}{\includegraphics{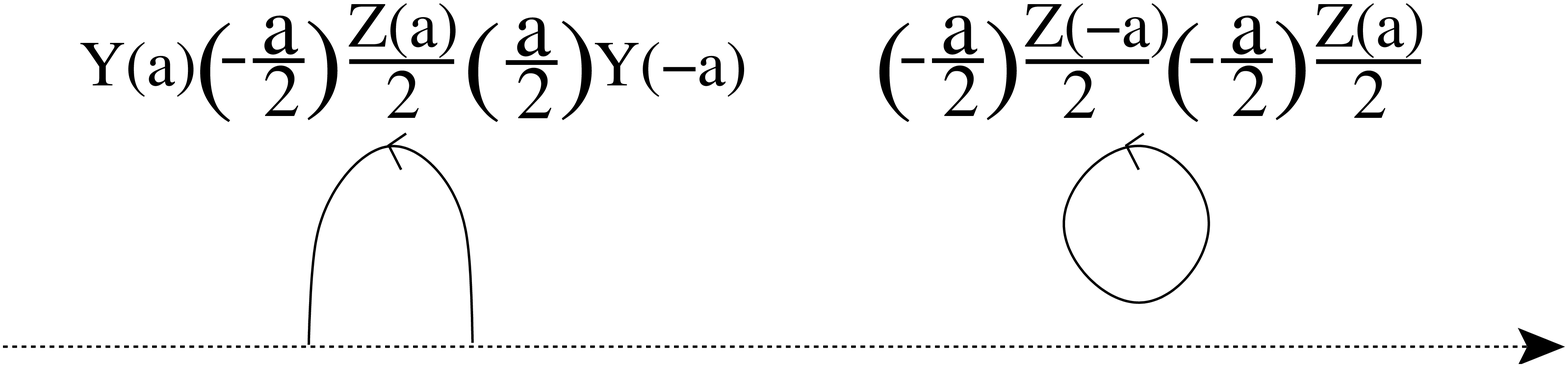}}},
\\[0.2cm]
& =& (+1)\frac{1}{4!}\frac{1}{5!}\,
\raisebox{-3ex}{\scalebox{0.175}{\includegraphics{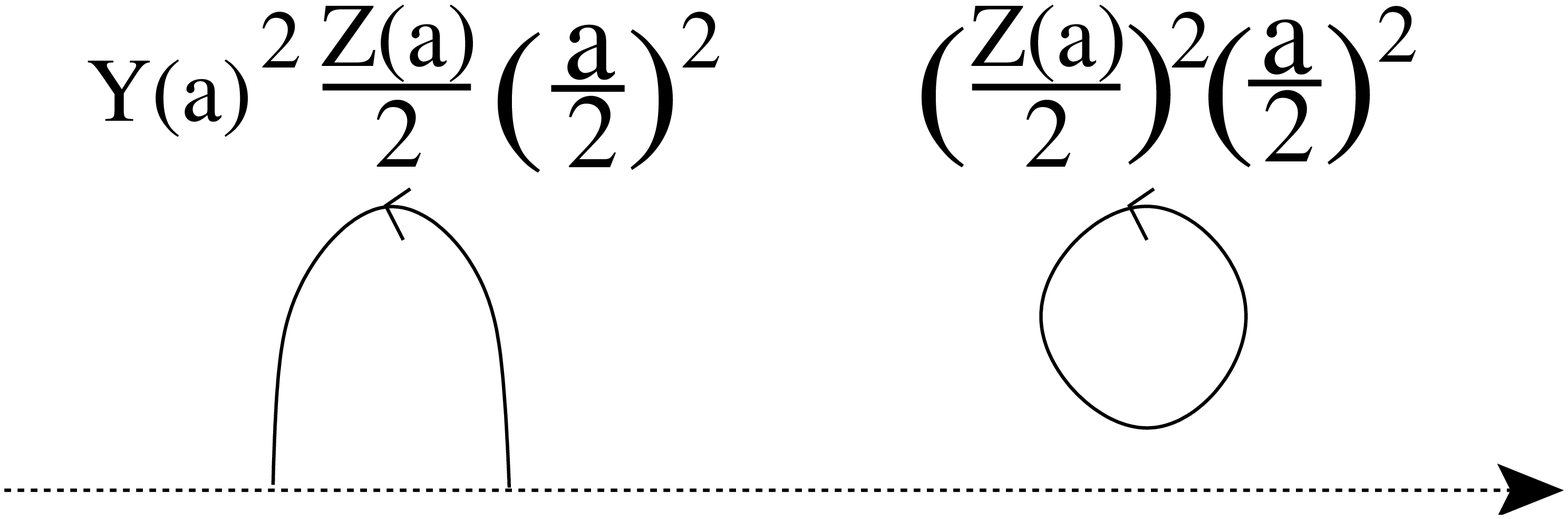}}}.\\[0.05cm]
\end{eqnarray*}
Here we have replaced $a$ by $-a$ whenever we have had to use AS
relations to make $a$-legs lie on the correct side of the
(oriented) edge. Also recall that $Y$ is assumed to have only even powers of $a$, so $Y(-a)=Y(a)$, while $Z$ is assumed to only have odd powers
of $a$, giving $Z(-a)=-Z(a)$.
\subsubsection{The formal development.}
Now let's set this up more formally. We'll begin by defining a set
$\nabla$ which will index the different terms that arise when
Expression \ref{tobecomputed} is evaluated. The elements of
$\nabla$ are triples $(n,w,\sigma)$, where $n$ is a positive
integer, where $w$ is a word in the symbols $Y$ and $Z$ such that
$(\#Y)+2(\#Z)=2n$, and where $\sigma$ is a bijection
\[
\sigma : \{1,2,\ldots,2n\} \rightarrow \{1,2,\ldots,2n\}.
\]
Given some triple $(n,w,\sigma)\in \nabla$, let
$\widetilde{T}_{(n,w,\sigma)}$ denote the corresponding term
(where the $\widetilde{T}$ has a tilde to logically distinguish it
from the $T$'s of Section \ref{computingtheoperatorexpressionA}).
For example, we just showed that
\[
\widetilde{T}_{\left(\,4\,,\,YZYZZ\,,\,\left({1 2 3 4 5 6 7 8
\atop 1 5 2 7 4 6 8 3}\right)\,\right)} =
(+1)\frac{1}{4!}\frac{1}{5!}\,
\raisebox{-3ex}{\scalebox{0.175}{\includegraphics{contributionexampC}}}.
\]

Let $\nabla_{\mathcal{C}}\subset \nabla$ denote the elements of
$\nabla$ whose corresponding diagram is {\bf connected}. (For
example, the case just treated is {\it not} an element of this
subset.) We'll omit the proof of the following proposition, which
is a tedious combinatorial argument closely analogous to the proof
of Proposition \ref{decompintoconnecteds}.
\begin{prop}
The expression to be computed, Expression \ref{tobecomputed}, is
equal to
\[
\exp_{\#}\left(\sum_{\tau\in\nabla_{\mathcal{C}}}
\widetilde{T}_\tau \right).
\]
\end{prop}
\subsubsection{The computation of
$\sum_{\tau\in\nabla_{\mathcal{C}}} \widetilde{T}_\tau$\ .} We can
group this sum into two contributions:
\[
\sum_{\tau\in\nabla_{\mathcal{C}}} \widetilde{T}_\tau = C_0 + C_2,
\]
where:
\begin{itemize}
\item{$C_{0}$ denotes the series of terms $\widetilde{T}_\tau$
whose underlying diagrams are connected and have no legs. An
example is $\widetilde{T}_{\left(4\,,\,ZZZZ\,,\,\left({1 2 3 4 5 6
7 8 \atop 4 1 6 8 5 3 7 2}\right)\,\right)}$, which is:
\[
\frac{1}{4!4!}(-1)^{13}\, \raisebox{-8ex}{
\scalebox{0.17}{\includegraphics{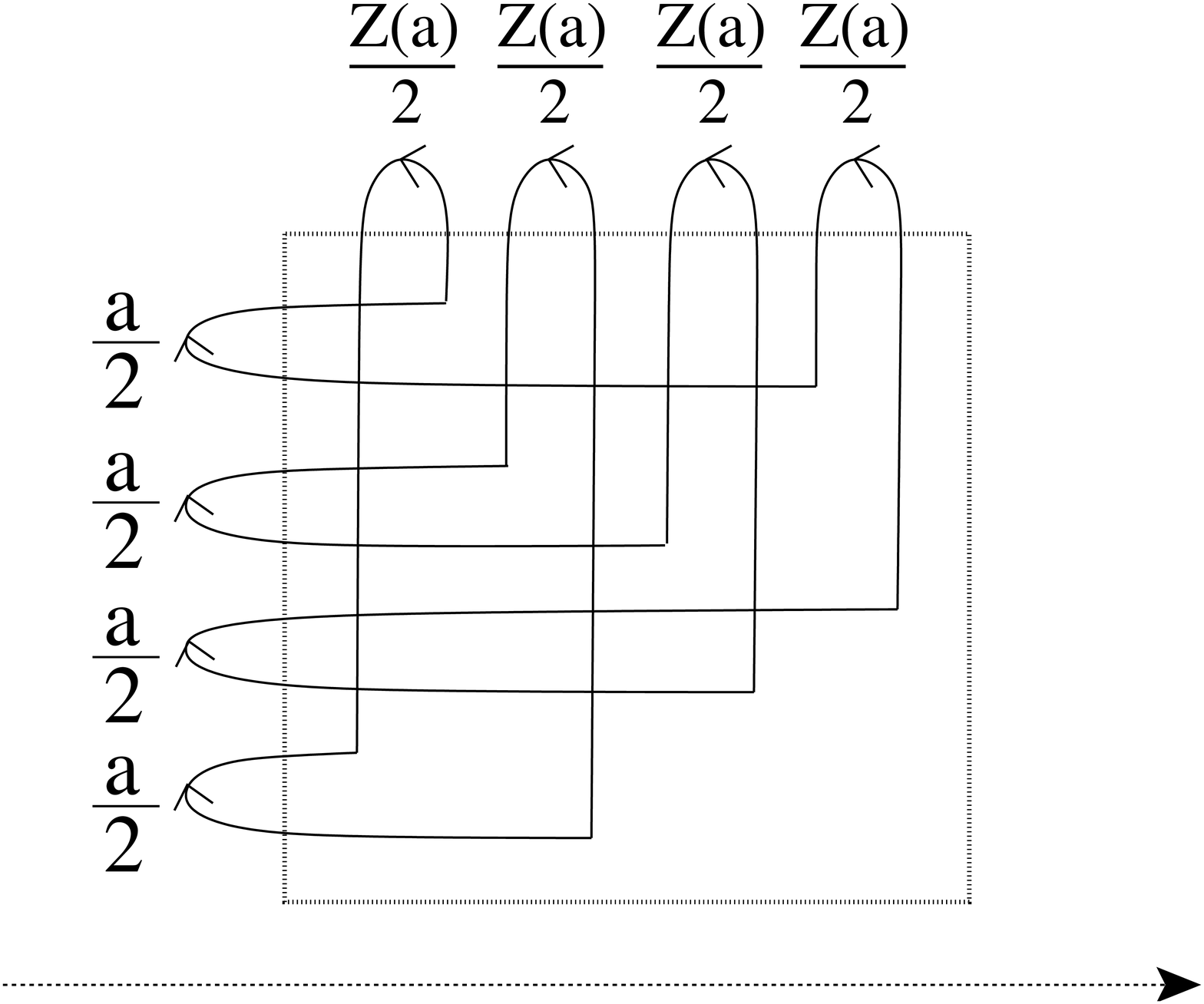}}}\ \ . \\[0.2cm]
\]}
 \item{$C_{2}$ denotes the series of terms $\widetilde{T}_\tau$
whose underlying diagrams are connected and have 2 legs, such as:
\[
\widetilde{T}_{\left(\,2\,,YZY\,,\,\left({1 2 3 4 \atop 4 2 3
1}\right)\,\right)} = \frac{1}{2!3!}(-1)^{6}\, \raisebox{-6ex}{
\scalebox{0.16}{\includegraphics{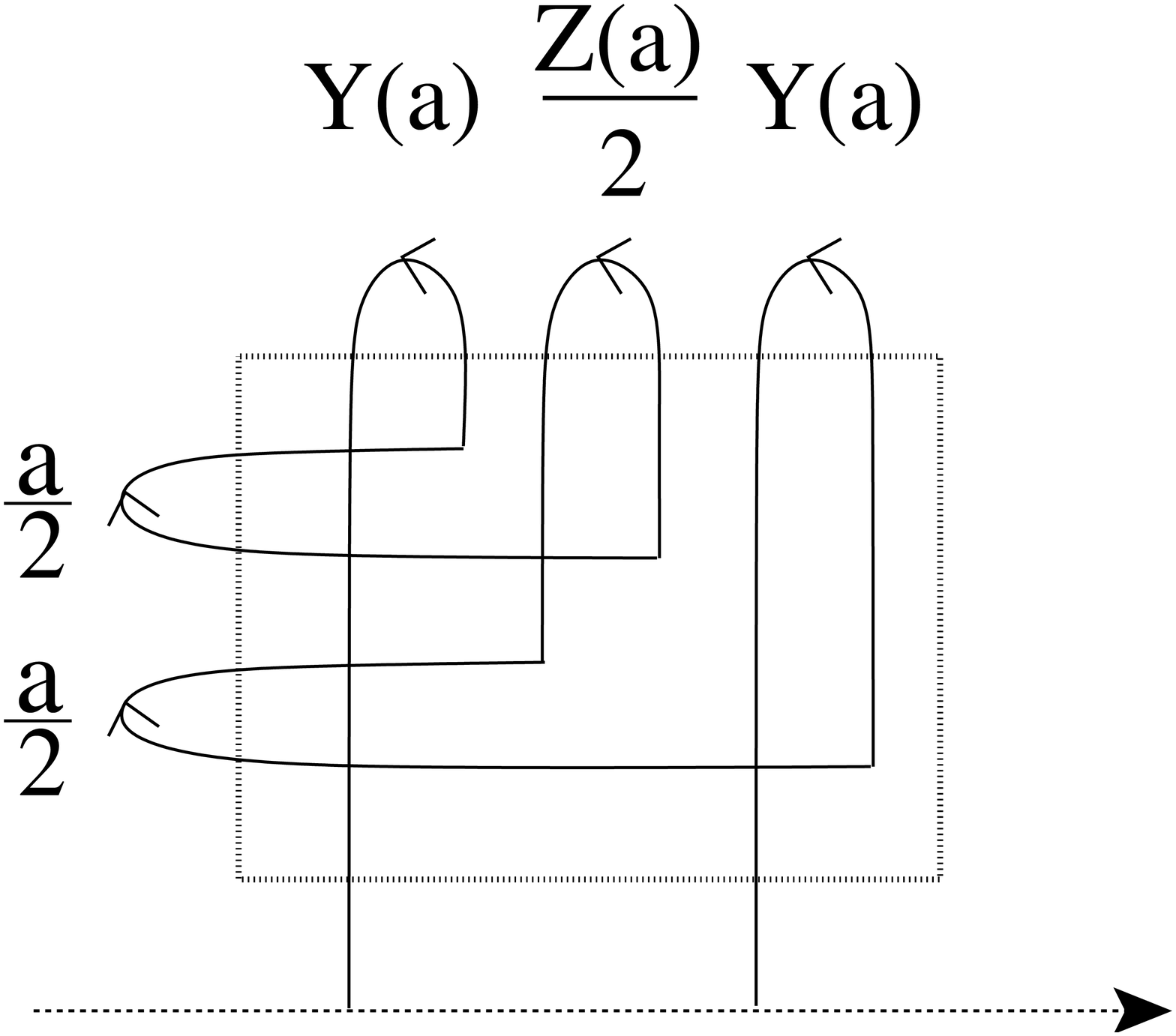}}}\ . \\[0.2cm]
\]}
\end{itemize}
These classes are the only combinatorial possibilities. We'll
compute these two contributions in turn.
\subsubsection{The contribution $C_0$.} The goal of this
subsection is the computation that
\[
C_0\ =\ -\frac{1}{2}\sum_{n=1}^{\infty}\frac{1}{n}\,
\raisebox{-3ex}{\scalebox{0.175}{\includegraphics{loopanswerB}}}.
\]
We'll begin by defining a set $\overrightarrow{\Phi}_n$ which will
index, in a convenient manner, the various terms which contribute
to $C_0$. The elements of $\overrightarrow{\Phi}_n$ will be pairs
$(w_1,w_2)$ of words, the first word $w_1$ using each of the
symbols $\{1\,,\,\ldots\,,\,n\}$ exactly once, and the second word
$w_2$ using each of the symbols $\{2\,,\,\ldots\,,\,n\}$ exactly
once. In addition, every symbol $s$ of the first word $w_1$ is
decorated by either an arrow pointing to the left
$\overleftarrow{s}$ or an arrow pointing to the right
$\overrightarrow{s}$, and every symbol $s$ of the second word
$w_2$ is decorated by either an arrow pointing up $s\uparrow$ or
an arrow pointing down $s\downarrow$. The set
$\overrightarrow{\Phi}_n$ is defined to be the set of all such
pairs. For example:
$\left(\overrightarrow{4}\overleftarrow{1}\overleftarrow{3}\overrightarrow{2}\,,\,2\uparrow3\downarrow4\downarrow\right)
\in \overrightarrow{\Phi}_4.$ To write down the pair of words
corresponding to some gluing giving a connected diagram with no
legs, start at the top-most factor $\frac{a}{2}$. The point
referred to is decorated in the following example with a bullet:
\begin{equation}\label{traverseexample} \raisebox{-8ex}{
\scalebox{0.18}{\includegraphics{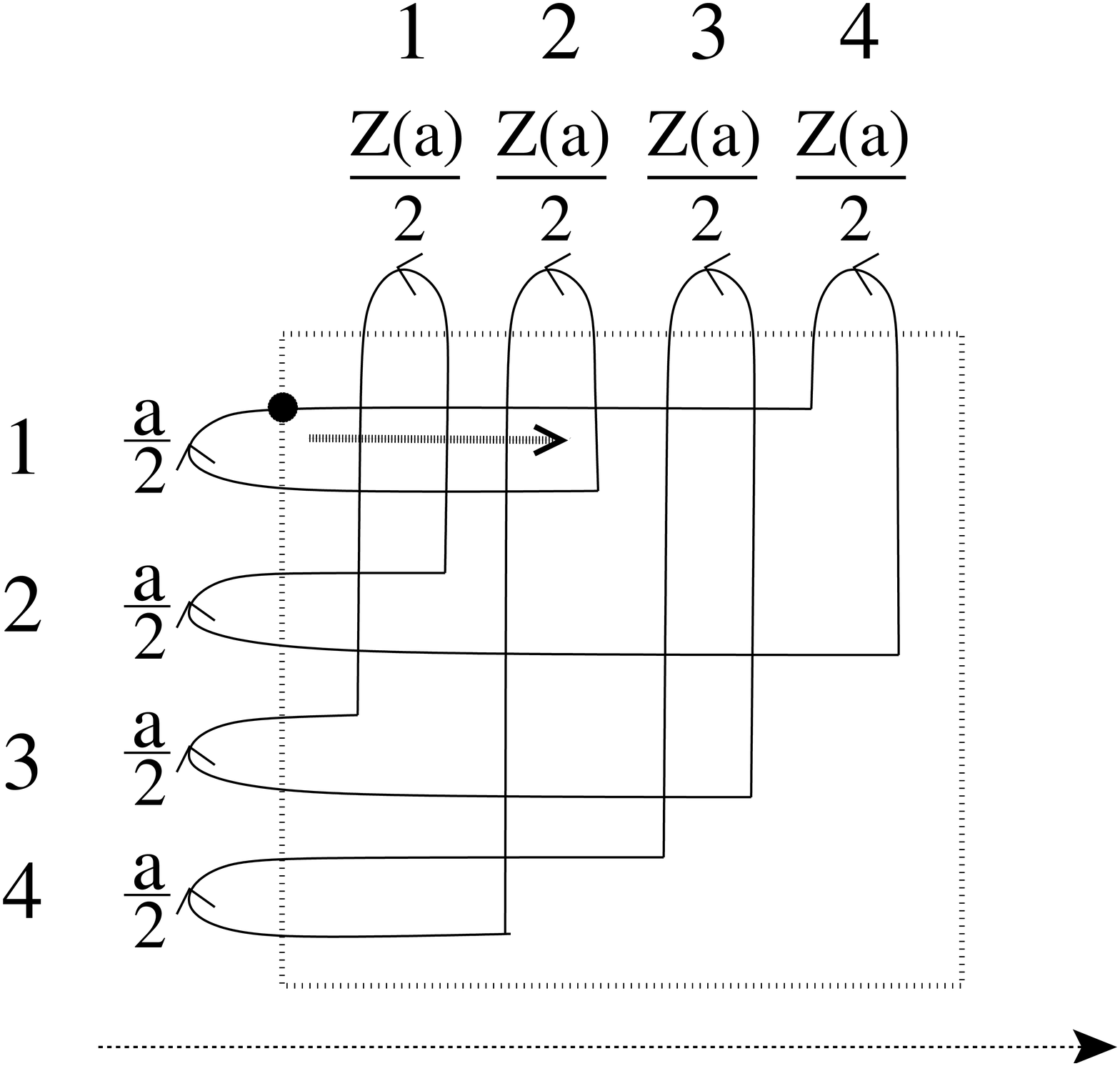}}}\ .
\end{equation}
Now traverse the diagram until you return to where you started
(the arrow in the example indicates how to begin this traverse).
The first word records the order and direction in which you
encounter the factors written along the top line as you traverse
(in this example the corresponding word is
$\overrightarrow{4}\overleftarrow{1}\overleftarrow{3}\overrightarrow{2}$),
and the second word records the order in which you encounter the
factors written down the left-hand side (in this case,
$2\uparrow3\downarrow4\downarrow$).

Given some element $(w_1,w_2)\in \overrightarrow{\Phi}_n$, let
$\phi_{(w_1,w_2)}$ denote the corresponding term. In the example
at hand:
\[
\phi_{
\left(\overrightarrow{4}\overleftarrow{1}\overleftarrow{3}\overrightarrow{2}\,,\,2\uparrow3\downarrow4\downarrow\right)
}\ =\ (-1)^{17}\frac{1}{4!4!} \raisebox{-8ex}{
\scalebox{0.18}{\includegraphics{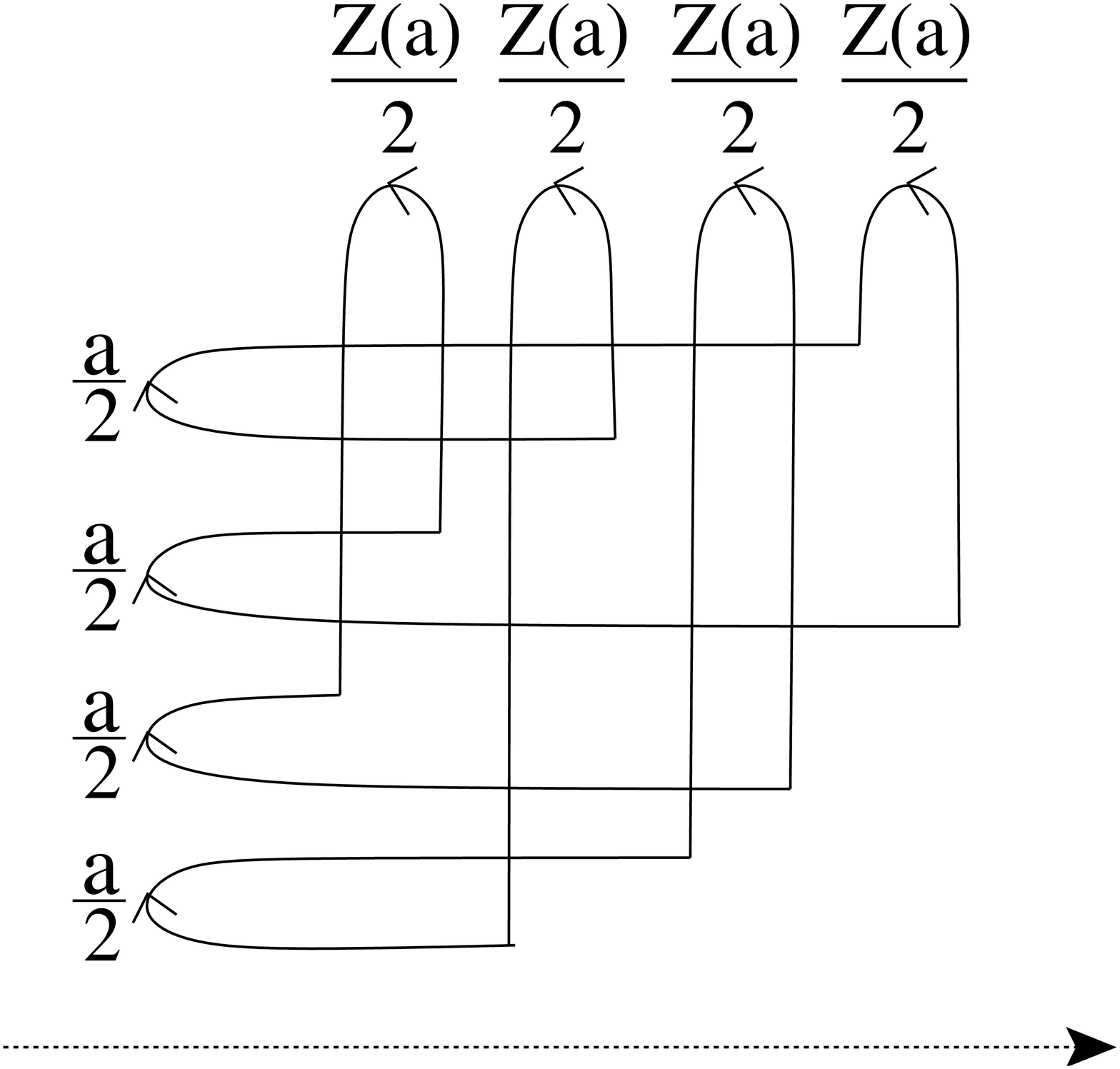}}}.
\]
The series to be calculated can now be written \[C_0 =
\sum_{n=1}^{\infty} \sum_{(w_1,w_2)\in
\overrightarrow{\Phi}_n}\phi_{(w_1,w_2)}\ .\] This calculation is
greatly simplified by the observation that, for a fixed $n$, all
the $\phi_{(w_1,w_2)}$, for $(w_1,w_2)\in\overrightarrow{\Phi}_n$,
are {\it precisely {equal}}. This fact is part of the following
lemma, whose proof appears later in this section.
\begin{lem}\label{allthesame}
Consider some $n\geq 1$ and some
$(w_1,w_2)\in\overrightarrow{\Phi}_n$. Then:
\[
\phi_{(w_1,w_2)} = -\frac{1}{n!n!}\ \
\raisebox{-4ex}{\scalebox{0.2}{\includegraphics{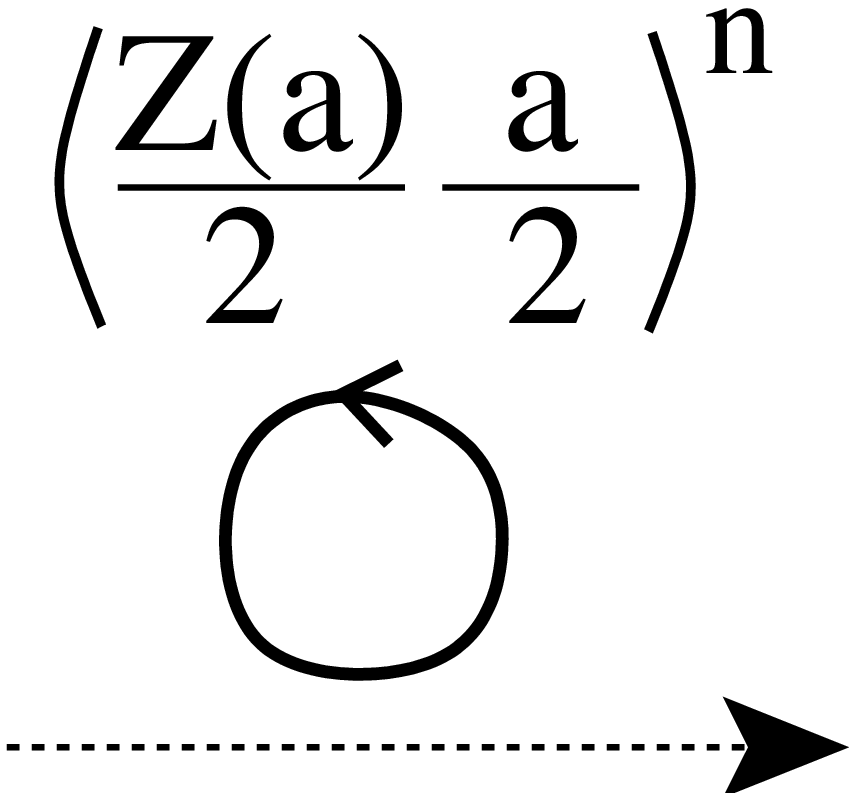}}}\ .
\]
\end{lem}
With this information in hand, $C_0$ is easily computed:
\begin{eqnarray*}
C_0 & = & \sum_{n=1}^{\infty} \sum_{(w_1,w_2)\in
\overrightarrow{\Phi}_n}\phi_{(w_1,w_2)} \\
& = &
(-1)\sum_{n=1}^{\infty}|\overrightarrow{\Phi}_n|\frac{1}{n!n!}
\raisebox{-4ex}{\scalebox{0.2}{\includegraphics{loopanswer}}}\ ,
\\[0.2cm]
& = & (-1)\sum_{n=1}^{\infty}\frac{2^nn!2^{n-1}(n-1)!}{n!n!}
\raisebox{-4ex}{\scalebox{0.2}{\includegraphics{loopanswer}}}\ ,
\\[0.2cm]
& = & -\frac{1}{2}\sum_{n=1}^{\infty}\frac{1}{n}
\raisebox{-3ex}{\scalebox{0.2}{\includegraphics{loopanswerB}}}\
,\\[0.05cm]
\end{eqnarray*}
as required.

 {\it Proof of Lemma
\ref{allthesame}.} We'll begin by introducing some notation. Given
a gluing datum $(w_1,w_2)\in\overrightarrow{\Phi}_n$, let
$D_{(w_1,w_2)}$ denote the series of diagrams (in
$\Whatwedge\abpow$) represented by the drawing you get when you
wire up a grid according to $(w_1,w_2)$. Let $x(w_1,w_2)$ denote
the number of intersections displayed by that drawing. According
to these definitions, the corresponding contribution is written:
\[
\phi_{(w_1,w_2)} = \frac{1}{n!n!}(-1)^{x(w_1,w_2)}D_{(w_1,w_2)}.
\]
This proof is based on two moves, which we'll call R-moves (for
tRansposition) and W-moves (for tWist), that we can perform on
gluing data:
\[
(w_1,w_2)\stackrel{\text{R-move}}{\longrightarrow}(w_1',w_2')\ \ \
\text{and}\ \ \
(w_1,w_2)\stackrel{\text{W-move}}{\longrightarrow}(w_1',w_2'),
\]
whose key property is that:
\begin{equation}\label{proofskeyproperty}
\frac{1}{n!n!}(-1)^{x(w_1,w_2)}D_{(w_1,w_2)} =
\frac{1}{n!n!}(-1)^{x(w_1',w_2')}D_{(w_1',w_2')}.
\end{equation}
We'll begin by introducing these two moves and establishing that
the key property holds for them.

({\it T-moves}.) This move is: transposition of adjacent columns
or adjacent rows. Here is an example of a T-move, (where the arcs
which are unaltered by the move have not been drawn in):
\[
\raisebox{-9ex}{\scalebox{0.12}{\includegraphics{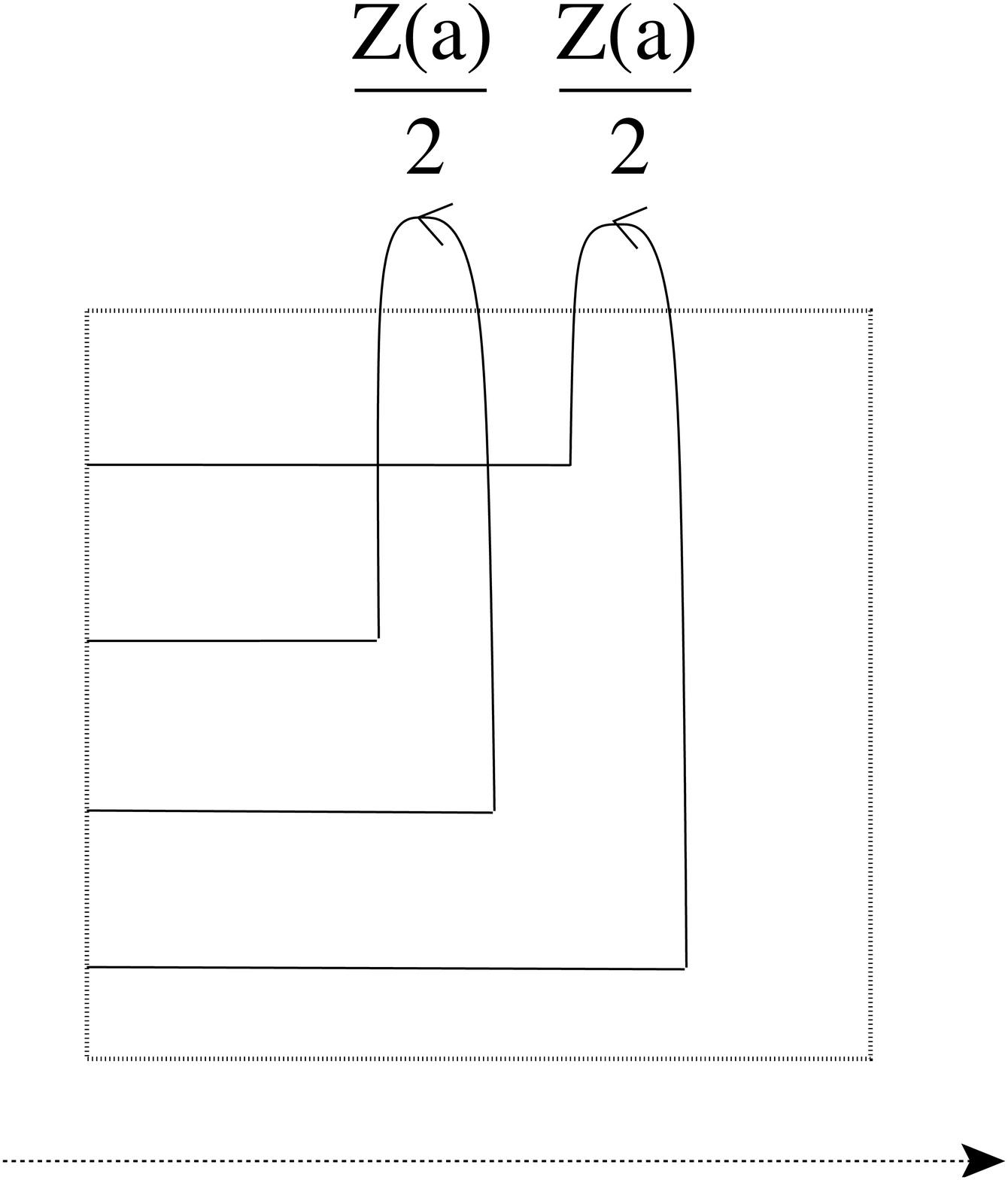}}}
\ \rightleftharpoons\
\raisebox{-9ex}{\scalebox{0.12}{\includegraphics{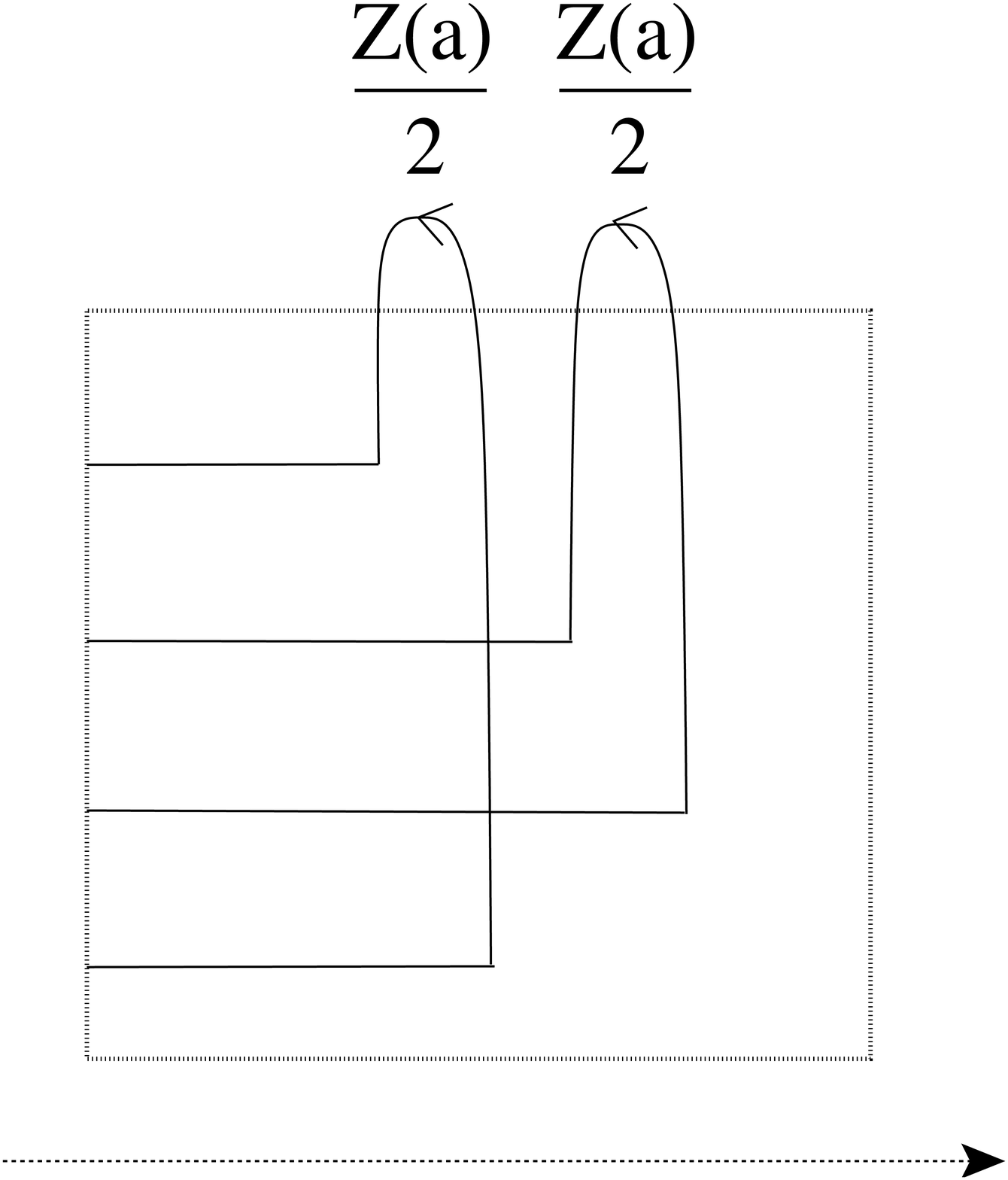}}}\
.
\]
This move can only change the number of displayed intersections by
an even number (indeed, observe that the relative positions of the
ends of all arcs down the left-hand edge of the grid is unaltered
by this move). Thus:
\[
\begin{array}{clp{0cm}l}
 & \frac{1}{n!n!}(-1)^{x(w_1,w_2)}D_{(w_1,w_2)} \\ = &
\frac{1}{n!n!}(-1)^{x(w_1,w_2)}D_{(w_1',w_2')}& &
\text{(As $D_{(w_1',w_2')}=D_{(w_1,w_2)}$),} \\
 = & \frac{1}{n!n!}(-1)^{x(w_1',w_2')}D_{(w_1',w_2')} & &
\text{(As $x(w_1',w_2')=x(w_1,w_2)+\mbox{an even number}$).}
\end{array}
\]

 ({\it W-moves}.) This move is a `half-twist' of a single column
or row: \[
\raisebox{-10ex}{\scalebox{0.13}{\includegraphics{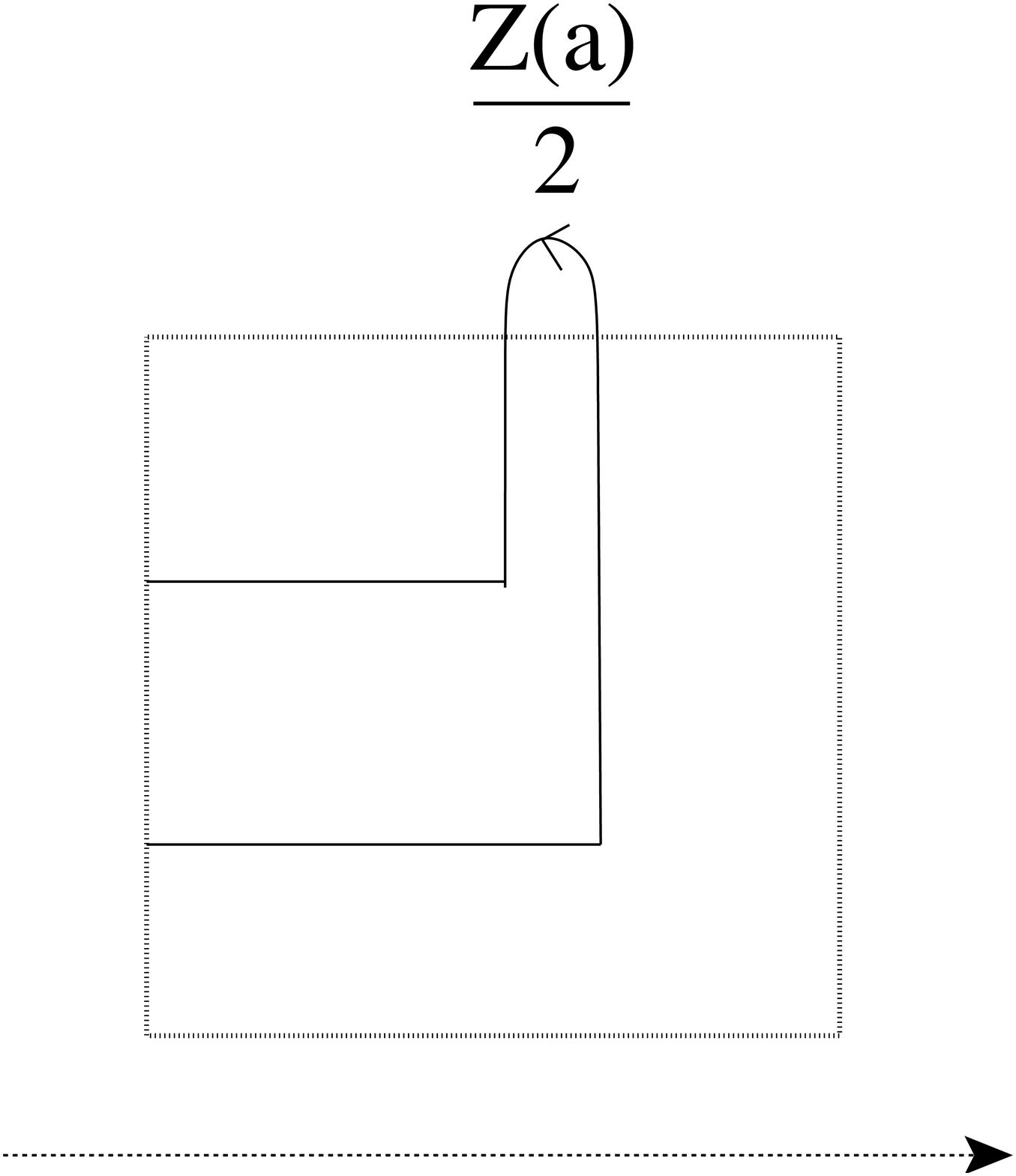}}}
\rightleftharpoons
\raisebox{-10ex}{\scalebox{0.13}{\includegraphics{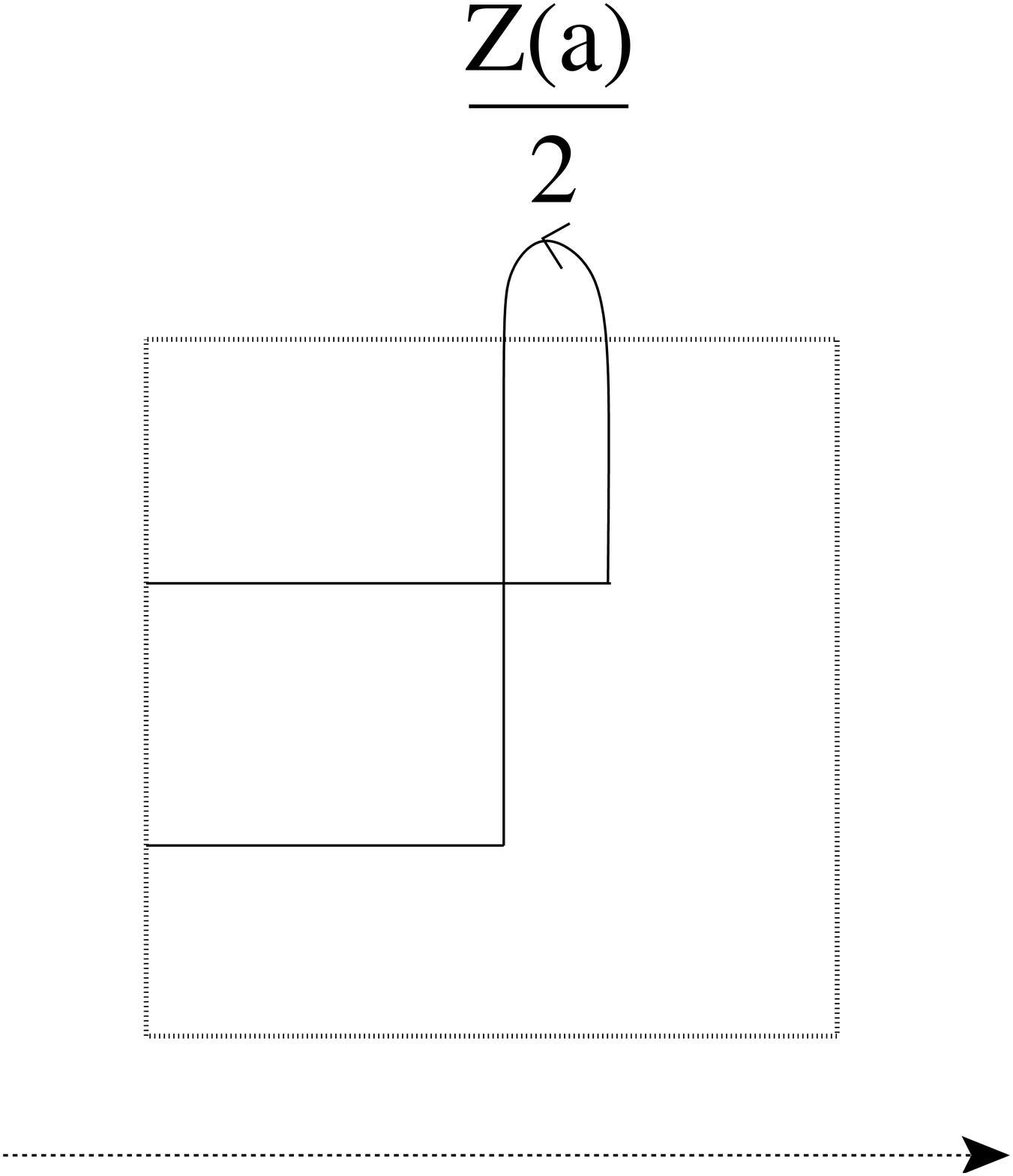}}}\ .
\]
Note that in this case, $D_{(w_1',w_2')} = (-1)D_{(w_1,w_2)}$,
because
\[
\raisebox{-10ex}{\scalebox{0.13}{\includegraphics{twistitA}}} =
\raisebox{-10ex}{\scalebox{0.13}{\includegraphics{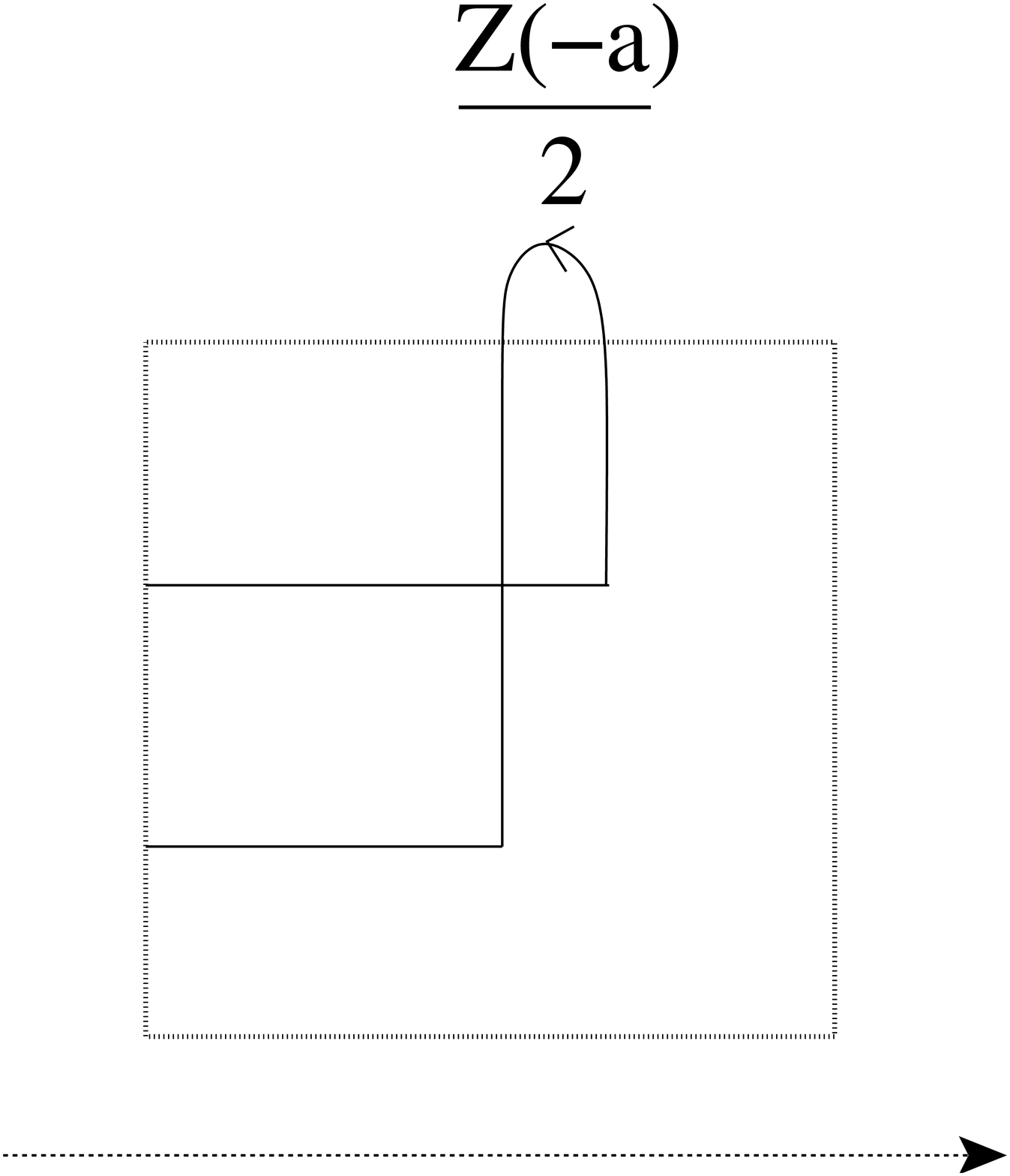}}} =
(-1)\,
\raisebox{-10ex}{\scalebox{0.13}{\includegraphics{twistitC}}}\ .
\]
So:
\[
\begin{array}{cclp{0.15cm}l}
& & \frac{1}{n!n!}(-1)^{x(w_1,w_2)}D_{(w_1,w_2)} \\ & = &
(-1)\frac{1}{n!n!}(-1)^{x(w_1,w_2)}D_{(w_1',w_2')}& &
\text{(As $D_{(w_1',w_2')}=(-1)D_{(w_1,w_2)}$),} \\
& = & \frac{1}{n!n!}(-1)^{x(w_1',w_2')}D_{(w_1',w_2')} & &
\text{(As $x(w_1',w_2')=x(w_1,w_2)\pm 1$),}
\end{array}
\]
as required.

With these moves in hand, we can turn to the general argument. The
simple idea is to show that we can transform any gluing
$(w_1,w_2)$ that we are given, via a sequence of R- and W-moves,
into a {\it standard} gluing:
\[
(w_1,w_2)\rightarrow (w_1',w_2') \rightarrow (w_1'',w_2'')
\rightarrow \ldots \rightarrow
(\overrightarrow{1}\overrightarrow{2}\ldots\overrightarrow{n},2\downarrow
3\downarrow \ldots\, n\downarrow).
\]
For example, the standard gluing in the case $n=4$ is:
\[
D_{\left(\overrightarrow{1}\overrightarrow{2}\overrightarrow{3}\overrightarrow{4},2\downarrow
3\downarrow 4\downarrow\right)}\ =\
\raisebox{-10ex}{\scalebox{0.175}{\includegraphics{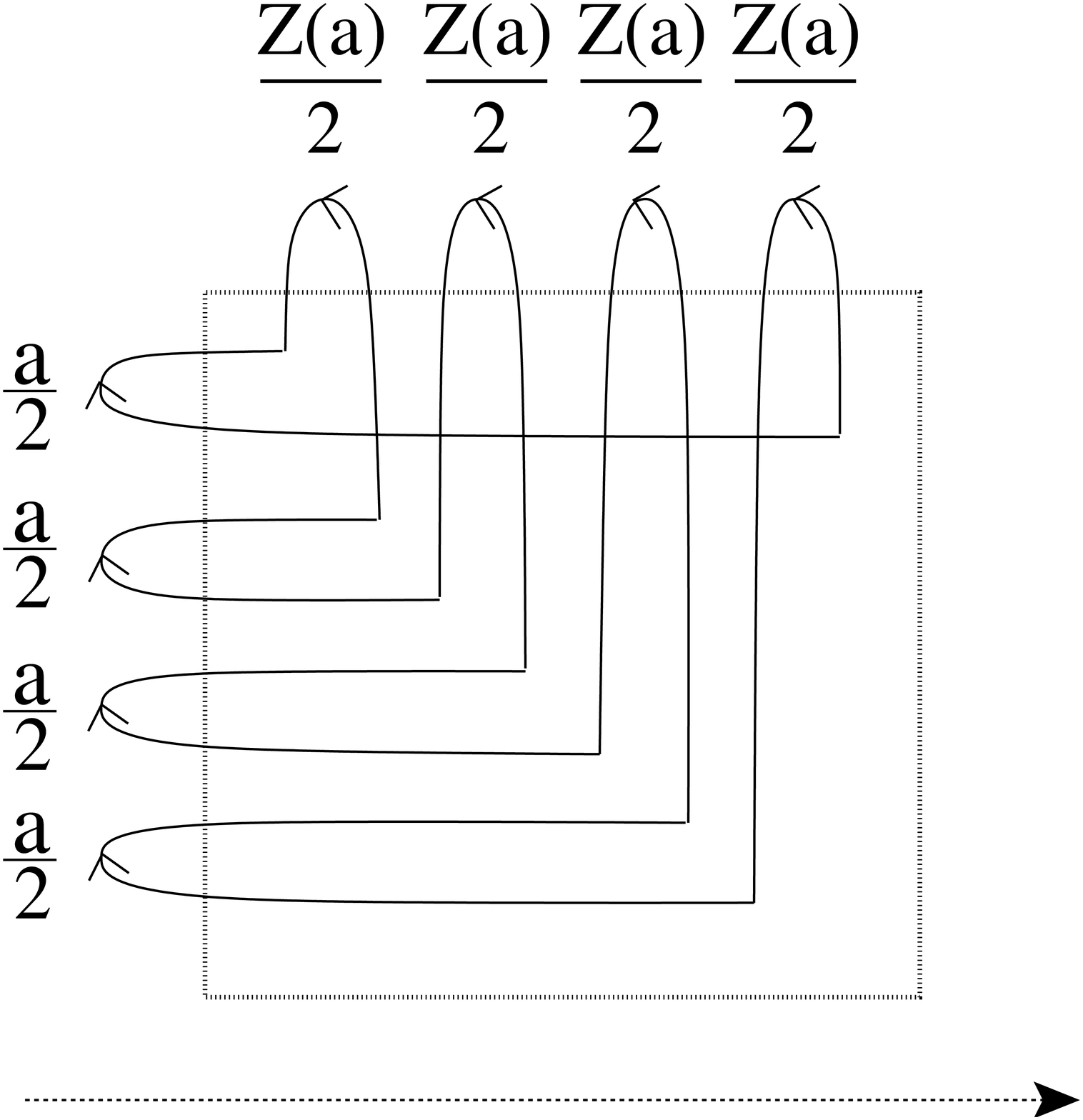}}}\
\ . \\[0.2cm]
\]

We'll now explain how to transform any gluing $(w_1,w_2)$ into the
standard gluing using R- and W-moves. The explanation will be
illustrated by the example of the gluing
$\left(\overleftarrow{1}\overleftarrow{2}\overrightarrow{3}\overleftarrow{4},4\downarrow
2\downarrow 3\uparrow\right)$:
\begin{equation}\label{dotransposeexample}
\raisebox{-10ex}{
\scalebox{0.175}{\includegraphics{nolegsexamp}}}\ \ . \\[0.2cm]
\end{equation}
There are two steps in the procedure.

({\it The first step.}) Begin by doing R-moves to put the factors
along the top and down the side into the order in which they
appear in the words $w_1$ and $w_2$. Let's go through this in the
case of the example shown in line $(\ref{dotransposeexample})$.
We'll start by swapping row 3 and row 4, to get:
\[
\stackrel{\text{R-move}}{\longrightarrow}\raisebox{-10ex}{
\scalebox{0.16}{\includegraphics{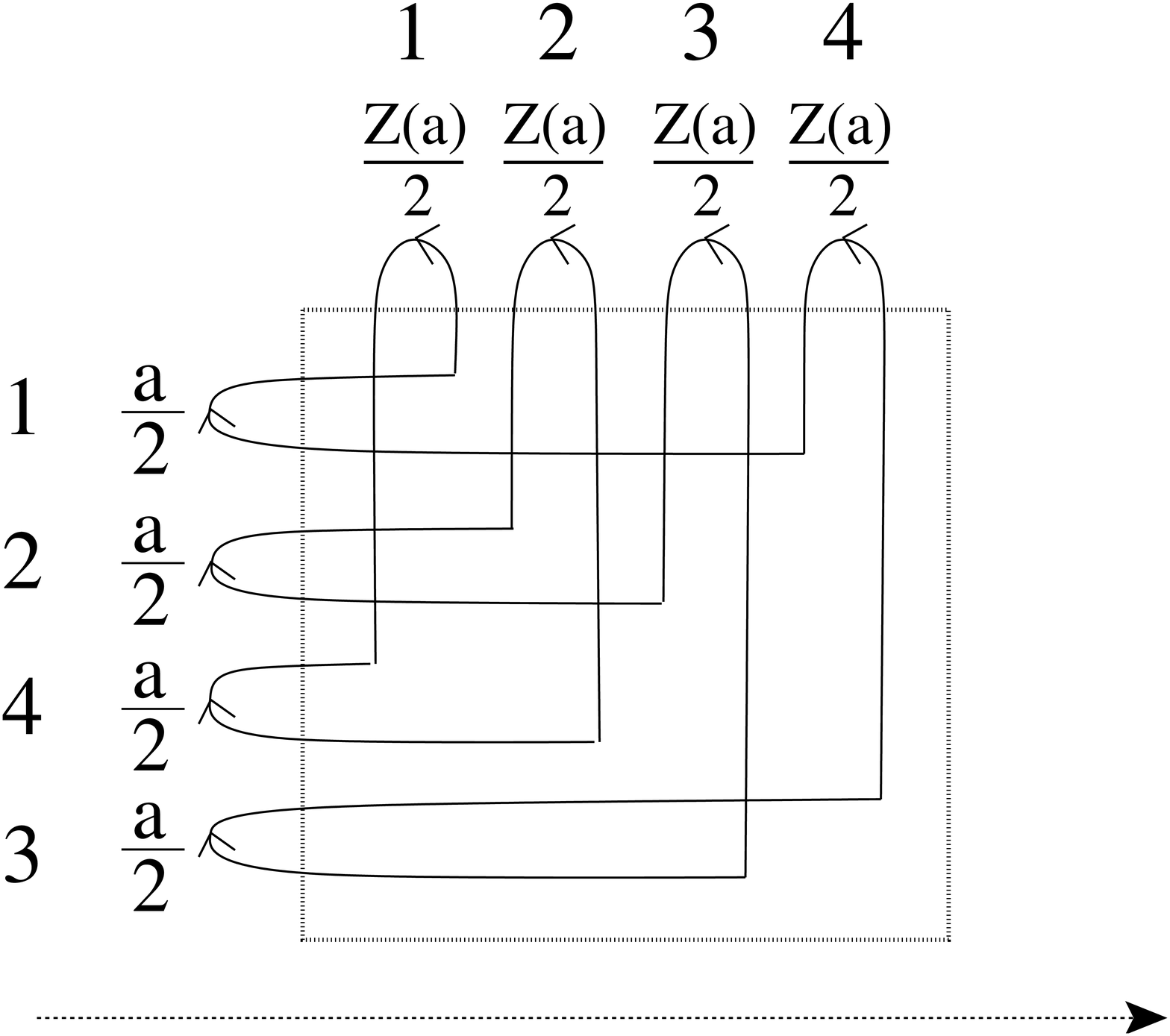}}}\ .\\[0cm]
\]
The we'll swap row 2 and row 3 (i.e. this swaps factor 2 and
factor 4), giving:
\begin{equation}\label{fromhere}
\stackrel{\text{R-move}}{\longrightarrow} \raisebox{-10ex}{
\scalebox{0.16}{\includegraphics{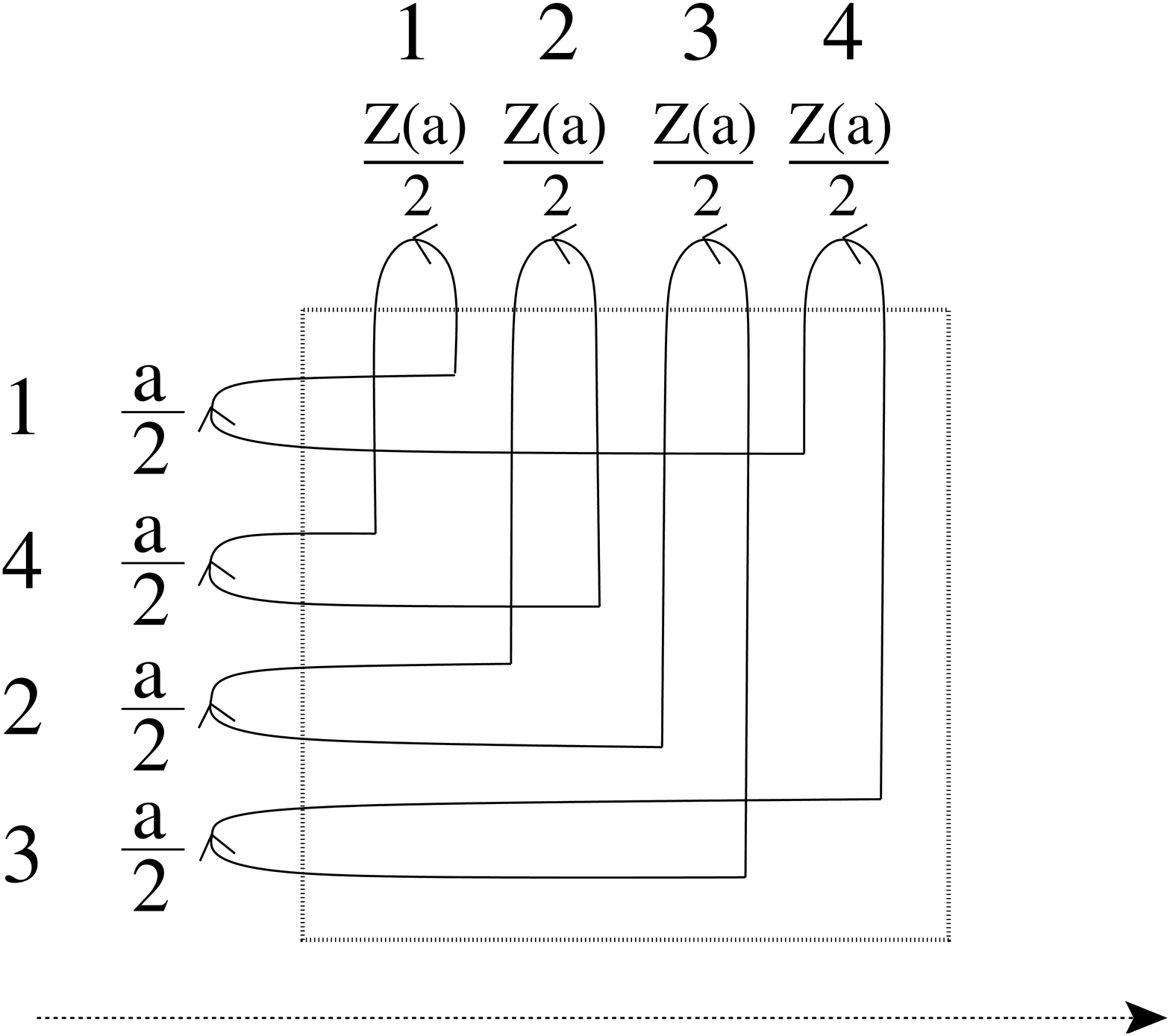}}}\ .
\end{equation}
This finishes the first step of the procedure (notice that after
this step the factors along the top are visited in order from left
to right, and the factors down the side are visited in order from
top to bottom).

({\it The second step.}) The second step of the procedure is to
employ W-moves to arrange it so that, as the drawing is traversed
(as indicated in Line \ref{traverseexample}), the factors along
the top are traversed from left-to-right, and the factors up the
side (except the top-most) are traversed from top-to-bottom.

Our example requires four such twists. Continuing from line
\ref{fromhere}:
\[ \stackrel{\text{W-move}}{\longrightarrow}\raisebox{-9ex}{
\scalebox{0.16}{\includegraphics{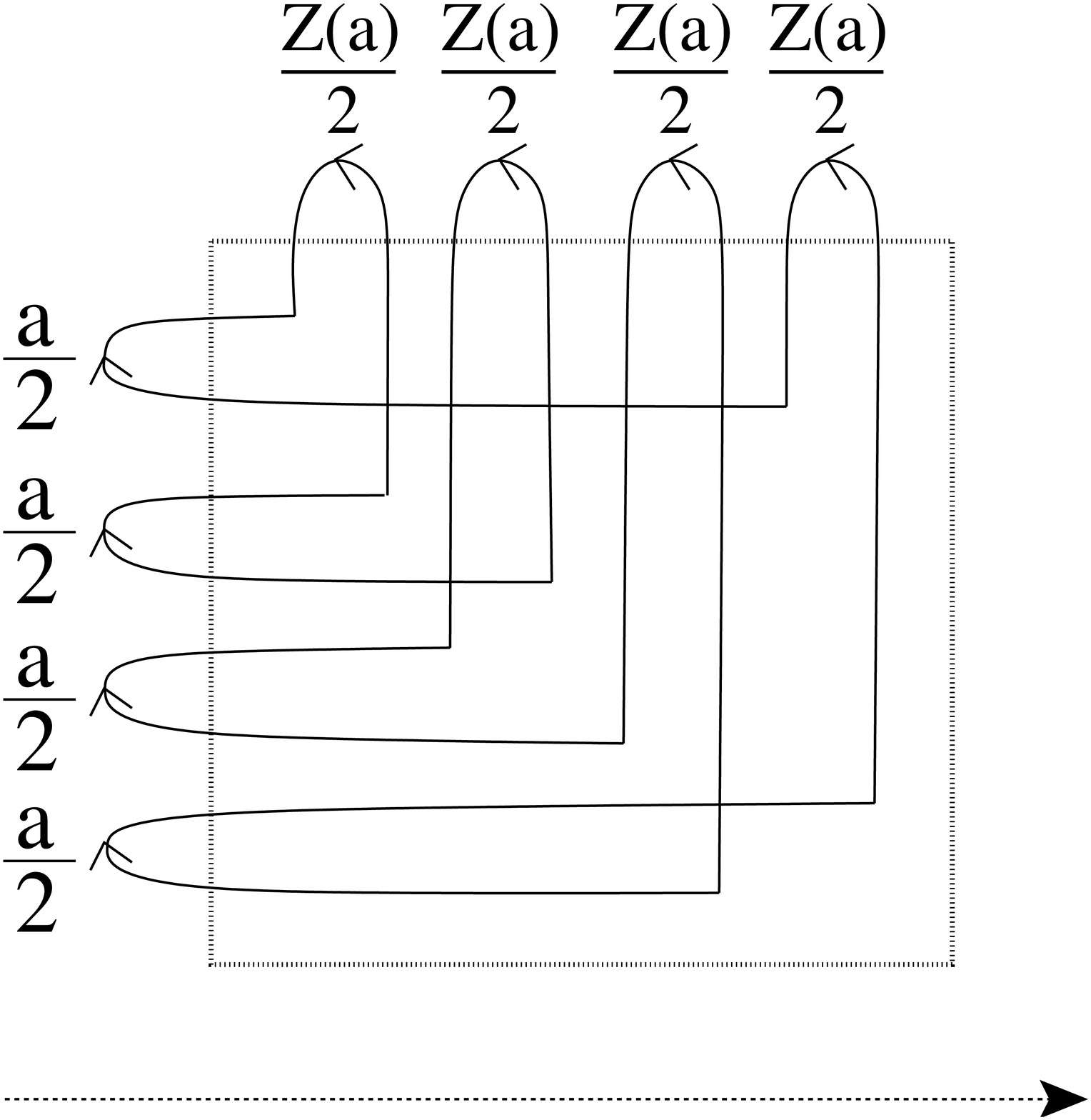}}}
\stackrel{\text{W-move}}{\longrightarrow} \raisebox{-9ex}{
\scalebox{0.16}{\includegraphics{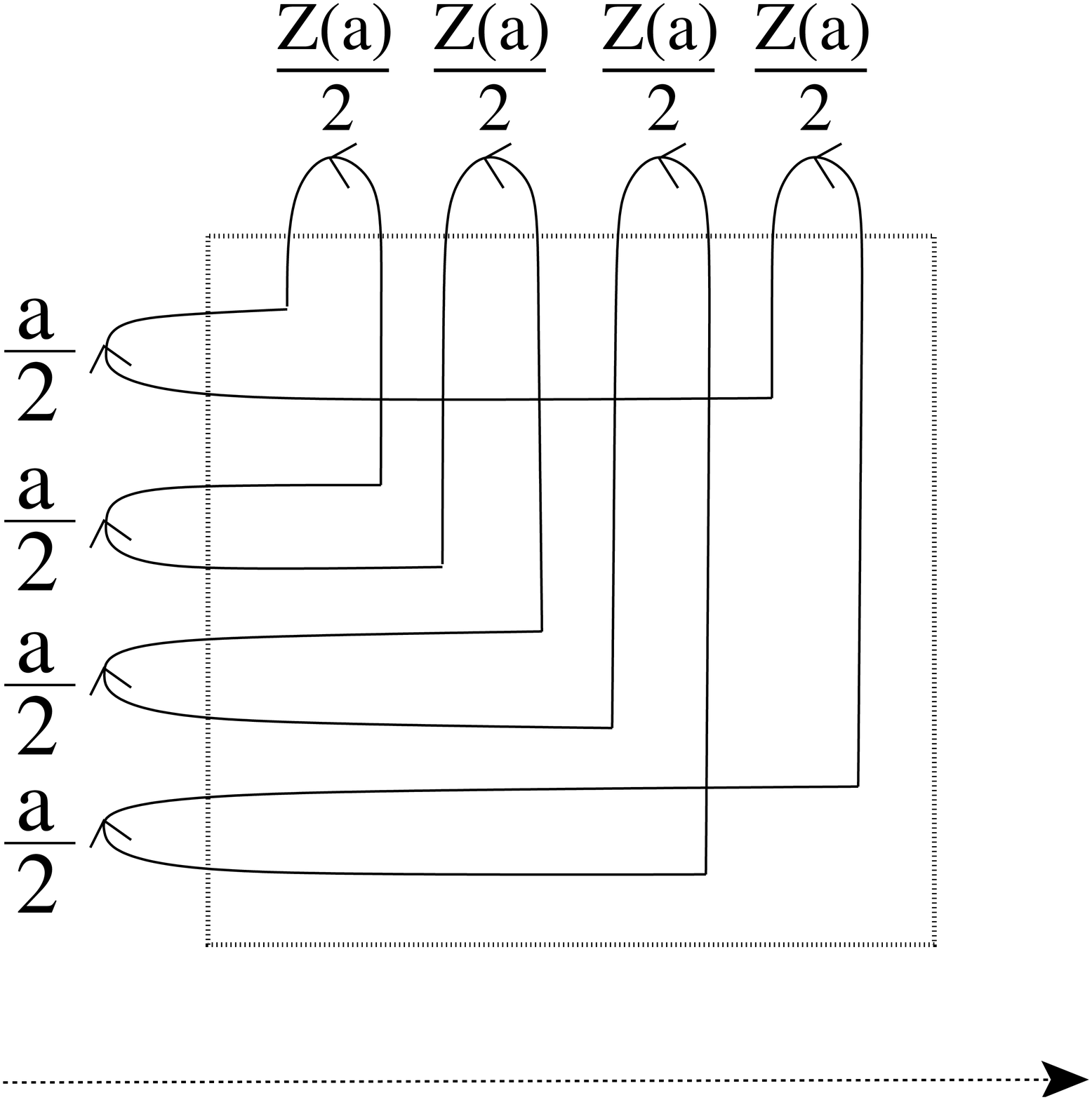}}}.
\]
And:
\[
\stackrel{\text{W-move}}{\longrightarrow}
 \raisebox{-9ex}{
\scalebox{0.16}{\includegraphics{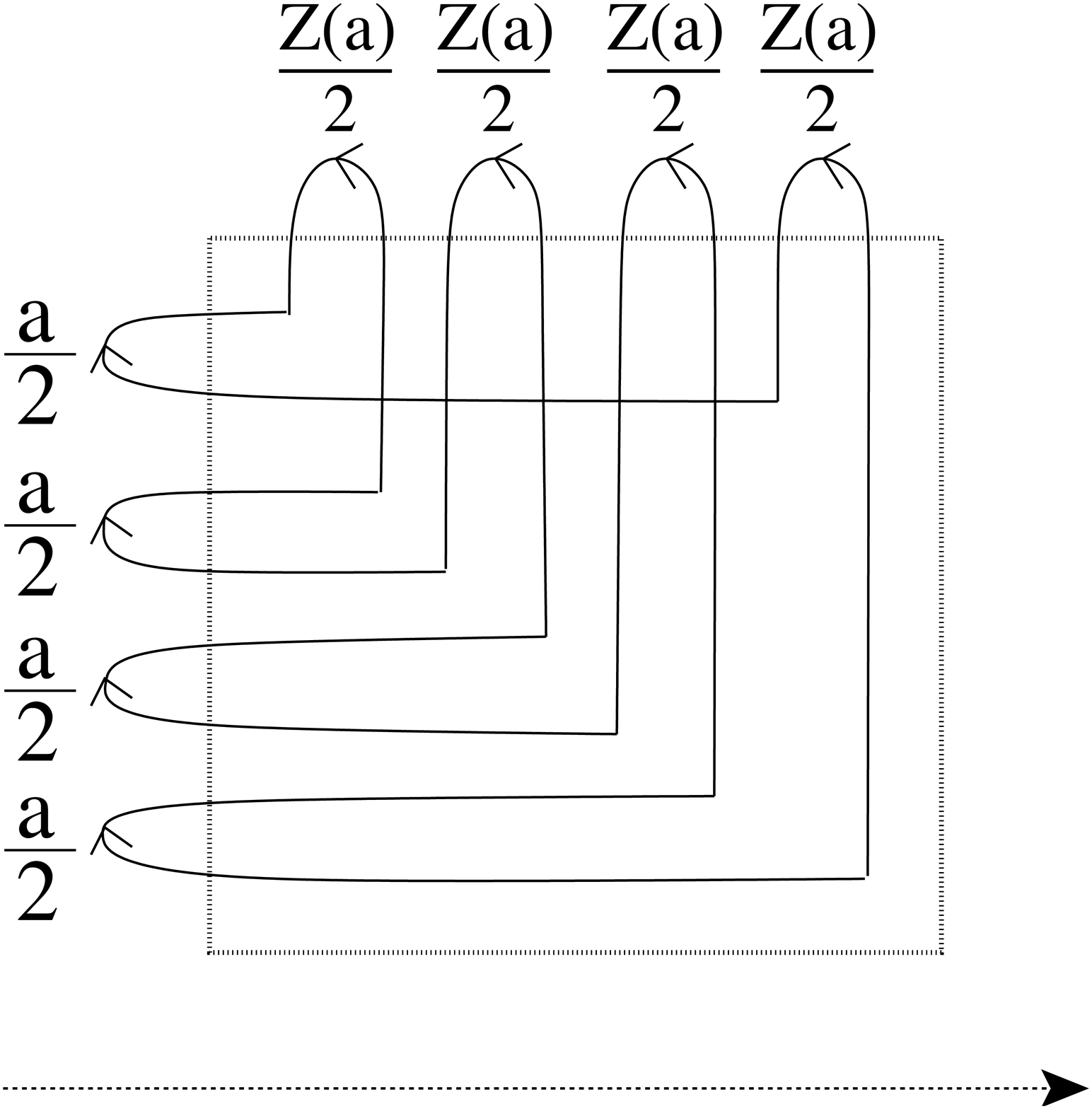}}}
\stackrel{\text{W-move}}{\longrightarrow}
\raisebox{-9ex}{
\scalebox{0.16}{\includegraphics{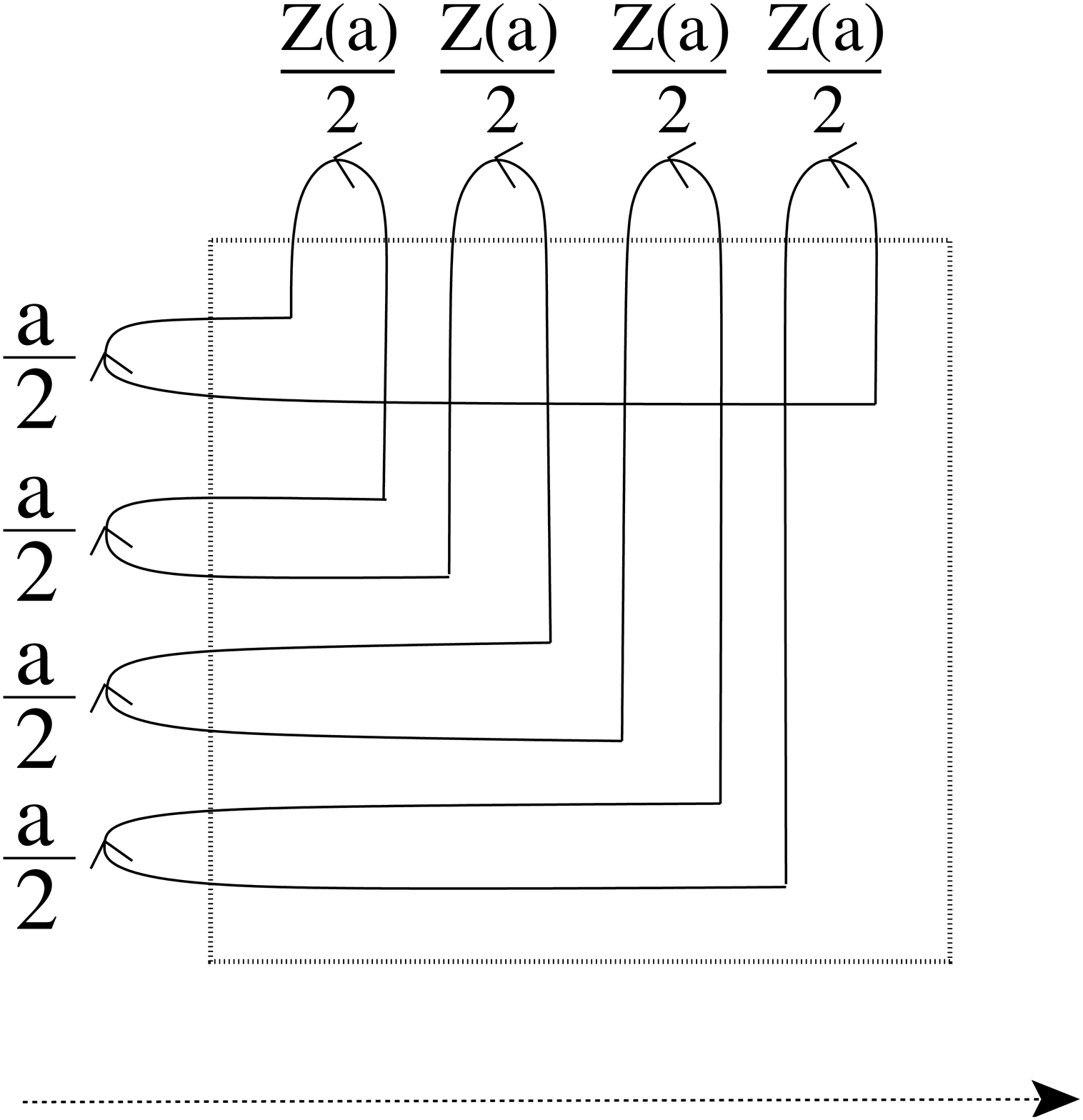}}}\   . \\[0.2cm]
\]
These two steps will transform any given $(w_1,w_2)$ into the
standard gluing, and so, by repeated application of Equation
\ref{proofskeyproperty}, for any
$(w_1,w_2)\in\overrightarrow{\Phi}_n$,
\[
\phi_{(w_1,w_2)}=\frac{1}{n!n!}(-1)^{x(w_1,w_2)}D_{(w_1,w_2)}=\frac{1}{n!n!}(-1)^{x(w_1^s,w_2^s)}D_{(w_1^s,w_2^s)},
\]
where
$(w_1^s,w_2^s)=(\overrightarrow{1}\overrightarrow{2}\ldots\overrightarrow{n},2\downarrow3\downarrow\ldots
n\downarrow)$.

Now let's work out what term that standard gluing represents. We
can put it in a simplified form in the following way (taking the
$n=4$ case as a representative example):
\begin{multline*}
\phi_{(w_1^s,w_2^s)}\ =\  \frac{1}{4!4!}(+1)\raisebox{-12ex}{
\scalebox{0.16}{\includegraphics{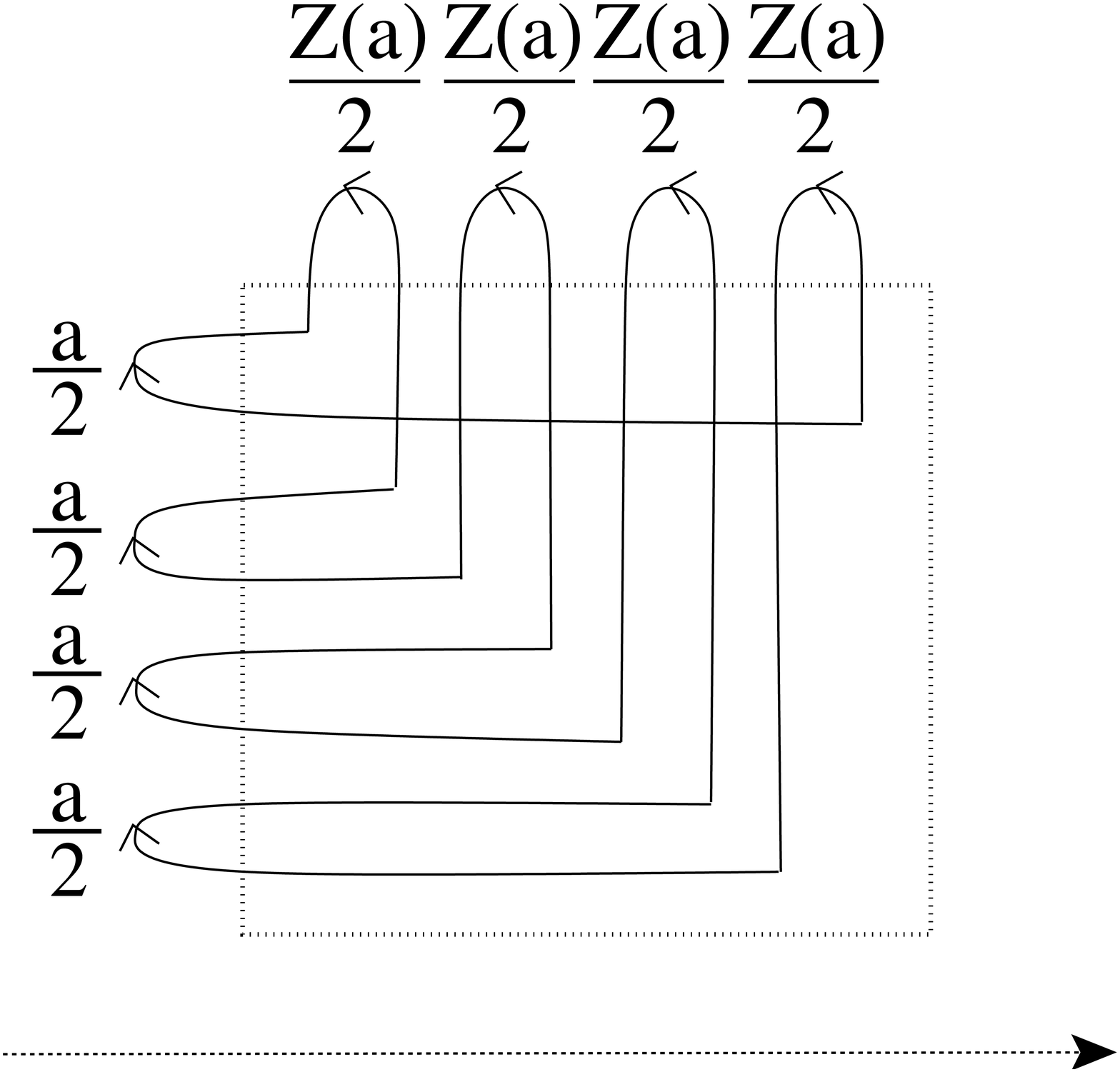}}} \\[0.3cm] =
\frac{1}{4!4!}(+1)\raisebox{-9ex}{
\scalebox{0.16}{\includegraphics{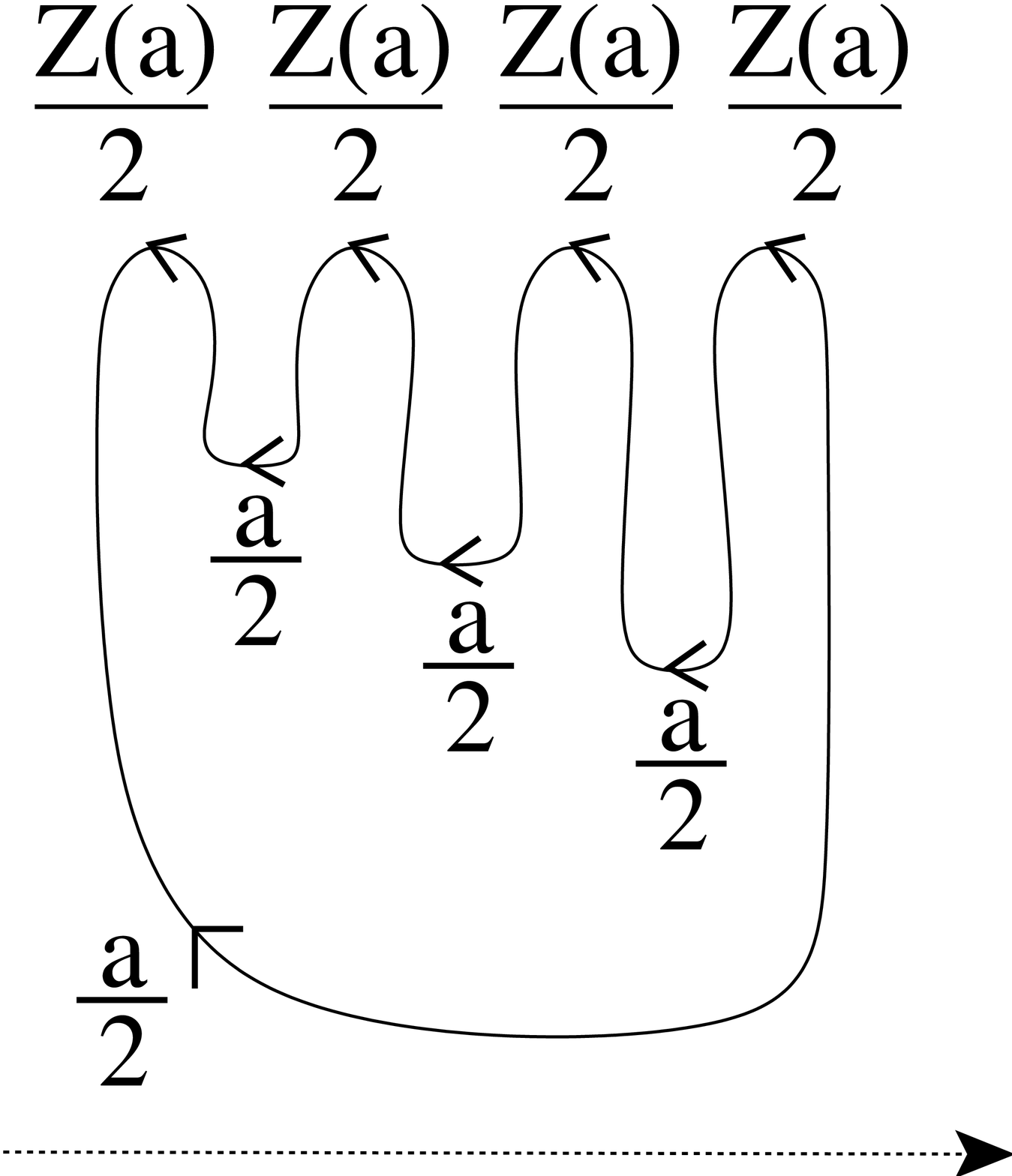}}}\ =\
\frac{1}{4!4!}(-1)\raisebox{-3ex}{\scalebox{0.175}{\includegraphics{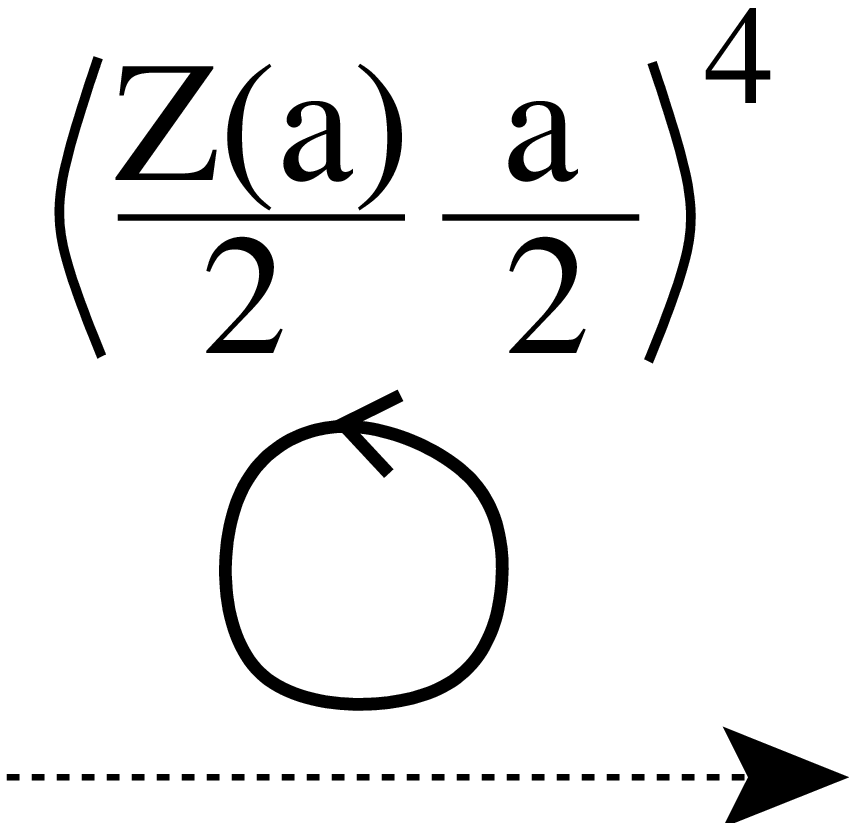}}}\
,
\end{multline*}
as required.
\begin{flushright}
$\Box$
\end{flushright}

\subsubsection{The contribution $C_2$.}
The computation of $C_2$ is closely analogous to the computation
of $C_0$, so we will only provide a sketch of it.

Consider some integer $n\geq 1$. The contributions to $C_2$ from
terms which use $n$ copies of
$\raisebox{-2ex}{\scalebox{0.125}{\includegraphics{arightZ}}}$ can
be indexed by a certain set $\overrightarrow{\Theta}_n$. An
element of the set $\overrightarrow{\Theta}_n$ is a pair of words
$(w_1,w_2)$ where:
\begin{itemize}
\item{$w_1$ is a word which uses each of the symbols
$\{1,\ldots,n+1\}$ precisely once. The last symbol of $w_1$ has
greater value than the first symbol of $w_1$.} \item{Every symbol
$s$ of $w_1$, except the first and last symbol, is decorated by
either an arrow pointing to the right $\overrightarrow{s}$ or an
arrow pointing to the left $\overleftarrow{s}$.} \item{The word
$w_2$ is a word using each of the symbols $\{1,2,\ldots,n\}$
exactly once.} \item{Every symbol $s$ of $w_2$ is decorated by
either an arrow pointing up $s\uparrow$ or an arrow pointing down
$s\downarrow$.}
\end{itemize}
To every element $(w_1,w_2)$ of $\overrightarrow{\Theta}_n$ there
correspond a contribution $\theta_{(w_1,w_2)}$ to $C_2$. Consider
the following example: \[
\raisebox{-9ex}{\scalebox{0.15}{\includegraphics{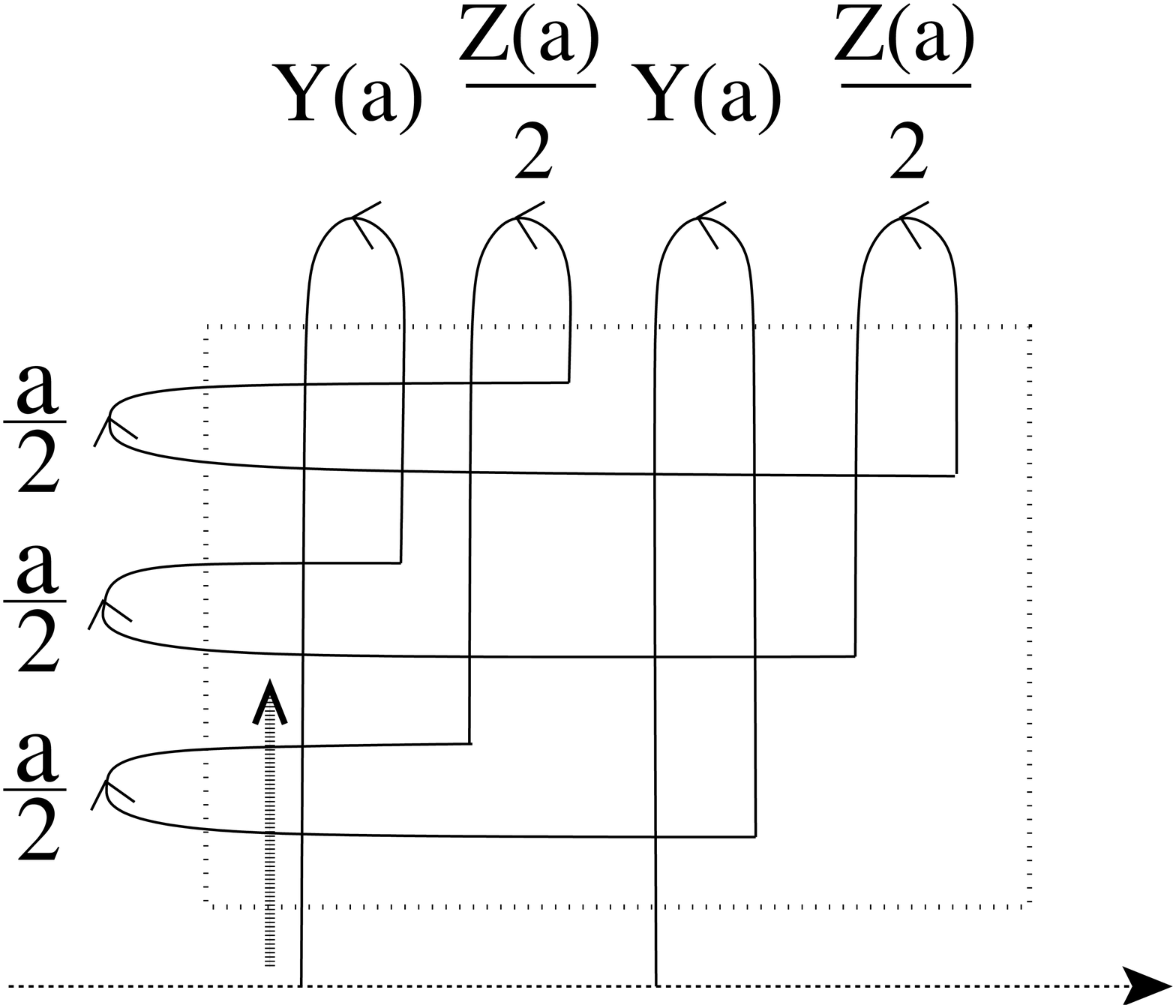}}}\ .
\]
To write down the pair $(w_1,w_2)$ that corresponds with this
gluing start at the base of the left-most of the 2 legs. Then
trace the diagram. The word $w_1$ records the order in which you
encounter the factors along the top of the diagram; the word $w_2$
records the order in which you encounter the factors written down
the left-hand side of the diagram. Thus:
\[
\theta_{\left(1\overrightarrow{4}\overleftarrow{2}3,2\downarrow1\uparrow3\downarrow\right)}=
(-1)^{17}\frac{1}{3!4!}\
\raisebox{-9ex}{\scalebox{0.15}{\includegraphics{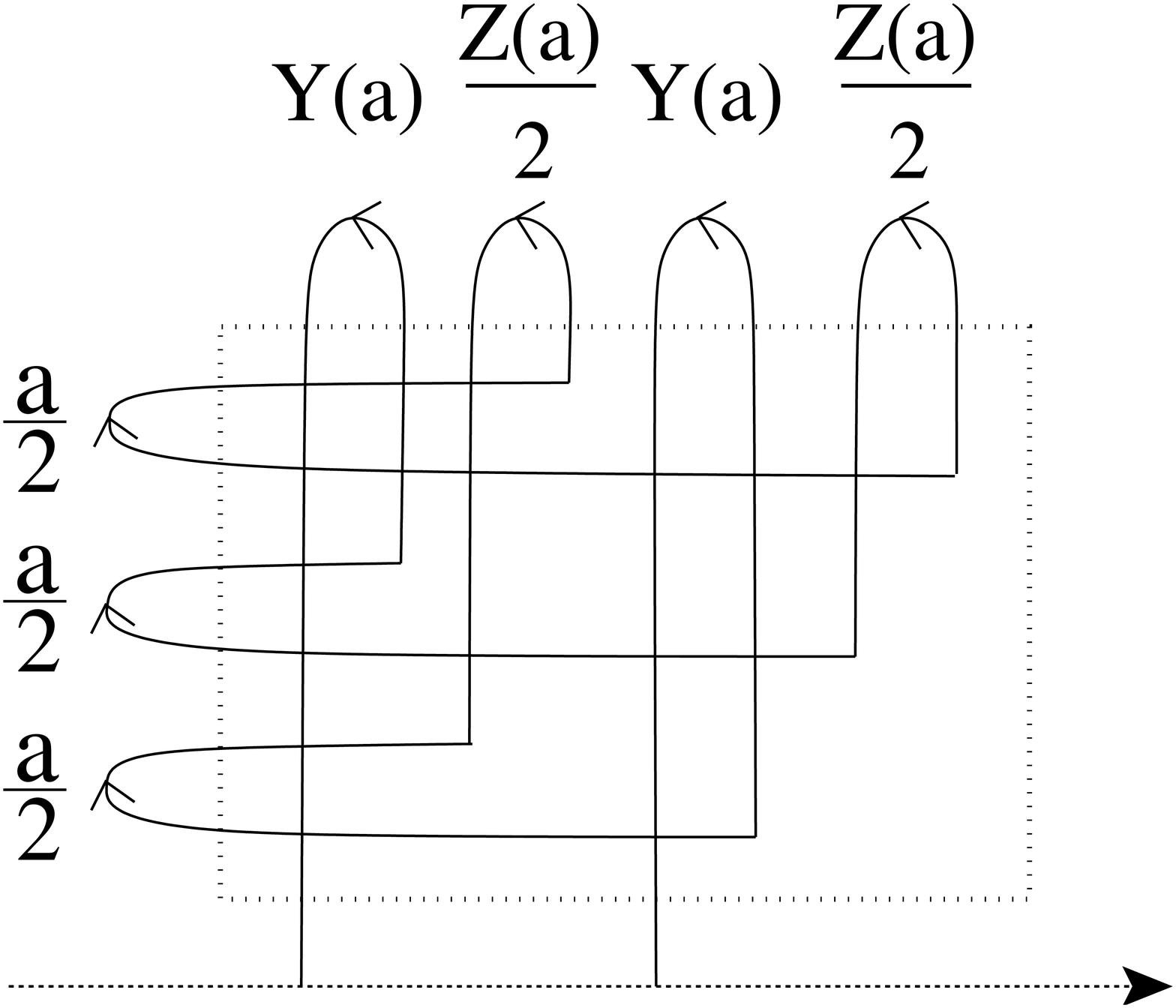}}}\ .
\]
The following lemma says that every contribution
$\theta_{(w_1,w_2)}$, for $(w_1,w_2)\in\overrightarrow{\Theta}_n$,
is equal.
\begin{lem}
Let $n$ be an integer $n\geq 1$ and let $(w_1,w_2)\in
\overrightarrow{\Theta}_n$. Then:
\[
\theta_{(w_1,w_2)} = -\frac{1}{n!(n+1)!}
\raisebox{-2ex}{\scalebox{0.18}{\includegraphics{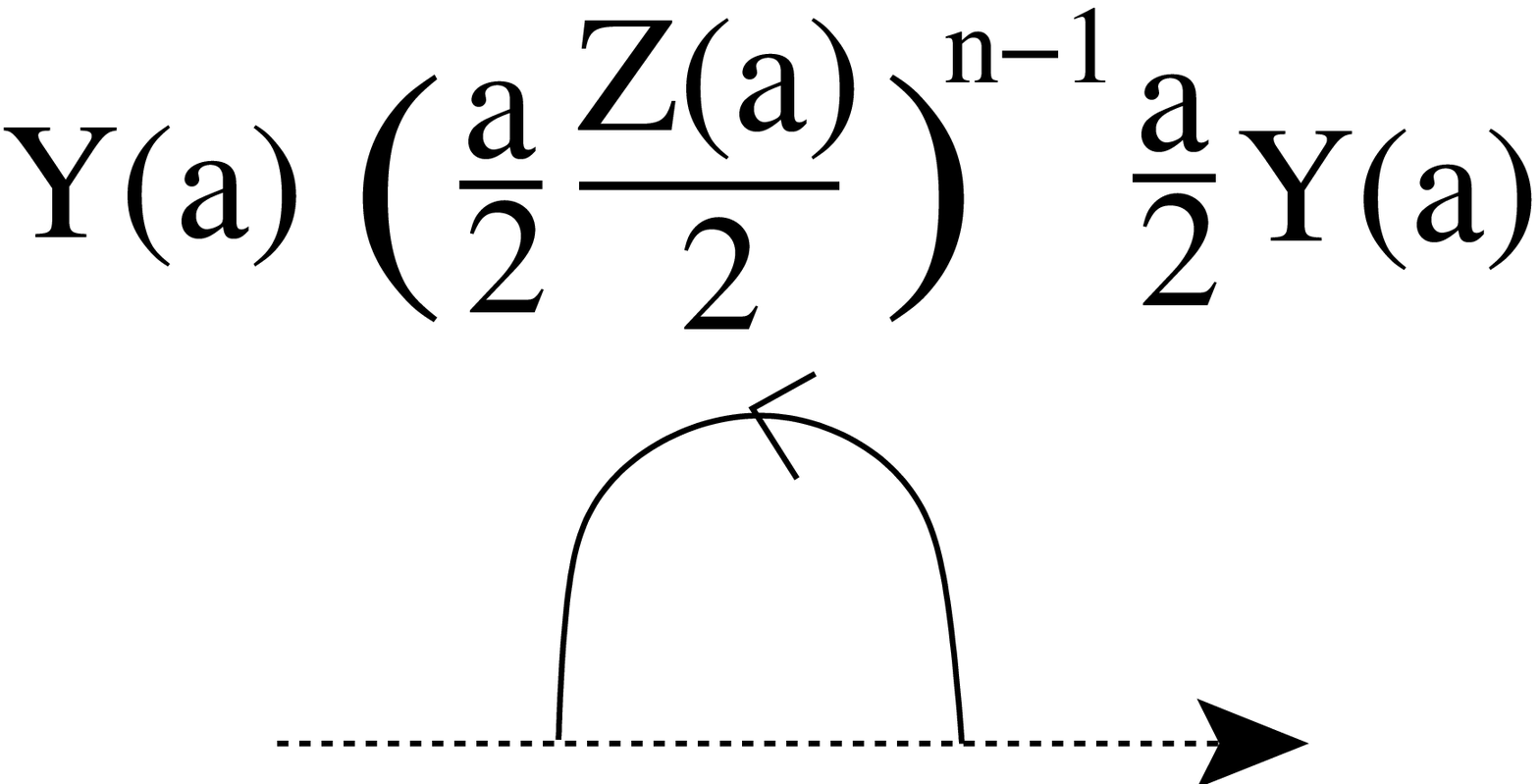}}}.
\]
\end{lem}
{\it Comments on the proof.} This proof is analogous to the proof
of Lemma \ref{allthesame}. In the present case, the ``standard
form" we wish to put the diagram into by means of column/row
transpositions (R-moves) and twists (W-moves) is
\[
(1\overrightarrow{2}\ldots\overrightarrow{n}\,n+1 , 1\downarrow
2\downarrow \ldots n\downarrow).
\]
For example, the standard contribution for $n=3$ is:
\[
\theta_{(1\overrightarrow{2}\overrightarrow{3}4,1\downarrow
2\downarrow 3\downarrow)}=\frac{1}{3!4!}(-1)^7\
\raisebox{-8ex}{\scalebox{0.13}{\includegraphics{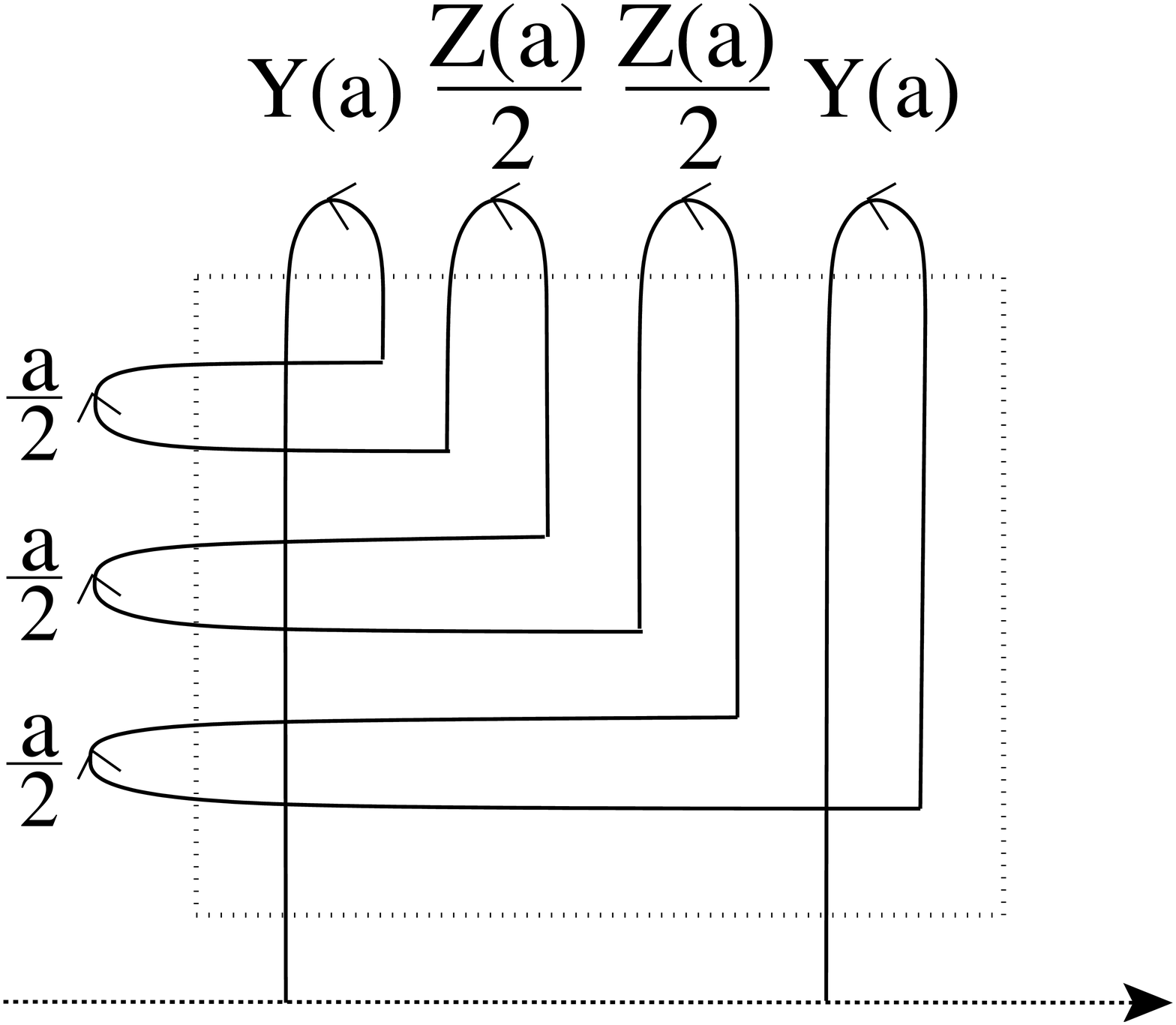}}}\ ,
\]
which is equal to
\[
\frac{1}{3!4!}(-1)\
\raisebox{-7ex}{\scalebox{0.13}{\includegraphics{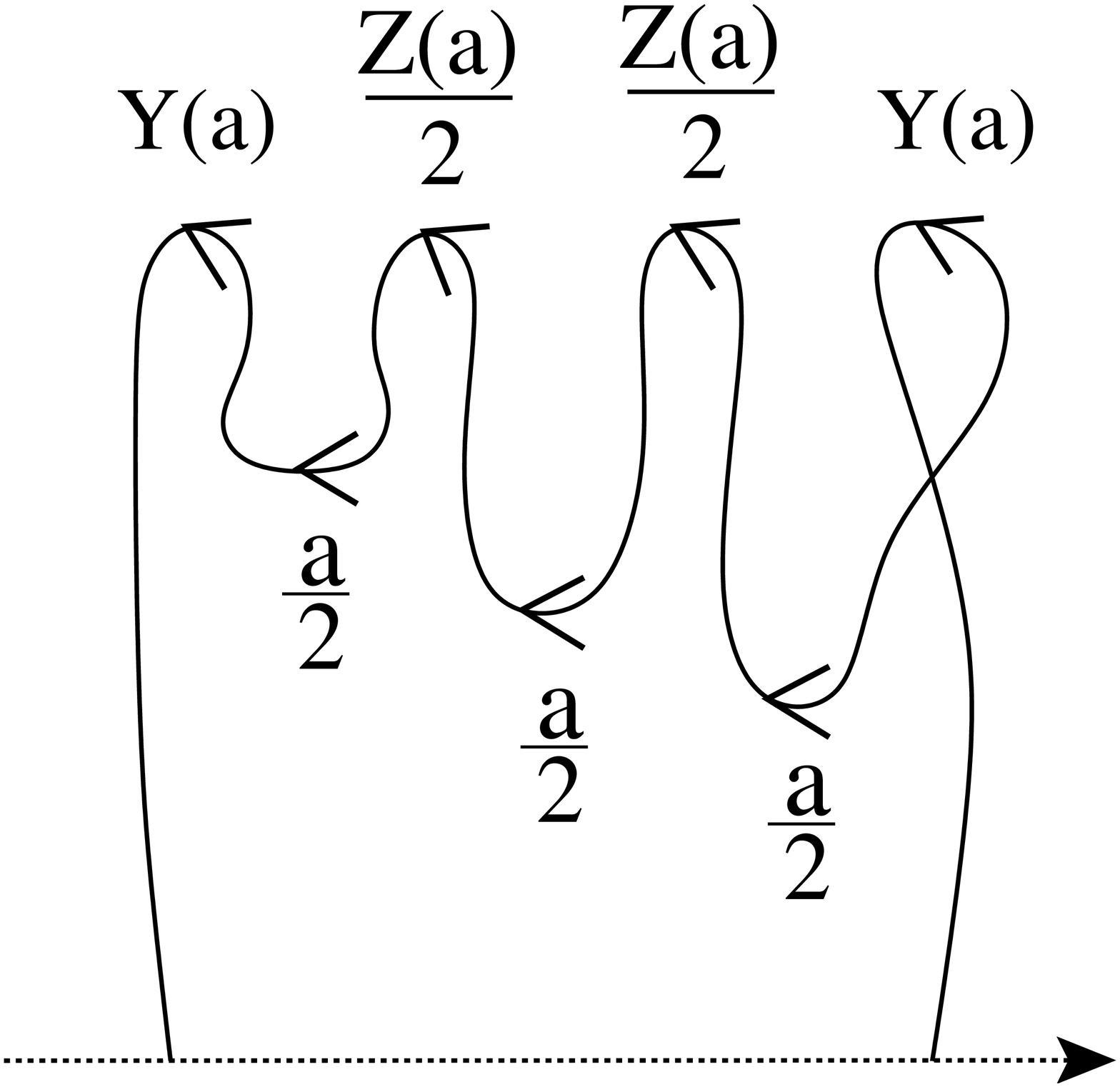}}}\ \
=\ \frac{1}{3!4!}(-1)\
\raisebox{-7ex}{\scalebox{0.13}{\includegraphics{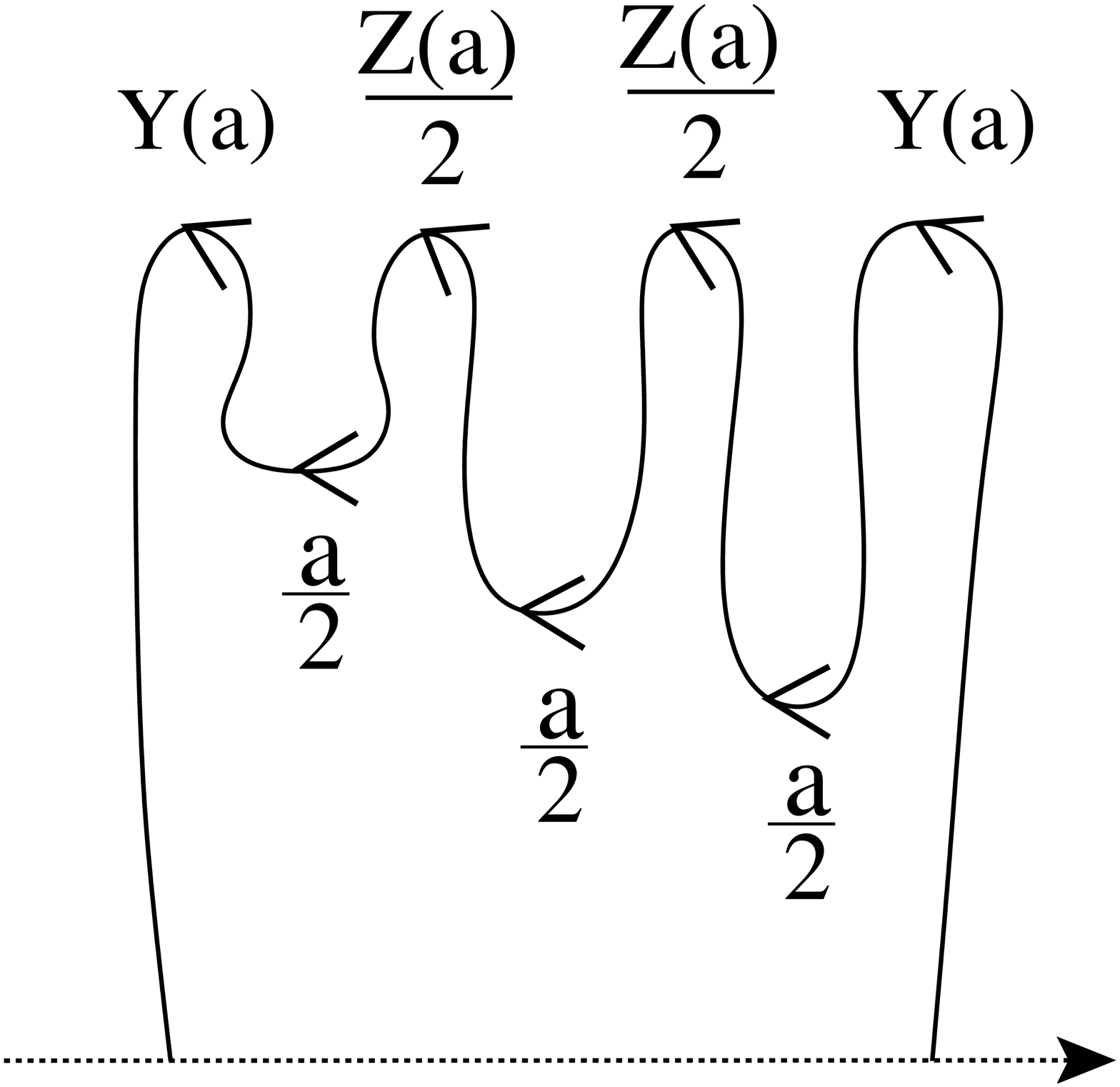}}}\
,\] as required. The final equality above used the fact that
$Y(a)$ was only assumed to have even powers of $a$.
\begin{flushright}
$\Box$
\end{flushright}

Finally, we can compute:
\begin{eqnarray*}
C_2 & = & \sum_{n=1}^{\infty}
\sum_{(w_1,w_2)\in\overrightarrow{\Theta}_n}\theta_{(w_1,w_2)} \\
& = & -\sum_{n=1}^{\infty}\left(
\left|\overrightarrow{\Theta}_n\right| \frac{1}{n!(n+1)!}
\raisebox{-2.5ex}{\scalebox{0.18}{\includegraphics{anstermsecond}}}\right)
\\
& = & -\sum_{n=1}^{\infty}\left(n!2^n\frac{(n+1)!}{2}2^{n-1}
\frac{1}{n!(n+1)!}
\raisebox{-2.5ex}{\scalebox{0.18}{\includegraphics{anstermsecond}}}\right)
\\ & = &
-\frac{1}{2}\sum_{n=1}^{\infty}\raisebox{-2ex}{\scalebox{0.18}{\includegraphics{anstermsecondB}}}\
.
\end{eqnarray*}
The ends the computation of $C_2$ and the proof of Proposition
\ref{opcompprop}.

\end{document}